\documentclass[a4paper,12pt,english]{smfbook}

\def\DD{{\SSS{D}}}

\def\logexp{{\mathrm{lex}}}
\def\KZ{{\mathrm{KZ}}}
\def\libertesymetrie{{(Free-Sym)}}
\def\HHC{{\!\SSS{\HC}}}
\def\HHb{{\!\SSS{\Hb}}}
\def\Hbd{{\dot{\Hb}}}
\def\opp{{\mathrm{opp}}}
\def\ln{{\mathrm{ln}}}
\def\nr{{\mathrm{nr}}}
\def\sec{{\mathrm{sec}}}
\usepackage{stmaryrd}

\usepackage{amssymb}
\usepackage{url}
\usepackage[T1]{fontenc}
\RequirePackage{calrsfs}
\DeclareSymbolFont{rsfscript}{OMS}{rsfs}{m}{b}
\DeclareSymbolFontAlphabet{\mathrsfs}{rsfscript}
\usepackage[utopia,expert]{mathdesign}
\usepackage{lscape}

\usepackage[leftbars]{changebar}
\usepackage{todonotes}
\usepackage{pgf}

\usepackage{palatino}
\usepackage{rotating}
\usepackage{graphicx}

\usepackage{enumerate}

\usepackage{color}
\definecolor{shadecolor}{gray}{0.90}

\def\bfit{\bfseries\itshape}

\input xypic
\xyoption{all}
\xyoption{arc}

\usepackage{imakeidx}
\makeindex
\def\indexnot#1#2{\index{#1@$#2$ |textbf  {\hskip0.5cm} \textbf }}
\def\indexname{Index of notation}

\newtheorem{theo}{Theorem}[section]
\def\thetheo{\thesection.\arabic{theo}}
\newtheorem{prop}[theo]{Proposition}

\newtheorem{lem}[theo]{Lemma}

\newtheorem{coro}[theo]{Corollary}

\newtheorem{defi}[theo]{Definition}

\newtheorem{defivide}{Definition}

\newtheorem{conj}{Conjecture}[section]

\renewcommand\theequation{\thetheo}
\def\equat{\refstepcounter{theo}\begin{equation}}
\def\endequat{\end{equation}}

\renewcommand\thetable{\thetheo}

\renewcommand\thesection{\thechapter.\arabic{section}}
\renewcommand\thesubsection{\thesection.\Alph{subsection}}

\def\contentsname{Table of contents}
\def\appendixname{Appendix}
\def\partname{Chapter}

\def\bibname{Bibliography}

\def\BZ{{\mathbb{Z}}}
\def\colim{\operatorname{colim}\nolimits}
\def\mMOD{\operatorname{\!-Mod}\nolimits}
\def\mmod{\operatorname{\!-mod}\nolimits}
\def\mmodgr{\operatorname{\!-modgr}\nolimits}

\def\mfree{\operatorname{\!-free}\nolimits}
\def\mproj{\operatorname{\!-proj}\nolimits}

\def\AG{{\mathfrak A}}  \def\aG{{\mathfrak a}}  \def\AM{{\mathbb{A}}}
  \def\bG{{\mathfrak b}}  
\def\CG{{\mathfrak C}}    \def\CM{{\mathbb{C}}}

    \def\FM{{\mathbb{F}}}
  \def\gG{{\mathfrak g}}  
  \def\hG{{\mathfrak h}}  
\def\IG{{\mathfrak I}}

  \def\lG{{\mathfrak l}}  \def\LM{{\mathbb{L}}}
  \def\mG{{\mathfrak m}}  
    \def\NM{{\mathbb{N}}}
    
  \def\pG{{\mathfrak p}}  
  \def\qG{{\mathfrak q}}  \def\QM{{\mathbb{Q}}}
  \def\rG{{\mathfrak r}}  \def\RM{{\mathbb{R}}}
\def\SG{{\mathfrak S}}

    \def\XM{{\mathbb{X}}}
    
  \def\zG{{\mathfrak z}}  \def\ZM{{\mathbb{Z}}}

\def\Ab{{\mathbf A}}  \def\ab{{\mathbf a}}  \def\AC{{\mathcal{A}}}
  \def\bb{{\mathbf b}}  \def\BC{{\mathcal{B}}}
    \def\CC{{\mathcal{C}}}
  \def\db{{\mathbf d}}  \def\DC{{\mathcal{D}}}
\def\Eb{{\mathbf E}}  \def\eb{{\mathbf e}}  \def\EC{{\mathcal{E}}}
    \def\FC{{\mathcal{F}}}
\def\Gb{{\mathbf G}}    \def\GC{{\mathcal{G}}}
\def\Hb{{\mathbf H}}    \def\HC{{\mathcal{H}}}
    \def\IC{{\mathcal{I}}}
    \def\JC{{\mathcal{J}}}
\def\Kb{{\mathbf K}}  \def\kb{{\mathbf k}}  \def\KC{{\mathcal{K}}}
\def\Lb{{\mathbf L}}    \def\LC{{\mathcal{L}}}
\def\Mb{{\mathbf M}}    \def\MC{{\mathcal{M}}}
    \def\NC{{\mathcal{N}}}
    \def\OC{{\mathcal{O}}}
\def\Pb{{\mathbf P}}    \def\PC{{\mathcal{P}}}
  \def\qb{{\mathbf q}}

\def\Tb{{\mathbf T}}  \def\tb{{\mathbf t}}  \def\TC{{\mathcal{T}}}
  \def\ub{{\mathbf u}}  
    
    \def\WC{{\mathcal{W}}}
  \def\xb{{\mathbf x}}  
    
    \def\ZC{{\mathcal{Z}}}

\def\Crm{{\mathrm{C}}}    \def\CCB{{\boldsymbol{\mathcal{C}}}}
    \def\DCB{{\boldsymbol{\mathcal{D}}}}

\def\Hrm{{\mathrm{H}}}    \def\HCB{{\boldsymbol{\mathcal{H}}}}

    \def\KCB{{\boldsymbol{\mathcal{K}}}}
    
\def\Mrm{{\mathrm{M}}}    \def\MCB{{\boldsymbol{\mathcal{M}}}}

    \def\PCB{{\boldsymbol{\mathcal{P}}}}
\def\Qrm{{\mathrm{Q}}}    \def\QCB{{\boldsymbol{\mathcal{Q}}}}
    \def\RCB{{\boldsymbol{\mathcal{R}}}}
    
\def\Trm{{\mathrm{T}}}    \def\TCB{{\boldsymbol{\mathcal{T}}}}

\def\Wrm{{\mathrm{W}}}    
    \def\XCB{{\boldsymbol{\mathcal{X}}}}
    \def\YCB{{\boldsymbol{\mathcal{Y}}}}
\def\Zrm{{\mathrm{Z}}}    \def\ZCB{{\boldsymbol{\mathcal{Z}}}}

\def\Qti{{\tilde{Q}}}    
\def\Rti{{\tilde{R}}}    \def\RCt{{\tilde{\mathcal{R}}}}

\def\Zti{{\tilde{Z}}}  \def\zti{{\tilde{z}}}  

  \def\ahat{{\hat{a}}}

  \def\ehat{{\hat{e}}}

\def\Qhat{{\hat{Q}}}

  \def\zhat{{\hat{z}}}

          \def\aba{{\bar{a}}}
          \def\bba{{\bar{b}}}
          
\def\Dba{{\bar{D}}}          
          \def\eba{{\bar{e}}}
\def\Fba{{\bar{F}}}          \def\fba{{\bar{f}}}
\def\Gba{{\bar{G}}}          \def\gba{{\bar{g}}}
          \def\hba{{\bar{h}}}
\def\Iba{{\bar{I}}}

\def\Pba{{\bar{P}}}          
          
          \def\rba{{\bar{r}}}
          
          \def\tba{{\bar{t}}}
          
          \def\vba{{\bar{v}}}

\def\Zba{{\bar{Z}}}          \def\zba{{\bar{z}}}

\def\Tov{{\overline{T}}}

\def\Hbt{{\widetilde{\Hb}}}

\def\Mbt{{\tilde{\Mb}}}

\def\a{\alpha}
\def\b{\beta}
\def\g{\gamma}
\def\G{\Gamma}
\def\d{\delta}
\def\D{\Delta}
\def\e{\varepsilon}
\def\ph{\varphi}
\def\l{\lambda}
\def\L{\Lambda}

\def\o{\omega}
\def\O{\Omega}
\def\r{\rho}
\def\s{\sigma}
\def\Sig{\Sigma}
\def\th{\theta}
\def\Th{\Theta}
\def\t{\tau}

\def\z{\zeta}

\def\betb{{\boldsymbol{\beta}}}

\def\delb{{\boldsymbol{\delta}}}        
\def\Delb{{\boldsymbol{\Delta}}}        
   
        \def\pht{{\tilde{\varphi}}}

\def\chib{{\boldsymbol{\chi}}}

\def\mub{{\boldsymbol{\mu}}}            \def\mut{{\tilde{\mu}}}
            
\def\omeb{{\boldsymbol{\omega}}}        
\def\Omeb{{\boldsymbol{\Omega}}}        
\def\pib{{\boldsymbol{\pi}}}            
            
          \def\rhot{{\tilde{\rho}}}
\def\sigb{{\boldsymbol{\sigma}}}

\def\taub{{\boldsymbol{\tau}}}          \def\taut{{\tilde{\t}}}

           \def\gamh{{\hat{\gamma}}}

               \def\muh{{\hat{\mu}}}
               
\def\omeba{{\bar{\omega}}}

\def\rhoba{{\bar{\rho}}}

\DeclareMathOperator{\ad}{{\mathrm{ad}}}

\DeclareMathOperator{\Aut}{{\mathrm{Aut}}}

\DeclareMathOperator{\End}{{\mathrm{End}}}
\DeclareMathOperator{\Ext}{{\mathrm{Ext}}}
\DeclareMathOperator{\Gal}{{\mathrm{Gal}}}
\DeclareMathOperator{\Hom}{{\mathrm{Hom}}}

\DeclareMathOperator{\Id}{{\mathrm{Id}}}
\DeclareMathOperator{\id}{{\mathrm{Id}}}
\DeclareMathOperator{\im}{{\mathrm{Im}}}

\DeclareMathOperator{\Ind}{{\mathrm{Ind}}}

\DeclareMathOperator{\Irr}{{\mathrm{Irr}}}
\DeclareMathOperator{\Ker}{{\mathrm{Ker}}}
\DeclareMathOperator{\Log}{{\mathrm{Log}}}

\DeclareMathOperator{\Mat}{{\mathrm{Mat}}}

\DeclareMathOperator{\Proj}{{\mathrm{Proj}}}

\DeclareMathOperator{\Rad}{{\mathrm{Rad}}}
\DeclareMathOperator{\Res}{{\mathrm{Res}}}

\DeclareMathOperator{\Spec}{{\mathrm{Spec}}}

\DeclareMathOperator{\Stab}{{\mathrm{Stab}}}

\DeclareMathOperator{\Tr}{{\mathrm{Tr}}}
\DeclareMathOperator{\Tor}{{\mathrm{Tor}}}

\DeclareMathOperator{\val}{{\mathrm{val}}}
\DeclareMathOperator{\disc}{{\mathrm{disc}}}
\DeclareMathOperator{\classe}{{\mathrm{Cl}}}
\DeclareMathOperator{\longueur}{{\mathrm{Length}}}

\def\to{\rightarrow}
\def\longto{\longrightarrow}
\def\injto{\hookrightarrow}

\def\fonction#1#2#3#4#5{\begin{array}{rccc}
{#1} : & {#2} & \longto & {#3} \\
& {#4} & \longmapsto & {#5} 
\end{array}}

\def\fonctio#1#2#3#4{\begin{array}{ccc}
{#1} & \longto & {#2} \\
{#3} & \longmapsto & {#4} 
\end{array}}

\def\bijectio#1#2#3#4{\begin{array}{ccc}
{#1} & \stackrel{\sim}{\longto} & {#2} \\
{#3} & \longmapsto & {#4} 
\end{array}}

\def\DS{\displaystyle}
\def\SS{\scriptstyle}
\def\SSS{\scriptscriptstyle}

\def\finl{~$\blacksquare$}

\def\lexp#1#2{\kern\scriptspace\vphantom{#2}^{#1}\kern-\scriptspace#2}
\def\le{\hspace{0.1em}\mathop{\leqslant}\nolimits\hspace{0.1em}}
\def\ge{\hspace{0.1em}\mathop{\geqslant}\nolimits\hspace{0.1em}}

\mathchardef\inferieur="321E
\mathchardef\superieur="321F

\def\eqna{\begin{eqnarray*}}
\def\endeqna{\end{eqnarray*}}

\def\itemth#1{\item[${\mathrm{(#1)}}$]}

\catcode`\@=11
\long\def\@car#1#2\@nil{#1}
\long\def\@first#1#2{#1}
\long\def\@second#1#2{#2}
\long\def\ifempty#1{\expandafter\ifx\@car#1@\@nil @\@empty
  \expandafter\@first\else\expandafter\@second\fi}
\catcode`\@=12

\def\Hbov{{\bar{\Hb}}}
\def\Kbov{{\bar{\Kb}}}
\def\Lbov{{\bar{\Lb}}}
\def\Mbov{{\bar{\Mb}}}

\def\pGt{{\tilde{\pG}}}
\def\qGt{{\tilde{\qG}}}

\def\rGt{{\tilde{\rG}}}
\def\GL{{\mathrm{GL}}}

\DeclareMathOperator{\REF}{Ref}

\def\boitegrise#1#2{\begin{centerline}{\fcolorbox{black}{shadecolor}{~
    \begin{minipage}[t]{#2}{\vphantom{~}#1\vphantom{$A_{\DS{A_A}}$}}
            \end{minipage}~}}\end{centerline}\medskip}

\def\ve{{\SS{\vee}}}
\def\cow{{\mathrm{co}(W)}}

\def\module{{\text{-}}{\mathrm{mod}}}

\def\BIL{LR}
\def\GAUCHE{L}
\def\CAR{C}
\def\FAM{FAM}
\def\COH{COH}
\def\ECOH{ECOH}
\def\FIX{FIX}

\theoremstyle{remark}
\newtheorem{rema}[theo]{Remark}
\newtheorem{remavide}{Remark}

\newtheorem{exemple}[theo]{Example}
\newtheorem{contre}[theo]{Counter-example}
\theoremstyle{plain}

\newtheorem{conjecturem}{Conjecture \FAM}

\newtheorem{conjecturebil}{Conjecture \BIL}

\newtheorem{conjectureleft}{Conjecture \GAUCHE}

\newtheorem{conjecturecar}{Conjecture \CAR}

\def\Frac{{\mathrm{Frac}}}
\newtheorem{conjecturecoh}{Conjecture \COH}

\newtheorem{conjectureecoh}{Conjecture \ECOH}

\newtheorem{conjecturefix}{Conjecture \FIX}

\def\groth{K_0}
\DeclareMathOperator{\dec}{dec}

\def\reg{{\mathrm{reg}}}
\def\blocs{{\mathrm{Idem_{pr}}}}
\def\refw{{\REF(W)/W}}
\def\grad{{\mathrm{gr}}}
\def\gradauto{{\mathrm{bigr}}}
\def\alg{{\mathrm{alg}}}
\def\euler{{\eb\ub}}
\def\eulerq{{\mathrm{eu}}}
\def\eulertilde{\widetilde{\euler}}
\def\casimir{{\mathrm{cas}}}

\def\pGba{\bar{\pG}}

\def\moins{{\hskip0.2mm -}}

\def\rGba{{\bar{\rG}}}

\def\qGba{{\bar{\qG}}}
\def\zGba{{\bar{\zG}}}
\def\decba{\overline{\dec}}

\def\calo{{\Crm\Mrm}}
\def\eval{{\mathrm{ev}}}
\def\KER{\KC \!\!\! e\!\! r}

\def\xyinj{\ar@{^{(}->}}
\def\xysur{\ar@{->>}}

\def\gauche{{\mathrm{left}}}
\def\droite{{\mathrm{right}}}

\def\isomorphisme#1{{\boldsymbol{[}}\hskip0.5mm #1\hskip0.5mm {\boldsymbol{]}}}
\def\res{{\mathrm{res}}}
\def\gen{{\mathrm{gen}}}
\def\parti{{\mathrm{par}}}
\def\bigrad{{\mathrm{bigr}}}
\def\partname{Chapitre}
\DeclareMathOperator{\carac}{{\mathrm{Car}}}
\bigskip\def\unb{{\boldsymbol{1}}}

\def\petitespace{\vphantom{$\DS{\frac{\DS{A^A}}{\DS{A_A}}}$}}
\def\trespetitespace{\vphantom{$\DS{\frac{\DS{A}}{\DS{A}}}$}}

\def\kl{{\mathrm{KL}}}

\def\singulier{{\mathrm{sing}}}
\def\ramif{{\mathrm{ram}}}
\def\cmcellules{\lexp{\calo}{\mathrm{Cells}}}

\def\klcellules{\lexp{\kl}{\mathrm{Cells}}}

\def\mult{{\mathrm{mult}}}

\def\cyclo{{\mathrm{cyc}}}
\def\heckegenerique{\HCB}
\def\heckecyclotomique{\HCB^\cyclo}

\makeatletter
\def\hlinewd#1{%
\noalign{\ifnum0=`}\fi\hrule \@height #1 %
\futurelet\reserved@a\@xhline}
\makeatother

\newlength\epaisLigne
\newcommand\clinewd[2]{\noalign{\global\epaisLigne\arrayrulewidth\global\arrayrulewidth #1}%
\cline{#2} \noalign{\global\arrayrulewidth\epaisLigne}}

\usepackage{multirow}

\def\prel{\leqslant_{L}^c}

\def\prer{\leqslant_{R}^c}
\def\prelr{\leqslant_{LR}^c}
\def\rell{\stackrel{L,c}{\longleftarrow}}

\def\relr{\stackrel{R,c}{\longleftarrow}}
\def\siml{\sim_{L}^{\kl,c}}
\def\simr{\sim_{R}^{\kl,c}}
\def\simlr{\sim_{LR}^{\kl,c}}

\def\CCBt{{\widetilde{\CCB}}}
\def\CGt{{\widetilde{\CG}}}

\def\copie{{\mathrm{cop}}}
\def\iso{{\mathrm{iso}}_0}

\newcommand{\longiso}{\stackrel{\sim}{\longrightarrow}}
\newcommand{\longbij}{\stackrel{\sim}{\longleftrightarrow}}

\def\MCov{{\bar{\MC}}}
\def\LCov{{\bar{\LC}}}

\def\schur{{\mathrm{sch}}}
\def\carac{{\mathrm{car}}}

\def\attractif{{\mathrm{att}}}
\def\repulsif{{\mathrm{rep}}}
\def\limiteattractive{{{\mathrm{lim}}_{{\mathrm{att}}}}}
\def\limiterepulsive{{{\mathrm{lim}}_{{\mathrm{rep}}}}}
\def\limiteattractiveinverse{{{\mathrm{lim}}_{{\mathrm{att}}}^{-1}}}

\def\limitegauche{{{\mathrm{lim}}_{{\mathrm{left}}}}}

\def\sym{{\mathrm{sym}}}

\renewcommand\proofname{Proof}
\newcommand{\longsurto}{\relbar\joinrel\twoheadrightarrow}
\newcommand{\longinjto}{\lhook\joinrel\longrightarrow}

\def\gaudin{{\mathrm{Gau}}}
\def\Liegaudin{{\mathfrak{gau}}}

\begin{document}

\baselineskip=16pt

\title{Cherednik algebras and Calogero-Moser Cells}

\author{{\sc C\'edric Bonnaf\'e}}
\address{
Institut de Montpelli\'erain Alexander Grothendieck (CNRS: UMR 5149), 
Universit\'e Montpellier 2,
Case Courrier 051,
Place Eug\`ene Bataillon,
34095 MONTPELLIER Cedex,
FRANCE} 

\makeatletter
\email{cedric.bonnafe@umontpellier.fr}
\makeatother

\author{{\sc Rapha\"el Rouquier}}

\address{UCLA Mathematics Department
Los Angeles, CA 90095-1555, 
USA}
\email{rouquier@math.ucla.edu}



\thanks{The first author is partly supported by the ANR: 
Projects No ANR-16-CE40-0010-01 (GeRepMod) and ANR-18-CE40-0024-02 (CATORE). \\
The second author was partly 
 supported by the NSF (grant DMS-1161999 and DMS-1702305) and by a
grant from the Simons Foundation (\#376202)}

\date{\\ March 7, 2022}



\maketitle

\pagestyle{myheadings}

\markboth{\sc C. Bonnaf\'e \& R. Rouquier}{\sc Cherednik algebras and Calogero-Moser cells}

\tableofcontents

%
%
%
%
%
%
%
%
%
%
%
%
%
%
%
%
%
%

\chapter*{Introduction}

\vskip-2cm
\noindent

\section{Reconstruction of Lie-theoretic structures from Weyl groups
and extension to complex reflection groups}

\medskip
A number of Lie-theoretic questions have their answer in terms of the
associated Weyl group. Our work is part of a program to reconstruct 
combinatorial and categorical structures arising in Lie-theoretic
representation theory from rational Cherednik algebras. Such algebras
are associated by Etingof and Ginzburg to more general complex reflection
groups, and an aspect of the program is to generalize those combinatorial
and categorical structures to complex reflection groups, that will not arise
from Lie theory in general.

\smallskip
To be more precise, consider a complex semisimple Lie algebra $\gG$ and
let $W$ be its Weyl group. Consider also a reductive algebraic group $\Gb$
over $\ZM$, with $\gG$ the Lie algebra of $\Gb(\CM)$. Consider the following:

\begin{itemize}
\item[(i)] (Parabolic) Blocks of (deformed) category $\OC$ for $\gG$,
blocks of categories of Harish-Chandra bimodules.
\item[(ii)] The set of unipotent characters of $\Gb(\FM_q)$, their
generic degrees, Lusztig's Fourier transform matrices.
\item[(iii)] The square part of decomposition matrices of unipotent blocks of
$\Gb(\FM_q)$ over a field of characteristic prime to $q$.
\item[(iv)] The Hecke algebra of $W$.
\item[(v)]  Lattices in the Hecke algebra arising from the Kazhdan-Lusztig
basis, Lusztig's asymptotic algebra $J$.
\item[(vi)] Kazhdan-Lusztig cells of $W$ and left cell
representations, families of characters of $W$.
\item[(vii)] Lusztig's modular categories associated to two-sided cells.
\end{itemize}

It is known or conjectured that the structures above depend only on $W$,
viewed as a reflection group. One can hope that (possibly super or
derived) versions of
those structures still make sense for $W$ a complex reflection group.

\smallskip
Consider the case where $W$ is a real reflection group.
A solution to (i) is provided by Soergel's bimodules \cite{Soe}. A
solution to (ii) was found \cite{BrMa,lusztigexotic,malleexotic}. 
The combinatorial theory
in (v,vi) extends (partly conjecturally) to that setting.
Categories as in (vii) were constructed by Lusztig when $W$ is a dihedral group
\cite{lusztigexotic}.

\smallskip
The structures above might
make sense for arbitrary ("unequal") parameters, and this is already an open
problem for
$W$ a Weyl group. A partly conjectural theory for (v,vi) has been developed
by Lusztig \cite{lusztig}, who is developing an interpretation via
character sheaves on disconnected groups \cite{lusztigdisconnected}.

\smallskip
Hecke algebras are a starting point: they have a topological definition that
makes sense for complex reflection groups
\cite{BMR}, providing a solution to (iv) (cf also 
\cite{ariki,ariki-koike,BMR,chavli1, 
chavli2,chavli3,marin1,marin2,marin3, 
marin-pfeiffer,tsuchioka}).
The Hecke algebras are deformations of $\ZM W$ over the space
of class functions on $W$ supported on reflections.

\smallskip
For certain complex reflection groups ("spetsial", see~\cite{malle generic}), 
a combinatorial set (a "spets") has
been associated by Brou\'e, Malle and Michel, that plays the role of unipotent
characters, providing an answer to (ii) \cite{malleunipotent,BMM,BMM2}.
Generic degrees are associated, building on Fourier
transforms generalizing Lusztig's constructions for Weyl groups. There are
generalized induction and restriction functors, and a $d$-Harish-Chandra theory.

\smallskip
When $W$ is a cyclic group and for equal parameters, a solution to
(vii) has been constructed in \cite{BR}. It is a derived version of
a modular category. It gives rise
to the Fourier transform defined by Malle \cite{malleunipotent}. 
This has been extended by Lacabanne to $G(d,1,n)$~\cite{lacabanne gd1n,lacabanne Fourier} and
to some twisted groups~\cite{lacabanne tordu} 
using ``super'' versions of modular categories~\cite{lacabanne super},
following 
suggestions of Etingof.

\smallskip
In this book, we provide a conjectural extension of (vi) to complex
reflection groups, using the geometry of Calogero-Moser spaces.



\section{Etingof-Ginzburg's rational Cherednik algebras and Calogero-Moser spaces}

\smallskip
Consider a non-trivial finite group $W$ acting on a finite-dimensional complex
vector space $V$ and let $V^\reg$ be the complement of the
ramification locus of the quotient map $V\to V/W$.
Assume $W$ is a reflection group, i.e., $W$ is
generated by its set of reflections $\REF(W)$ (equivalently: $V/W$ is smooth; equivalently:
$V/W$ is an affine space).
The quotient
variety $(V\times V^*)/\Delta W$ by the diagonal action of $W$ is singular.
It is a ramified covering of $V/W\times V^*/W$. These varieties carry a
$\CM^\times$-action coming from the symplectic action on $V\times V^*=T^*V$
induced by the scalar action on $V$.

Etingof and Ginzburg have constructed
a flat deformation $\Upsilon:\ZCB\to\PCB=\CCB\times V/W\times V^*/W$
of this covering with a $\CM^\times$-action \cite{EG}. 
Here, $\CCB$ is a vector space with basis the quotient of $\REF(W)$ by the
conjugation action of $W$. The variety $\ZCB$ is the Calogero-Moser space. The original
covering corresponds to the point $0\in\CCB$.

Etingof and Ginzburg define $\ZCB$ 
as the spectrum of the center
of the rational Cherednik algebra $\Hb$ associated with
$W$ at $t=0$. It is a remarkable feature of their work that those important
but complicated Calogero-Moser spaces have an explicit description based
on non-commutative algebra. We will now explain their constructions.

\medskip
The rational Cherednik algebra $\Hbt$ associated to $W$ is a flat
deformation defined by generators and relations
of the graded algebra $\CM[V\times V^*]\rtimes W$ over the space
of parameters $(c,t)\in \CCBt=\CCB\times\AM^1$. Its specialization at
$(c=0,t=1)$ is the crossed product of the Weyl algebra of $V$ by $W$.
The Cherednik algebra has a triangular decomposition
$\Hbt=\CM[V]\otimes \CM[\CCBt]W\otimes\CM[V^*]$. Equivalently,
it satisfies a PBW Theorem. 
The algebra $\Hbt$ has a faithful representation by Dunkl operators on
$\CM[\CCBt\times V^\reg]$. The Euler element admits a deformation
that acts by derivation as multiplication by $dT$ on the degree $d$ part
of $\Hbt$.

Consider the algebra $\Hb$ obtained by specializing $\Hbt$ at $t=0$ and
let $Z$ be its center. It contains $P=\CM[\CCB]\otimes\CM[V]^W\otimes
\CM[V^*]^W$ as a subalgebra.
The Calogero-Moser variety is defined as $\ZCB=\Spec Z$ and the
inclusion $P\subset Z$ defines the covering $\Upsilon$.

\smallskip
Our main object of study is the representation theory of $\Hb$ as a
$P$-algebra and its interaction with the ramification of $\Upsilon$
 above $c\times 0\times 0$, $c\times 0\times V^*/W$ and $c\times V/W\times 0$.

\section{Families}

\smallskip
Fix a parameter $c\in\CCB$. We  consider $\Hb_c$ the
specialization of $\Hb$ at $c$ and the restricted rational Cherednik
algebra $\bar{\Hb}_c$, the specialization of $\Hb$ at 
$(c\times 0\times 0)\in\CCB\times V/W\times V^*/W$.

Given a simple $\CM W$-module $E$, we have a Verma $\Hb_c$-module 
$\Delta_c(E)=\Hb_c\otimes_{\CM[V]\rtimes W}E$, where $V^*$ acts by $0$ on $E$.
We have also a restricted version $\bar{\Delta}_c(E)=\bar{\Hb}_c
\otimes_{\Hb_c}\Delta_c(E)$ which has a unique simple quotient $L_c(E)$. The
$L_c(E)$'s for $E\in\Irr(W)$ give all simple $\bar{\Hb}_c$-modules \cite{gordon}.

The partition into blocks of
those modules gives a partition of $\Irr(W)$ into {\em Calogero-Moser
families}. They  are in bijection with $\Upsilon^{-1}(c\times 0\times 0)=
\ZCB_c^{\CM^\times}$ \cite{gordon}.

We show that,
in a given Calogero-Moser family, the matrix of multiplicities
$[\bar{\Delta}_c(E):L_c(F)]$ has rank $1$, a property
conjectured by Ulrich Thiel \cite{thiel}.

In each Calogero-Moser family 
there is a unique irreducible representation of $W$ with minimal
$\bb$-invariant (the invariant $\bb_E$ is the minimal $d$ such that
$E$ occurs in $S^d(V)$).

Families satisfy a semi-continuity property with respect to specialization 
of the parameter.  We show that families are minimal subsets that
are unions of families for a generic parameter and unions of blocks of Hecke 
algebras for certain specializations (see Theorem~\ref{th:CMfamiliesblocksHecke}). 
In particular, the subvariety of $\CCB$ where families are not generic
is contained in the union of hyperplanes $c_E-c_F=0$
where $E$ and $F$ are in distinct generic families.



\section{Cellular characters}

Let $\Kb_c^{\gauche}$ be the function field of $c\times V/W\times 0\subset\PCB$.
The simple $\Kb_c^{\gauche}\Hb$-modules are determined by the action
of the center and this provides a bijection from $\Irr(\Kb_c^{\gauche}\Hb)$
to the set of irreducible components of $\Upsilon^{-1}(c\times V/W\times 0)$.

Let $\zG$ be the defining prime ideal of an irreducible component of
$\Upsilon^{-1}(c\times V/W\times 0)$.
Let $E\in\Irr(W)$. We define $\mult_{\zG,E}^{\calo}$ as the multiplicity
of the simple $\Kb_c^{\gauche}\Hb$-module corresponding to $\zG$ in 
$\Kb_c^{\gauche}\Delta_c(E)$.

We put $\gamma_\zG^{\calo}=\sum_{E\in\Irr(W)}\mult_{\zG,E}^{\calo}E$: this
is the {\it cellular representation} of $W$ associated with $\zG$.

If $\mult_{\zG,E}^{\calo}\neq 0$, then $E$ is contained in the family
corresponding to the unique $\CM^\times$-fixed point of $\Spec(Z/\zG)$.
If that fixed point is smooth in $\ZCB_c$, then there is a unique $E$ in
the corresponding family \cite{gordon} and 
\equat\label{eq:intro-lisse}
\gamma_\zG^{\calo}=E.
\endequat

In each cellular character,
there is a unique irreducible representation of $W$ with minimal
$\bb$-invariant.

We show that cellular characters are sums of characters of projective
modules of Hecke algebras for certain specializations 
(see Corollary~\ref{coro:proj-cellular}).

\section{Galois closure and ramification}

\smallskip

The covering $\Upsilon$, of degree $|W|$, is not Galois (unless
$W=(\BZ/2)^n$). Let $\r : \RCB \to \PCB$ denote a Galois closure 
(with $\RCB$ a normal variety) and let $G$ be its Galois group. At $0\in\CCB$, a
Galois closure of the covering $(V\times V^*)/\Delta W\to V/W\times V^*/W$ 
is given by $(V\times V^*)/\Delta \Zrm(W)$. This leads
to a realization of $G$ as a group of permutations of the set $W$.

This can be reformulated in terms of representations of $\Hb$: the
semisimple $\CM(\PCB)$-algebra $\CM(\PCB)\Hb$
is not split and $\CM(\RCB)$ is a splitting field. The simple
$\CM(\RCB)\Hb$-modules are in bijection with $W$. Our
work can be viewed as the study of the partition 
of these modules into blocks corresponding to a given prime ideal of
$\CM[\RCB]$.

We show that $G$ is the Galois group of the minimal polynomial of the Euler
element.
This
element plays an important role in the study of ramification, but is not
enough to separate cells in general.

\section{Calogero-Moser cells}

Let $X$ be an irreducible closed subvariety of $\RCB$. We define the
$X$-cells of $W$ as the orbits of the inertia group of $X$ in $G$. 
The partition into $X$-cells corresponds to the partition of the simple
$\CM(\RCB)\Hb$-modules into blocks for the defining prime ideal of $X$.

Given a
parameter $c\in\CCB$, we define the {\em two-sided $c$-cells} (resp.
{\em left $c$-cells}, resp. {\em right $c$-cells}) as the $X$-cells 
defined for $X$ contained in $\r^{-1}(c\times 0\times 0)$ (resp. in
$\r^{-1}(c\times V/W\times 0)$, resp. in $\r^{-1}(c\times 0\times V^*/W)$).

The set of two-sided $c$-cells is in bijection with the
set of families, i.e. with $\Upsilon^{-1}(c\times 0\times 0)$. Given $\G$ a two-sided cell, 
we have
\equat\label{eq:cardinal-bilatere}
|\G|=\sum_{E \in \FC_\G} (\dim E)^2,
\endequat
where $\FC_\G$ is the family associated with $\G$. 

The orbits of left cells under the decomposition group for $X$
are in bijection with the 
irreducible components of $\Upsilon^{-1}(c \times V/W \times 0)$. This allows 
to associate to each left cell $C$ a cellular character $\isomorphisme{C}$. We have
\equat\label{eq:cardinal-gauche}
|C|=\dim \isomorphisme{C}.
\endequat


We analyze in detail the case where $W$ is cyclic (and $\dim V=1$): this is
the only
case where we have a complete description of all objects studied in this
book.

\section{Gaudin operators and topology}

\smallskip
Our discussion so far has been in the algebraic setting. The
study of ramification can be done in the topological setting.

Let $\Hb^{\reg}=\Hb\otimes_{\CM[V]^W}\CM[V^{\reg}]^W$.
Fix $(c,v,v^*)\in\CCB\times V\times V^*$ and consider 
$L(c,v,v^*)=\CM\otimes_{\CM[\CCB\times V^{\reg}]}\Hb^{\reg}\otimes_{\CM[V^*]}
\CM$
where the morphism $\CM[\CCB\times V^{\reg}]\to\CM$ corresponds to the point
$(c,v)$ and the morphism $\CM[V^*]\to\CM$ corresponds to the point $v^*$.

The triangular decomposition of $\Hb$ provides an isomorphism
$\CM W\xrightarrow{\sim} L(c,v,v^*)$.

Using Dunkl operators, Etingof and Ginzburg constructed an isomorphism
of algebras $(\CM[V^{\reg}]\otimes\CM[V^*])\rtimes W\xrightarrow{\sim}
\Hb^{\reg}$. The image of $V\subset\CM[V^*]$ in
$\Hb^{\reg}$ is a commutative Lie subalgebra
that acts on $L(c,v,v^*)$ by left multiplication. This provides
a commutative Lie subalgebra $\Liegaudin_{c,v,v^*}$ of $\End_{\CM}(\CM W)$ ("Gaudin operators").

Taking $v$ generic, the generalized eigenspaces of $\Liegaudin_{c,v,0}$ 
are the left cell representations of $W$.

When $v^*$ is also generic,
the generalized eigenspaces of $\Liegaudin_{0,v,v^*}$ are one-dimensional and
can be parametrized by $W$. Consider now a continuous
path in $\CCB\times V^{\reg}\times V^*$ starting
at $(0,v,v^*)$ and ending at $(c,v,0)$, avoiding the ramification locus
except at the end. The parametrization
by $W$ of the eigenvalues can be extended by continuity along the path.
This gives a family of paths in $V^*$ parametrized by $W$.
Two elements of $W$ are in the same left cell if and only if the corresponding
paths have the same endpoints.

\section{Coxeter groups}

We assume here that $W$ is a finite Coxeter group (i.e. a real reflection group) 
and that $c$ takes real values. Forty years ago, Kazhdan-Lusztig~\cite{KL} and 
Lusztig~\cite{Lu2} used their basis of the Hecke algebra to define notions 
of families, cellular representations, and cells. Our work aims to 
generalize these notions to complex reflection groups using Cherednik algebras 
instead of Hecke algebras. We conjecture that our notions coincide with their 
notions for Coxeter groups. 

The conjecture on families was first stated by Gordon-Martino~\cite{gordon martino} and 
is known to hold in many cases~\cite{gordon martino, bellamy these, martino 2, bonnafe thiel}. 

The conjecture on cellular representations is known to hold in type $A$ and in type $B$ for some 
generic values of $c$ (thanks to~\eqref{eq:intro-lisse}) and in type $I_2(m)$ for any 
parameter~\cite{bonnafe diedral}.

We prove here that the conjectures on cells hold in type $B_2$. They also 
hold in type $I_2(m)$ when $c$ is constant~\cite{bonnafe germoni} and in type $A$~\cite{BGW}. 

We prove in the Calogero-Moser setting some properties known in the Kazhdan-Lusztig setting:
\begin{itemize}
\itemth{1} The cardinality of cells given by~\eqref{eq:cardinal-bilatere} and~\eqref{eq:cardinal-gauche}.

\itemth{2} If $w_0=-\Id_V \in W$ and $\G$ (resp. $C$) is a two-sided (resp. left) cell, 
then $\G w_0$ (resp. $C w_0$) is a two-sided (resp. left) cell and 
$$\FC_{\G w_0}=\FC_\G\cdot \e\qquad\text{and}\qquad 
\isomorphisme{C w_0}=\isomorphisme{C} \cdot \e.$$

\itemth{3} Unicity of a representation with minimal $\bb$-invariant.
\end{itemize}
In Lusztig's theory, for $c$ constant, there is a unique irreducible representation 
with minimal $\bb$-invariant in a family, the special representation. 
It occurs in every left cell representation in the family. This last 
fact does not hold for $c$ non-constant but (3) is known to hold in 
the Kazhdan-Lusztig setting~\cite{bonnafe b} modulo a conjecture of Lusztig~\cite{lusztig}. 
It is an instance of the Calogero-Moser theory shedding some
light on the Kazhdan-Lusztig and Lusztig theory.

We also provide a detailed study of 
type $B_2$: the Galois group is a Weyl group of type $D_4$ and we show
that the Calogero-Moser cells coincide with the Kazhdan-Lusztig cells. Our
approach for $B_2$ is based on a detailed study of $\ZCB$ and of the
ramification of the covering, without constructing explicitly the variety
$\RCB$.

\section{Description of the chapters}

\medskip
We start in Chapter~\ref{chapter:notation} with general notations about 
algebras, modules and gradings.
We review in Chapter \ref{chapter:reflexions} the basic theory of complex
reflection groups:
invariant theory, hyperplane complement,
rationality of representations, Hilbert series and
fake degrees. We close that
chapter with the particular case of real reflection groups endowed with
the choice of a real chamber, i.e., finite Coxeter groups. Throughout the
book, we devote special sections to the case of Coxeter groups when
particular features arise in their case.

\smallskip
Chapters  \ref{chapter:cherednik-1} and \ref{chapter:cherednik-0}
are devoted to the basic structure
theory of rational Cherednik algebras, following Etingof and Ginzburg 
\cite{EG}. The definition of generic Cherednik algebras is given in
Chapter \ref{chapter:cherednik-1}, followed by the fundamental PBW
Decomposition
Theorem and the faithful polynomial representation via Dunkl
operators.
We discuss the spherical algebra and 
some of its basic properties, in particular
the Double Endomorphism Theorem.
We also introduce the Euler element and
gradings, filtrations, and automorphisms.

Chapter \ref{chapter:cherednik-0} is devoted to the Cherednik algebra
at $t=0$. An important result is the Satake isomorphism between
the center of the Cherednik algebra and the spherical subalgebra. We
discuss localizations and cases of Morita equivalence between the
Cherednik algebra and its spherical subalgebra. We provide some
complements: filtrations, symmetrizing form, Poisson structure and Hilbert series.

\medskip
Our original work starts in Part \ref{part:reps}: this part introduces and
studies families of characters and cellular characters. There are
the representation theoretic shadows of the Calogero-Moser 
cells that will be constructed later. An important aspect is that
this part does not involve a Galois closure of the covering $\Upsilon$.

\smallskip
Chapter \ref{se:representations} introduces a certain category $\OC$ of (graded)
representations of rational Cherednik algebras. It is a highest category
in the generalized sense of Appendix  \S\ref{ap:hw}. Of particular importance
are Verma modules and the action of the Euler element on them. When
$t=1$, we recover the category $\OC$ of \cite{ggor}. When $t=0$, we explain
representation-theoretic interpretations of the smoothness of the Calogero-Moser
space~\cite{EG}. We introduce a generalization of Gaudin operators (cf
\cite{MuTaVa1,MuTaVa3} for the symmetric group case), as
endomorphisms of a family of representations of $\Hb$. The spectral scheme
of Gaudin operators identifies with a pullback of the covering $\Upsilon$.

\smallskip
Chapter \ref{ch:Hecke} is devoted to Hecke algebras. We recall
in \S\ref{se:definitionsHecke} the definition of Hecke algebras of complex
reflection groups and some of their basic (partly conjectural) properties.
We introduce a "cyclotomic" version, where the Hecke parameters are powers
of a fixed indeterminate.
We explain in~\S\ref{se:KZ} the construction of the
Knizhnik-Zamolodchikov functor \cite{ggor} realizing the category of
representations of the Hecke algebra as a quotient of a (non-graded)
category $\OC$
for the Cherednik algebra at $t=1$. Thanks to the Double Endomorphism
Theorem, the semisimplicity of the Hecke algebra is equivalent to that
of the category $\OC$.
We present in \S\ref{section:representations-hecke} Malle's splitting
result~\cite{malle} for irreducible representations of Hecke algebras and we consider
central characters.
We discuss in \S\ref{se:Heckefamilies} the notion of Hecke families.
We finish in \S\ref{section:cellules-kl} with a brief exposition of the
theory of Kazhdan-Lusztig cells of $W$ and of families of characters of $W$ and
$c$-cellular characters.

\smallskip
Chapter \ref{chapter:bebe-verma} is devoted to the representation theory
of restricted Cherednik algebras and to Calogero-Moser families.
We recall in \S\ref{se:represtricted} and \S\ref{section:familles CM}
some basic results of Gordon
\cite{gordon} on representations of restricted Cherednik algebras and
Calogero-Moser families. Graded
representations give rise to a highest weight category, as proven by
Bellamy and Thiel \cite{BelTh1}. We show
in \S\ref{section:dim graduee} the existence of a unique representation with
minimal
$\bb$-invariant in each family and generalize results of \cite{gordon}
on graded dimensions. We discuss in \S\ref{section:geometrie CM} and
\S\ref{se:smoothnessCMfamilies} the relation
between the
geometry of the Calogero-Moser space at $\Upsilon^{-1}(0)$ and the
Calogero-Moser families, with a focus on smoothness
\cite{EG, GK,gordon,bellamy g4,BelSchTh}. The final section \S\ref{se:blocksCMfamilies} relates 
Calogero-Moser families with blocks of category $\OC$ at $t=1$ and
with blocks of Hecke algebras. 

\smallskip
We give in Chapter~\ref{chapter:gauche-cellulaire} the definition
of Calogero-Moser (left) 
cellular characters. Our first definition is in terms of representations
of the specialization of $\Hb$
at the prime ideal of $\CM[\PCB]$ defining the subvariety 
$V/W \times 0$ of $\PCB=V/W \times V^*/W$: to each 
irreducible representation of this specialization we associate
a character of $W$. 
We provide a second equivalent definition based on the representation 
theory of Gaudin algebras, which is more convenient for explicit 
computations. We prove that Calogero-Moser 
cellular characters are sums of characters of projective modules 
of the suitably specialized Hecke algebra.

\smallskip
We analyze in Chapter \ref{chapter:bb} the $\CM^\times$-action: its fixed points
are in bijection with families. We relate attracting sets and cellular characters.
We prove that there is a unique cellular character in a family corresponding
to a smooth point of a Calogero-Moser space in $\Upsilon^{-1}(0)$ and
that it is irreducible.

\medskip

We build in Part~\ref{part:extension} the foundations 
for defining general Calogero-Moser cells: the particular 
cases of left, right and two-sided cells will be studied 
in Part~\ref{part:verma}. 

\smallskip
We introduce in Chapter \ref{chapter:galois-CM}
some basic objects of our study, namely a Galois closure
of the covering $\Upsilon$ and its Galois group $G$. At parameter $0$, the
corresponding
data is easily described, and its embedding in the family depends on
a choice. We explain this in \S\ref{subsection:specialization galois 0}, and
show that this allows an identification of the generic fiber of $\Upsilon$
with $W$. We show in \S\ref{section:galois euler} that the extension
$\CM(\ZCB)/\CM(\PCB)$ is generated by the Euler element. The semisimple
algebra $\CM(\PCB)\Hb$ is not split and we provide in
\S\ref{section:deploiement} a decomposition of the $\CM(\RCB)$-algebra
$\CM(\RCB)\Hb$ as a product of matrix algebras over $\CM(\RCB)$ indexed
by $W$. 
In other parts of \S\ref{chapter:galois-CM}, we discuss
gradings and automorphisms, and construct a central element of order $2$ of $G$
when all reflections of $W$ have order $2$ and $-\Id_V\in W$.
The last section~\S\ref{se:geometrie-CM} is a geometrical translation of the
previous constructions.

\smallskip
We introduce general Calogero-Moser cells in Chapter \ref{chapter:cellules-CM}. They
are defined in \S\ref{section:definition cellules} as orbits of inertia groups
on $W$ and shown in
\S\ref{subsection:cellules et blocs} to coincide with blocks of the Cherednik
algebra. We study next the ramification locus and smoothness. We give
two more equivalent definitions of Calogero-Moser cells:
via irreducible components of the base change by $\Upsilon$ of the Galois
cover (\S\ref{se:geometryCMcells}) and via lifting of paths
(\S\ref{se:cellstopology}).

\medskip
Part \ref{part:verma} is the heart of the book. It discusses Calogero-Moser
cells
associated with the ramification above $0\times 0$, $V/W\times 0$ and
$0\times V^*/W$, and relations with representations of Cherednik
algebras, as well as (conjectural) relations with Hecke algebras.

\smallskip
In Chapter \ref{chapter:bilatere}, we go back to the Galois cover $\RCB/\PCB$ and study
two-sided cells. We
construct a bijection between the set of two-sided cells and the set of
families.

\smallskip
We continue in Chapter \ref{chapter:gauche} with the study of left (and right)
cells. We analyze in
\S\ref{section:choix-gauche-R} the choices involved in the definition of left
cells using Verma modules. We study the relevant decomposition groups,
and we reinterpret left cells as blocks of a suitable specialization of
the Cherednik algebra. We finish in \S\ref{se:backtocellular} with basic
properties relating cellular characters and left cells and we give
an alternative definition of cellular characters as the socle of the
restriction of a projective module. We also show that a cellular
character involves a unique irreducible representation with minimal
$\bb$-invariant. Section~\ref{se:chemins} shows, following a suggestion of Etingof, 
that cells can 
be interpreted in terms of spectra of Gaudin operators.
This provides a topological approach to cells and cellular characters.

\smallskip
Chapter \ref{ch:decmat} brings decomposition matrices in the study of cells.
We show in \S\ref{se:isobabyVerma}
that the decomposition matrix of baby Verma modules in a block has rank $1$,
as conjectured by Thiel \cite{thiel}.

\smallskip
The next two chapters are devoted to conjectures. 
Chapter \ref{part:coxeter} discusses the motivation of this book, namely
the expected relation between Calogero-Moser cells and Kazhdan-Lusztig cells
when $W$ is a Coxeter group. We start in \S\ref{chapter:hecke} with
Martino's conjecture that Calogero-Moser families are unions of Hecke
families. \S\ref{section:conjectures} and \S\ref{chapter:arguments}
state and discuss our main conjectures. We providence some evidence and
describe some cases where the conjectures hold.

\smallskip
Chapter \ref{ch:conjgeometry} gives a conjecture on the cohomology ring and
the $\CM^\times$-equivariant cohomology ring of Calogero-Moser spaces,
extending the description of Etingof-Ginzburg in the smooth case.
We also conjecture that irreducible components of the fixed points
of finite order automorphisms on the Calogero-Moser space are Calogero-Moser
spaces for reflection subquotients. This conjecture has been extended to
symplectic leaves in \cite{bonnafe autosymp}.

\medskip
Part~\ref{part:exemples} is based on the study of particular cases.
Chapter \ref{chapitre nul} presents the theory for the parameter $c=0$.

\smallskip
Chapter \ref{chapitre:rang 1} is devoted to the case where $V$ has dimension $1$.
We give a description of the objects introduced earlier, in
particular the Galois closure $R$. We show that generic decomposition groups
can be very complicated for some values of the parameters. We also 
compute the cohomology of the Calogero-Moser space, even when it is singular, 
confirming the conjecture on cohomology.

\smallskip
Chapter \ref{chapitre:b2} analyzes the case of $W$ a Coxeter group
of type $B_2$. We determine in \S\ref{section:quotient B2}
the ring of diagonal invariants and the
minimal polynomial of the Euler element. We continue in \S\ref{section:Q B2}
with the determination of the corresponding deformed objects. The
Calogero-Moser families are then easily found. We move next to the
determination of the Galois group $G$. Section \S\ref{se:cellsB2}
is the more complicated study of ramification and the determination of the
Calogero-Moser cells. We finish in \S\ref{sec:fixed-b2} with a discussion of
fixed points
of the action of groups of roots of unity, confirming the conjecture on
fixed points.

\smallskip
Chapter~\ref{chapter:diedral} illustrates the topological approach 
(via Gaudin algebras) of the computation of cells in the particular 
case of finite dihedral groups. No proof is given, but we provide
pictures that give evidence for our main conjectures.

\medskip
We have gathered in six appendices some general algebraic considerations,
few of which are original.

\smallskip
Appendix \ref{appendice:filtration} is a brief exposition of
filtered modules and filtered algebras. We analyze in particular 
properties of an algebra with a filtration
that are consequences of the corresponding properties for the associated
graded algebra. We discuss symmetrizing forms in \S\ref{se:symmalg} and
we finish with Weyl algebras in \S\ref{se:Weylalgebra}.

\smallskip
Appendix \ref{chapter:galois-rappels} gathers some basic facts on
ramification theory for commutative rings around decomposition and
inertia groups. We recollect some properties of Galois groups,
discriminants and integral closures. We close the chapter with a topological
version of ramification theory and its connection with commutative ring 
theory.

\smallskip
Appendix \ref{appendice graduation} is a discussion of some aspects of the
theory of graded rings. We consider general rings in \S\ref{intro graduation}.
We next discuss in \S\ref{section:graduation integrale}
gradings in the setting of integral extensions of commutative rings. We finally
consider gradings and invariant rings in \S\ref{section:GR}.

\smallskip
We present in Appendix \ref{appendice: blocs} some results on blocks and
base change for algebras finite and free over a base.
We discuss in particular decomposition maps,
central characters and idempotents, and the locus where the
block decomposition changes.

\smallskip
Appendix \ref{appendice:invariant} deals with finite group actions on
rings (commutative or not), and compares the cross-product
algebra and the invariant ring. We consider in particular
the module categories and the centers.

\smallskip
Appendix \ref{ap:hw} provides a generalization of the theory of highest weight
categories over commutative rings \cite{CPS,CPS2,BelTh1}.
We discuss in particular base change
(\S\ref{se:basechange}), Grothendieck groups (\S\ref{hw:K0} and
\S\ref{se:completedK0}), decomposition maps (\S\ref{se:hwdecmap}) and blocks
(\S\ref{se:hwblocks}). A particular class of highest weight categories
arises from graded algebras with a triangular decomposition 
(\S\ref{se:Appendixtriangular}), generalizing \cite{ggor} to non-inner
gradings.

\smallskip
Before the index of notations, we have
included in "Prime ideals and geometry" some diagrams
summarizing the commutative algebra and geometry studied in this book.

\bigskip
We would like to thank G.~Bellamy, P.~Etingof, I.~Gordon, G.~Malle,
M.~Martino
and U.~Thiel for their help and their suggestions.

\bigskip

\noindent{\bfit Commentary. --- } 
This book contains an earlier version of our work \cite{cm 1}. 
The structure of the text has changed and the presentation of classical
results on Cherednik algebras is now mostly self-contained.
 There are a number of new results
(see for instance Chapters~\ref{se:representations},~\ref{chapter:bebe-verma} 
and~\ref{chapter:gauche-cellulaire}) based on appropriate highest weight category
considerations (Appendix~\ref{ap:hw} is new), and a
new topological approach 
(see Sections~\ref{ch:Gaudin},~\ref{se:cellstopology} and~\ref{se:chemins} 
on Gaudin algebras in particular).



%
%


\medskip

\part{Reflection groups and Cherednik algebras}
\label{part:cherednik}

\chapter{Notations}\label{chapter:notation}

\section{Integers}
We put $\NM=\BZ_{\ge 0}$\indexnot{N}{\NM}.

\section{Modules}
Let $A$ be a ring. Given $L$ a subset of $A$, we denote by $<L>$
\indexnot{L}{<L>}
the two-sided ideal of $A$ generated by $L$.
Given $M$ an $A$-module, we denote by $\Rad(M)$
\indexnot{Rad}{\Rad(M)} the intersection of the maximal
proper $A$-submodules of $M$.
We denote by $A\mMOD$ \indexnot{Mod}{A\mMOD}
 the category of $A$-modules and by $A\mmod$
\indexnot{mod}{A\mmod} the category of finitely generated $A$-modules. We put
$G_0(A)=K_0(A\mmod)$\indexnot{G}{G_0(A)}, where $K_0(\CC)$ denotes the
Grothendieck group of an exact category $\CC$.
We denote
by $A\mproj$ \indexnot{proj}{A\mproj} the category of finitely generated
projective $A$-modules.
Given $\AC$ an abelian category, we denote by $\Proj(\AC)$ its full
subcategory of projective objects.

Given $M\in A\mmod$, we denote by
$[M]_A$ (or simply $[M]$)
\indexnot{M}{[M]_A} its class in $G_0(A)$.

\smallskip
We denote by $\Irr(A)$\indexnot{IrrA}{\Irr(A)}
the set of isomorphism classes of simple $A$-modules.
Assume $A$ is a finite-dimensional algebra over a field $\kb$. We have an
isomorphism $\BZ\Irr(A)\longiso G_0(A),\ M\mapsto [M]$.
If $A$ is semisimple, we have a bilinear form
$\langle-,-\rangle_A$ 
\index{ZZZ@ ${{\langle}},{{\rangle}}_A$ |textbf {\hskip0.5cm} \textbf} on
$G_0(A)$ given by $\langle [M],[N]\rangle=\dim_\kb\Hom_A(M,N)$.
When $A$ is split semisimple, $\Irr(A)$ provides an orthonormal basis.

\smallskip
Let $W$ be a finite group and $\kb$ a field.
We denote by $\Irr_\kb(W)$ (or simply by $\Irr(W)$)
\indexnot{I}{\Irr(W),~\Irr(\kb W)} the set of irreducible characters
of $W$ over $\kb$. When $|W|\in\kb^\times$, there is a bijection
$\Irr_{\kb}(W)\xrightarrow{\sim}
\Irr(\kb W),\ \chi\mapsto E_\chi$\indexnot{Echi}{E_\chi}.
The group $\Hom(W,\kb^\times)$ of linear characters of
$W$ with values in $\kb$ is denoted by
$W^{\wedge_\kb}$ (or $W^\wedge$) \indexnot{W}{W^\wedge}. We have
an embedding $W^\wedge \subset \Irr(W)$, and equality holds if and only if
$W$ is abelian and $\kb$ contains all $e$-th roots of unity, where $e$ is the exponent of $W$.

\section{Gradings}
\subsection{}
Let $\kb$ be a ring and $X$ a set. We denote by $\kb X=\kb^{(X)}$ the free
$\kb$-module with basis $X$.
We sometimes denote elements of $\kb^X$ as formal sums: $\sum_{x\in X}\alpha_x
x$, where $\alpha_x\in \kb$.

\subsection{}
Let $\Gamma$ be a commutative monoid.
We denote by $\kb\Gamma$ (or $\kb[\Gamma]$)
the monoid algebra of $\Gamma$ over $\kb$. Its basis of elements of
$\Gamma$ is denoted by $\{t^\gamma\}_{\gamma\in\Gamma}$.

A {\em $\Gamma$-graded $\kb$-module} is a $\kb$-module
$L$ with a decomposition $L=\bigoplus_{\gamma\in\Gamma}L_\gamma$ (that is
the same as a comodule over the coalgebra $\kb\Gamma$).
Given $\gamma_0\in\Gamma$, we denote by $L\langle \gamma_0\rangle$ the
$\Gamma$-graded $\kb$-module given by
$(L\langle \gamma_0\rangle)_\gamma=L_{\gamma+\gamma_0}$.
We denote by $\kb\mfree^\Gamma$ the additive
category of $\Gamma$-graded $\kb$-modules $L$ such that
$L_\gamma$ is a free $\kb$-module of finite rank for all $\gamma\in\Gamma$.
Given $L\in \kb\mfree^\Gamma$, we put 
$$\dim_\kb^\Gamma(L)=\sum_{\gamma\in\Gamma} \mathrm{rank}_\kb(L_\gamma) t^\gamma
\in \BZ^\Gamma.\indexnot{dimk}{\dim_\kb^\Gamma}$$
We have defined an isomorphism of abelian
groups $\dim_\kb^\Gamma:K_0(\kb\mfree^\Gamma)\longiso \BZ^\Gamma$.
This construction provides a bijection from the set of isomorphism classes of
objects of $\kb\mfree^\Gamma$ to $\NM^\Gamma$.
Given $P=\sum_{\gamma\in\Gamma}p_\gamma t^\gamma$ with $p_\gamma\in\NM$,
we define the $\Gamma$-graded $\kb$-module $\kb^P$ by $(\kb^P)_\gamma=\kb^{p_\gamma}$.
We have $\dim_\kb^\Gamma(\kb^P)=P$.

We say that a subset $E$ of a $\Gamma$-graded module $L$ is {\em homogeneous}
if every element of $E$ is a sum of elements in $E\cap L_\gamma$ for various elements
$\gamma\in\Gamma$.

\subsection{}
\label{se:weightseq}
A {\em graded $\kb$-module} $L$ is a $\BZ$-graded $\kb$-module.
We put $L_+=\bigoplus_{i>0} L_i$. Assume $L\in \kb\mfree^\BZ$.
If $L_i=0$ for $i\ll 0$ (for example, if $L$ is $\NM$-graded),
then $\dim^\BZ_\kb(L)$
is an element of the ring of Laurent power series $\BZ((\tb))$:
this is the Hilbert series of $L$.
Similarly, if $L_i=0$ for $i\gg 0$, then $\dim^\BZ_\kb(L)\in \BZ((\tb^{-1}))$.

When $L$ has finite rank over
$\kb$, we define the {\em weight sequence}
of $L$ as the unique sequence of integers
$r_1\le\cdots\le r_m$ such that $\dim_k^\BZ(L)=t^{r_1}+\cdots+t^{r_m}$.

A {\em bigraded $\kb$-module} $L$ is a $(\ZM\times\ZM)$-graded $\kb$-module. We put
$\tb=t^{(1,0)}$ and $\ub=t^{(0,1)}$, so that $\dim_{\kb}^{\BZ\times\BZ}(L)=
\sum_{i,j}\dim_\kb(L_{i,j})\tb^i\ub^j$ for $L\in \kb\mfree^{\BZ\times\BZ}$.
When $L$ is $(\NM\times\NM)$-graded, we have
$\dim_\kb^{\NM\times\NM}(L)\in\BZ[[\tb,\ub]]$.

\medskip
 When $A$ is a graded ring
and $M$ is a finitely generated graded $A$-module,
we denote by $[M]_A^\grad$ (or simply $[M]^\grad$)\indexnot{M}{[M]_A^\grad}
 its class in the Grothendieck group of
the category $A\mmodgr$\indexnot{modgr}{A\mmodgr}
 of finitely generated graded $A$-modules.
Note that $K_0(A\mmodgr)$ is a $\ZM[\tb^{\pm 1}]$-module, with
$\tb [M]^\grad=[M\langle -1\rangle]^\grad$.

\subsection{}
Assume $\kb$ is a commutative ring. There is a tensor product of
$\Gamma$-graded $\kb$-modules given by $(L\otimes_{\kb} L')_\gamma=
\bigoplus_{\gamma'+\gamma''=\gamma}L_{\gamma'}\otimes_{\kb} L_{\gamma''}$.
When the fibers of the multiplication map $\Gamma\times\Gamma\to\Gamma$
are finite, the multiplication in $\Gamma$ provides
$\BZ^\Gamma$ with a ring structure, the tensor product preserves
$\kb\mfree^\Gamma$, and $\dim_\kb^\Gamma(L\otimes_{\kb} L')=
\dim_\kb^\Gamma(L)\dim_\kb^\Gamma(L')$.

\smallskip
A {\em $\Gamma$-graded $\kb$-algebra} is a $\kb$-algebra $A$ with a structure
of $\Gamma$-graded $k$-modules
such that $A_\gamma\cdot A_{\gamma'}\subset A_{\gamma
\gamma'}$ for all $\g$, $\g' \in \G$.

\chapter{Reflection groups}\label{chapter:reflexions}

\boitegrise{{\it Throughout this book, we consider a fixed
characteristic $0$ field $\kb$\indexnot{ka}{\kb}, a finite-dimensional $\kb$-vector
 space $V$
\indexnot{V}{V} of dimension 
$n$ \indexnot{na}{n} and a {\it finite} subgroup $W$ \indexnot{W}{W}
of $\GL(V)$. 
We will write $\otimes$ for $\otimes_\kb$. 
We denote by
$$\REF(W)=\{s \in W~|~\dim_\kb \im(s-\Id_V)=1\}\indexnot{R}{\REF(W)}$$
the set of {\it reflections} of $W$. 
{\it We assume that $W$ is generated by $\REF(W)$.}}}{0.75\textwidth}

\bigskip

\section{Determinant, roots, coroots}

\medskip

We denote by $\e$ the determinant representation of $W$
$$\fonction{\e}{W}{\kb^\times}{w}{\det_V(w).}\indexnot{ez}{\e}$$
We have a perfect pairing between $V$ and its dual $V^*$
$$\langle,\rangle : V \times V^* \longto \kb.\index{ZZZ@ ${{\langle}}-,-{{\rangle}}$ |textbf {\hskip0.5cm} \textbf}  $$
Given $s \in \REF(W)$, we choose
$\a_s\in V^*$ and $\a_s^\ve\in V$ \indexnot{az}{\a_s,~\a_s^\ve}
such that
$$\Ker(s-\Id_V)=\Ker \a_s\quad\text{and}\quad \im(s-\Id_V)=\kb \a_s^\ve$$
or equivalently
$$\Ker(s-\Id_{V^*})=\Ker \a_s^\ve\quad\text{and}
\quad \im(s-\Id_{V^*})=\kb \a_s.$$
Note that, since $\kb$ has characteristic $0$ and $W$ is finite, all
elements of $\REF(W)$ are diagonalizable, hence
\equat\label{non nul}
\langle \a_s^\ve,\a_s \rangle \neq 0.
\endequat
Given $x \in V^*$ and $y \in V$ we have
\equat\label{action s V}
s(y)=y - (1 -\e(s)) \frac{\langle y,\a_s\rangle}{\langle \a_s^\ve,\a_s\rangle} \a_s^\ve
\endequat
and
\equat\label{action s V*}
s(x)=x - (1 -\e(s)^{-1}) \frac{\langle \a_s^\ve,x\rangle}{\langle \a_s^\ve,\a_s\rangle} \a_s.
\endequat

\section{Invariants}

\medskip

We denote by $\kb[V]=S(V^*)$ (respectively $\kb[V^*]=S(V)$)
\indexnot{ka}{\kb[V],~\kb[V^*]} the symmetric algebra
of $V^*$ (respectively $V$). We identify it with the algebra of
polynomial functions on $V$ (respectively $V^*$).
The action of $W$ on $V$ induces an action by algebra automorphisms on
$\kb[V]$ and $\kb[V^*]$ and we will consider the graded subalgebras of
invariants $\kb[V]^W$ and $\kb[V^*]^W$.
 \indexnot{ka}{\kb[V]^W,~\kb[V^*]^W} The 
{\it coinvariant algebras} $\kb[V]^\cow$ and $\kb[V^*]^\cow$ 
\indexnot{ka}{\kb[V]^\cow,~\kb[V^*]^\cow} 
are the graded finite-dimensional $\kb$-algebras
$$\kb[V]^\cow=\kb[V]/<\kb[V]_+^W>\quad\text{and}
\quad \kb[V^*]^\cow=\kb[V^*]/<\kb[V^*]_+^W>.$$
Shephard-Todd-Chevalley's Theorem asserts that the property of $W$ to be
generated by reflections is equivalent to structural properties of
$\kb[V]^W$. We provide here a version augmented with
quantitative properties
(see for example~\cite[Theorem~4.1]{broue}).
We state a version with $\kb[V]$, while the
same statements hold with $V$ replaced by $V^*$.

\smallskip
Let us define the
sequence $d_1\le\cdots\le d_n$ of {\it degrees of $W$} as the weight sequence
of $<\kb[V]^W_+>/<\kb[V]^W_+>^2$ (cf. \S\ref{se:weightseq}): it has length $n$ 
thanks to the following classical theorem.

\bigskip

\begin{theo}[Shephard-Todd, Chevalley]\label{chevalley}
\begin{itemize}
\itemth{a} The algebra $k[V]^W$ is a polynomial algebra generated by
homogeneous elements of degrees $d_1,\ldots,d_n$.
We have
$$|W|=d_1\cdots d_n\quad \text{and}\quad |\REF(W)|=\sum_{i=1}^n (d_i-1).$$

\itemth{b} The $(\kb[V]^W[W])$-module $\kb[V]$ is free of rank $1$.

\itemth{c} The $\kb W$-module $\kb[V]^\cow$ is free of rank $1$. In particular
$\dim_\kb \kb[V]^\cow=|W|$.
\end{itemize}
\end{theo}

\bigskip

\begin{rema}\label{rema:deg-v-v}
Note that when $\kb=\CM$, there is a skew-linear isomorphism between the
representations $V$ and $V^*$ of $W$, hence the sequence of 
degrees for the action of
$W$ on $V$ is the same as the one for the action of $W$ on $V^*$. In general,
note that the representation $V$ of $W$ can be defined over a finite extension
of $\QM$, which can be embedded in $\CM$: so, the equality of degrees
for the actions on $V$ and $V^*$ holds for any $\kb$.

This equality can also be deduced from Molien's formula
\cite[Lemma~3.28]{broue}.\finl
\end{rema}

\bigskip

Let $N=|\REF(W)|$\indexnot{N}{N}. Since $\dim_\kb^\BZ(\kb[V]^{\cow})=
\prod_{i=1}^n\frac{1-\tb^{d_i}}{1-\tb}$, we deduce that
$\dim_\kb \kb[V]^{\cow}_N=1$. A generator is given by the image
of $\prod_{s\in\REF(W)}\alpha_s$: this provides an isomorphism
$h:\kb[V]^{\cow}_N\longiso\kb$.

The composition
$$\kb[V]_N\otimes \kb[V]^W\stackrel{\mathrm{mult}}{\longrightarrow}
 \kb[V]\xrightarrow{\mathrm{can}} 
\kb[V]/(\kb[V]^W \kb[V]_{<N})$$
 factors through an isomorphism
$g:\kb[V]^{\cow}_N\otimes \kb[V]^W\longiso \kb[V]/(\kb[V]^W \kb[V]_{<N})$.
We denote by $p_N$ the composition
$$p_N: \kb[V]\xrightarrow{\mathrm{can}} \kb[V]/(\kb[V]^W \kb[V]_{<N})
\xrightarrow{g^{-1}} \kb[V]^{\cow}_N\otimes \kb[V]^W\xrightarrow[\sim]{h\otimes\id} \kb[V]^W.$$

We refer to \S \ref{se:symmalg} for basic facts on symmetric algebras.

\begin{prop}
\label{pr:symmkV}
$p_N$ is a symmetrizing form for the $\kb[V]^W$-algebra $\kb[V]$.
\end{prop}

\begin{proof}
We need to show that the morphism of graded $\kb[V]^W$-modules
$$\hat{p}_N:\kb[V]\to \Hom_{\kb[V]^W}(\kb[V],\kb[V]^W),\
a\mapsto (b\mapsto p_N(ab))$$
is an isomorphism. By the graded Nakayama lemma, it is enough to do so
after applying $-\otimes_{\kb[V]^W}\kb$. We have
$\hat{p}_N\otimes_{\kb[V]^W}\kb=\hat{\bar{p}}_N$, where 
$\bar{p}_N:\kb[V]^{\cow}\to \kb[V]^{\cow}_N\xrightarrow[\sim]{h}\kb$
 is the projection onto the
homogeneous component of degree $N$. This is a symmetrizing form
for $\kb[V]^{\cow}$ \cite[Theorem 4.25]{broue}, hence $\hat{\bar{p}}_N$ is
an isomorphism.
\end{proof}

Note that the same statements hold for $V$ replaced by $V^*$.

\bigskip

\section{Hyperplanes}\label{section:hyperplans}

\medskip
\boitegrise{{\bf Notation.} 
{\it We fix an embedding of the group of roots of unity of $\kb$ in $\QM/\ZM$.
When the class of $\frac{1}{e}$ is in the image of this embedding, we denote by
$\zeta_e$ the corresponding element of $\kb$.}}{0.75\textwidth}

\medskip

We denote by $\AC$ \indexnot{A}{\AC} the set of reflecting hyperplanes of
$W$:
$$\AC=\{\Ker(s-\Id_V)~|~s \in \REF(W)\}.$$
There is a surjective $W$-equivariant map $\REF(W)\to\AC,\ s\mapsto\Ker(s-\Id_V)$.
Given $X$ a subset of $V$, we denote by $W_X$ \indexnot{W}{W_X} the 
pointwise stabilizer of $X$:
$$W_X=\{w \in W~|~\forall~x \in X,~w(x)=x\}.$$

\bigskip

Given $H \in \AC$, denote by $\a_H \in V^*$\indexnot{az}{\a_H, \a_H^\ve} a linear form such that
$H=\Ker(\a_H)$ and let 
$\a_H^\ve \in V$ such that $V=H \oplus \kb \a_H^\ve$ and
$\kb\a_H^\ve$ is stable under $W_H$. 
We denote by $e_H$ \indexnot{ea}{e_H} the order of the
cyclic subgroup $W_H$ of $W$. 
We denote by $s_H$ \indexnot{sa}{s_H} the generator of
$W_H$ with determinant $\z_{e_H}$. This is a reflection with hyperplane $H$. We have
$$\REF(W)=\{s_H^j~|~H \in \AC\text{ and }1 \le j \le e_H-1\}.$$
The following lemma is clear.

\bigskip

\begin{lem}\label{conjugaison ref}
$s_H^j$ and $s_{H'}^{j'}$ are conjugate in $W$ if and only
if $H$ and $H'$ are in the same $W$-orbit and $j=j'$.
\end{lem}

\bigskip

\def\orbite{{\mho}}
\def\orbiteb{{\boldsymbol{\orbite}}}

Given $\orbite$ a $W$-orbit of hyperplanes of $\AC$, we denote by $e_\orbite$
 \indexnot{ea}{e_\orbite} the common value of the
$e_H$ for $H \in \orbite$. Lemma~\ref{conjugaison ref} 
provides a bijection from $\refw$ to the set
$\orbiteb$ \indexnot{oz}{\orbiteb} of pairs $(\orbite,j)$ where $\orbite \in \AC/W$ and
$1 \le j \le e_\orbite-1$.

We denote by $\orbiteb^\circ$ \indexnot{oz}{\orbiteb^\circ}  the set of pairs $(\orbite,j)$ with
$\orbite \in \AC/W$ and $0 \le j \le e_\orbite -1$.

\bigskip
Let $V^{\reg}=\{v \in V~|~\Stab_W(v)=1\}$.
Define the discriminant $\delta=\prod_{H\in\AC}\alpha_H^{e_H}\in \kb[V]^W$.
\indexnot{delta}{\delta}
The following result shows that points outside reflecting hyperplanes
have trivial stabilizers \cite[Theorem~4.7]{broue}.

\begin{theo}[Steinberg]
\label{th:Steinberg}
Given $X \subset V$, the group $W_X$ is generated by its
reflections. As a consequence,
$V^{\reg}=V \setminus \bigcup_{H \in \AC} H$ and $\kb[V^{\reg}]=
k[V][\delta^{-1}]$.
\end{theo}

\medskip
We say that $W$ is indecomposable if it acts irreducibly on $V/V^W$.

There is a unique decomposition $V=V^W\oplus V_1\oplus\cdots\oplus V_r$
as $\kb W$-modules where the $V_i$'s are irreducible. Define
$W_i$ to be the pointwise stabilizer in $W$ of $\bigoplus_{j\neq i}V_j$.
We have $W=W_1\times\cdots\times W_r$ and $W_i$ is an indecomposable
reflection group on $V$ (and irreducible on $V_i$).

\bigskip

\section{Irreducible characters}

\medskip

The rationality property of the reflection representation of
$W$ is classical.

\begin{prop}\label{pr:deploiementV}
Let $\kb'$ be a subfield of $\kb$ containing the traces of the elements of
$W$ acting on $V$. Then there exists a $\kb'W$-submodule $V'$ of $V$
such that $V=\kb\otimes_{\kb'}V'$.
\end{prop}

\begin{proof}
Assume first $V$ is irreducible.
Let $V''$ be a simple $\kb'W$-module such that $\kb\otimes_{\kb'}V''
\simeq V^{\oplus m}$ for some integer $m\ge 1$. Let $s\in\REF(W)$. Since $s$
has only one non-trivial eigenvalue on $V$, it also has only one
non-trivial eigenvalue on $V''$.
Let $L$ be
the eigenspace of $s$ acting on $V''$ for the non-trivial eigenvalue.
This is an $m$-dimensional $\kb'$-subspace of $V''$, stable under the action
of the division algebra $\End_{\kb' W}(V'')$. Since that division algebra
has dimension $m^2$ over $\kb'$ and has a module $L$ that has dimension
$m$ over $\kb'$, we deduce that $m=1$.
The proposition follows by
taking for $V'$ the image of $V''$ by an isomorphism
$\kb\otimes_{\kb'}V''\xrightarrow{\sim} V$.

Assume now $V$ is arbitrary.
Let $V=V^W\oplus \bigoplus_{i=1}^l V_i$ be a decomposition of the $\kb W$-module
$V$, where $V_i$ is irreducible and non-trivial for $1\le i\le l$.
Let $W_j$ be the subgroup of $W$ of elements acting trivially
on $\bigoplus_{i{\not=}j}V_i$. The group $W_j$ is a reflection group on
$V_j$. The discussion above shows there is
a $\kb'W_j$-submodule $V'_j$ of $V_j$ such that $V_j=\kb\otimes_{\kb'}V'_j$.
Let $V''$ be a $\kb'$-submodule of $V^W$ such
that $V^W=\kb\otimes_{\kb'}V''$.
Let $V'=V''\oplus\bigoplus_{j=1}^l V'_j$. 
We have $W=\prod_{j=1}^l W_j$ and
$V=\kb\otimes_{\kb'}V'$: this proves the proposition.
\end{proof}

The following rationality property of all
representations of complex reflection
groups is proven using the classification of those groups
\cite{benard,bessis}.

\bigskip

\begin{theo}[Benard, Bessis]\label{deploiement}
Let $\kb'$ be a subfield of $\kb$ containing the traces of the elements of
$W$ acting on $V$. Then the algebra
$\kb' W$ is split semisimple. In particular, $\kb W$ is split semisimple.
\end{theo}

\bigskip

\bigskip

\section{Hilbert series}

\subsection{Invariants}
The algebra $\kb[V \times V^*]=\kb[V] \otimes \kb[V^*]$ admits a standard bigrading, 
by giving to the elements of $V^* \subset \kb[V]$ 
the bidegree $(0,1)$ and to those of $V \subset \kb[V^*]$ the bidegree 
$(1,0)$. We have
\equat\label{hilbert kv}
\dim_\kb^{\BZ\times\BZ}(\kb[V \times V^*]) = \frac{1}{(1-\tb)^n(1-\ub)^n}.
\endequat

Using the notation of Theorem~\ref{chevalley}(a) and Remark~\ref{rema:deg-v-v}, we find that 
\equat\label{hilbert biinv}
\dim_\kb^{\BZ\times\BZ}(\kb[V \times V^*]^{W \times W})
=\prod_{i=1}^n \frac{1}{(1-\tb^{d_i})(1-\ub^{d_i})}.
\endequat
On the other hand, the bigraded Hilbert series of the diagonal invariant algebra 
$\kb[V \times V^*]^{\Delta W}$ is given by a formula {\it \`a la Molien}
\equat\label{hilbert molien}
\dim_\kb^{\BZ\times\BZ}(\kb[V \times V^*]^{\Delta W})=\frac{1}{|W|} 
\sum_{w \in W} \frac{1}{\det(1-w\tb)~\det(1-w^{-1}\ub)},
\endequat
whose proof is obtained word by word from the proof of the usual 
Molien formula.

\bigskip

\subsection{Fake degrees} 
We identify $K_0(\kb W\mmodgr)$ with $G_0(\kb W)[\tb,\tb^{-1}]$:
given
$M = \bigoplus_{i \in \BZ} M_i$ a finite dimensional $\BZ$-graded
$\kb W$-module, 
we make the identification
\indexnot{M}{\isomorphisme{M}_{\kb W}^\grad}
$$\isomorphisme{M}_{\kb W}^\grad = \sum_{i \in \BZ}~\isomorphisme{M_i}_{\kb W}~\tb^i.$$
The class $\isomorphisme{M}_{\kb W}$ is the evaluation at $\tb=1$ of 
$\isomorphisme{M}_{\kb W}^\grad$.
When $M$ is a bigraded $\kb W$-module, we identify similarly 
$\isomorphisme{M}_{\kb W}^{\BZ\times\BZ}$ with an element of 
$\groth(\kb W)[\tb,\ub,\tb^{-1},\ub^{-1}]$.

Let $(f_\chi(\tb))_{\chi \in \Irr(W)}$ \indexnot{fa}{f_\chi(\tb)}
denote the unique family of elements of $\NM[\tb]$ such that 
\equat\label{fchi def}
\isomorphisme{\kb[V^*]^\cow}_{\kb W}^{\BZ\times\BZ}=\sum_{\chi \in \Irr(W)} f_\chi(\tb) ~\chi.
\endequat

\bigskip
\begin{defi}\label{defi:degre-fantome}
The polynomial $f_\chi(\tb)$ is called the {\bfit fake degree} of $\chi$. 
Its $\tb$-valuation is denoted by $\bb_\chi$ \indexnot{ba}{\bb_\chi} 
and is called the {\bfit $\bb$-invariant} of $\chi$.
\end{defi}

\bigskip

The fake degree of $\chi$  satisfies 
\equat\label{chi 1}
f_\chi(1) = \chi(1).
\endequat
Note that 
\equat\label{fchi}
\isomorphisme{\kb[V]^\cow}_{\kb W}^{\BZ\times\BZ}=\sum_{\chi \in \Irr(W)} f_\chi(\ub) ~\chi^*,
\endequat
(here, $\chi^*$ denotes the dual character of $\chi$, that is, 
$\chi^*(w)=\chi(w^{-1})$). Note also that, if $\unb_W$ denotes the trivial character of $W$, then 
$$\isomorphisme{\kb[V^*]^\cow}_{\kb W}^{\BZ\times\BZ} \equiv \unb_W \mod \tb G_0(\kb W)[\tb]$$
$$\isomorphisme{\kb[V]^\cow}_{\kb W}^{\BZ\times\BZ} \equiv \unb_W \mod \ub  
G_0(\kb W)[\ub].\leqno{\text{and}}$$
We deduce the following lemma.

\bigskip

\begin{lem}\label{non zero}
The elements $\isomorphisme{\kb[V]^\cow}_{\kb W}^{\BZ\times\BZ}$ and 
$\isomorphisme{\kb[V^*]^\cow}_{\kb W}^{\BZ\times\BZ}$ 
are not zero divisors in $\groth(\kb W)[\tb,\ub,\tb^{-1},\ub^{-1}]$.
\end{lem}

\bigskip

\begin{rema}
Note that 
$$\isomorphisme{\kb[V]^\cow}_{\kb W}=\isomorphisme{\kb[V^*]^\cow}_{\kb W}=
\isomorphisme{\kb W}_{\kb W} = 
\sum_{\chi \in \Irr(W)} \chi(1) \chi$$ 
is a zero divisor in $\groth(\kb W)$ (as soon as $W \neq 1$).\finl
\end{rema}

\bigskip

We can now give another formula for the Hilbert series $\dim_\kb^{\BZ\times\BZ}(\kb[V \times V^*]^{\Delta W})$.

\bigskip

\begin{prop}\label{dim bigrad Q0 fantome}
$\DS{\dim_\kb^{\BZ\times\BZ}(\kb[V \times V^*]^W)=\frac{1}{\prod_{i=1}^n (1-\tb^{d_i})(1-\ub^{d_i})} 
\sum_{\chi \in \Irr(W)} f_\chi(\tb) ~f_\chi(\ub).}$
\end{prop}

\begin{proof}
Let $\HC$ be a $W$-stable graded complement to $<\kb[V]^W_+>$ in $\kb[V]$. 
Since $\kb[V]$ is a free $\kb[V]^W$-module, we have isomorphisms 
of graded $\kb[W]$-modules 
$$\kb[V] \simeq \kb[V]^W \otimes \HC\qquad\text{and}\qquad 
\kb[V]^\cow \simeq \HC.$$
Similarly, if $\HC'$ is a $W$-stable graded complement of 
$<\kb[V^*]^W_+>$ in $\kb[V^*]$, then we have isomorphisms of graded
$\kb[W]$-modules 
$$\kb[V^*] \simeq \kb[V^*]^W \otimes \HC'\qquad\text{and}\qquad 
\kb[V^*]^\cow \simeq \HC'.$$
In other words, we have isomorphisms of graded $\kb[W]$-modules
$$\kb[V] \simeq \kb[V]^W \otimes \kb[V]^\cow\qquad\text{and}\qquad 
\kb[V^*] \simeq \kb[V^*]^W \otimes \kb[V^*]^\cow.$$
We deduce an isomorphism of bigraded $\kb$-vector spaces 
$$(\kb[V] \otimes \kb[V^*])^{\Delta W} \simeq 
(\kb[V]^W \otimes \kb[V^*]^W) \otimes 
(\kb[V]^\cow \otimes \kb[V^*]^\cow)^{\Delta W}.$$
By (\ref{fchi def}) and (\ref{fchi}), we have
$$\dim_\kb^{\BZ\times\BZ}(\kb[V]^\cow \otimes \kb[V^*]^\cow)^{\Delta W} 
= \sum_{\chi, \psi \in \Irr(W)} f_\chi(\tb) f_\psi(\ub) \langle \chi\psi^*,\unb_W\rangle_W.$$
So the formula follows from the fact that 
$\langle \chi\psi^*,\unb_W\rangle  = \langle \chi,\psi \rangle_W$. 
\end{proof}

\bigskip

To conclude this section, we gather in a same formula Molien's Formula (\ref{hilbert molien})
and Proposition~\ref{dim bigrad Q0 fantome}:
\eqna
\DS{
\dim_\kb^{\BZ\times\BZ}(\kb[V \times V^*]^{\Delta W})
}
&=&
\DS{
\frac{1}{|W|} \sum_{w \in W} \frac{1}{\det(1-w\tb)~\det(1-w^{-1}\ub)}
}\\
&=&
\DS{
\frac{1}{\prod_{i=1}^n (1-\tb^{d_i})(1-\ub^{d_i})} 
\sum_{\chi \in \Irr(W)} f_\chi(\tb) ~f_\chi(\ub).
}
\endeqna

\medskip

\section{Coxeter groups} 
\label{se:Coxetergroups}
Let us recall the following classical equivalences:

\bigskip

\begin{prop}\label{prop:coxeter}
The following assertions are equivalent:
\begin{itemize}
\itemth{1} There exists a subset $S$ of $\REF(W)$ such that $(W,S)$ is a Coxeter system.

\itemth{2} $V \simeq V^*$ as $\kb W$-modules.

\itemth{3} There exists a $W$-invariant non-degenerate symmetric bilinear form 
$V \times V \to \kb$.

\itemth{4} There exists a subfield $\kb_\RM$ of $\kb$ and a $W$-stable $\kb_\RM$-vector subspace 
$V_{\kb_\RM}$ of $V$ such that $V = \kb \otimes_{\kb_\RM} V_{\kb_\RM}$ and 
$\kb_\RM$ embeds as a subfield of $\RM$.
\end{itemize}
\end{prop}

\bigskip

Whenever one (or all the) assertion(s) of Proposition~\ref{prop:coxeter} 
is (are) satisfied, we say that {\it $W$ is a Coxeter group}. 
In this case, the text will be followed by a gray line on the left, 
as below.

\bigskip

\cbstart

\boitegrise{{\bf Assumption, choice.} {\it From now on, and until the end of
\S\ref{se:Coxetergroups},
we assume that $W$ is a Coxeter group. We fix a subfield $\kb_\RM$ of $\kb$
that embeds as a subfield of $\RM$ and a
$W$-stable $\kb_\RM$-vector
subspace $V_{\kb_\RM}$  \indexnot{VR}{V_{\kb_\RM}}  of $V$ such that 
$V=\kb \otimes_{\kb_\RM} V_{\kb_\RM}$. We also fix a connected component $C_\RM$
\indexnot{CR}{C_\RM} of
$\{v\in \RM\otimes_{\kb_\RM} V_{\kb_\RM} | \Stab_W(v)=1\}$.
We denote by $S$\indexnot{S}{S} the set of $s \in \REF(W)$ such that 
$\overline{C_\RM}\cap \ker_{\RM\otimes_{\kb_\RM} V_{\kb_\RM}}(s-1)$ has 
real codimension $1$ in $\overline{C_\RM}$.
So, $(W,S)$ is a Coxeter system. This notation will be used all along this 
book, 
provided that $W$ is a Coxeter group.}}{0.75\textwidth}

\bigskip

The following is a particular case of Theorem~\ref{deploiement}.

\bigskip

\begin{theo}\label{lem:coxeter}
The $\kb_\RM$-algebra $\kb_\RM W$ is split. In particular, 
the characters of $W$ are real valued, that is, $\chi=\chi^*$ 
for all characters $\chi$ of $W$.
\end{theo}

\bigskip

Recall also the following.

\bigskip

\begin{lem}\label{lem:coxeter-2}
If $s \in \REF(W)$, then $s$ has order $2$ and $\e(s)=-1$.
\end{lem}

\bigskip

\begin{coro}\label{coro:coxeter}
The map $\REF(W) \to \AC$, $s \mapsto \Ker(s-\Id_V)$ is bijective and $W$-equivariant. 
In particular, $|\AC|=|\REF(W)|=\sum_{i=1}^n (d_i-1)$ and $|\AC/W|=|\REF(W)/W|$.
\end{coro}

\bigskip

Let $\ell : W \to \NM$\indexnot{l}{\ell}
 denote the length function with respect to $S$:
given $w\in W$, the integer $\ell(w)$ is minimal such that $w$ is a product
of $\ell(w)$ elements of $S$.
When $w=s_1s_2\cdots s_l$ with $s_i \in S$ and $l=\ell(w)$, we say that 
$w=s_1s_2\cdots s_l$ is a {\bfit reduced decomposition} of $w$. 
We denote by $w_0$ \indexnot{wa}{w_0}  the longest element of $W$: we have $\ell(w_0)=|\REF(W)|=|\AC|$. 

\bigskip

\begin{rema}\label{rem:w0-central}
If $-\Id_V \in W$, then $w_0=-\Id_V$. Conversely, if 
$w_0$ is central and $V^W=0$, then $w_0=-\Id_V$.\finl
\end{rema}

\cbend
%
%
%
%
%
%
%

\chapter{Generic Cherednik algebras}\label{chapter:cherednik-1}

Let $\CCB$ \indexnot{C}{\CCB}  be the $\kb$-vector space of maps
$c : \REF(W) \to \kb$, $s \mapsto c_s$ 
\indexnot{ca}{c_s}  
that are  constant on conjugacy classes: this is the {\it space of
parameters}, which we identify with the space of maps $\refw \to \kb$. 

Given $s \in \REF(W)$ (or $s \in \refw$), we denote by $C_s$ \indexnot{C}{C_s}  the linear form on
$\CCB$ given by evaluation at $s$. The algebra $\kb[\CCB]$ of polynomial
functions on $\CCB$ is the algebra of polynomials on the set of 
indeterminates $(C_s)_{s \in \refw}$:
$$\kb[\CCB]=\kb[(C_s)_{s \in \refw}].$$
We denote by $\CCBt$ \indexnot{C}{\CCBt}  the $\kb$-vector space
$\kb \times \CCB$ and
we introduce $T : \CCBt \to \kb$, \indexnot{T}{T}  $(t,c) \mapsto t$. We have
$T \in \CCBt^*$ and
$$\kb[\CCBt]=\kb[T,(C_s)_{s \in \refw}].$$

\section{Structure}

\medskip

\subsection{Symplectic action}
We consider here the action of $W$ on $V\oplus V^*$.
Lemma \ref{le:idempotentinvariants} and Proposition \ref{pr:codim2doublecentralizer}
give the following result.

\begin{prop}
\label{pr:centerdouble}
We have $\Zrm(\kb[V \times V^*]\rtimes W)=\kb[V \times V^*]^W=\kb[(V \times V^*)/W]$ and there
is an isomorphism
$$\Zrm(\kb[V \times V^*]\rtimes W)\longiso e(\kb[V \times V^*]\rtimes W)e,\ z\mapsto ze.$$
The action by left multiplication gives an isomorphism
$$\kb[V \times V^*]\rtimes W\longiso\End_{\kb[(V \times V^*)/W]^{\mathrm{opp}}}
\bigl((\kb[V \times V^*]\rtimes W)e\bigr)^\opp.$$
\end{prop}

\subsection{Definition}

\medskip

The {\it generic rational Cherednik algebra} (or simply
the {\it generic Cherednik algebra}) is the $\kb[\CCBt]$-algebra $\Hbt$
defined as the quotient 
of $\kb[\CCBt] \otimes \bigl(\Trm_\kb(V \oplus V^*) \rtimes W\bigr)$ 
by the following relations (here, $\Trm_\kb(V \oplus V^*)$ 
is the tensor algebra of $V \oplus V^*$):
\equat\label{relations-1}\begin{cases}
[x,x']=[y,y']=0, \\
\\
[y,x] = T \langle y,x\rangle + \DS{\sum_{s \in \REF(W)} (\e(s)-1)\hskip1mm C_s 
\hskip1mm\frac{\langle y,\a_s \rangle \cdot \langle \a_s^\ve,x\rangle}{\langle \a_s^\ve,\a_s\rangle}
\hskip1mm s,} 
\end{cases}\endequat
for $x$, $x' \in V^*$ and $y$, $y' \in V$.

\bigskip

\begin{rema}
Thanks to (\ref{action s V}), the second relation is equivalent to
\equat\label{eq:relations-1-sans-alpha}
[y,x] = T \langle y,x\rangle + \sum_{s \in \REF(W)} C_s 
\langle s(y)-y,x\rangle 
\hskip1mm s
\endequat
and to
$$[y,x] = T \langle y,x\rangle + \sum_{s \in \REF(W)} C_s 
\langle y,s^{-1}(x)-x\rangle
\hskip1mm s.
$$
This avoids the use of $\a_s$ and $\a_s^\ve$.\finl
\end{rema}

\bigskip

\subsection{PBW Decomposition}\label{subsection:PBW-1} 
Given the relations~(\ref{relations-1}), 
the following assertions are clear:
\begin{itemize}
\item[$\bullet$] There is a unique morphism of $\kb$-algebras 
$\kb[V] \to \Hbt$ sending $y \in V^*$ to the class of
$y \in \Trm_\kb(V \oplus V^*) \rtimes W$ in $\Hbt$.

\item[$\bullet$]There is a unique morphism of $\kb$-algebras 
$\kb[V^*] \to \Hbt$ sending $x \in V$ to the class of
$x \in \Trm_\kb(V \oplus V^*) \rtimes W$ in $\Hbt$.

\item[$\bullet$] There is a unique morphism of $\kb$-algebras  
$\kb W \to \Hbt$ sending $w \in W$ to the class of
$w \in \Trm_\kb(V \oplus V^*) \rtimes W$ in $\Hbt$.

\item[$\bullet$] The $\kb[\CCBt]$-linear map $\kb[\CCBt] \otimes \kb[V] \otimes \kb W 
\otimes \kb[V^*] \longto \Hbt$ induced by the three morphisms defined above
and the multiplication map is surjective.
\end{itemize}
The last statement is strengthened by the following fundamental result
\cite[Theorem 1.3]{EG}, for which we will provide a proof in
Theorem \ref{th:PBWDunkl}.

\bigskip

\begin{theo}[Etingof-Ginzburg]\label{PBW-1}
The multiplication map  $\kb[\CCBt] \otimes \kb[V] \otimes \kb W 
\otimes \kb[V^*] \longto \Hbt$ is an isomorphism of $\kb[\CCBt]$-modules.
\end{theo}

\bigskip

\subsection{Specialization}\label{subsection:specialization-1}
Given $(t,c) \in \CCBt$, we denote by $\CGt_{t,c}$ \indexnot{C}{\CGt_{t,c}}
the maximal ideal of $\kb[\CCBt]$ given by
$\CGt_{t,c}=\{f \in \kb[\CCBt]~|~f(t,c)=0\}$: this is the ideal
generated by $T-t$ and
$(C_s-c_s)_{s \in \refw}$. We put
$$\Hbt_{t,c} = (\kb[\CCBt]/\CGt_{t,c}) \otimes_{\kb[\CCBt]} \Hbt = \Hbt/\CGt_{t,c} \Hbt.\indexnot{H}{\Hbt_{t,c}}$$
The  $\kb$-algebra
$\Hbt_{t,c}$ is the quotient of
$\Trm_\kb(V \oplus V^*) \rtimes W$ by the ideal generated by the
following relations:
\equat\label{relations specialisees}\begin{cases}
[x,x']=[y,y']=0, \\
\\
[y,x] = t \langle y,x \rangle + \DS{\sum_{s \in \REF(W)} (\e(s)-1)\hskip1mm c_s 
\hskip1mm\frac{\langle y,\a_s \rangle \cdot \langle \a_s^\ve,x\rangle}{\langle \a_s^\ve,\a_s\rangle}
\hskip1mm s,} 
\end{cases}\endequat
for $x$, $x' \in V^*$ and $y$, $y' \in V$. 

\bigskip

\begin{exemple}\label{exemple zero-1} 
We have $\Hbt_{0,0}=\kb[V \times V^*] \rtimes W$ and $\Hbt_{T,0}=\DCB_T(V)\rtimes W$ 
(see \S\ref{se:Weylalgebra}).\finl
\end{exemple}

More generally, given $\CGt$ a prime ideal of $\kb[\CCBt]$, we put
$\Hbt(\CGt)=\Hbt/\CGt\Hbt$.

\subsection{Filtration}\label{section:filtration-1}

\medskip

We endow the $\kb[\CCBt]$-algebra $\Hbt$ with the filtration defined as follows:
\begin{itemize}
\item[$\bullet$] $\Hbt^{\le -1}=0$
\item[$\bullet$] $\Hbt^{\le 0}$ is the $\kb[\CCBt]$-subalgebra generated by
$V^*$ and $W$
\item[$\bullet$]  $\Hbt^{\le 1}=\Hbt^{\le 0}V+\Hbt^{\le 0}$.
\item[$\bullet$] $\Hbt^{\le i}= (\Hbt^{\le 1})^i$ for $i\ge 2$.
\end{itemize}
Specializing at $(t,c) \in \CCBt$, we have an induced filtration
of $\Hbt_{t,c}$.

\medskip
The canonical maps $\kb[\CCBt]\otimes\kb[V]\rtimes W\to (\grad\,\Hbt)^0$ and
$V\to (\grad\,\Hbt)^1$ induce a surjective morphism of algebras
$\rho:\kb[\CCBt]\otimes\kb[V \times V^*]\rtimes W\twoheadrightarrow \grad\,\Hbt$.

\bigskip

\subsection{Localization at $V^\reg$} 
Recall that
$$V^\reg=V \setminus \bigcup_{H \in \AC} H=\{v \in V~|~\Stab_G(v)=1\}
\quad\text{ and }\quad\kb[V^{\reg}]=\kb[V][\delta^{-1}].$$
We put $\Hbt^\reg=\Hbt[\delta^{-1}]$, the non-commutative localization
of $\Hbt$ obtained by adding a two-sided inverse to the image of $\delta$.
Note that the filtration of $\Hbt$ induces a filtration of $\Hbt^\reg$,
with $(\Hbt^\reg)^{\le i}=\Hbt^{\le i}[\delta^{-1}]$.

\medskip
Note that multiplication induces an isomorphism of $\kb$-vector spaces
$\kb[V^{\reg}]\otimes\kb[V^*]\longiso \DCB(V^{\reg})=\DCB(V)[\delta^{-1}]$
(cf. Appendix \S \ref{se:Weylalgebra}).

\begin{lem}
\label{le:Weylfaithful}
We have the following properties.
\begin{itemize}
\itemth{a} There is a Morita equivalence between $\kb[V^\reg\times V^*]\rtimes W$ and
$\kb[V^{\reg}\times V^*]^W$ given by the bimodule $\kb[V^\reg\times V^*]$.

\itemth{b} There is a Morita equivalence between $\DCB(V^\reg)\rtimes W$ and
$\DCB(V^{\reg})^W=\DCB(V^\reg/W)$ given by the bimodule $\DCB(V^\reg)$.

\itemth{c} The action of $\DCB(V^{\reg})\rtimes W$ on $\kb[V^{\reg}]$ is faithful.
\end{itemize}
\end{lem}

\begin{proof}
(a) follows from Corollary~\ref{le:Moritafree}. 

(b) becomes (a) after taking associated graded, hence
(b) follows from Lemmas ~\ref{le:Moritafiltered} and
\ref{le:idempotentinvariants}.

(c) It follows from (b)
that every two-sided ideal of $\DCB(V^{\reg})\rtimes W$ is generated by
its intersection with $\DCB(V^{\reg})^W$. Since $\DCB(V^{\reg})$ acts faithfully on
$\kb[V^{\reg}]$ (cf. \S \ref{se:Weylalgebra}), we deduce that the kernel of the action of 
$\DCB(V^{\reg})\rtimes W$ vanishes.
\end{proof}

\bigskip

\subsection{Polynomial representation and Dunkl operators}

Given $y\in V$, we define $D_y$, a $\kb[\CCBt]$-linear endomorphism of $\kb[\CCBt]\otimes \kb[V^{\reg}]$ by
$$D_y=T\partial_y-\DS{\sum_{s \in \REF(W)} \e(s) C_s 
\langle y,\a_s \rangle \alpha_s^{-1}s},$$
where $\partial_y$ is the partial derivative along $y$. 
Note that $D_y\in\kb[\CCB]\otimes\DCB_T(V^{\reg})\rtimes W\subset
\kb[\CCBt]\otimes\DCB(V^{\reg})\rtimes W$.

\begin{rema}
	The Dunkl operators are traditionally defined (cf for example
	\cite[p.280]{EG}) as
$$D'_y=T\partial_y-\DS{\sum_{s \in \REF(W)} \e(s) C_s 
\langle y,\a_s \rangle \alpha_s^{-1}(s-1)}.$$
	The operator $D_y$ is constructed in \cite[p.281 and 284]{EG}
	and $D'_y=\delta_C^{-1}D_y\delta_C$, where
	$\delta_C=\prod_{s\in\REF(W)}\alpha_s^{C_s}$. The operator $D_y$
	behaves better than $D'_y$ at $t=0$. On the other hand, the operator
	$D'_y$ preserves
	$\kb[\CCBt]\otimes \kb[V]$ while $D_y$ does not.
	\finl
\end{rema}

\bigskip

\begin{prop}\label{pr:polynomialrep}
There is a unique structure of $\Hbt$-module on $\kb[\CCBt]\otimes \kb[V^{\reg}]$ where $\kb[\CCBt]\otimes \kb[V^{\reg}]$ acts by
multiplication, $W$ acts through its natural action on $V$ and $y\in V$ acts by $D_y$.
\end{prop}

\begin{proof}
The following argument is due to Etingof. 
Let $y\in V$ and $x\in V^*$.
We have
$$[\alpha_s^{-1}s,x]=(\e(s)^{-1}-1)\frac{\langle\alpha_s^\vee,x\rangle}{
\langle\alpha_s^\vee,\alpha_s\rangle}s,$$
hence
$$[D_y,x]=T\langle y,x\rangle+\sum_s (\e(s)-1)C_s\frac{\langle y,\alpha_s
\rangle\cdot\langle\alpha_s^\vee,x\rangle}{\langle\alpha_s^\vee,\alpha_s\rangle}s.$$
Given $w\in W$, we have $wD_yw^{-1}=D_{w(y)}$.

Consider $y'\in V$. We have
$$[[D_y,D_{y'}],x]=[[D_y,x],D_{y'}]-[[D_{y'},x],D_y]$$
and
\begin{align*}
[[D_y,x],D_{y'}]&=
\sum_s (\e(s)-1)C_s\frac{\langle y,\alpha_s
\rangle\cdot\langle\alpha_s^\vee,x\rangle}{
\langle\alpha_s^\vee,\alpha_s\rangle}[s,D_{y'}]\\
&=\sum_s (\e(s)-1)^2C_s\frac{\langle y,\alpha_s
\rangle\cdot\langle y',\alpha_s\rangle\cdot
\langle\alpha_s^\vee,x\rangle}{
\langle\alpha_s^\vee,\alpha_s\rangle^2}D_{\alpha_s^\vee}s\\
&=[[D_{y'},x],D_y].
\end{align*}
We deduce that $[[D_y,D_{y'}],x]=0$ for all $x\in V^*$. On the 
other hand, $[D_y,D_{y'}]$ acts by zero on $1\in \kb[\CCBt]\otimes 
\kb[V^{\reg}]$,
hence $[D_y,D_{y'}]$ acts by zero on $\kb[\CCBt]\otimes \kb[V^{\reg}]$.
The proposition follows.
\end{proof}

\bigskip

Proposition \ref{pr:polynomialrep} provides a morphism of
$\kb[\CCBt]$-algebras
$\Theta:\Hbt\to\kb[\CCB]\otimes\DCB_T(V^\reg)$. We denote
by
$\Theta^\reg:\Hbt^\reg\to\kb[\CCB]\otimes\DCB_T(V^\reg)$ its extension to $\Hbt^\reg$.

\bigskip

\begin{theo}
\label{th:PBWDunkl}
We have the following statements:
\begin{itemize}
\itemth{a}
The morphism $\Theta$ is injective, hence the
polynomial representation of $\Hbt$ is faithful.
\itemth{b}
The multiplication map is an isomorphism
$\kb[\CCBt]\otimes\kb[V]\otimes\kb[V^*]\otimes\kb W\longiso \Hbt$.
\itemth{c}
We have an isomorphism of algebras
$\rho:\kb[\CCBt]\otimes\kb[V \times V^*]\rtimes W\longiso \grad\,\Hbt$.
\itemth{d}
The morphism $\Theta^\reg$ is an isomorphism
$\Hbt^\reg \longiso \kb[\CCB]\otimes\DCB_T(V^\reg)\rtimes W$.
\itemth{e}
Given $\CG$ a prime ideal of $\kb[\CCBt]$, the morphism
$(\kb[\CCBt]/\CG)\otimes_{\kb[\CCBt]}\Theta$ is injective. If $T{\not\in}\CG$, then
the polynomial representation of
$\Hbt(\CG)$ is faithful and
$\Zrm(\Hbt(\CG))=\kb[\CCBt]/\CG$.
\end{itemize}
\end{theo}

\begin{proof}
Let $\eta$ be the composition
$$\eta:\kb[\CCBt]\otimes\kb[V^\reg]\otimes\kb[V^*]\otimes\kb W
\xrightarrow{\text{mult}} \Hbt^\reg 
\xrightarrow{\Theta^\reg}
\DCB_T(V^\reg)\rtimes W.$$
Note that $\mathrm{gr}~\eta$ is an isomorphism, since it is equal to 
the graded map associated to the multiplication isomorphism
$$\kb[\CCBt]\otimes\kb[V^\reg]\otimes\kb[V^*]\otimes\kb W\longiso
\kb[\CCBt]\otimes\DCB_T(V^\reg)\rtimes W.$$
We deduce that $\eta$ is an isomorphism (Lemma \ref{le:injsurjfiltered}).
Since the multiplication map is surjective, it follows that it is
an isomorphism and
$\Theta^\reg$ is an isomorphism as well.
We deduce also that $\rho$ is injective, hence it is an isomorphism.

Since $\kb[T]\otimes\kb[V^\reg]$ is a faithful representation of
$\DCB_T(V^\reg)\rtimes W$ (Lemma \ref{le:Weylfaithful}),
 we deduce that the
polynomial representation induces an injective map
$$\kb[\CCB]\otimes\DCB_T(V^\reg)\rtimes W
\hookrightarrow \kb[\CCBt]\otimes \End_{\kb}(\kb[V^\reg]).$$

There is a commutative diagram
$$\xymatrix{
\kb[\CCBt]\otimes\kb[V]\otimes\kb[V^*]\otimes\kb W \ar@{^(->}[d] 
\ar@{->>}[rr]^-{\text{mult}} && \Hbt \ar[rr]^-{\text{pol. rep.}} \ar[d]_{\text{can}}
&&
\kb[\CCBt]\otimes\End_{\kb}(\kb[V^\reg]) \\
\kb[\CCBt]\otimes\kb[V^\reg]\otimes\kb[V^*]\otimes\kb W
\ar@/_2pc/[rrrr]_-\eta^\sim
\ar@{->>}[rr]_-{\text{mult}} && \Hbt^\reg 
\ar[rr]_-{\Theta^\reg} &&
\kb[\CCB]\otimes\DCB_T(V^\reg)\rtimes W 
\ar[u]_-{\text{pol. rep.}}
}$$
It follows that the multiplication
$$\kb[\CCBt]\otimes\kb[V]\otimes\kb[V^*]\otimes\kb W\to \Hbt$$
is an isomorphism and the polynomial representation of $\Hbt$ is faithful.

\smallskip
Consider now $\CG$ a prime ideal of $\CCBt$ and let $A=\kb[\CCBt]/\CG$.
There is a commutative diagram
$$\xymatrix{
A\otimes\kb[V]\otimes\kb[V^*]\otimes\kb W \ar@{^(->}[d] 
\ar@{->>}[rr]^-{\text{mult}} && \Hbt(\CG)
 \ar[rr]^-{\text{pol. rep.}} \ar[d]_{\text{can}}
&&
A\otimes\End_{\kb}(\kb[V^\reg]) \\
A\otimes\kb[V^\reg]\otimes\kb[V^*]\otimes\kb W
\ar@/_2pc/[rrrr]_-\eta^\sim
\ar[rr]_-{\text{mult}}^-{\sim} && \Hbt^\reg(\CG)
\ar[rr]_-{\Theta^\reg}^-{\sim} &&
A\otimes_{\kb[\CCBt]}(\kb[\CCB]\otimes\DCB_T(V^\reg))\rtimes W 
\ar[u]_-{\text{pol. rep.}}
}$$
We deduce as above that $(\kb[\CCBt]/\CG)\otimes_{\kb[\CCBt]}\Theta$ is
injective.
Assume now $T{\not\in}\CG$. Then the polynomial representation of 
$A\otimes_{\kb[\CCBt]}(\kb[\CCB]\otimes\DCB_T(V^\reg))\rtimes W$ is
faithful, hence the polynomial representation of
$\Hbt(\CG)$ is faithful as well. Since $\Zrm(\DCB(V^{\reg}))=\kb$, we deduce
that $\Zrm(A\otimes_{\kb[\CCBt]}(\kb[\CCB]\otimes\DCB_T(V^\reg))\rtimes W)=A$,
hence $\Zrm(\Hbt(\CG))=\kb[\CCBt]/\CG$.
\end{proof}

\begin{coro}\label{coro:commutateur-t}
Given $f\in\kb[V]$ and $y \in V$, we have 
$$[y,f]=T\partial_y(f)-\sum_{s \in \REF(W)} \e(s) C_s 
\langle y,\a_s \rangle \frac{s(f)-f}{\alpha_s}s.$$
\end{coro}

\begin{proof}
The result follows from Theorem~\ref{th:PBWDunkl}.
Note that the corollary can also be proven directly by induction on the
degree of $f$.
\end{proof}

\bigskip

\subsection{Hyperplanes and parameters}\label{section:hyperplans-variables}

Let $\KCB$\indexnot{KC}{\KCB} be the $\kb$-vector space of maps
$k:\orbiteb^\circ\to\kb$, $(\orbite,j) \mapsto k_{\orbite,j}$ such that
for all $\orbite\in \AC/W$, we have $\sum_{j=0}^{e_\orbite-1}k_{\orbite,j}=0$.
Let $(K_{\orbite,j})_{(\orbite,j)\in
\orbiteb^\circ}$\indexnot{Komega}{K_{\orbite,j},K_{H,j}} be the canonical
basis of $\kb^{\orbiteb^\circ}$.
We put $K_{H,j}=K_{\orbite,j}$, where $\orbite$ is the
$W$-orbit of $H$.

There is an isomorphism of $\kb$-vector spaces
$$\CCB^*\xrightarrow{\sim}\KCB^*,\
C_{s_{H^i}}\mapsto \sum_{j=0}^{e_H-1}\z_{e_H}^{i(j-1)} K_{H,j}.$$
The surjectivity 
is a consequence of the invertibility of the Vandermonde determinant.
We denote the dual of that isomorphism by
$\kappa:\KCB\xrightarrow{\sim}\CCB$\indexnot{kappa}{\kappa}.
We will often identify $\CCB$ and $\KCB$ via the isomorphism $\kappa$.
Note that the canonical bases of $\CCB$ and $\KCB$ provide them with
$\QM$-forms (and $\ZM$-forms).
Unless all reflections of $W$ have order $2$, the isomorphism
$\kappa$ is not compatible with the $\QM$-forms.

\smallskip
Note that
$$\sum_{w\in W_H}\e(w)C_w w=e_H\sum_{j=0}^{e_H-1}\e_{H,j}K_{H,j}$$
where $\e_{H,i}=e_H^{-1}\sum_{w\in W_H}\e(w)^i w$ and
\equat\label{escs}
\sum_{s \in \REF(W)} \e(s)~C_s = \sum_{H \in \AC} e_H K_{H,0} = 
-\sum_{H \in \AC} \sum_{i=1}^{e_H-1} e_H K_{H,i}.
\endequat

%
%
%

\bigskip

Via $\kappa$, we can view
$\Hbt$ as a $\kb[\KCB]$-algebra and the second
relation in~(\ref{relations-1}) becomes
\equat\label{eq:rel-1-H}
[y,x]=T\langle y,x\rangle +\sum_{H\in\mathcal{A}} \sum_{i=0}^{e_H-1}
e_H(K_{H,i}-K_{H,i+1}) \frac{\langle y,\a_H \rangle \cdot \langle \a_H^\ve,x\rangle}{\langle \a_H^\ve,\a_H\rangle} 
\e_{H,i}
\endequat
for $x\in V^*$ and $y\in V$, where $K_{H,e_H}=K_{H,0}$.
\medskip

Given $y\in V$, we have
$$\Theta(y)=\partial_y-\sum_{H\in\mathcal{A}}\sum_{i=0}^{e_H-1}\frac{\langle y,\alpha_H\rangle}{\alpha_H}
e_H K_{H,i}\e_{H,i}.$$

\bigskip

\noindent{\sc Comment - } 
Our convention for the definition of Cherednik algebras differs from that
of \cite[\S{3.1}]{ggor}: we have added a coefficient $\e(s)-1$ 
in front of the term $C_s$. On the other hand, our convention is the same
as \cite[\S{1.15}]{EG}, 
with $c_s=c_{\a_s}$ (when $W$ is a Coxeter group). Note that the
$\mathrm{k}_{H,i}$'s from \cite{ggor} are related to the
$K_{H,i}$'s above by the relation $\mathrm{k}_{H,i}=K_{H,0}-K_{H,i}$.\finl

%

\begin{rema}
The endomorphism
$K_{\orbite,j}\mapsto K_{\orbite,j}-\frac{1}{e_\orbite}\sum_{j'=0}^{e_\orbite-1}
K_{\orbite,j'}$ of $\kb^{\orbiteb^\circ}$ induces an injection
$\KCB^*\hookrightarrow \kb^{\orbiteb^\circ}$.
The dual map $\sec:\kb^{\orbiteb^\circ}\twoheadrightarrow
\KCB$\indexnot{sec}{\sec} provides
a section to the inclusion of $\KCB$ in $\kb^{\orbiteb^\circ}$. It is defined
over $\QM$.\finl
\end{rema}

\bigskip

\section{Gradings}\label{section:graduation-1}

\medskip

The algebra $\Hbt$ admits a natural $(\NM\times\NM)$-grading, 
thanks to which we can associate, to each morphism of monoids 
$\NM \times \NM \to \BZ$ (or $\NM \times \NM \to \NM$), a $\BZ$-grading 
(or an $\NM$-grading) of $\Hbt$.

\medskip

We endow the extended tensor algebra 
$\kb[\CCBt] \otimes \bigl(\Trm_\kb(V \oplus V^*) \rtimes W\bigr)$ with an
$(\NM \times \NM)$-grading 
by giving the elements of $V$ the bidegree $(1,0)$, 
the elements of $V^*$ the bidegree $(0,1)$, the elements of $\CCBt^*$ 
the bidegree $(1,1)$ and those of $W$ 
the bidegree $(0,0)$. The relations~(\ref{relations-1}) are
homogeneous. 
Hence, $\Hbt$ inherits an $(\NM \times \NM)$-grading whose homogeneous component 
of bidegree $(i,j)$ will be denoted by 
$\Hbt^{\NM \times \NM}[i,j]$. \indexnot{H}{\Hbt^{\NM \times \NM}[i,j],~\Hbt^\ph[i],~\Hbt^\NM[i],\Hbt^\BZ[i]}  
We have
$$\Hbt=\mathop{\bigoplus}_{(i,j) \in \NM \times \NM} 
\Hbt^{\NM \times \NM}[i,j]\quad\text{and}\quad \Hbt^{\NM \times \NM}[0,0]=\kb W.$$
Note that all homogeneous components have finite dimension over $\kb$. 

\medskip

If $\ph : \NM \times \NM \to \BZ$ is a morphism of monoids, 
then $\Hbt$ inherits a $\BZ$-grading whose homogeneous component of degree $i$ 
will be denoted by $\Hbt^\ph[i]$:
$$\Hbt^\ph[i]=\mathop{\bigoplus}_{\ph(a,b) = i} \Hbt^{\NM \times \NM}[a,b].\indexnot{H}{\Hbt_\ph[i]}$$
In this grading, the elements of $V$ have degree $\ph(1,0)$, 
the elements of $V^*$ have degree $\ph(0,1)$, the elements of $\CCBt^*$ 
have degree $\ph(1,1)$ and those of $W$ have degree $0$.

\bigskip

\begin{exemple}[$\BZ$-grading]\label{Z graduation-1}
The morphism of monoids $\NM \times \NM \to \BZ$, $(i,j) \mapsto j-i$ induces a 
$\BZ$-grading on $\Hbt$ for which the elements of $V$ have degree $-1$, 
the elements of $V^*$ have degree $1$ and the elements of $\CCBt^*$ and $W$ 
have degree $0$. We denote by $\Hbt^\BZ[i]$ the homogeneous component of degree $i$. We have
$$\Hbt=\mathop{\bigoplus}_{i \in \BZ} \Hbt^\BZ[i].$$
By specialization at $(t,c) \in \CCBt$, the algebra $\Hbt_{t,c}$ inherits a $\BZ$-grading 
whose homogeneous component of degree $i$ will be denoted by
$\Hbt_{t,c}^\BZ[i]$.

\smallskip
	Assume $W$ is irreducible.
	Let $w_z$\indexnot{wz}{w_z} be a generator of 
	$Z(W)=W\cap \Zrm(\GL(V))$ and let
$z_W$\indexnot{zW}{z_W} be its order. We have $w_z=\zeta^{-1}\Id_V$ for some
root of unity $\zeta$\indexnot{zeta}{\zeta} of order $z_W$ of $\kb$.
When $\kb$ is a subfield of $\CM$, we take 
$w_z=e^{2i\pi/z_W}\Id_V$.
Note that $z_W=\gcd(d_1,\ldots,d_n)$
(a consequence of Theorem \ref{chevalley}).
Given $h\in \Hbt^\BZ[i]$, we have $w_zhw_z^{-1}=\zeta^i h$. So, the 
$(\BZ/z_W\BZ)$-grading on $\Hbt$ deduced from the $\BZ$-grading is given
by an inner automorphism of $\Hbt$.

\smallskip
	Consider again a general $W$ and let 
	$(W=W_1\times\cdots\times W_r,V=V^W\oplus V_1\oplus\cdots\oplus V_r)$
	be the decomposition
	into indecomposable components (cf \S\ref{section:hyperplans}).
	Let $w_i$ be the elements of $W_i\cap Z(\GL(V_i))$ defined as
	above.  We put $w_z=w_1\times\cdots\times w_r\in Z(W)$ and
	$z_W=z_1\cdots z_r=|Z(W)|$.
\finl
\end{exemple}

\bigskip

\begin{exemple}[$\NM$-grading]\label{N graduation-1}
The morphism of monoids $\NM \times \NM \to \NM$, $(i,j) \mapsto i+j$ induces  
an $\NM$-grading on $\Hbt$ for which the elements of $V$ and $V^*$ 
have degree $1$, the elements of $\CCBt^*$ have degree $2$ and the 
elements of $W$ have degree $0$. 
We denote by $\Hbt^\NM[i]$ the homogeneous component of degree $i$. We have
$$\Hbt=\mathop{\bigoplus}_{i \in \NM} \Hbt^\NM[i]\quad\text{and}\quad \Hbt^\NM[0]=\kb W.$$
Note that $\dim_\kb \Hbt^\NM[i] < \infty$ for all $i$. 
This grading is not inherited after specialization at $(t,c) \in \CCBt$, 
except when $(t,c)=(0,0)$: we retrieve 
the usual $\NM$-grading on $\Hbt_{0,0} = \kb[V \times V^*] \rtimes W$ 
(see Example~\ref{exemple zero-1}).\finl
\end{exemple}

\bigskip

\section{Euler element}\label{section:eulertilde}

\medskip

Let $(x_1,\dots,x_n)$ be a $\kb$-basis of $V^*$ and let $(y_1,\dots,y_n)$ 
be its dual basis. We define the {\it generic Euler element} of $\Hbt$
$$\eulertilde=~-nT + \sum_{i=1}^n y_i x_i + \sum_{s \in \REF(W)} C_s s\in\Hbt.$$\indexnot{ea}{\eulertilde}
Note that 
$$\eulertilde=~\sum_{i=1}^n x_i y_i + \sum_{s \in \REF(W)} \e(s) C_s s =~\sum_{i=1}^n x_i y_i +
\sum_{H\in\AC}\sum_{j=0}^{e_H-1}e_H\ K_{H,j}\varepsilon_{H,j}.$$
It is easy to check that $\eulertilde$ does not depend on the choice of the basis 
$(x_1,\dots,x_n)$ of $V^*$. Note that 
\equat\label{eq:eulertilde-1-1}
\eulertilde \in \Hbt^{\NM \times \NM}[1,1].
\endequat
We have
$$\Theta(\eulertilde)=T\sum_{i=1}^n y_i x_i$$
Thanks to Theorem \ref{th:PBWDunkl}, we deduce the following
result~\cite[\S 3.1(4)]{ggor}.

\bigskip

\begin{prop}\label{lem:eulertilde}
Given $x \in V^*$, $y \in V$ and $w \in W$, we have
$$[\eulertilde,x]=Tx,\qquad[\eulertilde,y]=-Ty\qquad\text{and}\qquad [\eulertilde,w]=0.$$
\end{prop}

\bigskip

In~\cite{ggor}, the Euler element plays a fundamental role in the study of the category $\OC$ associated 
with $\Hbt_{1,c}$. We will see in this book the role it plays in the theory of Calogero-Moser 
cells.

\bigskip

\begin{prop}\label{prop:eulertilde}
Given $h \in \Hbt^\BZ[i]$, we have $[\eulertilde,h]=iTh$.
\end{prop}

\bigskip

\section{Spherical algebra}\label{section:spherique-1}

\medskip

\boitegrise{{\bf Notation.} 
{\it Throughout this book, we denote by $e$ \indexnot{ea}{e}  the primitive central idempotent 
of $\kb W$ defined by 
$$e=\frac{1}{|W|}\sum_{w \in W} w.$$
The $\kb[\CCBt]$-algebra $e \Hbt e$ will be called the 
{\bfit generic spherical algebra}.}}{0.75\textwidth}

\medskip

By specializing at $(t,c)$, and since $e\Hbt e$ is a direct summand of the $\kb[\CCBt]$-module 
$\Hbt$, we get 
\equat\label{spherique-1}
e \Hbt_{t,c} e = (\kb[\CCBt]/\CGt_{t,c}) \otimes_{\kb[\CCBt]} e\Hbt e.
\endequat
Since $e$ has degree $0$, the filtration of $\Hbt$ induces a filtration
of the generic spherical algebra given by
$(e\Hbt e)^{\le i}= e (\Hbt^{\le i})e$.
It follows from Theorem \ref{th:PBWDunkl} and Proposition \ref{pr:centerdouble}
that 
\equat\label{spherique graduee-1}
\grad(e\Hbt e) = e \grad(\Hbt) e \simeq \kb[\CCBt] \otimes 
\kb[V \times V^*]^{\Delta W}.
\endequat

\bigskip

\begin{theo}[Etingof-Ginzburg]\label{EG spherique-1}
Let $\CGt$ be a prime ideal of $\kb[\CCBt]$. 

\begin{itemize}
\itemth{a} The algebra $e \Hbt(\CGt) e$ is a finitely generated $\kb$-algebra
without zero divisors.

\itemth{b} $\Hbt(\CGt) e$ is a finitely generated right $e\Hbt(\CGt)e$-module.

\itemth{c} Left multiplication of $\Hbt(\CGt)$ on the 
projective module $\Hbt(\CGt) e$ induces an isomorphism 
$\Hbt(\CGt)\stackrel{\sim}{\longto} \End_{(e\Hbt(\CGt)e)^\opp}(\Hbt(\CGt)e)^\opp$.

\itemth{d} There is an isomorphism of algebras $\Zrm(\Hbt(\CGt))\longiso
\Zrm(e \Hbt(\CGt)e),\ z\mapsto ze$.
\itemth{e} The $(\Hbt(\CGt)\otimes_{\kb[\CCBt]/\CGt}\Hbt(\CGt)^\opp)$-module $\Hbt(\CGt)$
	has finite projective dimension.
	If $\kb[\CCBt]/\CGt$ is regular, then $\Hbt(\CGt)$ has finite global dimension.
%
\end{itemize}
\end{theo}

\begin{proof}
The assertion (a) follows from Lemmas \ref{le:generatefilt} and
 \ref{le:zerodivisorsfilt}.
The assertion (b) follows from Lemma \ref{le:generatefilt}.

Let $\alpha:\Hbt(\CGt)\to
 \End_{(e\Hbt(\CGt)e)^\opp}(\Hbt(\CGt)e)^\opp$ be the morphism of the theorem.
Lemma \ref{le:injgrHom} provides an injective morphism
$$\beta:\grad\,\End_{(e\Hbt(\CGt)e)^\opp}(\Hbt(\CGt)e)^\opp\hookrightarrow
\End_{\grad\,(e\Hbt(\CGt)e)^\opp}(\grad\,\Hbt(\CGt)e)^\opp.$$
The composition
$$\grad\,\Hbt(\CGt)\xrightarrow{\grad\,\alpha}
\grad\,\End_{(e\Hbt(\CGt)e)^\opp}(\Hbt(\CGt)e)^\opp\xrightarrow{\beta}
\End_{\grad\,(e\Hbt(\CGt)e)^\opp}(\grad\,\Hbt(\CGt)e)^\opp$$
is given by the left multiplication action.
Via the isomorphism $\rho$ (Theorem \ref{th:PBWDunkl}), it corresponds to the 
morphism given by left multiplication
$$\gamma:\kb[\CCBt]\otimes\kb[V \times V^*]\rtimes W\to
\End_{\kb[\CCBt]\otimes(e(\kb[V \times V^*]\rtimes W)e)^\opp}
(\kb[\CCBt]\otimes(\kb[V \times V^*]\rtimes W)e)^\opp.$$
Since the codimension of the complement of $(V\times (V^*\setminus (V^*)^{\reg}))\cup
((V\setminus V^{\reg})\times V^*)$ in $V\times V^*$ is $\ge 2$, it follows from
Proposition \ref{pr:codim2doublecentralizer} that $\gamma$ is an isomorphism.
So, $\grad\,\alpha$ is
an isomorphism, hence $\alpha$ is an isomorphism by
Lemma \ref{le:injsurjfiltered}.

The assertion (d) follows from (c) by Lemma \ref{lem:ZA-ZB}.

	The first part of assertion (e) follows from Lemma \ref{le:gldimfilt}
	and the second part is an immediate consequence.
\end{proof}

\begin{rema}
It can actually be shown \cite[Theorem 1.5]{EG} that
if $\kb[\CCBt]/\CGt$ is Gorenstein (respectively Cohen-Macaulay), then so is
the algebra  
$e\Hbt(\CGt)e$ as well as the right $e\Hbt(\CGt)e$-module $\Hbt(\CGt)e$.\finl 
\end{rema}

\bigskip
%

%
%
%

\section{Some automorphisms of $\Hbt$}\label{section:automorphismes-1}

\medskip

Let $\Aut_{\kb\text{-}\alg}(\Hbt)$ denote the group of automorphisms of the $\kb$-algebra $\Hbt$.

\bigskip

\subsection{Bigrading}\label{subsection:bi-graduation-1}
The bigrading on $\Hbt$ can be seen as an action of the algebraic group 
$\kb^\times \times \kb^\times$ on $\Hbt$. Indeed, 
if $(\xi,\xi') \in \kb^\times \times \kb^\times$, we define the automorphism 
$\gradauto_{\xi,\xi'}$ \indexnot{ba}{\gradauto_{\xi,\xi'}} of $\Hbt$ by the following formula:
$$\forall~(i,j) \in \NM \times \NM,~\forall~h \in \Hbt^{\NM \times \NM}[i,j],~
\gradauto_{\xi,\xi'}(h)=\xi^i\xi^{\prime j} h.$$
Then
\equat\label{graduation automorphisme-1}
\bigrad : \kb^\times \times \kb^\times \longto \Aut_{\kb\text{-}\alg}(\Hbt)
\endequat
is a morphism of groups. Concretely,
$$
\begin{cases}
\forall~y \in V,~\gradauto_{\xi,\xi'}(y)=\xi y,\\
\forall~x \in V^*,~\gradauto_{\xi,\xi'}(x)=\xi' x,\\
\forall~C \in \CCBt^*,~\gradauto_{\xi,\xi'}(C)=\xi\xi' C,\\
\forall~w \in W,~\gradauto_{\xi,\xi'}(w)=w.\\
\end{cases}
$$
After specialization, for all $\xi \in \kb^\times$ and all 
$(t,c) \in \CCBt$, the action of $(\xi,1)$ induces an isomorphism of $\kb$-algebras 
\equat\label{xi c}
\Hbt_{t,c} \longiso \Hbt_{\xi t, \xi c}.
\endequat

\bigskip

\subsection{Linear characters}\label{subsection:lineaires-1} 
Let $\g : W \longto \kb^\times$ be a linear character. 
It provides an automorphism 
of $\CCB$ by multiplication: 
given $c \in \CCB$, we define $\g \cdot c$ as the map $\REF(W) \to \kb$, $s \mapsto \g(s)c_s$. 
This induces an automorphism $\g_\CCB : \kb[\CCB] \to \kb[\CCB]$, \indexnot{gz}{\g_\CCB}
$f \mapsto (c \mapsto f(\g^{-1} \cdot c))$ sending
$C_s$ on $\g(s)^{-1} C_s$. It extends to an automorphism $\g_\CCBt$ of $\kb[\CCBt]$ 
by setting $\g_\CCBt(T)=T$. 

On the other hand, $\g$ induces also an automorphism of the group algebra $\kb W$ given by $W\ni w\mapsto \g(w) w$. Hence, $\g$ induces an automorphism 
of the $\kb[\CCBt]$-algebra $\kb[\CCBt] \otimes \bigl(\Trm_\kb(V \oplus V^*) \rtimes W\bigr)$ 
acting trivially on $V$ and $V^*$: it will be denoted by $\g_\Trm$. Of course,
$$(\g\g')_\Trm=\g_\Trm \g_\Trm'.$$
Since the relations~(\ref{relations-1}) are stable by the action of $\g_\Trm$, 
it follows that $\g_\Trm$ induces an automorphism $\g_*$ of 
the $\kb$-algebra $\Hbt$. The map 
\equat\label{action caracteres lineaires-1}
\fonctio{W^\wedge}{\Aut_{\kb\text{-}\alg}(\Hbt)}{\g}{\g_*}
\endequat
is an injective morphism of groups. Given $(t,c) \in \CCBt$ 
and $\g \in W^\wedge$, then $\g_*$ induces an isomorphism of $\kb$-algebras 
\equat\label{gamma c-1}
\Hbt_{t,c} \longiso \Hbt_{t,\g \cdot c}.
\endequat


\bigskip

\subsection{Normalizer}\label{subsection:normalisateur-1}
Let $\NC$ \indexnot{N}{\NC}  denote the normalizer of $W$ in $\GL(V)$. Then:
\begin{itemize}
\item[$\bullet$] $\NC$ acts naturally on $V$ and $V^*$;

\item[$\bullet$] $\NC$ acts on $W$ by conjugation;

\item[$\bullet$] The action of $\NC$ on $W$ stabilizes $\REF(W)$ and so $\NC$ acts on 
$\CCB$: if $g \in \NC$ and $c \in \CCB$, then 
$\lexp{g}{c} : \REF(W) \to \kb$, $s \mapsto c_{g^{-1}sg}$.

\item[$\bullet$] The action of $\NC$ on $\CCB$ induces an action of $\NC$ on $\CCB^*$ 
(and so on $\kb[\CCB]$) such that, if $g \in \NC$ and $s \in \REF(W)$, then 
$\lexp{g}{C_s}=C_{gsg^{-1}}$.

\item[$\bullet$] $\NC$ acts trivially on $T$. 
\end{itemize}
Consequently, $\NC$ acts on the $\kb[\CCBt]$-algebra 
$\kb[\CCBt] \otimes \bigl(\Trm_\kb(V \oplus V^*) \rtimes W\bigr)$ and it is easily checked, 
thanks to the relations~(\ref{relations-1}), that 
this action induces an action on $\Hbt$: if $g \in \NC$ and $h \in \Hbt$, 
we denote by $\lexp{g}{h}$ the image of $h$ under the action of $g$.
By specialization at $(t,c) \in \CCBt$, an element $g \in \NC$ induces an isomorphism of $\kb$-algebras 
\equat\label{N c}
\Hbt_{t,c} \longiso \Hbt_{t,\lexp{g}{c}}.
\endequat

\bigskip
%
%
%

\begin{exemple}\label{Z graduation et normalisateur-1}
If $\xi \in \kb^\times$, then we can see $\xi$ as an automorphism of $V$ 
(by scalar multiplication) normalizing (and even centralizing) $W$. We then recover the 
automorphism of $\Hbt$ inducing the $\BZ$-grading (up to a sign): 
if $h \in \Hbt$, then $\lexp{\xi}{h}=\gradauto_{\xi,\xi^{-1}}(h)$.\finl
\end{exemple}

\bigskip

\subsection{Compatibilities}\label{subsection:compilation-1} 
The automorphisms induced by $\kb^\times \times \kb^\times$ and 
those induced by $W^\wedge$ commute. On the other hand, the group  
$\NC$ acts on the group $W^\wedge$ and on the $\kb$-algebra $\Hbt$. 
This induces an action of $W^\wedge \rtimes \NC$ 
on $\Hbt$ preserving the bigrading, that is, commuting
with the action of $\kb^\times \times \kb^\times$. 
Given $\g \in W^\wedge$ and $g \in \NC$, we will denote by $\g \rtimes g$ the corresponding 
element of $W^\wedge \rtimes \NC$. 
We have a morphism of groups 
$$\fonctio{\kb^\times \times \kb^\times \times (W^\wedge \rtimes \NC)}{\Aut_{\kb\text{-}\alg}(\Hbt)}{
(\xi,\xi',\g \rtimes g)}{(h \mapsto \gradauto_{\xi,\xi'} \circ \g_*(\lexp{g}{h})).}$$
Given $\t=(\xi,\xi',\g \rtimes g) \in \kb^\times \times \kb^\times \times (W^\wedge \rtimes \NC)$ 
and $h \in \Hbt$, we set
$$\lexp{\t}{h}=\gradauto_{\xi,\xi'}\bigl(\g_*(\lexp{g}{h})\bigr).$$\indexnot{ha}{{{^\t{h}}}}
The following lemma is immediate.

\bigskip

\begin{lem}\label{lem:automorphismes-1}
Let $\t=(\xi,\xi',\g \rtimes g) \in \kb^\times \times \kb^\times \times (W^\wedge \rtimes \NC)$. 
Then:
\begin{itemize}
\itemth{a} $\t$ stabilizes the subalgebras $\kb[\CCBt]$, $\kb[V]$, $\kb[V^*]$ and $\kb W$.

\itemth{b} $\t$ preserves the bigrading.

\itemth{c} $\lexp{\t}{\eulertilde}=\xi\xi' \eulertilde$.

\itemth{d} $\lexp{\t}{e}=e$ if and only if $\g=1$.
\end{itemize}
\end{lem}

%
%

\section{Special features of Coxeter groups}\label{section:coxeter-htilde}

\medskip

\cbstart

\boitegrise{{\bf Assumption.} 
{\it In this section \ref{section:coxeter-htilde}, we assume that 
$W$ is a Coxeter group, and we use the notation of \S\ref{se:Coxetergroups}.
}}{0.75\textwidth}

\bigskip

By Proposition~\ref{prop:coxeter}, there exists a non-degenerate symmetric bilinear $W$-invariant form 
$\betb : V \times V \to \kb$. \indexnot{bz}{\betb}  
We denote by $\s : V \stackrel{\sim}{\longto} V^*$ \indexnot{sz}{\s}  the isomorphism induced by 
$\betb$: given $y$, $y' \in V$, we have
$$\langle y,\s(y') \rangle = \betb(y,y').$$
The $W$-invariance of $\betb$ implies that $\s$ is an isomorphism of $\kb W$-modules 
and the symmetry of $\betb$ implies that 
\equat\label{eq:sigma-sigma-inverse}
\langle y , x \rangle = \langle \s^{-1}(x),\s(y) \rangle
\endequat
for all $x \in V^*$ and $y \in V$. We denote by 
$\s_\Trm : \Trm_\kb(V \oplus V^*) \to \Trm_\kb(V \oplus V^*)$ the automorphism 
of algebras induced by the automorphism of the vector space 
$V \oplus V^*$ defined by $(y , x) \mapsto (-\s^{-1}(x), \s(y))$. It is $W$-invariant, hence
extends to an automorphism of $\Trm_\kb(V \oplus V^*) \rtimes W$, 
with trivial action on $W$. By extension of scalars, we get another automorphism, 
still denoted by $\s_\Trm$, of $\kb[\CCBt] \otimes (\Trm_\kb(V \oplus V^*) \rtimes W)$. 
It is easy to check that $\s_\Trm$ induces an automorphism $\s_\Hbt$ \indexnot{sz}{\s_\Hbt}  of $\Hbt$. 
We have proven the following proposition.

\bigskip

\begin{prop}\label{prop:auto-coxeter-1}
There exists a unique automorphism $\s_\Hbt$ of $\Hbt$ such that 
$$
\begin{cases}
\s_\Hbt(y)=\s(y) & \text{if $y \in V$,}\\
\s_\Hbt(x)=-\s^{-1}(x) & \text{if $x \in V^*$,}\\
\s_\Hbt(w)=w & \text{if $w \in W$,}\\
\s_\Hbt(C)=C & \text{if $C \in \CCBt^*$.}\\
\end{cases}
$$
\end{prop}

\bigskip

\begin{prop}\label{prop-bis:auto-coxeter-1}
The following hold:
\begin{itemize}
\itemth{a} $\s_\Hbt$ stabilizes the subalgebras $\kb[\CCBt]$ and $\kb W$ and exchanges 
the subalgebras $\kb[V]$ and $\kb[V^*]$.

\itemth{b} If $h \in \Hbt^{\NM \times \NM}[i,j]$, 
then $\s_\Hbt(h) \in \Hbt^{\NM \times \NM}[j,i]$.

\itemth{c} If $h \in \Hbt^\NM[i]$ (respectively $h \in \Hbt^\BZ[i]$), 
then $\s_\Hbt(h) \in \Hbt^\NM[i]$ (respectively $\s_\Hbt(h) \in \Hbt^\BZ[-i]$).

\itemth{d} $\s_\Hbt$ commutes with the action of $W^\wedge$ on $\Hbt$.

\itemth{e} If $(t,c) \in \CCBt$, then $\s_\Hbt$ induces an automorphism 
of $\Hbt_{t,c}$, still denoted by $\s_\Hbt$ (or $\s_{\Hbt_{t,c}}$ if necessary).

\itemth{f} $\s_\Hbt(\eulertilde)=-nT - \eulertilde$.
\end{itemize}
\end{prop}

\begin{rema}[Action of $\Gb\Lb_2(\kb)$]\label{rem:sl2-1}
Let $\r = \begin{pmatrix} a & b \\ c & d \end{pmatrix} \in \Gb\Lb_2(\kb)$. 
The $\kb$-linear map 
$$\fonctio{V \oplus V^*}{V \oplus V^*}{y \oplus x}{ay+b\s^{-1}(x) \oplus c \s(y) + dx}$$
is an automorphism of the $\kb W$-module $V \oplus V^*$. It extends to an automorphism 
of the $\kb$-algebra $\Trm_\kb(V \oplus V^*) \rtimes W$ by letting it
	act trivially on $W$ and to an 
automorphism $\r_\Trm$ of $\kb[\CCBt] \otimes (\Trm_\kb(V \oplus V^*) \rtimes W)$ 
by $\r_\Trm(C)=\det(\r) C$ for $C \in \CCBt^*$. 

It is easy to check that $\r_\Trm$ induces an automorphism $\r_\Hbt$ \indexnot{rz}{\r_\Hbt}  
of $\Hbt$. Moreover, $(\r \r')_\Hbt=\r_\Hbt \circ \r_\Hbt'$ for all $\r$, $\r' \in \Gb\Lb_2(\kb)$. So, we obtain
an action of $\Gb\Lb_2(\kb)$ on $\Hbt$. This action preserves the $\NM$-grading 
$\Hbt^\NM$. 

Finally, note that, for $\r=\begin{pmatrix} 0 & -1 \\ 1 & 0 \end{pmatrix}$, we have 
$\r_\Hbt=\s_\Hbt$ and, if $\r=\begin{pmatrix} \xi & 0 \\ 0 & \xi' \end{pmatrix}$, then 
$\r_{\Hbt}=\gradauto_{\xi,\xi'}$. Hence we have extended the action of 
$\kb^\times \times \kb^\times \times (W^\wedge \rtimes \NC)$ to an action of 
$\Gb\Lb_2(\kb) \times (W^\wedge \rtimes \NC)$.\finl
\end{rema}

\cbend

\chapter{Cherednik algebras at $t=0$}\label{chapter:cherednik-0}

\boitegrise{{\bf Notation.} {\it We put $\Hb=\Hbt/T\Hbt$. \indexnot{H}{\Hb}  The
$\kb$-algebra 
$\Hb$ is called the {\bfit Cherednik algebra at $t=0$}.}}{0.75\textwidth}

\bigskip

\section{Generalities}

\medskip

We gather here those properties that are immediate consequences of results
discussed in Chapter~\ref{chapter:cherednik-1}. We also introduce some
notation.

Let us rewrite the defining relations~(\ref{relations-1}).
The algebra $\Hb$ is the $\kb[\CCB]$-algebra quotient 
of $\kb[\CCB] \otimes \bigl(\Trm_\kb(V \oplus V^*) \rtimes W\bigr)$ 
by the ideal generated by the following relations:
\equat\label{relations-0}\begin{cases}
[x,x']=[y,y']=0, \\
\\
[y,x] = \DS{\sum_{s \in \REF(W)} (\e(s)-1)\hskip1mm C_s 
\hskip1mm\frac{\langle y,\a_s \rangle \cdot \langle \a_s^\ve,x\rangle}{\langle \a_s^\ve,\a_s\rangle}
\hskip1mm s,} 
\end{cases}\endequat
for $x$, $x' \in V^*$ and $y$, $y' \in V$. 

\smallskip
The PBW-decomposition (Theorem~\ref{PBW-1}) takes the following form.

\bigskip

\begin{theo}[Etingof-Ginzburg]\label{PBW-0}
The multiplication map gives an isomorphism of
$\kb[\CCB]$-modules
$$\kb[\CCB] \otimes \kb[V] \otimes \kb W 
\otimes \kb[V^*] \longiso \Hb.$$
\end{theo}

\bigskip

Given $c \in \CCB$, we denote by $\CG_c$ \indexnot{C}{\CG_c} the maximal ideal of $\kb[\CCB]$ defined by  
$\CG_c=\{f \in \kb[\CCB]~|~f(c)=0\}$: it is the ideal generated by 
$(C_s-c_s)_{s \in \refw}$. We set
$$\Hb_c = (\kb[\CCB]/\CG_c) \otimes_{\kb[\CCB]} \Hb = \Hb/\CG_c \Hb = \Hbt_{0,c}.\indexnot{H}{\Hb_c}$$
The $\kb$-algebra $\Hb_c$ is the quotient of the $\kb$-algebra 
$\Trm_\kb(V \oplus V^*) \rtimes W$ by the ideal generated by the
following relations: 
\equat\label{relations specialisees-0}\begin{cases}
[x,x']=[y,y']=0, \\
\\
[y,x] = \DS{\sum_{s \in \REF(W)} (\e(s)-1)\hskip1mm c_s 
\hskip1mm\frac{\langle y,\a_s \rangle \cdot \langle \a_s^\ve,x\rangle}{\langle \a_s^\ve,\a_s\rangle}
\hskip1mm s,} 
\end{cases}\endequat
for $x$, $x' \in V^*$ and $y$, $y' \in V$. 

Since $T$ is bi-homogeneous, the $\kb$-algebra $\Hb$ 
inherits all the gradings, filtrations of the algebra $\Hbt$: 
we will use the obvious notation $\Hb^{\NM \times \NM}[i,j]$, 
$\Hb^\NM[i]$, $\Hb^\BZ[i]$ and $\Hb^{\le i}$ for the constructions obtained
by quotient from $\Hbt$.
We will denote by $\euler$ \indexnot{ea}{\euler}  the image of $\eulertilde$ in $\Hb$. 
This is the {\it generic Euler element} of $\Hb$. 
Note that 
\equat\label{eq:euler-1-1}
\euler \in \Hb^{\NM \times \NM}[1,1]\subset\Hb^{\ZM}[0]
\endequat
The ideal generated by $T$ is also stable by the action of 
$\kb^\times \times \kb^\times \times (W^\wedge \rtimes \NC)$, so  
$\Hb$ inherits an action of 
$\kb^\times \times \kb^\times \times (W^\wedge \rtimes \NC)$.
The action of
$\t \in \kb^\times \times \kb^\times \times (W^\wedge \rtimes \NC)$ on
$h \in \Hb$ is still denoted by
$\lexp{\t}{h}$.
The following lemma is immediate from Lemma~\ref{lem:automorphismes-1}.

\bigskip

\begin{lem}\label{lem:automorphismes-0}
Let $\t=(\xi,\xi',\g \rtimes g) \in \kb^\times \times \kb^\times \times (W^\wedge \rtimes \NC)$. 
Then
\begin{itemize}
\itemth{a} $\t$ stabilizes the subalgebras $\kb[\CCB]$, $\kb[V]$, $\kb[V^*]$ and $\kb W$.

\itemth{b} $\t$ stabilizes the bigrading.

\itemth{c} $\lexp{\t}{\euler}=\xi\xi'~ \euler$.
\end{itemize}
\end{lem}

Theorem~\ref{EG spherique-1} implies
the following result on the spherical algebra.

\bigskip

\begin{theo}[Etingof-Ginzburg]\label{EG spherique-0}
Let $\CG$ be a prime ideal of $\kb[\CCB]$ and let $\Hb(\CG)=\Hb/\CG\Hb$. \indexnot{H}{\Hb(\CG)}  Then
\begin{itemize}
\itemth{a} The algebra $e \Hb(\CG) e$ is a finitely generated $\kb$-algebra 
without zero divisors.
%
%
\itemth{b} Left multiplication of $\Hb(\CG)$ on the projective module 
$\Hb(\CG)e$ induces an 
isomorphism $\Hb(\CG) \stackrel{\sim}{\longto} 
\End_{(e\Hb(\CG)e)^\opp}(\Hb(\CG) e)^\opp$.
\end{itemize}
\end{theo}

\bigskip

Let $\Hb^\reg=\kb[\CCB]\otimes_{\kb[\CCBt]}\Hbt^{\reg}$.
 \indexnot{H}{\Hb^\reg}
Theorem \ref{th:PBWDunkl} becomes the following result.

\begin{theo}[Etingof-Ginzburg]\label{thm:eg}
There exists a unique isomorphism of $\kb[\CCB]$-algebras 
$$\Th : \Hb^\reg \longiso \kb[\CCB] \otimes (\kb[V^\reg \times V^*] \rtimes W)\indexnot{ty}{\Th}$$
such that 
$$
\begin{cases}
\Th(w) = w & \text{for $w \in W$},\\
\Th(y) = y - \DS{\sum_{s \in \REF(W)} 
\e(s)C_s \hskip1mm\frac{\langle y,\a_s\rangle}{\a_s}\hskip1mm s} & \text{for $y \in V$},\\
\Th(x) = x & \text{for $x \in V^*$.}\\
\end{cases}
$$

Given $\CG$ a prime ideal of $\kb[\CCB]$, the restriction of
$(\kb[\CCB]/\CG)\otimes_{\kb[\CCB]}\Theta$ to
$(\kb[\CCB]/\CG)\otimes_{\kb[\CCB]}\Hb$ is injective.
\end{theo}


\section{Center}

\medskip

\boitegrise{{\bf Notation.} {\it Throughout this book, 
we denote by $Z=\Zrm(\Hb)$ \indexnot{Z}{Z,~Z_c}  the center of $\Hb$.
Given $c \in \CCB$, we set 
$Z_c=Z/\CG_c Z$. Let $P$ \indexnot{P}{P}  denote the $\kb[\CCB]$-algebra 
obtained by tensor product of algebras $P=\kb[\CCB] \otimes \kb[V]^W \otimes \kb[V^*]^W$. We identify $P$ with a $\kb[\CCB]$-submodule of $\Hb$
via Theorem \ref{PBW-0}.}}{0.75\textwidth}

\bigskip

\subsection{A subalgebra of $Z$} 
The first fundamental result about the center $Z$ of $\Hb$ is the next 
one~\cite[Proposition~4.15]{EG} (we follow
\cite[Proposition~3.6]{gordon} for the proof).

\bigskip

\begin{lem}\label{P stable}
$P$ is a subalgebra of $Z$ stable under the action of 
$\kb^\times \times \kb^\times \times (W^\wedge \rtimes \NC)$. In particular, 
it is $(\NM \times \NM)$-graded.
\end{lem}

\begin{proof}
The subalgebra $\kb[V]^W$ is central in $\Hb$
by Corollary \ref{coro:commutateur-t}.
Dually, $\kb[V^*]^W$ is central as well. The stability property is clear.
\end{proof}

\bigskip

\begin{coro}\label{coro:P-libre}
The PBW-decomposition is an isomorphism of $P$-modules. In particular, we have isomorphisms of 
$P$-modules:
\begin{itemize}
\itemth{a} $\Hb \simeq \kb[\CCB] \otimes \kb[V] \otimes \kb W \otimes \kb[V^*]$.

\itemth{b} $\Hb e \simeq \kb[\CCB] \otimes \kb[V] \otimes \kb[V^*]$.

\itemth{c} $e\Hb e \simeq \kb[\CCB] \otimes \kb[V \times V^*]^{\Delta W}$.
\end{itemize}
Hence, $\Hb$ (respectively $\Hb e$, respectively $e\Hb e$) is a free $P$-module 
of rank $|W|^3$ (respectively $|W|^2$, respectively $|W|$). 
\end{coro}

\bigskip

The main theme of this book is the
study the algebra $\Hb$ viewed as a $P$-algebra: 
given $\pG$ a prime ideal of $P$, 
we will be interested in the finite dimensional $k_P(\pG)$-algebra  
$k_P(\pG) \otimes_P \Hb$ (splitting, 
simple modules, blocks, standard modules, decomposition matrix...). Here,
$k_P(\pG)$ is the fraction field of $P/\pG$, cf. Appendix
 \ref{chapter:galois-rappels}.

\bigskip

\begin{rema}\label{rem:P-libre}
Let $(b_i)_{1 \le i \le |W|}$ be a $\kb[V]^W$-basis of $\kb[V]$ and let 
$(b_i^*)_{1 \le i \le |W|}$ be a $\kb[V^*]^W$-basis of $\kb[V^*]$. Corollary~\ref{coro:P-libre} 
shows that $(b_i w b_j^*)_{\substack{1 \le i,j \le |W| \\ w \in W}}$ is a $P$-basis of 
$\Hb$ and that $(b_i b_j^* e)_{1 \le i,j \le |W|}$ is a $P$-basis of $\Hb e$.\finl
\end{rema}

\bigskip

Set 
$$P_\bullet = \kb[V]^W \otimes \kb[V^*]^W.\indexnot{P}{P_\bullet}$$
If $c \in \CCB$, then 
$$P_\bullet \simeq \kb[\CCB]/\CG_c \otimes_{\kb[\CCB]} P = P/\CG_c P .$$
We deduce from Corollary~\ref{coro:P-libre} the next result.

\bigskip

\begin{coro}\label{coro:P-libre-c}
We have isomorphisms of $P_\bullet$-modules:
\begin{itemize}
\itemth{a} $\Hb_c \simeq \kb[V] \otimes \kb W \otimes \kb[V^*]$.

\itemth{b} $\Hb_c e \simeq \kb[V] \otimes \kb[V^*]$.

\itemth{c} $e\Hb_c e \simeq \kb[V \times V^*]^{\Delta W}$.
\end{itemize}
In particular, $\Hb_c$ (respectively $\Hb_c e$, respectively $e\Hb_c e$) is a free 
$P_\bullet$-module of rank $|W|^3$ (respectively $|W|^2$, respectively $|W|$). 
\end{coro}

\bigskip

\subsection{Satake isomorphism} 
It follows from Proposition~\ref{lem:eulertilde} that
\equat\label{eq:euler-central}
\euler \in Z.
\endequat
Given $c \in \CCB$, we denote by $\euler_c$ \indexnot{ea}{\euler_c}  the image of $\euler$ in $\Hb_c$.

\bigskip

The next structural theorem is a cornerstone of the representation theory of $\Hb$. 

\bigskip

\begin{theo}[Etingof-Ginzburg]\label{theo:satake}
The morphism of algebras $Z \longto e\Hb e$, $z \mapsto ze$ is an isomorphism of 
$(\NM \times \NM)$-graded algebras. In particular, $e\Hb e$ is commutative.
\end{theo}

\bigskip

\begin{proof}
Recall (Theorem \ref{EG spherique-1}) that
the map $\pi_e : \Zrm(\Hb) \to \Zrm(e\Hb e)$, $z \mapsto ze$ is an isomorphism
of algebras.
Theorem \ref{thm:eg} shows that 
$\Theta(e\Hb e)=\kb[\CCB] \otimes (e\kb[V^\reg \times V^*]^W)$ and that
$\Theta$ is injective, hence $e\Hb e$ is commutative. The theorem follows.
\end{proof}

\bigskip

\begin{coro}\label{coro:endo-bi}
Let $\CG$ be a prime ideal of $\kb[\CCB]$. 
Let $Z(\CG)=Z/\CG Z$ and $P(\CG)=P/\CG P$. \indexnot{Z}{Z(\CG)}
We have:
\begin{itemize}
\itemth{a}
$Z(\CG)=\Zrm(\Hb(\CG))$.
\itemth{b} The map $Z(\CG) \to e\Hb(\CG) e$, $z \mapsto ze$ is
an isomorphism.
\itemth{c} $\End_{\Hb(\CG)}(\Hb(\CG) e) = Z(\CG)$ and $\End_{Z(\CG)}(\Hb(\CG) e)=
\Hb(\CG)$.

\itemth{d} $\Hb(\CG)=Z(\CG) \oplus e\Hb(\CG) (1-e) \oplus (1-e)\Hb(\CG) e \oplus (1-e)\Hb(\CG)(1-e)$.
In particular, $Z(\CG)$ is a direct summand of the $Z(\CG)$-module $\Hb(\CG)$.

\itemth{e} $Z(\CG)$ is a free $P(\CG)$-module of rank $|W|$. 
\itemth{f} If $\kb[\CCB]/\CG$ is integrally closed, then
$Z(\CG)$ is an integrally closed domain.
%
\end{itemize}
\end{coro}

\bigskip

\begin{proof}
Assertion (b) follows from 
Theorem \ref{theo:satake}. We deduce now (c) from Theorem \ref{EG spherique-0}
and (e) from  Corollary \ref{coro:P-libre}.
We deduce also that
$Z(\CG)(1-e)\cap Z(\CG)=0$. It follows also that
$e\Hb(\CG) e=Z(\CG)e$, hence $e\Hb(\CG) e\subset Z(\CG)+\Hb(\CG) (1-e)$.
The decomposition
$\Hb(\CG)=e\Hb(\CG) e \oplus e\Hb(\CG) (1-e) \oplus (1-e)\Hb(\CG) e \oplus
(1-e)\Hb(\CG)(1-e)$ implies (d). 

The canonical map $\Zrm(\Hb(\CG))\to e\Hb(\CG)e,\ z\mapsto ze$ is injective since
$\Hb(\CG)$ acts faithfully on $\Hb(\CG)e$ by (c). Since
$Z(\CG)$ is a direct summand of $\Hb(\CG)$ contained in
$\Zrm(\Hb(\CG))$, the assertion (a) follows from (b).

\smallskip
The fact that $Z(\CG) \simeq e\Hb(\CG) e$ is an integrally closed domain
follows from the fact that $\grad\,(e \Hb(\CG) e) \simeq (\kb[\CCB]/\CG)
\otimes \kb[V \times V^*]^{\Delta W}$ is an
integrally closed domain (Lemma \ref{le:normalfilt}).
\end{proof}

\medskip
\begin{exemple}
\label{ex:centerc=0}
Recall (Example \ref{exemple zero-1}) that $\Hb_0=\kb[V \times V^*]\rtimes W$.
It follows from Proposition \ref{pr:centerdouble} and Corollary 
\ref{coro:endo-bi} (a)
that $Z_0=\kb[V \times V^*]^W \simeq \kb[(V \times V^*)/W]$.\finl
\end{exemple}

%


\bigskip
\def\pbw{{\mathrm{pbw}}}
\def\eval{{\mathrm{ev}}}

\subsection{Morphism to the center of $\kb[\CCB]W$}\label{subsection:omega}
Let $\pbw : \kb[\CCB] \otimes \kb[V] \otimes \kb W \otimes \kb[V^*] \xrightarrow{\sim} \Hb$ denote
the 
isomorphism given by the PBW-decomposition (see Theorem~\ref{PBW-0}) and let 
$\eval_{0,0} : \kb[\CCB] \otimes \kb[V] \otimes \kb W \otimes \kb[V^*] \to \kb[\CCB]W$ 
denote the $\kb[\CCB]$-linear map defined by
$$\eval_{0,0}(a \otimes f \otimes w \otimes g) = f(0)g(0)a w.$$
We set
$$\Omeb^\Hb = \eval_{0,0} \circ \pbw^{-1}: \Hb \longto \kb[\CCB]W.$$
This is a $W$-equivariant morphism of bigraded $\kb[\CCB]$-modules but 
it is not a morphism of algebras (except if $W=1$). Nevertheless, we have the following result:

\bigskip

\begin{prop}\label{prop:omega-mult}
If $z \in Z$ and $h \in \Hb$, then $\Omeb^\Hb(z) \in \Zrm(\kb[\CCB] W)$ and 
$$\Omeb^\Hb(zh)=\Omeb^\Hb(z)\Omeb^\Hb(h).$$
\end{prop}

\bigskip

\begin{proof}
First, the fact that $\Omeb^\Hb(z) \in \Zrm(\kb[\CCB] W)$ follows from the $W$-equivariance 
of $\Omeb^\Hb$. Now, 
write $\pbw^{-1}(z) = \sum_{i \in I} a_i \otimes f_i \otimes w_i \otimes g_i$ 
and $\pbw^{-1}(h) = \sum_{j \in J} a_j' \otimes f_j' \otimes w_j' \otimes g_j'$. 
We have 
$$zh=\sum_{j \in J} z a_j' f_j'  w_j'  g_j'=\sum_{j \in J} a_j' f_j' z w_j'  g_j'$$
hence
$$\pbw^{-1}(zh)=\sum_{\substack{i \in I \\ j \in J}} a_ia_j' \otimes f_if_j' \otimes w_i w_j' 
\otimes w_j^{\prime -1}{g_i} g_j'w_j'.$$
This proves that $\Omeb^\Hb(zh)=\Omeb^\Hb(z)\Omeb^\Hb(h)$. 
\end{proof}

\bigskip

Let $\Omeb : Z \to \Zrm(\kb[\CCB] W)$ denote the restriction of $\Omeb^\Hb$ to $Z$. Since 
$\Omeb$ respects the bigrading and so respects the $\ZM$-grading, we have 
\equat\label{eq:omega-z0}
\Omeb(z)=0
\endequat
if $z \in Z$ is $\ZM$-homogeneous of non-zero $\ZM$-degree. 

\bigskip

\begin{coro}\label{coro:omega-mult}
The map $\Omeb : Z \to \Zrm(\kb[\CCB] W)$ is a morphism of bigraded $\kb[\CCB]$-algebras. 
\end{coro}

\bigskip

If $\CG$ is a prime ideal of $\kb[\CCB]$ (resp. if $c \in \CCB$), the map $\Omeb$ 
induces a morphism of $\ZM$-graded algebras $\Omeb^\CG : Z(\CG) \to (\kb[\CCB]/\CG) \otimes \Zrm(\kb W)$ 
(resp. $\Omeb^c : Z_c \to \Zrm(\kb W)$). 

\bigskip

\begin{rema}\label{rem:v-v*}
By exchanging the roles of $V$ and $V^*$, one gets an isomorphism 
$\pbw_* : \kb[\CCB] \otimes \kb[V^*] \otimes \kb W \otimes \kb[V] \to \Hb$ coming 
from the PBW-decomposition and a map 
$\eval_{0,0}^* : \kb[\CCB] \otimes \kb[V^*] \otimes \kb W \otimes \kb[V] \to \kb[\CCB]W$ 
obtained by evaluating at $(0,0) \in V^* \times V$. One obtains another morphism 
of bigraded $\kb[\CCB]$-algebras 
$$\Omeb^* : Z \longto \Zrm(\kb[\CCB] W).$$
It turns out that $\Omeb \neq \Omeb^*$. Indeed, using the automorphism 
$\e_*$
of the $\kb[\CCB]$-algebra $\kb[\CCB] W$ given by $w \mapsto \e(w) w$, then 
it follows from Section~\ref{section:eulertilde} that
$$\Omeb(\euler) = \sum_{s \in \REF(W)} \e(s) C_s s = \e_*(\Omeb^*(\euler)).$$
We will see in Corollary~\ref{coro:v-v*} that
$$\Omeb(z)=\e_*(\Omeb^*(z))$$
for all $z \in Z$.\finl
\end{rema}

\bigskip

\section{Localization}\label{sec:morita}

Recall that 
$$V^\reg=V \setminus \bigcup_{H \in \AC} H=\{v \in V~|~\Stab_W(v)=1\}.$$
Set $P^\reg=\kb[\CCB] \otimes \kb[V^\reg]^W \times \kb[V^*]^W$ \indexnot{P}{P^\reg} 
and $Z^\reg=P^\reg \otimes_P Z$ \indexnot{Z}{Z^\reg}, so that
$\Hb^\reg=P^\reg \otimes_P \Hb=Z^\reg \otimes_Z \Hb$. 
Given $s \in \REF(W)$, let $\a_s^W=\prod_{w \in W} w(\a_s) \in P$. \indexnot{az}{\a_s^W}
We have $P^\reg=P[\delta^{-1}]$ and $Z^\reg=Z[\delta^{-1}]$.
As a consequence,
\equat\label{inversible alpha}
\text{\it $\a_s$ is invertible in $\Hb^\reg$.}
\endequat

\begin{coro}\label{center reg}
$\Th$ restricts to an isomorphism of $\kb[\CCB]$-algebras 
$Z^\reg \longiso \kb[\CCB] \otimes \kb[V^\reg \times V^*]^W$. In particular, 
$Z^\reg$ is regular. 
\end{coro}

\begin{proof}
The first statement follows from the comparison between the centers of 
$\Hb^\reg$ and $\kb[V^\reg \times V^*] \rtimes W$ (Theorem \ref{thm:eg}).
The second statement follows from the fact that 
$W$ acts freely on $V^\reg \times V^*$.
\end{proof}

\bigskip

Given $c \in \CCB$, let $Z^\reg_c=Z_c[\delta^{-1}]$.
\indexnot{Z}{Z_c^\reg}
Corollary~\ref{center reg} shows that 
\equat\label{lisse reg}
\text{\it $Z^\reg_c \simeq \kb[V^\reg \times V^*]^W$ is a regular ring.}
\endequat

\bigskip

\section{Geometry}\label{section:geometrie-zp}

\medskip

The spectra of the
$\kb$-algebras $P$, $Z$, $P_\bullet$ and $Z_c$ are
algebraic $\kb$-varieties that will be denoted by 
$\PCB$, $\ZCB$, $\PCB_{\!\!\!\bullet}$ and $\ZCB_c$ respectively. 
\indexnot{P}{\PCB,~\PCB_{\!\!\!\bullet}}\indexnot{Z}{\ZCB,~\ZCB_c}  
Note that
$$\PCB=\CCB \times V/W \times V^*/W\quad\text{and}\quad \PCB_{\!\!\!\bullet}=V/W \times V^*/W$$
$$\ZCB_0=(V \times V^*)/\Delta W.\leqno{\text{and that}}$$
It follows from Corollary~\ref{coro:endo-bi}(f) that 
\equat\label{irr normale}
\text{\it the varieties $\ZCB$ and $\ZCB_c$ are irreducible and normal.}
\endequat
Note that this statement holds over any base field. 
The inclusions $P \subset Z$ and $P_\bullet \subset Z_c$ 
induce morphisms of varieties 
$$\Upsilon : \ZCB \longto \PCB=\CCB\times V/W \times V^*/W \indexnot{uz}{\Upsilon,~\Upsilon_c}$$
$$\Upsilon_c : \ZCB_c \longto \PCB_{\!\!\!\bullet}=V/W \times V^*/W.\leqno{\text{and}}$$
The surjective maps $P \to P/\CG_cP \simeq P_\bullet$  and $Z \to Z_c$ induce 
closed immersions $j_c : \ZCB_c \longinjto \ZCB$ and 
$i_c : \PCB_{\!\!\!\bullet} \longinjto \PCB$, \indexnot{ia}{i_c,~j_c}  
$p \mapsto (c,p)$. Moreover, the canonical injective map  
$\kb[\CCB] \longinjto P$ induces the canonical projection 
$\pi : \PCB \to \CCB$ \indexnot{pz}{\pi}  and, in the diagram 
\equat\label{diagramme geometrie}
\diagram
&\ZCB_c \ar@{^{(}->}[rr]^{\DS{j_c}} \ddto_{\DS{\Upsilon_c}} && \ZCB \ddto^{\DS{\Upsilon}}& \\
&&&&\\
V/W \times V^*/W \ar@{=}[r] & \PCB_{\!\!\!\bullet} 
\ar@{^{(}->}[rr]^{\DS{i_c}} \ddto && \PCB \ddto^{\DS{\pi}} \ar@{=}[r]& 
\CCB \times V/W \times V^*/W \\
&&\\
&\{c\} \ar@{^{(}->}[rr] && \CCB =\Ab^{\refw}&
\enddiagram
\endequat
all the squares are cartesian. Note also that, 
by Corollary~\ref{coro:endo-bi},
\equat\label{platitude}
\text{\it the morphisms $\Upsilon$ and $\Upsilon_c$ are finite and flat.}
\endequat
Moreover,
\equat\label{eq:pi-plat}
\text{\it $\pi$ is smooth,}
\endequat
since $V/W \times V^*/W$ smooth.

\bigskip

\begin{exemple}\label{upsilon c=0}
We have $\ZCB_0 = (V \times V^*)/W$ and 
$\Upsilon_0 : (V \times V^*)/W = \ZCB_0 \to \PCB_\bullet = V/W \times V^*/W$ 
is the canonical morphism.\finl 
\end{exemple}

\bigskip

Let $\ZCB^\reg$ \indexnot{Z}{\ZCB^\reg} denote the open subset $\Spec(Z^\reg)$ of $\ZCB$.  
Corollary~\ref{center reg} shows that 
\equat\label{eq:zreg}
\text{\it $\ZCB^\reg \simeq \CCB\times (V^\reg \times V^*)/W$ is smooth.}
\endequat
Let $\ZCB_\singulier$  \indexnot{Z}{\ZCB_\singulier}  denote the singular locus 
of $\ZCB=\Spec(Z)$ and $\zG_\singulier$ 
\indexnot{Z}{\zG_\singulier}  
its defining ideal. Since $Z$ is integrally closed, it follows that
\equat\label{codimension deux Q}
\text{\it $\ZCB_\singulier$ has codimension $\ge 2$ in $\ZCB$.}
\endequat
Of course, $\zG_\singulier$ needs not be a prime ideal.
Since $\Upsilon : \ZCB \to \PCB$ is finite and flat, we deduce that 
\equat\label{codimension deux P}
\text{\it $\Upsilon(\ZCB_\singulier)$ is closed and of codimension $\ge 2$ in $\PCB$.}
\endequat
The defining ideal of $\Upsilon(\ZCB_\singulier)$ is $\sqrt{\zG_\singulier \cap P}$. 
\begin{rema}
	Note that the condition that 
$\zG_\singulier \cap P \not\subset \pG$ of Proposition \ref{morita lissite} and
	Corollary \ref{morita lissite p} is
equivalent to the fact
that $\Spec(P/\pG)$ is not contained in $\Upsilon(\ZCB_\singulier)$.
\end{rema}
\bigskip

\section{Morita equivalences}
While $Z$ and $\Hb$ are only related by a double endomorphism theorem,
after restricting to a smooth open subset of $Z$, they become Morita
equivalent.

\smallskip
Let us first state a consequence of Corollary  \ref{coro:endo-bi}.

\begin{lem}
	\label{le:critereMorita}
	Let $R$ be a commutative $Z$-algebra. The following assertions are
	equivalent
	\begin{itemize}
		\item $\Hb e\otimes_Z R$ is a projective $R$-module
		\item $\Hb e\otimes_Z R$ induces a Morita equivalence between
			$\Hb \otimes_Z R$ and $R$.
	\end{itemize}
\end{lem}

\begin{prop}\label{prop:morita}
Let $U$ be a multiplicative subset of $Z$ such that $Z[U^{-1}]$ is regular.
Then $\Hb[U^{-1}]e$ induces a Morita equivalence between
the algebras $\Hb[U^{-1}]$ and $Z[U^{-1}]$.
\end{prop}

\begin{proof}
Let $\mG$ be a maximal ideal of $Z$ such that $Z_\mG$ is regular. Let
$i$ be maximal such that $\Tor_i^Z(\Hb e,Z/\mG)\not=0$. Given any
finite length $Z$-module $L$ with support $\mG$, we have
$\Tor_i^Z(\Hb e,L)\not=0$.

Let $\pG=P\cap\mG$.
We have $\Tor_*^Z(\Hb e,Z/\pG Z)\simeq
\Tor_*^Z(\Hb e,Z\otimes_P P/\pG)
\simeq\Tor_*^P(\Hb e,P/\pG)$ since $Z$ is
a free $P$-module (Corollary \ref{coro:endo-bi}). Since $\Hb e$ is a
free $P$-module (Corollary \ref{coro:P-libre}), it follows that
$\Tor_{>0}^Z(\Hb e,Z/\pG Z)=0$, hence
$\Tor_{>0}^Z(\Hb e,(Z/\pG Z)_\mG)=0$. We deduce that $i=0$, hence 
$(\Hb e)_\mG$ is a free $Z_\mG$-module.

We have shown that $\Hb[U^{-1}]e$ is a projective $Z[U^{-1}]$-module.
The Morita equivalence follows from Lemma \ref{le:critereMorita}.
\end{proof}

\begin{coro}\label{coro:reg-morita}
The $(\Hb^\reg,Z^\reg)$-bimodule $\Hb^\reg e$ induces a Morita equivalence 
between $\Hb^\reg$ and $Z^\reg$. 
\end{coro}

\begin{proof}
This follows from Proposition~\ref{prop:morita} and Corollary~\ref{center reg}.
\end{proof}


\medskip

\bigskip

\begin{prop}\label{morita lissite}
	Let $\pG$ be a prime ideal of $P$ that does not contain
	$\zG_\singulier \cap P$.

The $(\Hb_\pG,Z_\pG)$-bimodule $\Hb_\pG e$ is both left and right projective and induces 
a Morita equivalence between $\Hb_\pG$ and $Z_\pG$. 
\end{prop}

\begin{proof}
By assumption, there exists $p \in \zG_\singulier \cap P$ such that $p \not\in \pG$. 
So $Z[1/p]=P[1/p] \otimes_P Z \subset Z_\pG= P_\pG \otimes_P Z$ and $\Spec(Z[1/p])$ is 
regular. It follows from Proposition~\ref{prop:morita} that 
$P[1/p] \otimes_P \Hb$ and $Q[1/p]$ are Morita equivalent 
via the bimodule $P[1/p] \otimes_P \Hb e$. The proposition follows
by scalar extension.
\end{proof}
\bigskip

By reducing modulo $\pG$, one gets the following consequence.

\bigskip

\begin{coro}\label{morita lissite p}
	Let $\pG$ be a prime ideal of $P$ that does not contain
	$\zG_\singulier \cap P$.

The $(k_P(\pG)\Hb,k_P(\pG)Z)$-bimodule $k_P(\pG)\Hb e$ is both left and right projective and induces 
a Morita equivalence between $k_P(\pG)\Hb$ and $k_P(\pG)Z$. 
\end{coro}

\bigskip

\begin{exemple}\label{exemple lisse}
Taking Corollary~\ref{center reg} into account, the condition $\zG_\singulier \cap P \not\subset \pG$ 
is satisfied if $\Spec(P/\pG)$ meets the open subset 
$\PCB^\reg = \CCB \times V^\reg/W \times V^*/W$ or, by symmetry, the open 
subset $\CCB \times V/W \times V^{* \reg}/W$.\finl
\end{exemple}

\bigskip

Let $\Kb$ \indexnot{K}{\Kb} denote the fraction field of $P$ and let $\Kb Z = \Kb \otimes_P Z$. 
Since $Z$ is a domain and is integral over $P$, it follows that
\equat\label{eq:KZ-fraction}
\text{\it $\Kb Z$ is the fraction field of $Z$.}
\endequat
In particular, $\Kb Z$ is a regular ring.

\bigskip

\begin{theo}\label{theo:KH-mat}
The $\Kb$-algebras $\Kb\Hb$ and $\Kb Z$ are Morita equivalent, 
the Morita equivalence being induced by $\Kb \Hb e$. More precisely, 
$$\Kb\Hb \simeq \Mat_{|W|}(\Kb Z).$$
	Furthermore, $\Kb\Hb e$ is a free $(\Kb Z)W$-module of rank $1$.
\end{theo}

\begin{proof}
Proposition~\ref{prop:morita} shows the Morita equivalence.
	Since $\kb(V\times V^*)$ is a free $\kb(V\times V^*)^WW$-module of rank $1$ and
	since $\Theta(\Kb\Hb e)\simeq \kb[\CCB](V\times V^*)$, it follows from
	Corollary \ref{center reg} that $\Kb\Hb e$ is a free $\Kb Z W$-module of rank $1$.
In particular, $\Kb\Hb e$ is a $\Kb Z$-vector space of dimension $|W|$, whence the isomorphism
	of $\Kb\Hb$ with a matrix algebra.
\end{proof}

\bigskip

\section{Complements}

\bigskip

\subsection{Poisson structure}

The PBW-decomposition induces an isomorphism of {\it $\kb$-vector spaces} 
$\kb[T] \otimes \Hb \longiso \Hbt$. Given $h \in \Hb$, let $\tilde{h}$ denote the image of 
$1 \otimes h \in \kb[T] \otimes \Hb$ in $\Hbt$ through this isomorphism. 
If $z$, $z' \in Z$, then $[z,z']=0$, hence
$[\zti,\zti'] \in T \Hbt$. We denote by $\{z,z'\}$ the image of $[\zti,\zti']/T \in \Hbt$ in  
$\Hb = \Hbt/T\Hbt$. It is easily checked that $\{z,z'\} \in Z$ and that 
\equat\label{eq:poisson}
\{-,-\} : Z \times Z \longto Z\indexnot{ZZZ}{{{\{,\}}}}
\endequat
is a $\kb[\CCB]$-linear {\it Poisson bracket}. Given $c \in \CCB$,
it induces a Poisson bracket
\equat\label{eq:poisson-c}
\{-,-\} : Z_c \times Z_c \longto Z_c.
\endequat

\bigskip

\subsection{$\Hbt$ as a Rees algebra}
\label{sub:Rees}

Let us define $\Hbd=\Hbt\otimes_{\kb[T]}\bigl(
\kb[T]/(T-1)\bigr)$\indexnot{Hd}{\Hbd}, an algebra over $\kb[\CCBt]/(T-1)$
(we identify that ring with $\kb[\CCB]$).

Given $c \in \CCB$, we set
$\Hbd_c =  \Hbt_{1,c}$\indexnot{Hdc}{\Hbd_c}.

\medskip

We define a $\kb$-algebra filtration of 
$\Hbd$\indexnot{Ht}{\Hbd^{\trianglelefteq}}
$$\Hbd^{\trianglelefteq -1}=0,\
\Hbd^{\trianglelefteq 0}=\kb[V]\rtimes W,\
\Hbd^{\trianglelefteq 1}=\Hbd^{\trianglelefteq 0}V+\Hbd^{\trianglelefteq 0}\CCB^*+\Hbd^{\trianglelefteq 0}$$
$$\text{ and }\Hbd^{\trianglelefteq i}= (\Hbd^{\trianglelefteq 1})^i\text{ for }i\ge 2.$$

The canonical maps $\kb[V]\rtimes W\to 
(\grad^{\trianglelefteq}\Hbd)^0$ and
$V\oplus\CCB^*\to (\grad^{\trianglelefteq}\Hbd)^1$ induce a morphism
of $\NM$-graded
algebras $g:\Hb\to\grad^{\trianglelefteq}\Hbd$

The PBW decomposition (Theorem \ref{PBW-1}) shows the following result.

\begin{prop}
The morphism $g$ is an isomorphism.
\end{prop}

Note that this proposition shows that $\Hbt$ is the Rees algebra
of $\Hbd$ for this filtration.

\bigskip

\subsection{Additional filtration}\label{sub:filtration}
\medskip

Define a $P$-algebra filtration of $\Hb$ by 
$$\Hb^{\preccurlyeq -1}=0,\
\Hb^{\preccurlyeq 0}=P[W],\ \Hb^{\preccurlyeq 1}=\Hb^{\preccurlyeq 0}+
\Hb^{\preccurlyeq 0}V+\Hb^{\preccurlyeq 0}V^* \text{ and }
\Hb^{\preccurlyeq i}=\Hb^{\preccurlyeq 1}\Hb^{\preccurlyeq i-1}
\text{ for } i\ge 2.$$
Note that $\Hb^{\preccurlyeq 2N-1}{\not=}\Hb$ and
$\Hb^{\preccurlyeq 2N}=\Hb$.

Let $V'$ be the $\kb W$-stable complement to $V^W$ in $V$.
We have an injection of $P$-modules
$P\otimes (V^{\prime *}\oplus V')\otimes \kb[W]\hookrightarrow
 \Hb^{\preccurlyeq 1}$.
It extends to a morphism of graded $P$-algebras
$$f:P\otimes (\kb[V']^\cow\otimes \kb[V^{\prime *}]^\cow)\rtimes W\to
\mathrm{gr}^{\preccurlyeq}\Hb$$
where $P$ and $W$ are in degree $0$ and $V^{\prime *}$ and $V'$ are
in degree $1$.

\begin{prop}
\label{pr:filttilde}
The morphism $f$ is an isomorphism of graded $P$-algebras.
\end{prop}

\begin{proof}
This follows from the PBW decomposition (Corollary \ref{coro:P-libre}).
\end{proof}

\bigskip

\subsection{Symmetrizing form}\label{sub:symetrisante}
\medskip
Recall (Proposition \ref{pr:symmkV}) that we have 
symmetrizing forms $p_N:\kb[V]\to \kb[V]^W$ and
$p_N^*:\kb[V^*]\to \kb[V^*]^W$.

We define a $P$-linear map
\begin{align*}
\taub_\HHb : \Hb=\kb[\CCB]\otimes \kb[V]\otimes \kb W\otimes \kb[V^*] &
\longto  P \indexnot{tx}{\taub_\HHb} \\
a\otimes b\otimes w\otimes c &\longmapsto  a\delta_{1w}p_N(b)p_N^*(c).
\end{align*}

\begin{theo}[{\cite[Theorem 4.4]{BGS}}]
\label{symetrique}
The form $\taub_\HHb$ is symmetrizing for the $P$-algebra $\Hb$.
\end{theo}

\begin{proof}
We have an isomorphism
\begin{align*}
(\mathrm{gr}^{\preccurlyeq}\Hb)^{2N} &\longiso PW\\
\kb[\CCB]\otimes \kb[V]\otimes \kb W\otimes \kb[V^*]\ni 
a\otimes b\otimes w\otimes c & \longmapsto p_N(b) aw p_N^*(c).
\end{align*}
Via the isomorphism of Proposition \ref{pr:filttilde},
the $P$-linear form on $\mathrm{gr}^{\preccurlyeq}\Hb$
induced by $\taub_\HHb$ is given by 
	$$P\otimes (\kb[V']_W\otimes \kb[V^{\prime *}]_W)\rtimes W\ni
a\otimes (b\otimes c)\otimes w\mapsto a\delta_{1w}p_N(b) p_N^*(c).$$
It follows from Lemma \ref{le:symmtwisted} that this is a symmetrizing form.

Let $L=V'\oplus V^{\prime *}$.
We have $S^{N+1}(V')\subset S(V')^W_{>0}\cdot S^{\le N-1}(V')$ and
	$S^{\ge N+2}(V')\subset S(V')^W_{>0}\cdot S^{\le N}(V')$
(and similarly with $V^{\prime *}$), hence
$$L^{2N+1}\subset \Hb^{\preccurlyeq 2N-1}.$$
It follows from Lemma \ref{le:trace} that $\tau_\HHb$ is a trace.
We deduce from
Proposition \ref{pr:filtsymm}  that $\taub_\HHb$ is symmetrizing.
\end{proof}

\begin{rema}
Note that while there is no canonical isomorphism
$\kb[V]^\cow_N\longiso \kb$, there is a canonical choice of isomorphism
$\kb[V]^\cow_N\otimes_{\kb}\kb[V^*]^\cow_N\longiso \kb$ obtained by
requiring $\langle\alpha_s^\vee,\alpha_s\rangle=1$ for all $s\in\REF(W)$.
This provides a canonical trace $\taub_\HHb$.\finl
\end{rema}

We denote by $\casimir_\HHb\in Z$ the central Casimir element of $\Hb$ (cf. \S \ref{se:symmalg}).


%
%

\bigskip

\subsection{Hilbert series}
We compute here the bigraded Hilbert series of 
$\Hb$, $P$, $Z$ and $e\Hb e$. First of all, note that 
$$\dim_\kb^{\BZ\times\BZ}(\kb[\CCB]) = \frac{1}{(1-\tb\ub)^{|\refw|}},$$
so that it becomes easy to deduce the Hilbert series for $\Hb$, 
using the PBW-decomposition:
\equat\label{hilbert H}
\dim_\kb^{\BZ\times\BZ}(\Hb) = \frac{|W|}{(1-\tb)^n~(1-\ub)^n
~ (1-\tb\ub)^{|\refw|}}.
\endequat
On the other hand, using the notation of Theorem~\ref{chevalley}, 
we get, thanks to~(\ref{hilbert biinv}), 
\equat\label{hilbert P}
\dim_\kb^{\BZ\times\BZ}(P)=
\frac{1}{(1-\tb\ub)^{|\refw|} ~\DS{\prod_{i=1}^n (1-\tb^{d_i})(1-\ub^{d_i})}}.
\endequat
Finally, note that the PBW-decomposition is a $W$-equivariant isomorphism 
of bigraded $\kb[\CCB]$-modules, from which we deduce that 
$\Hb e \simeq \kb[\CCB] \otimes \kb[V] \otimes \kb[V^*]$ as bigraded  
$\kb W$-modules. So
\equat\label{iso bigrad}
\text{\it the bigraded $\kb$-vector spaces $Z$ and 
$\kb[\CCB] \otimes \kb[V \times V^*]^{\Delta W}$ are isomorphic.}
\endequat
We deduce that $\dim_\kb^{\BZ\times\BZ}(Z)=\dim_\kb^{\BZ\times\BZ}(\kb[\CCB]) \cdot 
\dim_\kb^{\BZ\times\BZ}(\kb[V \times V^*]^{\Delta W})$. By
(\ref{hilbert molien}) and Proposition~\ref{dim bigrad Q0 fantome},
we get 
\equat\label{hilbert Q molien}
\dim_\kb^{\BZ\times\BZ}(Z) = \frac{1}{|W|~(1-\tb\ub)^{|\refw|}} 
\sum_{w \in W} \frac{1}{\det(1-w\tb)~\det(1-w^{-1}\ub)}
\endequat
and
\equat\label{hilbert Q fantome}
\dim_\kb^{\BZ\times\BZ}(Z) = 
\frac{\DS{\sum_{\chi \in \Irr(W)} 
f_\chi(\tb)~f_\chi(\ub)}}{(1-\tb\ub)^{|\refw|} ~\DS{\prod_{i=1}^n (1-\tb^{d_i})(1-\ub^{d_i})}}.
\endequat

\bigskip

\section{Special features of Coxeter groups}\label{section:coxeter-h}

\cbstart

\boitegrise{{\bf Assumption.} 
{\it In this section \S\ref{section:coxeter-h}, we assume that 
$W$ is a Coxeter group, and we use the notation of \S\ref{se:Coxetergroups}.}}{0.75\textwidth}

\bigskip

In relation with the aspects studied in this chapter, one of the features 
of the situation is that the algebra $\Hb$ admits another automorphism $\s_\Hb$, \indexnot{sz}{\s_\Hb}  
induced by the isomorphism of $W$-modules $\s : V \longiso V^*$. It is the 
reduction modulo $T$ of the automorphism $\s_\Hbt$ of $\Hbt$ defined in 
\S\ref{section:coxeter-htilde}. Propositions~\ref{prop:auto-coxeter-1} 
and~\ref{prop-bis:auto-coxeter-1} now becomes the following.

\bigskip

\begin{prop}\label{prop:auto-coxeter-0}
There exists a unique automorphism $\s_\Hb$ of $\Hb$ such that 
$$
\begin{cases}
\s_\Hb(y)=\s(y) & \text{if $y \in V$,}\\
\s_\Hb(x)=-\s^{-1}(x) & \text{if $x \in V^*$,}\\
\s_\Hb(w)=w & \text{if $w \in W$,}\\
\s_\Hb(C)=C & \text{if $C \in \CCB^*$.}\\
\end{cases}
$$
\end{prop}

\bigskip

\begin{prop}\label{prop-bis:auto-coxeter-0}
We have the following statements:
\begin{itemize}
\itemth{a} $\s_\Hb$ stabilizes the subalgebras $\kb[\CCB]$ and $\kb W$ and exchanges 
the subalgebras $\kb[V]$ and $\kb[V^*]$.

\itemth{b} Given $h \in \Hb^{\NM \times \NM}[i,j]$, 
we have $\s_\Hb(h) \in \Hb^{\NM \times \NM}[j,i]$.

\itemth{c} Given $h \in \Hb^\NM[i]$ (respectively $h \in \Hb^\BZ[i]$), 
we have $\s_\Hb(h) \in \Hb^\NM[i]$ (respectively $h \in \Hb^\BZ[-i]$).

\itemth{d} $\s_\Hb$ commutes with the action of $W^\wedge$ on $\Hb$.

\itemth{e} Given $c \in \CCB$, then $\s_\Hb$ induces an automorphism 
of $\Hb_{c}$, still denoted by $\s_\Hb$ (or $\s_{\Hb_{c}}$ if necessary).

\itemth{f} $\s_\Hb(\euler)=- \euler$.
\end{itemize}
\end{prop}

Similarly, there exists an action of $\Gb\Lb_2(\kb)$ on $\Hb$, 
which is obtained by reduction modulo $T$ of the action on $\Hbt$ defined in 
Remark~\ref{rem:sl2-1}.

\cbend

\bigskip

\part{Families and cellular characters}\label{part:reps}

\chapter{Representations}
\label{se:representations}

\section{Highest weight categories}
We define highest weight category structures on categories of (graded)
representations of $\Hbt$, following Appendix
\ref{se:Appendixtriangular}. The existence of such structures for
the restricted Cherednik algebras (cf. \S \ref{chapter:bebe-verma}) is due
to Bellamy and Thiel \cite{BelTh1}.

\medskip

We consider the $\ZM$-grading on $A=\Hbt$ and take $k=\kb[\CCBt]$.
We have three graded $k$-subalgebras
$B_-=\kb[\CCBt\times V^*]$,
$B_+=\kb[\CCBt\times V]$ and $H=\kb[\CCBt]W$ of $A$.
Theorem \ref{PBW-1} and Example \ref{ex:Hsplit} show that
the conditions ($\bigtriangleup$i)-($\bigtriangleup$ix) of Appendix \ref{se:Appendixtriangular} are satisfied
with $I=\{\kb[\CCBt]\otimes E\}_{E\in\Irr(\kb W)}$.

\medskip

We put $\Hbt_-=\kb[\CCBt\times V^*]\rtimes W$ and
$\Hbt_+=\kb[\CCBt\times V]\rtimes W$.
Given $E$ a graded $(\kb[V^*]\rtimes W)$-module, we define the Verma module
$$\tilde{\Delta}(E)=\Ind_{\Hbt_-}^{\Hbt}(\kb[\CCBt]\otimes E)=
\Hbt\otimes_{\Hbt_-}(\kb[\CCBt] \otimes E).
\indexnot{Deltat}{\tilde{\Delta}(E)}$$
Note that the functor $\tilde{\D}$ is exact, since $\Hbt$ is a free 
$\Hbt_-$-module, by the PBW decomposition, which also 
gives a canonical isomorphism of $\Hbt_+$-modules
\equat
\label{dimension left}
\kb[\CCBt\times V]\otimes E\longiso \tilde{\Delta}(E)
\endequat
where $W$ acts diagonally and $\kb[\CCBt\times V]$ acts by left multiplication
on $\kb[\CCBt\times V]\otimes E$.
We will view $\kb W$-modules as graded $(\kb[V^*]\rtimes W)$-modules
concentrated
in degree $0$ by letting $V$ act by $0$.


\medskip
Let $R$ be a noetherian commutative $\kb[\CCBt]$-algebra.
Let $\tilde{\OC}(R)$\indexnot{Ot}{\tilde{\OC}(R)}
be the category of finitely generated $\ZM$-graded 
$(R\otimes_{\kb[\CCBt]}\Hbt)$-modules that are locally nilpotent for the
action of $V$.

\smallskip
Theorem \ref{th:hwgraded} provides the following result.

\begin{theo}
The category $\tilde{\OC}(R)$ is a highest weight category
over $R$ with set of
standard modules $\{R\tilde{\Delta}(E)\langle i\rangle\}_{E\in\Irr(\kb W),
i\in\BZ}$ and partial order $R\tilde{\Delta}(E)\langle i\rangle<
\tilde{\Delta}(F)\langle j\rangle$ if $i<j$.
\end{theo}

\smallskip
We put $\tilde{\Delta}(\mathrm{co})=\tilde{\Delta}(\kb[V^*]^\cow)=
\Hbt e\otimes_{\kb[V^*]^W}\kb$.\indexnot{Deltaco}{\tilde{\Delta}(\mathrm{co})}

\begin{lem}
\label{le:classDeltaco}
We have
$[\tilde{\Delta}(\mathrm{co})]=\sum_{E\in\Irr(W)}f_E(\tb^{-1})
[\tilde{\Delta}(E)]$ in $K_0(\tilde{\OC}(\kb[\CCBt]))$.
\end{lem}

\begin{proof}
The $(\kb[V^*]\rtimes W)$-module
$M=\kb[V^*]^\cow$ has a finite filtration given by $M^{\le i}=
(\kb[V^*]^\cow)^{\le i}$ and the associated graded module is
isomorphic to $\bigoplus_{E\in\Irr(W)}f_E(\tb^{-1})E$.
Consequently, we have a filtration of the $\Hbt$-module 
$\tilde{\Delta}(\mathrm{co})$ given by
$\tilde{\Delta}(\mathrm{co})^{\le i}=\tilde{\Delta}((\kb[V^*]^\cow)^{\le i})$
and the associated graded $\Hbt$-module $\mathrm{gr}\
\tilde{\Delta}(\mathrm{co})$
is isomorphic to
$\bigoplus_{E\in\Irr(W)}f_E(\tb^{-1})\tilde{\Delta}(E)$.
The result follows.
\end{proof}

\begin{rema}
Note that, for the filtration introduced in the proof of Lemma
\ref{le:classDeltaco}, the module
 $\mathrm{gr}\ \tilde{\Delta}(\mathrm{co})$ is isomorphic to
$\tilde{\Delta}(\kb W)$, as an ungraded $\Hbt$-module.\finl
\end{rema}

\section{Euler action on Verma modules}
The classical formula 
describing the action of the Euler element on Verma modules is given in the 
proposition below (cf \cite[\S 3.1(4)]{ggor}).

\smallskip
Given $(\orbite,j) \in \orbiteb^\circ$, $H \in \orbite$ and $E \in \Irr(W)$, we put
$$m_{\orbite,j}^E = \langle \Res_{W_H}^W E, {\det}^j \rangle_{W_H}\indexnot{ma}{m_{\orbite,j}^\chi}$$
$$C_E=\frac{1}{\dim_\kb E}\sum_{s \in \REF(W)} \e(s)\Tr(s,E)~C_s.\leqno{\text{and}}\indexnot{ca}{C_E}$$

\begin{lem}
\label{le:cchi}
We have
$$C_E=\sum_{(\orbite,j) \in \orbiteb^\circ} 
\frac{m_{\orbite,j}^E|\orbite|e_\orbite}{\dim_\kb E}\cdot K_{\orbite,j}
\in\bigoplus_{(\orbite,j) \in \orbiteb^\circ} \ZM_{\ge 0}K_{\orbite,j}.$$
\end{lem}

\begin{proof}
We have
$$C_E=\sum_{(\orbite,j) \in \orbiteb^\circ}\sum_{H\in\orbite} \frac{1}{\dim_\kb E} \Tr(e_H\varepsilon_{H,j},E)
K_{\orbite,j}.$$
Now, given $(\orbite,j) \in \orbiteb^\circ$,
the central element $\sum_{H\in\orbite}e_H\varepsilon_{H,j}$ of $\kb W$ acts on
$E$ by the scalar $\frac{m_{\orbite,j}^E|\orbite|e_\orbite}{\dim_\kb E}$. The lemma follows.
\end{proof}

\begin{prop}
\label{pr:OmegaEuler}
Let $E\in\Irr(W)$. The element
$\eulertilde$ acts on $\tilde{\Delta}(E)_i$ by multiplication by
$Ti+C_E$.
\end{prop}

\begin{proof}
Recall (\S\ref{section:eulertilde}) that if $(x_1,\dots,x_n)$ denotes a 
$\kb$-basis of $V^*$ and if $(y_1,\dots,y_n)$ denotes its dual basis, then 
$$\eulertilde=\sum_{i=1}^n x_i y_i + \sum_{s \in \REF(W)} \e(s)C_ss.$$
Let $h\in (\Hbt_+)_i$, $v\in E$ and $m=h\otimes
 v\in\tilde{\Delta}(E)_i$. We have
$$\eulertilde \cdot m=h\eulertilde \otimes v+Tim=h\otimes (
\sum_{s \in \REF(W)} \e(s)C_s s \cdot v)+Tim.$$
Since $\sum_{s \in \REF(W)} \e(s)C_s s$ acts on $\kb[\CCBt] E$ by
multiplication by $C_E$,
the result follows.
\end{proof}

\begin{rema}
	\label{re:CEproduct}
	Assume $W=W_1\times\cdots\times W_r$ is the decomposition into
	indecomposables. Let $E_i\in\Irr(W_i)$ and $E=E_1\otimes\cdots\otimes
	E_r\in\Irr(W)$. We have $C_E=C_{E_1}+\cdots+C_{E_r}$.\finl
\end{rema}

\bigskip

\section{Case $T=1$}

It follows from Proposition \ref{lem:eulertilde} that the $\ZM$-grading
on $\Hbd$ is inner (i.e. given by the eigenspaces of $[\euler,-]$). 
We put $\dot{\Delta}(E)=
\Hbd\otimes_{\Hbt}\tilde{\Delta}(E)=
\tilde{\Delta}(E)\otimes_{\kb[T]}\bigl(\kb[T]/
<T-1>\bigr)$.\indexnot{Deltap}{\dot{\Delta}(E)}

Let $R$ be a commutative noetherian $\kb[\CCB]$-algebra. We assume that
given any family $E_1,E_2,\ldots,E_{n-1},E_n=E_1\in\Irr(\kb W)$ and given
any family $a_1,\ldots,a_n\in\BZ$
with $a_i\neq a_{i+1}$ and
$(C_{E_i}-C_{E_{i+1}}+a_i-a_{i+1})1_R$ non invertible in $R$ for 
$1\le i\le n-1$, then $a_1=a_n$. Note that this assumption is automatically
satisfied if $R$ is a local ring.

\smallskip
Let $\dot{\OC}(R)$\indexnot{Od}{\dot{\OC}(R)} be the category of finitely
generated $R\Hbd$-modules
that are locally nilpotent for the action of $V$.
Theorem \ref{th:hwinner} shows that this is a highest weight category
 \cite[\S 3]{ggor}.

\begin{theo}
\label{th:OT=1}
$\dot{\OC}(R)$ is a
highest weight category over $R$ with set of standard objects
$\{R\dot{\Delta}(E)\bigr)\}_{E\in\Irr(\kb W)}$.
The order is given by $E>F$ if $(C_E-C_F)1_R\in \ZM_{>0}$.
\end{theo}

\smallskip
Let $\FC_1$\indexnot{F1}{\FC_1} be the set of height one prime ideals $\CG$ of
$\kb[\CCB]$ such that $\dot{\OC}(\kb(\CG))$ is not 
semisimple.
Note that given $\mG$ a maximal ideal of $\kb[\CCB]$, the category
$\dot{\OC}(\kb(\mG))$ is not semisimple if and only if there is $\CG\in\FC_1$
such that $\CG\subset\mG$.

\smallskip

We have the following classical semisimplicity result (an improvement of which will be given in
Corollary \ref{co:semisimplicityO} below, using the KZ functor).

\begin{theo}
\label{th:semisimplicity}
The ideals in $\FC_1$ are of the form
$\langle (C_E-C_F-r)1_R \rangle$ for some 
$E,F\in\Irr(W)$ such that $C_E\neq C_F$
and some $r\in\ZM\setminus\{0\}$. They correspond to affine hyperplanes
in $\KCB(\QM)$.

Assume $R$ is a field and
$(C_E-C_F)1_R{\not\in}\ZM-\{0\}$ for all $E,F\in\Irr(\kb W)$.
Then the category $\dot{\OC}(R)$ is semisimple.

 In particular, if $R$ is a field and
$\QM\cap\bigl(\sum_{(\orbite,j)\in\orbiteb^\circ}\ZM K_{\orbite,j}1_R\bigr)=
\{0\}$, then
$\dot{\OC}(R)$ is semisimple.
\end{theo}

\begin{proof}
Under the assumption on $C_E$'s, the order on $\Irr(W)$ is trivial,
hence $\dot{\OC}(R)$ is semisimple.
The second assertion follows from Lemma \ref{le:cchi}.
\end{proof}

Consider now the case where $R$ is a field  and extend the structure of $\kb[\CCB]$-algebra on
$R$ to a structure of $\kb[\CCBt]$-algebra by setting $T 1_R=1_R$. Consider
a map $\logexp:R\to\ZM$ such that $\logexp(z+n)=\logexp(z)+n$ 
 for all $z\in R$ and $n\in\ZM$.
Proposition \ref{pr:gradedungraded} shows the following.

\begin{prop}
\label{pr:gradednongradedVerma}
Assume $R$ is a field.
There is an equivalence of graded highest weight categories
$$\dot{\OC}(R)^{(\ZM)}\xrightarrow{\sim}\tilde{\OC}(R),\
\dot{\Delta}(E)\mapsto
	\tilde{\Delta}(E)\langle -\logexp(C_E 1_R)\rangle.$$
\end{prop}

\section{Case $T=0$}
\label{se:caseT=0}
\subsection{}
Given $E$ a $\ZM$-graded $(\kb[V^*] \rtimes W)$-module, we put 
$\Delta(E)=\tilde{\Delta}(E)\otimes_{\kb[T]}
\kb[T]/<T>$\indexnot{Delta}{\Delta(E)}: it is a graded 
$\Hb$-module. We define $\Hb^+=\Hbt^+/T\Hbt^+=\kb[\CCB \times V] \rtimes W$ 
and $\Hb^-=\Hbt^-/T\Hbt^- = \kb[\CCB \times V^*] \rtimes W$. We have
$$\D(E)=\Ind_{\Hb^-}^{\Hb} (\kb[\CCB] \otimes E)=\Hb \otimes_{\Hb^-} (\kb[\CCB] \otimes E).$$
Assume $E \in \Irr(W)$. We denote by $\o_E : \Zrm(\kb W) \to
\kb$\indexnot{omegaE}{\o_E}
its associated central character and we set
$$\O_E = (\Id_{\kb[\CCB]} \otimes \o_E) \circ \Omeb : Z \longto \kb[\CCB],
\indexnot{OmegaE}{\O_E}$$
where $\Omeb : Z \to \Zrm(\kb[\CCB] W)$ is the morphism of algebras defined 
in~\S\ref{subsection:omega}. Note that $\O_E$ is a morphism of algebras. 

Recall that $Z^0$ denotes the $\ZM$-homogeneous 
component of $Z$ of degree $0$.

%
%
%

\medskip

\begin{prop}\label{prop:action-z0}
Given $E \in \Irr(W)$, an element $b \in Z^0$ acts on $\D(E)$ by
multiplication by $\O_E(b)$. 
\end{prop}

\bigskip

\begin{proof}
Using the PBW-decomposition, we can write 
$$b=\sum_{i \in I} a_i f_i w_i g_i,$$
where $a_i \in \kb[\CCB]$, $f_i \in \kb[V]$, $w_i \in W$ and $g_i \in \kb[V^*]$. 
Since $b$ is homogeneous of degree $0$, we can choose the $f_i$'s and $g_i$'s to be homogeneous
elements such that
$\deg_\ZM(f_i)+\deg_\ZM(g_i)=0$ for all $i \in I$. 

Let $h \in \Hb$ and $v \in E$. We have
$$b \cdot (h \otimes_{\Hb^-} v)=bh \otimes_{\Hb^-} v = hb \otimes_{\Hb^-} v,$$
so
$$b \cdot (h \otimes_{\Hb^-} v)=\sum_{i \in I} a_i h f_i \otimes_{\Hb^-} (w_i g_i \cdot v).$$ 
Let $I_0$ denote the set of $i \in I$ such that $\deg_\ZM(f_i)=\deg_\ZM(g_i)=0$. Then 
$g_i v=0$ if $i \not\in I_0$, and $f_i$, $g_i \in \kb$ if $i \in I_0$, so 
$$b \cdot (h \otimes_{\Hb^-} v)=\sum_{i \in I_0}  h \otimes_{\Hb^-} (a_i f_i w_i g_i\cdot v).$$
Since $\sum_{i \in I_0} a_i f_i g_iw_i = \Omeb(b)$, the result follows.
\end{proof}

Note that $\O_E(\euler)=C_E$, so the next proposition follows also from
Proposition \ref{pr:OmegaEuler}.

\begin{prop}
\label{action euler verma}
The element $\euler$ acts by $\Omega_E(\euler)=C_E$ on $\Delta(E)$.
\end{prop}

\begin{rema}
	Let us assume in this remark that $W$ is irreducible.
Recall (Example \ref{Z graduation-1}) that the $(\BZ/z_W\BZ)$-grading
on $\Hb$ deduced from the $\ZM$-grading 
is induced by conjugation by an element $w_z\in \Zrm(W)$. As a 
consequence, the category of $\BZ$-graded $\Hb$-modules decomposes
as a direct sum, parametrized by $l\in \BZ/z_W\BZ$,
of subcategories with objects the graded modules $M$
such that $w_z$ acts on $M^i$ by $\zeta^{i+l}$.

Given $R$ a noetherian commutative $\kb[\CCB]$-algebra, this
induces a corresponding decomposition of the category $\OC(R)$
of finitely generated graded $(\Hb\otimes_{\kb[\CCB]}R)$-modules that are
locally nilpotent for the action of $V$.
\end{rema}

\subsection{}

Let $\CG$ be a prime ideal of $\kb[\CCB]$. Recall that 
$\Hb(\CG)$ is a finitely
generated module over its center $Z(\CG)$
(Lemma \ref{P stable} and Corollary \ref{coro:P-libre})
and $Z(\CG)$ is a finitely generated
$\kb(\CG)$-algebra (Corollary \ref{coro:endo-bi} and Theorem
\ref{EG spherique-0}). It follows that all simple
$\Hb(\CG)$-modules are finite-dimensional over $\kb(\CG)$ and their
annihilators in $Z(\CG)$ are maximal ideals.

\begin{prop}\label{pr:simpleinregW}
Let $\zG$ a prime ideal of $Z$ and let 
	$L$ be a simple $(\Hb\otimes_Z Z_\zG)$-module.
Then 
	\begin{itemize}
		\item $L$ is a composition factor of $\Hb e\otimes_ZZ(\zG)$
		\item the restriction of $L$ to $Z(\zG)W$ is isomorphic to a direct
			summand of $Z(\zG)W$.
	\end{itemize}
\end{prop}

\begin{proof}
	Let $\pG=P\cap\zG$, a prime ideal of $P$.
	We have $\Hb\otimes_P\Kb\simeq
	(\Hb e\otimes_P\Kb)^{\oplus|W|}$ as
	$(\Hb\otimes_P\Kb)$-modules (Theorem \ref{theo:KH-mat}). Since $L$ is a composition factor of
	$\Hb\otimes_PP(\pG)$, it follows from
	Proposition \ref{prop:geck-rouquier}
	that it is a composition factor of
	$\Hb e\otimes_PP(\pG)$. Since $L$ is annihilated by $\zG$,
	we deduce that
	$L$ is a composition factor of
	$\Hb e\otimes_ZZ(\zG)$.

	Since $\Hb e\otimes_Z\Frac(Z)\simeq
	\Frac(Z)W$ as $\Frac(Z)W$-modules (Theorem \ref{theo:KH-mat}), it follows from
	Proposition \ref{prop:geck-rouquier} that
	$\Hb e\otimes_ZZ(\zG)$ and
	$Z(\zG)W$ have the same class in $K_0(Z(\zG)W)$, hence
	the $Z(\zG)W$-module $\Hb e\otimes_ZZ(\zG)$ is
	free of rank $1$. The proposition follows.
\end{proof}

%
%
%
%

The following theorem is \cite[Theorem 1.7]{EG}.

\begin{theo}
	\label{th:smoothsimples}
	Let $\zG$ be a prime ideal of $Z$. The following assertions are equivalent
\begin{itemize}
\itemth{1} $Z$ is smooth at $\zG$
\itemth{2} $\Hb e\otimes_Z Z_\zG$ is a progenerator for $\Hb\otimes_Z Z_\zG$
\itemth{3} there is a unique simple $(\Hb\otimes_Z Z_\zG)$-module
\itemth{4} $\Hb e\otimes_ZZ(\zG)$ is a simple $(\Hb\otimes_Z Z_\zG)$-module
\itemth{5} there is a simple $(\Hb\otimes_Z Z_\zG)$-module with dimension
	$\ge |W|$ over $Z(\zG)$
\itemth{6} the simple  $(\Hb\otimes_Z Z_\zG)$-modules are free $Z(\zG)W$-modules of
	rank $1$.
	\end{itemize}
\end{theo}

\begin{proof}
	Recall that the $Z(\zG)W$-module $\Hb e\otimes_ZZ(\zG)$ is
	free of rank $1$ (proof of Proposition \ref{pr:simpleinregW}). 
	It follows now from Proposition \ref{pr:simpleinregW} that
	(3)$\Leftarrow$(4)$\Leftrightarrow$(5)$\Leftrightarrow$(6).

	Assume (3). It follows from Proposition \ref{pr:simpleinregW}
	that $\Hb e\otimes_Z Z(\zG)$ is a progenerator for 
	$\Hb\otimes_Z Z(\zG)$, hence there is a surjective morphism
	of $\Hb\otimes_Z Z(\zG)$-modules
	$f:(\Hb e\otimes_Z Z(\zG))^r\twoheadrightarrow \Hb\otimes_Z Z(\zG)$ for
	some $r>0$. The composition
	$$(\Hb e\otimes_Z Z_\zG)^r\xrightarrow{\mathrm{can}}
	(\Hb e\otimes_Z Z(\zG))^r\xrightarrow{f} \Hb\otimes_Z Z(\zG)$$
	lifts to a morphism of $\Hb\otimes_Z Z(\zG)$-modules
	$(\Hb e\otimes_Z Z_\zG)^r\to \Hb\otimes_Z Z_\zG$ that is surjective
	by Nakayama's Lemma. It follows that $\Hb e\otimes_Z Z_\zG$ is
	a progenerator for $\Hb\otimes_Z Z_\zG$. So (3)$\Rightarrow$(2).

	Assume (2). It follows from Theorem \ref{theo:satake} that $Z_\zG$ is Morita equivalent
	to $\Hb\otimes_Z Z_\zG$. Since 
	$\Hb$ has finite projective dimension as a 
	$(\Hb\otimes_{\kb[\CCB]}\Hb^\opp)$-module
(Theorem \ref{EG spherique-1}), it follows that
	$\Hb\otimes_Z Z_\zG$ has finite projective dimension as a 
	$((\Hb\otimes_Z Z_\zG)\otimes_{\kb[\CCB]}(\Hb\otimes_Z Z_\zG)^\opp)$-module, hence
	$Z_\zG$ has finite projective dimension as a 
	$(Z_\zG\otimes_{\kb[\CCB]}Z_\zG)$-module.
	it follows that $Z$ is smooth at $\zG$. So (2)$\Rightarrow$(1).

	Assume (1). Since $\Hb e\otimes_P P_{P\cap\zG}$ is a free
	$P_{P\cap\zG}$-module (Corollary \ref{coro:P-libre-c}), it follows that
	its depth is equal to the Krull dimension of $P_{P\cap\zG}$.
	Since $Z_\zG$ is a finite flat $P_{P\cap\zG}$-module
	(Corollary \ref{coro:P-libre}), it follows that the $Z_{\zG}$-module
	$\Hb e\otimes_Z Z_\zG=(\Hb e\otimes_P P_{P\cap\zG})
	\otimes_{P_{P\cap\zG}}Z_\zG$ has depth equal to the Krull dimension
	of $Z_{\zG}$ \cite[\S 1, Proposition 11]{bourbaki10}.
	As $Z_\zG$ is regular, it follows that
	$\Hb e\otimes_Z Z_\zG$ is a projective $Z_\zG$-module 
	\cite[\S 4, Proposition 3]{bourbaki10}
	hence $\Hb e\otimes_Z Z_\zG$ induces a Morita equivalence
	between $\Hb\otimes_Z Z_\zG$ and $Z_\zG$ (Lemma \ref{le:critereMorita}).
	We deduce that (1)$\Rightarrow$(4).
\end{proof}


\section{Gaudin algebras}
\label{ch:Gaudin}

\bigskip

Gaudin operators have been introduced when $W$ is a symmetric group via
Schur-Weyl duality in \cite{MuTaVa1,MuTaVa3}.

\subsection{$W$-covering of $\ZCB$}
\label{se:Wcovering}

Let $\ZCB'=\ZCB^\reg\times_{V^\reg/W}V^\reg$\indexnot{zp}{\ZCB'}.
We have a cartesian square

$$\xymatrix{
\ZCB' \ar[rr]^-{\mathrm{can}}\ar[d]_-{\mathrm{can}} && \ZCB^\reg \ar[d]^-\Upsilon \\
\CCB\times V^\reg\times V^*/W\ar[rr]_-{\mathrm{can}} && 
\CCB\times V^\reg/W\times V^*/W
}$$

There is an action of $W$ on $\ZCB'$ given by $w(z,v)=(z,w(v))$ for
$z\in\ZCB^\reg$ and $v\in V^\reg$. Let 
$Z'=\kb[\ZCB']=Z^\reg\otimes_{\kb[V^\reg]^W}\kb[V^\reg]$\indexnot{zp}{Z'}.

\begin{lem}
\label{le:Zprime}
The multiplication map gives an isomorphism $Z'\rtimes W\xrightarrow{\sim}
\Hb^\reg$. The image of $Z'$ by that map is
$C_{\Hb^\reg}(V^*)=\kb[\CCB\times V^\reg]\otimes\Theta^{-1}(\kb[V^*])$.

There are isomorphisms
$$\CCB\times V^\reg\times V^*\xrightarrow[\mathrm{can}]{\sim}
\CCB\times \bigl((V^\reg\times V^*)/\Delta W\bigr)\times_{V^\reg/W}V^\reg
\xrightarrow[\Theta^\#\times\mathrm{id}]{\sim}\ZCB'.$$
\end{lem}

\begin{proof}
There is a commutative diagram
$$\xymatrix{
Z^\reg\otimes_{\kb[V^\reg]^W}(\kb[V^\reg]\rtimes W)\ar[rr]^-{\mathrm{mult}}
\ar[d]_{\Theta\otimes \mathrm{id}}^\sim &&
\Hb^\reg \ar[d]^{\Theta}_\sim \\
\kb[\CCB\times (V^\reg\times V^*)/\Delta W]\otimes_{\kb[V^\reg]^W}
(\kb[V^\reg]\rtimes W)\ar[rr]_-{\mathrm{mult}} &&
\kb[\CCB\times V^\reg\times V^*]\rtimes W
}$$
Since the canonical map $V^\reg\times V^*\to
(V^\reg\times V^*)/\Delta W\times_{V^\reg/W} V^\reg$ is an isomorphism,
it follows that the bottom horizontal map in the diagram is an
isomorphism, hence the top horizontal map is an isomorphism as well. 
The other assertions of the lemma are clear.
\end{proof}

Recall that $\Hb^\reg e$ induces a Morita equivalence between
$\Hb^\reg$ and $Z^\reg$ (Corollary \ref{coro:reg-morita}).
Through the isomorphisms of Lemma
\ref{le:Zprime}, this corresponds to the Morita equivalence between
$Z'\rtimes W$ and $Z^{\prime W}=Z^\reg$ given by $Z'$.

\medskip
\def\Hilb{{\mathrm{Hilb}}}

Let $\Hilb^{|W|}(V^*)$ be the Hilbert scheme of $0$-dimensional subschemes
of length $|W|$ of $V^*$ and let $F_{|W|}\subset \Hilb^{|W|}(V^*)\times V^*$
be the associated universal family.
Consider the inverse $\ZCB'\xrightarrow{\sim}\CCB\times V^\reg\times V^*$ of the
isomorphism of Lemma \ref{le:Zprime} and its composition with the 
projection $\CCB\times V^\reg\times V^*\to V^*$. This defines a morphism
$f:\ZCB'\to V^*$. The product map 
$$\textrm{can}\times f:\ZCB'\to (\CCB\times V^\reg\times V^*/W)\times V^*$$
is a closed immersion.
Since the canonical map $\ZCB'\to
\CCB\times V^\reg\times V^*/W$ is finite and flat of degree $|W|$,
the universal property of $F_{|W|}$ shows that there is a unique morphism 
$\eta:\CCB\times V^\reg\times V^*/W\to \Hilb^{|W|}(V^*)$ and a (unique) isomorphism
$$\CCB\times V^\reg\times V^*/W\times_{\Hilb^{|W|}(V^*)} F_{|W|}\xrightarrow{\sim}\ZCB'$$
over $\CCB\times V^\reg\times V^*/W\times V^*$.

\smallskip
Let us describe the construction above at the level of points.
Given $m$ a (non-necessarily closed) point of
$\CCB\times V^\reg\times V^*/W$, we put
$L(m)=\kb(m)\otimes_{\kb[\CCB\times V^\reg\times V^*/W]}
\Hb^\reg e\xrightarrow{\sim}
\kb(m)\otimes_{\kb[\CCB\times V^\reg\times V^*/W]}Z'$\indexnot{Lm}{L(m)}.
We have $\dim_{\kb(m)} L(m)=|W|$.

The action of $\Theta^{-1}(\kb[V^*])$ by left multiplication on $Z'\xrightarrow{\sim}
\Hb^\reg e$ induces an action on $L(m)$. It comes from a
surjective morphism of algebras 
$$\kb(m)\otimes_\kb\Theta^{-1}(\kb[V^*])
\twoheadrightarrow \kb(m)\otimes_{\kb[\CCB\times V^\reg\times V^*/W]}Z'.$$
The kernel of that morphism defines a point of $\Hilb^{|W|}(V^*)$. It corresponds to
the schematic support of the $\Theta^{-1}(\kb[V^*])$-module $L(m)$.

\smallskip
Let $\ZC''$ be the spectral scheme of $\Theta^{-1}(V)$ acting on
the family $\{L(m)\}_m$.
This is the closed subscheme of 
$\CCB\times V^\reg\times V^*/W\times V^*$ given by
$$\ZC''=\{(c,v,u,\lambda)
\ |\ {\det}_{L(c,v,u)}(\Theta^{-1}(y)-\langle y,\lambda\rangle)=0\
\forall y\in V\}.$$
We have obtained the following proposition.

\begin{prop}
\label{pr:spectrum}
There is an isomorphism
$\ZC''\xrightarrow{\sim} \ZC',\ (c,v,u,\lambda)\mapsto
(\Theta^\#(c,v,\lambda),u)$.
\end{prop}

\begin{rema}
\label{le:eigenspaces}
	Note that
	the indecomposable direct summands of the
	$\bigl(\kb(m)\otimes_\kb \Theta^{-1}(\kb[V^*])\bigr)$-module $L(m)$
	have no composition factors in common.
\end{rema}

\subsection{Gaudin operators}
\label{se:Gaudinoperators}
\medskip

We consider the $\bigl(\Theta^{-1}(\kb[V^*])\otimes \kb[\CCB\times V^\reg]
\otimes\kb[V^*]\bigr)$-module
$\bar{L}=\Hb^\reg$, where $\Theta^{-1}(\kb[V^*])\otimes\kb[\CCB\times V^\reg]$
acts by left multiplication and $\kb[V^*]$ acts by right multiplication.
Note that $\bar{L}$ is a free 
$(\kb[\CCB\times V^\reg] \otimes\kb[V^*])$-module
with basis $W$. The action of $\Theta^{-1}(\kb[V^*])$ on that module
provides a morphism of varieties
$\bar{\eta}:\CCB\times V^\reg \times V^*\to \mathrm{Sym}^{|W|}(V^*)$, where
$\mathrm{Sym}^{|W|}(V^*)$ denotes the $|W|$-th symmetric power of the variety
$V^*$.

Given a closed point $(c,v,v^*)\in\CCB\times V^\reg\times V^*$,
we put $\bar{L}(c,v,v^*)=\kb(c,v)\otimes_{\kb[\CCB\times V^\reg]}\Hb^\reg
\otimes_{\kb[V^*]}\kb(v^*)$, a $\Theta^{-1}(\kb[V^*])$-module.
We define similarly $\bar{L}(\tilde{m})$ for any point $\tilde{m}$ of
$\CCB\times V^\reg\times V^*$.
We denote
by $(e_w)_{w\in W}$ the $\kb(\tilde{m})$-basis of $\bar{L}(\tilde{m})$ obtained as the
image of $W$.

Given $y\in V$, the action of $\Theta^{-1}(y)$ on $\bar{L}(c,v,v^*)$ is given by the
operator
$$D_y^{c,v,v^*}:e_w\mapsto \langle y,w(v^*)\rangle e_w+
\sum_{s\in\REF(W)}\varepsilon(s) c_s
\frac{\langle y,\alpha_s\rangle}{\langle v,\alpha_s\rangle}e_{sw}.$$
\indexnot{Dy}{D_y^{c,v,v^*}}

Let $u$ be the image of $v^*$ in $V^*/W$. The
$(\kb[V^*]\rtimes W)$-module $\Ind_{\kb[V^*]}^{\kb[V^*]\rtimes W}\kb(v^*)$
is isomorphic to the semisimplification of 
$\Ind_{\kb[V^*]^W\otimes \kb W}^{\kb[V^*]\rtimes W}(\kb(u)\otimes\kb)=
\kb[V^*]\otimes_{\kb[V^*]^W}\kb(u)$.
Consequently, 
$\bar{L}(c,v,v^*)$ is isomorphic to the graded module associated with
a filtration of $L(c,v,u)$ (the filtration does not depend on
$v^*$). Similarly, given any $\tilde{m}$ with
image $m\in \CCB\times V^\reg\times V^*/W$, the
$(\kb(\tilde{m})\otimes_\kb\Theta^{-1}(\kb[V^*])$-module $\bar{L}(\tilde{m})$ 
is isomorphic to the graded module associated with a filtration of $L(m)$
(in particular, $\bar{L}(\tilde{m})$ depends only on $m$).
As a consequence, the morphism $\bar{\eta}$ is the composition
$$\CCB\times V^\reg \times V^*\xrightarrow{\mathrm{can}} \CCB\times V^\reg \times V^*/W
\xrightarrow{\eta}
\Hilb^{|W|}(V^*) \xrightarrow{\mathrm{Hilbert-Chow}} \mathrm{Sym}^{|W|}(V^*).$$

\smallskip
We now introduce the spectral scheme of $\Theta^{-1}(V)$ acting on
the family $\{\bar{L}(c,v,v^*)\}_{c,v,v^*}$.
Let $\bar{\ZC}''$ be the closed subscheme of 
$\CCB\times V^\reg\times V^*/W\times V^*$ given by
$$\bar{\ZC}''=\{(c,v,u,\lambda)
\ |\ {\det}_{\bar{L}(c,v,v^*)}(\Theta^{-1}(y)-\langle y,\lambda\rangle)=0\
\forall y\in V\}$$
where $v^*\in V^*$ has image $u$ in $V^*/W$.
Forgetting $\lambda$ gives a morphism $\bar{\ZC}''\to 
\CCB\times V^\reg\times V^*/W$ and the fiber at $(c,v,u)$ is the
spectrum of $\Theta^{-1}(V)$.

\smallskip
From Proposition \ref{pr:spectrum}, we deduce the following description of that
variety.

\begin{prop}
\label{pr:spectrumGaudin}
There is an isomorphism
$\bar{\ZC}''\xrightarrow{\sim} \ZC',\ (c,v,u,\lambda)\mapsto
(\Theta^\#(c,v,\lambda),u)$.
\end{prop}

\bigskip

\section{Automorphisms}

The group $\kb^\times \times \kb^\times \times (W^\wedge \rtimes \NC)$ acts on $\Hbt$, hence it acts on the
category of $\Hbt$-modules. The action is given as follows.
Let $M$ be an $\Hbt$-module and let
$\t \in \kb^\times \times \kb^\times \times (W^\wedge \rtimes \NC)$.
We denote by $\lexp{\t}{M}$ the $\Hbt$-module whose underlying
$\kb$-module is $M$ and where the action of $h\in\Hbt$ on an element
of $\lexp{\t}{M}$ is given by the action of
$\lexp{\t^{-1}}{h}$ on the corresponding element of $M$.

This defines a functor
$$\t : \Hbt\module \longto \Hbt\module$$
and this induces an action of 
$\kb^\times \times \kb^\times \times (W^\wedge \rtimes \NC)$ on the category 
$\Hbt\module$. Similarly, we can define a functor 
$$\t : \Hbt^-\module \longto \Hbt^-\module$$
and an action of $\kb^\times \times \kb^\times \times (W^\wedge \rtimes \NC)$ on the category 
$\Hbt^-\module$. There is a commutative diagram
\equat\label{M lineaire}
\xymatrix{
\Hbt^-\mmodgr \ar[rr]^{\Ind} \ddto_{\DS{\t}} && \Hbt\mmodgr
\ddto^{\DS{\t}}\\
&& \\
\Hbt^-\mmodgr \ar[rr]_{\Ind} && \Hbt\mmodgr
}
\endequat
The next proposition is now clear.

\bigskip

\begin{prop}\label{omega lineaire}
Given $E \in \Irr(\kb W)$ and 
$\t=(\xi,\xi',\g \rtimes g) \in \kb^\times \times \kb^\times \times (W^\wedge \rtimes \NC)$, we have
$$\lexp{\t}{\tilde{\Delta}(E)} \simeq \tilde{\Delta}(\lexp{g}{E}\otimes \g^{-1})$$
and
$$\O_E(\lexp{\t}{z})=\lexp{\t}{\bigl(\O_{\lexp{g}{E}\otimes \g^{-1}}(z)\bigr)}$$
for all $z \in Z$.
\end{prop}

\bigskip

\begin{coro}\label{omega gradue}
Given $E\in \Irr(\kb W)$, then $\O_E: Z \to \kb[\CCB]$ is a bigraded morphism. 
In particular, $\Ker(\O_E)$ is a bi-homogeneous ideal of $Z$. 
\end{coro}

\bigskip
%
%
%
%

\chapter{Hecke algebras}
\label{ch:Hecke}
\boitegrise{{\bf Notation.} 
{\it From now on, and until the end of this book, we fix a number field $F$ 
contained in $\kb$, \indexnot{F}{F}
which is Galois over $\QM$ and contains all the traces of elements of $W$, 
and we denote by $\OC$ 
\indexnot{O}{\OC} the integral closure of $\ZM$ in $F$. We also fix an embedding $F \longinjto \CM$. 
By Proposition \ref{pr:deploiementV},
there exists 
a $W$-stable $F$-vector subspace $V_F$ \indexnot{V}{V_F}  of $V$ such that 
$V = \kb \otimes_F V_F$. 
Let $a \mapsto \aba$ denote the complex conjugation (it stabilizes $F$ since $F$ is Galois over 
$\QM$). Finally, we denote by $\mub_W$ \indexnot{mz}{\mub_W}  the group of roots of unity 
of the field generated by the traces of elements of $W$.}}{0.75\textwidth}

\bigskip

The existence of such a field $F$ is easy: we can take the field generated 
by the traces of elements of $W$ (it is Galois over $\QM$ as it is contained in a cyclotomic 
number field). Note also that $F$ contains all the roots of unity of the form 
$\z_{e_H}$, where $H \in \AC$.

\bigskip

\section{Definitions}
\label{se:definitionsHecke}

\bigskip

\subsection{Braid groups}
\label{se:braidgroups}

Set $V_\CM= \CM \otimes_F V_F$. Given
$H \in \AC$, let $H_\CM = \CM \otimes_F (H \cap V_F)$. 
We define 
$$V_\CM^\reg = V_\CM \setminus \bigcup_{H \in \AC} H_\CM \indexnot{V}{V_\CM^\reg}.$$
We fix a point $v_\CM \in V_\CM^\reg$.
Given $v \in V_\CM$, we denote by 
$\vba$ its image in the quotient variety $V_\CM/W$. The {\it braid group} associated with $W$, 
denoted $B_W$, is defined as
$$B_W = \pi_1(V_\CM^\reg/W,\vba_\CM).\indexnot{B}{B_W}$$
The {\it pure braid group} associated with $W$, denoted by $P_W$, 
is then defined as
$$P_W = \pi_1(V_\CM^\reg,v_\CM).\indexnot{P}{P_W}$$
The covering $V_\CM^\reg \to V_\CM^\reg/W$ being unramified
(Steinberg's Theorem \ref{th:Steinberg}), we obtain an exact sequence 
\equat\label{eq:suite exacte braid}
1 \longto P_W \longto B_W \stackrel{p_W}{\longto} W \longto 1.
\endequat
Given $H \in \AC$, we denote by $\sigb_{\!\!H}$ \indexnot{sz}{\sigb_{\!\!H}}  a {\it generator of the monodromy} 
around the hyperplane $H$, as defined in~\cite[\S{2.A}]{BMR}, and such that 
$p_W(\sigb_{\!\!H})=s_H$. This is an element of $B_W$ well-defined up to conjugacy by an element of $P_W$.
Recall~\cite[Theorem~2.17]{BMR} that 
\equat\label{eq:BW engendre}
\text{\it $B_W$ is generated by $(g\sigb_{\!\!H}g^{-1})_{H \in \AC,g\in
P_W}$.}
\endequat
It can be proven \cite{bessis diagram,BMR} that $B_W$ is already generated by $(\sigb_{\!\!H})_{H \in \AC}$,
for a suitable choice of the elements $\sigb_{\!\!H}$.

\medskip
Assume $W$ is irreducible.
We denote by $\pib_z$ \indexnot{pzz}{\pib_z} the image in $B_W$ of the path in
$V_\CM^\reg$ defined by 
$$\fonction{\pib_z}{[0,1]}{V_\CM^\reg}{t}{e^{2i\pi t/z} v_\CM.}$$
Recall that $z_W=|\Zrm(W)|$. 
Note that~\cite[Lemma~2.22]{BMR}
\equat\label{eq:pi central}
\pib_z\in \Zrm(B_W).
\endequat
The image of $\pib_z$ in $W$ is the generator $w_z=e^{2i\pi/z_W}\Id_V$
of $W\cap \Zrm(\GL(V))$.
We put $\pib=(\pib_z)^{z_W}\in P_W\cap \Zrm(B_W)$\indexnot{pi}{\pib}.

\smallskip
Consider now a general $W$ with decomposition $W=W_1\times\cdots\times W_r$
into indecomposables. We have a canonical isomorphism
$B_{W_1}\times\cdots\times B_{W_r}\xrightarrow{\sim}B_W$.
Let $\pib_i$ be the element $\pib_z$ for $W_i$.
We put $\pib_z=\pib_1\cdots\pib_r\in Z(B_W)$. Its image in $W$ is
$w_z$. We put $\pib=(\pib_z)^{z_W}\in P_W\cap \Zrm(B_W)$.
\bigskip

\subsection{Generic Hecke algebra}\label{sub:hecke}

Recall that $\orbiteb^\circ$ is the set of pairs $(\orbite,j)$ with $\orbite \in \AC/W$ and $0 \le j \le e_\orbite -1$ 
(see \S\ref{section:hyperplans}). 

\bigskip

Consider the affine variety $\TCB$\indexnot{T}{\TCB} over $\OC$ with ring of functions
$\OC[\TCB]=\OC[(\qb_{\orbite,j}^{\pm 1})_{(\orbite,j) \in \orbiteb^\circ}]$. 
Its fraction field is
$F(\TCB)$.
Given $\orbite \in \AC/W$, $H \in \orbite$ and $0 \le j \le e_H-1=e_\orbite-1$, we put
$\qb_{H,j}=\qb_{\orbite,j}$.\indexnot{qa}{\qb_{H,j}}

\bigskip

The {\it generic Hecke algebra} associated with $W$, denoted by $\heckegenerique$, \indexnot{H}{\heckegenerique}
is the quotient of the group algebra $\OC[\TCB] B_W$ by the ideal generated by the elements 
\equat\label{eq:relations Hecke}
\prod_{j = 0}^{e_H-1} (\sigb_{\!\!H} - \z_{e_H}^j \qb_{H,j}^{|\mub_W|}),
\endequat
where $H$ runs over $\AC$. Given $H\in\AC$, let $\Tb_H$ \indexnot{T}{\Tb_H} denote the image of 
$\sigb_{\!\!H}$ in $\heckegenerique$. By~(\ref{eq:BW engendre}),
\equat\label{eq:HW-engendre}
\text{\it $\heckegenerique$ is generated by $(b\Tb_Hb^{-1})_{H \in \AC,b\in P_W}$,}
\endequat
where we still denote by $b$ the image in $\heckegenerique$ of an
element $b\in B_W$.
If $H \in \AC$, then 
\equat\label{eq:TH-relation}
\prod_{j = 0}^{e_H-1} (\Tb_H - \z_{e_H}^j \qb_{H,j}^{|\mub_W|})=0.
\endequat
As a consequence,
\equat\label{eq:TH-inversible}
\text{\it $\Tb_H$ is invertible in $\heckegenerique$.}
\endequat
The next lemma follows immediately from~\cite[Proposition~2.18]{BMR}:

\bigskip

\begin{lem}\label{lem:hecke specialization}
The specialization $\qb_{\orbite,j} \mapsto 1$ gives
an isomorphism of $\OC$-algebras 
$\OC \otimes_{\OC[\TCB]} \heckegenerique \longiso \OC W$.
\end{lem}

\bigskip

Let $a \mapsto \aba$ denote the unique automorphism of the $\BZ$-algebra 
$\OC[\TCB]$ which extends the complex conjugation on $\OC$ 
and such that $\overline{\qb}_{\orbite,j}=\qb_{\orbite,j}^{-1}$. 

The following Theorem was first conjectured in~\cite[\S{4.C}]{BMR} 
and then proved case-by-case~\cite{ariki}, \cite{ariki-koike}, \cite{BMR}, \cite{chavli1}, 
\cite{chavli2}, \cite{chavli3}, \cite{marin1}, \cite{marin2}, \cite{marin3}, 
\cite{marin-pfeiffer} and~\cite{tsuchioka}. 

\bigskip

\begin{theo}\label{theo:liberte}
The Hecke algebra $\heckegenerique$ is a free $\OC[\TCB]$-module of rank $|W|$.
\end{theo}

\bigskip

Let us now state a basic conjecture~\cite[\S{2.A}]{BMM}.

\bigskip

\begin{conj}
\label{ce:libertesymetrie}
Assume $W$ is irreducible.
There exists a symmetrizing form $\taub_\HHC : \heckegenerique \to \OC[\TCB]$ 
such that:
\begin{itemize}
\itemth{a} After the specialization of Lemma~\ref{lem:hecke specialization} (i.e. 
$\qb_{\orbite,j} \mapsto 1$), $\taub_\HHC$ specializes to the canonical symmetrizing form 
of $\OC W$ (i.e. $w \mapsto \d_{w,1}$). 

\itemth{b} If $b \in B_W$, then 
$$\taub_\HHC(\pib) \overline{\taub_\HHC(b^{-1})} = \taub_\HHC(b\pib).\indexnot{tx}{\taub_\HHC}$$
\end{itemize}
\end{conj}

\bigskip

\begin{rema}\label{rem:hecke}
$\bullet\ $
 There is at most one form $\taub_\HHC$ satisfying (a) and (b) 
	\cite[Theorem 2.1]{BMM}.

\smallskip
$\bullet\ $
If $W=W_1\times\cdots\times W_r$ is the decomposition of $W$ into irreducible
components and if conjecture \ref{ce:libertesymetrie} holds for the $W_i$'s, then
we obtain a symmetrizing form on $\heckegenerique$ as the tensor product
of the forms on the Hecke algebras of the $W_i$'s.

$\bullet\ $
If Conjecture \ref{ce:libertesymetrie} holds, then
$\taub_\HHC(\pib) \neq 0$ since, by the property (a) and by~(\ref{eq:pi central}), 
$\taub_\HHC(\pib)$ specializes to $1$ through $\qb_{H,j} \mapsto 1$.\finl
\end{rema}

\smallskip

Conjecture~\ref{ce:libertesymetrie} holds for real reflection groups
\cite[Theorem 2.1]{BMM}.
It also holds for the following non-real complex reflection groups:
\begin{itemize}
\item The group $G_4$ by~\cite{malle michel}, \cite{marin wagner}, 
\cite{BCCK} (three independent proofs);

\item The groups $G_5$, $G_6$, $G_7$, $G_8$ by~\cite{BCCK};

\item The group $G _{12}$ by~\cite{malle michel};

\item The group $G_{13}$ by~\cite{BCC};

\item The groups $G_{22}$, $G_{24}$ by~\cite{malle michel}.
\end{itemize}
A symmetrizing form satisfying property~(a) of Conjecture~\ref{ce:libertesymetrie} 
has been obtained for the groups of the infinite series $G(de, e, n)$ by~\cite{bremke malle} 
and~\cite{malle mathas}.

\bigskip

\subsection{Cyclotomic Hecke algebras} 
We will not use here the classical definition of cyclotomic Hecke 
algebras~\cite[\S{6.A}]{BMM},~\cite[Definition~4.3.1]{chlouveraki LNM}, 
since we will need to work over a sufficiently large ring allowing us to let the parameters 
vary as much as possible.

\bigskip

\boitegrise{{\bf Notation.} 
{\it Following~\cite{bonnafe two},~\cite{bonnafe continu},~\cite{bonnafe faux} or~\cite{bonnafe livre}, 
we will use an exponential notation for the group algebra $\OC[\CM]$, which will be denoted by 
$\OC[\qb^\CM]$: $\OC[\qb^\CM]= \bigoplus_{r \in \CM} ~\OC~\!\qb^r$,  \indexnot{O}{\OC[\qb^\CM]}  with 
$\qb^r \qb^{r'}=\qb^{r+r'}$. Since $\OC$ is integral and 
$\CM$ is torsion-free, $\OC[\qb^\CM]$is also integral 
and we denote by $F(\qb^\CM)$ \indexnot{F}{F(\qb^\CM)}  its fraction field. 
If $a = \sum_{r \in \CM} a_r\qb^r$, we denote by 
$\deg(a)$ \indexnot{da}{\deg}  (respectively $\val(a)$) \indexnot{va}{\val}  its degree 
(respectively its valuation), 
that is, the element of $\RM \cup \{-\infty\}$ 
(respectively $\RM \cup \{+\infty\}$) defined by \\
\centerline{$\deg(a) = \max\{r \in \RM~|~a_r \neq 0\}$}\\
\centerline{(respectively \quad $\val(a)=\min\{r \in \RM~|~a_r \neq 0\}$).}}}{0.75\textwidth}

\bigskip

We have $\deg(a) = -\infty$ (respectively $\val(a)=+\infty$) if and only if $a=0$. 
The usual properties of degree and valuation (with respect to the sum and the product) 
are of course satisfied. Let us start with an easy remark:

\medskip

\begin{lem}\label{lem:A-normal}
The ring $\OC[\qb^\CM]$ is integrally closed.
\end{lem}

\begin{proof}
This follows from the fact that $\OC[\qb^\CM]=\bigcup_{\L \subset \CM} \OC[\qb^\L]$, 
where $\L$ runs over the finitely generated subgroups of $\CM$, and that, if $\L$ has $\ZM$-rank $e$, then 
$\OC[\qb^\L] \simeq \OC[\tb_1^{\pm 1},\dots,\tb_e^{\pm 1}]$ is integrally closed.
%
\end{proof}

\bigskip

Fix a family $k=(k_{\orbite,j})_{(\orbite,j) \in \orbiteb^\circ}$ of complex numbers 
(as usual, if $H \in \orbite$ and $0 \le j \le e_H -1$, then we set $k_{H,j}=k_{\orbite,j}$).
The {\it cyclotomic Hecke algebra} ({\it with parameter $k$}) is the $\OC[\qb^\CM]$-algebra
$\heckecyclotomique(k)=\OC[\qb^\CM] \otimes_{\OC[\TCB]} \heckegenerique$,
\indexnot{H}{\heckecyclotomique(k)}
where 
$\OC[\qb^\CM]$ is viewed as an $\OC[\TCB]$-algebra through the morphism 
$$\fonction{\Th_k^\cyclo}{\OC[\TCB]}{\OC[\qb^\CM]}{\qb_{\orbite,j}}{\qb^{k_{\orbite,j}}.}
\indexnot{ty}{\Th_k^\cyclo}$$
Let $T_H$ \indexnot{T}{T_H} denote the image of 
$\Tb_H$ in $\heckecyclotomique(k)$; then 
\equat\label{eq:hecke-engendre-cyclotomique}
\text{\it $\heckecyclotomique(k)$ is generated by $(\bar{g}T_H\bar{g}^{-1})_{H \in \AC,g\in P_W}$}
\endequat
and, if $H \in \AC$, then 
\equat\label{eq:TW-relation-cyclotomique}
\prod_{j = 0}^{e_H-1} (T_H - \z_{e_H}^j \qb^{|\mub_W| k_{H,j}})=0.
\endequat

\bigskip

\begin{rema}
It follows from Lemma~\ref{lem:hecke specialization} that, after the specialization 
$\OC[\qb^\CM] \to \OC$, $\qb^r \mapsto 1$ (this is the augmentation morphism for the group 
$\CM$), we obtain $\OC \otimes_{\OC[\qb^\CM]} \heckecyclotomique(k) \simeq \OC W$. 

Similarly, $\heckecyclotomique(0)\simeq \OC[\qb^\CM] W$.\finl 
\end{rema}

\bigskip


\bigskip

\begin{rema}\label{rem:translation}
Let $(\l_\orbite)_{\orbite \in \AC/W}$ be a family of complex numbers and, if $H \in \orbite$, 
set $\l_H=\l_\orbite$. Let $k_{\orbite,j}'=k_{\orbite,j} + \l_\orbite$ and let 
$k'=(k_{\orbite,j}')_{(\orbite,j) \in \orbiteb^\circ}$. 
The map $\bar{g}T_H\bar{g}^{-1} \mapsto \qb^{-\l_H} \bar{g}T_H\bar{g}^{-1}$ extends to an isomorphism 
of $\OC[\qb^\CM]$-algebras $\heckecyclotomique(k) \xrightarrow{\sim} \heckecyclotomique(k')$. 

Hence, if we take $\l_\orbite=-(k_{\orbite,0}+k_{\orbite,1}+\cdots+k_{\orbite,e_\orbite-1})/e_\orbite$, 
then $\heckecyclotomique(k)\simeq \heckecyclotomique(k')$, with 
$k' \in \KCB(\CM)$.

This shows that, in the study of cyclotomic Hecke algebras, it is enough to
consider the case of parameters in the subspace $\KCB$ of $\CM^{\orbiteb^\circ}$.\finl 
\end{rema}

\begin{rema}
The group algebra $\OC[\qb^K]$ of any characteristic $0$ field $K$ is 
integrally closed.\finl
\end{rema}

\section{Coxeter groups}
\label{se:HeckeCoxeter}

\cbstart
We assume in \S \ref{se:HeckeCoxeter} that $W$ is a Coxeter group (cf. 
\S \ref{se:Coxetergroups} for the notation).
We assume $F\subset\RM$ and $v_\CM\in C_\RM$, and we will in fact denote $v_\CM$ by $v_\RM$.

\subsection{Braid groups} 
For $s$, $t \in S$, let $m_{st}$ \indexnot{ma}{m_{st}}  denote the order of $st$ in $W$. For $s \in S$ 
and $H=\Ker(s-\Id_V)$, let $\sigb_{\!\!s} = \sigb_{\!\!H}$ \indexnot{sz}{\sigb_{\!\!s}}  be the loop in $V_\CM^\reg/W$ that
is the image of the path
$$\fonctio{[0,1]}{V_\CM^\reg}{t}{e^{i\pi t} \Bigl(\DS{\frac{v_\RM-s(v_\RM)}{2}\Bigr) 
+ \frac{v_\RM+s(v_\RM)}{2}}}$$
from $v_\RM$ to $s(v_\RM)$. 
With this notation, $B_W$ admits the following presentation~\cite{brieskorn}:
\equat\label{eq:coxeter-tresse}
B_W\quad : \quad 
\begin{cases}
\text{Generators:} & (\sigb_{\!\!s})_{s \in S}, \\
\text{Relations:} & \forall~s,t \in S,~
\underbrace{\sigb_{\!\!s}\sigb_{\!\!t}\sigb_{\!\!s}\cdots}_{\text{$m_{st}$ times}} 
= \underbrace{\sigb_{\!\!t}\sigb_{\!\!s}\sigb_{\!\!t}\cdots}_{\text{$m_{st}$ times}}.
\end{cases}
\endequat
Given $w=s_1s_2\cdots s_l$ a reduced decomposition of $w$, we set 
$\sigb_{\!\!w}=\sigb_{\!\!s_1}\sigb_{\!\!s_2}\cdots \sigb_{\!\!s_l}$: \indexnot{sz}{\sigb_{\!\!w}}  it is a classical fact that 
$\sigb_{\!\!w}$ does not depend on the choice of the reduced decomposition. 
We have
\equat\label{eq:pi-coxeter}
\pib = \sigb_{\!\!w_0}^2.
\endequat

\subsection{Hecke algebras}\label{sub:hecke coxeter}

\medskip

\subsubsection*{Generic case} 
Given $s \in S$, we put
$\qb_{s,j}=\qb_{V^s,j}$. \indexnot{qa}{\qb_{s,j}}
It follows from~(\ref{eq:coxeter-tresse}) that the generic Hecke algebra $\heckegenerique$ 
admits the following presentation, where $\Tb_s$ \indexnot{T}{\Tb_s}  
denotes the image of $\sigb_{\!\!s}$ in $\heckegenerique$:
\equat\label{eq:coxeter-hecke}
\heckegenerique \quad : \quad 
\begin{cases}
\text{Generators:} & (\Tb_s)_{s \in S}, \\
\text{Relations:} & \forall~s \in S,~(\Tb_s-\qb_{s,0}^2)(\Tb_s+\qb_{s,1}^2)=0,\\
& \forall~s,t \in S,~\underbrace{\Tb_s\Tb_t\Tb_s\cdots}_{\text{$m_{st}$ times}} 
= \underbrace{\Tb_t\Tb_s\Tb_t\cdots}_{\text{$m_{st}$ times}}.
\end{cases}
\endequat
Given $w=s_1s_2\cdots s_l$ a {\it reduced decomposition} of $w$, we set 
$\Tb_w=\Tb_{s_1}\Tb_{s_2}\cdots\Tb_{s_l}$. \indexnot{T}{\Tb_w}  
This is the image of $\sigb_{\!\!w}$ in $\heckegenerique$, hence
$\Tb_w$ does not depend on the choice of the reduced decomposition. Moreover, 
\equat\label{eq:hecke-base}
\heckegenerique=
\mathop{\bigoplus}_{w \in W} \OC[\TCB]~\! \Tb_w.
\endequat
Note that $\Tb_w\Tb_{w'}=\Tb_{ww'}$ if $\ell(ww')=\ell(w)+\ell(w')$. Note also that 
the basis $(\Tb_w)_{w \in W}$ of $\heckegenerique$ depends on the choice of $S$, 
that is, of $C_\RM$. 
% ...

\subsubsection*{Cyclotomic case}
We take $k=(k_{\orbite,j})_{\orbite \in \AC/W, j \in \{0,1\}} \in \KCB(\CM)$.
Remark~\ref{rem:translation} shows that assuming $k_{\orbite,0}+k_{\orbite,1}=0$ 
does not restrict the class of algebras we are interested in. 
Recall that for $H \in \AC$, we set 
$c_{s_H} = - k_{H,0} + k_{H,1}=-2k_{H,0}=2k_{H,1}$. We write $k_{s_H}=k_{H,0}$. 
The cyclotomic Hecke algebra $\heckecyclotomique(k)$ is the $\OC[\qb^\CM]$-algebra 
with the following presentation:
\equat\label{eq:coxeter-hecke-cyclotomique}
\heckecyclotomique(k) \quad : \quad 
\begin{cases}
\text{Generators:} & (T_s)_{s \in S}, \\
\text{Relations:} & \forall~s \in S,~(T_s-\qb^{2k_s})(T_s+\qb^{-2k_s})=0,\\
& \forall~s,t \in S,~\underbrace{T_s T_t T_s\cdots}_{\text{$m_{st}$ times}} 
= \underbrace{T_t T_s T_t\cdots}_{\text{$m_{st}$ times}}.
\end{cases}
\endequat

\cbend

\section{KZ functor}
\label{se:KZ}
In this section \S\ref{se:KZ}, we assume that $\kb=\CM$.
\subsection{KZ functor and properties}

Consider a pair $\CG\subset\mG$ consisting of a prime ideal and a maximal ideal of
$\CM[\CCB]$ and let $R$ be the completion of $\CM[\CCB]/\CG$ at
$\mG/\CG$.
There is a highest
weight category structure on $\dot{\OC}(R)$ (Theorem \ref{th:OT=1}).

Let $\exp:\CCB\to\TCB$\indexnot{exp}{\exp} be the analytic map given by
$q_{H,j}=e^{2i\pi k_{H,-j}/|\mub_W|}$. It endows the $\CM[\CCB]$-algebra
$R$ with a structure of $\CM[\TCB]$-algebra.

\medskip
Let us recall the construction of the KZ functor
$\KZ:\dot{\OC}(R)\to (R\heckegenerique)\mmod$ and its
main properties as in \cite[\S 5.3]{ggor}. Our change of Dunkl operators
corresponds to a twist of the monodromy representation of
\cite{ggor} by the one-dimensional
representation of $B_W$ given by $\sigb_H\mapsto \qb_{H,0}^{|\mub_W|}$.

\medskip
Let $M\in\dot{\OC}(R)$ and let $M^\reg=R\Hbd^\reg\otimes_{R\Hbd}M$.
The isomorphism $\Theta^\reg$ (Theorem \ref{th:PBWDunkl}(d)) makes $M^\reg$
into
a $(R\otimes\DCB(V^\reg)\rtimes W)$-module and the Morita
equivalence of Lemma \ref{le:Weylfaithful}(b) produces a 
$(R\otimes\DCB(V^\reg/W))$-module $\bar{M}^\reg$. It has regular
singularities and taking horizontal sections, we obtain an
$(RB_W)$-module $\mathrm{dR}(\bar{M}^\reg)_{\bar{v}_\CM}$,
finitely generated as an
$R$-module. The action of
$RB_W$ on $\mathrm{dR}(\bar{M}^\reg)_{\bar{v}_\CM}$
factors through an action of
$R\heckegenerique$: the resulting $(R\heckegenerique)$-module is
$\KZ(M)$\indexnot{KZ}{\KZ}.

\medskip
The KZ functor satisfies a ``Double Endomorphism Theorem''
\cite[Theorems 5.14 and 5.16]{ggor}.

\begin{theo}
\label{th:doubleendo}
The functor $\KZ:\dot{\OC}(R)\to (R\heckegenerique)\mmod$ is exact and its
restriction to $\Proj(\dot{\OC}(R))$ is fully faithful.
It induces an isomorphism 
$$\Zrm(\dot{\OC}(R))\xrightarrow{\sim} \Zrm(R\heckegenerique)$$
and an equivalence
$$\dot{\OC}(R)/\{M|M^\reg=0\}\xrightarrow{\sim}(R\heckegenerique)\mmod.$$
\end{theo}

\subsection{Semi-simplicity}
\label{se:semisimplicity}

Let $c\in\CCB(\CM)$ and $q=\exp(c)$. Let $\CM_c=\CM$ \indexnot{Cc}{\CM_c}
with the $\CM[\CCB]$-algebra structure
given by $C\mapsto c$. Similarly, 
let $\CM_q=\CM$ \indexnot{Cq}{\CM_q} with the $\CM[\TCB]$-algebra structure
given by $q=\exp(c)$.

Theorem \ref{th:doubleendo} shows that the semisimplicity of
$\dot{\OC}(\CM_c)$ is equivalent to that of
$\CM_q\heckegenerique$.
From Theorem \ref{th:semisimplicity}, we deduce the following
\cite[Proposition 5.4]{rouquier schur}. Note that this result is equivalent to
a result of Chlouveraki on Schur elements obtained independently 
\cite[Theorem~4.2.5]{chlouveraki LNM}.

\begin{coro}
\label{co:semisimplicityHecke}
If the subgroup of $\CM^\times$ generated by
$\{q_{\orbite,j}^{|\mub_W|}\}_{(\orbite,j)\in \orbiteb^\circ}$ is torsion-free,
then $\CM_q\heckegenerique$ is semisimple.
\end{coro}

\begin{proof}
Let $k=\kappa(c)$ and let
$\Gamma_0$ be the subgroup of $\CM$ generated by the $k_{\orbite,j}$'s for
$(\orbite,j)\in\orbiteb^\circ$.
By assumption, $\Gamma_0/(\ZM\cap\Gamma_0)$ is torsion free. As a consequence, there is a subgroup
$\Gamma'$ of $\Gamma_0$ such that $\Gamma_0=\Gamma'\times (\ZM\cap\Gamma_0)$. Let $p:\Gamma_0\to
\Gamma'$ be the projection, let $k'=p(k)$ and let $c'=\kappa^{-1}(k')$.
Theorem \ref{th:semisimplicity} shows that $\dot{\OC}(\CM_{c'})$ is
semisimple, hence
$\CM_q\heckegenerique$ is semisimple, since
	$\CM_{\exp(c')}\heckegenerique\simeq \CM_{\exp(c)}\heckegenerique$.
\end{proof}

Corollary \ref{co:semisimplicityHecke} provides an improvement
of Theorem \ref{th:semisimplicity} using that the semisimplicity of
$\CM_q\heckegenerique$ implies that of $\dot{\OC}(\CM_c)$ 
\cite[Proposition 5.4]{rouquier schur}.

\begin{coro}
\label{co:semisimplicityO}
The prime ideals in $\FC_1$ correspond to affine hyperplanes of
$\KCB$ of the form $\sum_{(\orbite,j)\in\orbiteb^\circ}a_{\orbite,j}
K_{\orbite,j}=\frac{b}{r}$ for some $a\in\ZM^{\orbiteb^\circ}$
with $\gcd(\{a_{\orbite,j}\})=1$ and $r,b\in\ZM$, $r\ge 2$, $b\ge 1$ and
$\gcd(r,b)=1$.

If $(\QM \setminus \ZM)\cap
\bigl(\sum_{(\orbite,j)\in\orbiteb^\circ}\ZM k_{\orbite,j}\bigr)=\emptyset$, then
$\dot{\OC}(\CM_c)$ is semisimple.
\end{coro}

Note that the affine hyperplane of $\KCB$ coming from a prime ideal $\CG$ in
$\FC_1$ has a unique equation of the form 
$\sum_{(\orbite,j)\in\orbiteb^\circ}a_{\CG,\orbite,j}
K_{\orbite,j}=\frac{b_\CG}{r_\CG}$ as in the corollary.

We define a map $\FC_1\to \ZM^{\orbiteb^\circ}$ by setting
$m(\CG)_{\orbite,j}=a_{\CG,\orbite,j}$.
The elements of $\ZM^{\orbiteb^\circ}$ can be viewed
as functions on $(\CM^\times)^{\orbiteb^\circ}$.
The previous results have the following corollary.

\begin{coro}
\label{co:Heckesingular}
The algebra $\CM_q\heckegenerique$ is semisimple if and only if
	$m(\CG)(q)\neq e^{2i\pi |\mub_W|^{-1}b_\CG/r_\CG}$ for all $\CG\in\FC_1$.
\end{coro}

Note that the set $\{(m(\CG),|\mub_W|^{-1}b_\CG/r_\CG\pmod\ZM)\}_{\CG\in\FC_1}$
is finite, so that Corollary \ref{co:Heckesingular} provides a finite
set of conditions. Given $\CG\in\FC_1$, let $\Psi_\CG\in 
F[t]$ be the minimal polynomial of 
$e^{2i\pi |\mub_W|^{-1}b_\CG/r_\CG}$.

\bigskip
\section{Representations}\label{section:representations-hecke}

\medskip

The following result is due to Malle~\cite[Theorem~5.2]{malle}, the difficulty being the statement on splitting.

\bigskip

\begin{theo}[Malle]\label{theo:hecke-deployee}
The $F(\TCB)$-algebra $F(\TCB)\heckegenerique$ is split semisimple.
\end{theo}

\bigskip

Since the algebra $FW$ is also split semisimple (by Benard-Bessis Theorem~\ref{deploiement}), 
it follows from Tits Deformation Theorem~\cite[Theorem~7.4.6]{geck} that we have a bijective map 
$$
\begin{array}{ccc}
\Irr(W) & \longiso & \Irr(F(\TCB)\heckegenerique) \\
E & \longmapsto & E^\gen
\end{array}
$$
defined by the following property: the character of $E$ is the specialization
of the character of $E^\gen$ 
through $\qb_{\orbite,j} \mapsto 1$.

\bigskip

Let $\omeb_E^\gen : \Zrm(F(\TCB)\heckegenerique) \longto F(\TCB)$ \indexnot{ozz}{\omeb_E^\gen}
denote the central character associated with the representation $E$: 
given $a \in \Zrm(F(\TCB)\heckegenerique)$, 
we define $\omeb_E^\gen(a)$ as the element of $F(\TCB)$ by which $a$ acts 
on $E^\gen$. This is a morphism of $F(\TCB)$-algebras. 
Since $\OC[\TCB]$ is integrally closed, $\omeb_E^\gen$ 
restricts to a morphism of $\OC[\TCB]$-algebras 
$\omeb_E^\gen : \Zrm(\heckegenerique) \longto \OC[\TCB]$. 
We denote by $\omega_E:\Zrm(\OC W)\to\OC$\indexnot{omegaE}{\omega_E}
 the usual central character
(specialization of $\omeb_E^\gen$ at $\qb=1$).

\smallskip
The image of
$\pib_z \in \Zrm(B_W)$ in $\heckegenerique$ belongs to the 
center of this algebra. Hence, one can evaluate $\omeb_E^\gen$ at $\pib_z$
and we recover the formula of \cite[Proposition~4.16]{broue-michel}.

\begin{prop}
\label{pr:action-pi}
Assume $W$ is irreducible.
Given $E\in\Irr(W)$, we have
$$\omeb_E^\gen(\pib_z) = \omega_E(w_z)
\prod_{(\orbite,j) \in \orbiteb^\circ} 
\qb_{\orbite,j}^{\frac{|\mub_W|}{z_W}
\frac{\SS{m_{\orbite,j}^E|\orbite|e_\orbite}}{\SS{\dim E}}}.$$
\end{prop}

\begin{proof}
	Let $\CG$ be the  maximal ideal ideal of $\CM[\CCB]$ corresponding to
	the point $c$ and let $R$ the
	completion of $\CM[\CCB]_\CG$.
The element $w_ze^{2i\pi\euler/z_W}$
	acts on objects $M$ of $\dot{\OC}(R)$ and
	defines an element of $\Zrm(\dot{\OC}(R))^\times$.
Its action on $\KZ(M)$ is given by $\pib_z$. We deduce from
Proposition \ref{pr:OmegaEuler} that
$\pib_z$ acts on $\KZ(\Delta_c(E))$ by $\omega_E(w_z)
e^{2i\pi c_E/z_W}$ and the proposition follows from Lemma \ref{le:cchi}.
\end{proof}

\bigskip

The following result is a consequence of Corollary \ref{co:semisimplicityHecke}.

\begin{coro}\label{coro:hecke-deployee}
The $F(\qb^\CM)$-algebra $F(\qb^\CM)\heckecyclotomique(k)$ 
is split semisimple. 
\end{coro}

\bigskip

By Tits Deformation Theorem, we get a sequence of bijective maps 
$$\begin{array}{ccccc}
\Irr(W) &\longiso & \Irr(F(\qb^\CM)\heckecyclotomique(k)) & \longiso & \Irr(F(\TCB)\heckegenerique) \\
\chi & \longmapsto & \chi_k^\cyclo & \longmapsto & \chi^\gen
\end{array}\indexnot{kz}{\chi^\gen,~\chi_k^\cyclo} 
$$
such that $\chi_k^\cyclo = \Th_k^\cyclo \circ \chi^\gen$. 

Finally, let $\o_{\chi,k}^\cyclo : \Zrm(\heckecyclotomique(k)) \longto \OC[\qb^\CM]$ 
denote the central character associated with $\chi_k^\cyclo$. \indexnot{ozz}{\o_{\chi,k}^\cyclo} 
It follows from Proposition \ref{pr:action-pi} that when $W$ is irreducible, we
have
\equat\label{eq:action-pi-cyclo}
\o_{\chi,k}^\cyclo(\pib) = \qb^{|\mub_W| C_\chi(k)}.
\endequat

\bigskip

\section{Hecke families}
\label{se:Heckefamilies}

\subsection{Definition}\label{sub:rouquier}
Let $\OC^\cyclo[\qb^\CM]$, 
\indexnot{O}{\OC^\cyclo[\qb^\CM]} denote the ring 
$$\OC^\cyclo[\qb^\CM] = \OC[\qb^\CM][\bigl((1-\qb^r)^{-1}\bigr)_{r \in \CM\setminus\{0\}}].$$
Given $b$  a central idempotent (not necessarily primitive) 
of $\OC^\cyclo[\qb^\CM]\heckecyclotomique(k)$, we denote by 
$\Irr_\HC(W,b)$ \indexnot{I}{\Irr_\HC(W,b)}  the set of irreducible
representations $E$ of $W$ such that $E_k^\cyclo \in \Irr(F(\qb^\CM)\heckecyclotomique(k)b)$. 

\bigskip

\begin{defi}\label{defi:famille-rouquier}
A {\bfit Hecke $k$-family} is a subset of $\Irr(W)$ of the form 
$\Irr_\HC(W,b)$, where $b$ is a primitive central idempotent of $\OC^\cyclo[\qb^\CM]\heckecyclotomique(k)$.
\end{defi}

\bigskip

The Hecke $k$-families form a partition of $\Irr(W)$.

\bigskip

\begin{lem}[Brou\'e-Kim]\label{lem:c-contant}
If $E$ and $E'$ are in the same Hecke $k$-family, then 
$C_E(k)=C_{E'}(k)$.
\end{lem}

\begin{proof}
We could apply the argument contained in~\cite[Proposition~2.9(2)]{broue kim}. 
However, our framework is slightly different and we propose 
a different proof, based on the particular form of $\o_{E,k}^\cyclo(\pib)$ 
(see~\ref{eq:action-pi-cyclo}). 

\smallskip
	Remark \ref{re:CEproduct} shows that it is enough to prove the lemma for
$W$ irreducible. So we assume now that $W$ is irreducible.

Let $\EC=\{r_1,r_2,\dots, r_m\}$, with $r_i \neq r_j$ if $i \neq j$, 
denote the image of the map $\Irr(W) \to \CM$, $E \mapsto |\mub_W|C_E(k)$. 
Given $1 \le j \le m$, we set 
	$$\FC(j)=\{E\in \Irr(W)~\bigl|~|\mub_W| C_E(k)=r_j\}.$$
Given $E\in \Irr(W)$, let $e_{E,k}$ denote the associated primitive central idempotent 
of $F(\qb^\CM)\heckecyclotomique(k)$. We set
	$$b_j = \sum_{E\in \FC(j)} e_{E,k}.$$
To show the lemma, it is sufficient to check that $b_j \in \OC^\cyclo[\qb^\CM]\heckecyclotomique(k)$. 
In $\OC^\cyclo[\qb^\CM]\heckecyclotomique(k)$, we have
$$\pib=\qb^{r_1} b_1 + \qb^{r_2} b_2 + \cdots + \qb^{r_m} b_m.$$
Hence, 
$$
\left\{\begin{array}{ccccccccc}
b_1 &+& b_2 &+& \cdots &+& b_m &=& 1\\
\qb^{r_1} b_1 &+& \qb^{r_2} b_2 &+& \cdots &+& \qb^{r_m} b_m &=&  \pib \\
 && \cdots & \\
\qb^{(m-1)r_1} b_1 &+&  \qb^{(m-1)r_2} b_2 &+& \cdots &+& \qb^{(m-1)r_m} b_m &=&  \pib^{m-1}.
\end{array}\right.
$$
The determinant of this system is a Vandermonde determinant, equal to
$$\prod_{1 \le i < j \le m} (\qb^{r_i}-\qb^{r_j}),$$
which is invertible in $\OC^\cyclo[\qb^\CM]$ by construction. 
Since $1$, $\pib$,\dots, $\pib^{m-1} \in \heckecyclotomique(k)$, 
the result follows.
\end{proof}

\bigskip

\subsection{Reflections of order $2$}\label{sec:ordre-2}

\medskip

In this section \S\ref{sec:ordre-2}, we assume that all the reflections of 
$W$ have order $2$. We explain results and constructions due to Maria Chlouveraki,
whom we thank for her explanations.

\bigskip

Let $\OC[\TCB] \to \OC[\TCB]$, $f \mapsto f^\dagger$ \indexnot{ZZZ}{\dagger}  
denote the unique involutive automorphism of $\OC$-algebra exchanging 
$\qb_{\orbite,0}$ and $\qb_{\orbite,1}$ for all $\orbite \in \AC/W$. Let 
$\OC[\TCB] B_W \to \OC[\TCB] B_W$, $a \mapsto a^\dagger$ denote also 
the unique semilinear (for the involution $f \mapsto f^\dagger$ of $\OC[\TCB]$) 
automorphism such that $\b^\dagger=\e(p_W(\b))\b$ for all $\b \in B_W$. 
The relations~(\ref{eq:relations Hecke}) are stable under this automorphism. 
So it induces a semilinear automorphism $\heckegenerique \to \heckegenerique$, 
$h \mapsto h^\dagger$ of the generic Hecke algebra. 

This automorphism, after the specialization $\qb_{\orbite,j} \mapsto 1$, becomes 
the unique $\OC$-linear automorphism of $\OC W$ which sends $w \in W$ to $\e(w)w$. In other words, 
it is the automorphism induced by the linear character $\e$. 
 
Similarly, since $k_{\orbite,0}+k_{\orbite,1}=0$, if we still denote by $\OC[\qb^\CM] \to \OC[\qb^\CM]$, $f \mapsto f^\dagger$ 
the unique automorphism of $\OC$-algebra such that $(\qb^r)^\dagger=\qb^{-r}$, then the 
specialization $\qb_{\orbite,j} \mapsto \qb^{k_{\orbite,j}}$ induces an $\OC[\qb^\CM]$-semilinear 
automorphism of the algebra $\heckecyclotomique$, still denoted by $h \mapsto h^\dagger$. 
If $\chi \in \Irr(W)$, let $(\chi^\gen)^\dagger$ (respectively $(\chi^\cyclo_k)^\dagger$) 
denote the composition of $\chi^\gen$ (respectively $\chi_k^\cyclo$) 
with the automorphism $\dagger$: it is a new irreducible character of $F(\TCB)\heckegenerique$ 
(respectively $F(\qb^\CM)\heckecyclotomique(k)$). 
Since it is determined by its specialization through $\qb_{\orbite,j} \mapsto 1$ 
(respectively $\qb^r \mapsto 1$), we have
\equat\label{eq:chi-dagger}
(\chi^\gen)^\dagger = (\chi\e)^\gen\qquad\text{and}\qquad (\chi^\cyclo_k)^\dagger = (\chi\e)^\cyclo_k.
\endequat
As the automorphism $f \mapsto f^\dagger$ of $\OC[\qb^\CM]$ extends to the ring
$\OC^\cyclo[\qb^\CM]$, the next lemma follows immediately.

\bigskip

\begin{lem}\label{lem:rouquier-ordre-2}
Assume that all the reflections of $W$ have order $2$. If 
$\FC$ is a Hecke $k$-family, then $\FC\e$ is a Hecke $k$-family. 
\end{lem}

\bigskip

\subsection{About the coefficient ring}

\medskip

It might seem strange to work with such a large coefficient ring 
(far from being Noetherian for instance). A first argument for this choice 
is that this ring is still integral and integrally closed.

Moreover, this choice allows to work with a fixed ring, whatever the value of the parameter $k$ is:
as we let $k$ vary in a {\it real} vector space of parameters, this choice becomes more 
natural. Also, as it has been seen in Corollary~\ref{coro:hecke-deployee}, the fact that it is possible 
to extract $n$-th roots of all ``powers'' of $\qb$ implies immediately the splitness 
of all the cyclotomic Hecke algebras over the same fixed ring.

This ring is of the form $\OC[\G]$, where $\G$ is a totally ordered abelian group: 
this allows to define for instance, thanks to the notion of degree and valuation, 
the $\ab$ and $\Ab$-invariants associated with irreducible characters of $W$ 
(even though we will not used them in this book). 
Finally, as we will see in \S \ref{section:cellules-kl}, it is also 
the general framework for Kazhdan-Lusztig theory, which we aim to generalize 
to complex reflection groups. 

It is nevertheless necessary to compare the notion of Hecke families we have introduced 
in \S\ref{se:Heckefamilies} with the classical definitions. In order to do
so,
let $B$ be a commutative integral $\OC[\TCB]$-algebra, with fraction 
field $F_B$.
Let $L_B$ denote the subgroup of $B^\times$ generated by
$\{1_B\qb_{\orbite,j}\}_{(\orbite,j)\in\orbiteb^\circ}$,
a quotient of $\BZ^{\orbiteb^\circ}$.
We assume that $L_B$ has no torsion, that $\OC[L_B]$ embeds in $B$,
and that $F[L_B]\cap B=\OC[L_B]$.
As in Corollary~\ref{coro:hecke-deployee}, the $F_B$-algebra 
$F_B\HCB$ is split semisimple and we get a bijective map 
$\Irr(W)\longiso\Irr(F_B\HCB)$.

\bigskip

\begin{exemple}
\label{groupeabelien}
Let $\L$ be a torsion-free abelian group. 
As in Lemma~\ref{lem:A-normal}, note that $\OC[\L]$ is integrally closed.
Let $q:\orbiteb^\circ\to \L$ be a map. It extends to a morphism of groups 
$\BZ^{\orbiteb^\circ}\to \L$ and to a morphism between group algebras 
 $\OC[\TCB]\to \OC[\L]$. The algebra $B=\OC[\L]$ satisfies the previous 
assumption.\finl
\end{exemple}

\bigskip

Let $B^\cyclo=B[(1-v)_{v\in L_B-\{0\}}^{-1}]$ (we write additively the abelian group $L_B$). 
As in~\S\ref{sub:rouquier}, we have a notion of {\it Hecke $B$-family}.

\bigskip

\begin{prop}
\label{anneaubase}
The Hecke $B$-families coincide with the Hecke
$\OC[L_B]$-families.
\end{prop}

\begin{proof}
Let $\L=L_B$.
We have $B^\cyclo\cap F_{\OC[\L]}=\OC[\L]^\cyclo$. We deduce that, 
if $b$ is a primitive central idempotent of $F_{\OC[\L]}\HCB$
such that $b\in B^\cyclo\HCB$, then $b\in \OC[\L]^\cyclo\HCB$.
\end{proof}

\bigskip

The previous proposition reduces the study of Hecke families to the case where 
$B=\OC[\L]$ and $\L$ is a torsion-free quotient of $\BZ^{\orbiteb^\circ}$.

\medskip
Let $\MCB=m(\FC_1)$, a finite subset 
$\BZ^{\orbiteb^\circ}$ (cf. end of \S\ref{se:semisimplicity}).
Consider now a torsion-free abelian group $\L$ and a map $q:\orbiteb^\circ\to \L$ 
as in Example~\ref{groupeabelien}. Let $\L'$ be a torsion-free abelian group and let 
$f:\L\to \L'$ be a surjective morphism of groups. 

\begin{prop}
\label{quotientabelien}
If $q(\MCB)\cap \Ker f\subset\{0\}$, then the Hecke $\OC[\L]$-families 
coincide with the Hecke $\OC[\L']$-families.
\end{prop}

\begin{proof}
The morphism $f$ induces a surjective morphism between group algebras 
$\OC[\L]\to\OC[\L']$ which extends to a surjective morphism of algebras between localizations 
$f:\OC[\L][(1-v)^{-1}_{v\in \L-\Ker f}]\to \OC[\L']^\cyclo$.
	Let $h\in F[\L][\{(\Psi_\CG\circ m(\CG))^{-1}\}_{\CG\in\FC_1,m(\CG)\not\in\Ker q}]$
(cf. end of \S\ref{se:semisimplicity}).
If $h\in \OC[\L]^\cyclo$, then $h\in f^{-1}(\OC[\L']^\cyclo)$.

It follows from Corollary \ref{co:Heckesingular}
that the idempotents of 
$\Zrm(F(\L)\HCB)$ are in the algebra 
$F[\L][\{\Psi_\CG(m(\CG))^{-1}\}_{\CG\in\FC_1,m(\CG)\not\in\Ker q}]
\HCB$. Consequently, any idempotent of $\Zrm(\OC[\L]^\cyclo\HCB)$ is contained
in $\OC[\L][(1-v)^{-1}_{v\in \L-\Ker f}]\HCB$. 
Proposition~\ref{muller} shows that the central idempotents of 
$\OC[\L][(1-v)^{-1}_{v\in \L-\Ker f}]\HCB$ are in bijection 
with those of $\OC[\L']^\cyclo\HCB$ and the result follows.
\end{proof}

\bigskip

Given $\L$ and $q$ as above, there exists a morphism of groups 
$f:\L\to\BZ$ such that $\Ker f\cap q(\MCB)\subset\{0\}$. So Proposition 
\ref{quotientabelien} reduces the study of Hecke $\OC[\L]$-families 
(and so of Hecke $B$-families, by the above arguments) 
to the case of Hecke $\OC[t^{\pm 1}]$-families, for a choice of integers 
$m_{\orbite,j}\in\BZ$ defining a morphism of groups 
$\BZ^{\orbiteb^\circ}\to t^\BZ,\ q_{\orbite,j}\mapsto t^{m_{\orbite,j}}$. 
This is the usual framework for Hecke families.

\section{Kazhdan-Lusztig cells}\label{section:cellules-kl}

\medskip
The constructions and results of \S\ref{section:cellules-kl} are due
to Kazhdan-Lusztig \cite{KL} (equal parameter case) 
and Lusztig \cite{Lu2} (general case).

\medskip

\cbstart
\boitegrise{{\bf Assumption.} 
{\it From now on, and until the end of \S \ref{section:cellules-kl}, we 
assume that $W$ is a Coxeter group and we
fix a family $k=(k_{\orbite,j})_{\orbite \in \AC/W, j \in \{0,1\}} \in \KCB(\RM)$. 
We denote by $c : \REF(W) \to \RM$, \indexnot{ca}{c} the map defined 
by $c_{s_H}=-2k_{H,0}$ for all $H \in \AC$. It is constant on conjugacy 
classes and recall that we write $k_s=-c_s/2$ for $s \in \REF(W)$.}}{0.75\textwidth}

\bigskip

Giving the map $c : \REF(W) \to \RM$ constant on conjugacy classes of reflections 
is equivalent to giving the family $k \in \KCB(\RM)$. 

\bigskip

\subsection{Kazhdan-Lusztig basis} 
The involution $a \mapsto \aba$ of $\OC[\qb^\RM]$ extends to an $\OC[\qb^\RM]$-semilinear 
involution of the algebra $\heckecyclotomique(k)$ by setting 
$$\overline{T}_w = T_{w^{-1}}^{-1}.\indexnot{T}{\Tov_w}$$
If $\XM$ is a subset of $\RM$, we set $\OC[\qb^\XM]=\mathop{\bigoplus}_{r \in \XM} \OC~\qb^r$. 
\indexnot{O}{\OC[\qb^\XM]}  
We also set 
$$\heckecyclotomique(k)_{> 0}=\mathop{\bigoplus}_{w \in W} \OC[\qb^{\RM_{>0}}]~T_w.
\indexnot{H}{\heckecyclotomique(k)_{>0}}$$

\bigskip

\begin{theo}[Kazhdan-Lusztig]\label{theo:base-kl}
For $w \in W$, there exists a unique $C_w \in \heckecyclotomique(k)$ 
\indexnot{C}{C_w}  such that 
$$
\begin{cases}
\overline{C}_w  =  C_w, \\
C_w \equiv T_w \mod \heckecyclotomique(k)_{> 0}.
\end{cases}
$$
The family $(C_w)_{w \in W}$ is an $\OC[\qb^\RM]$-basis of $\heckecyclotomique(k)$. 
\end{theo}

\bigskip

Note that $C_w$ depends only on $k$ (i.e., on $c$). For example, if $s \in S$, then 
$$C_s=
\begin{cases}
T_s - \qb^{2k_s} & \text{if $k_s > 0$,}\\
T_s & \text{si $k_s=0$,}\\
T_s + \qb^{-2k_s} & \text{if $k_s < 0$.}
\end{cases}
$$
Similarly, as well as $T_w$, $C_w$ depends on the choice of $S$. The basis 
$(C_w)_{w \in W}$ is called the {\it Kazhdan-Lusztig basis} of 
$\heckecyclotomique(k)$.

\bigskip

\subsection{Definition} 
For $x$, $y \in W$, we will write $x \rell y$ \indexnot{ZZZ}{\rell,~\relr}  
if there exists $h \in \heckecyclotomique(k)$ 
such that $C_x$ appears with a non-zero coefficient in the decomposition of 
$hC_y$ in the Kazhdan-Lusztig basis. Let $\prel$ \indexnot{ZZZ}{\prel,~\prer,~\prelr}  
denote the transitive closure of this relation; 
it is a pre-order and we denote by $\siml$ 
\indexnot{ZZZ}{\siml,~\simr,~\simlr}  the associated equivalence relation. 

We define similarly $\relr$ by multiplying on the right by $h$ as well as 
$\prer$ and $\simr$. Let $\prelr$ be the transitive relation 
generated by $\prel$ and $\prer$, and let $\simlr$ denote the associated equivalence 
relation.

\bigskip

\begin{defi}\label{defi:cellules}
We call {\bfit Kazhdan-Lusztig left $c$-cells} of $W$ the equivalence classes 
for the relation $\siml$. We define similarly {\bfit Kazhdan-Lusztig right $c$-cells} 
and {\bfit Kazhdan-Lusztig two-sided $c$-cells} using $\simr$ and $\simlr$ respectively.
If $? \in \{L,R,LR\}$, we denote by $\klcellules_?^c(W)$ 
the corresponding set of Kazhdan-Lusztig $c$-cells in $W$.  
\end{defi}

\bigskip

If $? \in \{L,R,LR\}$ and if $\G$ is an equivalence class for the relation 
$\sim_?^{\kl,c}$ (that is, a Kazhdan-Lusztig $c$-cell of the type associated with $?$), we set
$$\heckecyclotomique(k)_{\leqslant_?^{\kl,c} \G} = 
\mathop{\bigoplus}_{w \leqslant_?^{\kl,c} \G} \OC[\qb^\RM]~C_w
\qquad\text{and}\qquad
\heckecyclotomique(k)_{<_?^{\kl,c} \G} = 
\mathop{\bigoplus}_{w <_?^{\kl,c} \G} \OC[\qb^\RM]~C_w,
\indexnot{H}{\heckecyclotomique(k)_{\leqslant_?^{\kl,c} \G},~
\heckecyclotomique(k)_{<_?^{\kl,c} \G}}
$$
as well as 
$$\MC_\G^? = \heckecyclotomique(k)_{\leqslant_?^{\kl,c} \G}/\heckecyclotomique(k)_{<_?^{\kl,c} \G}.$$
By construction, $\heckecyclotomique(k)_{\leqslant_?^{\kl,c} \G}$ and 
$\heckecyclotomique(k)_{<_?^{\kl,c} \G}$ are 
ideals (left ideals if $?=L$, right ideals if $?=R$ or two-sided ideals if $?=LR$) 
and $\MC_\G^?$ is a left (respectively right) $\heckecyclotomique(k)$-module if $?=L$ 
(respectively $?=R$), or an $(\heckecyclotomique(k),\heckecyclotomique(k))$-bimodule 
if $?=LR$. Note that 
\equat\label{eq:liberte-cellulaire}
\text{\it $\MC_\G^?$ is a free $\OC[\qb^\RM]$-module with basis the 
image of $(C_w)_{w \in \G}$.}
\endequat

\bigskip

\begin{defi}\label{defi:cellulaires-familles}
If $C$ is a Kazhdan-Lusztig left $c$-cell of $W$, we denote by 
$\isomorphisme{C}_c^\kl$ \indexnot{C}{\isomorphisme{C}_c^\kl}  the class 
of $\kb \otimes_{\OC[\qb^\RM]} \MC_C^L$ in the Grothendieck group $\groth(\kb W)=\BZ\Irr(W)$ 
(here, the tensor product $\kb \otimes_{\OC[\qb^\RM]} -$ is viewed through the 
specialization $\qb^r \mapsto 1$). We will call {\bfit $c$-cellular $\kl$-character} of $W$ 
every character of the form $\isomorphisme{C}_c^\kl$, where $C$ is a left 
Kazhdan-Lusztig $c$-cell. 

If $\G$ is a Kazhdan-Lusztig two-sided $c$-cell of $W$, 
we denote by $\Irr_\G^\kl(W)$ \indexnot{I}{\Irr_\G^\kl(W)}  the set of irreducible 
characters of $W$ appearing in $\kb \otimes_{\OC[\qb^\RM]} \MC_\G^{LR}$, 
viewed as a left $\kb W$-module. We will call {\bfit Kazhdan-Lusztig $c$-family} 
every subset of $\Irr(W)$ of the form $\Irr_\G^\kl(W)$ where $\G$ is a Kazhdan-Lusztig 
two-sided $c$-cell. We will say that $\Irr_\G^\kl(W)$ is the Kazhdan-Lusztig $c$-family 
{\bfit associated} with $\G$, or that $\G$ is the Kazhdan-Lusztig two-sided $c$-cell 
{\bfit covering} $\Irr_\G^\kl(W)$. 
\end{defi}

\bigskip

Since $\kb W$ is semisimple and since 
$\kb \otimes_{\OC[\qb^\RM]} \MC_\G^{LR}$ is a quotient of two-sided ideals of $\kb W$, 
the Kazhdan-Lusztig $c$-families form a partition of $\Irr(W)$ 
\equat\label{eq:partition-familles}
\Irr(W) = \mathop{\coprod}_{\G \in \klcellules_{LR}^c(W)} \Irr_\G^\kl(W)
\endequat
and, since $\kb W$ is split,
\equat\label{eq:cardinal-cellule-kl}
|\G|=\sum_{\chi \in \Irr_\G^\kl(W)} \chi(1)^2.
\endequat
Moreover, if $C$ is a Kazhdan-Lusztig {\it left} $c$-cell of $W$, we set  
$$\isomorphisme{C}_c^\kl = \sum_{\chi \in \Irr(W)} 
\mult_{C,\chi}^\kl~\chi,\indexnot{ma}{\mult_{C,\chi}^\kl}$$
where $\mult_{C,\chi}^\kl \in \NM$. 
Then:

\bigskip

\begin{lem}\label{lem:mult-kl}
With the previous notation, we have:
\begin{itemize}
\itemth{a} If $C\in \klcellules_L^c(W)$, then 
$\DS{\sum_{\chi \in \Irr(W)} \mult_{C,\chi}^\kl~\chi(1)=|C|}$.

\smallskip

\itemth{b} If $\chi \in \Irr(W)$, then 
$\DS{\sum_{C \in \klcellules_L^c(W)} \mult_{C,\chi}^\kl = \chi(1)}$. 
\end{itemize}
\end{lem}

\begin{proof}
The equality (a) simply says that the dimension of $\isomorphisme{C}_c^\kl$ 
is equal to $|C|$ by~(\ref{eq:liberte-cellulaire}). The equality~(b) 
translates the fact that, since $W$ is a disjoint union of Kazhdan-Lusztig left $c$-cells, we have 
$\isomorphisme{\kb W}_{\kb W} = \sum_{C \in \klcellules_L^c(W)} \isomorphisme{C}_c^\kl$. 
\end{proof}

\subsection{Properties of cells}\label{sub:proprietes-cellules} 
The algebra $\heckecyclotomique(k)$ is endowed with an $\OC[\qb^\RM]$-linear involutive 
anti-automorphism which sends $T_w$ on $T_{w^{-1}}$: it will be denoted by $h \mapsto h^*$. 
\indexnot{ZZZ}{*}  It is immediate that
\equat\label{eq:cw-etoile}
C_w^*=C_{w^{-1}},
\endequat
which implies that, if $x$ and $y$ are two elements of $W$, then 
\equat\label{eq:rell-relr}
\text{\it $x \prel y$ if and only if $x^{-1} \prer y^{-1}$}
\endequat
and so  
\equat\label{eq:siml-simr}
\text{\it $x \siml y$ if and only if $x^{-1} \simr y^{-1}$.}
\endequat
In other words, the map $\klcellules_L^c(W) \to \klcellules_R^c(W)$, 
$\G \mapsto \G^{-1}$ is well-defined and bijective.

The next property is less obvious~\cite[Corollary~11.7]{lusztig}
\equat\label{eq:rel-w0}
\text{\it $x \leqslant_?^c y$ if and only if $w_0 y \leqslant_?^c w_0 x$ 
if and only if $y w_0 \leqslant_?^c x w_0$.}
\endequat
It follows that
\equat\label{eq:sim-w0}
\text{\it $x \sim_?^c y$ if and only if $w_0x \sim_?^c w_0y$ if and only if $xw_0 \sim_?^c yw_0$.}
\endequat
Moreover, if $C \in \klcellules_L^c(W)$, then~\cite[Proposition~21.5]{lusztig}
\equat\label{eq:caractere-cellulaire-w0}
\isomorphisme{w_0 C}_c^\kl = \isomorphisme{C w_0}_c^\kl =\isomorphisme{C}_c^\kl \cdot \e.
\endequat
Similarly, if $\G \in \klcellules_{LR}^c(W)$, then~\cite[proposition~21.5]{lusztig}
\equat\label{eq:familles-cellulaire-w0}
\Irr_{w_0 \G}^\kl(W)=\Irr_{\G w_0}^\kl(W) = \Irr_\G^\kl(W) \cdot \e.
\endequat
This shows in particular that 
\equat\label{eq:w0gw0}
w_0\G w_0 = \G.
\endequat
So tensoring by $\e$ induces a permutation of Kazhdan-Lusztig $c$-families 
and of $c$-cellular $\kl$-characters.

If $\g : W \to \kb^\times$ is a linear character (note that 
$\g$ has values in $\{1,-1\}$), we set $\g\cdot c : \REF(W) \to \RM$, \indexnot{gz}{\g \cdot c}  
$s \mapsto \g(s)c_s$. The following Lemma is proven 
in~\cite[Corollary~2.5~and~2.6]{bonnafe continu}:

\bigskip

\begin{lem}\label{lem:gamma-c-kl}
If $\g \in W^\wedge$ and let $? \in \{L,R,LR\}$. Then:
\begin{itemize}
\itemth{a} The relations $\leqslant_?^c$ and $\leqslant_?^{\g \cdot c}$ coincide.

\itemth{b} The relations $\sim_?^{\kl,c}$ and $\sim_?^{\kl,\g \cdot c}$ coincide.

\itemth{c} If $C \in \klcellules_L^c(W)=\klcellules_L^{\g \cdot c}(W)$, then 
$\isomorphisme{C}_{\g \cdot c}^\kl = \g \cdot \isomorphisme{C}_c^\kl$. 

\itemth{d} If $\G \in \klcellules_{LR}^c(W)=\klcellules_{LR}^{\g \cdot c}(W)$, 
then $\Irr_{\G}^{\kl,\g \cdot c}(W)=\g \cdot \Irr_{\G}^{\kl,c}(W)$.
\end{itemize}
\end{lem}

\bigskip

The next result is easy~\cite[Lemma~8.6]{lusztig}:

\bigskip

\begin{lem}\label{lem:1-cellule}
Assume that $c_s \neq 0$ for all $s \in \REF(W)$. Then:
\begin{itemize}
 \itemth{a} $\{1\}$ and $\{w_0\}$ are Kazhdan-Lusztig left, right or two-sided 
$c$-cells. 

\itemth{b} Let $\g : W \to \kb^\times$ be the unique linear 
character such that $\g(s)=1$ if $k_s > 0$ and $\g(s)=-1$ if $k_s < 0$. Then 
$\isomorphisme{1}_c^\kl = \g$ and $\isomorphisme{w_0}_c^\kl=\g\e$.
\end{itemize}
\end{lem}

\bigskip

\begin{rema}\label{rem:lusztig-positif}
In fact,~\cite[Lemma~8.6]{lusztig} is proven whenever $k_s > 0$ for all $s$. 
To obtain the general statement of Lemma~\ref{lem:1-cellule}, it is sufficient 
to apply Lemma~\ref{lem:gamma-c-kl}.\finl
\end{rema}

\bigskip

\begin{exemple}[Vanishing parameters]\label{exemple:c=0}
If $c=0$ (i.e. if $c_s=0$ for all $s$), then $C_w=T_w$, $\heckecyclotomique(0)=\OC[\qb^\RM][W]$ 
and there is only one Kazhdan-Lusztig left, right or two-sided $0$-cell, namely $W$.  
We then have $\Irr_W^{\kl,0}(W)=\Irr(W)$ and 
$\isomorphisme{W}_0^{\kl} = \sum_{\chi \in \Irr(W)} \chi(1)\chi$.\finl
\end{exemple}

\cbend

\bigskip

\cbend

\chapter{Restricted Cherednik algebra and Calogero-Moser families}\label{chapter:bebe-verma}

In this chapter, we start by recalling in \S\ref{se:represtricted} and 
\S\ref{section:familles CM}
some results of Gordon~\cite{gordon} 
on the representations of the restricted Cherednik algebras. We do not need
to extend the defining field for representations, as the algebras are split.
This is useful to
derive consequences about the partition into Calogero-Moser two-sided cells
(\S~\ref{chapter:bilatere}). 

\bigskip

\boitegrise{{\bf Notation.}\label{notationscellsVerma} {\it
Given $\CG$ a prime ideal of
$\kb[\CCB]$, we denote by
$\CCB(\CG)=\Spec \kb[\CCB]/\CG$ \indexnot{C}{\CCB(\CG)} the
closed irreducible subscheme of $\CCB$ defined by $\CG$.
We denote by $\pGba_\CG$ \indexnot{pa}{\pGba_\CG}  
the prime ideal of $P$ corresponding to the closed irreducible subscheme 
$\CCB(\CG) \times \{0 \} \times \{0\}$.
We set 
$$
\Pba_\CG=P/\pGba_\CG \simeq \kb[\CCB]/\CG
\indexnot{P}{\Pba_\CG} 
$$
and we define 
the $\Pba_\CG$-algebra $\Hbov_\CG=\Hb/\pGba_\CG \Hb$.
\indexnot{H}{\Hbov_\CG} 
We denote by $\Zba_\CG$ 
the image of $Z$ in $\Hbov_\CG$.  \indexnot{Z}{\Zba_\CG}
We also define 
$\Kbov_\CG=k_P(\pGba_\CG)$.
\indexnot{K}{\Kbov_\CG}
\hphantom{A} To simplify the notation, when $\CG=0$,
the index $\CG$ will be omitted 
in all the previous notations ($\Pba$, $\pGba$, $\Kbov)$. Given
$c \in \CCB$ and $\CG=\CG_c$, the index $\CG_c$ will be replaced by $c$   
\indexnot{pa}{\pGba,  \pGba_c}  
\indexnot{P}{\Pba,  \Pba}  
\indexnot{H}{\Hbov, \Hbov_c}  
\indexnot{K}{\Kbov, \Kbov_c} 
\indexnot{Z}{\Zba,\Zba_c}
($\Kbov_c$,\dots). 
Note for instance that $\pGba_\CG=\pGba + \CG~P$ and that $\Kbov_c \simeq \kb$.}}{0.75\textwidth}

\section{Representations of restricted Cherednik algebras}
\label{se:represtricted}
The {\it restricted Cherednik algebra} is
the $\kb[\CCB]$-algebra $\Hbov$ defined by
$$\Hbov=\Hb/\pGba\Hb = \kb[\CCB] \otimes_P \Hb.$$
Theorem~\ref{PBW-0} gives a PBW decomposition for that algebra.

\bigskip

\begin{prop}\label{PBW restreint}
The map $\kb[\CCB] \otimes \kb[V]^\cow \otimes \kb W \otimes \kb[V^*]^\cow \to \Hbov$ 
induced by the product is an isomorphism of $\kb[\CCB]$-modules. In particular, 
$\Hbov$ is a free $\kb[\CCB]$-module of rank $|W|^3$.
\end{prop}

The algebra $\Hbov$ inherits an $(\NM \times \NM)$-grading, a $\ZM$-grading and an $\NM$-grading from $\Hbt$ 
(cf. \S \ref{section:graduation-1}).

 Given $E\in\Irr(\kb W)$, we put
$\bar{\Delta}(E)=\Delta(E)\otimes_P \kb[\CCB]$\indexnot{Dbar}{\bar{\Delta}(E)},
a $\ZM$-graded $\Hbov$-module.
Note that $\bar{\Delta}(E)$ is isomorphic to 
$\kb[\CCB] \otimes \kb[V]^\cow\otimes E$ as a graded
$(\kb[\CCB]W)$-module.

We put $\bar{\Delta}(\mathrm{co})=\Delta(\mathrm{co})\otimes_P \kb[\CCB]=
\Delta(\mathrm{co})\otimes_{\kb[V]^W}\kb=\Hbov e$.

\bigskip
Let $\CG$ be a prime ideal of $\CCB$. 
Theorem \ref{th:hwgraded} has the following consequence~\cite{BelTh1}.

\begin{theo}[Bellamy-Thiel]
$\Hbov_\CG\mmodgr$ is a split highest weight category over $\kb(\CG)$
with standard objects the baby Verma modules
$\bar{\Delta}_\CG(E)=\bar{\Delta}(E)\otimes_{\kb[\CCB]}\kb(\CG)$,
$E\in\Irr(\kb W)$, and their shifts.
\end{theo}

\begin{prop}\label{pr:verma}
Given $E\in \Irr(\kb W)$, the $\Hbov_\CG$-module $\bar{\Delta}_\CG(E)$
has a unique simple quotient $L_\CG(E)$. 
The map $\Irr(\kb W) \longto \Irr(\Hbov_\CG)$, $E \mapsto L_\CG(E)$ 
is bijective and the algebra $\Hbov_\CG$ is split. 
\end{prop}

\bigskip

\section{Calogero-Moser families}\label{section:familles CM} 

\medskip

Let $\th_\CG:\kb[\CCB]\to \kb(\CG)$ be the quotient map and let 
$\O_E^\CG=\th_\CG\circ\O_E =\th_\CG \circ \o_E \circ \Omeb : Z \to \kb(\CG)$ for
$E \in\Irr(W)$.

\bigskip

\begin{lem}\label{lem:action-z-verma}
If $z \in Z$, then $z$ acts by multiplication by $\O_E^\CG(z)$ on $L_\CG(E)$.
\end{lem}

\bigskip

\begin{proof}
We may assume that $z$ is $\ZM$-homogeneous of degree $i$. If $i=0$, then the lemma follows from 
Proposition~\ref{prop:action-z0}. If $i \neq 0$, then, as $L_\CG(E)$ is $\ZM$-graded 
and finite-dimensional, $z$ acts nilpotently on $L_\CG(E)$ and, as it also acts by a scalar, 
this scalar must be equal to $0$. But $\Omeb(z)=0$ by~(\ref{eq:omega-z0}), and the result follows. 
\end{proof}

\bigskip

Given $b \in \blocs(\Zrm(\Hbov_\CG))$, let $\Irr_\Hb(W,b)$
\indexnot{I}{\Irr_\Hb(W,b)}  
denote the set of irreducible representations $E$ of $W$ such that
$\bar{\Delta}_\CG(E)$ is in the block $\Hbov_\CG b$, 
i.e., $b\bar{\Delta}_\CG(E)\not=0$. Note that
$E\in\Irr_\Hb(W,b)$ if and only if $L_\CG(E)$ belongs to
$\Irr(\Hbov_\CG b)$.
It follows from Proposition~\ref{muller} that
$$\blocs(\Zba_\CG)=\blocs(\Zrm(\Hbov_\CG)).$$
So,
\equat\label{cm familles}
\Irr(W) = \coprod_{b \in \blocs(\Zba_\CG)} \Irr_\Hb(W,b).
\endequat
A {\it Calogero-Moser $\CG$-family}
is a subset of $\Irr W$ of the form 
$\Irr_\Hb(W,b)$, where $b \in \blocs(\Zba_\CG)$.

\medskip

The next lemma follows from Corollary~\ref{coro:r-blocs}, Proposition \ref{pr:verma} 
and Lemma~\ref{lem:action-z-verma}.

\bigskip

\begin{lem}\label{caracterization blocs CM}
Let $E$, $E' \in \Irr W$. Then $E$ and $E'$ are in the same 
Calogero-Moser $\CG$-family if and only if $\O_E^\CG=\O_{E'}^\CG$.
Moreover, the map 
\equat\label{...}
\fonction{\Th_\CG}{\Irr W}{\Upsilon^{-1}(\pGba_\CG)}{E}{\Ker~\O_E^\CG}
\indexnot{ty}{\Th_\CG}
\endequat
is surjective and its fibers are the Calogero-Moser $\CG$-families.
\end{lem}

\bigskip
Let $b \in \blocs(\Zba_\CG)$. We denote by $\O_b^\CG$ the common value of the $\O_E^\CG$ for
$E\in \Irr_\Hb(W,b)$. When $\CG=0$, we put $\O_b=\O_b^\CG$. When $\CG=\CG_c$ for some $c\in\CCB$, we put
$\O_b^c=\O_b^{\CG_c}$.

\begin{coro}\label{Q extention kc}
If $\zG$ is a prime ideal of $Z$ lying over $\pGba_\CG$, then the inclusion $P \longinjto Z$ 
induces an isomorphism $P/\pGba_\CG \longiso Z/\zG$.
\end{coro}

\begin{proof}
By Lemma~\ref{caracterization blocs CM}, there exists $E\in \Irr(W)$
such that 
$\zG=\Ker(\O_E^\CG)$. Since
$\O_E^\CG : Z \to \kb(\CG)$ factors through 
a surjective morphism $Z \to P/\pGba_\CG$, the corollary follows.
\end{proof}

\bigskip

\begin{exemple}\label{cm generique}
We will call {\it generic Calogero-Moser family} every Calogero-Moser 
$\CG'$-family, where $\CG'=0$. In this case, the map $\Th_\CG$ will be simply denoted by $\Th$.  \indexnot{ty}{\Th}  
Every Calogero-Moser $\CG$-family is a union of generic Calogero-Moser families.\finl
\end{exemple}

\bigskip

\begin{exemple}\label{cm c}
Given $c \in \CCB$, we define a {\it Calogero-Moser $c$-family} 
to be a Calogero-Moser $\CG_c$-family. In this case, $\O_\chi^{\CG_c}$ will
be denoted by 
$\O_\chi^c$   \indexnot{oz}{\O_\chi^c}  and $\Th_{\CG_c}$ will be denoted  \indexnot{ty}{\Th_c}  
by $\Th_c$.\finl
\end{exemple}

\begin{exemple}\label{ex:cm0}
We have $|\Upsilon_0^{-1}(0)|=1$, hence there is a unique Calogero-Moser
$0$-family.\finl
\end{exemple}

\section{Linear characters and Calogero-Moser families}\label{section:lineaire CM}

\bigskip

From Proposition \ref{omega lineaire} we deduce the following result.

\begin{prop}\label{coro:action-tau-L}
Given $E \in \Irr(W)$ and 
$\t=(\xi,\xi',\g \rtimes g) \in \kb^\times \times \kb^\times \times (W^\wedge \rtimes \NC)$ 
stabilizing $\CG$, we have
$$\lexp{\t}{L_\CG(E)} \simeq L_\CG(\lexp{g}{E}\otimes\g^{-1}).$$
\end{prop}

\bigskip
The following result follows from Proposition~\ref{omega lineaire} with $\xi=\xi'=1$.

\begin{coro}\label{gamma c familles}
Let $c \in \CCB$, let $\g$ be a linear character of $W$ and 
let $\FC$ be a Calogero-Moser $c$-family. Then $\FC \g$ is a 
Calogero-Moser $\g \cdot c$-family.
\end{coro}

\bigskip

Using Proposition~\ref{omega lineaire} again, we obtain the following result.

\begin{coro}\label{familles lineaires}
Let $\t=(\xi,\xi',\g \rtimes g) \in \kb^\times \times \kb^\times \times (W^\wedge \rtimes \NC)$ 
and let $\FC$ be a Calogero-Moser $\CG$-family. 
If $\t$ stabilizes $\CG$, then $\FC \g$ is a Calogero-Moser $\CG$-family.
\end{coro}

\bigskip

\begin{coro}\label{familles lineaires generiques}
Let $\g$ be a linear character of $W$ 
and let $\FC$ be a {\bfit generic} Calogero-Moser family. Then  
$\FC\g$ is a generic Calogero-Moser family.
\end{coro}

\bigskip

\begin{coro}\label{ordre 2}
Assume that all the reflections of $W$ have order $2$. Let $\FC$ be a Calogero-Moser  
$\CG$-family. Then $\FC \e$ is a Calogero-Moser $\CG$-family (recall that $\e$ 
is the determinant).
\end{coro}

\begin{proof}
The element $\t=(-1,1,\e \rtimes 1)$ of $\kb^\times \times \kb^\times \times (W^\wedge \rtimes \NC)$
acts trivially on $\kb[\CCB]$. The result follows now from
Corollary~\ref{familles lineaires}.
\end{proof}

\bigskip

\begin{exemple}[Generic families and linear characters]\label{lineaire}
Let $\g \in W^\wedge$ and $\chi \in \Irr(W)$ be in the same {\it generic} Calogero-Moser family.
Then $\O_\chi(\euler)=\O_\g(\euler)$, hence
$\chi(s)=\g(s) \chi(1)$ for all $s \in \REF(W)$.
In other words, all the reflections of $W$ 
are in the center of $\chi$ (that is, the normal subgroup of $W$ 
consisting of elements which acts on $E_\chi$ by scalar multiplication). 
It follows that the center of $\chi$ is $W$ itself, hence $\chi=\g$. 

Consequently, a linear character is alone in its generic Calogero-Moser family. 
This result applies in particular to $\unb_W$ and $\e$, 
and is compatible with Corollary~\ref{familles lineaires generiques}.\finl
\end{exemple}

\bigskip

\section{Graded dimension, $\bb$-invariant}\label{section:dim graduee}

\medskip

By Proposition~\ref{graduation idem}, the elements of 
$\blocs(\Zba_\CG)$ have $\BZ$-degree $0$. In particular, given
$b \in \blocs(\Zba_\CG)$, 
then $b\Zba_\CG$ is a finite-dimensional graded $\kb(\CG)$-algebra.
The aim of this section is to study this grading. We put
$\bar{\Delta}_\CG(\mathrm{co})=\bar{\Delta}(\mathrm{co})
\otimes_{\kb[\CCB]}\kb(\CG)$.


\bigskip

\begin{theo}\label{dim graduee bonne}
Let $b \in \blocs(\Zba_\CG)$ and let $\FC=\Irr_\Hb(W,b)$. Then:
\begin{itemize}
\itemth{a} $\DS{\dim_{\kb(\CG)}^\grad b\Zba_\CG =\sum_{\chi \in \FC} f_{\chi}(\tb^{-1}) ~f_{\chi}(\tb)}$.

\itemth{b} There exists a unique $\chi \in \FC$ with minimal $\bb$-invariant,
which we denote by $\chi_\FC$.

\itemth{c} The coefficient of $\tb^{\bb_{\chi_\FC}}$ in  
$f_{\chi_\FC}(\tb)$ is equal to $1$.

\itemth{d} $b\bar{\Delta}_\CG(\mathrm{co})=
b\Hbov_\CG e$ is a projective cover of $L_\CG(\chi_\FC)$.

\itemth{e} The algebra $\End_{\Hbov_\CG}(\bar{\Delta}_\CG(\chi_\FC))$ is a quotient of $b\Zba_\CG$.
In particular, it is commutative.
\end{itemize}
\end{theo}

\medskip
The character $\chi_\FC$ is called the {\em special} character of the family $\FC$
(relative to $\CG$).

\bigskip

By~(\ref{chi 1}), we obtain the following immediate consequence:

\bigskip

\begin{coro}\label{dim bonne}
Given $b \in \blocs(\Zba_\CG)$, we have
$$\dim_{\kb(\CG)} b\Zba_\CG= \sum_{\chi \in \Irr_\Hb(W,b)} \chi(1)^2.$$
\end{coro}

\bigskip

\begin{rema}\label{generalization gordon}
Theorem~\ref{dim graduee bonne} 
generalizes~\cite[Theorem~5.6]{gordon} and Corollary~\ref{dim bonne} 
generalizes~\cite[Corollary~5.8]{gordon} (case of families with only
one element). As pointed 
out to us by Gordon, in~\cite[Theorem~5.6]{gordon},
$p_\chi(\tb)=\tb^{\bb_{\chi^*}-\bb_\chi}f_{\chi}(\tb)f_{\chi^*}(\tb^{-1})$ should be replaced by
$p_\chi(\tb)=f_{\chi^*}(\tb)f_{\chi^*}(\tb^{-1})$ with the notation of~\cite{gordon}.
Note that the difference with 
our result comes from the fact that we have used a $\ZM$-grading opposed 
to the one of~\cite[\S 4.1]{gordon}, which amounts to swapping $V$ 
and $V^*$ and so to swap, in this formula, $\chi$ and $\chi^*$.\finl
\end{rema}

%

\bigskip

\begin{proof}[Proof of Theorem~\ref{dim graduee bonne}]
Lemma \ref{le:classDeltaco} shows that
$$[b\bar{\Delta}_\CG(\mathrm{co})]_{\Hbov_\CG}^\grad=\sum_{\chi\in\FC}
f_\chi(\tb^{-1}) [\bar{\Delta}_\CG(\chi)]_{\Hbov_\CG}^\grad.$$

The right action of $b\Zba_\CG$ on $b\Hbov_\CG e$ induces an isomorphism of
graded
algebras $b\Zba_\CG\xrightarrow{\sim} eb\Hbov_\CG e$
(Corollary~\ref{coro:endo-bi}).
We deduce now assertion (a):
$$\dim_{\kb(\CG)}^\grad (b\Zba_\CG)=
\dim_{\kb(\CG)}^\grad(eb\bar{\Delta}_\CG(\mathrm{co}))=
\sum_{\chi\in\FC}f_\chi(\tb^{-1})\dim_{\kb(\CG)}^\grad
\bigl((\kb(\CG)[V]^\cow\otimes E_\chi)^W\bigr)=
\sum_{\chi\in\FC}f_\chi(\tb^{-1})f_{\chi}(\tb).$$

Since $\End_{\Hbov_\CG}(b\Hbov_\CG e)$ is isomorphic to the local
commutative algebra $b\Zba_\CG$, the $\Hbov_\CG$-module
$b\Hbov_\CG e$ is indecomposable (and of course projective), so it admits
a unique simple quotient $L_\CG(\chi_\FC)$, for some $\chi_\FC\in\FC$.
The highest weight category structure of $\Hbov_\CG\mmodgr$ shows that 
$$[b\Hbov_\CG e]-\tb^{b_{\chi_\FC}}[\bar{\Delta}_\CG(\chi_\FC)]
\in \bigoplus_{\chi\in\FC}
\tb^{b_{\chi_\FC}+1}\ZM[\tb][\bar{\Delta}_\CG(\chi)].$$
The assertions (b), (c) and (d) follow.

\medskip
Let $M$ be the kernel of a surjection
$b\Hbov_\CG e\to\bar{\Delta}_\CG(\chi_\FC)$. Since $\End_{\Hbov_\CG}
(b\Hbov_\CG e)$ is
generated by $\Zba_\CG$, it follows that the $\Hbov_\CG$-endomorphisms
of $b\Hbov_\CG e$
stabilize $M$. We obtain by restriction a morphism of $\kb(\CG)$-algebras 
$b\Zba_\CG \to \End_{\Hbov_\CG}(\bar{\Delta}_\CG(\chi_\FC))$
which is surjective since $b\Hbov_\CG e$ is projective. 
\end{proof}

\bigskip
%
%

\begin{coro}\label{polynome caracteristique}
Let $z \in Z$ and let $\carac_z(\tb) \in P[\tb]$ denote the characteristic polynomial 
of the multiplication by $z$ in the $P$-module $Z$. Then
$$\carac_z(\tb) \equiv \prod_{\chi \in \Irr(W)} (\tb - \O_\chi(z))^{\chi(1)^2} \mod \pGba.$$
\end{coro}

\begin{proof}
Let $b \in \blocs(\kb(\CCB)\Zba)$. Since
$z-\O_b(z)$ is a nilpotent endomorphism of 
$b\kb(\CCB)\Zba$, the characteristic polynomial of $z$ on 
$b\kb(\CCB)\Zba$ is $(\tb - \O_b(z))^{\dim_{\kb(\CCB)} b\kb(\CCB)\Zba}$. 
Consequently, 
$$\carac_z(\tb) \equiv \prod_{b \in \blocs(\kb(\CCB)\Zba)} 
(\tb - \O_b(z))^{\dim_{\kb(\CCB)} b\kb(\CCB)\Zba } \mod \pGba.$$
Since $\O_b(z)=\O_\chi(z)$ for all $\chi \in \Irr_\Hb(W,b)$, the result follows 
from~(\ref{cm familles}) and from Corollary~\ref{dim bonne}.
\end{proof}

\bigskip

\bigskip

\begin{coro}\label{multiplicite 1}
Let $\g : W \longto \kb^\times$ be a linear character. Then 
$Z$ is unramified over $P$ at $\Ker(\O_\g)$.
\end{coro}

\begin{proof}
Indeed, if $b_\g$ denotes the primitive idempotent of $\kb(\CCB)\Zba$ associated with $\g$, 
then $\Irr_\Hb(W,b_\g)=\{\g\}$ by Example~\ref{lineaire}. This implies, by Corollary~\ref{dim bonne}, that 
$\dim_{\kb(\CCB)}(b_\g \kb(\CCB)\Zba)=1$. 

Set $\zG_\g=\Ker(\O_\g)$ (we have $\zG_\g \cap P = \pGba$). Then $Z/\zG_\g \simeq \kb[\CCB]$, hence
$Z_{\zG_\g}/\zG_\g Z_{\zG_\g}\simeq \kb(\CCB)$. But, on the other hand, 
$Z_{\zG_\g}/\pGba Z_{\zG_\g} = b_\g \kb(\CCB) \Zba$. So 
$\dim_{\kb(\CCB)}(Z_{\zG_\g}/\pGba Z_{\zG_\g})=1$, which implies that 
$\pGba Z_{\zG_\g} = \zG_\g Z_{\zG_\g}$, as desired.
\end{proof}

\bigskip

\section{Exchanging $V$ and $V^*$} 

\medskip

If $E$ is a graded $(\kb[V] \rtimes W)$-module, one can define
$$\D^*(E)=\Ind_{\Hb^+}^\Hb (\kb[\CCB] \otimes E),$$
which is a graded $\Hb$-module, as well as its reduction modulo $\pGba$, 
denoted by $\bar{\D}^*(E)$, which is a graded $\Hbov$-module. 
If $\CG$ is a prime ideal, we also set $\bar{\D}_\CG^*(E)=\kb(\CG) \otimes_{\kb[\CCB]} \bar{\D}^*(E)$, 
which is a graded $\Hbov_\CG$-module. 

Assume now that $E \in \Irr(W)$. Then, as in Proposition~\ref{pr:verma}, the $\Hbov_\CG$-module 
$\bar{\D}_\CG^*(E)$ is indecomposable and 
admits a unique simple quotient, which will be denoted by 
$L_\CG^*(E)$. Moreover, the map 
\equat\label{eq:bij-simple-*}
\fonctio{\Irr(W)}{\Irr(\Hbov_\CG)}{E}{L_\CG^*(E)}
\endequat
is bijective.

\bigskip

\begin{rema}\label{rem:auto-bij}
From Proposition~\ref{pr:verma} and~(\ref{eq:bij-simple-*}), it follows that there exists 
a unique permutation $\bigstar_\CG$ of $\Irr(W)$ such that
$$L_\CG^*(E)=L_\CG(\bigstar_\CG(E))$$
for all $E \in \Irr(W)$. It turns out that the permutation $\bigstar_\CG$ 
is in general difficult to compute, and that it 
depends heavily on the prime ideal $\CG$, as the reader can already check when
$\dim_\kb V = 1$.\finl
\end{rema}

\bigskip

We say that $E$ and $F$ belong to the same $*$-Calogero-Moser $\CG$-family 
if $L_\CG^*(E)$ and $L_\CG^*(F)$ are two simple modules belonging to the same block of $\Hbov_\CG$. 
It follows from Remark~\ref{rem:auto-bij} that $E$ and $F$ belong to the same $*$-Calogero-Moser $\CG$-family 
if and only if $\bigstar_\CG(E)$ and $\bigstar_\CG(F)$ belong to the same Calogero-Moser $\CG$-family. 

However, there is another easier description of $*$-families which has been explained to us by Gwyn
Bellamy. 
It follows from~(\ref{dimension left}) that 
$$\Res_{\Hbov_\CG^+}^{\Hbov_\CG} \bar{\D}_\CG(E) \simeq \kb(\CG) \otimes \kb[V]^\cow \otimes E.$$
Recall that $N=|\REF(W)|$ and that $\kb[V]^\cow_N$ is one dimensional, affording 
the character $\e$ (as a $\kb W$-module) and that $\kb[V]^\cow_{> N}=0$. So one 
gets an injective morphism of $\Hbov_\CG^+$-modules 
$E \otimes \e \injto \Res_{\Hbov_\CG^+}^{\Hbov_\CG} \bar{\D}_\CG(E)$. 
By adjunction, one gets a non-zero morphism
\equat\label{eq:non-zero morphism}
\bar{\D}_\CG^*(E \otimes \e) \longto \bar{\D}_\CG(E).
\endequat
In particular, 
\equat\label{eq:chief}
\text{\it the simple module $L_\CG^*(E \otimes \e)$ is a composition factor of $\bar{\D}_\CG(E)$.}
\endequat
Since $\bar{\D}_\CG^*(E)$ is indecomposable with unique simple quotient $L_\CG^*(E)$, this 
implies that
\equat\label{eq:v-v*-bloc}
\text{\it the simple modules $L_\CG^*(E \otimes \e)$ and $L_\CG(E)$ belong to the same block of $\Hbov_\CG$.}
\endequat

\bigskip

\begin{prop}\label{prop:cm-*}
The simple $\kb W$-modules $E$ and $F$ belong to the same Calogero-Moser $\CG$-family if and only if 
$E \otimes \e$ and $F \otimes \e$ belong to the same $*$-Calogero-Moser $\CG$-family.
\end{prop}

\bigskip

The following consequence was announced in Remark~\ref{rem:v-v*}.

\bigskip

\begin{coro}\label{coro:v-v*}
If $z \in Z$, then $\Omeb(z)=\e(\Omeb^*(z))$.
\end{coro}

\begin{proof}
It is sufficient to prove that, for all $E \in \Irr(W)$, we have 
$\o_E(\Omeb(z))=\o_E(\lexp{\e}{\Omeb^*(z)})$. Since
$\o_E(\e(\Omeb^*(z)))=\o_{E \otimes \e}(\Omeb^*(z))$, it 
is sufficient to prove that 
$$\o_E(\Omeb(z))=\o_{E \otimes \e}(\Omeb^*(z)).$$
Take $\CG$ to be the zero ideal of $\kb[\CCB]$, so that $\kb(\CG)=\kb(\CCB)$. 
By Lemma~\ref{lem:action-z-verma} (and its version obtained by swapping $V$ and $V^*$), 
we get that $\o_E(\Omeb(z))$ is the scalar by which $z$ acts on the simple module 
$L_\CG(E)$ while $\o_{E \otimes \e}(\Omeb^*(z))$ is the scalar by which $z$ acts on the simple module 
$L_\CG^*(E \otimes \e)$. So the result follows from~(\ref{eq:v-v*-bloc}).
\end{proof}

\bigskip

\section{Geometry}\label{section:geometrie CM}

\medskip

The composition 
\equat\label{section upsilon}
\diagram
 \kb[\CCB] \xyinj[rr] && Z \xysur[rr]^{\DS{\O_b}} && \kb[\CCB]
\enddiagram
\endequat
is the identity, which means that the morphism of $\kb$-varieties  
$\O_b^\sharp : \CCB \longto \ZCB$  \indexnot{oz}{\O_b^\sharp}  induced by $\O_b$ is 
a section of the morphism $\pi \circ \Upsilon : \ZCB \longto \CCB$ (see the diagram~\ref{diagramme geometrie}). 
Lemma~\ref{caracterization blocs CM} says that the map 
$$\fonctio{\blocs(\kb(\CCB)\Zba)}{\Upsilon^{-1}(\pGba)}{b}{\Ker(\O_b)}$$
is bijective or, in geometric terms, that 
the irreducible components of $\Upsilon^{-1}(\CCB \times 0)$ 
are in bijection with $\blocs(\kb(\CCB)\Zba)$, through the map $b \mapsto \O_b^\sharp(\CCB)$. We deduce the following proposition.

\bigskip

\begin{prop}\label{generique particulier}
Let $c \in \CCB$. Then the following are equivalent:
\begin{itemize}
\itemth{1} $|\blocs(\kb(\CCB)\Zba)|=|\blocs(\Kbov_c\Zba)|$.

\itemth{2} $|\Upsilon^{-1}_c(0)|$ is equal to the number of irreducible components of 
$\Upsilon^{-1}(\CCB \times 0)$.

\itemth{3} Every element of $\Upsilon^{-1}_c(0)$ belongs to a unique 
irreducible component of $\Upsilon^{-1}(\CCB \times 0)$. 

\itemth{4} If $b$ and $b'$ are two distinct elements of $\blocs(\kb(\CCB)\Zba)$, 
then $\th_c \circ \O_b \neq \th_c \circ \O_{b'}$.
\end{itemize}
\end{prop}

\bigskip

We say that $c\in\CCB$ is {\it generic} if it satisfies one of the equivalent conditions of
Proposition~\ref{generique particulier}. It will be called {\it particular} 
otherwise. We will denote by $\CCB_\gen$ (respectively $\CCB_\parti$) the set of 
generic (respectively particular) elements of $\CCB$. 

\bigskip

\begin{coro}\label{particulier ferme}
$\CCB_\gen$ is a Zariski dense and open subset of $\CCB$ and
$\CCB_\parti$ is Zariski closed in $\CCB$. If $W \neq 1$, 
then $\CCB_\parti$ is of pure codimension $1$ and contains $0$. 

Moreover, $\CCB_\gen$ and $\CCB_\parti$ are stable under the action of 
$\kb^\times \times \kb^\times \times (W^\wedge \rtimes \NC)$. 
\end{coro}

\begin{proof}
The stability under the action of 
$\kb^\times \times \kb^\times \times (W^\wedge \rtimes \NC)$ is obvious. 
The fact that $\CCB_\gen$ (respectively $\CCB_\parti$) is open 
(respectively closed) follows from Proposition~\ref{codimension un}(2). 
Whenever $W \neq 1$, the trivial character is alone in its generic Calogero-Moser family 
(see Example~\ref{lineaire}) while its Calogero-Moser $0$-family is $\Irr(W)$. 
This shows that $0 \in \CCB_\parti$ and, by
Proposition~\ref{codimension un}(1), $\CCB_\parti$ has
pure codimension $1$. 
%
%
\end{proof}

\bigskip

We deduce the following from Example~\ref{lineaire}.

\bigskip

\begin{coro}\label{generique lineaire}
If $c \in \CCB$ is generic, then any linear character of $W$ 
is alone in its Calogero-Moser $c$-family.
\end{coro}

\bigskip


\begin{coro}\label{multiplicite 1 generique}
Let $\g : W \longto \kb^\times$ be a linear character and assume that $c$ is generic. 
Then $Z$ is unramified over $P$ at $\Ker(\O_\g^c)$.
\end{coro}

\bigskip

\section{Smoothness and Calogero-Moser families}
\label{se:smoothnessCMfamilies}

\medskip

Let $b \in \blocs(\Zba_\CG)$ and let $\bar{\zG}_b$  \indexnot{z}{\bar{\zG}_b}  denote the prime ideal 
of $Z$ equal to the kernel of $\O_b^\CG : Z \to \kb(\CG)$.
Note that $b\Zba_\CG$ is a local finite dimensional $\kb(\CG)$-algebra with residue
field $\kb(\CG)$.

By~\cite[\S 5]{gordon}, we have the following characterization of smoothness.

\bigskip

\begin{prop}\label{prop lissite}
The ring $Z$ is regular at $\bar{\zG}_b$ if and only if $|\Irr_\Hb(W,b)|=1$. 
Moreover, if $Z$ is regular at $\bar{\zG}_b$, then
$$\Hbov_\CG b \simeq \Mat_{|W|}(b\Zba_\CG).$$
\end{prop}

\begin{proof}
	Theorem \ref{th:smoothsimples} provides the first equivalence.
	Assume now $Z$ is regular at $\bar{\zG}_b$.  Theorem \ref{th:smoothsimples} shows that
	$\Hbov_\CG be$ induces a Morita equivalence between 
	$\Hbov_\CG b$ and $\Mat_{|W|}(b\Zba_\CG)$. Since $b\Zba_\CG$ is local, it follows that
	the finitely generated projective $b\Zba_\CG$-module $\Hbov_\CG be$ is free.
	On the other hand,  Theorem \ref{th:smoothsimples} shows that
	$\Hbov_\CG be\otimes_{b\Zba_\CG}Z(\bar{\zG}_b)$ is a vector space of dimension $|W|$ over
	$Z(\bar{\zG}_b)$, hence $\Hbov_\CG be$ is a free $b\Zba_\CG$-module of rank $|W|$. The proposition
	follows.
\end{proof}

\medskip

Consider now $c \in \CCB$ and $b \in \blocs(\Zba_c)$.
Let $z_b$ denote the point of $\Upsilon_c^{-1}(0) \subset \ZCB_c \subset \ZCB$ 
corresponding to $b$. 

\bigskip

\begin{prop}\label{lissite generique}
The following assertions are equivalent:
\begin{itemize}
\itemth{1} $\ZCB$ is smooth at $z_b$.

\itemth{2} $\ZCB_c$ is smooth at $z_b$.
\end{itemize}
\end{prop}

\begin{proof}
Let us first recall the following Lemma:

\bigskip

\begin{quotation}
{\small
\begin{lem}\label{somme tangents}
Let $\ph : \YCB \to \XCB$ be a morphism of $\kb$-varieties (not necessarily reduced), 
let $y \in \YCB$ and let $x=\ph(y)$. 
We assume that there exists a morphism of $\kb$-varieties 
$\s : \XCB \to \YCB$ such that $y=\s(x)$ and $\ph \circ \s = \Id_\XCB$. Then 
$$\TC_y(\YCB)=\TC_y(\ph^*(x)) \oplus \TC_y(\s(\XCB)).$$
Here, $\TC_y(\YCB)$ denotes the tangent space to the $\kb$-variety $\YCB$ and 
$\ph^*(x)$ denotes the (scheme-theoretic) fiber of $\ph$ at $x$, viewed as a closed 
$\kb$-subvariety of $\YCB$, not necessarily reduced.
\end{lem}}
\end{quotation}

\bigskip

Let $\chi \in \Irr_\Hb(W,b)$. The morphism of varieties 
$\O_\chi^\sharp : \CCB \to \ZCB$ is a section of the morphism 
$\pi \circ \Upsilon : \ZCB \to \CCB$. Moreover, by assumption, 
$z_b = \O_\chi^\sharp(c)$. By Lemma~\ref{somme tangents} above, 
we have
$$\TC_{z_b}(\ZCB) = \TC_{z_b}(\ZCB_c) \oplus \TC_{z_b}(\O_\chi^\sharp(\CCB)).$$
Since $\TC_{z_b}(\O_\chi^\sharp(\CCB)) \simeq \TC_c(\CCB)$, the Proposition follows from 
the smoothness of $\CCB$ and from the fact that 
$\dim(\ZCB)=\dim(\ZCB_c)+\dim(\CCB)$.
\end{proof}

\bigskip

After the work of Etingof-Ginzburg~\cite{EG}, Ginzburg-Kaledin~\cite{GK},
Gordon~\cite{gordon} and Bellamy~\cite{bellamy g4}, a complete classification 
of complex reflection groups $W$ such that there exists $c \in \CCB$ 
such that $\ZCB_c$ is smooth has been obtained. Note that the statements 
{\it ``There exists $c \in \CCB$ such that $\ZCB_c$ is smooth''} and 
{\it ``The ring $\kb(\CCB) \otimes_{\kb[\CCB]} Z=\kb(\CCB)Z$ is regular''} 
are equivalent. We recall now the result.

\bigskip

\begin{theo}[Etingof-Ginzburg, Ginzburg-Kaledin, Gordon, Bellamy]\label{les lisses}
Assume that $W$ is irreducible. Then the ring $\kb(\CCB)Z$ is regular 
if and only if we are in one of the following two cases:
\begin{itemize}
 \itemth{1} $W$ has type $G(d,1,n)$, with $d$, $n \ge 1$.

\itemth{2} $W$ is the group denoted $G_4$ in the Shephard-Todd classification.
\end{itemize}
\end{theo}

\bigskip

The following proposition is a consequence of 
the work of Etingof-Ginzburg~\cite{EG}, Gordon~\cite{gordon} and
Bellamy~\cite{bellamy g4}, by using the Shephard-Todd classification 
of complex reflection groups. Recently, Bellamy-Schedler-Thiel 
gave a proof of this fact which does not rely on the 
classification~\cite[Corollary~1.4]{BelSchTh}. 

\bigskip

\begin{prop}\label{lissite en 0}
Let $c \in \CCB$. Then the following are equivalent:
\begin{itemize}
\itemth{1} The variety $\ZCB_c$ is smooth.

\itemth{2} The points belonging to 
$\Upsilon_c^{-1}(0)$ are smooth in $\ZCB_c$.
\end{itemize}
\end{prop}

\section{Blocks and Calogero-Moser families}
\label{se:blocksCMfamilies}
We assume in \S\ref{se:blocksCMfamilies} that $W$ is irreducible.

\smallskip
Calogero-Moser families and blocks of the category
$\tilde{\OC}$ are closely related, as the following lemma shows.
Given $E\in\Irr(W)$, we put
$\tilde{\Delta}_\CG(E)=\tilde{\Delta}(E)\otimes_{\kb[\CCB]}\kb(\CG)$.
As in Example \ref{Z graduation-1}, we consider $w_z=\zeta^{-1}\Id_V$ a
generator of $W\cap \Zrm(\GL(V))$.

\begin{prop}
\label{pr:familiesasblocks}
Let $E,F\in\Irr(W)$ and let $i\in\ZM$. The standard objects
	$\tilde{\Delta}_\CG(E)$ and 
	$\tilde{\Delta}_\CG(F)\langle i\rangle$ are in the same block
	of $\tilde{\OC}(\kb(\CG))$
if and only if $E$ and $F$ are in the same Calogero-Moser $\CG$-family and
$\omega_E(w_z)=\zeta^i\omega_F(w_z)$.
\end{prop}

\begin{proof}
We use the notation of Example \ref{Z graduation-1}.
The element $w_z\in \Zrm(W)$ acts on the degree $r$ part of
	$\tilde{\Delta}_\CG(F)\langle i\rangle$ by
$\omega_F(w_z)\zeta^{r+i}$. It follows from \S \ref{se:caseT=0} that if
	$\tilde{\Delta}_\CG(E)\otimes_{\kb[\CCB]}\kb(\CG)$ and 
	$\tilde{\Delta}_\CG(F)\otimes_{\kb[\CCB]}\kb(\CG)\langle i\rangle$ are in the same block
	of $\tilde{\OC}(\kb(\CG))$, then $\omega_E(w_z)=\omega_F(w_z) \zeta^i$.

\smallskip
	Note that $\tilde{\Delta}_\CG(E)$ has a filtration whose successive
quotients are isomorphic to $\bar{\Delta}_\CG(E)\otimes_{\kb}(\CM[V]^W)^i$.
As a consequence, $\bar{\Delta}_\CG(E)$ and
$\bar{\Delta}_\CG(E)\langle -i\rangle$ are in the same block, whenever
$(\CM[V]^W)^i\neq 0$. Since $z_W=\gcd(d_1,\ldots,d_n)$, it follows that
$\bar{\Delta}_\CG(E)$ and
$\bar{\Delta}_\CG(E)\langle z_W\rangle$ are in the same block.

\smallskip
Assume now $E$ and $F$ are in the same Calogero-Moser $\CG$-family
and $\omega_E(w_z)=\zeta^i\omega_F(w_z)$.
There is an integer $j$ such that $\bar{\Delta}_\CG(E)$ and
	$\bar{\Delta}_\CG(F)\langle j\rangle$ are in the same block of $\tilde{\OC}(\kb(\CG))$.
It follows that $\omega_E(w_z)=\omega_F(w_z) \zeta^j$, hence $d|(i-j)$.
So, $\bar{\Delta}_\CG(E)$ and
	$\bar{\Delta}_\CG(F)\langle i\rangle$ are in the same block of $\tilde{\OC}(\kb(\CG))$.
\end{proof}

Let $\tilde{\FC}$ be the set of height one prime ideals $\CG$ of $\kb[\CCBt]$
such that the blocks of $\tilde{\OC}(\kb(\CG))$ are not trivial, i.e., there
exists
$E,F\in\Irr(W)$ and $r\in\ZM$ such that 
$\tilde{\Delta}_\CG(E)$ and $\tilde{\Delta}_\CG(F)\langle r\rangle$ are in
the same block and $E{\not\simeq}F\langle r\rangle$.
Note that the ideals of $\tilde{\FC}$ are homogeneous
for the $\ZM$-grading on $\kb[\CCBt]$.

\smallskip

\medskip
We assume for the remainder of \S \ref{se:blocksCMfamilies} that $V\neq 0$.

\begin{prop}
\label{pr:idealsT}
The ideals in $\tilde{\FC}$ are $(T)$ and the ideals
$(C_E-C_F-rT)$ such that $(C_E-C_F-r)\in\FC_1$, where
$E,F\in\Irr(W)$ and $r\in\ZM\setminus\{0\}$.
\end{prop}

\begin{proof}
	Proposition \ref{pr:familiesasblocks} shows that $\tilde{\OC}(\kb(\CG))$ has
non-trivial blocks for all prime ideals $\CG$ of $\kb[\CCB]$.
It follows that $(T)\in\tilde{\FC}$.

Let $\CG$ be an ideal of $\tilde{\FC}$ distinct from $(T)$. Since
$\CG$ is homogeneous, it follows that it is generated by some
irreducible polynomial $P(T)=\sum_{i=0}^r a_iT^i$,
where $r\ge 0$ and $a_i$ is a homogeneous polynomial of degree $d-i$ in the
indeterminates $C_s$, for some $d\in\ZM$. 
Consider the proper ideal $\qG=(T-1,P(T))$ of $\kb[\CCBt]$.
Since $\CG\in\tilde{\FC}$ and $\CG\subset\qG$, it follows that 
$P(1)\kb[\CCB]$ contains an ideal in $\FC_1$, hence $P(1)$
is divisible by $C_E-C_F-r$ for some $E,F\in\Irr(W)$ with 
$C_E\neq C_F$ and some $r\in\ZM\setminus\{0\}$. Since $P(T)$ is homogeneous,
it follows that it is divisible by $C_E-C_F-rT$. We deduce that
$(P(T))=(C_E-C_F-rT)$.
\end{proof}

Define $\FC_0$ to be the set of height one prime ideals $\CG$ of $\kb[\CCB]$
such that the Calogero-Moser $\CG$-families are different from the generic
Calogero-Moser families. Propositions \ref{pr:familiesasblocks} and
\ref{pr:idealsT} have the following consequence. The fact that the ideals
in $\FC_0$ define hyperplanes of $\CC$ (and not merely hypersurfaces) is
due to Bellamy, Schedler and Thiel \cite[Theorem 5.1]{BelSchTh}.

\begin{coro}
The ideals in $\FC_0$ are those ideals of the form $(C_E-C_F)$ for some
$E,F\in\Irr(W)$ such that there exists $r\in\ZM\setminus\{0\}$ with
$(C_E-C_F-r)\in\FC_1$.
\end{coro}

\begin{theo}
Let $c\in\CCB$. 
The Calogero-Moser $c$-families are the smallest subsets of $\Irr(W)$
that are unions of generic Calogero-Moser families and unions of
blocks of $\dot{\OC}(\kb(\hbar))$ for all morphisms of $\kb$-algebras
$\kb[\CCB]\to\kb(\hbar)$ of the form $C\mapsto \hbar c+c'$
with $\kappa(c')\in\KCB(\QM)$.
\end{theo}

\begin{proof}
Let $I$ be the set of prime ideals $\CG=(C_E-C_F-r)\in\FC_1$
such that $C_E(c)=C_F(c)$.

By Propositions \ref{pr:familiesasblocks} and
 \ref{pr:idealsT}, the Calogero-Moser $c$-families
are the smallest subsets of $\Irr(W)$
that are unions of generic Calogero-Moser families and unions of
blocks of $\dot{\OC}(\CG)$ for all $\CG\in I$.

Consider a morphism of $\kb$-algebras
$\kb[\CCB]\to\kb(\hbar)$ of the form $C\mapsto \hbar c+c'$
with $\kappa(c')\in\KCB(\QM)$.
Let $I(c')$ be the set of 
$\CG=(C_E-C_F-r)\in\FC_1$ such that 
$(C_E-C_F-r)(\hbar c+c')=0$, i.e., $C_E(c)=C_F(c)$ and
$(C_E-C_F)(c')=r$.
The blocks of $\dot{\OC}(\kb(\hbar))$ are the
smallest subsets of $\Irr(W)$ that are unions of blocks of
	$\dot{\OC}(\kb(\CG))$ for all $\CG\in I(c')$.
Since $I=\bigcup_{c'\in\CC(\QM)}I(c')$, the theorem follows.
\end{proof}

Since blocks of $\OC$ correspond to blocks of the Hecke algebra
(cf. \S \ref{se:semisimplicity}), we can reformulate the previous result.

\begin{theo}
\label{th:CMfamiliesblocksHecke}
Let $c\in\CCB$ and $k=\kappa(c)$.
The Calogero-Moser $c$-families are the smallest subsets of $\Irr(W)$
that are unions of generic Calogero-Moser families and unions of
blocks of $\CM(\qb^\kb)\heckegenerique$ for all morphisms
of $\CM$-algebras $\CM[\TCB]\to\CM(\qb^\kb)$ of the form
$\qb_{\orbite,j} \mapsto \zeta_{\orbite,j} \qb^{k_{\orbite,-j}}$ where $(\zeta_{\orbite,j})_{(\orbite,j)\in
\orbiteb^\circ}$ is a family of roots of unity.
\end{theo}

Using Proposition \ref{prop lissite}, we deduce a description of
Calogero-Moser families from blocks of the Hecke algebra, when the
Calogero-Moser space is smooth for generic values of the parameter.

\begin{coro}
\label{cor:genericCMsmoothfamilies}
Assume $\ZCB_\eta$ is smooth, for $\eta$ the generic point of $\CCB$. 
Let $c\in\CC(\kb)$. Let $I$ be a subset of $\Irr(W)$.
The following are equivalent:

\begin{itemize}
	\item[(i)] $I$ is a union of Calogero-Moser $c$-families.
	\item[(ii)] $I$ is a union of blocks of $\dot{\OC}(\kb(\hbar))$ for all morphisms of
$\kb$-algebras $\kb[\CCB]\to\kb(\hbar)$ of the form $C\mapsto \hbar c+c'$
with $\kappa(c')\in\KCB(\QM)$.
\item[(iii)] $I$ is a union of blocks of $\CM(\qb^\kb)\heckegenerique$ for all morphisms
of $\CM$-algebras $\CM[\TCB]\to\CM[\qb^\kb]$ of the form
$\qb_{\orbite,j} \mapsto \zeta_{\orbite,j} \qb^{k_{\orbite,-j}}$ where $(\zeta_{\orbite,j})_{(\orbite,j)\in
\orbiteb^\circ}$ is a family of roots of unity.
\end{itemize}
\end{coro}

\begin{rema}
When $W$ has a unique conjugacy class of reflections, the previous results
are trivial: when $c\neq 0$, the algebras $\CM(\qb^\kb)\heckegenerique$
in Theorem \ref{th:CMfamiliesblocksHecke} and
Corollary \ref{cor:genericCMsmoothfamilies} are semisimple.\finl
\end{rema}

\bigskip

\chapter{Calogero-Moser cellular characters}\label{chapter:gauche-cellulaire}

In this Chapter~\ref{chapter:gauche-cellulaire}, $\CG$ denotes a prime ideal of $\kb[\CCB]$. 
We use Verma modules for $\Hb^\gauche$ to define the notion 
of {\it Calogero-Moser $\CG$-cellular characters}. We expect 
that they coincide with the Kazhdan-Lusztig cellular characters when $W$ is a
Coxeter group. 

This chapter will mainly consider {\it left} Calogero-Moser $\CG$-cellular
characters and (left) Verma modules:
definitions and results can be immediately transposed to the {\it right}
setting.

\bigskip

\boitegrise{{\bf Notation.} {\it Given
$\CG$ a prime ideal of $\kb[\CCB]$,
we denote by
$\pG_\CG^\gauche$ (resp. $\pG_\CG^\droite$)  \indexnot{pa}{\pG_\CG^\gauche, \pG_\CG^\droite}  
the prime ideal of $P$ corresponding to the closed irreducible subscheme 
$\CCB(\CG) \times V/W \times \{0\}$
(resp. $\CCB(\CG) \times \{0\} \times V^*/W$). 
We set 
$$
\begin{cases}
P^\gauche_\CG = P/\pG_\CG^\gauche \simeq \kb[\CCB]/\CG \otimes \kb[V]^W,\\ 
P^\droite_\CG=P/\pG^\droite \simeq \kb[\CCB]/\CG \otimes \kb[V^*]^W,
\end{cases}
\indexnot{P}{P_\CG^\gauche, P_\CG^\droite} 
$$
and we define 
$$\begin{cases}
\text{\it the $P_\CG^\gauche$-algebra $\Hb^\gauche_\CG=\Hb/\pG_\CG^\gauche\Hb$,}\\
\text{\it the $P_\CG^\droite$-algebra $\Hb_\CG^\droite=\Hb/\pG_\CG^\droite\Hb$.} 
\indexnot{H}{\Hb_\CG^\gauche, \Hb_\CG^\droite} 
\end{cases}$$
We denote by $Z_\CG^\gauche$ (resp. $Z_\CG^\droite$) 
the image of $Z$ in $\Hb_\CG^\gauche$ (resp.
$\Hb_\CG^\droite$). 
\indexnot{Z}{Z_\CG^\gauche, Z_\CG^\droite}
We also define 
$$\begin{cases}
\Kb_\CG^\gauche=k_P(\pG_\CG^\gauche),\\ 
\Kb_\CG^\droite = k_P(\pG_\CG^\droite).
\indexnot{K}{\Kb_\CG^\gauche, \Kb_\CG^\droite}
\end{cases}$$
\hphantom{A} To simplify the notation, when $\CG=0$,
the index $\CG$ will be omitted 
in all the previous notations ($\pG^\gauche$, $\Hb^\droite$, $\Kb^\gauche$,\dots). Given
$c \in \CCB$ and $\CG=\CG_c$, the index $\CG_c$ will be replaced by $c$   
\indexnot{pa}{\pG^\gauche, \pG^\droite,\pG_c^\gauche, \pG_c^\droite}  
\indexnot{P}{P^\gauche, P_\CG^\droite,  P_c^\gauche, P_c^\droite}  
\indexnot{H}{\Hb^\gauche, \Hb^\droite, \Hb_c^\gauche, \Hb_c^\droite}  
\indexnot{K}{\Kb^\gauche, \Kb^\droite, \Kb_c^\gauche, \Kb_c^\droite} 
\indexnot{Z}{Z^\gauche, Z^\droite,  Z_c^\gauche, Z_c^\droite}
($\pG_c^\droite$, $\Hb_c^\gauche$, $\Kb_c^\droite$, \dots). 
}}{0.75\textwidth}

\bigskip

\section{Verma modules and cellular characters}

\medskip

\subsection{Morita equivalence}
Let $P^{\reg,\gauche}=\kb[\CCB \times V^\reg/W\times\{0\}] = P^\reg \otimes_P P^\gauche$ 
\indexnot{P}{P^{\reg,\gauche}}, 
$Z^{\reg,\gauche}=P^{\reg,\gauche}Z$ \indexnot{Z}{Z^{\reg,\gauche}}, 
 and $\Hb^{\reg,\gauche}=P^{\reg,\gauche}\Hb$. 
\indexnot{H}{\Hb^{\reg,\gauche}}
Note that, by Example~\ref{exemple lisse}, $\zG_\singulier \cap P \not\subset \pG^\gauche$. Hence, 
Theorem~\ref{lissite et simples} can be applied. Thanks to
Corollary~\ref{coro:reg-morita}, we obtain the following result.  

\bigskip

\begin{theo}\label{theo:morita-left}
The $(\Hb^{\reg,\gauche},Z^{\reg,\gauche})$-bimodule $\Hb^{\reg,\gauche} e$ induces a Morita equivalence 
between the algebras $\Hb^{\reg,\gauche}$ and $Z^{\reg,\gauche}$. Consequently, 
the $(\Kb_\CG^\gauche \Hb^\gauche,\Kb_\CG^\gauche Z^\gauche)$-bimodule 
$\Kb_\CG^\gauche \Hb^\gauche e$ induces a Morita equivalence 
between the algebras $\Kb_\CG^\gauche \Hb^\gauche$ and $\Kb_\CG^\gauche Z^\gauche$. 
\end{theo}

\bigskip

The Morita equivalence of Theorem~\ref{theo:morita-left} induces 
a bijection
\equat\label{eq:simple-left}
\bijectio{\Irr(\Kb_\CG^\gauche\Hb^\gauche)}{\Irr(\Kb_\CG^\gauche Z^\gauche)}{L}{eL.}
\endequat
On the other hand, the (isomorphism classes of) 
simple $\Kb_\CG^\gauche Z^\gauche$-modules are in bijection with the 
maximal ideals of $\Kb_\CG^\gauche Z^\gauche$, that is, with the minimal prime ideals 
of $Z_\CG^\gauche$ or, in other words, with $\Upsilon^{-1}(\pG_\CG^\gauche)$. 
Using~(\ref{eq:simple-left}), we obtain a bijection
\equat\label{eq:simple-left-prime}
\bijectio{\Upsilon^{-1}(\pG_\CG^\gauche)}{\Irr(\Kb_\CG^\gauche\Hb^\gauche)}{\zG}{L_\CG^\gauche(\zG).}
\endequat

\bigskip

\subsection{Cellular characters} 

Let $\zG \in \Upsilon^{-1}(\pG_\CG^\gauche)$ and let $e_\zG$ be the
corresponding primitive idempotent of $\Kb_\CG^\gauche Z^\gauche$.

We identify $G_0((Z_\CG^\gauche)_\zG[W])$ with $G_0(\kb W)$ by
$[(Z/\zG)\otimes E]\mapsto [E]$ for $E\in \Irr(W)$.

The right action by multiplication of $\kb W$ on $\kb W$ induces a 
right action of $\kb W$ on $\Delta(\kb W)$.
\bigskip

\begin{defi}\label{defi:cellular-cm}
We define the {\bfit Calogero-Moser $\CG$-cellular character} associated with
$\zG$ as the character of $W$ given by 
$$\g_\zG^\calo =
([Z_\zG e \Kb_\CG^\gauche\Delta(\kb W)]_{(Z_\CG^\gauche)_\zG W})^*.$$
\indexnot{gzCM}{\g_\zG^\calo}
\end{defi}

When
$\CG=0$ (respectively $\CG=\CG_c$ for 
some $c \in \CCB$), they will be called
{\bfit generic Calogero-Moser cellular characters} 
(respectively {\bfit Calogero-Moser $c$-cellular characters}).

\bigskip

Given $\zG \in \Upsilon^{-1}(\pG_\CG^\gauche)$ and $\chi \in \Irr(W)$, we denote by 
$\mult_{\zG,\chi}^\calo$ the multiplicity of the simple module 
$L_\CG^\gauche(\zG)$ in a composition series of the Verma module 
$\Kb_\CG^\gauche\Delta(\chi)$.

\bigskip

The above definition can be expressed in terms of length of $Z_\zG$-modules, 
through the Morita equivalence. We first need the following lemma.
Given $M$ a finitely generated $Z_\CG^\gauche$-module, the
$Z_\zG$-module $M_\zG$ has
finite length and we denote this length by $\longueur_{Z_\zG}(M_\zG)$.

\bigskip

\begin{lem}\label{le:multiplicite-Z}
Let $M$ be a finitely generated $\Hb_\CG^\gauche$-module. 
Then $eM_\zG$ is a $Z_\zG$-module of finite length and 
$\longueur_{Z_\zG}(eM_\zG)$ is equal to the multiplicity of $L_\CG^\gauche(\zG)$ 
in the $\Kb_\CG^\gauche\Hb^\gauche$-module $\Kb_\CG^\gauche M$.
\end{lem}

\begin{proof}
By construction, $\longueur_{Z_\zG}(eM_\zG)$ is equal to 
the multiplicity of $eL_\CG^\gauche(\zG)$ in the 
$\Kb_\CG^\gauche Z^\gauche$-module $e\Kb_\CG^\gauche M$. The result 
follows now from the Morita equivalence of Theorem~\ref{theo:morita-left}.
\end{proof}

\bigskip
Given  $\zG \in \Upsilon^{-1}(\pG_\CG^\gauche)$, we put
$$\mult_{\zG,\chi}^\calo=\longueur_{Z_\zG}\bigl(e \Kb_\CG^\gauche \Delta(\chi)\bigr)_\zG\indexnot{mult}{\mult_{\zG,\chi}^\calo}.$$

\smallskip
\begin{prop}\label{pr:mult-z-left}
Let $\zG \in \Upsilon^{-1}(\pG_\CG^\gauche)$. We have
$$\g_\zG^\calo
=\sum_{\chi \in \Irr(W)} \mult_{\zG,\chi}^\calo \cdot \chi.$$
\end{prop}

\begin{proof}
We have $\Delta(\kb W)=\bigoplus_{E\in\Irr(W)}\Delta(E)\otimes E^*$, 
hence 
$$Z_\zG e \Kb_\CG^\gauche\Delta(\kb W)=\bigoplus_E
Z_\zG e \Kb_\CG^\gauche\Delta(E)\otimes E^*.$$
Since $[Z_\zG e \Kb_\CG^\gauche\Delta(E)]_{(Z_\CG^\gauche)_\zG}=
\longueur_{Z_\zG}\bigl(e \Kb_\CG^\gauche \Delta(E)\bigr)_\zG [Z/\zG]$,
we deduce the first equality. The second equality follows from
Lemma \ref{le:multiplicite-Z}.
\end{proof}

\bigskip
\begin{rema}
	We expect that the cellular characters associated
to two different prime ideals are distinct and that the cellular
characters are linearly independent.
\end{rema}

\bigskip

\section{Choices}\label{section:choix-gauche}

\medskip

\begin{lem}\label{absolue simplicite}
Given $\g$ a {\bfit linear} character of $W$, the
$\Kb^\gauche\Hb^\gauche$-module $\Kb^\gauche\Delta(\g)$ is absolutely simple.
\end{lem}

\begin{proof}
This follows from Theorem~\ref{theo:morita-left} and~(\ref{dimension left}).
\end{proof}

\bigskip

Fix now a linear character $\g$ of $W$. By Lemma~\ref{absolue simplicite}, 
the endomorphism algebra of $\Kb^\gauche\Delta(\g)$ is equal to 
$\Kb^\gauche$. This induces a morphism of $P$-algebras 
$\O_\g^\gauche : Z \to \Kb^\gauche$ \indexnot{oz}{\O_\g^\gauche}
whose restriction to $P$ is the canonical morphism $P \to P^\gauche$. 
Since $Z$ is integral over $P$, the image of $\O_\g^\gauche$ is integral over 
$P^\gauche$ and contained in $\Kb^\gauche=\Frac(P^\gauche)$. Since $P^\gauche \simeq \kb[\CCB \times V/W]$ 
is integrally closed, this forces $\O_\g^\gauche$ to factor through $P^\gauche$. 
We set
$$\zG^\gauche = \Ker(\O_1^\gauche)\quad\text{and}
\quad \qG^\gauche=\copie(\zG^\gauche).\indexnot{za}{\zG^\gauche}\indexnot{qa}{\qG^\gauche}$$

\bigskip 

\begin{prop}\label{q left}
The ideal $\qG^\gauche$ of $Q$ satisfies the following properties:
\begin{itemize}
\itemth{a} $\qG^\gauche$ is a prime ideal of $Q$ lying over $\pG^\gauche$.

\itemth{b} $\qG^\gauche \subset \qGba$.

\itemth{c} $P^\gauche=P/\pG^\gauche \simeq Q/\qG^\gauche$.
\end{itemize}
\end{prop}

\begin{proof}
Since $\Kb^\gauche$ is a field, $\qG^\gauche$ is prime. Since the restriction of $\O_1^\gauche$ to $P$ 
is the canonical morphism $P \to P^\gauche$, $\qG^\gauche \cap P = \pG^\gauche$. This shows~(a). 

\medskip

By construction, $\Delta(\g)/\pGba \Delta(\g)=\bar{\Delta}(\g)$ and so the morphism 
$\O_\g : Z \to \Pba=P/\pGba$ factors through the morphisms $\O_\g^\gauche : Z \to P^\gauche$ and  
$P^\gauche \to \Pba$, whence~(b).

\medskip

Finally, the isomorphism (c) follows from the fact that the image of $\O_\g$ is $P^\gauche$.
\end{proof}

\bigskip

Proposition~\ref{q left} allows us to choose a prime ideal of $Q$ lying over 
$\pG^\gauche$ and compatible with our choice of $\qGba$. The next lemma shows that this choice is 
unique:

\bigskip

\begin{lem}\label{unicite qgauche plus}
We have $\pG^\gauche Q_\qGba = \qG^\gauche Q_\qGba$.
\end{lem}

\begin{proof}
It is sufficient to prove that $\pG^\gauche Q_{\qGba}$ is a prime ideal of $Q_{\qGba}$. 
By Lemma~\ref{qba}, the local morphism of local rings 
$P_{\pGba} \to Q_{\qGba}$ is \'etale. Moreover, $P/\pG^\gauche \simeq \kb[\CCB \times V^*/W]$ 
is integrally closed (it is a polynomial algebra) and so 
$P_{\pGba}/\pG^\gauche P_{\pGba}$ is also integrally closed. By base change, the ring morphism 
$P_{\pGba}/\pG^\gauche P_{\pGba} \longinjto Q_{\qGba}/\pG^\gauche Q_{\qGba}$ 
is \'etale, which implies that $Q_{\qGba}/\pG^\gauche Q_{\qGba}$, 
which is a local ring (hence is connected), is also integrally closed 
(by~\cite[Expos\'e I, Corollaire~9.11]{sga}) and so is a domain 
(because it is connected). This shows that $\pG^\gauche Q_{\qGba}$ is 
a prime ideal of $Q_{\qGba}$, as desired.
\end{proof}

\bigskip

\begin{coro}\label{unicite qgauche}
The ideal $\qG^\gauche$ is the unique prime ideal of $Q$ lying over $\pG^\gauche$ 
and contained in $\qGba$. Moreover, $Q$ is \'etale over $P$ at $\qG^\gauche$. 
\end{coro}

\bigskip

Since $Q/\qG^\gauche \simeq P/\pG^\gauche=\kb[\CCB \times V/W]$, we get that 
$Q/(\qG^\gauche +\CG\, Q) \simeq \kb[\CCB]/\CG \otimes \kb[V/W]$ and so 
$\qG^\gauche + \CG\, Q$ is a prime ideal of $Q$. We will denote it by 
$\qG_\CG^\gauche$.\indexnot{qa}{\qG_\CG^\gauche}

\bigskip

\begin{coro}\label{q gauche unique}
We have $Q/\qG^\gauche_\CG \simeq P/\pG^\gauche_\CG$. Moreover, 
$\qG_\CG^\gauche$ is the unique prime ideal of $Q$ lying over 
$\pG_\CG^\gauche$ and contained in $\qGba$. 
\end{coro}

\bigskip

\begin{proof}
The first statement is immediate and the second one follows from the first one. 
\end{proof}

\bigskip

\begin{rema}\label{R rgauche}
It has been shown in Corollary~\ref{Q extention kc} that, if 
$\qGba_\star$ is a prime ideal of $Q$ lying over $\pGba$, then 
$Q/\qGba_\star \simeq P/\pGba$. Even though $Q/\qG^\gauche \simeq P/\pG^\gauche$, 
we will see in Chapter~\ref{chapitre:b2} that this cannot be extended 
in general to other prime ideals of $Q$ lying over 
$\pG^\gauche$: indeed, if $W$ has type $B_2$, then there exists 
a prime ideal $\qG_\star^\gauche$ of $Q$ lying over $\pG^\gauche$ 
such that $P/\pG^\gauche$ is a proper subring of $Q/\qG^\gauche_\star$ 
(see Lemma~\ref{lem:qgauche-b2}(c)). So, in general, 
$\Kb^\gauche \varsubsetneq \Mb^\gauche_\CG$, where $\CG=0$.\finl
\end{rema}

\bigskip

%
%
%
%
%
%
%

\bigskip

\bigskip

%

\begin{prop}\label{coro:app-lim}
Let $\zG_*^\gauche$ be a prime ideal of $Z$ lying over $\pG_\CG^\gauche$. 
Then there exists a unique prime ideal of $Z$ lying over $\pGba_\CG$ and containing $\zG_*^\gauche$: 
it is equal to $\zG_*^\gauche + \langle Z_{<0}, Z_{>0} \rangle$. 
\end{prop}

\bigskip

\begin{proof}
	The proposition follows from Lemma \ref{le:barfromleft} applied
	to the extension $Z/P$ and the prime ideals $\pGba_\CG$ and $\zG_*^\gauche$.
\end{proof}

\bigskip

Proposition~\ref{coro:app-lim} provides a surjective map 
\equat\label{eq:limitegauche}
\fonction{\limitegauche}{\Upsilon^{-1}(\pG_\CG^\gauche)}{\Upsilon^{-1}(\pGba_\CG)}{\zG_*^\gauche}{\zG_*^\gauche + 
\langle Z_{<0}, Z_{>0} \rangle.}
\endequat
The notation $\limitegauche$ will be justified in Chapter~\ref{chapter:bb}. 

\bigskip

\section{Gaudin algebra and cellular characters}
\subsection{Left specialization}

Let $Z^{\prime\gauche}=Z'\otimes_{\kb[V^*]^W}\kb=Z^{\reg,\gauche}
\otimes_{\kb[V^\reg]^W}\kb[V^\reg]$, a 
$\kb[\CCB\times V^\reg]$-algebra, free of rank $|W|$ as a 
$\kb[\CCB\times V^\reg]$-module. It is acted on by $W$ and
$(Z^{\prime\gauche})^W=Z^{\reg,\gauche}$.

Lemma \ref{le:Zprime} shows that $\Theta$ gives an isomorphism
$Z^{\prime\gauche}\rtimes W\xrightarrow{\sim}\Hb^{\reg,\gauche}$ and the
image of $Z^{\prime\gauche}$ by that map is the 
$\kb[\CCB\times V^\reg]$-subalgebra generated by $\Theta^{-1}(V)$.

There is a Morita equivalence between $\Hb^{\reg,\gauche}$ and
$Z^{\reg,\gauche}$ given by the bimodule
$\Delta(\mathrm{co})^\reg=\Delta(\mathrm{co})\otimes_{\kb[V]}\kb[V^\reg]
=\Hb^\reg e\otimes_{\kb[V^*]^W}\kb=\Hb^{\reg,\gauche}e$.
It corresponds to the Morita equivalence between $Z^{\prime\gauche}\rtimes W$
and $Z^{\reg,\gauche}$ given by $Z^{\prime\gauche}$. Note in 
particular that $Z^{\prime\gauche}\rtimes W$ acts faithfully on
$\Delta(\mathrm{co})^\reg$.

\subsection{Gaudin algebra}
There is a canonical isomorphism $\kb[\CCB\times V^\reg]\otimes \kb W
\xrightarrow{\sim}
\Hb^\reg\otimes_{\kb[V^*]}\kb=\Delta^\reg(\kb W)$. Through this isomorphism,
the action of $\Theta^{-1}(y)$ (for $y\in V)$ is given by left multiplication by
$$\DC_y=\sum_{s \in \REF(W)} 
\e(s)C_s \frac{\langle y,\a_s\rangle}{\a_s} s\,\,\in \kb[\CCB \times V^\reg][W]
\indexnot{DyC}{\DC_y}
$$
where
$\kb[\CCB \times V^\reg][W]$ \indexnot{ka}{\kb[\CCB \times V^\reg][W]}
denote the group algebra of $W$ over the algebra $\kb[\CCB \times V^\reg]$: 
note that, in this algebra, the elements of $\kb[\CCB \times V^\reg]$ and those of $W$ commute. 
In other words, $\kb[\CCB \times V^\reg][W] =\kb[\CCB \times V^\reg]\otimes
\kb W$ as an algebra. Note that the action of $\DC_y$ by left multiplication
on the specialization of $\kb[\CCB \times V^\reg][W]$ at a closed
point $(c,v)$ is the operator $D_y^{c,v,0}$ of \S\ref{se:Gaudinoperators}.

We denote by $\gaudin(W)$ \indexnot{G}{\gaudin(W)} (resp. $\Liegaudin(W)$ 
\indexnot{G}{\Liegaudin(W)})
the $\kb[\CCB \times V^\reg]$-subalgebra (resp. submodule) of 
$\kb[\CCB \times V^\reg][W]$ generated by 
the $\DC_y$'s ($y \in V$). It will be called the {\it generic Gaudin algebra}  (resp.
{\it generic Gaudin Lie algebra}) associated with $W$. Note that it is commutative.

Gaudin Lie algebras and Proposition \ref{pr:Liegaudin} below are in
\cite{AgFeVe} when $W$ is a real reflection group.

\begin{prop}
	\label{pr:Liegaudin}
$\Liegaudin(W)$ is a commutative Lie subalgebra of $\kb[\CCB \times V^\reg][W]$.
	The $\kb[\CCB \times V^\reg]$-module $\Liegaudin(W)$
	is free of rank $\dim(V/V^W)$.

	Let $y'\in V^\reg$ and $c\in \CCB$ be $\kb$-points
	with $c_s\neq 0$ for all
	$s\in\REF(W)$. Then the image of $\Liegaudin(W)$ in
	$\kb[W]$, the specialization of $\kb[\CCB \times V^\reg][W]$ at $(y',c)$,
	is a $\kb$-vector space of dimension $\dim(V/V^W)$.

\end{prop}

\begin{proof}
	Note that the map $a\otimes y\mapsto a\DC_y$ for $y\in V$ and 
	$a\in \kb[\CCB \times V^\reg]$ gives a surjective morphism of 
	$\kb[\CCB \times V^\reg]$-modules
	$f:\kb[\CCB \times V^\reg]\otimes V\to \Liegaudin(W)$. We have
	$f(a\otimes y)=0$ if $y\in V^W$. Consider now 
	$b\in \kb[\CCB \times V^\reg]\otimes V$ with $f(b)=0$.
	Given $s\in\REF(W)$, we have
	$\langle b,\alpha_s\rangle=0$. It follows that $b\in
	\kb[\CCB \times V^\reg]\otimes V^W$. The first part of the 
	proposition follows.

	The second part of the proposition follows by a similar argument.
\end{proof}

Let $R=\kb[\CCB \times V^\reg][\{C_s^{-1}\}_{s\in\REF(W)}$.
	One can deduce from the proposition above that the $R$-module
$R\otimes_{\kb[\CCB \times V^\reg]}\Liegaudin(W)$
is a direct summand of $R[W]$, because the quotient module is flat, hence
projective.

\smallskip
Note also that the external action of $W$ stabilizes $\gaudin(W)$:
we have $\lexp{w}{\DC_y}=\DC_{w(y)}$
for $y \in V$ and $w \in W$. Therefore, we can consider the algebra 
$\gaudin(W) \rtimes W$.

\bigskip
Recall that $\Delta^\reg(\mathrm{co})$ has a filtration with
$\mathrm{gr}\ \Delta^\reg(\mathrm{co})\simeq\Delta(\kb W)$.
There is a filtration of
$\Hb^{\reg,\gauche}$ given for $r\le 0$ by
$$\Hb^{\reg,\gauche,\le r}=
\{h\in\Hb^{\reg,\gauche}\ |\ h\Delta^\reg(\mathrm{co})^{\le i}\subset
\Delta^\reg(\mathrm{co})^{\le i+r},\ \forall i\le 0\}.$$
Via the isomorphism $Z^{\prime\gauche}\rtimes W\xrightarrow{\sim}
\Hb^{\reg,\gauche}$,
we obtain a filtration with
$(Z^{\prime\gauche}\rtimes W)^{\le r}=(Z^{\prime\gauche})^{\le r}\otimes\kb W$.
We deduce the following:
\begin{itemize}
\item The algebra $\gaudin(W)$ is the image of 
$Z'=\Theta^{-1}(\kb[\CCB \times V^\reg \times V^*])$ in $\End(\Delta(\kb W))$
\item The algebra $\gaudin(W)^W$ is the image of 
$Z^\reg=\Theta^{-1}(\kb[\CCB \times V^\reg \times V^*]^{\Delta W})$
in $\End(\Delta(\kb W))$
\item The algebra $\gaudin(W)\rtimes W$ is the image of
$\Hb^{\reg,\gauche}$ in $\End(\Delta(\kb W))$
\item the kernel of the action of $\Hb^{\reg,\gauche}$ on
$\Delta(\kb W)$ is a nilpotent ideal.
\end{itemize}

Thanks to Proposition~\ref{pr:equivMorita}, we also deduce that
$\gaudin(W)$ induces a Morita equivalence between
$\gaudin(W)\rtimes W$ and $\gaudin(W)^W$.

\bigskip
\def\gaudincellulaire#1{{\g^{\mathrm{Gau}}_{\! #1}}}

\subsection{Cellular characters}\label{sub:cellular-gaudin}

Let $\mG$ be a maximal ideal of $\Kb_\CG^\gauche\gaudin(W)$.
We define a character of $W$
$$\gaudincellulaire{\mG}=([(\kb(\CG)(V)\otimes\kb W)_{\mG}]_{
(\Kb_\CG^\gauche\gaudin(W))_\mG})^*.$$
\indexnot{gzGau}{\gaudincellulaire{\mG}}

By Proposition~\ref{pr:equivMorita}, the restriction map induces
a bijection
\equat\label{eq:gaudin-simples}
\left(\Irr(\Kb_\CG^\gauche \gaudin(W))\right)/W \xrightarrow{\sim}
\Irr(\Kb_\CG^\gauche \gaudin(W)^W).
\endequat

Let $\zG$ be the maximal ideal of $Z_\CG^\gauche$ corresponding to
the orbit of $\mG$ via this bijection.

\begin{theo}\label{theo:gaudin-cellulaire}
We have $\gaudincellulaire{\mG}=\g_\zG^\calo$.
\end{theo}

Note that, as consequence, we have
$\gaudincellulaire{{\lexp{w}{\mG}}}=\gaudincellulaire{\mG}$ for all $w \in W$. 

There is a corresponding result for cellular multiplicities:
$$\longueur_{(\Kb_\CG^\gauche \gaudin(W))_\mG}((\kb(\CG)(V)\otimes E)_\mG)=
\mult_{\zG,\chi}^\calo.$$

\medskip
Consider now an algebraically closed field $K$ and a $K$-point
$p$ of $\Spec(\kb(\CG)[V^\reg])$ outside the ramification locus of
$f:\Spec(\kb(\CG)Z^{\prime\gauche})\to \Spec(\kb(\CG))\times V^\reg$.

There is a bijection $\Irr(\Kb_\CG^\gauche \gaudin(W))\xrightarrow{\sim}
f^{-1}(p)$. Denote by $z_\mG$ the point of $f^{-1}(p)$ corresponding to
$\mG$. We have $\g_\mG^{\mathrm{Gau}}=[K(f^{-1}(p))]_{KW}$.
So, the $\CG$-cellular characters are the generalized eigenspaces of the Gaudin
operators at $p$.

\bigskip

\begin{rema}\label{rem:compute}
The description of Calogero-Moser cellular characters provided by
Theorem~\ref{theo:gaudin-cellulaire} allows
efficient computations in small groups.\finl
\end{rema}

\bigskip

\bigskip

\section{Projective modules for Hecke algebras}
\label{se:T=1to0}

Let $c,c'\in\CCB(\CM)$ with $c\neq 0$.
Consider the morphism of $\CM$-algebras
$\CM[\CCB]\to\CM(\hbar),\ C_s\mapsto \hbar c_s+c'_s$ and the corresponding
embedding  $\CM\hookrightarrow \CCB$ given by $\hbar \mapsto \hbar c+c'$.

We consider also the morphism of $\CM$-algebras
$\CM[\CCBt]\to\CM(\hbar),\ T\mapsto\hbar^{-1},\ C_s\mapsto c_s+\hbar^{-1}c'_s$.
There is an isomorphism of $\CM(\hbar)$-algebras
$\CM(\hbar)\otimes_{\CM[\CCB]}\Hbd\xrightarrow{\sim}
\CM(\hbar)\otimes_{\CM[\CCBt]}\Hbt$
(cf. \S \ref{subsection:bi-graduation-1}). It induces an equivalence
$\tilde{\OC}\left(\CM(\hbar)\otimes_{\CM[\CCB]}\CM[\CCBt]/(T-1)\right)
\xrightarrow{\sim} \tilde{\OC}(\CM(\hbar))$
and an isomorphism 
$K_0(\tilde{\OC}\left(\CM(\hbar)\otimes_{\CM[\CCB]}\CM[\CCBt]/(T-1)\right))
\xrightarrow{\sim} K_0(\tilde{\OC}(\CM(\hbar)))$.

Recall (Proposition \ref{pr:gradednongradedVerma})
that there is an equivalence 
$$\dot{\OC}(\CM(\hbar))^{(\ZM)}\xrightarrow{\sim}
\tilde{\OC}\left(\CM(\hbar)\otimes_{\CM[\CCB]}\CM[\CCBt]/(T-1)\right).$$
Composing with the equivalence above provides an equivalence
$\dot{\OC}(\CM(\hbar))^{(\ZM)}\xrightarrow{\sim}
\tilde{\OC}(\CM(\hbar))$,
hence an isomorphism of $\ZM[\tb^{\pm 1}]$-modules
$$K_0(\dot{\OC}(\CM(\hbar)))[\tb^{\pm 1}]
\xrightarrow{\sim} K_0(\tilde{\OC}(\CM(\hbar))),\
[\dot{\Delta}(E)]\mapsto [\tilde{\Delta}(E)]\cdot\tb^{\logexp(c_E+c'_E)}$$
where $\logexp(z)=z-\Log(\exp(z))$ and $\Log$ is the principal branch of logarithm.

\medskip
Consider the discrete valuation ring $\CM[\hbar^{-1}]_{(\hbar^{-1})}$.
There is a decomposition map (cf. \S \ref{se:hwdecmap}) 
$K_0(\tilde{\OC}(\CM(\hbar)))\to K_0(\tilde{\OC}(\CM_c))$
corresponding to the specialization $\hbar^{-1}\mapsto 0$, hence $T\mapsto 0$.
Composing with the isomorphism
above provides a morphism of $\ZM[\tb^{\pm 1}]$-modules
$K_0(\dot{\OC}(\CM(\hbar)))[\tb^{\pm 1}]\to K_0(\tilde{\OC}(\CM_c))$.

Forgetting the gradings, i.e. setting $\tb=1$, we obtain a morphism of
abelian groups $d'$ from $K_0(\dot{\OC}(\CM(\hbar)))$ to the Grothendieck
group $L$ of the category of finitely generated $\Hb_c$-modules
that are locally nilpotent for $V$.

\smallskip
Applying $\CM(V)^W\otimes_{\CM[V]^W}-$ provides a morphism
from $L$ to the Grothendieck
group $L'$ of the category of finitely generated 
$(\CM(V)^W\otimes_{\CM[V]^W}\Hb_c)$-modules
that are locally nilpotent for $V$. Composing with $d'$, we obtain
a morphism $d'' : K_0(\dot{\OC}(\CM(\hbar))) \to L'$.

Since every finitely generated $(\CM(V)^W\otimes_{\CM[V]^W}\Hb_c)$-module
that is locally nilpotent for $V$ is a finite extension of
$\Kb_c^{\gauche}\Hb$-modules, it follows that pullback through the
quotient map $\CM(V)^W\otimes_{\CM[V]^W}\Hb_c\twoheadrightarrow
\Kb_c^{\gauche}\Hb$
induces an isomorphism $K_0(\Kb_c^{\gauche}\Hb\mmod)\xrightarrow{\sim}L'$.
Composing $d''$ with the inverse of that map provides a morphism
$d''':K_0(\dot{\OC}(\CM(\hbar)))\to K_0(\Kb_c^{\gauche}\Hb\mmod)$.

\medskip

Consider the morphisms of $\CM$-algebras
$\CM[\TCB]\to \CM[\qb^\CM],\ \qb_{\orbite,j}\mapsto
\qb^{k_{\orbite,-j}}e^{2i\pi r k'_{H,-j}/|\mub_W|}$ and
$\CM[\qb^\CM]\to \CM((\hbar)),\ \qb^r\mapsto e^{2i\pi\hbar r/|\mub_W|}$.
We have an isomorphism $K_0(\CM(\qb^\CM)\heckegenerique\mmod)
\xrightarrow{\sim} K_0(\CM((\hbar))\heckegenerique\mmod)$.


Given $M\in\dot{\OC}(\CM(\hbar))$ with $\CM[V^\reg]\otimes_{\CM[V]}M=0$,
we have $d''([M])=0$. It follows from Theorem \ref{th:doubleendo}
that $d''$ factors through
$K_0(\dot{\OC}(\CM(\hbar)))\xrightarrow{\sim}
K_0(\dot{\OC}(\CM((\hbar))))\xrightarrow{KZ} K_0(\CM((\hbar))
\heckegenerique\mmod)$. This provides a morphism
$$d_{c,c'}:K_0(\CM(\qb^\CM)\heckegenerique\mmod)\to
K_0(\Kb_c^{\gauche}\Hb\mmod).$$

\medskip
Let us summarize these constructions in the following theorem.

\begin{theo}
\label{th:decmapHeckecells}
There is a (unique) morphism
$d_{c,c'}:K_0(\CM(\qb^\CM)\heckegenerique\mmod)\to K_0(\Kb_c^\gauche\Hb\mmod)$
such that
$d_{c,c'}([E^\gen])=[\Kb_c^\gauche\Delta(E)]$ for all $E\in\Irr(W)$.

Given $L$ a simple $\CM(\qb^\CM)\heckegenerique$-module, there are
non-negative integers $d_{L,\zG}$ such that
	$d_{c,c'}([L])=\sum_{\zG\in\Upsilon^{-1}(\pG_c^\gauche)}
	d_{L,\zG}([L_c^\gauche(\zG)])$.
\end{theo}

\smallskip

Here is a commutative diagram summarizing the situation (we indicate above each column which 
specialization it corresponds to):

$$\xymatrix{
{\SS{T\mapsto 1,\ C_s\mapsto\hbar c_s+c_s'}} & 
{\SS{T\mapsto\hbar^{-1},\ C_s\mapsto c_s+\hbar^{-1} c_s'}} & 
{\SS{T \mapsto 0,\ C_s \mapsto c_s}} \\
K_0(\dot{\OC}(\CM(\hbar)))[\tb^{\pm 1}]\ar[r]^-\sim \ar[d]_-{\tb=1} & 
 K_0(\tilde{\OC}(\CM(\hbar^{-1})))\ar[r]^{\text{dec. map}} &
 K_0(\tilde{\OC}(\CM_c))\ar[d]^{\text{forget grading}} \\
K_0(\dot{\OC}(\CM(\hbar)))\ar[d]_-{\KZ}\ar@{-->}[rr] &&
 K_0(\Hb_c\mmod_{V-\text{loc. nilp.}})\ar[d]^{\text{localization}}\\
K_0(\CM(\qb^\CM)\heckegenerique\mmod)\ar@{-->}[rr]\ar@{-->}[drr]_-{d_{c,c'}}&&
 K_0((\CM(V)^W\otimes_{\CM[V]^W}\Hb_c)\mmod_{V-\text{loc. nilp.}}) \\
&& K_0(\Kb_c^{\gauche}\Hb\mmod)\ar[u]_-{\sim}
}$$

\begin{rema}
The discussion above shows that the cellular multiplicities measure
something like the regular part of
the characteristic cycle of a Verma module $\CM_c\dot{\Delta}(E)$
(with respect to the filtration $\Hbd^\trianglelefteq$ of 
\S\ref{sub:Rees}), although this doesn't seem to fit with
the usual characteristic cycle theory.\finl
\end{rema}

\begin{coro}\label{coro:proj-cellular}
The Calogero-Moser $c$-cellular 
characters are sums of
classes of projective indecomposable
$(\CM(\qb^\CM)\heckegenerique)$-modules.
\end{coro}

\begin{rema}
When $W$ has a unique class of reflections and $c\neq 0$, the algebra
$\CM(\qb^\CM)\heckegenerique\mmod$ is
semisimple, so Theorem \ref{th:decmapHeckecells}
brings no information on cellular characters.\finl
\end{rema}

\begin{rema}
When $W$ is a Coxeter group of type $B_n$, the Lusztig cellular characters
for equal parameters are characters of projective indecomposable 
	modules \cite{LeMi}.
\end{rema}

\bigskip

\chapter{Bialynicki-Birula cells of $\ZCB_c$}\label{chapter:bb}

\boitegrise{{\bf Assumption.} 
{\it In this chapter \S\ref{chapter:bb},
we assume that $\kb=\CM$ and we fix an element $c \in \CCB$.}}{0.75\textwidth}

\bigskip

The group $\CM^\times$ acts on the algebraic variety $\ZCB_c$. We shall interprete geometrically 
several notions introduced in this book (families, cellular characters,...) using this action 
(fixed points, attractive or repulsive sets...). The main result of this chapter is concerned with the case of a family corresponding to a smooth point of $\ZCB_c$: 
we will show that the associated cell characters are irreducible. This result will be seen 
as a geometric result. Indeed, the smoothness of the fixed point implies that the attractive 
and repulsive sets are affine spaces which intersect properly and transversally; 
a computation of the intersection multiplicity will conclude the proof 
(see Theorem~\ref{theo:cellulaire-lisse-geometrique}). Note that another proof can be obtained
using results of Bellamy that appeared after the first version of this book
\cite{bellamyVerma}.

\bigskip

\def\action{{\hskip0.2mm\SS{\bullet}\hskip0.5mm}}

\section{Generalities on $\CM^\times$-actions}

\medskip

Let $\XCB$ be an {\it affine} algebraic variety endowed with a regular 
$\CM^\times$-action $\CM^\times \times \XCB \to \XCB$, $(\xi,x) \mapsto \xi \action x$. 
We will denote by $\XCB^{\CM^\times}$ the closed subvariety consisting of 
the fixed points under the action of $\CM^\times$. 
Given $x \in \XCB$, we say that {\it $\lim_{\xi \to 0} \xi \action x$ exists and is equal to $x_0$} 
if there exists a morphism of varieties $\ph : \CM \to \XCB$ such that, 
if $\xi \in \CM^\times$, then $\ph(\xi)=\xi \action x$ 
and $\ph(0)=x_0$. It is then clear that $x_0 \in \XCB^{\CM^\times}$. 
Similarly, we will say that {\it $\lim_{\xi \to 0} \xi^{-1} \action x$ exists and is equal to $x_0$} 
if there exists a morphism of varieties  $\ph : \CM \to \XCB$ such that, 
if $\xi \in \CM^\times$, then $\ph(\xi)=\xi^{-1} \action x$ and $\ph(0)=x_0$. 

We denote by $\XCB^\attractif$ (respectively $\XCB^\repulsif$) the set of $x \in \XCB$ 
such that $\lim_{\xi \to 0} \xi \action x$ (respectively $\lim_{\xi \to 0} \xi^{-1} \action x$) exists. 
It is a closed subvariety of $\XCB$ and the maps 
$$\fonction{\limiteattractive}{\XCB^\attractif}{\XCB^{\CM^\times}}{x}{\lim_{\xi \to 0} \xi \action x}$$
$$\fonction{\limiterepulsive}{\XCB^\repulsif}{\XCB^{\CM^\times}}{x}{\lim_{\xi \to 0} \xi^{-1} \action x}
\leqno{\text{and}}$$
are morphisms of varieties (which are of course surjective: a section is given by the 
closed immersion $\XCB^{\CM^\times} \subset \XCB^\attractif \cap \XCB^\repulsif$), because  
$\XCB$ is affine by assumption (this follows from the facts that this is true for the affine space 
$\CM^N$endowed with a linear action of $\CM^\times$ and that $\XCB$ can be seen as a $\CM^\times$-stable 
closed subvariety of such a $\CM^N$). Note that this is no longer true in general if $\XCB$ is not affine, 
as it is shown by the example $\Pb^1(\CM)$ endowed with the action $\xi \action [x;y]=[\xi x;y]$. 

Finally, given $x_0 \in \XCB^{\CM^\times}$, we denote by $\XCB^\attractif(x_0)$ (respectively 
$\XCB^\repulsif(x_0)$) the inverse image of $x_0$ by the map $\limiteattractive$ 
(respectively $\limiterepulsive$). The closed subvariety $\XCB^\attractif(x_0)$
(respectively $\XCB^\repulsif(x_0)$) 
will be called the {\it attractive set} (respectively the {\it repulsive set}) of $x_0$: 
it is a closed subvariety of $\XCB$. Let us recall the following classical fact, due to 
Bialynicki-Birula~\cite{bialynicki}:

\bigskip

\begin{prop}\label{prop:bb-lisse}
If $x_0 \in \ZCB^{\CM^\times}$ is a smooth point of $\XCB$, then there exists $N \ge 0$ such that 
$\XCB^\attractif(x_0) \simeq \CM^N$. In particular, $\XCB^\attractif(x_0)$ is smooth 
and irreducible.

The same statements hold for $\XCB^\repulsif(x_0)$.
\end{prop}

\bigskip

We will describe the notions developed in the previous chapters
(families, cellular characters) 
via fixed points and attractive sets of the
$\CM^\times$-action on $\ZCB_c$. 

\bigskip

\section{Fixed points and families}\label{se:fixedpointsfamilies}

\medskip

The results of this section \S\ref{se:fixedpointsfamilies}
are due to Gordon~\cite{gordon}. 
There are several $\CM^\times$-actions on all of our varieties
($\PCB$, $\ZCB$, $\RCB$,...). 
We will use the one which induces the $\BZ$-grading of Example~\ref{Z graduation-1}. 
In other words, an element $\xi \in \CM^\times$ acts on  
$\Hb$ as the element $(\xi^{-1},\xi,1 \rtimes 1)$ of 
$\CM^\times \times \CM^\times \times (\Hom(W,\CM^\times) \rtimes \NC)$. 
Therefore, for the action on $\Hb$, $\xi$ acts trivially on $\CM[\CCB] \otimes \CM W$, 
acts with non-negative weights on $\CM[V]$, with non-positive weights on $\CM[V^*]$. 
We get an action on 
$P$ and $Z_c$, which induces regular actions of $\CM^\times$ on the varieties 
$\PC_\bullet \simeq V/W \times V^*/W$ 
and $\ZCB_c$ making the morphism 
$$\Upsilon_c : \ZCB_c \longto \PCB_\bullet = V/W \times V^*/W$$
$\CM^\times$-equivariant. Given $\xi \in \CM^\times$ and $z \in \ZCB_c$, the image of $z$ through this action 
of $\xi$ will be denoted by $\xi \action z$. The unique fixed point of $\PCB_\bullet$ is $(0,0)$:
\equat\label{eq:fixe-P}
\PCB_\bullet^{\CM^\times} = (0,0).
\endequat
Since $\Upsilon_c$ is a finite morphism, we deduce that 
\equat\label{eq:fixe-Z}
\ZCB_c^{\CM^\times} = \Upsilon_c^{-1}(0,0).
\endequat

\begin{prop}
The construction above
provides a bijection between $\ZCB_c^{\CM^\times}$ and the set of
Calogero-Moser $c$-families.
\end{prop}

\bigskip

\section{Attractive sets and cellular characters} 

\medskip

First of all, note that 
\equat\label{eq:P-attractif}
\PC_\bullet^\attractif = V/W \times 0 \subset V/W \times V^*/W\quad\text{and}\quad 
\PC_\bullet^\repulsif = 0 \times V^*/W \subset V/W \times V^*/W.
\endequat
In other words, $\PC_\bullet^\attractif$ is the irreducible subvariety of $\PCB_\bullet$ associated 
with the prime ideal $\pG_c^\gauche$. 
Moreover, since $\Upsilon_c$ is a finite morphism, we have 
\begin{lem}\label{eq:Z-attractif}
We have
$\ZCB_c^\attractif = \Upsilon_c^{-1}(V/W \times 0)$
and
$\ZCB_c^\repulsif = \Upsilon_c^{-1}(0 \times V^*/W)$.
\end{lem}

\begin{proof}
Let $\r : Z_c \to \CM[\tb,\tb^{-1}]$ 
be a morphism of $\CM$-algebras such that $\r(P_\bullet) \subset \CM[\tb]$.
Since $Z_c$ is integral over $P_\bullet$, it follows that $\r(Z_c)$ is integral over
$\r(P_\bullet)$. As $\CM[\tb]$ is integrally closed, we deduce that 
$\r(Z_c) \subset \CM[\tb]$. This shows that
$\PCB_\bullet^\attractif = V/W \times 0\subset
\Upsilon_c(\ZCB_c^\attractif)$. The reverse inclusion is clear, and the other equality is
proven similarly.
\end{proof}

We have the following immediate consequence.

\begin{prop}
There is a bijection from
$\Upsilon_c^{-1}(\pG_c^\gauche)$ to
the set of irreducible components of $\ZCB_c^\attractif$ sending
$\zG$ to the corresponding irreducible closed subvariety $\ZCB_c^\attractif[\zG]$.
\end{prop}

\bigskip

Since $\limiteattractive : \ZCB_c^\attractif \to \ZCB_c^{\CM^\times}$ is a morphism of varieties, 
the image of $\ZCB_c^\attractif[\zG]$ is irreducible. As $\ZCB_c^{\CM^\times}$ is a
finite set, we deduce that 
$\limiteattractive(\ZCB_c^\attractif[\zG])$ is reduced to a point. Hence, the morphism of 
varieties $\limiteattractive : \ZCB_c^\attractif \to \ZCB_c^{\CM^\times}$ induces  
a surjective map $\Upsilon_c^{-1}(\pG_c^\gauche) \longto \Upsilon_c^{-1}(\pGba_c)$: this
is the map $\limitegauche$ defined in (\ref{eq:limitegauche}).

%
%
%
%
%
%
%
%

\bigskip

\section{The smooth case}\label{se:smooth}

\medskip

\boitegrise{{\bf Assumption and notation.} 
{\it We fix in \S\ref{se:smooth} a point $z_0 \in \ZCB_c^{\CM^\times}$ 
which is assumed to be {\bfit smooth} in $\ZCB_c$. We denote by $\chi$ the unique 
irreducible character of the associated Calogero-Moser $c$-family.}}{0.75\textwidth}

\bigskip

Since $z_0$ is smooth (and isolated), we have
$$\ZCB_c^\attractif(z_0) \simeq \CM^N$$
for some $N$, hence
$\ZCB_c^\attractif(z_0)$ is smooth and irreducible (Proposition~\ref{prop:bb-lisse}). 
This shows that $\limiteattractiveinverse(z_0)$ is irreducible and isomorphic to an affine 
space. We denote by $\zG_L$ the prime ideal of $Z_c$ corresponding to this irreducible 
subvariety of $\ZCB_c$. 
The aim of this section is to show the following result.

\bigskip

\begin{theo}\label{theo:cellulaire-lisse-geometrique}
The celular character $\g_{\zG_L}^\calo$ is irreducible, i.e. $\g_\zG^\calo=\chi$.
\end{theo}

%
%
%
%

\smallskip
We will provide a geometrical proof of Theorem~\ref{theo:cellulaire-lisse-geometrique}.
An entirely algebraic proof can be deduced from
\cite[Theorem 10(3)]{bellamyVerma}. In type $A_n$, this can also be deduced,
using Gaudin operators (cf. Chapter \ref{ch:Gaudin}), from
\cite{MuTaVa2}.

\bigskip

\noindent{\sc Notation - } Given $A$ is a commutative local ring with maximal ideal $\mG$
and given
$M$ a finitely generated $A$-module, we denote by $e_\mG(M)$ the {\it multiplicity} of $M$ 
for the ideal $\mG$, as it is defined in~\cite[Chapitre~\MakeUppercase{\romannumeral 5},~\S{A.2}]{serre}. 

Let $A$ be a regular commutative ring (not necessarily local) and $M$ and $N$ two 
finitely generated $A$-modules such that $M \otimes_A N$ has finite length.
Given $\aG$ a prime ideal of $A$, we put
$$\chi_\aG(M,N)=\sum_{i=0}^{\dim A} (-1)^i \longueur_{A_\aG}(\Tor_i^A(M,N)_\aG),$$
as in~\cite[Chapitre~\MakeUppercase{\romannumeral 5},~\S{B},~Th\'eor\`eme~1]{serre} .\finl

\bigskip

\begin{proof}
%
%
%
%
We define similarly $\zG_R$ as 
being the defining ideal of $\ZCB_c^\repulsif(z_0)$: we denote by $\g_{\zG_R}^{\calo,\droite}$ 
the associated {\it right} Calogero-Moser $c$-cellular character.
We have
$$\g_{\zG_L}^\calo=m_L \chi(1) \quad\text{and}\quad \g_{\zG_R}^{\calo,\droite}=
m_R \chi\leqno{(\clubsuit)}$$
for some $m_L$, $m_R$. We must show that $m_L=m_R=1$. 

Let us first compute the multiplicity of the $Z_{c,\zGba}$-module $Z_{c,\zGba}/\pG_c^\gauche Z_{c,\zGba}$ 
for $\zGba Z_{c,\zGba}$. The Krull dimension of this module is $n=\dim_\kb V$. 
By the additivity formula~\cite[Chapitre~\MakeUppercase{\romannumeral 5},~\S{A.2}]{serre}, 
we have
$$e_{\zGba Z_{c,\zGba}}(Z_{c,\zGba}/\pG_c^\gauche Z_{c,\zGba}) = \sum_{{\mathrm{coht}}(\zG)=n} 
\longueur_{Z_{c,\zG}}(Z_{c,\zG}/\pG_c^\gauche Z_{c,\zG}) e_{\zGba Z_{c,\zGba}}(Z_{c,\zGba}/\zG Z_{c,\zGba}).
\leqno{(\diamondsuit)}$$
Here, ${\mathrm{coht}}(\zG)$ denotes the coheight of the prime ideal $\zG$ of $Z_{c,\zGba}$. 
Since $\ZCB_c^\attractif(z_0)$ is irreducible of dimension $n$, there is only one prime 
ideal of $Z_{c,\zGba}$ with coheight $n$ which contains $\pG_c^\gauche Z_{c,\zGba}$ (and so such that  
$\longueur_{Z_{c,\zG}}(Z_{c,\zG}/\pG_c^\gauche Z_{c,\zG})$ is non-zero), this is the prime 
ideal $\zG_L$. Moreover, since $Z_{c,\zGba}/\zG_L Z_{c,\zGba}$ is a regular ring 
(because $\ZCB_c^\attractif(z_0)$ is smooth), 
the multiplicity $e_{\zGba Z_{c,\zGba}}(Z_{c,\zGba}/\zG Z_{c,\zGba})$ is equal to $1$ 
(see~\cite[Chapitre~\MakeUppercase{\romannumeral 4}]{serre}). Hence, it follows from $(\clubsuit)$ and 
$(\diamondsuit)$ that 
$$e_{\zGba Z_{c,\zGba}}(Z_{c,\zGba}/\pG_c^\gauche Z_{c,\zGba})=m_L \chi(1).\leqno{(\heartsuit_L)}$$
By symmetry, 
$$e_{\zGba Z_{c,\zGba}}(Z_{c,\zGba}/\pG_c^\droite Z_{c,\zGba})=m_R \chi(1).\leqno{(\heartsuit_R)}$$

\medskip

On the other hand, $P/\pG_c^\gauche$ is a polynomial algebra and $Z_c/\pG_c^\gauche Z_c$ is a free 
$P/\pG_c^\gauche$-module of rank $|W|$. So $Z_{c,\zGba}/\pG_c^\gauche Z_{c,\zGba}$ is a Cohen-Macaulay 
$Z_{c,\zGba}$-module of dimension $n$. Similarly, $Z_{c,\zGba}/\pG_c^\droite Z_{c,\zGba}$ 
is a Cohen-Macaulay $Z_{c,\zGba}$-module of dimension $n$. Since $Z_c$ has dimension $2n$, 
it follows from~\cite[Chapitre~\MakeUppercase{\romannumeral 5},~\S{B},~Corollaire~du~Th\'eor\`eme~4]{serre}
that 
$$\chi_{\zGba}(Z_{c,\zGba}/\pG_c^\gauche Z_{c,\zGba},Z_{c,\zGba}/\pG_c^\droite Z_{c,\zGba}) 
= \longueur_{Z_{c,\zGba}}(Z_{c,\zGba}/\pG_c^\gauche Z_{c,\zGba} \otimes_{Z_{c,\zGba}} Z_{c,\zGba}/\pG_c^\droite Z_{c,\zGba}) 
\!=\! \chi(1)^2 > 0.\leqno{(\spadesuit)}$$
The last equality follows from the fact that $\pG_c^\gauche+\pG_c^\droite=\pGba_c$ and 
Corollary~\ref{dim bonne}. 

Consequently~\cite[Chapitre~\MakeUppercase{\romannumeral 5},~\S{B},~Compl\'ement~du~Th\'eor\`eme~1]{serre}, 
$$e_{\zGba Z_{c,\zGba}}(Z_{c,\zGba}/\pG_c^\gauche Z_{c,\zGba}) \cdot 
e_{\zGba Z_{c,\zGba}}(Z_{c,\zGba}/\pG_c^\droite Z_{c,\zGba}) \le 
\chi_{\zGba}(Z_{c,\zGba}/\pG_c^\gauche Z_{c,\zGba},Z_{c,\zGba}/\pG_c^\droite Z_{c,\zGba}).$$
From this last equality and
$(\heartsuit_L)$, $(\heartsuit_R)$ and $(\spadesuit)$, we deduce that
$$m_L m_R \le 1.$$
We obtain $m_L=m_R=1$, as desired.
%
%
%
%
\end{proof}

\bigskip

\part{The extension $Z/P$}\label{part:extension}

\boitegrise{{\bf Important notation.} 
{\it 
Throughout this book, we fix a copy $Q$ \indexnot{Q}{Q} of the $P$-algebra $Z$, 
as well as an isomorphism of $P$-algebras $\copie : Z \longiso Q$. \indexnot{ca}{\copie}  
This means that $P$ will be seen as a $\kb$-subalgebra of both $Z$ and $Q$, but that 
$Z$ and $Q$ will be considered as different.\\ 
\hphantom{A} We then denote $\Kb=\Frac(P)$   and $\Lb=\Frac(Q)$  \indexnot{L}{\Lb} 
and we fix a {\bfit Galois closure} $\Mb$ \indexnot{M}{\Mb}  of the extension $\Lb/\Kb$. 
Set $G=\Gal(\Mb/\Kb)$ \indexnot{G}{G}  and $H=\Gal(\Mb/\Lb)$.  \indexnot{H}{H} 
We denote by $R$  \indexnot{R}{R}  the integral closure of $P$ in $\Mb$. We then have 
$P \subset Q \subset R$ and, by Corollary~\ref{coro:endo-bi}, 
$Q=R^H$ and $P=R^G$. This is the Galois context of Appendix~\ref{chapter:galois-rappels} 
which will be used extensively in this part.\\ 
\hphantom{A} Recall that $\Kb Z=\Kb \otimes_P Z$ is the fraction field of $Z$ (see~(\ref{eq:KZ-fraction})). 
We still denote by $\copie : \Frac(Z) \longiso \Lb$ the extension of $\copie$ to the fraction fields.\\
\hphantom{A} Let $Z^{\NM \times \NM}$, $Z^\NM$ and $Z^\ZM$ denote respectively the $(\NM \times \NM)$-grading, 
the $\NM$-grading, the $\ZM$-grading induced by the corresponding one of $\Hbt$ (see~\S\ref{section:graduation-1}, 
and the examples~\ref{N graduation-1} and~\ref{Z graduation-1}). Through the isomorphism 
$\copie$, we obtain gradings 
$Q^{\NM \times \NM}$, $Q^\NM$ and $Q^\ZM$ on $Q$.
}}{0.75\textwidth}

\bigskip
\medskip

This Galois extension is the main object studied in this book: 
we shall be particularly interested in the inertia groups of prime ideals of $R$, and their relation  
with the representation theory of $\Hb$. Throughout this part, we will use the results 
of the Appendices~\ref{chapter:galois-rappels} and~\ref{appendice graduation}, which deal 
with generalities about Galois theory, integral extensions and gradings.

\chapter{Galois theory}\label{chapter:galois-CM}

%
%
%

\section{Action of $G$ on the set $W$}

\medskip

Since $Q$ is a free $P$-module of rank $|W|$, the field extension $\Lb/\Kb$ has degree $|W|$:
\equat\label{degre lk}
[\Lb : \Kb] = |W|.
\endequat
Recall that the fact that $\Mb$ is a Galois closure of $\Lb/\Kb$ 
implies that
\equat\label{intersection}
\bigcap_{g \in G} \lexp{g}{H}=1.
\endequat
It follows from~(\ref{degre lk}) that
\equat\label{ghw}
|G/H|=|W|.
\endequat
This equality establishes a first link between the pair $(G,H)$ and the group $W$. 
We will now construct, using Galois theory, a bijection (depending on some choices) 
between $G/H$ and $W$.

\medskip

\subsection{Specialization}\label{subsection:specialization galois}
We fix here $c \in \CCB$. Recall that $\CG_c$ is the maximal ideal of $\kb[\CCB]$ 
whose elements are maps which vanish at $c$. We set
$$\pG_c=\CG_c P\qquad\text{and}\qquad \qG_c = \CG_c Q=\pG_c Q.\indexnot{pa}{\pG_c}\indexnot{qa}{\qG_c}
\indexnot{pa}{\pG_c}\indexnot{qa}{\qG_c}$$
Since $P_\bullet=P/\pG_c \simeq \kb[V]^W \otimes \kb[V^*]^W$ \indexnot{P}{P_\bullet}  and 
$Q_c=Q/\qG_c$ \indexnot{Q}{Q_c}  are domains (see Corollary~\ref{coro:endo-bi}(f)), 
we deduce that $\pG_c$ and $\qG_c$ are prime ideals of $P$ and $Q$ respectively. 
Fix a prime ideal $\rG_c$ \indexnot{ra}{\rG_c}  of $R$ lying over $\pG_c$ and let 
$R_c=R/\rG_c$. \indexnot{R}{R_c}  Now, let $D_c$ \indexnot{D}{D_c}  
(respectively $I_c$) \indexnot{I}{I_c}  be the decomposition (respectively inertia) group
$G_{\rG_c}^D$ (respectively $G_{\rG_c}^I$). 
Let
$$\Kb_c=\Frac(P_\bullet),\qquad \Lb_c=\Frac(Q_c)\qquad\text{and}\qquad \Mb_c=\Frac(R_c).
\indexnot{K}{\Kb_c}\indexnot{L}{\Lb_c}\indexnot{M}{\Mb_c}$$
In other words, $\Kb_c=k_P(\pG_c)$, $\Lb_c=k_Q(\qG_c)$ and 
$\Mb_c=k_R(\rG_c)$.

\bigskip

\begin{rema}\label{premier probleme} 
Here, the choice of the ideal $\rG_c$ is relevant. We will meet
such issues all along this book.\finl
\end{rema}

\bigskip

Since $\qG_c=\pG_c Q$ is a prime ideal, 
it follows from Proposition~\ref{reduction} that 
\equat\label{G=HD}
G=H \cdot D_c = D_c \cdot H.
\endequat
We also obtain that $Q$ is unramified in $P$ at  
$\qG_c$ (by definition). 
Theorem~\ref{raynaud} implies that $I_c \subset H$. 
Since $I_c$ is normal in $D_c$, we deduce from~(\ref{intersection}) and~(\ref{G=HD}) that  
$I_c \subset \bigcap_{d \in D_c} \lexp{d}{H} = \bigcap_{g \in G} \lexp{g}{H}=1$, so that 
\equat\label{eq:net}
I_c=1.
\endequat
It follows now from Proposition~\ref{cloture galoisienne} 
that 
\equat\label{cloture Lc}
\text{\it $\Mb_c$ is a Galois closure of the extension $\Lb_c/\Kb_c$.}
\endequat
Finally, by~(\ref{eq:net}) and Theorem~\ref{bourbaki}, we get 
\equat\label{gal Dc}
\Gal(\Mb_c/\Kb_c) = D_c\qquad\text{and}\qquad \Gal(\Mb_c/\Lb_c)= D_c \cap H.
\endequat
We denote by $\copie_c : Z_c \to Q_c$ \indexnot{ca}{\copie_c}  the specialization of $\copie$ at $c$ 
and we still denote by $\copie_c : \Frac(Z_c) \to \Lb_c$ the extension of $\copie_c$ 
to the fraction fields.

\bigskip

\begin{rema}\label{rema:impossible}
In \S \ref{subsection:specialization galois 0},
we will study the particular case where $c=0$, and obtain an explicit description 
of $D_0$. However, obtaining an explicit description of $D_c$ in general seems to be very difficult, 
as it will be shown by the examples treated in chapter~\ref{chapitre:rang 1}
(case $\dim_\kb(V)=1$), see \S\ref{rema:dc-cyclique}.\finl
\end{rema}

\bigskip

\subsection{Specialization at $0$}\label{subsection:specialization galois 0} 
Recall that  $P_0=P_\bullet = \kb[V]^W \otimes \kb[V^*]^W$ and  
$Q_0 \simeq Z_0=\Zrm(\Hb_0) \simeq \kb[V \times V^*]^{\D W}$, 
where $\D : W \to W \times W$, $w \mapsto (w,w)$ is the diagonal morphism.  \indexnot{dz}{\D}  
So, 
$$\Kb_0=\kb(V \times V^*)^{W \times W}\qquad\text{and}\qquad
\Lb_0 = \kb(V \times V^*)^{\D W},$$
On the other hand, the extension $\kb(V \times V^*)/\Kb_0$ is Galois with group 
$W \times W$, whereas the extension $\kb(V \times V^*)/\Lb_0$ is Galois 
with group $\D W$. Since $\D \Zrm(W)$ is the biggest normal subgroup of $W \times W$ contained in 
$\D W$, it follows from~(\ref{cloture Lc}) that 
$$\Mb_0 \simeq \kb(V \times V^*)^{\D \Zrm(W)}.$$

\medskip

\boitegrise{{\bf Fundamental choice.} 
 {\it 
We fix once and for all a prime ideal $\rG_0$ of $R$ 
lying over $\qG_0=\CG_0 Q$ as well as a field isomorphism 
$$\iso : \kb(V \times V^*)^{\D \Zrm(W)} \stackrel{\sim}{\longto} \Mb_0 \indexnot{ia}{\iso}$$
whose restriction to $\kb(V \times V^*)^{\D W}$ is the canonical isomorphism  
$\kb(V \times V^*)^{\D W} \longiso \Frac(Z_0) \longiso \Lb_0$. Here, 
the isomorphism $\Frac(Z_0) \longiso \Lb_0$ is $\copie_0$.\\
~\\
{\bf Convention.} 
The action of the group $W \times W$ on the field $\kb(V \times V^*)$ is as follows: 
$V \times V^*$ will be seen as a vector subspace of $\kb(V \times V^*)$ which generates this field, 
and the action of $(w_1,w_2)$ sends $(y,x) \in V \times V^*$ to $(w_1(y),w_2(x))$.
}}{0.75\textwidth}

\bigskip

\begin{rema}\label{rema:action-dif}
The action of $W \times W$ on $\kb(V \times V^*)$ described above 
is not the one obtained by first making $W \times W$ act on 
the variety $V \times V^*$ and then making it act on the function field 
$\kb(V \times V^*)$ by precomposition: one is deduced from the other 
thanks to the isomorphism $W \times W \longiso W \times W$, $(w_1,w_2) \mapsto (w_2,w_1)$. 
Nevertheless, this slight difference is important (see Remark~\ref{rema:w-w}).\finl
\end{rema}

\medskip

These choices being made, we get a canonical isomorphism
$\Gal(\Mb_0/\Kb_0) \longiso (W \times W)/\D \Zrm(W)$, which induces a canonical 
isomorphism $\Gal(\Mb_0/\Lb_0) \longiso \D W /\D \Zrm(W)$. 
Since $D_0 = \Gal(\Mb_0/\Kb_0)$ by~(\ref{gal Dc}),  
we obtain a group morphism 
$$\iota : W \times W \longto G\indexnot{iz}{\iota}$$
satisfying the following properties:

\bigskip

\begin{prop}\label{WW}
\begin{itemize}
\itemth{a} $\Ker \iota = \D \Zrm(W)$.

\itemth{b} $\im \iota = D_0$.

\itemth{c} $\iota^{-1}(H) = \D W$.
\end{itemize}
\end{prop}

\bigskip

Using now~(\ref{G=HD}), Proposition \ref{WW} provides a bijection 
\equat\label{bij W}
(W \times W)/\D W \stackrel{\sim}{\longleftrightarrow} G/H.
\endequat
Of course, one can build a bijection between $(W \times W)/\D W$ and $W$ using left or right projection. 
We fix a choice: 

\medskip

\boitegrise{{\bf Identification.} {\it The morphism $W \to W \times W$, $w \mapsto (w,1)$ 
composed with the morphism $\iota : W \times W \to G$ is injective, and we will identify $W$ 
with its image in $G$.}}{0.75\textwidth}

\medskip

More concretely, $w \in W \subset G$ is the unique 
automorphism of the $P$-algebra $R$ such that 
\equat\label{definition W}
\bigl(w(r) \mod \rG_0\bigr) = (w,1)
\bigl(r \mod \rG_0\bigr)\quad\text{in }\kb(V \times V^*)^{\D \Zrm(W)}
\endequat
for all $r \in R$. Hence, by~(\ref{bij W}),
\equat\label{G=HW}
G=H \cdot W=W \cdot H\qquad\text{and}\qquad H \cap W = 1.
\endequat

\bigskip

\begin{coro}\label{Dc W}
Given $c \in \CCB$, the natural map $D_c \to G/H \stackrel{\sim}{\to} W$ 
induces a bijection $D_c/(D_c \cap H) \stackrel{\sim}{\longto} W$.
\end{coro}

\begin{proof}
This follows from~(\ref{G=HD}) and~(\ref{G=HW}).
\end{proof}

\bigskip

\subsection{Action of $G$ on $W$}\label{subsection:action G} 
Let $\SG_W$ denote the permutation group of the set $W$. 
We identify the group $\SG_{W\setminus\{1\}}$ of permutations of the set 
$W \setminus \{1\}$ with the stabilizer of $1$ in $\SG_W$. 
The identification $G/H \stackrel{\sim}{\longleftrightarrow} W$ and the action 
of $G$ by left translations on $G/H$ identify 
$G$ with a subgroup of $\SG_W$.
Summarizing, we have
\equat\label{GHW}
G \subset \SG_W\qquad\text{and}\qquad H = G \cap \SG_{W \setminus \{1\}}.
\endequat
Given $g \in G$ and $w \in W$, we denote by $g(w)$ the unique element of $W$ 
such that $g\iota(w,1) H = \iota(g(w),1)H$. Through this identification of $G$ 
as a subgroup of $\SG_W$, the map $\iota : W \times W \to G$ is described as follows. Given
$(w_1,w_2) \in W \times W$ and $w \in W$, then 
\equat\label{iota concret}
\iota(w_1,w_2)(w)=w_1 w w_2^{-1}.
\endequat
This is the action of $W \times W$ on the set $W$ by left and right translation. 
Since $\D W$ is the stabilizer of $1 \in W$ 
for this action, we get
\equat\label{DW S}
\iota(\D W) = \iota(W \times W) \cap \SG_{W \setminus \{1\}}.
\endequat
This is of course compatible with Proposition~\ref{WW}(c) 
and~(\ref{GHW}).

Finally, the choice of the embedding of $W$ in $G$ through $w \mapsto \iota(w,1)$ 
amounts to identify $W$ with a subgroup of $\SG_W$ through the action on itself 
by left translation. 

\bigskip

\subsection{Euler element and Galois group}\label{section:galois euler}

\bigskip

Let $\eulerq=\copie(\euler) \in Q$. \indexnot{ea}{\eulerq}  

\begin{prop}
\label{pr:minimal-Euler}
The minimal polynomial of $\eulerq$ over $P$ has degree $|W|$.
Its specialization at $c$ is the minimal polynomial of $\eulerq_c$ over
$P_\bullet$.

We have $\Lb=\Kb[\eulerq]$.
\end{prop}

\begin{proof}
Since $\Hb_0=\kb[V \times V^*] \rtimes W$, we have 
$\Zrm(\Hb_0) =\kb[V \times V^*]^{\Delta W}$ and $P_\bullet = \kb[V/W \times V^*/W] \subset Z_0$. 
Moreover, it follows from Theorem~\ref{chevalley} 
that $Z_0$ is a free $P_\bullet$-module of rank $|W|$. On the other hand, 
$\euler_0=\sum_{i=1}^n x_i y_i$ (using the notation of \S\ref{section:eulertilde}).
It corresponds to $\id_V$ via the canonical isomorphism
$V\otimes V^*\longiso \End_{\kb}(V)$. Since $W$ acts faithfully on $V$, it
follows that the different elements of $W$ define different elements of
$\End_{\kb}(V)$. Consequently, the orbit of $\euler_0$ under the action of
$W\times W$ has $|W|$ elements. We deduce that 
the minimal polynomial of $\euler_0$ over $P_\bullet$ has degree $|W|$. As
a consequence,
the field $\kb(V \times V^*)^{\Delta W}$ is generated by $\euler_0$ over
$\kb(V/W \times V^*/W)$

Let $F_\euler(\tb)\in P[\tb]$ be the minimal polynomial of $\euler$ over $P$.
Since $Z$ is a free $P$-module of rank $|W|$ (Corollary~\ref{coro:endo-bi}),
we have $\deg F_\euler\le |W|$.
Since the specialization $\euler_0$ 
has a minimal polynomial over $P_0$ of degree $|W|$, it follows that
$\deg F_\euler=|W|$.

Denote by $\CG$ the prime ideal
of $\kb[\CCB]$ corresponding to the line $\kb c$. Let $F$ be the minimal polynomial
over $P\otimes_{\kb[\CCB]}\kb[\CCB]/\CG$ of the image of $\euler$ in the
integrally closed domain $Z\otimes_{\kb[\CCB]}\kb[\CCB]/\CG$
(Corollary \ref{coro:endo-bi}). 
We have $\deg F\le |W|$.
Since $\euler_0$ has a minimal polynomial of degree $|W|$, it follows that
$\deg F\ge |W|$, hence $\deg F=|W|$, so $F$ is the specialization of
$F_\euler$.

The Euler element is homogeneous for the $\BZ$-grading of Example \ref{N graduation-1}
(cf. Lemma \ref{lem:automorphismes-0}).
It follows from Lemma \ref{le:minpolquotient} that the specialization of $F$ (hence of $F_\euler$)
is the minimal polynomial of $\euler_c$ over $P_\bullet$.

The last assertion follows from (\ref{degre lk}).
\end{proof}

%
%
%
The computation of the Galois group $G=\Gal(\Mb/\Kb)$ is now
equivalent to the computation 
of the Galois group of the minimal polynomial of $\euler$ (or $\eulerq$). 
Classical methods (reduction modulo a prime ideal, see for instance \S~\ref{sec:calcul}) 
will be useful in small examples.

Let us come back to the computation of the embedding $W \longinjto G \subset \SG_W$. 
Given $w \in W$, let $\eulerq_w=w(\eulerq) \in \Mb$. \indexnot{ea}{\eulerq_w}  
Recall (see~(\ref{GHW})) that
if $g \in G$ and $w \in W$, then $g(w)$ is defined by the equality $g(w)H=gwH$. 
Since $H$ acts trivially on $\eulerq$, we deduce that 
\equat\label{action G W euler}
g(\eulerq_w) = \eulerq_{g(w)}
\endequat
and so, given $(w_1,w_2) \in W \times W$, we have
\equat\label{action WW euler}
\iota(w_1,w_2)(\eulerq_w)=\eulerq_{w_1ww_2^{-1}}.
\endequat
This extends the equality 
\equat\label{action W euler}
w_1(\eulerq_w)=\eulerq_{w_1w}
\endequat
which is an immediate consequence of the definition of $\eulerq_w$. 
In particular, by~(\ref{G=HW}), 
\equat\label{euler minimal}
\text{\it the minimal polynomial of $\eulerq$ over $P$ is } 
\prod_{w \in W} (\tb-\eulerq_w).
\endequat
Note also that, using~(\ref{definition W}) and the convention used for the action of 
$W \times W$ on $\kb(V \times V^*)$, we obtain
$$\iso^{-1}(\eulerq_w \mod \rG_0) = \sum_{i=1}^n w(y_i) x_i \in \kb[V \times V^*]^{\D \Zrm(W)}.$$
%

\bigskip

\begin{prop}\label{Q=Pe}
We have $Z=P[\euler]$ if and only if $W$ is generated by a single reflection.
\end{prop}

\begin{proof}
Assume $W$ is generated by a single reflection. An immediate argument allows
to reduce to the case where $\dim_\kb V=1$. 
So we assume now $n=\dim_\kb(V) = 1$ and let $d=|W|$. 
Let $y \in V \setminus\{0\}$ and $x \in V^*$ 
with $\langle y,x\rangle = 1$. Then $P_\bullet=\kb[x^d,y^d]$, 
$\euler_0=xy$ and it is easily checked that $Z_0=\kb[x^d,y^d,xy]$, 
that is, $Z_0=P_\bullet[\euler_0]$. We will prove here that
$$Z=P[\euler].$$
Indeed, $\Irr(W)=\{\e^i~|~0 \le i \le d-1\}$ and $f_{\e^i}(\tb)=\tb^i$ for $0 \le i \le d-1$. 
Consequently, (\ref{hilbert Q fantome}) implies that
$$\dim_\kb^{\BZ\times\BZ}(Z) = \frac{1+(\tb\ub)+\cdots + (\tb\ub)^{d-1}}{(1-\tb\ub)^{d-1}~
(1-\tb^d)~(1-\ub^d)}$$
whereas, since $P[\euler]=P \oplus P \euler \oplus \cdots \oplus P \euler^{d-1}$ 
by Proposition~\ref{pr:minimal-Euler}, we have 
$$\dim_\kb^{\BZ\times\BZ}(P[\euler])=\frac{1+(\tb\ub)+\cdots + (\tb\ub)^{d-1}}{(1-\tb\ub)^{d-1}~
(1-\tb^d)~(1-\ub^d)}.$$
Hence, $\dim_\kb^{\BZ\times\BZ}(P[\euler])=\dim_\kb^{\BZ\times\BZ}(Z)$, so 
$Z=P[\euler]$.

\medskip

Conversely, if $Z=P[\euler]$, then
$$Z=\mathop{\bigoplus}_{j=0}^{|W|-1} P \euler^j$$
since the minimal polynomial of $\euler$ over $P$ 
has degree $|W|$ (by Proposition \ref{pr:minimal-Euler}). 
We deduce, using~(\ref{hilbert P}), that 
$$\dim_\kb^{\BZ\times\BZ}(Z)=\frac{\DS{\sum_{j=0}^{|W|-1} 
(\tb\ub)^j}}{(1-\tb\ub)^{|\refw|} ~\DS{\prod_{i=1}^n (1-\tb^{d_i})(1-\ub^{d_i})}}.$$
It then follows from~(\ref{hilbert Q molien}) that
$$\frac{1}{|W|}\sum_{w \in W} \frac{(1-\tb)^n}{\det(1-w\tb)~\det(1-w^{-1}\ub)} = 
\frac{\DS{\sum_{j=0}^{|W|-1}(\tb\ub)^j}}{\DS{\prod_{i=1}^n
(1+\tb+\cdots+\tb^{d_i-1})(1-\ub^{d_i})}}.$$
By specializing $\tb \mapsto 1$ in this equality, 
the left-hand side contributes only when $w=1$. Since
$|W|=d_1\cdots d_n$ by Theorem~\ref{chevalley}(a), we obtain
$$\frac{1}{(1-\ub)^n} = \frac{\DS{\sum_{j=0}^{|W|-1}\ub^j}}{\DS{\prod_{i=1}^n
(1-\ub^{d_i})}}.$$
In other words,
$$\prod_{i=1}^n (1+\ub + \cdots + \ub^{d_i-1}) = \sum_{j=0}^{|W|-1}\ub^j.$$
By comparison of the degrees, we get 
$$|W|-1 = \sum_{i=1}^n (d_i-1).$$
But, again by Theorem~\ref{chevalley}(a), we have $|\REF(W)|=\sum_{i=1}^n (d_i-1)$, 
which shows that
$$\REF(W)=W\setminus\{1\}.$$
Therefore, if $w$, $w' \in W$, then $ww'w^{-1}w^{\prime -1}$ has determinant $1$, 
so it cannot be a reflection, so it must be equal to $1$. 
In other words $ww'=w'w$ and $W$ is abelian, hence diagonalizable.  
The proposition follows.
\end{proof}

\bigskip

\section{Splitting the algebra $\Kb\Hb$}\label{section:deploiement}

\medskip

Recall that Theorem~\ref{theo:KH-mat}
shows the existence of an isomorphism
$$\Kb\Hb \simeq \Mat_{|W|}(\Kb Z).$$ 
Recall also that $\Kb Z$ is the fraction field of $Z$ (see~(\ref{eq:KZ-fraction})) 
and that $\copie : \Kb Z \longiso \Lb$ denotes the extension of 
$\copie : Z \longiso Q$. 
The $\Kb$-algebra $\Kb\Hb$ is semisimple, but not $\Kb$-split in general.

Given $g \in G$, the morphism 
$\Kb Z \to \Mb$, $z \mapsto g(\copie(z))$ obtained by restriction of $g$ to $\Lb$ 
(through the isomorphism $\copie$) 
is $\Kb$-linear and it extends uniquely to a morphism of $\Mb$-algebras 
$$\fonction{g_Z}{\Mb \otimes_\Kb \Kb Z}{\Mb}{m \otimes_\Kb z}{mg(\copie(z)).}$$
Of course, $g_Z=(gh)_Z$ for all $h \in H$ and it is a classical fact 
(see the Proposition~\ref{iso galois}) that 
$$(g_Z)_{gH \in G/H} : \Mb \otimes_\Kb \Kb Z \longto \prod_{gH \in G/H} \Mb$$
is an isomorphism of $\Mb$-algebras. Taking~(\ref{G=HW}) into account, this can be 
rewritten as follows: there is an isomorphism of $\Mb$-algebras
\equat\label{W M}
\bijectio{\Mb \otimes_\Kb \Kb Z}{\prod_{w \in W} \Mb,}{x}{(w_Z(x))_{w \in W}.}
\endequat
So, the $\Mb$-algebra $\Mb \otimes_\Kb \Kb Z$ is semisimple and split, and
its simple representations are the $w_Z$, for $w\in W$.

Theorem~\ref{theo:KH-mat} provides a Morita equivalence between 
$\Mb \otimes_\Kb \Kb Z$ and $\Mb\Hb$. We will denote by $\LC_w$ 
\indexnot{L}{\LC_w}  the simple $\Mb\Hb$-module corresponding to $w_Z$.

\smallskip
Fix an ordered $\Kb Z$-basis $\BC$ of $\Kb\Hb e$ (recall that $|\BC|=|W|$). This choice 
provides an isomorphism of $\Kb$-algebras 
$$\r^\BC : \Kb\Hb \stackrel{\sim}{\longto} \Mat_{|W|}(\Kb Z).\indexnot{rz}{\r^\BC}$$
Now, given $w \in W$, let $\r_w^\BC$ \indexnot{rz}{\r_w^\BC} denote the morphism of $\Mb$-algebras 
$\Mb\Hb \longto \Mat_{|W|}(\Mb)$ defined by 
$$\r_w^\BC(m \otimes_P h) = m \cdot  w(\copie(\r^\BC(h)))$$
for all $m \in \Mb$ and $h \in \Hb$. Then $\r_w^\BC$ is an irreducible 
representation of $\Mb\Hb$ 
corresponding to the simple module $\LC_w$.

\smallskip
Let $\Irr(\Mb\Hb)$ denote the set of isomorphism classes of simple
$\Mb\Hb$-modules. We have a bijection 
\equat\label{bij irr}
\bijectio{W}{\Irr \Mb\Hb}{w}{\LC_w}
\endequat
and an isomorphism of $\Mb$-algebras 
\equat\label{iso MH}
\prod_{w \in W} \r_w^\BC : \Mb\Hb \stackrel{\sim}{\longto} 
\prod_{w \in W} \Mat_{|W|}(\Mb).
\endequat
In particular,
\equat\label{deployee}
\text{\it the $\Mb$-algebra $\Mb\Hb$ is split semisimple.}
\endequat
Moreover, the bijection~(\ref{bij irr}) allows us to identify 
its Grothendieck group $\groth(\Mb\Hb)$ with the $\BZ$-module 
$\BZ W$:
\equat\label{identification}
\groth(\Mb\Hb) \longiso \BZ W,\ [\LC_w]\mapsto w.
\endequat
Since the $\Mb$-algebra $\Mb\Hb$ is split semisimple, it follows from~\cite[Theorem~7.2.6 and Proposition 7.3.9]{geck} 
that there exists a unique family $(\schur_w)_{w \in W}$ \indexnot{sa}{\schur_w}  of elements 
of $R$ such that 
$$\taub_{\!\SSS{\Mb\Hb}} = \sum_{w \in W} \frac{\carac_w}{\schur_w},\indexnot{tx}{\taub_{\!\SSS{\Mb\Hb}}}$$
where $\carac_w : \Mb\Hb \to \Mb$ \indexnot{ca}{\carac_w}  denotes the character of the simple $\Mb\Hb$-module 
$\LC_w$ and $\taub_{\!{\SSS{\Mb\Hb}}} : \Mb\Hb \to \Mb$ denotes the extension of the symmetrizing form 
$\taub_\HHb : \Hb \to P$. The element $\schur_w$ of $R$ is 
called the {\it Schur element} associated with the simple module $\LC_w$. 
By \cite[Theorem 7.2.1]{geck}, $|W|\cdot\schur_w$ is equal to the scalar by which the Casimir element  
$\casimir_\Hb \in Z$ (defined in \S~\ref{sub:symetrisante}) acts on the simple module $\LC_w$. Therefore,
\equat\label{eq:formule-schur}
\schur_w = |W|^{-1}\cdot w(\copie(\casimir_\Hb)).
\endequat

\bigskip

\begin{rema}
In the general theory of symmetric algebras, 
the Schur element $\schur_w$ is an important invariant, which can be useful to determine the blocks 
of a reduction of $R\Hb$ modulo some prime ideal of $R$. Here, the formula~(\ref{eq:formule-schur}) 
shows that this computation is equivalent to the resolution of the following two problems:
\begin{itemize}
\itemth{1} Compute the Casimir element $\casimir_\Hb$.

\itemth{2} Understand the action of $W$ (or $G$) on the image of $\casimir_\Hb$ 
in $Q \subset R$. 
\end{itemize}
If Problem (1) seems doable (and its solution would be interesting as it 
would provide, after the Euler element, a new element of the center $Z$ of $\Hb$), 
it seems however more difficult to attack Problem~(2), as the computation of the ring $R$ 
(and even of the Galois group $G$) is for the moment out of reach.\finl
\end{rema}

\bigskip

\section{Grading on $R$}\label{section:graduation R}

\bigskip

\begin{prop}\label{bigraduation sur R}
There exists a unique $(\NM \times \NM)$-grading on $R$ extending the one of $Q$.
The Galois group $G$ stabilizes this $(\NM \times \NM)$-grading.
\end{prop}

\begin{proof}
The proposition is a consequence of Propositions~\ref{R gradue} and
\ref{prop:graduation-positive} and of
Proposition~\ref{unicite graduation} and Corollary~\ref{graduation et automorphisme}.
\end{proof}

\bigskip

Let $R=\bigoplus_{(i,j) \in \NM \times \NM} R^{\NM \times \NM}[i,j]$ denote the $(\NM \times \NM)$-grading 
extending the one of $Q$. 
Similarly, $R=\bigoplus_{i \in \NM} R^\NM[i]$ (respectively $R=\bigoplus_{i \in \BZ} R^\BZ[i]$) 
will denote the $\NM$-grading (respectively $\BZ$-grading) 
extending the one of $Q$: in other words, 
$$R^\NM[i]=\mathop{\bigoplus}_{i_1+i_2=i} R^{\NM \times \NM}[i_1,i_2] \quad\text{and}\quad 
R^\BZ[i]=\mathop{\bigoplus}_{i_2-i_1=i} 
R^{\NM \times \NM}[i_1,i_2].$$

\bigskip

Corollary~\ref{premier homogene} provides the following stability result.

\begin{coro}\label{ideaux homogenes}
The prime ideal $\rG_0$ of $R$ chosen in \S~\ref{subsection:specialization galois 0} 
is bi-homogeneous (in particular, it is homogeneous).
\end{coro}

\bigskip

%

\begin{coro}\label{r0}
We have $R^{\NM \times \NM}[0,0] = \kb$.
\end{coro}

\begin{proof}
By Corollary~\ref{ideaux homogenes}, we have $\rG_0 \subset R_+$. Consequently, 
$R^{\NM \times \NM}[0,0]$ is isomorphic to the homogeneous component of bidegree $(0,0)$ of $R/\rG_0$. 
Since $k_R(\rG_0) \simeq \kb(V \times V^*)^{\D\Zrm(W)}$ 
and $R/\rG_0$ is integral over $Q_0=\kb[V \times V^*]^{\D W}$, it follows that
$R/\rG_0 \subset \kb[V \times V^*]^{\D\Zrm(W)}$, and this inclusion 
preserves the bigrading, by the uniqueness of the bigrading 
on $R/\rG_0$ extending the one of $Q_0=\kb[V \times V^*]^{\D W}$ 
(Proposition~\ref{unicite graduation}). This shows the result.
\end{proof}

\bigskip
Denote by
$R_+$ \indexnot{R}{R_+}  the unique maximal bi-homogeneous ideal of $R$.

\begin{coro}\label{r0 DI}
Let $D_+$  \indexnot{D}{D_+}  (respectively $I_+$)\indexnot{I}{I_+} be the decomposition 
(respectively inertia) group of $R_+$ in $G$. Then $D_+=I_+=G$.
\end{coro}

\begin{proof}
Let $\pG_+=R_+ \cap P$. Then $k_R(R_+)/k_P(\pG_+)$ is a Galois extension 
with Galois group $D_+/I_+$ (see Theorem~\ref{bourbaki}).
By Corollary~\ref{r0}, $k_R(R_+)=\kb=k_P(\pG_+)$, so 
$D_+/I_+=1$. Note finally that
$D_+=G$ by Proposition~\ref{bigraduation sur R}.
\end{proof}

\begin{rema}
	We don't know if $G$ acts as a reflection group on $R_+/(R_+)^2$: this
	is the case when $\dim V=1$ (see \S\ref{chapitre:rang 1}). When this
	properties holds,
	the algebra $R$ is a complete intersection (Proposition \ref{intersection complete R}).
\end{rema}

\bigskip

\section{Action on $R$ of natural automorphisms of $\Hb$}
\label{section:auto galois}

\medskip

The previous Section~\ref{section:graduation R} was concerned with the extension to $R$ of the automorphisms 
of $Q$ induced by $\kb^\times \times \kb^\times$. 
In Section~\ref{section:automorphismes-1}, we have introduced an action of 
$W^\wedge \rtimes \NC$ on $\Hb$ which stabilizes $Z$ (of course), $P$, 
but also $\pG_0$ and so $\pG_0 Z$: this action can be transferred to $Q \simeq Z$ and still stabilizes 
$\qG_0=\pG_0 Q$. We will show how to extend this action to $R$, and we will derive some consequences 
about the Galois group. For this, we will work in a more general framework:

\bigskip

\boitegrise{{\bf Assumption.} {\it In this section \ref{section:auto galois},
we fix a group $\GC$ acting both on $Z$ and on 
$\kb[V \times V^*]$ and satisfying the following properties:
\begin{itemize}
\itemth{1} $\GC$ stabilizes $P$ and $\pG_0$.
\itemth{2} The action of $\GC$ on $\kb[V \times V^*]$ normalizes the action of $W \times W$ 
and the one of $\D W$.
\itemth{3} The canonical isomorphism of $\kb$-algebras $Z_0 \longiso \kb[V \times V^*]^{\D W}$ is $\GC$-equivariant.
\end{itemize}}\vskip-0.5cm
}{0.75\textwidth}

\bigskip

The action of $\GC$ on $Z$ induces, through the isomorphism $\copie$, an action of $\GC$ on $Q$. 
If $\t \in \GC$, we denote by $\t^\circ$ the automorphism of $\kb[V \times V^*]$ 
induced by $\t$: by (2), $\t^\circ$ stabilizes $\kb[V \times V^*]^{\D\Zrm(W)}$, 
$\kb[V \times V^*]^{\D W}$ and $\kb[V \times V^*]^{W \times W}$. 

%
%
%
%
%
%
%
 \bigskip

\begin{prop}\label{extension tau}
If $\t \in \GC$, then there exists a unique extension $\taut$ of $\t$ to $R$ 
satisfying the following two properties:
\begin{itemize}
\itemth{1} $\taut(\rG_0)=\rG_0$.

\itemth{2} The automorphism of $R/\rG_0$ induced by $\taut$ is equal to $\t^\circ$, 
via the identification $\iso : \kb(V \times V^*)^{\D \Zrm(W)} \longiso \Mb_0$ of 
\S\ref{subsection:specialization galois 0}. 
\end{itemize}
\end{prop}

\begin{proof}
Let us start by showing the existence. First of all, $\Mb$ being a Galois closure 
of the extension $\Lb/\Kb$, there exists an extension $\tau_\Mb$ of $\tau$ to $\Mb$. Since $R$ is 
the integral closure of $Q$ in $\Mb$, it follows that
$\t_\Mb$ stabilizes $R$. Moreover, since 
$\t(\qG_0)=\qG_0$, there exists $h \in H$ such that $\t_\Mb(\rG_0)=h(\rG_0)$. 
Let $\taut_\Mb=h^{-1} \circ \t_\Mb$. Then
$$\taut_\Mb(\rG_0)=\rG_0\qquad\text{and}\qquad (\taut_\Mb)|_{\SSS{\Lb}} =\t.$$
Let $\taut_{\Mb,0}$ denote the automorphism of $R/\rG_0$ induced by $\taut_\Mb$. 

By construction, the restriction of $\taut_{\Mb,0}$ to $Q/\qG_0$ is equal to 
the restriction of $\iso \circ \t^\circ \circ \iso^{-1}$. 
Hence, there exists $d \in D_0 \cap H$ such that  
$\taut_{\Mb,0} = d \circ (\iso \circ \t_0 \circ \iso^{-1})$. 
We then set $\taut=d^{-1} \circ \taut_\Mb$: it is clear that $\taut$ 
satisfies (1) and (2).

\medskip

Let us now show the uniqueness. Let $\taut_1$ be another extension of $\t$ 
to $R$ satisfying (1) and (2) and let $\s=\taut^{-1}\taut_1$. 
We have $\s \in G$ and, by (1), $\s$ stabilizes $\rG_0$, hence
$\s \in D_0$. Moreover, by (2), $\s$ induces 
the identity on $R/\rG_0$, hence
$\s\in I_0=1$ (cf. (\ref{eq:net})), and so $\taut=\taut_1$.
\end{proof}

\bigskip

The existence and the uniqueness statements of
Proposition~\ref{extension tau} have the following consequences.

\bigskip

\begin{coro}\label{coro:extension-action}
The action of $\GC$ on $Q$ extends uniquely to an action of $\GC$ on $R$, which stabilizes 
$\rG_0$ and is compatible with the isomorphism $\iso$.
\end{coro}

\bigskip

In this book, we will denote again by $\t$ the extension $\taut$ of $\t$ defined in 
Proposition~\ref{extension tau}. 
Since $\GC$ stabilizes $P$, $Q$, $\pG_0$, $\qG_0=\pG_0 Q$ and $\rG_0$, we
deduce the following.

\bigskip

\begin{coro}\label{coro:action-stabilise}
The action of $\GC$ on $R$ normalizes $G$, $H$, $D_0=\iota(W \times W)$ and $D_0 \cap H=\iota(\D W)=W/\Zrm(W)$.
\end{coro}

\bigskip

From Corollary~\ref{coro:action-stabilise}, we deduce that $\GC$ acts on the set 
$G/H \simeq W$ and that 
\equat\label{eq:action-normalise}
\text{\it the image of $\GC$ in $\SG_W$ normalizes $G$.}
\endequat

\begin{exemple}\label{ex:action-k-k-h-n}
The group $\GC=\kb^\times \times \kb^\times \times (W^\wedge \rtimes \NC)$ 
acts on $\Hb$ and stabilizes $P$ and $\pG_0$; by the same formulas, it acts on 
$\kb[V \times V^*]$ and normalizes $W \times W$ and $\D W$ (in fact, 
$\kb^\times \times \kb^\times \times \Hom(W ,\kb^\times)$ commutes with $W \times W$ and 
only $\NC$ acts non-trivially on $W \times W$). 

It follows from the previous results that the action of 
$\kb^\times \times \kb^\times \times  (W^\wedge \rtimes \NC)$ on $Q$ 
extends uniquely to an action on $R$ which stabilizes $\rG_0$ and 
is compatible with the isomorphism $\iso$. By the uniqueness statement, the extension of the action of  
$\kb^\times \times \kb^\times \times W^\wedge$ to $R$ commutes with the action of $G$ 
whereas the one of $\NC$ is such that the morphism $G \longinjto \SG_W$ is $\NC$-equivariant.

Finally, still by the uniqueness statement, the extension of the action of the subgroup 
$\kb^\times \times \kb^\times$ corresponds to the extension to $R$ of the $(\NM \times \NM)$-grading described 
in Proposition~\ref{bigraduation sur R}.\finl
\end{exemple}

\section{A particular situation: reflections of order 2}\label{sec:w0}

\medskip

\boitegrise{\noindent{\bf Assumption and notation.} 
{\it In this section \ref{sec:w0}, we assume that 
all the reflections of $W$ have order $2$ and that $-\Id_V \in W$. 
We set $w_0=-\Id_V$ and \indexnot{tx}{\tau_0}  
$\t_0=(-1,1,\e) \in \kb^\times \times \kb^\times \times W^\wedge$.}}{0.75\textwidth}

\bigskip

By construction, the restriction of $\t_0$ to $\kb[\CCB]$ is equal to 
the identity. Since $-\Id_V \in W$, the restriction of $\t_0$ to $\kb[V]^W$ 
is equal to the identity. Similarly, the restriction of $\t_0$ to $\kb[V^*]^W$ 
is also equal to the identity. Consequently,
\equat\label{tau 0}
\forall~p \in P,~\t_0(p)=p.
\endequat
Recall that $\t_0$ denotes also the automorphism of $R$ 
defined by Proposition~\ref{extension tau}. By definition of the Galois group, 
we have $\t_0 \in G$. More precisely, we have the following description.

\bigskip

\begin{prop}\label{tau 0 in G}
Assume that all the reflections of $W$ have order $2$ and that $w_0=-\Id_V \in W$. 
Then $\t_0$ is a central element of $G$. Its action on $W$ is given by 
$\t_0(w)=w_0w$ (which means that $\t_0=w_0=\iota(w_0,1)$, through the canonical embedding 
$W \longinjto G$) and, through the embedding $G \longinjto \SG_W$, we have 
$$G \subset \{\s \in \SG_W~|~\forall~w \in W,~\s(w_0w)=w_0\s(w)\}.$$
Moreover, given $w \in W$, we have
$$\t_0(\eulerq_w)=-\eulerq_w=\eulerq_{w_0w}.$$
\end{prop}

\begin{proof}
By Lemma~\ref{lem:automorphismes-1}(c), we have $\t_0(\eulerq)=-\eulerq$. 
Moreover, by Example~\ref{ex:action-k-k-h-n}, the action of $\t_0$ on $R$ commutes 
with the action of $G$. Therefore, if $w \in W$, then $\t_0(\eulerq_w)=-\eulerq_w$. 

On the other hand, there exists $w_1 \in W$ such that $\t_0(\eulerq)=\eulerq_{w_1}$. 
As $-\eulerq_0 = w_0(\eulerq_0)$, it follows from the characterization of the action of 
$W$ on $\Lb$ that $\t_0(\eulerq)=\eulerq_{w_0}=-\eulerq$. Since $w_0$ is central in 
$W$, we have $w_0(\eulerq_w)=\eulerq_{w_0w}=\eulerq_{ww_0}=w(\eulerq_{w_0})=-\eulerq_w$. 
So $\t_0=w_0$ because $\Mb=\Kb[(\eulerq_w)_{w \in W}]$. 

Now, the fact that $G \subset \{\s \in \SG_W~|~\forall~w \in W,~\s(w_0w)=w_0\s(w)\}$ 
follows from the fact that $\t_0=w_0$ commutes with the action of $G$.
\end{proof}

\bigskip

Note that $w_0w=-w$ and so the inclusion of Proposition~\ref{tau 0 in G} can be rewritten
\equat\label{inclusion w0}
G \subset \{\s \in \SG_W~|~\forall~w \in W,~\s(-w)=-\s(w)\}.
\endequat
Viewed like this, it shows that, under the assumption of this section, 
$G$ is contained in a Weyl group of type $B_{|W|/2}$.

\bigskip

\section{Special features of Coxeter groups}
\label{se:specGalois}
\cbstart

\boitegrise{{\bf Assumption.} 
{\it In this section \ref{se:specGalois}, we assume that 
$W$ is a Coxeter group, and we use the notation of \S\ref{se:Coxetergroups}.}}{0.75\textwidth}

\bigskip

Recall (Proposition \ref{prop:auto-coxeter-0}) that 
the algebra $\Hb$ admits another automorphism $\s_\Hb$
stabilizing $P$.
\bigskip

\begin{prop}\label{coro:sigma-inverse}
The automorphism $\s_\Hb$ of $Q$ extends uniquely to an automorphism 
$\sigma_\Hb$ of $R$.
Given $g \in G \subset \SG_W$ and $w \in W$, we have
$(\lexp{\s_\Hb}{g})(w)=g(w^{-1})^{-1}$. 
\end{prop}

\begin{proof}
Note that $\s_\Hb$ induces an automorphism of $\kb[V \times V^*]$ 
which normalizes $W \times W$ and $\D W$. More precisely, consider
$\s_2 : V \oplus V^* \longiso V \oplus V^*$, $(y,x) \longmapsto (-\s^{-1}(x),\s(y))$ 
We have
\equat\label{eq:sigma-ww}
\s_2 (w,w') \s_2^{-1} = (w',w)
\endequat
for all $(w,w') \in W \times W$. By Proposition~\ref{extension tau}, 
$\s_\Hb$ extends uniquely to an automorphism of $R$
which stabilizes $\rG_0$ and which is compatible with $\iso$. 
Since $\s_\Hb$ normalizes $G$ and its subgroup $\iota(W \times W)$ 
(see~(\ref{eq:action-normalise})), it follows from (\ref{eq:sigma-ww}) that
its action on the elements of $W \subset G$ satisfies 
\equat\label{eq:sigma-w}
\lexp{\s_\Hb}{w} H = w^{-1} H\qquad\text{and}\qquad H\lexp{\s_\Hb}{w} = H w^{-1} 
\endequat
for all $w \in W$. The proposition follows now from~(\ref{eq:action-normalise}).
\end{proof}

\begin{rema}
Note that if $W{\not=1}$, then the action of $\Gb\Lb_2(\kb)$ on $\Hb$ does not 
induce an action on $R$, since $\Gb\Lb_2(\kb)$ does not normalize $W\times W$.\finl
\end{rema}

\cbend

\section{Geometry}
\label{se:geometrie-CM}

\subsection{Extension $R/P$}\label{section:geometrie galois}

\medskip

Since $R$ and $Q \simeq Z$ are also $\kb$-algebras of finite type, they are associated with 
$\kb$-varieties $\RCB$ \indexnot{R}{\RCB}  
and \indexnot{Q}{\QCB}  $\QCB \simeq \ZCB$: the isomorphism $\copie^* : \QCB \longiso \ZCB$ 
is induced by $\copie : Z \longiso Q$. 
The inclusion $P \longinjto R$ (respectively $Q \longinjto R$) 
defines a morphism of varieties $\r_G : \RCB \to \PCB$ \indexnot{rz}{\r_G}  
(respectively $\r_H : \RCB \to \QCB$)\indexnot{rz}{\r_H}   and the equalities $P=R^G$ and $Q=R^H$ 
show that $\r_G$ and $\r_H$ induce isomorphisms
\equat\label{R/G}
\RCB/G \longiso \PCB \qquad\text{and}\qquad \RCB/H \longiso \QCB.
\endequat
In this setting, the choice of a prime ideal $\rG_c$ lying over $\qG_c$ 
is equivalent to the choice of an irreducible component 
$\RCB_c$ \indexnot{R}{\RCB_c}  of $\r_H^{-1}(\QCB_c)$ 
(whose ideal of definition is $\rG_c$). Similarly, 
the argument leading to Proposition~\ref{WW} implies for instance that 
the number of irreducible components of $\r_G^{-1}(\QCB_0)$ is equal to 
$|G|\cdot|\D\Zrm(W)|/|W|^2$. It also shows that $\iota(W \times W)$ is the stabilizer 
of $\RCB_0$ in $G$ and $\RCB_0/\iota(W \times W) \simeq \PCB_0$, 
that $\iota(\D W)$ is the stabilizer of $\RCB_0$ in $H$ and that 
$\RCB_0/\iota(\D W) \simeq \QCB_0$. We have a commutative diagram 
\equat\label{diagramme geometrie bis}
\diagram
&\RCB_c \ar@{^{(}->}[rr] \ddto && \RCB \ddto^{\DS{\r_H}}\ar@/^4pc/[dddd]^{\DS{\r_G}}& \\
&&&&\\
&\QCB_c \ar@{^{(}->}[rr]^{\DS{j_c}} \ddto_{\DS{\Upsilon_c}} && \QCB \ddto^{\DS{\Upsilon}}& \\
&&&&\\
V/W \times V^*/W \ar@{=}[r] & \PCB_{\!\!\!\bullet} 
\ar@{^{(}->}[rr]^{\DS{i_c}} \ddto && \PCB \ddto^{\DS{\pi}} \ar@{=}[r]& 
\CCB \times V/W \times V^*/W \\
&&\\ 
&\{c\} \ar@{^{(}->}[rr] && \CCB&
\enddiagram
\endequat
which completes the diagram~(\ref{diagramme geometrie}) (if we identify 
$\QCB$ and $\ZCB$ through $\copie^*$). 
Only the two bottom squares of the diagram~(\ref{diagramme geometrie bis})
are cartesian.

\bigskip

\subsection{Automorphisms}\label{subsection:automorphismes} 
The group $\kb^\times \times \kb^\times \times \bigl(W^\wedge \rtimes \NC\bigr)$ 
(which acts on $\Hb$ through automorphisms of $\kb$-algebras)  
stabilizes the $\kb$-subalgebras $\kb[\CCB]$, $P$ and $Q$ of $\Hb$. 
Therefore, it acts by automorphisms of $\kb$-varieties on 
$\CCB$, $\PCB$ and $\QCB$.
Also, this action extends to an action on $\RCB$
(see Corollary~\ref{coro:extension-action}), making the morphisms
$\Upsilon$, $\pi$, $\r_H$ and $\r_G$ of 
diagram~(\ref{diagramme geometrie bis}) equivariant for that action.

\bigskip

\subsection{Irreducible components of $\RCB \times_\PCB \ZCB$} 
Given $w \in W$, we set 
$$\RCB_w = \{(r,\copie^*(\r_H(w(r))))~|~r \in \RCB\} \subset \RCB \times_\PCB \ZCB.\indexnot{R}{\RCB_w}$$

\bigskip

\begin{lem}\label{lem:compo-irr-r}
If $w \in W$, then $\RCB_w$ is an irreducible component of 
$\RCB \times_\PCB \ZCB$, isomorphic to $\RCB$. Moreover,
$$\RCB \times_\PCB \ZCB = \bigcup_{w \in W} \RCB_w$$
and $\RCB_w=\RCB_{w'}$ if and only if $w=w'$.
\end{lem}

\begin{proof}
This is only the geometric translation of the fact that the morphism 
$$\fonctio{R \otimes_P Z}{\prod_{w \in W} R}{x}{(w_Z(x))_{w \in W}}$$
defined by restriction from the morphism~(\ref{W M}) is finite 
and becomes an isomorphism after extending scalars to $\Kb$.
\end{proof}

\chapter{Calogero-Moser cells}\label{chapter:cellules-CM}

\boitegrise{{\bf Notation.} 
{\it From now on, and until the end of Chapter~\ref{chapter:cellules-CM}, 
we fix a prime ideal $\rG$ of $R$ and we set 
$\qG=\rG \cap Q$ and $\pG = \rG \cap P$. \indexnot{pa}{\pG,~\qG,~\rG}  We denote by 
$D_\rG$ (respectively $I_\rG$) the decomposition (respectively inertia) group 
of $\rG$ in $G$.}}{0.75\textwidth}

\bigskip

\section{Definition, first properties}\label{section:definition cellules}

\medskip

Recall that, since we have chosen once and for all a prime ideal $\rG_0$ 
as well as an isomorphism $k_R(\rG_0) \longiso \kb(V \times V^*)^{\D \Zrm(W)}$, 
we can identify the sets $G/H$ and $W$ (see~\S\ref{subsection:specialization galois 0}). 
So $G$ acts on the {\it set} $W$.

\bigskip

\begin{defi}\label{defi:CM}
A {\bfit Calogero-Moser $\rG$-cell} 
is an orbit of the inertia group $I_\rG$ in the {\bfit set} $W$. We will denote by 
$\sim^\calo_\rG$ \indexnot{ZZZ}{\sim_\rG^\calo}  the equivalence relation corresponding 
to the partition of $W$ into Calogero-Moser $\rG$-cells. 

The set of Calogero-Moser $\rG$-cells will be denoted by $\cmcellules_\rG(W)$. 
\indexnot{C}{{{{^\calo{\mathrm{Cells}}_\rG(W)}}}}
\end{defi}

\bigskip

Recall that $W$ can be identified with the set $\Hom_{P-\alg}(Q,R) = \Hom_{\Kb-\alg}(\Lb,\Mb)$. 
By Proposition~\ref{reduction}, 
if $w$ and $w'$ are two elements of $W$, then 
\equat\label{w sim}
\text{\it $w \sim^\calo_\rG w'$ if and only if $w(q) \equiv w'(q) \mod \rG$ 
for all $q \in Q$.}
\endequat

\bigskip

\begin{rema}\label{semicontinuite}
If $\rG$ and $\rG'$ are two prime ideals of $R$ such that $\rG \subset \rG'$, 
then $I_\rG \subset I_{\rG'}$ and so the Calogero-Moser $\rG'$-cells are unions 
of Calogero-Moser $\rG$-cells.\finl
\end{rema}

\begin{exemple}[Reflections of order 2]\label{cellules w0}
If all the reflections of $W$ have order $2$ and if $w_0=-\Id_V \in W$, 
then it follows from Proposition~\ref{tau 0 in G} that 
$G \subset \{\s \in \SG_W~|~\forall~w \in W,~\s(w_0w)=w_0\s(w)\}$. 
Consequently, if $\G$ is a Calogero-Moser $\rG$-cell, 
then $w_0\G=\G w_0$ is a Calogero-Moser $\rG$-cell.\finl 
\end{exemple}

\bigskip

The action of $G$ being compatible with the bigrading of $R$, 
the following result is not surprising.

\bigskip

\begin{prop}\label{prop:cellules-homogeneise}
Let $\Gamma$ be a finitely generated free abelian group and
$R =\bigoplus_{\gamma\in\Gamma} R_\gamma$ a $G$-stable $\Gamma$-grading on $R$.
Let $\rGt=\bigoplus_{\gamma\in\Gamma} \rG \cap R_\gamma$. 
Then $I_\rG=I_\rGt$, hence
the Calogero-Moser $\rGt$-cells and the Calogero-Moser $\rG$-cells 
coincide.
\end{prop}

\begin{proof}
This follows from Corollary~\ref{coro:homogeneise-inertie}.
\end{proof}

%

\bigskip

\section{Blocks}\label{subsection:cellules et blocs}

\medskip

Given $w \in W$, we denote by $e_w \in \blocs(\Zrm(\Mb\Hb))$ \indexnot{ea}{e_w}  the central primitive idempotent of 
$\Mb\Hb$ (which is split semisimple by~(\ref{deployee})) 
associated with the simple module $\LC_w$. It is the unique central primitive idempotent 
of $\Mb\Hb$ which acts as the identity on the simple $\Mb\Hb$-module $\LC_w$.
Given
$b \in \blocs(\Zrm(R_\rG\Hb))$, we denote by $\calo_\rG(b)$ \indexnot{C}{\calo_\rG(b)} the 
unique subset of $W$ such that\equat\label{calo idempotent}
b=\sum_{w \in \calo_\rG(b)} e_w.
\endequat
In other words, the bijection $W \stackrel{\sim}{\longleftrightarrow} \Irr \Mb\Hb$ 
restricts to a bijection 
$\calo_\rG(b) \stackrel{\sim}{\longleftrightarrow} \Irr \Mb\Hb b$. 
It is clear that $(\calo_\rG(b))_{b \in \blocs(R_\rG \Hb)}$ is a partition 
of $W$. In fact, this partition coincides with the partition into 
Calogero-Moser $\rG$-cells.

\bigskip

\begin{theo}\label{theo:calogero}
Let $w$, $w' \in W$ and let $b$ and $b'$ be the central primitive idempotents 
of $R_\rG\Hb$ such that $w \in \calo_\rG(b)$ and $w' \in \calo_\rG(b')$. 
Then $w \sim_\calo^\rG w'$ if and only if $b=b'$.
\end{theo}

\begin{proof}
Let $\o_w : \Zrm(R\Hb)=R \otimes_P Z \longto R$ \indexnot{ozz}{\o_w}  denote the central character 
associated with the simple $\Mb\Hb$-module $\LC_w$ (see \S~\ref{section centrale}). 
By the very definition of $\LC_w$, we have 
$$\o_w(r \otimes_P z)= r w(\copie(z))$$
for all $r \in R$ and $z \in Z$. Consequently, by~(\ref{w sim}), we have 
$w \sim_\calo^\rG w'$ if and only if $\o_w \equiv \o_{w'} \mod \rG$. 
The result follows now from Corollary~\ref{coro:r-blocs}.
\end{proof}

\bigskip

Via Proposition \ref{relevement idempotent}, we obtain bijections 
\equat\label{bij calogero}
\xymatrix{
\cmcellules_\rG(W) && \blocs(R_\rG Z) \ar[ll]_-\sim^-{\calo_\rG(b)\mapsfrom b } 
\ar[rr]^-\sim_-{b\mapsto\bba} &&
\blocs(k_R(\rG) Z)
}
\endequat
where $\bba$ denote the image of $b$ in $k_R(\rG)Z$. 

Since $\Mb Z$ is the center of $\Mb\Hb$, the fact that $\Mb\Hb$ is split semisimple 
implies immediately that
\equat\label{dim center calogero}
\dim_\Mb(\Mb Z b)=|\calo_\rG(b)|.
\endequat
Recall that, since $Z$ is a direct summand of $\Hb$, the algebra
$k_R(\rG)Z$ can be identified 
with its image in $k_R(\rG) \Hb$. Note however that
this image might be different 
from the center of $k_R(\rG)\Hb$. 

\bigskip

\begin{coro}\label{dim center r}
We have
$\dim k_R(\rG) Z \bba=|\calo_\rG(b)|$.
\end{coro}

\begin{proof}
The $R_\rG$-module $R_\rG Z$ is free (of rank $|W|$), so the 
$R_\rG$-module $R_\rG Z b$ is projective, hence free since $R_\rG$ is local. 
By~(\ref{dim center calogero}), the $R_\rG$-rank of $R_\rG Z b$ 
is $|\calo_\rG(b)|$. The corollary follows.
\end{proof}

\bigskip

\begin{exemple}[Specialization]\label{specialization} 
Let $c \in \CCB$. Let $\rG_c$ be a prime ideal of $R$ lying over 
$\pG_c$ and, as in~\S\ref{subsection:specialization galois}, define
$D_c=G_{\rG_c}^D$ and $I_c=G_{\rG_c}^I$. 
Then $I_c=1$ by~(\ref{eq:net}), hence
$$\text{\it the Calogero-Moser $\rG_c$-cells are singletons.\finl}$$
\end{exemple}

\bigskip

\section{Ramification locus}

\medskip

Let $\rG_\ramif$ \indexnot{ra}{\rG_\ramif}  denote the defining ideal 
of the ramification locus of the finite morphism $\Spec(R) \to \Spec(P)$: in other words, 
$R$ is \'etale over $P$ at $\rG$ if and only if $\rG_\ramif \not\subset \rG$. 
Recall~\cite[Expos\'e~V, Corollaire~2.4]{sga} 
that the following assertions are equivalent:
\begin{itemize}
\item $R$ is \'etale over $P$ at $\rG$;
\item $I_\rG=1$;
\item
$R$ is unramified over $P$ at $\rG$.
\end{itemize}
As $G$ acts faithfully on $W$, 
we deduce the following result
(taking into account Theorem~\ref{theo:calogero}).

\bigskip

\begin{prop}\label{ramification inertie}
The following are equivalent:
\begin{itemize}
\itemth{1} $I_\rG \neq 1$.

\itemth{2} $R$ is not \'etale over $P$ at $\rG$.

\itemth{3} $R$ is ramified over $P$ at $\rG$.

\itemth{4} $\rG_\ramif \subset \rG$.

\itemth{5} $|\blocs(R_\rG Q)| < |W|$.
\end{itemize}
\end{prop}

\bigskip

Note that $\rG_\ramif$ is not necessarily a prime ideal of $R$. However, 
the purity theorem~\cite[Expos\'e~X,~Th\'eor\`eme~3.1]{sga} implies that 
$\Spec(R/\rG_\ramif)$ is empty or of pure codimension $1$ in $\Spec(R)$ 
(since $R$ is integrally closed and $P$ is regular). 
By Corollary~\ref{r0 DI} and Proposition~\ref{ramification inertie}, 
the morphism $\Spec(R) \to \Spec(P)$ is not \'etale if $W{\not=}1$. Hence,
if $W{\not=}1$, we deduce that
\equat\label{codim 1 ramification}
\text{\it $\Spec(R/\rG_\ramif)$ is of pure codimension $1$ in $\Spec(R)$.}
\endequat
Of course, 
\equat\label{ramification stable}
\text{\it $\rG_\ramif$ is stable under the action of 
$\kb^\times \times\kb^\times \times \left((W^\wedge \times G) \rtimes \NC\right)$.}
\endequat
While it is difficult to determine the ideal $\rG_\ramif$ (we even do not know how to determine 
the ring $R$), the ideal $\pG_\ramif=\rG_\ramif \cap P$ is determined by
the extension $Q/P$. 
\indexnot{pa}{\pG_\ramif}  The following result is classical \cite[Proposition 4.10]{sga}.

\bigskip

\begin{lem}\label{discriminant ramification}
Let $\disc(Q/P)$ \indexnot{da}{\disc(D/P)}  denote the discriminant ideal of $Q$ in $P$. Then 
$\pG_\ramif=\sqrt{\disc(Q/P)}$. 
\end{lem}
%

\bigskip

\begin{rema}\label{autre preuve ramification}
We proved that $\Spec(R/\rG_\ramif)$ is of pure codimension $1$ in $\Spec(R)$ 
by using the purity theorem. Using the equivalence between (4) and (5) in 
Lemma~\ref{ramification inertie}, we obtain another proof using
Proposition~\ref{codimension un}.\finl
\end{rema}

\bigskip

%
%
%
%
%
%
%
%

\section{Smoothness}\label{section:cellules et lissite}

\medskip
Extending scalars, we have the following consequence of Corollary \ref{morita lissite p}.

\bigskip

\begin{coro}\label{morita lissite r}
Assume $\zG_\singulier \cap P \not\subset \pG$.

The $(k_R(\rG)\Hb,k_R(\rG)Z)$-bimodule $k_R(\rG)\Hb e$ is both 
left and right projective and induces 
a Morita equivalence between $k_R(\rG)\Hb$ and $k_R(\rG)Z$. 
\end{coro}

\bigskip

\begin{theo}\label{lissite et simples}
Assume $\zG_\singulier \cap P \not\subset \pG$.
	
The $k_R(\rG)$-algebra $k_R(\rG)\Hb$ is split. Every block of $k_R(\rG)\Hb$ 
admits a unique simple module, which has dimension $|W|$.
In particular, the simple $k_R(\rG)\Hb$-modules are parametrized 
by the Calogero-Moser $\rG$-cells, that is, by the $I_\rG$-orbits in $W$.
\end{theo}

\begin{proof}
Let us first show that $k_R(\rG)Z = k_R(\rG) \otimes_P Z = k_R(\rG) \otimes_{P_\pG} Z_\pG$ 
is a split $k_R(\rG)$-algebra. Let $\zG_1$,\dots, $\zG_l$ 
be the prime ideals of $Z$ lying over $\pG$: in other words, 
$k_P(\pG)\zG_1$,\dots, $k_P(\pG)\zG_l$ are the prime (so, maximal) ideals 
of $k_P(\pG)Z = Z_\pG/\pG Z_\pG$. Then $k_P(\pG)(\zG_1 \cap \dots \cap \zG_l)$ 
is the radical $I$ of $k_P(\pG)Z$. Moreover, 
$$(k_P(\pG) Z)/I \simeq k_Z(\zG_1) \times \cdots \times k_Z(\zG_l).$$
Since $k_R(\rG)$ is a Galois extension of $k_P(\pG)$
containing 
(the image through $\copie^{-1}$ of)
$k_Z(\zG_i)$, for all $i$,
it follows that $k_R(\rG) \otimes_{k_P(\pG)} k_Z(\zG_i)$ 
is a split $k_R(\rG)$-algebra (see Proposition~\ref{iso galois}). 
As a consequence, $k_R(\rG)Z$ is split. We deduce from
Corollary \ref{morita lissite r} that $k_R(\rG)\Hb$ is also split.

On the other hand, since $k_R(\rG)Z$ is commutative, every block of $k_R(\rG)Z$ 
admits a unique simple module. Using again the Morita equivalence, the same property holds 
for $k_R(\rG)\Hb$. Finally, as the projective $Z_\pG$-module $\Hb_\pG e$ has rank 
$|W|$, the same is true for the projective $k_R(\rG)Z$-module $k_R(\rG)\Hb e$, and so 
the simple $k_R(\rG)\Hb$-modules have dimension $|W|$.

The last statement of the theorem is now clear.
\end{proof}

\bigskip

\section{Geometry}
\label{se:geometryCMcells}

\medskip

By Lemma~\ref{lem:compo-irr-r}, the irreducible components of 
$\RCB \times_\PCB \ZCB$ are the
$$\RCB_w=\{(r,\copie^*(\r_H(w(r))))~|~r \in \RCB\},$$
where $w$ runs over $W$ and the morphism $\Upsilon_\RCB : \RCB \times_\PCB \ZCB \to \RCB$ 
obtained from $\Upsilon : \ZCB \to \PCB$ by base change induces an 
isomorphism between the irreducible component $\RCB_w$ and $\RCB$. 

Consequently, the inverse image through $\Upsilon_\RCB$ of the closed irreducible subvariety 
$\RCB(\rG)=\Spec(R/\rG)$ \indexnot{R}{\RCB(\rG)}  is a union of closed irreducible 
subvarieties 
\equat
\Upsilon^{-1}_\RCB(\RCB(\rG)) = \bigcup_{w \in W} \RCB_w(\rG),\indexnot{R}{\RCB_w(\rG)}
\endequat
where $\RCB_w(\rG) \simeq \RCB(\rG)$ is the inverse image of $\RCB(\rG)$ in 
$\RCB_w$. 

\bigskip

\begin{lem}\label{lem:cellules-geometrie}
Let $w$ and $w' \in W$. We have $\RCB_w(\rG)=\RCB_{w'}(\rG)$ if and only if 
$w \sim_\rG^\calo w'$.
\end{lem}

\begin{proof}
Indeed, $\RCB_w(\rG)=\RCB_{w'}(\rG)$ if and only if, for all $r \in \RCB(\rG)$, 
we have $\r_H(w(r))=\r_H(w'(r))$. Translated at the level of the rings $Q$ and $R$, this becomes 
equivalent to saying that, for all $q \in Q$, we have $w(q) \equiv w'(q) \mod \rG$. 
\end{proof}

\bigskip

In other words, Lemma~\ref{lem:cellules-geometrie} shows that the Calogero-Moser $\rG$-cells 
parametrize the irreducible components of the inverse image 
of  $\RCB(\rG)$ in the fiber product $\RCB \times_\PCB \ZCB$. 

\bigskip

\section{Topology}
\label{se:cellstopology}

We assume in \S \ref{se:cellstopology} that $\kb=\CM$ and we identify 
the schemes $\PCB$, $\QCB$ and $\RCB$ with their $\CM$-points and
we endow them with the usual complex topology. 
We fix for the remainder of the book
a pair $(v_\CM,v^*_\CM)\in V^\reg\times V^{*\reg}$\indexnot{vc}{v_\CM,v^*_\CM}.

\medskip

We denote by $\RCB^{\nr}$
the complement of the ramification locus of $\rho_G:\RCB\to \PCB$ and by
$\PCB^{\nr}$\indexnot{P(C)}{\PCB^{\nr}} its image under $\r_G$.

Let $\gamma:[0,1]\to \PCB$ be a path with $\gamma([0,1))\subset \PCB^{\nr}$ and
such that
$\gamma(0)=(0,W\cdot v_\CM,W\cdot v^*_\CM)$.
Given $w\in W$, there is a unique path $\gamma_w:[0,1]\to \QCB$ lifting
$\gamma$ and such that $\gamma_w(0)=(0,(w(v_\CM),v^*_\CM)\Delta W)$ 
(Lemma \ref{le:inertiechemins}).

\begin{defi}
We say that $w,w'\in W$ are in the same {\bfit Calogero-Moser $\gamma$-cell}
\indexnot{ga}{\gamma-\textrm{cell}}
if $\gamma_w(1)=\gamma_{w'}(1)$.
\end{defi}

The choice of the prime ideal $\rG_0$ 
(cf. \S \ref{subsection:specialization galois 0}) corresponds to the choice
of an irreducible component of $\rho_G^{-1}(\{0\}\times V/W \times V^*/W)$.
The isomorphism $\iso$ extends to an isomorphism
$\kb[V^\reg\times V^{*\reg}]^{\Delta \Zrm(W)}\xrightarrow{\sim}(R/\rG_0)
\otimes_{\kb[V\times V^*]^{W\times W}}\kb[V^\reg\times V^{*\reg}]^{W\times W}$.
We have a corresponding isomorphism of varieties 
$$\Spec(R/\rG_0)\times_{V/W\times V^*/W}(V^\reg/W\times V^{*\reg}/W)
\xrightarrow{\sim} (V^\reg\times V^{*\reg})/\Delta \Zrm(W).$$
We denote by $y_0$ \indexnot{y}{y_0}
the point of the component that is the inverse image
of $(v_\CM,v^*_\CM)\Delta \Zrm(W)$.
Let $x_0=\rho_G(y_0)$. The choice
of $y_0$ provides a bijection 
$H\setminus G\xrightarrow{\sim}\Upsilon^{-1}(x_0),\
Hg\mapsto \rho_H(g\cdot y_0)$.
The bijection $W\xrightarrow{\sim}H\setminus G
\xrightarrow{\sim}\Upsilon^{-1}(x_0)$ is given by
$w\mapsto (w(v_\CM),v^*_\CM)\Delta W$.

Let $y_1$ be a point of $\RCB$ that lies in the
irreducible component determined by $\rG$ and satisfies
$\Stab_G(y_1)=G_{\rG}^I$.

\smallskip
We fix a path $\tilde{\gamma}:[0,1]\to \RCB$ such that 
$\tilde{\gamma}([0,1))\subset \RCB^{\nr}$, $\tilde{\gamma}(0)=y_0$ and
$\tilde{\gamma}(1)=y_1$. We denote by $\gamma$ the image of $\tilde{\gamma}$
in $\PCB$.

\medskip
From \S \ref{se:topologyinertia} we deduce the following result.

\begin{prop}
\label{pr:rgammacells}
Two elements $w,w'\in W$ are in the same Calogero-Moser $\rG$-cell if
and only if they are in the same Calogero-Moser $\gamma$-cell.
\end{prop}

\part{Cells and families}\label{part:verma}

\bigskip

\begin{defivide}
Fix a prime ideal $\rGba_\CG$ (resp. $\rG_\CG^\gauche$, resp. $\rG_\CG^\droite$) 
of $R$ lying over $\pGba_\CG$ (resp. $\pG_\CG^\gauche$, resp. $\pG_\CG^\droite$).  
\indexnot{ra}{\rGba_\CG, \rG_\CG^\gauche, \rG_\CG^\droite}   
\indexnot{ra}{\rGba, \rG^\gauche, \rG^\droite, \rGba_c, \rG_c^\gauche, \rG_c^\droite}  
A {\bfit Calogero-Moser two-sided (resp. left, resp. right) $\CG$-cell} is
defined to be an
$\rGba_\CG$-cell (resp. $\rG_\CG^\gauche$-cell, resp. $\rG_\CG^\droite$-cell).

When $\CG=0$, they will also be called {\bfit generic Calogero-Moser 
	(two-sided, left or right) cells} and we write $\rGba$ (resp. 
	$\rG^\gauche$, resp. $\rG^\droite$) for $\rGba_\CG$ (resp. $\rG_\CG^\gauche$, resp. $\rG_\CG^\droite$).
	Given $c \in \CCB$ and $\CG=\CG_c$, 
they are called {\bfit Calogero-Moser (two-sided, left or right) $c$-cells}
	 and we write $\rGba_c$ (resp. 
	$\rG_c^\gauche$, resp. $\rG_c^\droite$) for $\rGba_{\CG_c}$
	(resp. $\rG_{\CG_c}^\gauche$, resp. $\rG_{\CG_c}^\droite$).
\end{defivide}

\bigskip

\begin{remavide}\label{rem:dependance-r}
Of course, the notion of Calogero-Moser (two-sided, left, or
right) $\CG$-cell depends on the choice of the ideal 
$\rGba_\CG$, $\rG_\CG^\gauche$ or $\rG_\CG^\droite$; however, 
as all the prime ideals of $R$ lying over a prime ideal of $P$ are $G$-conjugate, 
changing the ideal amounts to transforming the cells according to the action of $G$.\finl
\end{remavide}

\bigskip

\begin{remavide}[Semi-continuity]\label{rem:semicontinuite}
It is of course possible to choose the ideals $\rGba_\CG$, $\rG_\CG^\gauche$ 
or $\rG_\CG^\droite$ so that $\rGba_\CG$ contains $\rG_\CG^\gauche$ and $\rG_\CG^\droite$: 
in this case, by Remark~\ref{semicontinuite}, 
the Calogero-Moser two-sided $\CG$-cells are unions of Calogero-Moser left (resp. right) 
 $\CG$-cells.

Similarly, if $\CG'$ is another prime ideal of $\kb[\CCB]$ such that $\CG \subset \CG'$, then 
one can choose the ideals $\rGba_{\CG'}$, $\rG_{\CG'}^\gauche$ or $\rG_{\CG'}^\droite$ 
in such a way that they contain respectively $\rGba_\CG$, $\rG_\CG^\gauche$ 
or $\rG_\CG^\droite$. Then the Calogero-Moser two-sided (resp. left, resp. right) 
$\CG'$-cells are unions of Calogero-Moser two-sided (resp. left, resp. right) 
$\CG$-cells.\finl
\end{remavide}

\bigskip

With the definition of Calogero-Moser two-sided, left or right cells given above, 
this completes the first aim
of this book. The aim of this part is now to study these 
particular cells, in relation with the representation theory of $\Hb$: 
in each case, a family of {\it Verma modules} will help us in this study. More precisely:
\begin{itemize}
\item In Chapter~\ref{chapter:bilatere}, we will associate a {\it Calogero-Moser family} 
with each two-sided cell: the Calogero-Moser families define a partition of 
$\Irr(W)$.

\item In Chapter~\ref{chapter:gauche}, we will associate a {\it
Calogero-Moser cellular character} with each left cell.
\end{itemize}
We conjecture that, whenever $W$ is a Coxeter group, all these notions coincide 
with the analogous notions defined by Kazhdan-Lusztig in the framework of Coxeter groups. 
The conjectures will be stated precisely in Chapter~\ref{part:coxeter} 
and some
evidence will be given (see \S\ref{chapter:arguments}).


\chapter{Calogero-Moser two sided cells}\label{chapter:bilatere}

\bigskip

In this chapter, we fix a prime ideal $\CG$ of $\kb[\CCB]$.
We study the relation between Calogero-Moser two-sided $\CG$-cells and 
$\CG$-families.

\bigskip

\section{Choices}\label{section:premiers}

\medskip

In order to relate the Calogero-Moser two-sided cells with the Kazhdan-Lusztig ones, 
we need to make an appropriate choice of a prime ideal $\rGba_\CG$ of $R$ lying over 
$\pGba_\CG$. We do not have a general procedure to make this choice (see
\S\ref{se:cellsandcharacters} for a discussion in the case of Coxeter groups).


\medskip

Recall that $\pGba$ denotes the prime ideal of $P$ corresponding to the closed irreducible 
subvariety $\CCB \times \{0\} \times \{0\}$ of $\PCB$ (c.f. page~\pageref{notationscellsVerma}). 
There are several prime ideals of $Z$ lying over $\pGba$. They are described in 
Lemma~\ref{caracterization blocs CM}, which says that they are in bijection
with the set of generic Calogero-Moser families of $W$. Recall also that
(Example~\ref{lineaire}) the trivial character $\unb_W$ of $W$ is itself a generic 
Calogero-Moser family. We denote by $\zGba$  \indexnot{z}{\zGba}  
the associated prime ideal of $Z$:
$$\zGba=\Ker(\O_{\unb_W}).$$
We put $\qGba=\copie(\zGba)$.  \indexnot{q}{\qGba}  

\bigskip

\begin{lem}\label{qba}
The ideal $\qGba$ of $Q$ is the unique prime ideal lying over $\pGba$ and containing 
$\eulerq-\sum_{H \in \AC} e_H K_{H,0}$.
The algebra $Q$ is \'etale over $P$ at $\qGba$.
\end{lem}

\begin{proof}
It follows from Corollary~\ref{multiplicite 1} that $Q$ is unramified over $P$ at $\qGba$: 
as $Q$ is a free (hence flat) $P$-module and since, in characteristic zero, 
all the field extensions are separable, we deduce that 
$Q$ is \'etale over $P$ at $\qGba$. The fact that $\euler -\sum_{H \in \AC} e_H K_{H,0} \in \zGba$ 
follows from Corollary~\ref{action euler verma}. 
Now, if $\zGba'$ is a prime ideal of $Z$ lying over $\pGba$ and containing 
$\euler-\sum_{H \in \AC} e_H K_{H,0}$, 
then there exists $\chi \in \Irr(W)$ such that $\zGba'=\Ker(\O_\chi)$. 
In particular, $\O_\chi(\euler)=\sum_{H \in \AC} e_H K_{H,0}$ and so 
$\O_\chi(\euler)=\O_{\unb_W}(\euler)$. It has been shown in Example~\ref{lineaire} 
that this implies $\chi=\unb_W$.
\end{proof}

\bigskip

Let us use now the notation of \S\ref{chapter:bebe-verma}.
%
Lemma~\ref{caracterization blocs CM} says that the set of prime ideals of $Q$ 
lying over $\pGba_\CG$ is in bijection with the set of Calogero-Moser 
$\CG$-families. Let $\zGba_\CG$ denote the prime ideal corresponding to the 
$\CG$-family containing the trivial character $\unb_W$ of $W$:
$$\zGba_\CG=\Ker(\O_{\unb_W}^{\CG}).$$  \indexnot{z}{\zGba_\CG}  
We set $\qGba_\CG=\copie(\zGba_\CG)$.   \indexnot{q}{\qGba_\CG}  

\bigskip

\begin{lem}\label{qba c}
We have $\qGba_\CG=\qGba + \CG Q$.
\end{lem}

\begin{proof}
The morphism $\O_{\unb_W} : Z \to \kb[\CCB]$ induces an isomorphism 
$Z/\zGba \longiso \kb[\CCB]$.
Since $\zGba_\CG$ contains $\CG$, it follows that $\zGba+\CG Z$ 
is a prime ideal of $Z$, corresponding to the closed irreducible subscheme 
$\CCB_\CG$ of $\CCB$. 
\end{proof}

\bigskip

\begin{coro}\label{qba c plus}
The ideal $\qGba_\CG$ of $Q$ is the unique prime ideal lying over $\pGba_\CG$ and
containing  $\eulerq-\sum_{H \in \AC} e_H K_{H,0}$. 
\end{coro}

\bigskip

\section{Two-sided cells}\label{cellules et familles}

\bigskip

\boitegrise{{\bf Assumption.} 
{\it From now on, and until the end of \S\ref{cellules et familles}, 
we fix a prime ideal $\rGba_\CG$  \indexnot{r}{\rGba_\CG}  of $R$ lying over $\qGba_\CG$. 
Recall that $\Kbov_\CG=k_P(\pGba_\CG)=k_{\kb[\CCB]}(\CG)$. We put 
$\Lbov_\CG=k_Q(\qGba_\CG)$  \indexnot{L}{\Lbov_\CG, \Lbov, \Lbov_c}  and 
$\Mbov_\CG=k_R(\rGba_\CG)$.   \indexnot{M}{\Mbov_\CG, \Mbov, \Mbov_c}  
We denote by $\Kbov$, $\Lbov$ and $\Mbov$ 
(respectively $\Kbov_c$, $\Lbov_c$ and $\Mbov_c$) the fields $\Kbov_\CG$, $\Lbov_\CG$ and  
$\Mbov_\CG$ whenever $\CG=0$ (respectively $\CG=\CG_c$ for some $c \in \CCB$).\\
\hphantom{A} The decomposition (respectively inertia) group of $\rGba_\CG$ will be denoted 
by $\Dba_\CG$  \indexnot{D}{\Dba_\CG, \Dba,\Dba_c}  (respectively $\Iba_\CG$).
\indexnot{I}{\Iba_\CG,\Iba, \Iba_c}  
We define similarly $\Dba$, $\Iba$, $\Dba_c$ and $\Iba_c$.
}}{0.75\textwidth}

\bigskip

\subsection{Galois theory}\label{sub:galois-bilatere}
Recall that, by Corollary~\ref{Q extention kc}, the canonical embedding 
$P/\pGba_\CG \longinjto Q/\qGba_\CG$ is an isomorphism, hence
$\Kbov_\CG=\Lbov_\CG$. Since $\Gal(\Mbov_\CG/\Lbov_\CG)=\Dba_\CG/\Iba_\CG$ 
(Theorem~\ref{bourbaki}), we deduce that 
\equat\label{eq:dba-h}
(\Dba_\CG \cap H)/(\Iba_\CG \cap H) \simeq \Dba_\CG/\Iba_\CG.
\endequat
In particular,
\equat\label{eq:iba-h}
(\Dba_\CG \cap H) \Iba_\CG =\Dba_\CG.
\endequat
Moreover, since the algebra $Q$ is unramified over $P$ at $\qGba$ (Lemma~\ref{qba}), 
it follows from Theorem~\ref{raynaud} that $\Iba \subset H$. Combined with 
(\ref{eq:iba-h}), we obtain
\equat\label{eq:iba-dba-h}
\Iba \subset \Dba \subset H.
\endequat
Note that this last result is not true in general for $\Iba_\CG$, as
shown by~(\ref{exemple:dba-0}),

To conclude with basic Galois properties, note that, by Corollaries~\ref{Q extention kc} 
and~\ref{dgh igh}, we have 
\equat\label{eq:dba-g-h}
\Iba_\CG\backslash G /H = \Dba_\CG \backslash G/H.
\endequat

\bigskip

\subsection{Two-sided cells and grading} 
Let $\CGt=\bigoplus_{i \ge 0} \CG \cap \kb[\CCB]^\NM[i]$ be the maximal homogeneous ideal
of $\kb[\CCB]$ contained in $\CG$.
Then $\CGt$ is a prime ideal of $\kb[\CCB]$ (see Lemma~\ref{lem:homogeneise-premier}). Let  
$\rGba_\CGt$ denote the maximal homogeneous ideal of $R$ contained in 
$\rGba_\CG$: it is a prime ideal of $R$ lying over $\qGba_\CGt$ 
(see Corollary~\ref{coro:homogeneise-premier}). 
The following is a consequence of Proposition~\ref{prop:cellules-homogeneise}.

\bigskip

\begin{lem}\label{lem:cellules-bilateres-homogeneise}
We have $\Iba_\CG=\Iba_\CGt$. The Calogero-Moser two-sided $\CG$-cells 
and the Calogero-Moser two-sided $\CGt$-cells coincide. 
\end{lem}

\bigskip

\subsection{Two-sided cells and families} 
The $\kb[\CCB]$-algebra $\Mbov_\CG$ is a finite field extension 
of $\Kbov_\CG=k_P(\pGba_\CG)=\Frac(\kb[\CCB]/\CG)$ and  
$\Mbov_\CG\Hb=\Mbov_\CG\Hbov$.
Theorem~\ref{theo:calogero} says that there is a bijection between 
the Calogero-Moser two-sided $\CG$-cells and the Calogero-Moser 
$\CG$-families: given $b \in \blocs(R_{\rGba_\CG} Z)$, 
this bijection sends $\calo_{\rGba_\CG}(b)$ to  
$\Irr_\Hb(W,\bba)$, where $\bba$ denotes the image of $b$ in $\Mbov_\CG\Hbov$, and
$\bba\in \Kbov_\CG\Hbov$ since  $\Kbov_\CG\Hbov$ is split.

\boitegrise{\noindent{\bf Terminology, notation.} 
{\it 
Given $b \in \blocs(\Mbov_\CG Z)$, we say that the Calogero-Moser two-sided $\CG$-cell
$\calo_{\rGba_\CG}(b)$ {\bfit covers} the Calogero-Moser $\CG$-family $\Irr_\Hb(W,b)$. 
Given $\G$ a Calogero-Moser two-sided $\CG$-cell, we denote by 
$\Irr_\G^\calo(W)$  \indexnot{I}{\Irr_\G^\calo(W)}  the Calogero-Moser $\CG$-family 
covered by $\G$. 
The set of Calogero-Moser two-sided $\CG$-cells will be denoted by 
\indexnot{C}{{{{^\calo{{\mathrm{Cells}}_{LR}^\CG(W)}}}}}  $\cmcellules_{LR}^\CG(W)$.}}{0.75\textwidth}

\bigskip

\begin{rema}
The definition of Calogero-Moser two-sided cells
depends on the choice of the prime ideal $\rGba_\CG$ lying over $\qGba_\CG$. 
Given $\rGba_\CG'$ another prime ideal of $R$ lying over 
$\qGba_\CG$, there exists $h \in H$ such that 
$\rGba_\CG'=h(\rGba_\CG)$ and the Calogero-Moser $\rGba_\CG'$-cells are obtained from the 
Calogero-Moser $\rGba_\CG$-cells via the action of $h$ on 
$W  \stackrel{\sim}{\longleftrightarrow} G/H$.\finl 
\end{rema}

\bigskip

The link between Calogero-Moser two-sided $\CG$-cells and Calogero-Moser 
$\CG$-families is strengthened by the following theorem.

\bigskip

\begin{theo}\label{theo cellules familles}
\begin{itemize}
\itemth{a} The decomposition group $\Dba_\CG$ acts trivially on $\cmcellules_{LR}^\CG(W)$. 

%

\itemth{b} Given
$w \in W$ 
and $\chi \in \Irr(W)$, the following are equivalent
\begin{itemize}
\item
the Calogero-Moser two-sided $\CG$-cell of $w$ is associated with the Calogero-Moser 
$\CG$-family of $\chi$
\item
$w^{-1}(\rGba_\CG) \cap Q = \copie(\Ker(\O_\chi^\CG))$
\item
$(w(q) \mod \rGba_\CG) = \O_\chi^\CG(\copie^{-1}(q)) \in \Mbov_\CG=k_R(\rGba_\CG)$ for all $q \in Q$. 
\end{itemize}
\itemth{c} Given $\G$ is a Calogero-Moser two-sided $\CG$-cell, we have
$|\G| = \sum_{\chi \in \Irr_\G^\calo(W)} \chi(1)^2$. 
\end{itemize}
\end{theo}

\begin{proof}
(a) follows from~\ref{eq:dba-g-h}.

\medskip

(b) Let $\omeba_w : Q \to R/\rGba_\CG$ denote the morphism of $P$-algebras which sends 
$q \in Q$ to the image of $\o_w(q)=w(q) \in R$ in $R/\rGba_\CG$. 
Then $w$ belongs to the Calogero-Moser two-sided $\CG$-cell associated with the Calogero-Moser 
$\CG$-family of $\chi$ if and only if 
$\omeba_w = \O_\chi$. But, by Lemma~\ref{caracterization blocs CM}, 
this is equivalent to say that $\Ker(\omeba_w) = \copie(\Ker(\O_\chi^\CG))$. 
Since $\Ker(\omeba_w)=w^{-1}(\rGba_\CG) \cap Q$, the first equivalence follows.
Since $Q=(w^{-1}(\rG) \cap Q) + \kb[\CCB]$ (Corollary~\ref{Q extention kc}), the
second equivalence follows.

The assertion (c) follows from Corollaries~\ref{dim bonne} and~\ref{dim center r}. 
\end{proof}

\bigskip

\begin{coro}\label{coro:famille-semicontinu}
Let $\CG'$ be a prime ideal of $\kb[\CCB]$ contained in $\CG$ and let $\rGba_{\CG'}$ 
be a prime ideal of $R$ lying above $\pGba_{\CG'}$ and contained in $\rGba_{\CG}$. 
Then the Calogero-Moser two-sided $\CG$-cells are unions of Calogero-Moser 
two-sided $\CG'$-cells. Moreover, if $\G$ is a Calogero-Moser two-sided 
$\CG$-cell and if $\G=\G_1 \coprod \cdots \coprod \G_r$ where the 
$\G_i$'s are Calogero-Moser two-sided $\CG'$-cells, then 
$$\Irr_\G^{\calo,\CG}(W)= \Irr_{\G_1}^{\calo,\CG'}(W) \coprod \cdots \coprod \Irr_{\G_r}^{\calo,\CG'}(W).$$
\end{coro}

\bigskip


\bigskip

\begin{coro}\label{w0 epsilon}
Assume that all the reflections of $W$ have order $2$ and that 
$w_0=-\Id_V \in W$. If $\G$ is a Calogero-Moser two-sided $\CG$-cell
covering the Calogero-Moser $\CG$-family $\FC$, then 
$w_0\G=\G w_0$ is the Calogero-Moser two-sided $\CG$-cell covering 
the Calogero-Moser $\CG$-family $\e\FC$.
\end{coro}

\begin{proof}
First of all, $w_0\G=\G w_0$ is a Calogero-Moser two-sided $\CG$-cell by 
Example~\ref{cellules w0} whereas $\e \FC$ is a Calogero-Moser $\CG$-family 
by Corollary~\ref{ordre 2}. 

Let $w \in \G$, $\chi \in \FC$ and $q \in Q$. By Theorem~\ref{theo cellules familles}(b), 
we only need to show that 
$ww_0(q) \equiv \O_{\chi\e}(q) \mod \rGba_\CG$. Let 
$\t_0=(-1,1,\e) \in \kb^\times \times \kb^\times \times W^\wedge$. 
By Proposition~\ref{tau 0 in G}, we have 
$w_0(q)=\lexp{\t_0}{q}$ for all $q \in Q$. Moreover, by Corollary~\ref{omega lineaire}, 
we have
$\O_{\chi\e}(q)=\lexp{\t_0}{\O_\chi(\lexp{\t_0}{q})}$ (because $\t_0$ has order $2$). 
Since $\t_0$ acts trivially on $\kb[\CCB]$, we have $\O_{\chi\e}(q)=\O_\chi(\lexp{\t_0}{q})$. 
It is now sufficient to show that 
$w(\lexp{\t_0}{q}) \equiv \O_\chi(\lexp{\t_0}{q})\mod \rGba_\CG)$. 
This follows from Theorem~\ref{theo cellules familles}(b). 
\end{proof}

\bigskip

%
%
%

\begin{exemple}[Smoothness]\label{cellules lisses}
If the ring $Q$ is regular at $\qGba_b$ and if $\chi$ denotes the unique element 
of $\Irr_\Hb(W,b)$ (see Proposition~\ref{prop lissite}), 
then $|\calo_\rG(b)|=\chi(1)^2$ by Theorem~\ref{theo cellules familles}(c).\finl
\end{exemple}

\bigskip

\begin{rema}\label{rem:gen-spe}
If $\rGba_\CG$ and $\rGba$ are chosen so that $\rGba \subset \rGba_\CG$, 
then $\Iba \subset \Iba_\CG$ and so every Calogero-Moser 
$\CG$-cell is a union of generic Calogero-Moser two-sided cells. 
It is the ``cell version'' of the corresponding result 
on families.\finl
\end{rema} 

\bigskip

If $\g$ is a linear character, then it is alone in its generic Calogero-Moser family 
(Example~\ref{lineaire}) and its covering generic Calogero-Moser two-sided cell 
contains only one element (Theorem~\ref{theo cellules familles}(c)). 
Let $w_\g$ denote this element. By Theorem~\ref{theo cellules familles}(b), we have 
\equat\label{wgamma}
w_\g^{-1}(\rGba) \cap Q = \Ker(\O_\g).
\endequat

\begin{coro}\label{nettete triviale}
We have $w_{\unb_W}=1$. In other words $1$ is alone in its generic Calogero-Moser two-sided cell
and covers the generic Calogero-Moser family of the trivial character $\unb_W$ (which is 
a singleton).
\end{coro}

\begin{proof}
By Theorem~\ref{theo cellules familles} and~(\ref{wgamma}), $w_{\unb_W}$ is  
the unique element $w \in W$ such that $w^{-1}(\rGba) \cap Q=\Ker(\O_{\unb_W})=\qGba$. 
Since $\rGba \cap Q =\qGba$, we have $w_{\unb_W}=1$.
\end{proof}

%

\bigskip

\begin{prop}\label{nettete lineaire}
Let $\g \in W^\wedge$. Then $\Iba \subset w_\g H w_\g^{-1}$. 
\end{prop}

\begin{proof}
We shall give two proofs of this fact. First, 
$w_\g$ is alone in its generic Calogero-Moser two-sided cell, 
so $\Iba w_\g H = w_\g H$, whence the result.

\medskip

Let us now give a second proof. By Corollary~\ref{multiplicite 1}, 
$Q$ is unramified over $P$ at $\Ker(\O_\g)=w_\g^{-1}(\rGba) \cap Q$. So, 
by Theorem~\ref{raynaud}, $I_{w_\g^{-1}(\rGba)}\subset H$, 
which is exactly the desired statement, since $I_{w_\g^{-1}(\rGba)}=w_\g^{-1} \Iba w_\g$.
\end{proof}

\bigskip

\begin{rema}
The action of $H$ on $W \stackrel{\sim}{\longleftrightarrow} G/H$ stabilizes 
the identity element (that is, $H$ stabilizes $\euler$). 
This shows that the statement of Corollary~\ref{nettete triviale} 
does not depend on the choice of $\rGba$.\finl
\end{rema}

\bigskip

\chapter{Calogero-Moser left and right cells}\label{chapter:gauche}

In this Chapter~\ref{chapter:gauche}, $\CG$ denotes a prime ideal of $\kb[\CCB]$. 
We relate Calogero-Moser cellular characters to
Calogero-Moser left cells.
This chapter will mainly consider {\it left} Calogero-Moser $\CG$-cells and 
(left) Verma modules:
definitions and results can be immediately transposed to the {\it right}
setting.

\bigskip

\section{Choices}\label{section:choix-gauche-R}

\medskip

As in the case of two-sided cells, 
the notion of Calogero-Moser left $\CG$-cell depends on the choice of a 
prime ideal of $R$ lying over $\pG_\CG^\gauche$. 
We will use Verma modules to restrict choices.

\bigskip

\boitegrise{\noindent{\bf Choices, notation.} 
{\it From now on, and until the end of this Part~\ref{part:verma}, we fix a prime ideal 
$\rG_\CG^\gauche$  \indexnot{ra}{\rG_\CG^\gauche, \rG^\gauche, \rG_c^\gauche}  
of $R$ lying over $\qG_\CG^\gauche$ and contained in $\rGba_\CG$. 
We put
$$
\Mb^\gauche_\CG=k_R(\rG_\CG^\gauche). \indexnot{M}{\Mb^\gauche_\CG}
$$ 
The decomposition (respectively inertia) group of $\rG_\CG^\gauche$ is
denoted by 
$D_\CG^\gauche$ \indexnot{D}{D^\gauche_\CG, D^\gauche, D^\gauche_c}  
(respectively $I_\CG^\gauche$).  \indexnot{I}{I^\gauche_\CG, I^\gauche, I^\gauche_c}  \\
\hphantom{A} Whenever $\CG=0$ (respectively $\CG=\CG_c$ with $c \in \CCB$), the objects 
$\rG_\CG^\gauche$, $\Kb_\CG^\gauche$, $D_\CG^\gauche$ and 
$I_\CG^\gauche$ are denoted respectively by $\rG^\gauche$, $\Kb^\gauche$,
$D^\gauche$ and $I^\gauche$ (respectively $\rG_c^\gauche$, $\Kb_c^\gauche$,  
$D_c^\gauche$ and $I_c^\gauche$). }}{0.75\textwidth}

\bigskip

\begin{coro}\label{dleft}
We have $I_\CG^\gauche \subset D_\CG^\gauche \subset H$. 
If moreover $\rG_\CG^\gauche$ contains $\rG^\gauche$, 
then $I^\gauche \subset I_\CG^\gauche$. 
\end{coro}

\begin{proof}
By Corollary~\ref{q gauche unique} and Theorem~\ref{raynaud}, 
$\Iba_\CG^\gauche \subset H$ and 
$k_P(\pG_\CG^\gauche)=k_Q(\qG_\CG^\gauche)$. So, 
$(D_\CG^\gauche \cap H)/(I_\CG^\gauche \cap H) \simeq D_\CG^\gauche/I_\CG^\gauche$, 
and the first sequence of inclusions follows.

The last inclusion is obvious.
\end{proof}

\bigskip

We will prove that $\rG_\CG^\gauche$ determines $\rGba_\CG$: for this, we will use the 
$\BZ$-grading on $R$ defined in~\S\ref{section:graduation R}. Set 
$$R_{<0} = \mathop{\bigoplus}_{i < 0} R^\BZ[i]\quad\text{and}\quad R_{>0} = \mathop{\bigoplus}_{i > 0} R^\BZ[i].$$
Then:

\bigskip

\begin{prop}\label{prop:rgauche-rbarre}
$\rGba_\CG = \rG_\CG^\gauche + \langle R_{>0},R_{<0} \rangle$. 
\end{prop}

\bigskip

\begin{proof}
	The proposition follows from Lemma \ref{le:barfromleft} applied
	to the extension $R/P$ and the prime ideals $\pGba_\CG$ and $\rG_\CG^\gauche$.
\end{proof}

\bigskip

\begin{coro}\label{coro:dgauche-dbarre}
$D_\CG^\gauche \subset \Dba_\CG$.
\end{coro}

\bigskip

\begin{proof}
This follows immediately from Proposition~\ref{prop:rgauche-rbarre} and from the fact 
that $R_{<0}$ and $R_{>0}$ are $G$-stable 
(see Proposition~\ref{bigraduation sur R}).
\end{proof}

\bigskip

\begin{rema}\label{rema:bb cache}
The algebraic proof of Proposition~\ref{prop:rgauche-rbarre} given here is in fact the translation of 
a geometric fact, as will be explained in Chapter~\ref{chapter:bb}.\finl
\end{rema}

\bigskip

%

\bigskip

\section{Left cells}

\medskip

\subsection{Definitions} 
Recall the definitions given in the preamble of \S\ref{part:verma}.

\medskip

\begin{defi}\label{defi:gauche}
A {\bfit Calogero-Moser left $\CG$-cell} is a
Calogero-Moser $\rG_\CG^\gauche$-cell. 
Given $c \in \CCB$, a {\bfit Calogero-Moser left $c$-cell} is a Calogero-Moser $\rG_c^\gauche$-cell. 
A {\bfit generic Calogero-Moser left cell} is a 
Calogero-Moser $\rG^\gauche$-cell. 

The set of Calogero-Moser left $\CG$-cells is denoted by $\cmcellules_L^\CG(W)$.  
\indexnot{C}{{{{^\calo{\mathrm{Cells}}_L^\CG(W), ^\calo{\mathrm{Cells}}_L(W), ^\calo{\mathrm{Cells}}_L^c(W)}}}}
When $\CG=0$ (respectively $\CG=\CG_c$, with $c \in \CCB$), 
this set is denoted by $\cmcellules_L(W)$ (respectively $\cmcellules_L^c(W)$). 
\end{defi}

\bigskip

As usual, the notion of Calogero-Moser left $\CG$-cell depends on the choice of 
the prime ideal $\rG_\CG^\gauche$. The next proposition is immediate.

\bigskip

\begin{prop}\label{semicontinuite gauche}
If $\CG'$ is a prime ideal of $\kb[\CCB]$ contained in $\CG$ and if 
$\rG_{\CG'}^\gauche$ is contained in $\rG_\CG^\gauche$, 
then every Calogero-Moser left $\CG$-cell is a union of Calogero-Moser left 
$\CG'$-cells.
\end{prop}

\bigskip

Also, since $\rG_\CG^\gauche \subset \rGba_\CG$, we have the following result.

\bigskip

\begin{prop}\label{bilatere gauche}
Every Calogero-Moser two-sided $\CG$-cell is a union of Calogero-Moser left 
$\CG$-cells.
\end{prop}

\bigskip

Finally, let $\CGt$ be the maximal homogeneous ideal contained in 
$\CG$ (i.e. $\CGt=\bigoplus_{i \ge 0} \CG \cap \kb[\CCB]^\NM[i]$). 
Then $\CGt$ is a prime ideal of $\kb[\CCB]$ (see Lemma~\ref{lem:homogeneise-premier}). Let  
$\rG^\gauche_\CGt$ denote the maximal homogeneous ideal 
contained in $\rG^\gauche_\CG$: it is a prime ideal of $R$ lying over 
$\qG^\gauche_\CGt$ 
(see Corollary~\ref{coro:homogeneise-premier}). 
We deduce from Proposition~\ref{prop:cellules-homogeneise} the following result.

\bigskip

\begin{prop}\label{lem:cellules-gauches-homogeneise}
We have $I_\CG^\gauche=I_\CGt^\gauche$. In particular, 
the Calogero-Moser left $\CG$-cells and the Calogero-Moser left $\CGt$-cells coincide.
\end{prop}

\bigskip

\subsection{Left and two-sided cells}\label{subsection:gauche-bilatere}
We fix here a Calogero-Moser two-sided $\CG$-cell $\G$ as well as a Calogero-Moser left $\CG$-cell $C$
contained in $\G$. Since $\Dba_\CG$ stabilizes $\G$ (see Theorem~\ref{theo cellules familles}(a)) 
and since $D_\CG^\gauche \subset \Dba_\CG$ (see Corollary~\ref{coro:dgauche-dbarre}), 
the group $D_\CG^\gauche$ stabilizes $\G$ (and permutes the left cells contained in $\G$). 
Set
$$C^\DD=\bigcup_{d \in D_\CG^\gauche} d(C).\indexnot{C}{C^\DD} $$
Let $w \in C^\DD$. We set $\qGba_\CG(\G)=w^{-1}(\rGba_\CG) \cap Q$ and   
\indexnot{qa}{\qGba_\CG(\G), \qG_\CG^\gauche(C^\DD)}
$\qG_\CG^\gauche(C^\DD)=w^{-1}(\rG^\gauche_\CG) \cap Q$. We also set $\zGba_\CG(\G)=\copie^{-1}(\qGba_\CG(\G))$  
\indexnot{za}{\zGba_\CG(\G), \zG_\CG^\gauche(C^\DD), \zG_\CG^\gauche(C)}  
and $\zG^\gauche_\CG(C^\DD)=\copie^{-1}(\qG_\CG^\gauche(C^\DD))$. 
It follows from Proposition~\ref{reduction} 
that $\zGba_\CG(\G)$ depends only on $\G$ and not on the choice of $C$ or $w$, whereas 
$\zG_\CG^\gauche(C^\DD)$ depends only on $C^\DD$ and not on the choice of $w$. We set 
$$\deg_\CG(C^\DD)=[k_Z(\zG_\CG^\gauche(C^\DD)):k_P(\pG^\gauche_\CG)].
\indexnot{da}{\deg_\CG(C^\DD), \deg_\CG(C)}$$
We sometimes use the notation $\zG_\CG^\gauche(C)$ or $\deg_\CG(C)$ instead of 
$\zG_\CG^\gauche(C^\DD)$ or $\deg_\CG(C^\DD)$. 

\bigskip

\begin{prop}\label{prop:gauche-bilatere}
Let $w \in C^\DD$. Then:
\begin{itemize}
\itemth{a} $\zGba_\CG(\G)=\limitegauche(\zG_\CG^\gauche(C^\DD))$.

\itemth{b} $\DS{\deg_\CG(C^\DD)=\frac{|C^\DD|}{|C|}=\frac{|D_\CG^\gauche|}{|(D_\CG^\gauche \cap \lexp{w}{H})I_\CG^\gauche|}}$.

\itemth{c} The map $D_\CG^\gauche \backslash \G \longto \lim_\gauche^{-1}(\zGba_\CG(\G))$, 
$C^\DD \longmapsto \zG_\CG^\gauche(C^\DD)$ is bijective.
\end{itemize}
\end{prop}

\bigskip

\begin{proof}
(a) Note that $\zGba_\CG(\G) \in \Upsilon^{-1}(\pGba_\CG)$, 
$\zG_\CG^\gauche(C^\DD) \in \Upsilon^{-1}(\pG_\CG^\gauche)$ and 
$\zG^\gauche_\CG(C^\DD) \subset \zGba_\CG(\G)$, whence the result by Proposition~\ref{coro:app-lim}.

\medskip

(b) Through the action of $w$, 
the extension $k_R(w^{-1}(\rG_\CG^\gauche))/k_P(\pG^\gauche_\CG)$ is Galois with group 
$\lexp{w^{-1}}{D_\CG^\gauche}/\lexp{w^{-1}}{I_\CG^\gauche}$ whereas the extension 
$k_R(w^{-1}(\rG_\CG^\gauche))/k_Q(\qG^\gauche_\CG(C^\DD))$ 
is Galois with group $(\lexp{w^{-1}}{D_\CG^\gauche} \cap H)/(\lexp{w^{-1}}{I_\CG^\gauche} \cap H)$. 
Hence 
$$\deg_\CG(C^\DD)=\frac{|D_\CG^\gauche|}{|(D_\CG^\gauche \cap \lexp{w}{H})I_\CG^\gauche|}.$$
Moreover, $|C^\DD|/|C|$ is equal to the index of the stabilizer of $C$ in $D^\gauche_\CG$: 
but this stabilizer is exactly $(D_\CG^\gauche \cap \lexp{w}{H})I_\CG^\gauche$. 

\medskip

(c) follows essentially from the commutativity of the diagram 
$$\diagram
D_\CG^\gauche\backslash G/H \rrto^{\DS{\sim}} \ddto && \Upsilon^{-1}(\pG_\CG^\gauche) \ddto^{\DS{\limitegauche}}\\
&&\\
\Dba_\CG \backslash G/H \rrto^{\DS{\sim}} && \Upsilon^{-1}(\pGba_\CG),
\enddiagram$$
where the left vertical arrow is the canonical map (since $D_\CG^\gauche \subset \Dba_\CG$), 
and the horizontal bijective maps are given by Proposition~\ref{reduction}. 
The additional ingredient is the equality
$\Dba_\CG \backslash G/H=\Iba_\CG \backslash G/H$ (see~(\ref{eq:dba-g-h})).
\end{proof}

\bigskip

\subsection{Left cells and simple modules}
By Example~\ref{exemple lisse}, we have $\zG_\singulier \cap P \not\subset \pG^\gauche$. Consequently, 
the results of \S\ref{section:cellules et lissite} can be applied. 
Let us recall here some consequences (see Theorem~\ref{lissite et simples}):

\bigskip

\begin{theo}\label{theo:gauche}
We have:
\begin{itemize}
\itemth{a} The algebra $\Mb_\CG^\gauche \Hb^\gauche$ is split and its simple modules 
have dimension $|W|$.

\itemth{b} Every block $\Mb_\CG^\gauche \Hb^\gauche$ admits a unique simple module.
\end{itemize}
\end{theo}

\bigskip

Given $C \in \cmcellules_L^\CG(W)$, we denote by $L_\CG^\gauche(C)$ the unique 
\indexnot{L}{\LC_\CG^\gauche(C), L^\gauche(C), L_c^\gauche(C)}  simple
$\Mb_\CG^\gauche \Hb^\gauche$-module belonging to the block of $\Mb_\CG^\gauche \Hb^\gauche$ 
associated with $C$. When $\CG=0$ (respectively $\CG=\CG_c$, for some $c \in \CCB$), 
the module $L_\CG^\gauche(C)$ is denoted by $L^\gauche(C)$ (respectively $L_c^\gauche(C)$). 

\medskip

The decomposition group $D_\CG^\gauche$ acts on the commutative ring 
$\Mb_\CG^\gauche Z^\gauche$ (the action factors through a faithful action of $D_\CG^\gauche/I_\CG^\gauche$) 
and 
\equat\label{eq:fixed-gauche}
\Kb_\CG^\gauche Z^\gauche=\bigl(\Mb_\CG^\gauche Z^\gauche\bigr)^{D_\CG^\gauche}.
\endequat
So the primitive idempotents of the left-hand side (which are in one-to-one correspondence 
with the simple $\Kb_\CG^\gauche Z^\gauche$-modules) are in one-to-one correspondence with 
the $D_\CG^\gauche$-orbits of primitive idempotents of $\Mb_\CG^\gauche Z^\gauche$. Thus we get 
a bijection
\equat\label{eq:bijection-dc-orbites}
\Irr(\Kb_\CG^\gauche Z^\gauche) \longbij \Bigl(\Irr(\Mb_\CG^\gauche Z^\gauche)\Bigr)/D_\CG^\gauche.
\endequat
Similarly, the decomposition group $D_\CG^\gauche$ acts on the commutative ring 
$\Mb_\CG^\gauche \Hb^\gauche$ and
\equat\label{eq:fixed-gauche-encore}
\Kb_\CG^\gauche \Hb^\gauche=\bigl(\Mb_\CG^\gauche \Hb^\gauche\bigr)^{D_\CG^\gauche}.
\endequat
So, through the bijection~(\ref{eq:bijection-dc-orbites}) and the Morita equivalence 
of Theorem~\ref{theo:morita-left}, we get another bijective map
\equat\label{eq:bijection-dc-encore}
\Irr(\Kb_\CG^\gauche \Hb^\gauche) \longbij \Bigl(\Irr(\Mb_\CG^\gauche \Hb^\gauche)\Bigr)/D_\CG^\gauche.
\endequat
Using the one-to-one correspondence 
$D_\CG^\gauche\backslash W \longiso D_\CG^\gauche\backslash G/H \longiso \Upsilon^{-1}(\pG_\CG^\gauche)$, 
$C^\DD \mapsto \zG_\CG^\gauche(C^\DD)$ given by Proposition~\ref{reduction}, 
we obtain the following commutative diagram of bijective maps.
\equat\label{eq:diagramme-gauche}
\diagram
D_\CG^\gauche\backslash W \ar@{<->}[r] &\Upsilon^{-1}(\pG_\CG^\gauche) \ar@{<->}[r] 
& \Irr(\Kb_\CG^\gauche Z^\gauche) \ar@{<->}[r]\ar@{<->}[dd]
& \Irr(\Kb_\CG^\gauche \Hb^\gauche) \ar@{<->}[dd]\\
&&&\\
&& \Bigl(\Irr(\Mb_\CG^\gauche Z^\gauche)\Bigr)/D_\CG^\gauche\ar@{<->}[r]
& \Bigl(\Irr(\Mb_\CG^\gauche \Hb^\gauche)\Bigr)/D_\CG^\gauche
\enddiagram
\endequat

\bigskip

\begin{prop}\label{prop:oulala}
Let $C^\circ$ be a $D_\CG^\gauche$-orbit in $W$ and let $\zG=\zG_\CG^\gauche(C^\circ)$. 
We have
$$\Mb_\CG^\gauche \otimes_{\Kb_\CG^\gauche} L_\CG^\gauche(\zG) \simeq 
\bigoplus_{\substack{C \in \cmcellules_L^\CG(W) \\ C \subset C^\circ}} L_\CG^\gauche(C).$$
\end{prop}

\begin{proof}
First, it is clear that if $C$ is a Calogero-Moser left $\CG$-cell 
such that $L_\CG^\gauche(C)$ is contained in the $\Mb_\CG^\gauche\Hb^\gauche$-module 
$\Mb_\CG^\gauche \otimes_{\Kb_\CG^\gauche} L_\CG^\gauche(\zG)$, then 
$\zG_\CG^\gauche(C^\DD)=\zG$, hence $C^\DD \subset C^\circ$. It follows
that $C \subset C^\circ$.  

Since the field extension $\Mb_\CG^\gauche/\Kb_\CG^\gauche$ is separable, 
the $\Mb_\CG^\gauche \Hb^\gauche$-module 
$\Mb_\CG^\gauche \otimes_{\Kb_\CG^\gauche} L_\CG^\gauche(\zG)$ is semisimple. 
As it is stable under the action of $D_\CG^\gauche$, it is a multiple 
of the right-hand side of the formula. By Theorem~\ref{theo:gauche}, 
the proof of the Proposition can be reduced to the proof of the analogous 
statement for the algebra $\Mb_\CG^\gauche Z^\gauche$. Since this algebra is commutative, 
it follows that $\Mb_\CG^\gauche \otimes_{\Kb_\CG^\gauche} L_\CG^\gauche(\zG)$ 
is multiplicity-free.
\end{proof}

\section{Back to cellular characters}
\label{se:backtocellular}

\medskip

\subsection{Left cells and cellular characters}
Recall that the simple $\Mb_\CG^\gauche \Hb^\gauche$-modules are parametrized by $\cmcellules_\CG^\gauche(W)$.
There exists a unique family of non-negative integers
$(\mult_{C,\chi}^\calo)_{C \in \cmcellules_\CG^\gauche(W),\chi \in \Irr(W)}$  
\indexnot{m}{\mult_{C,\chi}^\calo}  such that 
$$\isomorphisme{\Mb_\CG^\gauche\Delta(\chi)}_{\Mb_\CG^\gauche \Hb^\gauche} = \sum_{C \in \cmcellules_L^\CG(W)} 
\mult_{C,\chi}^\calo ~\cdot ~\isomorphisme{L_\CG^\gauche(C)}_{\Mb_\CG^\gauche \Hb^\gauche}$$
for all $\chi \in \Irr(W)$. They can be used to define the cellular characters, since
\equat\label{eq:mult-mult}
\mult_{C,\chi}^\calo = \mult_{\zG_\CG^\gauche(C),\chi}^\calo.
\endequat

\begin{proof}
By construction, we have
$$\isomorphisme{\Mb_\CG^\gauche\Delta(\chi)}_{\Mb_\CG^\gauche \Hb^\gauche} 
= \sum_{\zG \in \Upsilon^{-1}(\pG_\CG^\gauche)} \mult_{\zG,\chi}^\calo \cdot 
\isomorphisme{\Mb_\CG^\gauche \otimes_{\Kb_\CG^\gauche} L_\CG^\gauche(\zG)}_{\Mb_\CG^\gauche \Hb^\gauche} 
$$
for all $\chi \in \Irr(W)$. The result follows now from Proposition~\ref{prop:oulala}.
\end{proof}

\bigskip

We can also prove the following family of identities, which are similar to identities 
for the Kazhdan-Lusztig multiplicities $\mult_{C,\chi}^\kl$ (see Lemma~\ref{lem:mult-kl}).

\bigskip

\begin{prop}\label{multiplicite cm}
With the above notation, we have:
\begin{itemize}
\itemth{a} If $\chi \in \Irr(W)$, then $\sum_{C \in \cmcellules_L^\CG(W)} \mult_{C,\chi}^\calo = \chi(1)$.
 
\itemth{b} If $C \in \cmcellules_L^\CG(W)$, then $\sum_{\chi \in \Irr(W)} \mult_{C,\chi}^\calo~\chi(1) = |C|$.

\itemth{c} If $C \in \cmcellules_L^\CG(W)$, if $\G$ is the unique Calogero-Moser 
two-sided $\CG$-cell containing $C$ and if $\chi \in \Irr(W)$ is such that 
$\mult_{C,\chi}^\calo \neq 0$, then $\chi \in \Irr_\G^\calo(W)$.
\end{itemize}
\end{prop}

\begin{proof}
(a) follows from the computation of the dimension of Verma modules (see~(\ref{dimension left})). 

\medskip

Let us now show (b). First of all, note that, thanks to the Morita equivalence of Theorem~\ref{theo:KH-mat}, 
we have
$$\isomorphisme{\Mb\Hb e}_{\Mb\Hb} = \sum_{w \in W} \isomorphisme{\LC_w}_{\Mb\Hb}.$$
By applying $\dec_\CG^\gauche$ to this equality, we deduce that
$$\isomorphisme{\Mb_\CG^\gauche \Hb^\gauche e}_{\Mb_\CG^\gauche \Hb^\gauche} 
= \sum_{C \in \cmcellules_L^\CG(W)} 
|C|~\cdot~\isomorphisme{L_\CG^\gauche(C)}_{\Mb_\CG^\gauche \Hb^\gauche}.$$
Since $\Mb_\CG^\gauche \Hb^\gauche e=\Mb_\CG^\gauche \Delta(\mathrm{co})$,
we have the following equality
\equat\label{he verma}
\isomorphisme{\Mb_\CG^\gauche \Hb^\gauche e}_{\Mb_\CG^\gauche \Hb^\gauche}
= \sum_{\chi \in \Irr(W)} 
\chi(1)~\cdot~\isomorphisme{\Mb_\CG^\gauche\Delta(\chi)}_{\Mb_\CG^\gauche \Hb^\gauche}
\endequat
and (b) follows.

\medskip

(c) is immediate, as the reduction modulo $\pGba$ of the Verma module 
is the corresponding baby Verma module, and so is indecomposable 
as a $\Kbov_\CG\Hbov$-module.
\end{proof}

\bigskip

Given $C$ a Calogero-Moser left $\CG$-cell, we set
\equat\label{eq:cellulaire}
\isomorphisme{C}_\CG^\calo = \sum_{\chi \in \Irr(W)} \mult_{C,\chi}^\calo \cdot \chi.  
\indexnot{C}{\isomorphisme{C}_\CG^\calo, \isomorphisme{C}^\calo, \isomorphisme{C}_c^\calo}
\endequat
In other words, 
\equat\label{eq:easy}
\isomorphisme{C}_\CG^\calo=\g_{\zG_\CG^\calo(C^\DD)}.
\endequat
Taking Proposition~\ref{multiplicite cm}(c) into account, the set of irreducible characters 
appearing with a non-zero multiplicity in a Calogero-Moser $\CG$-cellular character 
is contained in a unique Calogero-Moser $\CG$-family $\FC$: we will say that the 
Calogero-Moser $\CG$-cellular character {\it belongs to $\FC$}.

%
%
%

\medskip

If $d \in D_\CG^\gauche$ and $C$ is a Calogero-Moser left $\CG$-cell, then $d(C)$ 
is also a Calogero-Moser left $\CG$-cell. The equality~(\ref{eq:easy}) 
shows that the Calogero-Moser $\CG$-cellular characters 
associated with $C$ and $d(C)$ coincide:

\bigskip

\begin{coro}\label{coro:dec-cellulaire}
If $d \in D_\CG^\gauche$ and $C$ is a Calogero-Moser left $\CG$-cell, then 
$$\isomorphisme{d(C)}_\CG^\calo=\isomorphisme{C}_\CG^\calo.$$
\end{coro}

\bigskip

\begin{rema}\label{rema:decleft-decba}
The previous Corollary~\ref{coro:dec-cellulaire} shows in particular that 
$d(C)$ is contained in the same Calogero-Moser two-sided $\CG$-cell as $C$, 
which has already been proven by a different argument 
(see the beginning of~\S\ref{subsection:gauche-bilatere}).\finl
\end{rema}

\bigskip

\begin{coro}\label{coro:cardinal-C}
Let $C$ be a Calogero-Moser left $\CG$-cell and let $\zG=\zG_\CG^\gauche(C^\DD)$. Then 
$|C|=\longueur_{Z_\zG}(Z/\pG_\CG^\gauche Z)_\zG$. 
\end{coro}

\bigskip

\begin{proof}
We have $[P_\CG^\gauche \Hb e]=[P_\CG^\gauche\Delta(\mathrm{co})]=
\sum_\chi \chi(1)[P_\CG^\gauche \Delta(\chi)]$.
Via the Morita equivalence of Corollary~\ref{morita lissite p}, the module
$Z/\pG_\CG^\gauche Z$ corresponds to the module $P_\CG^\gauche \Hb e$,
hence
$$\longueur_{Z_\zG}(Z/\pG_\CG^\gauche Z)_\zG=\sum_{\chi \in \Irr(W)} \mult_{C,\chi}^\calo \cdot \chi(1)=|C|$$
using Proposition~\ref{multiplicite cm}(b).
\end{proof}

\bigskip

\begin{coro}\label{coro:cellulaire-ordre-2}
Assume that all the reflections of $W$ have order $2$ and let 
$\t_0=(-1,1,\e) \in \kb^\times \times \kb^\times \times W^\wedge$. Let $\zG$ be a prime 
ideal of $Z$ lying over $\pG_\CG^\gauche$ and let $C$ and $C_\e$ be two Calogero-Moser 
left $\CG$-cells such that $\zG_\CG^\gauche(C^\DD)=\zG$ and $\zG_\CG^\gauche(C_\e^\DD)=\t_0(\zG)$. 
Then 
$$\isomorphisme{C_\e}_\CG^\calo = \e \cdot \isomorphisme{C}_\CG^\calo.$$
If moreover $w_0=-\Id_V \in W$, then we can take $C_\e=Cw_0$, hence
$$\isomorphisme{C w_0}_\CG^\calo = \e \cdot \isomorphisme{C}_\CG^\calo.$$
\end{coro}

\begin{proof}
The first statement follows from the fact that 
$\lexp{\t_0}{\Delta(\chi)} \simeq \Delta(\chi \e)$ whereas 
the second can be proven as in Corollary~\ref{w0 epsilon}.
\end{proof}

\bigskip

\begin{rema}\label{rem:pas-si-trivial}
Corollary~\ref{coro:cardinal-C} is interesting in that
it provides a numerical invariant (the cardinality) of an 
object (a left cell) which is defined using the Galois extension $\Mb/\Kb$ (and the ring $R$) 
in terms of an invariant which is computable inside the extension $\Lb/\Kb$ 
(and the ring $Z$).\finl
\end{rema}

\bigskip

\subsection{Cellular characters and projective covers} 
Note that $\Mb_\CG^\gauche \Hbov^\moins$ is a $\Mb_\CG^\gauche$-subalgebra of 
$\Mb_\CG^\gauche\Hb^\gauche$ of dimension $|W|^2$, whose Grothendieck group 
is identified with $\BZ\Irr(W)=\groth(\kb W)$. 

Given $C$ a Calogero-Moser left $\CG$-cell, we denote by
$\PC^\gauche_\CG(C)$  
\indexnot{P}{\PC^\gauche_\CG(C)}  a projective cover of the simple 
$\Mb_\CG^\gauche \Hb^\gauche$-module $L_\CG^\gauche(C)$. We denote
by $\mathrm{Soc}(M)$ the largest semisimple submodule (the socle)
of a module $M$.

\bigskip

\begin{prop}\label{eq:socle-cellulaire}
We have
$$\isomorphisme{{\mathrm{Soc}}(\Res_{\Mb_\CG^\gauche\Hbov^\moins}^{\Mb_\CG^\gauche\Hb^\gauche} 
\PC_\CG^\gauche(C))}_{\Mb_\CG^\gauche \otimes \Hbov^\moins} 
= \sum_{\chi \in \Irr(W)} \mult_{C,\chi}^\calo \cdot \chi = \isomorphisme{C}_\CG^\calo.$$
\end{prop}

\bigskip

\begin{proof}
Let $\chi \in \Irr(W)$. Since the algebra $\Hb$ is symmetric (see~(\ref{symetrique})), 
$\PC_\CG^\gauche(C)$ is also an injective hull of $L_\CG^\gauche(C)$. So 
$$\mult_{C,\chi}^\calo = \dim_{\Mb_\CG^\gauche} \Hom_{\Mb_\CG^\gauche\Hb^\gauche}\bigl(
\Mb_\CG^\gauche \Delta(\chi),\PC_\CG^\gauche(C) \bigr).$$
As $\Mb_\CG^\gauche \Delta(\chi) = \Ind_{\Mb_\CG^\gauche\Hbov^\moins}^{\Mb_\CG^\gauche\Hb^\gauche} 
(\Mb_\CG^\gauche \otimes E_\chi)$, we deduce that
$$\mult_{C,\chi}^\calo = \dim_{\Mb_\CG^\gauche}\Hom_{\Mb_\CG^\gauche\Hbov^\moins}\bigl( E_\chi,
\Res_{\Mb_\CG^\gauche \otimes \Hbov^\moins}^{\Mb_\CG^\gauche\Hb^\gauche} \PC_\CG^\gauche(C)\bigr).$$
The result follows.
\end{proof}

\bigskip

\subsection{Cellular characters and $\bb$-invariant} 
The following theorem is an analogue of Theorem~\ref{dim graduee bonne} (statements~(b) and~(c)). 

\bigskip

\begin{theo}\label{theo:b-minimal-cellulaire}
Let $C$ be a Calogero-Moser left $\CG$-cell. Then there exists a unique 
irreducible character $\chi$ with minimal $\bb$-invariant such that 
$\mult_{C,\chi}^\calo \neq 0$. We denote this
character by $\chi_C$. 
The coefficient of $\tb^{\bb_{\chi_C}}$ in $f_{\chi_C}(\tb)$ is equal to $1$. 
\end{theo}

\medskip
The character $\chi_C$ is called the {\em special} character of the
left $\CG$-cell $C$.

\bigskip

\begin{proof}
Let $b_C$ be the primitive central idempotent of $\Mb_\CG^\gauche\Hb^\gauche$ associated with $C$. 
The endomorphism algebra of $b_C \Mb_\CG^\gauche\Hb^\gauche e$ is equal to 
$(\Mb_\CG^\gauche \otimes_P Z)b_C$ and this (commutative) algebra is local. This shows that 
the projective module $b_C \Mb_\CG^\gauche\Hb^\gauche e$ admits a unique simple quotient. The proof continues as that of Theorem~\ref{dim graduee bonne}.
\end{proof}

\section{Topology}
\label{se:chemins}

We assume in \S \ref{se:chemins} that $\kb=\CM$.

\subsection{$\gamma$-cells}

Recall (cf \S\ref{se:Gaudinoperators}) that given $(c,v,v^*)\in \CCB(\CM)\times V^\reg\times V^*$,
we have a family of commuting operators $\{D_y^{c,v,v^*}\}_{y\in V}$
acting on $\bigoplus_{w\in W} \CM e_w$:

$$D_y^{c,v,v^*}:e_w\mapsto \langle y,w(v^*)\rangle e_w+
\sum_{s\in\REF(W)}\e(s) c_s
\frac{\langle y,\alpha_s\rangle}{\langle v,\alpha_s\rangle}e_{sw}.$$

Let $\gamma:[0,1]\to \CCB(\CM)\times V^\reg/W\times
V^*/W$ be a path with
$\gamma([0,1))\subset\PCB^{\nr}$
and $\gamma(0)=(0,W\cdot v_\CM,
W\cdot v^*_\CM)$ as in \S \ref{se:cellstopology}.
We denote by $\hat{\gamma}$ the path in $\CCB(\CM)\times V\times V^*$
lifting $\gamma$ with $\hat{\gamma}(0)=(0,v_\CM,v^*_\CM)$.

\begin{theo}
\label{th:cellsGaudin}
Let $w\in W$. There is a unique path $\rho_w:[0,1]\to V^*$ such that
\begin{itemize}
\itemth{1} $\rho_w(0)=w^{-1}(v^*_\CM)$.
\itemth{2} $\langle y,\rho_w(t)\rangle$ is an eigenvalue of $D_y^{\hat{\gamma}(t)}$
for $y\in V$ and $t\in [0,1)$
\end{itemize}

Two elements $w',w''\in W$ are in the same Calogero-Moser $\gamma$-cell if
and only if $\rho_{w'}(1)=\rho_{w''}(1)$.
\end{theo}

\begin{proof}
Appendix \S\ref{se:topologyinertia} applied to the covering 
$\ZC'\to \CCB\times V^\reg\times V^*/W$ of \S\ref{se:Wcovering}
shows the existence of a path $\tilde{\gamma}_w$ in $\ZC'$
lifting the image
of $\hat{\gamma}$ in $\CCB(\CM)\times V\times V^*/W$ and
such that $\tilde{\gamma}_w(0)=(z_w,w^{-1}(v_\CM^*))$, where 
$z_w=(0,(v_\CM,w^{-1}(v_\CM^*))\Delta W)$. Note that the
image of $\tilde{\gamma}_w$ in $\ZC$ is the path $\gamma_w$ of
\S\ref{se:cellstopology}. 
Define $\rho_w(t)$ to be the $\lambda$-component of the
image of $\tilde{\gamma}_w(t)$ in $\bar{\ZC}''$, via the inverse of the
	isomorphism of
Proposition \ref{pr:spectrumGaudin}. It satisfies the required properties.

The last statement follows by base change via the unramified map
$V^\reg\to V^\reg/W$. 
\end{proof}

\subsection{Left cells}

Let $\CG$ be a prime ideal of $\CM[\CCB]$ and let $\rG_\CG^\gauche$ be a prime
ideal of $R$ as in the preamble to Part~\ref{part:verma}. Recall that
we have defined a point $y_0\in\RCB(\CM)$ above
$(0,v_\CM,v^*_\CM)$ in \S\ref{se:cellstopology}.

Let $y_1\in \RCB(\CM)$ be a point in the irreducible
component determined by $\rG_\CG^\gauche$ such that
$\Stab_G(y_1)=G_{\rG_\CG^\gauche}^I$.
Fix a path $\hat{\tilde{\gamma}}:[0,1]\to \RCB(\CM)\times_{V/W\times V^*/W}
(V^\reg\times V^*)$ such that $\hat{\tilde{\gamma}}([0,1))\subset 
\RCB(\CM)^{\nr}\times_{V/W\times V^*/W} (V^\reg\times V^*)$,
$\hat{\tilde{\gamma}}(0)=(y_0,(v_\CM,v^*_\CM))$ and
$\hat{\tilde{\gamma}}(1)$ maps onto $y_1$.
We denote by $\hat{\gamma}$ the image of $\hat{\tilde{\gamma}}$ in 
$\CCB(\CM)\times V\times V^*$.

\smallskip
Theorem \ref{th:cellsGaudin} and Proposition \ref{pr:rgammacells}
have the following consequence.

\begin{theo}\label{theo:cm-cells-topo}
Let $w\in W$. There is a unique path $\rho_w:[0,1]\to V^*$ such that
\begin{itemize}
\itemth{1} $\rho_w(0)=w^{-1}(v^*_\CM)$.
\itemth{2} $\langle y,\rho_w(t)\rangle$ is an eigenvalue of $D_y^{\hat{\gamma}(t)}$
for $y\in V$ and $t\in [0,1)$
\end{itemize}

Two elements $w',w''\in W$ are in the same Calogero-Moser left
$\CG$-cell if
and only if $\rho_{w'}(1)=\rho_{w''}(1)$.
\end{theo}

\section{The smooth case}\label{se:smooth-cell}

\medskip

\boitegrise{{\bf Assumption and notation.} 
{\it We assume here that $\kb=\CM$. 
We fix in \S\ref{se:smooth-cell} a point $z_0 \in \ZCB_c^{\CM^\times}$ 
which is assumed to be {\bfit smooth} in $\ZCB_c$. We denote by $\chi$ the unique 
irreducible character of the associated Calogero-Moser $c$-family. We denote by 
$\G$ the Calogero-Moser two-sided $c$-cell associated with $z_0$ 
and we fix a Calogero-Moser left $c$-cell $C$ contained in $\G$.}}{0.75\textwidth}

\bigskip

The next result is a translation of Theorem~\ref{theo:cellulaire-lisse-geometrique} 
in the cell world.

\bigskip

\begin{theo}\label{theo:cellulaire-lisse}
With the assumption and notation above, we have:
\begin{itemize}
\itemth{a} $|\G|=\chi(1)^2$.

\itemth{b} $\bigcup_{d \in D_c^\gauche} \lexp{d}{C}=\G$. 

\itemth{c} $\isomorphisme{C}_c^\calo = \chi$.

\itemth{d} $|C|=\chi(1)$.

\itemth{e} $\deg_c(C)=\chi(1)$. 
\end{itemize}
\end{theo}

%
%
%

\begin{proof}
(a) follows from Theorem~\ref{theo cellules familles}(c).

\medskip

(b) The set of irreducible components of $\limiteattractiveinverse(z_0)$ is in 
bijection with $D_c^\gauche\setminus\Gamma$
(see Proposition~\ref{prop:gauche-bilatere}(c)). 
As explained in~\S\ref{se:smooth}, $\limiteattractiveinverse(z_0)$ is irreducible because 
$z_0$ is smooth. In other words, with the notation 
introduced in \S\ref{subsection:gauche-bilatere}, we have $C^\DD=\G$. 

\medskip

(c) follows from Thereom~\ref{theo:cellulaire-lisse-geometrique} and 
(d) follows immediately from~(c).

\medskip

(e) follows from~(b) and from Proposition~\ref{prop:gauche-bilatere}(b).
\end{proof}

\bigskip

\chapter{Decomposition maps}
\label{ch:decmat}

\section{The general framework}\label{part:decomposition}


\medskip

\def\propdec{({\DC\!e\!c})}

Let $R_1$ be a commutative $R$-algebra which is a domain, 
and let $\rG_1$ be a prime ideal of $R_1$. 
We set $R_2=R_1/\rG_1$, $K_1=\Frac(R_1)$ and $K_2=\Frac(R_2)=k_{R_1}(\rG_1)$. 
We will say that the pair $(R_1,\rG_1)$ satisfies Property~$\propdec$ 
if the three following statements are satisfied (see the Appendix~\ref{appendice: blocs}):
\begin{quotation}
\begin{itemize}
\item[(D1)] $R_1$ is noetherian.

\item[(D2)] If $h \in R_1\Hb$ and if $\LC$ is a simple $K_1\Hb$-module, then 
the characteristic polynomial of $h$ (for its action on $\LC$) has coefficients in $R_1$ 
(note that this assumption is automatically satisfied if $R_1$ is integrally closed).

\item[(D3)] The algebras $K_1\Hb$ and $K_2\Hb$ are split.
\end{itemize}
\end{quotation}
In this context, completely similar to the one of
\S~\ref{section:decomposition}, the decomposition map 
$$\dec_{R_2\Hb}^{R_1\Hb} : \groth(K_1\Hb) \longto \groth(K_2\Hb)$$
is well-defined (see Proposition~\ref{prop:geck-rouquier}).

Let $\rG_2$ be a prime ideal of $R_1$ containing $\rG_1$. Let  
$R_3=R_1/\rG_2 =R_2/(\rG_2/\rG_1)$, $K_3=\Frac(R_3)=k_{R_1}(\rG_2)=k_{R_2}(\rG_2/\rG_1)$ 
and assume that $(R_2,\rG_2)$ satisfies $\propdec$. 
Then the maps $\dec_{R_2\Hb}^{R_1\Hb}$, $\dec_{R_3\Hb}^{R_1\Hb}$ and 
$\dec_{R_3\Hb}^{R_2\Hb}$ are well-defined and, 
by Corollary~\ref{geck rouquier}, the diagram 
\equat\label{transitivite decomposition}
\diagram
\groth(K_1\Hb) \rrto^{\DS{\dec_{R_2\Hb}^{R_1\Hb}}} \ddrrto_{\DS{\dec_{R_3\Hb}^{R_1\Hb}}} 
&& \groth(K_2\Hb) \ddto^{\DS{\dec_{R_3\Hb}^{R_2\Hb}}}\\
&&\\
&&\groth(K_3\Hb)\\
\enddiagram
\endequat
is commutative.

\bigskip

\begin{exemple}[Specialization]
\label{subsection:specialization cellules} 
Let $c \in \CCB$. Recall that $\qG_c=\pG_c Q$ is prime and let 
$\rG_c$ be a prime ideal of $R$ lying over $\qG_c$. Let us use the notations of 
example~\ref{specialization} and of~\S\ref{subsection:specialization galois}. 
Then $R/\rG_c$ is an $R$-algebra, with fraction field $\Mb_c$. 
As in the proof of Theorem~\ref{theo:KH-mat}, we deduce from Corollary~\ref{morita lissite p}
an isomorphism of $\Kb_c$-algebras 
$$\Hb_c \longiso \Mat_{|W|}(\Lb_c)$$
which induces, as in the generic case (see~\S\ref{section:deploiement}), an isomorphism of 
$\Mb_c$-algebras 
$$\Mb_c \Hb_c 
\stackrel{\sim}{\longto} \prod_{d (D_c \cap H) \in D_c /(D_c \cap H)} \Mat_{|W|}(\Mb_c).$$
So the $\Mb_c$-algebra $\Mb_c\Hb_c$ is split, as well as $\Mb \Hb$, 
and its simple modules are indexed by $D_c/(D_c \cap H)$: this last set 
is in one-to-one correspondence with $W$ (Corollary~\ref{Dc W}). 
So, the decomposition map 
$\dec_{R_c\Hb}^{R\Hb}$ is well-defined, and will be denoted by $\dec_c$.  \indexnot{da}{\dec_c}  
We can moreover identify $\groth(\Mb_c \Hb_c)$ with the $\BZ$-module $\BZ W$ and, 
through this identification, the diagram 
\equat\label{dec c}
\xymatrix{
\groth(\Mb\Hb) \ar[rr]^{\dec_c} \ar@{=}[d]&& \groth(\Mb_c\Hb_c) \ar@{=}[d] \\
\BZ W \ar[rr]_{\DS{\Id_{\BZ W}}} && \BZ W
}
\endequat
is commutative. This follows from the fact that the Morita equivalence between 
$\Kb_c\Hb_c$ and $\Lb_c$ is the ``specialization at $c$'' of the Morita equivalence 
between $\Hb$ and $\Lb$.\finl
\end{exemple}

\bigskip

%
%
%

\section{Cells and decomposition maps}

\medskip

Let $\rG$ be a prime ideal of $R$. We will denote by 
$D_\rG$ the decomposition group of $\rG$ in $G$ and $I_\rG$ its inertia group.
The Galois group $G$ (respectively the decomposition group $D_\rG$) 
acts naturally on the Grothendieck group $\groth(\Mb\Hb)$ (respectively $\groth(k_R(\rG)\Hb)$). 
Then:

\bigskip

\begin{lem}\label{lem:dec-constant}
Assume that the $k_R(\rG)$-algebra $k_R(\rG)\Hb$ is split. Then:
\begin{itemize}
\itemth{a} The decomposition map $\dec_{(R/\rG)\Hb}^{R\Hb}$ is well-defined 
(it will be denoted by $\dec_\rG : \groth(\Mb\Hb) \longto \groth(k_R(\rG)\Hb)$).  \indexnot{da}{\dec_\rG}  

\itemth{b} The decomposition map $\dec_\rG$ is $D_\rG$-equivariant.

\itemth{c} The group $I_\rG$ acts trivially $\groth(k_R(\rG)\Hb)$.

\itemth{d} If $w$ and $w'$ are in the same Calogero-Moser $\rG$-cell, then 
$\dec_\rG (\LC_w)=\dec_\rG(\LC_{w'})$. 
\end{itemize}
\end{lem}

\begin{proof}
Since $R$ is integrally closed, saying that the $k_R(\rG)$-algebra $k_R(\rG)\Hb$ is split 
is equivalent to say that $(R,\rG)$ satisfies $\propdec$. 
The decomposition maps being computed by reduction of the characteristic polynomials, 
the statement~(b) is immediate. The group $I_\rG$ acting trivially on $k_R(\rG)$ by definition, 
(c) is clear. The statement~(d) then follows from~(b) and~(c) because the 
Calogero-Moser $\rG$-cells are $I_\rG$-orbits.
\end{proof}

\bigskip

Lemma~\ref{lem:dec-constant} says that, when restricted to an $\rG$-block, 
the decomposition map $\dec_\rG$ has rank $1$.

\bigskip

\begin{exemple}\label{exemple:lissite-decomposition}
Set $\pG = \rG \cap P$ and assume in this example, and only in this example, 
that $\zG_\singulier \cap P \not\subset \pG$. 
Then Theorem~\ref{lissite et simples}(a) implies that 
the $k_R(\rG)$-algebra $k_R(\rG)\Hb$ is split. Consequently, the decomposition 
map $\dec_\rG : \groth(\Mb\Hb) \to \groth(k_R(\rG)\Hb)$ 
is well-defined. Since the simple modules of $\Mb\Hb$ have dimension $|W|$, 
as well as the simple $k_R(\rG)\Hb$-modules, the decomposition map sends 
the isomorphism class of a simple $\Mb\Hb$-module on the isomorphism class of a simple 
$k_R(\rG)\Hb$-module. So $\dec_\rG$ defines a surjective map 
\equat\label{dec irr}
\dec_\rG : W \longto \Irr(k_R(\rG)\Hb)
\endequat
whose fibers are the Calogero-Moser $\rG$-cells 
(see Lemma~\ref{lem:dec-constant}).\finl 
\end{exemple}

\bigskip

\begin{rema}\label{rema:gauche-droite-decomposition}
The previous example can be applied in the case where 
$\rG=\rG_\CG^\gauche$ or $\rG_\CG^\droite$, thanks to Theorem~\ref{theo:gauche}(a).\finl
\end{rema}

%
%

\bigskip

\section{Left, right, two-sided cells and decomposition maps}

\medskip

In order to define decomposition maps, one must check that 
some assumptions are satisfied (see the previous conditions (D1), (D2) and (D3)). 
It is the aim of the next proposition to check that these assumptions 
hold in the cases we are interested in:

\bigskip

\begin{prop}\label{prop:d1-d2-d3}
Let $\rG$ be a prime ideal of $R$ amongst $\rG_\CG$, $\rG_\CG^\gauche$, $\rG_\CG^\droite$ or 
$\rGba_\CG$. Then:
\begin{itemize}
\itemth{a} The $k_R(\rG)$-algebra $k_R(\rG)\Hb$ is split.

\itemth{b} Assume here that $\rG \neq \rGba_\CG$ or $\CG=0$ or 
	$\CG=\CG_c$ for some $c \in \CCB$. 
If $\LC$ is a simple $k_R(\rG)\Hb$-module and if $h \in \Hb/\rG\Hb=(R/\rG)\Hb$, then 
the characteristic polynomial of $h$ (for its action on $\LC$) has coefficients 
in $R/\rG$.
\end{itemize}
\end{prop}

\begin{proof}
(a) has been proven for $\rG=\rG_c$ in Example~\ref{subsection:specialization cellules}, 
for $\rG=\rG_\CG^\gauche$ or $\rG=\rG_\CG^\droite$ in Theorem~\ref{theo:gauche}(a) and for 
$\rG=\rGba_\CG$ in Proposition~\ref{pr:verma}.

\medskip

Let us now show (b). First of all, if $\rG=\rG_\CG$ or $\rG_\CG^\gauche$ or $\rG_\CG^\droite$, 
then the images in the Grothendieck group $\groth(k_R(\rG)\Hb)$ of simple $k_R(\rG)\Hb$-modules 
are the images of simple $\Mb\Hb$-modules through the decomposition map 
(see Example~\ref{exemple:lissite-decomposition} and Remark~\ref{rema:gauche-droite-decomposition}). 
So, if $h$ is the image in $\Hb/\rG\Hb$ of $h' \in \Hb$, then the characteristic polynomial 
of $h'$ acting on a simple $\Mb\Hb$-module has coefficients in $R$ (because $R$ 
is integrally closed) and so the characteristic polynomial of $h$ has coefficients 
in $R/\rG$ (it is the reduction modulo $\rG$ of the one of $h$).

Now, if $\rG=\rGba$ or $\rGba_c$, then the simple $k_R(\rG)\Hb$-modules 
are obtained by scalar extension from the simple $k_P(\rG \cap P)\Hb$-modules, 
and the result follows from the fact that $P/\pGba \simeq \kb[\CCB]$ and $P/\pGba_c \simeq \kb$ 
is integrally closed.
\end{proof}

\bigskip

Taking Proposition~\ref{prop:d1-d2-d3} into account, we can define decomposition maps giving rise to a commutative diagram
$$\xymatrix{
	&	\BZ W\ar[d]^\sim \\
	& \groth(\Mb\Hb) \ar[dl]_{\dec_\CG^\gauche} \ar[dd]_{\decba_\CG}
	\ar[dr]_{\dec_\CG^\droite}\ar[r]_\sim^{\dec_\CG} & 
	\groth(\Mb_\CG\Hb) \\
\groth(\Mb_\CG^\gauche \Hb^\gauche) \ar[dr]_{\decba_\CG^\gauche} &&
\groth(\Mb_\CG^\droite\Hb^\droite) \ar[dl]_{\decba_\CG^\droite}
\\
& \groth(\Mbov_\CG \Hbov) & \groth(\Mbov \Hbov) \ar[l]^{\dec_\CG^\res} \\
& \BZ \Irr(W)\ar[u]^\sim &
\BZ \Irr(W)\ar[u]_\sim\ar[l]
}$$

\indexnot{d}{\dec_\CG,\dec_\CG^\gauche,\dec_\CG^\droite\dec_\CG^\res}
As usual, the index $\CG$ will be omitted if $\CG=0$ or will be replaced by 
$c$ if $\CG=\CG_c$ (for some $c \in \CCB$). Recall that $\dec_\CG$ is an isomorphism 
(by Example~\ref{subsection:specialization cellules}, 
which extends easily to the case where $\CG_c$ is replaced by any prime ideal $\CG$ of $\kb[\CCB]$)  
and that  
$$\groth(\Mb_\CG\Hb) \simeq \BZ W\qquad\text{and}\qquad \groth(\Mbov_\CG\Hbov) \simeq \BZ \Irr(W).$$
Note however that 
$\dec_\CG^\res : \groth(\Mbov\hskip1mm \Hbov) \simeq 
\BZ \Irr(W) \longto \groth(\Mbov_\CG\Hbov) \simeq \BZ \Irr(W)$ 
is not an isomorphism in general. 
Some transitivity formulas follow from~\ref{transitivite decomposition}.

\section{Isomorphism classes of baby Verma modules}
\label{se:isobabyVerma}

\medskip

The Verma modules $\Delta(\chi)$ being defined over the ring $P$, the fundamental 
properties of decomposition maps show that 
\equat\label{eq:dec M}
\decba_\CG^\gauche \isomorphisme{\Mb_\CG^\gauche\Delta(\chi)}_{\Mb_\CG^\gauche\Hb^\gauche}=
\isomorphisme{\Mbov_\CG\bar{\Delta}(\chi)}_{\Mbov_\CG\Hbov}.
\endequat
The multiplicities $\mult_{C,\chi}^\calo$ are defined from the image of 
$\Mb_\CG^\gauche\Delta(\chi)$ in the Grothendieck group 
$\groth(\Mb_\CG^\gauche\Hb^\gauche)$. We will now be interested to the image of 
$\Mbov_\CG\bar{\Delta}(\chi)$ in the Grothendieck group $\groth(\Mbov_\CG\Hbov)$:

Fix now a Calogero-Moser two-sided $\CG$-cell $\G$ and set 
$L_\CG(\G)=\decba_\CG \isomorphisme{\LC_w}_{\Mb\Hb}$, for $w \in \G$. 
Note that $L_\CG(\G)$ does not depend on the choice of $w \in \G$ by Lemma~\ref{lem:dec-constant}.

\bigskip 

\begin{prop}\label{prop:verma-rang-1}
If $\chi \in \Irr_\G^\calo(W)$, then 
$$\isomorphisme{\Mbov_\CG\bar{\Delta}(\chi)}_{\Mbov_\CG\Hbov} = \chi(1) L_\CG(\G).$$
\end{prop}

\medskip

\begin{rema}\label{rema:dec-rang-1}
Proposition~\ref{prop:verma-rang-1} says that, inside a given Calogero-Moser $\CG$-family, 
the decomposition matrix of baby Verma modules in the basis 
of simple modules has rank $1$: this was conjectured by U. Thiel~\cite{thiel}.

A similar property holds for restricted enveloping algebras or 
Lusztig's small quantum groups at a root of 
unity~\cite[Prop.~4.16,~(1)~and~(2)]{BelTh2}.\finl
\end{rema}

\medskip

\begin{proof}
Let $C$ be a Calogero-Moser left $\CG$-cell. Then $\mult_{C,\chi}^\calo=0$ if $C$ is not contained 
in $\G$ (see Proposition~\ref{multiplicite cm}(c)). Hence, by~(\ref{eq:dec M}), we have 
$$\isomorphisme{\Mbov_\CG\bar{\Delta}(\chi)}_{\Mbov_\CG\Hbov}=\sum_{\substack{C \in \cmcellules^\CG_L(W) \\ C \subset \G}} 
\mult_{C,\chi}^\calo \cdot \decba^\gauche \isomorphisme{L_\CG(C)}_{\Mb^\gauche_\CG\Hb^\gauche}.$$
But, if $C \subset \G$, then 
$\isomorphisme{L_\CG(C)}_{\Mb^\gauche_\CG\Hb^\gauche}=\dec_\CG^\gauche \isomorphisme{\LC_w}_{\Mb\Hb}$ 
where $w \in C$. Then, by the transitivity of decomposition maps, we have 
$$\decba^\gauche \isomorphisme{L_\CG(C)}_{\Mb^\gauche_\CG\Hb^\gauche} = L_\CG(\G)$$
(by Lemma~\ref{lem:dec-constant}). The result follows now
from Proposition~\ref{multiplicite cm}(a).
\end{proof}

\bigskip

We conclude with a result comparing the Calogero-Moser $\CG$-cellular characters for different 
prime ideals $\CG$:

\bigskip

\begin{prop}\label{prop:cellulaire-semicontinu}
Let $\CG'$ be a prime ideal of $\kb[\CCB]$ contained in $\CG$ and choose a prime ideal 
$\rG_{\CG'}^\gauche$ lying over $\qG_{\CG'}^\gauche$ and contained in $\rG_\CG^\gauche$. 
Let $C$ be a Calogero-Moser left $\CG$-cell and let us write 
$C=C_1 \coprod \cdots \coprod C_r$, where the $C_i$'s are Calogero-Moser left 
$\CG'$-cells (see Proposition~\ref{semicontinuite gauche}). Then 
$$\isomorphisme{C}_{\CG}^\calo = \isomorphisme{C_1}_{\CG'}^\calo + \cdots + \isomorphisme{C_r}_{\CG'}^\calo.$$
\end{prop}

\begin{proof}
By Proposition~\ref{prop:d1-d2-d3}, the decomposition map 
$\db : \groth(\Mb_{\CG'}^\gauche\Hb^\gauche) \longto \groth(\Mb^\gauche_\CG \Hb^\gauche)$ 
is well-defined and it satisfies the transitivity property 
$\db \circ \dec_{\CG'}^\gauche = \dec_\CG^\gauche$. 
Moreover, we have 
$$\db~\isomorphisme{\Mb_{\CG'}^\gauche\Delta(\chi)}_{\Mb_{\CG'}^\gauche\Hb^\gauche} = 
\isomorphisme{\Mb_\CG^\gauche\Delta(\chi)}_{\Mb_\CG^\gauche\Hb^\gauche}.$$
The result then follows from the fact that $\db~\isomorphisme{L_{\CG'}(C_i)}_{\Mb_{\CG'}^\gauche\Hb^\gauche}=
\isomorphisme{L_\CG^\gauche(C)}_{\Mb_{\CG}^\gauche\Hb^\gauche}$ for all $i$ 
(see Example~\ref{exemple:lissite-decomposition}).
\end{proof}

\bigskip

\chapter{Calogero-Moser versus Kazhdan-Lusztig} \label{part:coxeter}

We have recalled in \S~\ref{section:cellules-kl} the definition of 
Kazhdan-Lusztig left, right or two-sided $c$-cells, of Kazhdan-Lusztig $c$-families, 
and of Kazhdan-Lusztig $c$-cellular characters, starting from the representation theory 
of Hecke algebras. On the other hand, the notions of Calogero-Moser left, right or 
two-sided $c$-cells, of Calogero-Moser $c$-families and of Calogero-Moser 
$c$-cellular characters have been defined and studied in Part~\ref{part:verma} 
of this book. We conjecture that these notions coincide. 
The aim of this chapter is to state precise conjectures and to give arguments 
which support these conjectures.

\section{Hecke families}\label{chapter:hecke}

The aim of this section is to recall the statement of Martino's Conjecture~\cite{martino} 
which relates Calogero-Moser families and Hecke families (see definition~\ref{defi:famille-rouquier}), 
to recall what is known about this conjecture, and to show some theoretical arguments 
which support it.

\bigskip

\boitegrise{
{\it 
Let $k^\sharp=(k_{\orbite,j}^\sharp)_{(\orbite,j) \in \orbiteb^\circ}$ denote the element 
of $\CCB$ defined by $k_{\orbite,j}^\sharp=k_{\orbite,-j}$ (the indices $j$ being viewed modulo $e_\orbite$). 
{\bfit We assume that Assumption~\libertesymetrie~ is satisfied} 
(see~\S\ref{sub:hecke}).}}{0.75\textwidth}

\bigskip

\subsection{Statement and known cases}

\medskip

We recall here the  statement given in~\cite[Conjecture~2.7]{martino}:

\bigskip

\begin{conjecturem}[Martino]
If $b \in \blocs(\Zba_c)$, then there exists a central idempotent 
$b^\HC$ of $\OC^\cyclo[\qb^\RM]\HC_W^\cyclo(k^\sharp)$ such that:
\begin{itemize}
\itemth{a} $\Irr_\Hb(W,b)=\Irr_\HC(W,b^\HC)$;

\itemth{b} $\dim_\CM(\Zba b) = \dim_{F(\qb^\RM)}\bigl(F(\qb^\RM)\HC_W^\cyclo(k^\sharp)b^\HC\bigr)$.
\end{itemize}
In particular, every Calogero-Moser $c$-family is a union of Hecke $k^\sharp$-families. 
\end{conjecturem}

\bigskip

This Conjecture has been checked in many cases by computing separately the Calogero-Moser families 
and Hecke families. At the time this book is written, no theoretical link has been made 
towards a proof of this Conjecture which does not rely on the Shephard-Todd classification.

\bigskip

\begin{theo}[Bellamy, Chlouveraki, Gordon, Martino]
Assume that $W$ has type $G(de,e,n)$ and assume that, if $n=2$, then $e$ is odd or $d=1$. 
Then the Conjecture~\FAM~holds.
\end{theo}

\bigskip

The proof of this Theorem follows from the following works:
\begin{itemize}
\item M. Chlouveraki has computed the Hecke families in~\cite{chlouveraki B} and~\cite{chlouveraki D}.

\item Whenever $e=1$, the Calogero-Moser families have been computed by I. Gordon~\cite{gordon B} 
for rational values of $k$ (using Hilbert schemes). This result has been extended to all values 
of $k$ by M. Martino~\cite{martino 2} using purely algebraic methods. 

\item M. Chlouveraki's combinatoric and I. Gordon's combinatoric have been compared by 
M. Martino~\cite{martino} to show that Conjecture~\FAM~holds whenever $e=1$.

\item Whenever $e$ is any non-negative integer (satisfying the conditions of the Theorem), 
the Calogero-Moser families have been computed by~\cite{bellamy} 
for rational values of $k$, because this computation relied on I. Gordon's result. 
His method can nevertheless be extended to any value of $k$, once M. Martino's result has been 
established~\cite{martino 2}.
\end{itemize}

\bigskip

It was also conjectured by M. Martino that, whenever $c$ is generic, then the Calogero-Moser 
$c$-families and the Hecke $k^\sharp$-families coincide. 
A counter-example has been found by U. Thiel~\cite{thiel}. 


U. Thiel has also obtained many cases of Conjecture~\FAM~amongst 
the exceptional complex reflection groups~\cite[Theorem~25.4]{thiel thesis}. 
His algorithm has recently been improved by Thiel and the first author~\cite{bonnafe thiel} 
and new cases have been settled~\cite[Theorem~5.15]{bonnafe thiel}. 
It must be noticed that M. Chlouveraki has computed the partitions into Hecke families 
for exceptional groups~\cite{chlouveraki LNM} in all cases, while 
the partition into Calogero-Moser families is known only 
for some exceptional groups, and mainly in the generic parameter case. 
Comparison of both sides gives the following result (note that 
$G_{23}=\Wrm(H_3)$ and $G_{28}=\Wrm(F_4)$):

\bigskip

\begin{theo}[Thiel]
If $W$ has type $G_n$ with 
$$n \in \{4, 5, 6, 7, 8, 9, 10, 11, 12, 13, 14, 15, 20, 
22, 23, 24, 25, 26, 27, 28$$
and $c$ is any parameter (except for $n=20$, where it is assumed generic), 
then the Conjecture~\FAM~holds.
\end{theo}

\bigskip

\subsection{Theoretical arguments} 

\medskip

Corollary~\ref{dim bonne} shows that:

\bigskip

\begin{prop}\label{prop:a-b}
In Conjecture~\FAM, the statement (a) implies the statement (b).
\end{prop}

\bigskip

\begin{proof}
Keep the notation of Conjecture~\FAM~($b$, $b^\HC$,...). 
Since the algebra $F(\qb^\RM)\HC_W^\cyclo(k^\sharp)$ is split semisimple, 
we have
$$\dim_{F(\qb^\RM)}\bigl(F(\qb^\RM)\HC_W^\cyclo(k^\sharp)b^\HC\bigr)=
\sum_{\chi \in \Irr_\HC(W,b^\HC)} \chi(1)^2.$$
But, on the other hand, it follows from Corollary~\ref{dim bonne} that
$$\dim_\CM(\Zba_c b)=\sum_{\chi \in \Irr_\Hb(W,b)} \chi(1)^2.$$
Whence the result.
\end{proof}

\bigskip

\begin{rema}\label{rema:c-constant}
An important result supporting Conjecture~\FAM~is the following. 
It has been proven that, if $\chi$ and $\chi'$ are in the same Calogero-Moser $c$-family 
(respectively Hecke $k^\sharp$-family), then $\O_\chi^c(\euler)=\O_{\chi'}^c(\euler)$ 
(respectively $C_\chi(k^\sharp)=C_{\chi'}(k^\sharp)$): see Lemma~\ref{caracterization blocs CM} 
(respectively Lemma~\ref{lem:c-contant}). But it follows from Corollary~\ref{action euler verma} 
and from the definition of $C_\chi(k^\sharp)$
that 
\equat\label{eq:o=c}
\O_\chi^c(\euler_c)=C_\chi(k^\sharp).
\endequat
Even though this numerical invariant is not enough for determining 
in general the Calogero-Moser families, it is relatively sharp.\finl
\end{rema}

\bigskip

A last argument is given by the next proposition, which follows from
Lemma~\ref{lem:rouquier-ordre-2} 
and Corollary~\ref{ordre 2}:

\bigskip

\begin{prop}\label{prop:ordre 2}
If $\FC$ is a Calogero-Moser 
$c$-family (respectively a Hecke $k^\sharp$-family), then $\FC \e$ is a 
Calogero-Moser $c$-family (respectively a Hecke $k^\sharp$-family).
\end{prop}

\cbstart

\section{Kazhdan-Lusztig cells}\label{section:conjectures}

\bigskip
\boitegrise{{\bf Assumption.} {\it In \S \ref{section:conjectures},
we assume that 
$W$ is a Coxeter group, that $\kb=\CM$ and that $\kb_\RM=\RM$. We fix an element 
$c \in \CCB(\RM)$, we set $k=\kappa(c) \in \KCB(\RM)$ and 
recall that we set $k_s=-c_s/2$ for all $s \in \REF(W)$.}}{0.75\textwidth}

\bigskip

\subsection{Cells and characters} 
\label{se:cellsandcharacters}

\medskip

The first conjecture is concerned with two-sided cells and their associated families.

\bigskip

\begin{conjecturebil}
There exists a choice of the prime ideal $\rGba_c$ lying over 
$\qGba_c$ such that:
\begin{itemize}
\itemth{a} The partition of $W$ into Calogero-Moser two-sided $c$-cells
coincides with the partition into Kazhdan-Lusztig two-sided $c$-cells.

\itemth{b} Assume that $k_s \ge 0$ for all $s \in \REF(W)$. 
If $\G \in \cmcellules_{LR}^c(W)=\klcellules_{LR}^c(W)$, then $\Irr_\G^\calo(W)=\Irr_\G^\kl(W)$. 
\end{itemize}
\end{conjecturebil}

\bigskip

We propose a similar conjecture for left cells and cellular characters.

\bigskip

\begin{conjectureleft}
There exists a choice of the prime ideal $\rG^\gauche_c$ lying over 
$\qG^\gauche_c$ such that:
\begin{itemize}
\itemth{a} The partition of $W$ into Calogero-Moser left $c$-cells coincides
with the 
partition into Kazhdan-Lusztig left $c$-cells.

\itemth{b} Assume that $k_s \ge 0$ for all $s \in \REF(W)$. 
If $C \in \cmcellules_L^c(W)=\klcellules_L^c(W)$, then $\isomorphisme{C}_c^\calo=\isomorphisme{C}_c^\kl$. 
\end{itemize}
\end{conjectureleft}

\bigskip

A similar conjecture can be stated for right cells. 
Also, if Conjectures~\BIL~and~\GAUCHE~have positive answer, it should be
true that $\rG_c^\gauche \subset \rGba_c$. 

\bigskip
We propose a specific choice of ideals $\rG_0$, $\rGba_c$ and $\rG^\gauche_c$.
We describe that choice in the equivalent setting of paths, cf.~\S\ref{se:cellstopology}.

Let $C_\RM^*$ be the dual chamber to $C_\RM$, obtained as the
image of $C_\RM$ through
some isomorphism of $\RM W$-modules $V_\RM\xrightarrow{\sim}V_\RM^*$.
We choose $(v_\CM,v^*_\CM)\in C_\RM\times C_\RM^*$ and we choose the
path $\gamh$ contained in $\CCB(\RM) \times \overline{C_\RM} \times \overline{C_\RM^*}$ 
and such that $\gamh([0,1[) \subset \CCB(\RM) \times C_\RM \times C_\RM^*$ and 
we denote by $\g$ its image in $\CCB \times V^\reg/W \times V^*/W$.
We conjecture that the image of $\CCB(\RM) \times C_\RM\times C_\RM^*$ 
in $\CCB \times V^\reg/W \times V^*/W$ is contained in $\RCB(\CM)^{\nr}$ 
so that we can define Calogero-Moser $\g$-cells as in Theorems~\ref{th:cellsGaudin} 
and~\ref{theo:cm-cells-topo}. Moreover, as $\CCB(\RM) \times C_\RM\times C_\RM^*$ is 
simply connected, Calogero-Moser $\g$-cells do not depend on the particular 
choice of $\gamh$. We then conjecture that the Conjectures LR and L hold with such choices.

\bigskip

\subsection{Characters}

\medskip

%
%
%
%
%
%

First of all, note that the set of Calogero-Moser $c$-families, as well as the set 
of Calogero-Moser $c$-cellular characters, do not depend
on the choice of the ideal $\rG_c^\gauche$. 
At the level of characters, the statements~(b) of Conjectures~\BIL~and~\GAUCHE~ 
imply the following simpler statement, 
which does refer to the choice of a prime ideal of $R$.

\bigskip

\begin{conjecturecar}
\begin{itemize}
\itemth{a} The partition of $\Irr(W)$ into Calogero-Moser $c$-families
coincides with the partition 
into Kazhdan-Lusztig $c$-families (Gordon-Martino).

\itemth{b} The set of Calogero-Moser 
$c$-cellular characters coincides with the set of Kazhdan-Lusztig $c$-cellular 
characters.
\end{itemize}
\end{conjecturecar}

\bigskip

Note that (a) above has been conjectured by Gordon and
Martino~\cite[Conjecture~1.3(1)]{gordon martino}. So 
Conjecture~\BIL~lifts Gordon-Martino's Conjecture 
at the level of two-sided cells.

\bigskip

\noindent{\sc Commentary - } 
The choices of the prime ideals $\rGba_c$ or $\rG_c^\gauche$ are not relevant for the 
Conjecture \CAR, but they are relevant for the Conjectures \BIL~and \GAUCHE.\finl

\section{Evidence}\label{chapter:arguments}

As will be explained in Part~\ref{part:exemples}, the conjectures stated in
\S~\ref{section:conjectures} hold when $W$ has type $A_1$ or $B_2$: they 
also hold in type $A_2$, but we have not included the computations in this book. 
The case of type $B_2$ will be treated in \S~\ref{chapitre:b2}. However,
the difficulty 
of the computations does not allow us for now to extend 
this list of examples. Note that Conjectures~LR and~L have been proved 
by the first author whenever $W$ is dihedral of order $2m$, 
with $m$ odd~\cite[Corollary~6.3]{bonnafe diedral} (it turns out that, 
in this case, the Galois group $G$ is $\SG_W$). The prime ideals involved
in the conjectures have not been determined. However, in the case of dihedral 
groups at equal parameters, Conjectures~L and~LR have been solved by Germoni 
and the first author~\cite{bonnafe germoni} with the specific choice of ideals 
as discussed above.

\medskip
Again with this specific choice of ideals, 
the best evidence for Conjecture L is the following result~\cite{BGW}:

\begin{theo}[Brochier-Gordon-White]
Conjecture $L$ holds for $W$ of type $A$.
\end{theo}

The article \cite{BGW} gives a new description of the 
Robinson-Schensted-Knuth correspondence, a bijection from matrices with 
non-negative integer entries to pairs of semistandard Young tableaux. This is realized using the collision
of spectra of Gaudin operators acting on sections of a vector bundle over
$\CM^n$ with fiber $\CM[x_1,\ldots,x_r]^{\otimes n}$, via
a combinatorial description using $\mathfrak{gl}_r$-crystals based on
\cite{HaKaRyWe}. The
representation $(\CM^r)^{\otimes n}$ of $\mathfrak{gl}_r$ is a direct summand of that
representation, and the collision pattern of Gaudin operators on a vector
bundle with that fiber can be deduced. The conclusion is obtained by using some
version of Schur-Weyl duality \cite{MuTaVa1,MuTaVa3} that relates this system with the 
Calogero-Moser one.

\medskip
The aim of \S\ref{chapter:arguments}
is to give some evidence in
support of these conjectures. Note however that Conjecture~\CAR, 
which only deals with characters (and not with the partition of $W$ into cells), 
holds for some infinite series of groups (see the details below).

\bigskip

\subsection{The case ${\boldsymbol{c=0}}$}

\medskip

The following facts will be shown in \S~\ref{chapitre nul}.

\bigskip

\begin{prop}\label{lem:nul}
When $c=0$, there is only one Calogero-Moser left, right or two-sided cell: it
is $W$ itself. Moreover,
$$\Irr_W^\calo(W)=\Irr(W)\qquad\text{and}\qquad 
\isomorphisme{W}_0^\calo=\isomorphisme{\CM W}_{\CM W} = \sum_{ \chi \in \Irr(W)} \chi(1) \chi.$$
\end{prop}

\bigskip

\begin{coro}\label{coro:nul}
Conjectures~\GAUCHE~and~\BIL~hold when $c=0$.
\end{coro}

\begin{proof}
This follows from the comparison of~\cite[Corollaries~2.13~et~2.14]{bonnafe continu} 
with Proposition~\ref{lem:nul}.
\end{proof}

\bigskip

\subsection{Constructible characters, Lusztig families} 

\medskip

In the sequel of this Chapter, we will only deal with positive parameters.
We do not 
know how to treat the case where only some parameters are equal to $0$ (in order 
to compare with~\cite[Corollaries~2.13~et~2.14]{bonnafe continu}). 

\bigskip

\boitegrise{From now on, and until the end \S\ref{part:coxeter}, we assume
that $k_s > 0$ for all $s \in \REF(W)$.}{0.75\textwidth}

\bigskip

\noindent{\sc Convention - } When $(W,S)$ has type $B_n$, we write 
$S=\{t,s_1,s_2,\dots,s_{n-1}\}$ with the convention that $t$ is not conjugate 
to some $s_i$, so the Dynkin diagram is 
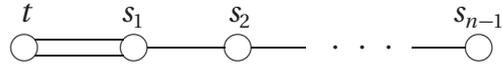
\begin{center}
\begin{picture}(220,30)
\put( 40, 10){\circle{10}}
\put( 44,  7){\line(1,0){33}}
\put( 44, 13){\line(1,0){33}}
\put( 81, 10){\circle{10}}
\put( 86, 10){\line(1,0){29}}
\put(120, 10){\circle{10}}
\put(125, 10){\line(1,0){20}}
\put(155,  7){$\cdot$}
\put(165,  7){$\cdot$}
\put(175,  7){$\cdot$}
\put(185, 10){\line(1,0){20}}
\put(210, 10){\circle{10}}
\put( 38, 20){$t$}
\put( 76, 20){$s_1$}
\put(116, 20){$s_2$}
\put(200, 20){$s_{n{-}1}$}
\end{picture}
\end{center}
In this case, we will set $b=c_t$ and $a=c_{s_1}=c_{s_2}=\cdots = c_{s_{n-1}}$.\finl

\bigskip

Lusztig~\cite[\S{22}]{lusztig} has defined a notion of {\it constructible characters} of $W$ 
(that we will call here {\it $c$-constructible characters}).
We can then define a graph $\GC_c(W)$ 
as follows: 
\begin{itemize}
\item[$\bullet$] The set of vertices of $\GC_c(W)$ is $\Irr(W)$.

\item[$\bullet$] Two distinct irreducible characters of $W$ are linked in $\GC_c(W)$ if they 
appear in the same $c$-constructible character. 
\end{itemize}
We then define {\it Lusztig $c$-families} as the connected components of $\GC_c(W)$. 
Lusztig conjectures that Lusztig $c$-families coincide with Kazhdan-Lusztig $c$-families 
and that $c$-constructible characters coincide with Kazhdan-Lusztig $c$-cellular characters. 
This conjecture is proven in the following cases:

\bigskip

\begin{prop}\label{prop:cellulaire-constructible}
Assume that one of the following hold:
\begin{itemize}
\itemth{1} $c$ is constant;

\itemth{2} $|S| \le 2$;

\itemth{3} $(W,S)$ has type $F_4$;

\itemth{4} $(W,S)$ has type $B_n$, $a \neq 0$ and $b/a \in \{1/2,1,3/2,2\} \cup ]n-1, + \infty)$. 
\end{itemize}
Then:
\begin{itemize}
\itemth{a} The $c$-constructible characters and the Kazhdan-Lusztig $c$-cellular characters 
coincide. 

\itemth{b} The Lusztig $c$-families and the Kazhdan-Lusztig $c$-families coincide. 
\end{itemize}
\end{prop}

\begin{proof}
Lusztig~\cite[Conjectures~14.2]{lusztig} has proposed a series of conjectures 
(numbered P1, P2,\dots, P15) about Kazhdan-Lusztig cells and the
{\it $\ab$-function}. 
They have been proven in the following cases:
\begin{itemize}
\itemth{1} when $c$ is constant in~\cite[chapitre~15]{lusztig};

\itemth{2} when $|S| \le 2$ in~\cite[chapitre~17]{lusztig};

\itemth{3} when $(W,S)$ has type $F_4$ in~\cite{geck f4};

\itemth{4} when $(W,S)$ has type $B_n$ and $a=0$ or when $a \neq 0$ and $b/a \in \{1/2,1,3/2,2\}$ 
in~\cite[Chapter~16]{lusztig};

\itemth{4'} when $(W,S)$ has type $B_n$, $a \neq 0$ and $b/a > n-1$ 
in~\cite{bonnafe iancu},~\cite{bonnafe two} et~\cite{geck iancu}.
\end{itemize}
Also, it is shown in~\cite[Lemma~22.2]{lusztig} and~\cite[\S{6}~and~\S{7}]{geck plus} that 
these conjectures implies that the $c$-constructible characters and the Kazhdan-Lusztig $c$-cellular 
characters coincide. This shows~(a). The statement~(b) 
now follows from~\cite[Corollary~1.8]{bonnafe geck}.
\end{proof}

\bigskip

\subsection{Conjectures about characters} 

\medskip

\subsubsection*{Families} 
The $c$-constructible characters (and so the Lusztig $c$-families) 
have been computed in all cases by Lusztig~\cite{lusztig}. 
We deduce from Proposition~\ref{prop:cellulaire-constructible} that the Kazhdan-Lusztig 
$c$-families are known in the cases (1), (2), (3) and (4) of Proposition~\ref{prop:cellulaire-constructible}. 
But, the explicit computation of Calogero-Moser $c$-families 
has been made in all irreducible types except $H_4$, $E_6$, $E_7$ and $E_8$ 
in the series of 
articles~\cite{bellamy these, bellamy, gordon, gordon B, gordon martino, martino 2, bonnafe thiel}. 
In all cases, they coincide with Lusztig families. 
Proposition~\ref{prop:cellulaire-constructible} then implies the following theorem.

\bigskip
 
\begin{theo}\label{theo:gordon-martino-bellamy}
Assume that one of the following holds:
\begin{itemize}
\itemth{1} $|S| \le 2$.

\itemth{2} $(W,S)$ has type $A_n$ or $D_n$.

\itemth{3} $(W,S)$ has type $B_n$, $a > 0$ and $b/a \in \{1/2,1,3/2,2\} \cup ]n-1, + \infty)$. 
\itemth{4} $(W,S)$ has type $H_3$ or $F_4$.
\end{itemize}
Then Conjecture~\CAR(a) holds.
\end{theo}

\bigskip

\subsubsection*{Cellular characters} 
If $(W,S)$ has type $A_n$ or if $(W,S)$ has type $B_n$ with $a > 0$ and 
$b/a \in \{1/2,3/2\} \cup ]n-1, + \infty)$, then it follows from the previous results 
that the Kazhdan-Lusztig $c$-cellular characters are irreducible. 
Moreover, it follows from the work of Gordon and Martino that, 
in those same cases, the Calogero-Moser space $\ZCB_c$ is smooth. 
Therefore, the next theorem follows from Theorem~\ref{theo:cellulaire-lisse}.

\bigskip

\begin{theo}\label{theo:cellulaire-conjecture}
Assume that one of the following holds:
\begin{itemize}
\itemth{1} $(W,S)$ has type $A_n$;

\itemth{2} $(W,S)$ has type $B_n$ with $a \neq 0$ and 
$b/a \in \{1/2,3/2\} \cup ]n-1, + \infty)$.
\end{itemize}
Then Conjecture~\CAR(b) holds (and the $c$-cellular characters are irreducible).
\end{theo}

\bigskip

Whenever $W$ is dihedral, the first author proved Conjecture~\CAR(b), 
by a direct explicit computation using Gaudin algebras~\cite[Table~4.14]{bonnafe diedral}. 

\bigskip

%
%
%
%

\subsubsection*{Other arguments} 
First of all, note that, if we assume that Lusztig's Conjectures~P1, P2,\dots, P15 hold  
(see~\cite[conjectures 14.2]{lusztig}), then the previous argument imply that
Conjecture~\CAR(a) holds in type $B_n$ and Conjecture~\CAR(b) holds in type $B_n$ 
whenever $a > 0$ and $b/a \not\in \{1,2,\dots,n-1\}$ (because in this case, 
the $c$-constructible characters are irreducible and the Calogero-Moser space is smooth).



\bigskip

\begin{rema}\label{rema:epsilon}
If $\FC$ is a Calogero-Moser (respectively Kazhdan-Lusztig) $c$-family, then 
$\FC\e$ is a Calogero-Moser (respectively Kazhdan-Lusztig) $c$-family: 
see Corollary~\ref{ordre 2} and~(\ref{eq:familles-cellulaire-w0}).

Similarly, if $\chi$ is a Calogero-Moser (respectively Kazhdan-Lusztig) 
$c$-cellular character, then $\chi \e$ is a Calogero-Moser (respectively Kazhdan-Lusztig) 
$c$-cellular character: see Corollaries~\ref{coro:cellulaire-ordre-2} 
and~(\ref{eq:caractere-cellulaire-w0}).\finl
\end{rema}

\bigskip

\begin{rema}\label{rema:b-invariant}
If $\FC$ is a Calogero-Moser (respectively Lusztig) $c$-family, then there exists 
a unique character $\chi \in \FC$ with minimal $\bb$-invariant: 
see Theorem~\ref{dim graduee bonne}(b) (respectively~\cite{bonnafe b}, 
or~\cite[Theorem~5.25~and~its~proof]{lusztig orange} 
whenever $c$ is constant).

Similarly, if $\chi$ is a Calogero-Moser $c$-cellular character (respectively a 
$c$-constructible character),
then there exists a unique irreducible component $\chi$ of minimal $\bb$-invariant: 
	see Theorem~\ref{theo:b-minimal-cellulaire} (respectively~\cite{bonnafe b}).\finl
\end{rema}

\bigskip

\subsection{Cells} 

\medskip

\subsubsection*{Two-sided cells} 
The first argument which supports Conjecture~\BIL~ comes from the comparison 
of the cardinality of cells, and from the fact that Conjecture~\CAR(a) 
has been proven in many cases.

\bigskip

\begin{rema}\label{rem:cardinal-cellules-kl-cm}
Assume here $(W,c)$ satisfies one of the assumptions of Theorem~\ref{theo:gordon-martino-bellamy}. 
Let $\FC$ be a Calogero-Moser $c$-family 
(that is, a Kazhdan-Lusztig $c$-family according to Theorem~\ref{theo:gordon-martino-bellamy}). 
Let $\G_\calo$ (respectively $\G_\kl$) denote the Calogero-Moser (respectively Kazhdan-Lusztig) two-sided 
$c$-cell covering $\FC$. It follows from Theorem~\ref{theo cellules familles}(c)  that 
$$|\G_\calo| = \sum_{\chi \in \FC} \chi(1)^2$$
and it follows from~(\ref{eq:cardinal-cellule-kl})  that 
$$|\G_\kl|=\sum_{\chi \in \FC} \chi(1)^2.$$
Therefore, 
$$|\G_\calo|=|\G_\kl|.$$
This is not sufficient to show that $\G_\calo=\G_\kl$. However, this shows 
Conjecture~\BIL~whenever the Galois group $G$ is equal to $\SG_W$:
indeed, by replacing $\rGba_c$ by some $g(\rGba_c)$ for some 
$g \in G=\SG_W$, we can arrange that $\G_\calo=\G_\kl$ (for all families
$\FC$). 
This also shows the importance of making the choice of $\rGba_c$ precise in Conjecture~\BIL.\finl
\end{rema}

\bigskip

\begin{rema}\label{rema:bil-w0}
Let $\G_\calo$ (respectively $\G_\kl$) be a Calogero-Moser (respectively Kazhdan-Lusztig) 
two-sided $c$-cell. Let $w_0$ denote the longest element of $W$. Then:
\begin{itemize}
 \item By~(\ref{eq:sim-w0}) and~(\ref{eq:w0gw0}), 
$w_0 \G_\kl=\G_\kl w_0$ is a Kazhdan-Lusztig $c$-cell and 
$\Irr_{w_0\G_\kl}^\kl(W)=\Irr_{\G_\kl}^\kl(W) \e$.

\item Since all the reflections of $W$ have order $2$, it has been shown in Corollary~\ref{w0 epsilon} that, 
{\it if $w_0$ is central in $W$}, then  
$w_0 \G_\calo=\G_\calo w_0$ is a Calogero-Moser two-sided $c$-cell and 
$\Irr_{w_0\G_\calo}^\calo(W)=\Irr_{\G_\calo}^\calo(W) \e$.
\end{itemize}
These results show some analogy {\it whenever $w_0$ is central in $W$}. 
For the second statement, it is not reasonable to expect that it is true 
whenever $w_0$ is not central (as is shown by the type $A_2$) without making a judicious 
choice of $\rGba_c$.\finl
\end{rema}

\bigskip

\subsubsection*{Left cells} 
Let us recall that numerous Kazhdan-Lusztig left cells give rise to 
the same Kazhdan-Lusztig cellular character. On the Calogero-Moser side, Corollary~\ref{coro:dec-cellulaire} 
also shows that numerous Calogero-Moser left cells give rise 
to the same Calogero-Moser cellular character (see for instance Theorem~\ref{theo:cellulaire-lisse} 
in the smooth case).

\bigskip

\begin{rema}\label{rema:gauche-w0}
Let $C_\calo$ (respectively $C_\kl$) be a Calogero-Moser (respectively Kazhdan-Lusztig) 
left $c$-cell. Let $w_0$ denote the longest element of $W$. Then:
\begin{itemize}
 \item It follows from~(\ref{eq:sim-w0}) and~(\ref{eq:caractere-cellulaire-w0})  
that $w_0 C_\kl$ and $C_\kl w_0$ are Kazhdan-Lusztig left $c$-cells and that 
$\isomorphisme{w_0C_\kl}_c^\kl= \isomorphisme{C_\kl w_0}_c^\kl=\isomorphisme{C_\kl}_c^\kl \e$.

\item Since all the reflections of $W$ have order $2$, it follows from Corollary~\ref{coro:cellulaire-ordre-2} 
that, {\it if $w_0$ is central in $W$}, then  
$w_0 C_\calo=C_\calo w_0$ is a Calogero-Moser left $c$-cell and that 
$\isomorphisme{w_0C_\calo}_c^\calo= \isomorphisme{C_\calo w_0}_c^\calo=\isomorphisme{C_\calo}_c^\calo \e$.\finl
\end{itemize}
\end{rema}

\bigskip

\begin{rema}\label{rema: lignes-colonnes}
Note also the analogy between the following equalities: if $C$ is a Calogero-Moser (respectively 
Kazhdan-Lusztig) left $c$-cell and if $\chi \in \Irr(W)$, then  
$$
\begin{cases}
|C|=\DS{\sum_{\psi \in \Irr(W)} \mult_{C,\psi}^\calo \psi(1),}\\
~\\
\chi(1)=\DS{\sum_{C' \in \cmcellules_L(W)} \mult_{C',\chi}^\calo }\\
\end{cases}
$$
(respectively
$$
\begin{cases}
|C|=\DS{\sum_{\psi \in \Irr(W)} \mult_{C,\psi}^\kl \psi(1),}\\
~\\
\chi(1)=\DS{\sum_{C' \in \klcellules_L(W)} \mult_{C',\chi}^\kl \quad ).}\\
\end{cases}
$$
See Proposition~\ref{multiplicite cm} (respectively Lemma~\ref{lem:mult-kl}). 
It would be interesting to study if other numerical properties of Kazhdan-Lusztig left cells 
(as for instance~\cite[Lemma~4.6]{geck plus}) are also satisfied by Calogero-Moser left cells.\finl
\end{rema}

\bigskip
%
%
%

\bigskip

A final argument to support Conjectures~\GAUCHE~and~\BIL~ is the following.

\bigskip

\begin{theo}\label{theo:gauche-presque}
Assume that we are in one of the following cases:
\begin{itemize}
\itemth{1} $(W,S)$ has type $A_n$ and $c > 0$;

\itemth{2} $(W,S)$ has type $B_n$ with $a > 0$ and $b/a \in \{1/2,3/2\} \cup ]n-1,+\infty)$. 
\end{itemize}
Then there exists a bijective map $\ph : W \to W$ such that:
\begin{itemize}
\itemth{a} If $\G$ is a Kazhdan-Lusztig two-sided $c$-cell, then $\ph(\G)$ is a Calogero-Moser 
two-sided $c$-cell and $\Irr_\G^\kl(W)=\Irr^\calo_{\ph(\G)}(W)$.

\itemth{b} If $C$ is a Kazhdan-Lusztig left $c$-cell, then $\ph(C)$ is a Calogero-Moser left 
$c$-cell and $\isomorphisme{C}_c^\kl=\isomorphisme{\ph(C)}_c^\calo$.
\end{itemize}
\end{theo}

\begin{proof}
Under the assumptions~(1) or~(2), the Calogero-Moser space $\ZCB_c$ is smooth (see~\cite[Theorem~1.24]{EG} 
in case~(1) and~\cite[Lemma~4.3~and~its~proof]{gordon}) in case~(2)) and so 
Theorem~\ref{theo:cellulaire-lisse} can be applied to all the Calogero-Moser $c$-cells of $W$. 
The result follows now from a comparison of cardinalities of cells.
\end{proof}

\cbend

\bigskip

\chapter{Conjectures about the geometry of $\ZCB_c$}
\label{ch:conjgeometry}

\boitegrise{{\bf Assumption.} {\it We assume in this chapter that $\kb=\CM$.}}{0.75\textwidth}

\bigskip

\def\rees{{\mathrm{Rees}}}
\def\codim{{\mathrm{codim}}}

\section{Cohomology}

\medskip

We follow some of the notation of Appendix~\ref{appendice:filtration}. 
Given $i \in \ZM$, we set
$$(\CM W)^i = \bigoplus_{\substack{w \in W \\ \codim_\CM(V^w)=i}} \CM w,$$
so that $\CM W = \bigoplus_{i \in \ZM} (\CM W)^i$. This is of course not 
a grading on the algebra $\CM W$, but the filtration by the vector subspaces 
$(\CM W)^{\le i} = \bigoplus_{j \le i} (\CM W)_j$ induces a structure 
of filtered algebra on $\CM W$. 

If $A$ is any subalgebra of $\CM W$, it inherits a structure of filtered 
algebra by setting $A^{\le i}=A \cap (\CM W)^{\le i}$. 
As in Appendix~\ref{appendice:filtration}, we can then 
define its associated Rees algebra $\rees(A)$ (which is contained in 
$\CM[\hbar] \otimes A$), 
as well as its associated graded algebra $\grad(A)$. 

If $\XCB$ is a quasi-projective complex algebraic variety, we denote by 
$\Hrm^i(\XCB)$ its $i$-th singular cohomology group with coefficients in $\CM$. 
We denote by $\Hrm^{2 \bullet}(\XCB)$ the graded algebra $\bigoplus_{i \in \NM} \Hrm^{2i}(\XCB)$. 
The Euler characteristic of $\XCB$ will be denoted by $\chib(\XCB)$. 
If $\XCB$ is endowed with an algebraic action of $\CM^\times$, we denote by 
$\Hrm_{\CM^\times}^i(\XCB)$ its $i$-th equivariant cohomology group, 
with coefficients in $\CM$. 
We denote by $\Hrm_{\CM^\times}^{2 \bullet}(\XCB)$ the graded algebra 
$\bigoplus_{i \in \NM} \Hrm^{2i}_{\CM^\times}(\XCB)$: it will be viewed as a $\CM[\hbar]$-algebra 
by identifying $\Hrm^{2\bullet}_{\CM^\times}({\mathrm{pt}})$ with $\CM[\hbar]$ in the usual way. 

Recall that, given $c \in \CCB$, we have defined in~\S\ref{subsection:omega} 
a morphism of algebras $\Omeb^c : Z_c \longto \Zrm(\CM W)$. We propose 
the following conjectures about the (equivariant) cohomology of the 
variety $\ZCB_c$.

\bigskip

\begin{conjecturecoh}\label{conj:coh}
Let $c \in \CCB$.
\begin{itemize}
\itemth{1} If $i \in \NM$, then $\Hrm^{2i+1}(\ZCB_c)=0$.

\itemth{2} We have an isomorphism of graded algebras $\Hrm^{2 \bullet}(\ZCB_c) \simeq \grad(\im \Omeb^c)$.
\end{itemize}
\end{conjecturecoh}

\bigskip

\begin{conjectureecoh}\label{conj:ecoh}
Let $c \in \CCB$.
\begin{itemize}
\itemth{1} If $i \in \NM$, then $\Hrm^{2i+1}_{\CM^\times}(\ZCB_c)=0$.

\itemth{2} We have an isomorphism of graded $\CM[\hbar]$-algebras 
$\Hrm^{2 \bullet}_{\CM^\times}(\ZCB_c) \simeq \rees(\im \Omeb^c)$.
\end{itemize}
\end{conjectureecoh}

We refer to \cite[Conjecture 3.3]{bonnafe shan} for a more precise version of the
conjecture.

\bigskip

\begin{exemple}
Assume here that $c=0$. Recall
(Example \ref{ex:centerc=0})
that $\Hb_0 = \CM[V \times V^*] \rtimes W$ and that $Z_0=\CM[V \times V^*]^{\D W}$. 
In particular, $\ZCB_{\! 0}= (V \times V^*)/\D W$ and $\im(\Omeb^0)=\CM$. Therefore, 
$$\Hrm^i(\ZCB_{\!0})=\Hrm^i(V \times V^*)^W=
\begin{cases}
\CM & \text{if $i=0$,}\\
0 & \text{otherwise,}
\end{cases}
\quad\text{and}\quad
\Hrm^i_{\CM^\times}(\ZCB_{\!0})=\Hrm^i_{\CM^\times}(V \times V^*)^W \simeq \Hrm_{\CM^\times}^i({\mathrm{pt}}),$$
so Conjectures~\COH~and~\ECOH~hold.\finl
\end{exemple}

\bigskip

Given $E \in \Irr(W)$, we denote by $e_E \in \Zrm(\CM W)$ the corresponding primitive central idempotent 
(the unique one such that $\o_E(e_E)=1$). Given $\FC$ is a subset of $\Irr(W)$, we set 
$e_\FC=\sum_{E \in \FC} e_E$. It follows from Lemma~\ref{caracterization blocs CM} 
and from (\ref{eq:fixe-Z}) that
\equat\label{eq:im-omega}
\im \Omeb^c = \bigoplus_{p \in \ZCB_c^{\CM^\times}} \CM e_{\Th_c^{-1}(p)}.
\endequat
In particular, 
\equat\label{eq:dim-im-omega}
\dim_\CM(\im \Omeb^c) = |\ZCB_c^{\CM^\times}|
\endequat
hence
\equat\label{eq:euler-car}
\chib(\ZCB_c) = \dim_\CM(\im \Omeb^c).
\endequat
This is compatible with Conjecture~\COH.

\bigskip

\begin{theo}[Etingof-Ginzburg]
If $\ZCB_c$ is smooth, then Conjecture~\COH~holds.
\end{theo}

\bigskip

\begin{proof}
Assume that $\ZCB_c$ is smooth. 
In~\cite[Theorem~1.8]{EG}, Etingof and Ginzburg proved that this implies that 
$\ZCB_c$ has no odd cohomology, and that 
$$\Hrm^{2 \bullet}(\ZCB_c) \simeq \grad(\Zrm(\CM W)).$$
By Proposition~\ref{prop lissite}, the smoothness of $\ZCB_c$ implies 
that $\Th_c : \Irr(W) \to \ZCB_c^{\CM^{\times}}$ is bijective, so that 
$\im \Omeb^c=\Zrm(\CM W)$.
\end{proof}

\bigskip

Based on Etingof-Ginzburg's result, Peng Shan and the first author proved the following, 
using localization methods~\cite[Theorem~A]{bonnafe shan}:

\bigskip

\begin{theo}
If $\ZCB_c$ is smooth, then Conjecture~\ECOH~holds.
\end{theo}

\bigskip

Apart from the smooth case and the case $c=0$, both conjectures are known
only in rank $1$.
For Conjecture~\COH, see the upcoming Chapter~\ref{chapitre:rang 1}
(Theorem~\ref{theo:cohomologie-1}).
For Conjecture~\ECOH, see~\cite[Proposition~1.7]{bonnafe shan}.

\bigskip

\section{Fixed points}

\medskip

\boitegrise{{\bf Assumption and notation.} {\it Recall that $\NC=N_{\Gb\Lb_\CM(V)}(W)$. 
We fix in this section an element of finite order $\t \in \NC$. That
element $\t$ acts on $\CCB$ and $\ZCB$ and, 
if $c \in \CCB^\t$, it acts on $\ZCB_c$. We denote by $\ZCB^\t$ (respectively $\ZCB_c^\t$, for $c \in \CCB^\t$) 
the {\bfit reduced} closed subvariety of $\ZCB$ (respectively $\ZCB_c$) consisting 
of fixed points of $\t$ in $\ZCB$ (respectively $\ZCB_c$): it is an affine variety whose 
algebra of regular functions is $Z/\sqrt{\langle \t(z)-z, z \in Z\rangle}$ (respectively 
$Z_c/\sqrt{\langle \t(z)-z, z \in Z_c\rangle}$).
}}{0.75\textwidth}

\medskip

We say that a pair $(V',W')$ is a {\it reflection subquotient} of $(V,W)$ if 
$V'$ is a subspace of $V$ and if there exists a subgroup $N'$ of the stabilizer 
of $V'$ in $W$ such that $W'=N'/N_1'$ is a reflection group
on $V'$, where $N_1'$ is the kernel of the action 
of $N'$ on $V'$. In this case, we denote by $\REF(W')$ the set of reflections 
of $W'$ for its action on $V'$, by $\CCB(W')$ the vector space of maps 
$c' : \REF(W') \to \CM$ constant on $W'$-conjugacy classes and  by
$\ZCB(V',W')$ the Calogero-Moser space associated with $(V',W')$. 

\bigskip

\begin{conjecturefix}\label{conj:springer-fixed}
Given $\XCB$ an irreducible component of $\ZCB^\t$, 
there exists a reflection subquotient $(V',W')$ of $(V,W)$ and a linear map 
$\ph : \CCB^\t \to \CCB(V',W')$ 
such that 
$$\XCB \simeq \ZCB(V',W') \times_{\CCB(W')} \CCB^\t.$$
\end{conjecturefix}

\bigskip

\begin{rema}
Conjecture~\FIX~has been extended to symplectic leaves of $\ZCB^\t$ 
and made more precise by the second author~\cite{bonnafe autosymp}. In that
more general case, the closures of symplectic leaves are not always normal. We
	do not know if this can occur in Conjecture \ref{conj:springer-fixed}
	\ref{conj:springer-fixed}.\finl
\end{rema}

\def\smooth{{\!\mathrm{sm}}}

\begin{exemple}\label{ex:root of unity}
Assume in this example that $\t \in \CM^\times$ is a root 
of unity, acting on $V$ by scalar multiplication. Note that $\CCB^\t=\CCB$ 
in this case. Under this assumption, we will show in \S~\ref{sec:fixed-b2} that 
Conjecture~\FIX~holds when $W$ is of type $B_2$. It is shown by the first 
author~\cite[Theorem~7.1~and~Proposition~8.3]{bonnafe diedral} that it also holds 
if $W$ is of type $G_2$ 
(computer calculation using the {\tt MAGMA} software~\cite{magma}), 
or if $W$ is dihedral of order $2m$, with $m$ odd, and $\t$ is a primitive 
$m$-th root of unity. 

When $\t$ is a root of unity, it is shown 
in~\cite{bonnafe maksimau} that Conjecture~\FIX~holds 
if $W$ is the group denoted $G_4$ in Shephard-Todd classification 
(this uses again the {\tt MAGMA} software). The best result 
about Conjecture~\FIX~has been obtained by Ruslan Maksimau and the first 
author~\cite{bonnafe maksimau}. We need some notation to state it. 
Let $\CCB_\smooth$ denote the (open) subset of $\CCB$ consisting 
of elements $c \in \CCB$ such that $\ZCB_c$ is smooth (it can be
empty, see Theorem~\ref{les lisses}). We put
$\ZCB_\smooth = \ZCB \times_\CCB \CCB_\smooth$. 

\bigskip

\begin{theo}\label{theo:ruslan}
Let $\XCB$ be an irreducible component of $\ZCB_\smooth$ and assume that $\t$ is a root 
of unity. Then there exists a reflection subquotient $(V',W')$ of $(V,W)$ and a linear 
map $\ph : \CCB \to \CCB(V',W')$ such that $\ph(\CCB_\smooth) \subset \CCB_\smooth(V',W')$ 
and
$$\XCB \simeq \ZCB_\smooth(V',W') \times_{\CCB_\smooth(W')} \CCB_\smooth.$$
\end{theo}

\bigskip

Note also that the linear map $\ph$ of the above theorem is explicitly 
described in~\cite{bonnafe maksimau}.\finl
\end{exemple}

\begin{exemple}
	Conjecture~\FIX~
	also holds if $W$ has type $D_n$ or $I_2(m)$ and
	$\tau$ is the order $2$ diagram automorphism \cite{bonnafe autosymp,bonnafe dihedral2}.
\end{exemple}

\part{Examples}\label{part:exemples}

\chapter{Case ${\boldsymbol{c=0}}$}\label{chapitre nul}

\bigskip

\section{Two-sided cells, families}

\medskip

Recall that $R_+$ denotes the unique maximal bi-homogeneous ideal of $R$ and that 
$$R/R_+ \simeq \kb$$
(see Corollary~\ref{r0}). Recall also that $D_+$ (respectively $I_+$) denotes its 
decomposition (respectively inertia) group and that 
$$D_+=I_+=G$$
(see Corollary~\ref{r0 DI}). 

\bigskip

\begin{prop}\label{prop:rba0}
$R_+$ is the unique prime ideal of $R$ lying over $\pGba_0$.
\end{prop}

\begin{proof}
Indeed, $\pGba_0=P_+$ and so $R_+$ is a prime ideal of $R$ lying over $\pGba_0$: 
since it is stabilized by $G$, the uniqueness is proven.
\end{proof}

\bigskip

So if we denote by $\rGba_0$ the unique prime ideal of $R$ lying over $\pGba_0=P_+$, 
$\Dba_0$ its decomposition group and $\Iba_0$ its inertia group, then
\equat\label{exemple:dba-0}
\rGba_0=R_+\qquad\text{and}\qquad \Dba_0=\Iba_0=G.
\endequat
Hence:

\bigskip

\begin{coro}\label{coro:familles-nulles}
$W$ contains only one Calogero-Moser two-sided $0$-cell, namely $W$ itself, and
$$\Irr_W^\calo(W)=\Irr(W).$$
\end{coro}

\bigskip

A feature of the specialization at $0$ is that the algebra 
$\Hbov_0$ inherits the $(\NM \times \NM)$-grading, and so the $\NM$-grading. 
If we set
$$\Hbov_{0,+} = \mathop{\bigoplus}_{i \ge 1} \Hbov_0^\NM[i],$$
then $\Hbov_{0,+}$ is a nilpotent two-sided ideal of $\Hbov_0$ and, since 
$\Hbov_0^\NM[0] = \kb W$, we get the following result:

\bigskip

\begin{prop}\label{rad h0}
$\Rad(\Hbov_0) = \Hbov_{0,+}$ and $\Hbov_0/\Rad(\Hbov_0) \simeq \kb W$.
\end{prop}

\bigskip

In particular, 
\equat\label{l 0}
\isomorphisme{\LCov_{\Kbov_0}(\chi)}_{\kb W}^\grad = \chi \in \groth(\kb W)[\tb,\tb^{-1}]
\endequat
and
\equat\label{m 0}
\isomorphisme{\Kbov_0\MCov(\chi)}_{\Hbov_0}
= \chi(1)~ \isomorphisme{\kb W}_{\kb W} \in \BZ\Irr(W)\simeq \groth(\Hbov_0).
\endequat

\bigskip

\section{Left cells, cellular characters}

\medskip

Recall that, in~\S\ref{subsection:specialization galois 0}, we have fixed a prime ideal 
$\rG_0$ of $R$ lying over $\qG_0=\CG_0 Q$ as well as a field isomorphism 
$$\iso : \kb(V \times V^*)^{\D\Zrm(W)} \longiso \Mb_0=k_R(\rG_0)$$
whose restriction to $\kb(V \times V^*)^{\D W}$ is the canonical isomorphism 
$\kb(V \times V^*)^{\D W} \longiso \Frac(Z_0) \longiso \Lb_0$. Hence, 
$R/\rG_0 \subset \iso(\kb[V \times V^*]^{\D\Zrm(W)})$ and these two rings have the same fraction field, 
the field $\Mb_0$. Recall also that we do not know if these two rings are
equal, or equivalently, if
$R/\rG_0$ is integrally closed or not. 

\bigskip

\begin{prop}\label{prop:unicite-r0gauche}
There exists a unique prime ideal of $R$ lying over $\pG_0^\gauche$ and containing $\rG_0$.
\end{prop}

\begin{proof}
Let $\pG^*=\iso^{-1}(\pG_0^\gauche/\pG_0)$. Since
$\kb[V \times V^*]^{W \times W}/\pG^* \simeq \kb[V \times 0]^{W \times W}$,
there is only one prime ideal $\rG^*$ of $\kb[V \times V^*]$ lying over $\pG^*$: it is the 
defining ideal of the irreducible closed subvariety $V \times 0$ of $V \times V^*$. In other words, 
$$\kb[V \times V^*]/\rG^*=\kb[V \times 0].$$
Consequently, the unique prime ideal $\rG_0^\gauche$ of $R$ lying over $\pG_0^\gauche$ 
and containing $\rG_0$ is defined by 
$\rG_0^\gauche/\rG_0 = \iso(\rG^* \cap \kb[V \times V^*]^{\D\Zrm(W)}) \cap (R/\rG_0)$. 
\end{proof}

\bigskip

Let $\rG_0^\gauche$ be the unique prime ideal of $R$ lying over $\qG_0^\gauche$ and containing $\rG_0$ 
(see Proposition~\ref{prop:unicite-r0gauche}) and let 
$D_0^\gauche$ (respectively $I_0^\gauche$) denote its decomposition (respectively inertia) group. 

\bigskip

\begin{prop}\label{prop:d0-left}
\begin{itemize}
\itemth{a} $\iota(W \times W) \subset D_0^\gauche$ and $\iota(W \times 1) \subset I_0^\gauche$.

\itemth{b} The canonical map $\bar{\iota} : W \times W \to D_0^\gauche/I_0^\gauche$ is 
surjective and its kernel contains $W \times \Zrm(W)$.

\itemth{c} $D_0^\gauche/I_0^\gauche$ is a quotient of $W/\Zrm(W)$.

\itemth{d} If $R/\rG_0$ is integrally closed (i.e. if $R/\rG_0 \simeq \kb[V \times V^*]^{\D\Zrm(W)}$), then 
$\Ker(\bar{\iota}) = W \times \Zrm(W)$ and $D_0^\gauche/I_0^\gauche \simeq W/\Zrm(W)$.
\end{itemize}
\end{prop}

\begin{proof}
The first statement of~(a) follows from the uniqueness of $\rG_0^\gauche$ 
(see Proposition~\ref{prop:unicite-r0gauche}). For the second statement, let us use here 
the notation of the proof of Proposition~\ref{prop:unicite-r0gauche}, 
and note that $W \times 1$ acts trivially on $\kb[V \times V^*]/\rG^*$.

\medskip

Now, let $B_0$ be the inverse image of $R/\rG_0$ in $\kb[V \times V^*]$ through $\iso$. Then 
$\kb[V \times 0]^{W \times W} \subset B_0/\rG^* \subset \kb[V \times 0]^{\D \Zrm(W)}=
\kb[V \times 0]^{W \times \Zrm(W)} \subset \kb[V \times 0]$. (b), (c) and (d) follow 
from these observations.
\end{proof}

This study of decomposition and inertia groups allows us to deduce the
following result.

\bigskip

\begin{coro}\label{coro:0-cm}
$W$ contains only one Calogero-Moser left $0$-cell, namely $W$ itself, and
$$\isomorphisme{W}_0^\calo = \isomorphisme{\kb W}_{\kb W} = \sum_{\chi \in \Irr(W)} \chi(1).$$
\end{coro}

\begin{proof}
The first statement follows from Proposition~\ref{prop:d0-left}(a) whereas the second one follows from 
Proposition~\ref{multiplicite cm}(a).
\end{proof}

\bigskip

Let us conclude with an easy remark, which, combined with Proposition~\ref{prop:d0-left}, 
shows that the pair $(I_0^\gauche,D_0^\gauche)$ has a surprising behaviour.

\bigskip

\begin{prop}\label{prop:dc-d0}
Let $\CG$ be a prime ideal of $\kb[\CCB]$. Then there exists $h \in H$ such that 
$\lexp{h}{I_\CG^\gauche} \subset I_0^\gauche$.
\end{prop}

\begin{proof}
Let $\CGt$ denote the maximal homogeneous ideal of $\kb[\CCB]$ contained in 
$\CG$. By Proposition~\ref{lem:cellules-gauches-homogeneise}, we have 
$I_\CG^\gauche=I_\CGt^\gauche$. This means that we may assume that $\CG$ is homogeneous. 
In particular, $\CG \subset \CG_0$. So $\qG_\CG^\gauche \subset \qG_0^\gauche$ and there exists 
$h \in H$ such that $h(\rG_\CG^\gauche) \subset \rG_0^\gauche$. 
Therefore, $\lexp{h}{I_\CG^\gauche} \subset I_0^\gauche$. 
\end{proof}

\bigskip

It would be tempting to think, after Proposition~\ref{prop:d0-left}, that $D_0^\gauche=\iota(W \times W)$ and  
$I_0^\gauche=\iota(W \times \Zrm(W))$. However, this would contradict
Proposition~\ref{prop:dc-d0}, if we assume that
Conjectures~\BIL~and~\GAUCHE~hold: 
indeed, $I_0^\gauche$ must contain conjugates of subgroups admitting as orbits the left cells.
We will see in Chapter~\ref{chapitre:rang 1} 
that if $\dim_\kb(V)=1$, then $D_0^\gauche=G$. 

%
%
%
%
%
%
%
%
%
%
%
%

\chapter{Groups of rank 1}\label{chapitre:rang 1}

\bigskip
 
\boitegrise{{\bf Assumption and notation.} {\it In
\S\ref{chapitre:rang 1}, 
we assume that $\dim_\kb V=1$, we fix a non-zero element 
$y$ of $V$ and we denote by $x$ the element of $V^*$ such that $\langle y,x\rangle = 1$. 
We also fix an integer $d \ge 2$ and we assume that $\kb$ contains a primitive $d$-th root 
of unity $\z$. 
We denote by $s$ the automorphism of $V$ defined by $s(y)=\z y$, so that
$s(x)=\z^{-1} x$. We assume finally that $W=\langle s \rangle$: $s$ is a reflection and $W$ 
is cyclic of order $d$.}}{0.75\textwidth}

\bigskip

\section{The algebra $\Hbt$}\label{section:H rang 1}

\medskip

\subsection{Definition} 
We have $\REF(W)=\{s^i~|~1 \le i \le d-1\}$. 
Given $1 \le i \le d-1$, we denote by $C_i$ the indeterminate $C_{s^i}$, 
so that $\kb[\CCBt]=\kb[T,C_1,C_2,\dots,C_{d-1}]$ and
$\kb[\CCB]=\kb[C_1,C_2,\dots,C_{d-1}]$.
The $k[\CCBt]$-algebra $\Hbt$ is generated by $x,y,s$ with the relations
\equat\label{relation d}
sys^{-1}=\zeta y,\ sxs^{-1}=\zeta^{-1}x \text{ and }
[y,x] = T+ \sum_{1 \le i \le d-1} (\z^i-1) C_i~s^i.
\endequat
We set $C_0=C_{s^0}=0$. 
The hyperplane arrangement $\AC$ is reduced to one element, 
and $\AC/W$ as well (we write $\AC/W=\{\orbite\}$), 
we put $K_j=K_{\orbite,j}$ (for $0 \le j \le d-1$). 
Recall that the family $(K_j)_{0 \le j \le d-1}$ is determined by the relations 
\equat\label{relations K}
\forall~0 \le i \le d-1,~C_i=\sum_{j=0}^{d-1} \z^{i(j-1)} K_j.
\endequat
We put
$$K_{di+j}=K_j$$
for all $i \in \BZ$ and $j \in \{0,1,\dots,d-1\}$. Recall that 
$$K_0+K_1+\cdots + K_{d-1}=0.$$
The last defining relation for $\Hbt$ can be rewritten as
$$[y,x] = T+ d\sum_{0 \le i \le d-1} (K_{H,i}-K_{H,i+1})\varepsilon_i,$$
where $\varepsilon_i=d^{-1}\sum_{j=0}^{d-1}\zeta^{ij}s^j$.

\subsection{Differential operators on $\CM^\times$}
We have 
$$D_y=T\partial_y-x^{-1}\sum_{i=1}^{d-1}\zeta^iC_is^i=
T\partial_y-dx^{-1}\sum_{i=0}^{d-1}K_i\varepsilon_i.$$

\smallskip
Given $L$ a $\CM[\CCB][y]\rtimes W$-module, the $W$-equivariant connection on 
$\OC_{\CM^\times}\otimes L$ is given by
$$\nabla(p\otimes l)=\frac{dp}{dx}\otimes l+p\otimes y\cdot l
+dx^{-1} \sum_{i=0}^{d-1}K_ip\otimes \varepsilon_i l.$$
When $L=\CM[\CCB][y]/(y)\otimes \det^n=\CM[\CCB]$, we obtain
$\nabla=\partial+x^{-1}K_n$.

\bigskip

\subsection{The variety ${\boldsymbol{(V \times V^*)/\Delta W}}$}\label{subsection:quotient rang 1}
Let $X=x^d$ and $Y=y^d$. Recall that $\euler_0 = xy$. We have
$$\kb[V \times V^*]^{\Delta W} = \kb[X,Y,\euler_0]$$
and the relation 
\equat\label{relation center d}
\euler_0^d = XY
\endequat
holds. It is easy to check that this relation generates the ideal of relations. 

\bigskip

\section{The algebra ${\boldsymbol{Z}}$}\label{section:Q rang 1}

\medskip

Recall that $\euler=yx + \sum_{i=1}^{d-1} C_i~s^i$ (its image in 
$\Hb_0$ is $\euler_0$) and that $\e : W \to \kb^\times$ is the determinant,
characterized by $\e(s)=\z$. We have $\e^d=1$ and 
$$\Irr W=\{1,\e,\e^2,\dots,\e^{d-1}\}.$$
The image of the Euler element by $\O_\chi$ can be computed thanks to Corollary~\ref{action euler verma}: we have
\equat\label{euler cyclique}
\O_{\e^i}(\euler) = d K_{-i}
\endequat
for all $i \in \BZ$.

\medskip

Recall that $\pGba=\langle X,Y \rangle_P$ and that $\pGba Z \subset \Ker(\O_\chi)$ 
for all $\chi \in \Irr(W)$. More precisely, we have
\equat\label{inter omega}
\bigcap_{i=1}^d \Ker(\O_{\e^i}) = \pGba Z = \langle X,Y\rangle_Z.
\endequat
\begin{proof}
It follows from Example~\ref{lineaire} that the generic Calogero-Moser families 
are reduced to one element. Given $1 \le i \le d$, let $b_i$ denote the primitive 
central idempotent of $\kb(\CCB)\Hbov$ such that 
$\Irr \kb(\CCB)\Hbov b_i = \{L_{(0)}(\e^i)\}$. We have
$$\kb(\CCB) \Zba \simeq \prod_{i=1}^{d} \kb(\CCB)\Zba b_i$$
and, by Theorem~\ref{dim graduee bonne}, 
\equat\label{dim bi}
\dim_{\kb(\CCB)} \kb(\CCB)\Zba b_i = 1.
\endequat
Since $b_j$ is characterized by $\O_{\e^i}(b_j)=\d_{i,j}$ (the Kronecker symbol) for all 
$i \in \{1,2,\dots,d\}$, the equality~(\ref{inter omega}) follows.
\end{proof}

\bigskip

The next result is well-known.

\bigskip

\begin{theo}\label{center rang 1}
We have
$Z=P[\euler]=\kb[C_1,\dots,C_{d-1},X,Y,\euler]=\kb[K_1,\dots,K_{d-1},X,Y,\euler]$ 
and the ideal of relations is generated by 
$$\prod_{i=1}^d (\euler - d K_i) = XY.$$
\end{theo}


\begin{proof}
The equality $Z=\kb[C_1,\dots,C_{d-1},X,Y,\euler]=P[\euler]$ has been proven in Proposition \ref{Q=Pe}. Define
$$z=\prod_{i=1}^d (\euler - d K_i).$$
Since $\O_{\e^i}(z)=0$ for all $i$ by~(\ref{euler cyclique}), it follows from
(\ref{inter omega}) that 
$$z \equiv 0 \mod \langle X,Y\rangle_Z.$$
Moreover, since $\euler_0^d=XY$, we have 
$$z \equiv XY \mod \langle C_1,\dots,C_{d-1} \rangle_Z.$$
Therefore,
$$z-XY \in \langle X,Y\rangle_Z \cap \langle C_1,\dots,C_{d-1} \rangle_Z = 
\langle C_1 X,C_1 Y, C_2 X, C_2 Y,\dots, C_{d-1} X,C_{d-1} Y \rangle_Z.$$
On the other hand, $z-XY$ is bi-homogeneous with bidegree $(d,d)$, whereas 
$C_i X$ and $C_i Y$ are bi-homogeneous with bidegree $(d+1,1)$ and 
$(1,d+1)$ respectively. Consequently, $z-XY=0$, which is the required relation.

\medskip

Since the minimal polynomial of $\euler$ over $P$ has degree $|W|=d$
(Proposition~\ref{pr:minimal-Euler}), we deduce that 
$$\prod_{i=1}^d (\tb - d K_i) - XY$$
is the minimal polynomial of $\euler$ over $P$: this concludes the proof of the theorem.
\end{proof}

\bigskip

\begin{coro}\label{inter 1}
The $\kb$-algebra $Z$ is a complete intersection.
\end{coro}

\bigskip

We denote by $F_\euler(\tb) \in P[\tb]$ the minimal polynomial of $\euler$ over $P$. 
Theorem~\ref{center rang 1} gives
\equat\label{polynome minimal euler rang 1}
F_\euler(\tb)=\prod_{i=1}^d (\tb - d K_i) - XY.
\endequat

\bigskip

\section{The ring ${\boldsymbol{R}}$, the group ${\boldsymbol{G}}$}\label{section:G rang 1} 

\medskip

\subsection{Symmetric polynomials} 
To take advantage of the fact that the minimal polynomial of the Euler element 
is symmetric in the variables $K_i$, we will recall here some classical facts about 
symmetric polynomials. 
Given $T_1$, $T_2$,\dots, $T_d$ a family of indeterminates and given $1 \le i \le d$, we
denote by
$\s_i(\Tb)$ the $i$-th elementary symmetric function 
$$\s_i(\Tb)=\s_i(T_1,\dots,T_d)=\sum_{1 \le j_1 < \cdots < j_i \le d} T_{j_1}\cdots T_{j_i}.$$
Recall the well-known formula
\equat\label{eq:jacobien}
\det\Bigl(\frac{\partial \s_i(\Tb)}{\partial T_j}\Bigr)_{1 \le i,j \le d} = 
\prod_{1 \le i < j \le d} (T_j-T_i).
\endequat
The group $\SG_d$ acts on $\kb[T_1,\dots,T_d]$ by permutation of the indeterminates.
Recall the following classical result (a particular case of Theorem~\ref{chevalley}).

\bigskip

\begin{prop}\label{prop:polynomes-symetriques}
The polynomials $\s_1(\Tb)$,\dots, $\s_d(\Tb)$ are algebraically independent and 
$\kb[T_1,\dots,T_d]^{\SG_d}=\kb[\s_1(\Tb),\dots,\s_d(\Tb)]$. Moreover, the $\kb$-algebra 
$\kb[T_1,\dots,T_d]$ is a free $\kb[\s_1(\Tb),\dots,\s_d(\Tb)]$-module of rank $d!$
\end{prop}

Recall also that $\s_1(\Tb)=T_1+\cdots + T_d$.

\begin{coro}\label{coro:polynomes-symetriques}
We have $\bigl(\kb[T_1,\dots,T_d]/\langle \s_1(\Tb) \rangle \bigr)^{\SG_d} \simeq 
\kb[\s_2(\Tb),\dots, \s_d(\Tb)]$ and the $\kb$-algebra 
$\kb[T_1,\dots,T_d]/\langle\s_1(\Tb)\rangle$ is a free $\kb[\s_2(\Tb),\dots, \s_d(\Tb)]$-module 
of rank $d!$.
\end{coro}

\bigskip

As a consequence of Proposition~\ref{prop:polynomes-symetriques}, 
there exists a unique polynomial $\D_d$ in $d$ variables such that 
\equat\label{eq:def-discriminant}
\prod_{1 \le i < j \le d} (T_j-T_i)^2=\D_d(\s_1(\Tb),\s_2(\Tb),\dots,\s_d(\Tb)).
\endequat

\bigskip

\subsection{Presentation of ${\boldsymbol{R}}$} 
Let $\s_i(\Kb)=\s_i(K_1,\dots,K_d)$ (in particular, $\s_1(\Kb)=0$). 
By Corollary~\ref{coro:polynomes-symetriques}, the ring
$P_\sym=\kb[\s_2(\Kb),\dots,\s_d(\Kb),X,Y]$ is the invariant ring, 
in $P$, of the group $\SG_d$ acting by permutation of the $K_i$'s. 
Moreover, 
\equat\label{eq:p-psym}
\text{\it $P$ is a free $P_\sym$-module of rank $d!$.}
\endequat
Let us introduce a new family of indeterminates $E_1$,\dots, $E_{d-1}$, and let 
$E_d=-(E_1+\cdots+E_{d-1})$ and $\s_i(\Eb)=\s_i(E_1,\dots,E_d)$ (in particular 
$\s_1(\Eb)=0$). Let $R_\sym=\kb[E_1,\dots,E_{d-1},X,Y]=\kb[E_1,\dots,E_d,X,Y]/\langle \s_1(\Eb)\rangle$, 
on which the symmetric group $\SG_d$ acts by permutation of the $E_i$'s. The ring $R_\sym^{\SG_d}$ 
is again a polynomial algebra equal to $\kb[\s_2(\Eb),\dots,\s_d(\Eb),X,Y]$ 
(still thanks to Corollary~\ref{coro:polynomes-symetriques}).

\medskip

\boitegrise{{\bf Identification.} {\it We identify the $\kb$-algebras $P_\sym$ and 
$R_\sym^{\SG_d}$ through the equalities 
$$\begin{cases}
\s_1(d\Kb)=\s_1(\Eb)=0 \\
\forall~2 \le i \le d-1,~\s_i(d\Kb)=\s_i(\Eb)\\
\s_d(d\Kb)=\s_d(\Eb)+(-1)^d XY
\end{cases}$$
Note that $\s_i(d\Kb)=d^i\s_i(\Kb)$.}}{0.75\textwidth}

\medskip

As a consequence,
\equat\label{eq:rsym-psym}
\text{\it $R_\sym$ is a free $P_\sym$-module of rank $d!$.}
\endequat

\bigskip

\begin{lem}\label{lem:integre-normal}
The ring $P \otimes_{P_\sym} R_\sym$ is an integrally closed domain.
\end{lem}

\begin{proof}
First of all, note that we may, and we will, assume in this proof that $\kb$ is integrally closed.
Let $\Rti=P \otimes_{P_\sym} R_\sym$. 
Then $\Rti$ admits the following presentation:
$$\begin{cases}
\text{Generators:} & K_1, K_2,\dots, K_d, E_1,E_2,\dots,E_d,X,Y\\
\text{Relations:} & 
\begin{cases}
\s_1(d\Kb)=\s_1(\Eb)=0 \\
\forall~2 \le i \le d-1,~\s_i(d\Kb)=\s_i(\Eb)\\
\s_d(d\Kb)=\s_d(\Eb)+(-1)^d XY
\end{cases}
\end{cases}\leqno{(\PC)}$$

The presentation $(\PC)$ of $\Rti$ shows that $\Rti$ is endowed with an $\NM$-grading 
such that $\deg(K_i)=\deg(E_i)=2$ and $\deg(X)=\deg(Y)=d$. Thus, the degree $0$ component of 
$\Rti$ is isomorphic to $\kb$, which shows that 
$$\text{\it $\Rti$ is connected.}\leqno{(\clubsuit)}$$

Since $\Rti$ is a free $P_\sym$-module of finite rank, it follows that
$\Rti$ has pure dimension $d+1$.
The presentation $(\PC)$ shows that 
$$\text{\it $\Rti$ is complete intersection.}\leqno{(\heartsuit)}$$

\def\Jac{{\mathrm{Jac}}}

Let us now show that 
$$\text{\it $\Rti$ is regular in codimension $1$.}\leqno{(\spadesuit)}$$
Let $\RCt=\Spec\Rti$, a closed subvariety of $\AM^{2d+2}(\kb)$ 
consisting of elements $r=(k_1,\dots,k_d,e_1,\dots,e_d,x,y)$ satisfying 
the equations $(\PC)$. The Jacobian $\Jac(r)$ of this system of equations $(\PC)$ in $r \in \RCt$ 
is given by
$${\scriptsize{\Jac(r)=
\begin{pmatrix}
d &  \cdots & d & 0 &  \cdots & 0 & 0 & 0 \\
&&&&&&&\\
0 &  \cdots & 0 & -1 &  \cdots & -1 & 0 & 0 \\
&&&&&&&\\
\DS{\frac{\partial \s_2(d\Kb)}{\partial K_1}}(r) & 
\cdots & \DS{\frac{\partial \s_2(d\Kb)}{\partial K_d}}(r) & 
-\DS{\frac{\partial \s_2(\Eb)}{\partial E_1}}(r) & 
\cdots & -\DS{\frac{\partial \s_2(\Eb)}{\partial E_d}}(r) & 0 & 0 \\
&&&&&&&\\
\vdots & & \vdots & \vdots & & \vdots & \vdots & \vdots \\
&&&&&&&\\
\DS{\frac{\partial \s_{d-1}(d\Kb)}{\partial K_1}}(r) & 
\cdots & \DS{\frac{\partial \s_{d-1}(d\Kb)}{\partial K_d}}(r) & 
-\DS{\frac{\partial \s_{d-1}(\Eb)}{\partial E_1}}(r) & 
\cdots & -\DS{\frac{\partial \s_{d-1}(\Eb)}{\partial E_d}}(r) & 0 & 0 \\
&&&&&&&\\
\DS{\frac{\partial \s_{d}(d\Kb)}{\partial K_1}}(r) & 
\cdots & \DS{\frac{\partial \s_{d-1}(d\Kb)}{\partial K_d}}(r) & 
-\DS{\frac{\partial \s_{d}(\Eb)}{\partial E_1}}(r) & 
\cdots & -\DS{\frac{\partial \s_{d}(\Eb)}{\partial E_d}}(r) & (-1)^{d+1} y & (-1)^{d+1} x \\
\end{pmatrix}}}$$

Since $\RCt$ is of pure dimension $d+1$, its singular locus is the 
closed subvariety $X$ of points $r$ where
the rank of $\Jac(r)$ is less than or equal to $d$. A point of $X$ satisfies the
equations
$$\det\Bigl(\frac{\partial \s_i(d\Kb)}{\partial K_j}(k_1,\dots,k_d)\Bigr)_{1 \le i , j \le d} = 
\det\Bigl(\frac{\partial \s_i(\Eb)}{\partial E_j}(e_1,\dots,e_d)\Bigr)_{1 \le i , j \le d} = 0.$$
By~(\ref{eq:jacobien}), this means that 
$$\prod_{1 \le i < j \le d}(k_j-k_i)=\prod_{1 \le i < j \le d}(e_j-e_i)=0$$
In particular, 
$$\D_d(\s_1(k_1,\dots,k_d),\dots,\s_d(k_1,\dots,k_d))=\D_d(\s_1(e_1,\dots,e_d),\dots,\s_d(e_1,\dots,e_d))=0.$$
It is well-known that $\D_d(0,U_2,\dots,U_d)$ is an irreducible polynomial in the 
indeterminates $U_2$,\dots, $U_d$. It follows that the variety of
$(a_2,\dots,a_d,x,y)\in\AM^{d+1}(\kb)$ such that 
$$\D_d(0,a_2,\dots,a_{d-1},a_d)=\D_d(0,a_2,\dots,a_{d-1},a_d+(-1)^d xy)=0.\leqno{(*)}$$
has dimension $\le d-1$. Consequently, $X$ has codimension $\ge 2$ in $\RCt$.

\medskip

The assertions $(\diamondsuit)$ and $(\spadesuit)$ imply that $\Rti$ is normal 
(see~\cite[\S{\MakeUppercase{\romannumeral 4}.D},~Th\'eor\`eme~11]{serre}). 
So it is a direct product of integrally closed domains. Since it is connected, 
it follows that it is an integrally closed domain.
\end{proof}


\bigskip

We can now describe the ring $R$.

\bigskip

\begin{theo}\label{theo:r-cyclique}
The ring $R$ satisfies the following properties:
\begin{itemize}
\itemth{a} $R$ is isomorphic to $P \otimes_{P_\sym} R_\sym$. It admits the following presentation:
$$\begin{cases}
\text{Generators:} & K_1, K_2,\dots, K_d, E_1,E_2,\dots,E_d,X,Y\\
\text{Relations:} & 
\begin{cases}
\s_1(d\Kb)=\s_1(\Eb)=0 \\
\forall~2 \le i \le d-1,~\s_i(d\Kb)=\s_i(\Eb)\\
\s_d(d\Kb)=\s_d(\Eb)+(-1)^d XY
\end{cases}
\end{cases}\leqno{(\PC)}$$

\itemth{b} $R$ is complete intersection and is a free $P$-module of rank $d!$.

\itemth{c} There exists a unique morphism of $P$-algebras $\copie : Z \to R$ such that $\copie(\euler)=E_d$. 
This morphism is injective, with $Q$ its image.

\itemth{d} For the action of $\SG_d$ by permutation of the $E_i$'s, 
we have $R^{\SG_d}=P$ and $R^{\SG_{d-1}}=Q$.

\itemth{e} $G=\SG_W \simeq \SG_d$; given $\s \in \SG_d$ and $1 \le i \le d$, we have
$\s(E_i)=E_{\s(i)}$. 

\itemth{f} $G$ is a reflection group for its action on $R_+/(R_+)^2$.
\end{itemize}
\end{theo}

\begin{proof}
Let $\Rti=P \otimes_{P_\sym} R_\sym$.
The relations $(\PC)$ show that, in the polynomial ring $\Rti[\tb]$, the equality
$$\prod_{i=1}^d(\tb-dK_i)-XY = \prod_{i=1}^d (\tb - E_i)$$
holds. It follows that $F_\euler(E_d)=0$. By Theorem~\ref{center rang 1}, 
we deduce that there exists a unique morphism of $P$-algebras 
$\copie : Z \to \Rti$ such that $\copie(\euler)=E_d$. 
Let $\zG=\Ker(\copie)$. We have $\zG \cap P=0$ since $P \subset \Rti$ and, 
since $Z$ is a domain and is integral over $P$, this forces $\zG=0$. 
So $\copie : Z \to \Rti$ is injective.

\medskip

Let $\Mbt$ be the fraction field of $\Rti$ (recall that $\Rti$ is a domain by 
Lemma~\ref{lem:integre-normal}). By construction, $\Rti$ is a free $P$-module of rank $d!$ and, 
by Corollary~\ref{coro:polynomes-symetriques}, $\Rti^{\SG_d}=P$. 
So the extension $\Mbt/\Kb$ is Galois, contains $\Lb$ (the fraction field of $Q$) 
and satisfies $\Gal(\Mbt/\Kb)=\SG_d$. Moreover,
$\Gal(\Mbt/\Lb)=\SG_{d-1}$ since $\SG_{d-1}$ is the stabilizer of $E_d$ in $\SG_d$. 
Since the unique normal subgroup of $\SG_d$ contained in $\SG_{d-1}$ is the 
trivial group, this shows that $\Mbt/\Kb$ is a Galois closure of $\Lb/\Kb$. 
So $\Mbt \simeq \Mb$.

\medskip

Since $\Rti$ is integrally closed (Lemma~\ref{lem:integre-normal}) and integral over $P$, 
this implies that $\Rti \simeq R$. Now all the statements of Theorem~\ref{theo:r-cyclique} 
can be deduced from these observations. For the statement~(f), we can use~(b),  and 
Proposition~\ref{intersection complete R} because $\SG_d$ acts trivially on the relations, 
or check it directly by noting that $R_+/(R_+)^2$ is the $\kb$-vector space of 
dimension $2d$ generated by 
$K_1$,\dots, $K_d$, $E_1$,\dots, $E_d$, $X$, $Y$, with the relations 
$K_1+\cdots+K_d=0$ and $E_1+ \cdots + E_d=0$: this shows that, as a representation of 
$\SG_d$, $R_+/(R_+)^2$ is the direct sum of the irreducible reflection representation 
and of $d+1$ copies of the trivial representation. 
\end{proof}

\bigskip

\subsection{Choice of the ideal ${\boldsymbol{\rG_0}}$} 
Let $\rG'$ denote the ideal of $R$ generated by the elements $E_i-\z^i E_d$. 
We then have
$$\s_1(\Eb) \equiv \s_2(\Eb) \equiv \cdots \equiv \s_{d-1}(\Eb) \equiv 0 \mod \rG'.$$
We choose the ideal of $R$
$\rG_0=\rG' + \langle K_1,\dots, K_d\rangle_R$. The $\kb$-algebra $R/\rG_0$ has the
following presentation:
$$\begin{cases}
\text{Generators:} & E_d,X,Y\\
\text{Relation:} & 
E_d^d = XY
\end{cases}\leqno{(\PC_0)}$$
Recall that $\Zrm(W)=W$. 
We recover the isomorphisms of $P$-algebras
$R/\rG_0 \simeq Q/\qG_0 \simeq \kb[V \times V^*]^{\D W}$ by mapping
$\eulerq=E_d$ to $\euler_0=yx \in \kb[V \times V^*]^{\D W}$. 
Recall that an element $w \in W$, viewed as an element of the Galois group $G=\SG_W\simeq \SG_d$, 
is characterized by the equality
$$(w(\eulerq) \mod \rG_0)\equiv w(y)x \in \kb[V \times V^*]^{\D W}.$$
Since $s^i(y)=\z^i y$, we have
\equat\label{eq:action-w-cyclique}
s^i(\eulerq) = E_i.
\endequat
For the action of $G =\SG_W \simeq \SG_d$, this corresponds to identifying the sets 
$\{1,2,\dots,d\}$ and $W$ via the bijective map $i \mapsto s^i$.

\bigskip

\bigskip

\subsection{Choice of the ideals ${\boldsymbol{\rG^\gauche}}$, ${\boldsymbol{\rG^\droite}}$ and 
${\boldsymbol{\rGba}}$}\label{section:choix rang 1}
Let $\rG''$ denote the ideal of $R$ generated by the $E_i-dK_i$'s. Then 
$$\forall~1 \le i \le d,~\s_i(d\Kb) \equiv \s_i(\Eb) \mod \rG''.$$
In particular, $XY \in \rG''$. 
We choose $\rG^\gauche=\rG'' + \langle Y \rangle_R$, $\rG^\droite=\rG''+\langle X \rangle_R$ 
and $\rGba=\rG'' + \langle X,Y \rangle_R$. Then
\equat\label{eq:iso-r}
\begin{cases}
 R/\rG^\gauche \simeq \kb[K_1,\dots,K_{d-1},X] =P/\pG^\gauche,\\
 R/\rG^\droite \simeq \kb[K_1,\dots,K_{d-1},Y] =P/\pG^\droite,\\
 R/\rGba \simeq \kb[K_1,\dots,K_{d-1}]=\kb[\CCB] =P/\pGba.\\
\end{cases}
\endequat
The next proposition follows easily.

\bigskip

\begin{prop}\label{prop:dec-cyclique}
$D^\gauche=I^\gauche=D^\droite=I^\droite=\Dba=\Iba=1$.
\end{prop}

\bigskip

\section{Cells, families, cellular characters}

\medskip

\boitegrise{\noindent{\bf Notation.} {\it We fix in this section a prime ideal 
$\CG$ of $\kb[\CCB]$ and we denote by $k_i$ the image of 
$K_i$ in $\kb[\CCB]/\CG$.}}{0.75\textwidth}

\medskip

By~(\ref{eq:iso-r}), we have 
\equat\label{eq:rc-cyclique}
\rG_\CG^\gauche=\rG^\gauche + \CG R,\quad \rG_\CG^\droite=\rG^\droite + \CG R
\quad\text{and}\quad \rGba_\CG=\rGba + \CG R
\endequat
and
\equat\label{eq:r-c-cyclique}
\begin{cases}
R/\rG_\CG^\gauche=\kb[\CCB]/\CG \otimes \kb[X]=P/\pG_\CG^\gauche,\\
R/\rG_\CG^\droite=\kb[\CCB]/\CG \otimes \kb[Y]=P/\pG_\CG^\droite,\\
R/\rGba_\CG=\kb[\CCB]/\CG=P/\pGba_\CG.
\end{cases}
\endequat
We will denote by $\SG[\CG]$ the subgroup of $\SG_d$ consisting of permutations 
stabilizing the fibers of the natural map $\{1,2,\dots,d\} \to \kb[\CCB]/\CG$, $i \mapsto k_i$. 
In other words,
$$\SG[\CG]=\{\s \in \SG_d~|~\forall~1 \le i \le d,~k_{\s(i)}=k_i\}.$$

\bigskip

\begin{prop}\label{prop:dec-c-cyclique}
$D^\gauche_\CG=I^\gauche_\CG=D^\droite_\CG=I^\droite_\CG=\Dba_\CG=\Iba_\CG=\SG[\CG]$.
\end{prop}

%

\medskip

\begin{coro}\label{coro:cellules}
Let $i$, $j \in \BZ$. Then $s^i$ and $s^j$ are in the same Calogero-Moser 
two-sided (respectively left, respectively right) $\CG$-cell if and only if $k_i=k_j$.
\end{coro}

\bigskip

Let us conclude with the description of families and cellular characters.

\bigskip

\begin{coro}\label{coro:familles-cellulaires-cyclique}
Let $i$, $j \in \BZ$. Then $\e^{-i}$ and $\e^{-j}$ are in the same
Calogero-Moser family if and only if $k_i=k_j$.

The map $\o \mapsto \sum_{i \in \o} \e^{-i}$ induces a bijective map between the set of 
$\SG[\CG]$-orbits in $\{1,2,\dots,d\}$ (that is, the set of fibers of the 
map $i \mapsto k_i$) and the set of Calogero-Moser $\CG$-cellular characters.
\end{coro}

\begin{proof}
Since $Z=P[\euler]$, we deduce that $\e^{-i}$ and $\e^{-j}$ are in the same Calogero-Moser 
$\CG$-family if and only if $\O_{\e^{-i}}^{\Kbov_\CG}(\euler)=\O_{\e^{-j}}^{\Kbov_\CG}(\euler)$. 
So the first statement follows from~(\ref{euler cyclique}).

\medskip

For the second statement, note that $\euler$ acts on $\LC_{s^i}$ by multiplication by $s^i(\eulerq)=E_i$. 
So, modulo $\rG_\CG^\gauche$ (or $\rGba_\CG$), the element $s^i(\eulerq)$ is congruent to 
$dk_{i}=\O_{\e^{-i}}^{\Kbov_\CG}(\euler)$. Hence, if $\o$ is an $\SG[\CG]$-orbit in 
$\{1,2,\dots,d\}$, then $C=\{s^i~|~i \in \o\}$ is a Calogero-Moser left, right or two-sided 
$\CG$-cell (see Corollary~\ref{coro:cellules}) and, 
as a two-sided cell, it covers the Calogero-Moser $\CG$-family 
$\{\e^{-i}~|~i \in \o\}$. Since $\Mb_\CG^\gauche \MC^\gauche(\e^{-i})$ is an absolutely simple 
$\Mb_\CG^\gauche\Hb^\gauche$-module (because it has dimension $|W|$), it must be
isomorphic to $\LC_\CG^\gauche(C)$. This shows that 
$\isomorphisme{C}_\CG^\calo=\sum_{i \in \o} \e^{-i}$.
\end{proof}

\bigskip

\section{Complements}

We will be interested here in geometric properties of $\ZCB$ 
(smoothness, ramification) and in the properties of the group $D_c$. 
To simplify the statements, we will make the following assumption:

\bigskip
\def\ram{{\mathrm{ram}}}

\boitegrise{{\bf Assumption and notation.} {\it In this section, and only in this section, 
we will assume that $\kb$ is {\bfit algebraically closed}. We will identify the variety $\ZCB$ with 
$$\ZCB=\{(k_1,\dots,k_d,x,y,e) \in \AM^{d+3}(\kb)~|~k_1+\cdots+k_d=0\text{ and } \prod_{i=1}^d(e-dk_i)=xy\}.$$
Similarly, $\PCB$ (respectively $\CCB$) will be identified with the affine space  
$$\PCB=\{(k_1,\dots,k_d,x,y) \in \AM^{d+2}(\kb)~|~
k_1+\cdots +k_d = 0\}$$
(respectively
$$\CCB=\{(k_1,\dots,k_d) \in \AM^{d}(\kb)~|~
k_1+\cdots +k_d = 0\}\quad),$$
which allows to redefine
$$\fonction{\Upsilon}{\ZCB}{\PCB}{(k_1,\dots,k_d,x,y,e)}{(k_1,\dots,k_d,x,y).}$$
Finally, we denote by $\ZCB_\singulier$ the singular locus of $\ZCB$ 
and $\ZCB_\ram$ the ramification locus of $\Upsilon$.
}}{0.8\textwidth}

\bigskip

\subsection{Smoothness} 
Let us start by the description of the singular locus of $\ZCB$:

\bigskip

\begin{prop}\label{prop:zsing-cyclique}
Given $1 \le i < j \le d$, let $\ZCB_{i,j} = \{(k_1,\dots,k_d,x,y,e) \in \ZCB~|~
e=dk_i=dk_j$ and $x=y=0\}$. Then 
$$\ZCB_\singulier=\bigcup_{1 \le i < j \le d} \ZCB_{i,j}.$$
Moreover, $\ZCB_{i,j} \simeq \AM^{d-2}(\kb)$ is an irreducible component of $\ZCB_\singulier$
and $\ZCB_\singulier$ is purely of codimension $3$.
\end{prop}

\begin{proof}
The variety $\ZCB$ being described as an hypersurface in the affine space 
$\{(k_1,\dots,k_d,x,y,e) \in \AM^{d+3}(\kb)~|~k_1+\cdots+k_d=0\} \simeq \AM^{d+2}(\kb)$, a point of 
$z=(k_1,\dots,k_d,x,y,e) \in \ZCB$ is singular if and only if 
the jacobian matrix of the equation vanishes at $z$. This is equivalent 
to the following system of equations:
$$
\begin{cases}
x=y=0,\\
\forall~1 \le i \le d, \prod_{j \neq i} (e-dk_j) = 0,\\
\sum_{i=1}^d \prod_{j \neq i} (e-dk_j) = 0.
\end{cases}
$$
The last equation is implied by the second family of equations, 
it is then easy to check that $\ZCB_\singulier$ is as expected. 

The last statements are immediate.
\end{proof}

\bigskip

\begin{coro}\label{coro:zsing-cyclique}
Given $c \in \CCB$ and $z \in \ZCB_c$, then $z$ is singular in $\ZCB$ if and only if it 
is singular in $\ZCB_c$.
\end{coro}

\bigskip

\subsection{Ramification} 
The variety $\ZCB$ being normal, the variety $\PCB$ being smooth and the morphism 
$\Upsilon : \ZCB \to \PCB$ being finite and flat, the purity of the branch 
locus~\cite[Expos\'e~\MakeUppercase{\romannumeral 10},~Th\'eor\`eme~3.1]{sga} 
tells us that the ramification locus of $\Upsilon$ is purely of codimension $1$. 
It is in fact easily computable:

\bigskip

\begin{prop}\label{prop:ramification-cyclique}
Let $z=(k_1,\dots,k_d,x,y,e) \in \ZCB$ and $p=(k_1,\dots,k_d,x,y)=\Upsilon(z) \in \PCB$. 
Let $F_{\euler,p}(\tb) \in \kb[\tb]$ denote the specialization of $F_\euler(\tb)$ at $p$. 
Then $\Upsilon$ is ramified at $z$ if and only if $F_{\euler,p}'(e) =0$, that is 
if and only if $e$ is a multiple root of $F_{\euler,p}$.
\end{prop}

\begin{proof}
Since $Z=P[\tb]/\langle F_\euler(\tb)\rangle$ (see Theorem~\ref{center rang 1}), 
this follows immediately from~\cite[Expos\'e~\MakeUppercase{\romannumeral 1},~Corollare~7.2]{sga}.
\end{proof}

\bigskip

\begin{coro}\label{coro:ramifiction-cyclique-p}
Let $c = (k_1,\dots,k_d) \in \CCB$ and $(x,y) \in \AM^2(\kb)$ (so that $(c,x,y) \in \PCB$). 
Then $(c,x,y) \in \Upsilon(\ZCB_\ram)$ if and only if 
$\D_d(0,\s_2(c),\dots,\s_{d-1}(c),\s_d(c)-(-1)^dxy)=0$.
\end{coro}

\bigskip

\subsection{About the group ${\boldsymbol{D_c}}$}\label{rema:dc-cyclique} 
Amongst the groups $D_c$, the only one we use is $D_0$. 
In this subsection, we will show that, even when $n=\dim_\kb(V)=1$, 
the groups $D_c$ have a very subtle behaviour.

The material of this subsection \S\ref{rema:dc-cyclique}
has been explained to us by G. Malle 
(any errors being of course our responsibility). We thank him warmly for his help. 

Fix $c \in \CCB$, and let $F_\euler^c(\tb)\in\kb[X,Y][\tb]$ denote the specialization of $F_\euler(\tb)$ at $c$. 
Note that
$D_c$ is the Galois group of $F_\euler^c(\tb)$, viewed as a polynomial 
with coefficients in the field $\kb(X,Y)$. Actually, we have
$F_\euler^c(\tb) \in \kb[T][\tb]$, where $T=XY$.
The following result will be helpful for computing $D_c$.

\bigskip

\begin{lem}\label{lem:dc-t}
$D_c$ is the Galois group of $F_\euler^c(\tb)$ viewed as a polynomial with coefficients in $\kb(T)$.
\end{lem}

\begin{proof}
If $L$ is a splitting field of $F_\euler^c(\tb)$ over $\kb(T)$, then 
the field $L(Y)$ of rational functions in one variable is a splitting field 
of the same polynomial over $\kb(T,Y)=\kb(X,Y)$. The result follows. 
\end{proof}

\bigskip

\begin{coro}\label{coro:dc-t}
The subgroup $D_c$ of $G=\SG_d$ contains a cycle of length $d$.
\end{coro}

\begin{proof}
Thanks to Lemma~\ref{lem:dc-t}, we may view $F_\euler^c(\tb)$ as an element of $\kb[T][\tb]$. 
Since $\kb$ has characteristic zero and $\kb[T]$ is regular of dimension $1$, 
the inertia group at infinity $I$ is cyclic. 
Since $d \ge 2$, the polynomial $F_\euler^c(\tb)$ is totally ramified at infinity, 
which implies that $I$ acts transitively on $\{1,2,\dots,d\}$. Whence the result.
\end{proof}

\bigskip

Corollary~\ref{coro:dc-t} gives very restrictive conditions on the group $D_c$. 
For instance, we have the following results, the first of which is due to
Schur, the second one to Burnside.

\bigskip

\begin{coro}\label{coro:schur-burnside}
\begin{itemize}
\itemth{a} If $d$ is not prime and $D_c$ is primitive, then $D_c$ is $2$-transitive.

\itemth{b} If $d$ is prime, then $D_c$ is $2$-transitive or $D_c$ contains a normal Sylow
$d$-subgroup.
\end{itemize}
\end{coro}

\bigskip
%
%
%
%
%
\def\frobenius{{\mathrm{Fr}}}

We will give now some examples that show that the description of $D_c$ in general 
can be rather complicated. In the following table, we assume that $\kb=\CM$, 
and $c \in \CCB$ is chosen so that $F_\euler^c(\tb) \in \QM[\tb]$. We will denote by $D_c^{(\QM)}$ 
the Galois group of $F_\euler^c(\tb)$ viewed as an element of $\QM[\tb]$: it might be different 
from $D_c$, as it is shown by table.
We denote by $\frobenius_{pr}$ the Frobenius group 
$(\BZ/r\BZ) \ltimes (\BZ/p\BZ)$, viewed as a subgroup of $\SG_p$, where $p$ is prime and $r$ divides $p-1$. 

\bigskip

\begin{centerline}{
\begin{tabular}{@{{{\vrule width 1.2pt}\,\,\,}}c@{{\,\,\,{\vrule width 1.2pt}\,\,\,}}c|c@{{\,\,\,{\vrule width 1.2pt}}}}
\hlinewd{1.2pt}
\petitespace $F_\euler^c(\tb)$ & $D_c^{(\QM)}$ & $D_c$ \\
\hlinewd{1.2pt}
\petitespace $(\tb^2+20\tb+180)(\tb^2-5\tb-95)^4 - XY$ & $\Aut(\AG_6)$ & $\Aut(\AG_6)$ \\
\hline
\petitespace $(\tb+1)^4(\tb-2)^2(\tb^3-3\tb-14) - XY$ & $(\SG_3 \wr \SG_3) \cap \AG_9$ & $(\SG_3 \wr \SG_3) \cap \AG_9$ \\
\hline
\petitespace 
$\tb(\tb^8+6\tb^4+25) - XY$ & $\AG_9$ & $\AG_9$ \\
\hline\petitespace 
$\tb(\tb^4+6\tb^2+25)^2 - XY$ & $\AG_9$ & $\AG_9$ \\
\hline
\petitespace $\tb^9-9\tb^7+27\tb^5-30\tb^3+9\tb -XY$
& $\SG_3 \ltimes (\BZ/3\BZ)^2$ & $\SG_3 \ltimes (\BZ/3\BZ)^2$ \\
\hline 
\petitespace
$\tb^{11}-11\tb^9+44\tb^7-77\tb^5+55\tb^3-11\tb - XY$ & $\frobenius_{110}$ & $\frobenius_{22}$  \\ 
\hline
\petitespace
$\tb^{13}-13\tb^{11}+65\tb^9-156\tb^7+182\tb^5-91\tb^3+13\tb - XY$ & $\frobenius_{156}$ & $\frobenius_{26}$ \\ 
\hlinewd{1.2pt}
\end{tabular}
}
\end{centerline}

\bigskip

To prove the results contained in the table, let us recall some classical facts:

\medskip

\begin{itemize}
\itemth{a} $D_c$ is a normal subgroup of $D_c^{(\QM)}$.

\medskip

\itemth{b} The computation of $D_c^\QM$ in all the cases can be performed thanks to 
the {\tt MAGMA} software~\cite{magma}.

\medskip

\itemth{c} Since there are no non-trivial unramified coverings of the complex
affine line, the group $D_c$ is generated by its inertia subgroups.

\medskip

\itemth{d} Given $z \in \CM$, let $F_\euler^{c,z}(\tb)$ denote the specialization 
$T \mapsto z$ of $F_\euler^c(\tb)$. If $\a \in \CM$ is a root of $F_\euler^{c,z}(\tb)$ with multiplicity 
$m$, then $D_c$ contains an element of order $m$.
\end{itemize}

\medskip

From these facts, the result in the table can be obtained as follows.
Let $\D_c(T) \in \kb[T]$ denote the discriminant of the polynomial $F_\euler^c(\tb)$.

\medskip

\begin{itemize}
\itemth{1} 
The computation of $D_c^{(\QM)}$ in the first example is done in~\cite[Theorem~I.9.7]{mama}. 
Note that, to retrieve the polynomial of~\cite[Theorem~I.9.7]{mama}, one must replace 
$\tb$ by $2\tb-5$, and renormalize: this operation allows to obtain a polynomial whose 
coefficient in $\tb^9$ is zero, as must be the case for all the $F_\euler^c(\tb)$. 
Going from $\QM$ to $\CM$ then follows, as this example comes from a rigid triple.

\medskip

\itemth{2} 
In the second example, the {\tt MAGMA} software tells us that $D_c^{(\QM)}=(\SG_3 \wr \SG_3) \cap \AG_9$. 
On the other hand, $D_c$ is a normal subgroup of $D_c^{(\QM)}$. 
Moreover, $-1$ is a root of $F_\euler^{c,0}(\tb)$ with multiplicity $4$, so 
$D_c$ contains an element of order $4$ (see~(d)). It also contains an element of order 
$9$ (see Corollary~\ref{coro:dc-t}). But $D_c^{(\QM)}$ contains only one normal subgroup 
containing both an element of order $9$ and an element of order $4$, namely itself. 
So $D_c=D_c^{(\QM)}$.

\medskip

\itemth{3\text{-}4} 
In the third and fourth examples, the equality $D_c^{(\QM)}=\AG_9$ is obtained by
the 
{\tt MAGMA} software. The fact that $D_c=\AG_9$ follows from the fact that 
$D_c$ is normal in $D_c^{(\QM)}$ and contains an element of order $9$.

\medskip

\itemth{5} Once the computation of $D_c^{(\QM)}$ done by the {\tt MAGMA} software, note that 
$$F_\euler^{c,2}(\tb)=(\tb-2)(\tb+1)^2(\tb^3-3\tb+1)^2.$$
This allows to say, thanks to~(d) and Corollary~\ref{coro:dc-t}, 
that $D_c$ contains an element of order $9$ 
and an element of order $2$. So $18$ divides $|D_c|$. 
But $D_c^{(\QM)}$ does not contain any normal subgroup of index $3$ and containing 
an element of order $9$.

\medskip

\itemth{6\text{-}7} The last two examples can be treated similarly. We will only deal with the last one. 
In this case, thanks to the {\tt MAGMA} software, we get 
$$\D_c(T)=13^{13} (T-2)^6(T+2)^6.$$
This discriminant is not a square in $\QM(T)$, but it is a square in $\CM(T)$. 
The non-trivial inertia groups in $D_c$, except the inertia group at infinity, 
lie above the ideals $\langle T-2 \rangle$ and $\langle T+2 \rangle$. But,
$$F_\euler^{c,2}(\tb)=(\tb-2)(\tb^6 + \tb^5 - 5\tb^4 - 4\tb^3 + 6\tb^2 + 3\tb - 1)^2$$
and
$$F_\euler^{c,-2}(\tb)=(\tb+2)(\tb^6 - \tb^5 - 5\tb^4 + 4\tb^3 + 6\tb^2 - 3\tb - 1)^2.$$
So the inertia groups have order $2$, which shows that $D_c$ is generated by its elements 
of order $2$. So $D_c=\frobenius_{26}$.


\end{itemize}

\bigskip

\subsection{Cohomology} 
We assume in this subsection that $\kb=\CM$. 
We aim to prove the following result:

\bigskip

\begin{theo}\label{theo:cohomologie-1}
If $\dim_\CM(V)=1$, then Conjecture~\COH~holds.
\end{theo}

\bigskip

\begin{proof}
Let $c \in \CCB$ and let $k_1$,\dots, $k_d$ be the images of $K_1$,\dots, $K_d$ in $\CM[\CCB]/\CG_c$. 
We put $r=|\ZCB_c^{\CM^\times}|=|\{k_1,k_2,\dots,k_d\}|$. 
The equations defining $\ZCB_c$ show that it is rationally smooth (it has only type $A$ 
singularities). So it follows from the Proposition~\ref{prop:r} below that 
$$
\dim_\CM \Hrm^i(\ZCB_c)=
\begin{cases}
1 & \text{if $i = 0$,}\\
r-1 & \text{if $i=2$,}\\
0 & \text{otherwise.}
\end{cases}
$$
On the other hand,
$$
\dim_\CM \grad(\im \Omeb^c)^i=
\begin{cases}
1 & \text{if $i = 0$,}\\
r-1 & \text{if $i=2$,}\\
0 & \text{otherwise.}
\end{cases}
$$
So, both $\Hrm^{2\bullet}(\ZCB_c)$ 
and $\grad(\im \Omeb^c)$ are isomorphic to the graded algebra
$k[a_1,\ldots,a_{r-1}]/(a_ia_j)_{1\le i,j\le r-1}$, with $a_i$'s in degree
$2$.
\end{proof}

\bigskip

In order to complete the proof of Theorem~\ref{theo:cohomologie-1}, 
we fix an infinite sequence of non-zero natural numbers $d_1$, $d_2$,\dots, as well as an infinite sequence 
$z_1$, $z_2$,\dots of complex numbers such that $z_i \neq z_j$ if $i \neq j$. 
We set 
$$\XCB(r)=\{(e,x,y) \in \Ab^3(\CM)~|~\prod_{i=1}^r (e - z_i)^{d_i}=xy\}.$$
First, note that $\XCB(r)$ admits an automorphism $\s : \XCB(r) \to \XCB(r)$, 
$(e,x,y) \mapsto (e,y,x)$. It is an involution. So $\s$ acts on the cohomology 
of $\XCB(r)$. We denote by $\CM_+$ (respectively $\CM_-$) the $\CM\langle \s \rangle$-module 
of dimension $1$ on which $\s$ acts by multiplication by $1$ (respectively $-1$). 
Finally, if $z \in \CM$, we set 
$$\XCB(r)_{\! \neq z}=\{(e,x,y) \in \XCB(r)~|~e \neq z\}$$
$$\XCB(r)_{\! = z}=\{(e,x,y) \in \XCB(r)~|~e = z\}.\leqno{\text{and}}$$
These are $\s$-stable subvarieties of $\XCB(r)$. The following result 
describes the cohomology of $\XCB(r)$ as a $\CM\langle \s \rangle$-module. 
We denote by $\Hrm_c^i(\XCB(r))$ the $i$-th cohomology group with compact support 
of $\XCB(r)$. 

\bigskip

\begin{prop}\label{prop:r}
With the above notation, we have:
\begin{itemize}
 \itemth{a} The cohomology with compact support of $\XCB(r)$ is given, as a $\CM\langle \s\rangle$-module, by
$$
\begin{cases}
\Hrm^2_c(\XCB(r)) = \CM_-^{r-1},\\
\Hrm^3_c(\XCB(r)) =0,\\
\Hrm^4_c(\XCB(r)) = \CM_+.\\
\end{cases}
$$

\itemth{b} If $r\neq 1$ or $z \neq z_1$, then the cohomology with compact 
support of $\XCB(r)$ is given, as a $\CM\langle \s\rangle$-module, by
$$
\begin{cases}
\Hrm^2_c(\XCB(r)_{\! \neq z}) = \begin{cases} \CM_-^r & \text{if $z \not\in \{z_1,z_2,\dots,z_r\}$,}\\
\CM_-^{r-1} & \text{if $z \in \{z_1,z_2,\dots,z_r\}$},\end{cases}\\
\Hrm^3_c(\XCB(r)_{\! \neq z}) = \CM_+,\\
\Hrm^4_c(\XCB(r)_{\! \neq z}) = \CM_+.\\
\end{cases}
$$

\itemth{c} Finally, 
$$\begin{cases}
\Hrm^2_c(\XCB(1)_{\! \neq z_1}) = \CM_-^r, \\
\Hrm^3_c(\XCB(1)_{\! \neq z_1}) = \CM_+ \oplus \CM_-,\\
\Hrm^4_c(\XCB(1)_{\! \neq z_1}) = \CM_+.\\
\end{cases}
$$
\end{itemize}
\end{prop}

\noindent{\sc Remark - } As $\XCB(r)$ and $\XCB(r)_{\!\neq z}$ are affine surface, 
their cohomology with compact support vanishes in degree different from $2$, $3$ or $4$.\finl

\bigskip

\begin{proof}
We first gather some elementary fact which will allow a proof of this Theorem by induction.
If $\l \in \CM$, let 
$$\HCB_\l=\{(x,y) \in \CM^2~|~xy =\l\}.$$
It is endowed with an action of $\s$ given by $(x,y) \mapsto (y,x)$. 
Then it is well-known that
$$
\begin{cases}
\Hrm_c^1(\HCB_\l) = \CM_- ,\\
\Hrm_c^2(\HCB_\l) = 
\begin{cases}
\CM_+ & \text{if $\l \neq 0$,}\\
\CM_+ \oplus \CM_- & \text{if $\l=0$.}\\
\end{cases}
\end{cases}
\leqno{(\clubsuit)}$$
Also, the quotient of $\CM^2$ by the action of $\s$ given by $(x,y) \mapsto (y,x)$ 
is isomorphic to $\CM^2$, through the map $(x,y) \mapsto (x+y,xy)$. Therefore, 
by setting $u=x+y$ and $v=xy$, we obtain that
$$\XCB(r)/\langle \s \rangle =\{(e,u,v) \in \CM^3~|~\prod_{i=1}^r(e-z_i)^{d_i} = v\},$$
so
$$\XCB(r)/\langle \s \rangle \simeq \CM^2.\leqno{(\diamondsuit)}$$
On the other hand, there is an obvious isomorphism of varieties 
$$\XCB(r)_{\! = z} \longiso \HCB_{\prod_{i=1}^r(z-z_i)^{d_i}}\leqno{(\heartsuit)}$$
which is $\s$-equivariant. Finally, if $\xi$ denotes a complex number such that 
$\xi^2=\prod_{i=1}^{r-1}(z_r-z_i)^{d_i}$, then the map
$$\fonctio{\XCB(r)_{\!\neq z_r}}{\XCB(r-1)_{\!\neq z_r}}{(e,x,y)}{(e,\xi^{-1}x,\xi^{-1} y)}\leqno{(\spadesuit)}$$
is a $\s$-equivariant isomorphism of varieties.

\medskip

We can now start the proof of the proposition by induction on $r$.

\medskip

$\bullet$ If $r=1$, then, if we translate $e \mapsto e-z_1$, 
we may assume that $z_1=0$. Then 
$$\XCB(1)=\{e,x,y) \in \CM^3~|~e^{d_1} = xy\} \simeq \CM^2 / \mub_{d_1},$$
where the group $\mub_{d-1}$ of $d_1$-th roots of unity in $\CM^\times$ acts on $\CM^2$ by 
by $\z \cdot (x,y) = (\z x,\z^{-1} y)$. So the cohomology of $\XCB(1)$ is 
equal to the space of invariants, under the action of $\mub_{d_1}$, of the cohomology 
of $\CM^2$. So (a) follows. Now, the ``open-closed'' long exact sequence 
gives, by using $(\heartsuit)$, the following exact sequence of $\CM\langle \s\rangle$-modules 
$$\diagram 
& 0 \rto & \Hrm^1_c(\HCB_{z^{d_1}}) \rto & \\
\Hrm^2_c(\XCB(1)_{\!\neq z}) \rto & \Hrm_c^2(\XCB(1)) \rto & \Hrm_c^2(\HCB_{z^{d_1}}) \rto & \\
\Hrm^3_c(\XCB(1)_{\!\neq z}) \rto & \Hrm_c^3(\XCB(1)) \rto & 0 \rto & \\
\Hrm^4_c(\XCB(1)_{\!\neq z}) \rto & \Hrm_c^4(\XCB(1)) \rto & 0 \rto &  \\
\enddiagram$$
But, by (a), $\Hrm_c^2(\XCB(1))=\Hrm^3_c(\XCB(1))=0$, so (b) follows from $(\clubsuit)$.

\medskip

$\bullet$ Let $r \ge 2$ and assume that the result cohomology of $\XCB(r-1)$ is given 
by Proposition~\ref{prop:r}. If $z \in \CM$, let $\l(z)=\prod_{i=1}^r(z-z_i)^{d_i}$. 
The ``open-closed'' long exact sequence gives, by using $(\heartsuit)$, 
the following exact sequence of $\CM\langle \s\rangle$-modules 
$$\diagram 
& 0 \rto & \Hrm^1_c(\HCB_{\l(z)}) \rto & \\
\Hrm^2_c(\XCB(r)_{\!\neq z}) \rto & \Hrm_c^2(\XCB(r)) \rto & \Hrm_c^2(\HCB_{\l(z)}) \rto & \\
\Hrm^3_c(\XCB(r)_{\!\neq z}) \rto & \Hrm_c^3(\XCB(r)) \rto & 0 \rto & \\
\Hrm^4_c(\XCB(r)_{\!\neq z}) \rto & \Hrm_c^4(\XCB(r)) \rto & 0 \rto &  \\
\enddiagram\leqno{(*)}$$
Assume first that $z=z_r$. Then $\l(z)=0$ and, by $(\spadesuit)$, 
$\XCB(r)_{\!\neq z_r} \simeq \XCB(r-1)_{\! \neq z_r}$. Using the induction hypothesis and $(\clubsuit)$, 
the exact sequence $(*)$ becomes
$$0 \longto \CM_- \longto \CM_-^{r-1} \longto \Hrm_c^2(\XCB(r)) \longto \CM_+ \oplus \CM_- \longto \CM_+ 
\longto \Hrm_c^3(\XCB(r)) \longto 0 \longto \CM_+ \longto \Hrm_c^4(\XCB(r)) \longto 0.$$
But it follows from $(\diamondsuit)$ that $\Hrm_c^i(\XCB(r))^\s=0$ if $i \in \{2,3\}$. Since the map 
$\CM_+ \to \Hrm^3_c(\XCB(r))$ is surjective, this forces $\Hrm_c^3(\XCB(r))=0$. 
Also, $\Hrm_c^2(\XCB(r))=\CM_-^l$ for some $l$ so, taking the $\CM_-$-isotypic component 
in the above exact sequence yields an exact sequence 
$$0 \longto \CM_- \longto \CM_-^{r-1} \longto \CM_-^l \longto \CM_- \longto 0.$$
So $l=r-1$, as desired. This proves (a).

\medskip

Now, if $z \in \{z_1,z_2,\dots,z_r\}$ then, by symmetry, we may assume that 
$z=z_r$ and the isomorphism $(\spadesuit)$ yields (b) in this case by the induction 
hypothesis. Now, assume that $z \not\in \{z_1,z_2,\dots,z_r\}$. Then $\l(z) \neq 0$ and it 
follows from (a) that the exact sequence $(*)$ becomes
$$0 \longto \CM_- \longto \Hrm_c^2(\XCB(r)_{\! \neq z}) \longto \CM_-^{r-1} \longto \CM_+ 
\longto \Hrm_c^3(\XCB(r)_{\! \neq z}) \longto 0 \longto 0 \longto \Hrm_c^2(\XCB(r)_{\! \neq z}) \longto 
\CM_+ \longto 0.$$
So (b) follows because the maps are $\s$-equivariant.
\end{proof}

\chapter{Type ${\boldsymbol{B_2}}$}\label{chapitre:b2}

\bigskip

\boitegrise{{\bf Assumption and notation.} {\it In \S\ref{chapitre:b2},
we assume that $\dim_\kb V=2$, we fix a 
$\kb$-basis $(x,y)$ of $V$ and we denote by $(X,Y)$ its dual basis. 
Let $s$ and $t$ be the two reflections of $\GL(V)$ whose matrices in the basis 
$(x,y)$ are given by 
$$s = \begin{pmatrix} 0 & 1 \\ 1 & 0 \end{pmatrix}
\qquad\text{and}\qquad t = \begin{pmatrix} -1 & 0 \\ 0 & 1 \end{pmatrix}.$$
We assume moreover that $W=\langle s, t \rangle$: it is a Weyl group of type $B_2$.}}{0.75\textwidth}

\bigskip

\section{The algebra $\Hb$}\label{section: H B2}

\medskip

We set $w=st$, $w'=ts$, $s'=tst$, $t'=sts$ and $w_0=stst=tsts=-\Id_V$. Then
$$W=\{1,s,t,w,w',s',t',w_0\}\qquad\text{and}\qquad \REF(W)=\{s,t,s',t'\}.$$
Moreover,
$$\refw=\{\{s,s'\},\{t,t'\}\}.$$
The matrices of the elements $w$, $w'$, $s'$, $t'$ and $w_0$ in the basis $(x,y)$ are given by
\equat\label{elements B2}
\begin{cases}
w = st=s't' = \begin{pmatrix} 0 & 1 \\ -1 & 0 \end{pmatrix},&\\
w' = ts=t's' = \begin{pmatrix} 0 & -1 \\ 1 & 0 \end{pmatrix},&\\
s' = tst =  \begin{pmatrix} 0 & -1 \\ -1 & 0 \end{pmatrix},&\\
t' = sts = \begin{pmatrix} 1 & 0 \\ 0 & -1 \end{pmatrix},&\\
w_0 = ss'=tt' = \begin{pmatrix} -1 & 0 \\ 0 & -1 \end{pmatrix}.&\\
\end{cases}
\endequat
We set $K_s=-C_s/2$ and $K_t=-C_t/2$ (as in~\S\ref{sub:hecke coxeter}). 
We set $A=-C_s=2K_s$ and $B=-C_t=2K_t$ so that, in $\Hb$, the following relations hold:
\equat\label{relations B2}
\begin{cases}
[x,X] = A(s+s') + 2B t, & \\
[x,Y] = A(s'-s), & \\
[y,X] = A(s'-s), & \\
[y,Y] = A(s+s') + 2B t'. & \\
\end{cases}
\endequat
We deduce for example that 
\equat\label{relations B2 plus}
\begin{cases}
[x,X^2] = A(s+s') X  + A(s-s')Y, & \\
[x,XY] = 2B t Y, & \\
[x,Y^2] = - A (s+s') X - A(s-s') Y, & \\
[y,X^2] = - A(s+s') Y - A (s-s') X, & \\
[y,XY]  = 2B t' X, &\\
[y,Y^2] = A(s+s') Y + A(s-s') X. &\\
\end{cases}
\endequat
Finally, note that $\Hb$ is endowed with an automorphism $\eta$ 
corresponding to $\begin{pmatrix} -1 & 1\\ 1 & 1 \end{pmatrix} \in \NC$:
$$\eta(x)=y-x,\quad\eta(y)=x+y,\quad \eta(X)=\frac{Y-X}{2},\quad \eta(Y)=\frac{X+Y}{2},$$
$$\eta(A)=B,\quad\eta(B)=A,\quad \eta(s)=t\quad \text{and}\quad \eta(s)=t.$$

\bigskip

\section{Irreducible characters}

\medskip

Let $\e_s$ (respectively $\e_t$) denote the unique linear character of $W$ such that 
$\e_s(s)=-1$ and $\e_s(t)=1$ (respectively $\e_t(s)=1$ and $\e_t(t)=-1$). Note that 
$\e_s\e_t=\e$. Let $\chi$ denote the character of the representation $V$ of $W$. Then
$$\Irr(W)=\{1,\e_s,\e_t,\e,\chi\}$$
and the character table of $W$ is given by Table~\ref{table b2}.
\begin{table}\refstepcounter{theo}
\centerline{
\begin{tabular}{@{{{\vrule width 2pt}\,\,\,}}c@{{\,\,\,{\vrule width 2pt}\,\,\,}}c|c|c|c|c@{{\,\,\,{\vrule width 2pt}}}}
\hlinewd{2pt}
\petitespace
$g$ & $~1~$ & $w_0$ & $s$ & $t$ & $w$
\vphantom{$\DS{\frac{a}{A_{\DS{A}}}}$}\\
\hlinewd{1pt}
\petitespace $|\classe_W(g)|$ & $1$ & $1$ & $2$ & $2$ & $2$ \\
\hline
\petitespace $o(g)$ & $1$ & $2$ & $2$ & $2$ & $4$ \\
\hline
\petitespace $\Crm_W(g)$ & $~W~$ & $~W~$ & $\langle s,s'\rangle$ & $\langle t,t'\rangle$ & 
$\langle w \rangle$  \\
\hlinewd{2pt}
\petitespace $1$& $1$ & $1$  & $1$ & $1$ & $1$ \\
\hline
\petitespace $\e_s$& $1$ & $1$  & $-1$ & $1$ & $-1$ \\
\hline
\petitespace $\e_t$& $1$ & $1$  & $1$ & $-1$ & $-1$ \\
\hline
\petitespace $\e$& $1$ & $1$  & $-1$ & $-1$ & $1$ \\
\hline
\petitespace $\chi$& $2$ & $-2$  & $0$ & $0$ & $0$ \\
\hlinewd{2pt}
\end{tabular}}

\bigskip

\caption{Character Table of $W$}\label{table b2}
\end{table}
The fake degrees are given by 
\equat\label{fantome b2}
\begin{cases}
f_1(\tb)=1, & \\
f_{\e_s}(\tb)=\tb^2, & \\
f_{\e_t}(\tb)= \tb^2, & \\
f_\e(\tb)=\tb^4, \\
f_\chi(\tb)=\tb+\tb^3.
\end{cases}
\endequat

\bigskip

\section{Computation of ${\boldsymbol{(V \times V^*)/\Delta W}}$}\label{section:quotient B2}

\medskip

Before computing the center $Z$ of $\Hb$, we will compute its specialization $Z_0$ at 
$(A,B) \mapsto (0,0)$. By Example~\ref{upsilon c=0}, 
$$Z_0=\kb[V \times V^*]^{\Delta W}.$$
Thanks to~(\ref{fantome b2}) and Proposition~\ref{dim bigrad Q0 fantome}, 
the bigraded Hilbert series of $Z_0$ is given by 
\equat\label{hilbert}
\dim_\kb^{\BZ\times\BZ}(Z_0)=\frac{1+\tb\ub+ \tb\ub^3 + 2 \tb^2\ub^2 + 
\tb^3\ub + \tb^3\ub^3 + \tb^4\ub^4}{(1-\tb^2)(1-\tb^4)(1-\ub^2)(1-\ub^4)}.
\endequat
Set
$$\s=x^2+y^2,\quad \pi=x^2y^2,\quad \Sig = X^2+Y^2\quad \text{and}\quad 
\Pi=X^2Y^2.$$
Then
\equat\label{invariants B2}
\kb[V^*]^W = \kb[\s,\pi]\qquad\text{and}\qquad\kb[V]^W=\kb[\Sig,\Pi].
\endequat
So the bigraded Hilbert series of 
$P_\bullet = \kb[V \times V^*]^{W \times W} = 
\kb[V]^W\otimes\kb[V^*]^W = \kb[\s,\pi,\Sig,\Pi]$ 
is given by
\equat\label{hilbert double}
\dim_\kb^{\BZ\times\BZ}(P_\bullet)=\frac{1}{(1-\tb^2)(1-\tb^4)(1-\ub^2)(1-\ub^4)}.
\endequat
Now, set
$$\euler_0=xX + yY,\hskip2mm \euler_0'=(xY+yX) XY,\hskip2mm 
\euler_0''=xy(xY+yX),\hskip2mm \euler_0'''=xy(xX+yY)XY,$$
$$\delb_0=xyXY,\quad\delb_0'=(x^2-y^2)(X^2-Y^2)\quad
\text{and}\quad \Delb_0=xy(x^2-y^2)XY(X^2-Y^2).$$
It is then easy to check that the family 
$(1,\euler_0,\euler_0',\euler_0'',\euler_0''',\delb_0,\delb_0',\Delb_0)$ 
is $P_\bullet$-linearly independent and is contained in 
$\kb[V \times V^*]^{\Delta W}$. 
On the other hand, $1$, $\euler_0$, $\euler_0'$, $\euler_0''$, $\euler_0'''$, 
$\delb_0$, $\delb_0'$ and $\Delb_0$ have respective bidegrees 
$(0,0)$, $(1,1)$, $(1,3)$, $(3,1)$, $(3,3)$, $(2,2)$, $(2,2)$ and $(4,4)$. So the 
bigraded Hilbert series of the free $P_\bullet$-module 
with basis $(1,\euler_0,\euler_0',\euler_0'',\euler_0''',\delb_0,\delb_0',\Delb_0)$ 
is equal to the one of $Z_0$ (see~(\ref{hilbert}) and~(\ref{hilbert double})). 
Hence
\equat\label{invariants diagonaux B2}
\kb[V \times V^*]^{\Delta W} = P_\bullet \oplus P_\bullet\, \euler_0 \oplus 
P_\bullet\, \euler_0' \oplus P_\bullet\, \euler_0'' \oplus P_\bullet\, \euler_0''' 
\oplus P_\bullet\, \delb_0 \oplus P_\bullet\, \delb_0' \oplus P_\bullet\, \Delb_0.
\endequat
The following result, already known (see for instance~\cite{alev}), describes the algebra $Z_0$:

\bigskip

\begin{theo}\label{theo:invariants B2}
$Z_0=\kb[V \times V^*]^{\Delta W} = \kb[\s,\pi,\Sig,\Pi,\euler_0,\euler_0',\euler_0'',\delb_0]$ 
and the ideal of relations is generated by the following ones:
$$
\begin{cases}
(1) \hskip2cm &\euler_0 ~\euler_0' = \s\Pi + \Sig ~\delb_0,\\
(2) \hskip2cm &\euler_0 ~\euler_0''= \Sig \pi + \s ~\delb_0,\\
(3) \hskip2cm &\delb_0~\euler_0' = \Pi ~\euler_0'',\\
(4) \hskip2cm &\delb_0~\euler_0''= \pi~\euler_0',\\
(5) \hskip2cm &\delb_0^2=\pi\Pi,\\
(6) \hskip2cm &\euler_0^{\prime 2} = \Pi(4~\delb_0 -\euler_0^2+ \s\Sig ),\\
(7) \hskip2cm &\euler_0^{\prime\prime 2} = \pi(4~\delb_0 -\euler_0^2+ \s\Sig ),\\
(8) \hskip2cm &\euler_0'~\euler_0''=4 \pi \Pi+ \s\Sig~\delb_0 -\delb_0~\euler_0^2,\\
(9) \hskip2cm &\euler_0(4~\delb_0 -\euler_0^2+ \s\Sig )=\s ~\euler_0'+\Sig~\euler_0''.
\end{cases}
$$
Moreover, $Z_0=P\oplus P\euler_0\oplus P\euler^2_0 \oplus P\delb_0 \oplus P\delb_0 \euler_0 
\oplus P\delb_0\euler_0^2 \oplus P\euler_0' \oplus P \euler_0''$.
\end{theo}

\begin{proof}
It is easily checked that 
\equat\label{autres invariants}
\delb_0'=2~\euler_0^2 - \s\Sig - 4~\delb_0,\quad
\Delb_0=\delb_0~\delb_0'\quad\text{and}\quad \euler_0''' = \delb_0~\euler_0.
\endequat
According to~(\ref{invariants diagonaux B2}), these three relations imply immediately that 
$\kb[V \times V^*]^{\Delta W} = \kb[\s,\pi,\Sig,\Pi,\euler_0,\euler_0',\euler_0'',\delb_0]$. 
This shows the first statement.

\medskip

The relations given in Theorem~\ref{theo:invariants B2} follow from direct computations. 
Taking~(5) into account, the relation~(8) can be rewritten
$$\euler_0'~\euler_0''=\delb_0(4~\delb_0 + \s\Sig -\euler_0^2),\leqno{(8')}$$
whereas~(6) and~(7) imply 
$$\pi ~\euler_0^{\prime 2} = \Pi ~\euler_0^{\prime\prime 2}.\leqno{(10)}$$
Let $E$, $E'$, $E''$ and $D$ be indeterminates over the field 
$\kb(\s,\pi,\Sig,\Pi)$ and let 
$$\rho : \kb[\s,\pi,\Sig,\Pi,E,E',E'',D] \longto \kb[V \times V^*]^{\Delta W}$$
denote the unique morphism of $\kb$-algebras which sends the sequence 
$(\s,\pi,\Sig,\Pi,E,E',E'',D)$ 
to $(\s,\pi,\Sig,\Pi,\euler_0,\euler_0',\euler_0'',\delb_0)$. 
Then $\r$ is surjective. Let $f_i$ denote the element of 
$\kb[\s,\pi,\Sig,\Pi,E,E',E'',D]$ corresponding to the relation~$(i)$ of the Theorem 
(for $1 \le i \le 9$), by subtracting the right-hand side to the left-hand side.
Set
$$\IG = \langle f_1,f_2,f_3,f_4,f_5,f_6,f_7,f_8,f_9 \rangle ~\subset \Ker \r.$$
Let $\Zti_0=\kb[\s,\pi,\Sig,\Pi,E,E',E'',D]/\IG$ and let $e$, $e'$, $e''$ and $d$ 
denote the respective images of $E$, $E'$, $E''$ and $D$ in $\Zti_0$. Set 
$$\Zti_0'=P_\bullet + P_\bullet\, e + P_\bullet\, e^2 + P_\bullet\, e' 
+ P_\bullet\, e'' + P_\bullet\, d + P_\bullet\, de + P_\bullet\, de^2.$$
Then $\Zti_0'$ is a $\kb$-vector subspace of 
$\Zti_0$. The relations given by $(f_i)_{1 \le i \le 9}$ 
show that $\Zti_0'$ is a $\kb$-subalgebra of $\Zti_0$. 
As moreover $\s$, $\pi$, $\Sig$, $\Pi$, $e$, $e'$, $e''$ and $d$ belong to $\Zti_0'$, 
we deduce that $\Zti_0=\Zti_0'$. 

Consequently, $\Zti_0$ is a quotient of the graded $\kb$-vector space
$$\EC=P_\bullet \oplus P_\bullet[-2] \oplus (P_\bullet[-4])^3 \oplus P_\bullet[-6] 
\oplus P_\bullet[-8],$$
and $Z_0$ is a quotient of $\Zti_0$. As the Hilbert series of the $\kb$-vector space 
$\EC$ is equal to the Hilbert series of $Z_0$, we deduce that 
$$\Zti_0'=P_\bullet \oplus P_\bullet e \oplus P_\bullet e^2 \oplus P_\bullet e' 
\oplus P_\bullet e'' \oplus P_\bullet d \oplus P_\bullet de \oplus P_\bullet de^2$$
and that $\Zti_0 \simeq Z_0$. This shows that $\Ker \r = \IG$, as desired.
\end{proof}

\bigskip

\begin{coro}\label{intersection complete}
The relations $(1)$, $(2)$,\dots, $(9)$ is a {\bfit minimal} system of relations. 
In particular, the $\kb$-algebra $Z_0=\kb[V \times V^*]^{\Delta W}$ is not complete intersection.
\end{coro}

\begin{proof}
Let us use the notation of the proof of Theorem~\ref{theo:invariants B2}. 
It is sufficient to show that $(f_i)_{1 \le i \le 9}$ is a minimal system of generators of 
$\IG$. Let $\Zba_0=Z_0/\langle \s,\pi,\Sig,\Pi\rangle$ 
and let $e$, $e'$, $e''$ and $d$ denote the respective images of $\euler_0$, 
$\euler_0'$, $\euler_0''$ and $\delb_0$ in $\Zba_0$. Then it follows from~(\ref{invariants B2}) 
and from the relations~(\ref{autres invariants}) that 
$$\Zba_0 = \kb \oplus \kb e \oplus \kb e^2 \oplus \kb e' \oplus \kb e'' \oplus \kb d 
\oplus \kb de \oplus \kb de^2.$$
Let $\fba_i \in \kb[E,E',E'',D]$ denote the reduction of the polynomial 
$f_i$ modulo $\langle \s,\pi,\Sig,\Pi \rangle$. We only need to show that 
$(\fba_i)_{1 \le i \le 9}$ is a minimal system of generators of the kernel 
of the morphism of $\kb$-algebras 
$$\rhoba : \kb[E,E',E'',D] \longto \Zba_0$$
which sends $E$, $E'$, $E''$ and $D$ on $e$, $e'$, $e''$ and $d$ respectively. 

\medskip

The algebra $N=k[E,E',E'',D]$ is bigraded, with 
$E$, $E'$, $E''$ and $D$ of respective bidegrees $(1,1)$, $(1,3)$, $(3,1)$
and $(2,2)$, and the elements $\bar{f}_1,\ldots,\bar{f}_9$ are homogeneous 
of respective bidegrees $(2,4)$, $(4,2)$, $(3,5)$, $(5,3)$,
$(4,4)$, $(2,6)$, $(6,2)$, $(4,4)$, $(3,3)$. We deduce that 
$$\bigl(\sum_{i=1}^9\mathbf{k}\bar{f_i}\bigr)\cap
\bigl(\sum_{i=1}^9 N_+ \bar{f}_i\bigr)\subset
\bigl(\sum_{i=1}^9\mathbf{k}\bar{f_i}\bigr)
\cap\bigl(\sum_{i=1}^9 \mathbf{k}E\bar{f}_i\bigr)$$
Since all these spaces are bigraded, this intersection is contained in 
$$(\mathbf{k}\bar{f}_3)\cap(\mathbf{k}E\bar{f}_1)+
(\mathbf{k}\bar{f}_4)\cap(\mathbf{k}E\bar{f}_2)+
(\mathbf{k}\bar{f}_5+\mathbf{k}\bar{f}_8)\cap (\mathbf{k}Ef_9).$$
Since $E$ divides neither $\bar{f}_3$ nor $\bar{f}_4$, nor any non-zero element 
of $\mathbf{k}\bar{f}_5+\mathbf{k}\bar{f}_8$, we conclude that 
$\bigl(\sum_{i=1}^9\mathbf{k}\bar{f_i}\bigr)\cap
\bigl(\sum_{i=1}^9 N_+ \bar{f}_i\bigr)=0$, so $(\bar{f}_i)_{1\le i\le 9}$
is a minimal system of generators of $\Ker\bar{\rho}$.
\end{proof}

\bigskip

\begin{coro}\label{euler non}
The minimal polynomial of $\euler_0$ over $P_\bullet$ is
$$\tb^8 - 2\s\Sig~ \tb^6 + 
\bigl(\s^2\Sig^2 + 2(\s^2\Pi + \Sig^2\pi - 8 \pi\Pi)\bigr)~ \tb^4 - 
2\s\Sig\,(\s^2\Pi + \Sig^2\pi - 8\pi\Pi)~ \tb^2 + (\s^2\Pi - \Sig^2\pi)^2.$$
\end{coro}

\begin{proof}
By multiplying the relation~(9) by $\euler_0$ and by using the relations~(1) and~(2), 
we get 
$$\euler_0^2(4\delb_0+\s\Sig-\euler_0^2) = \s^2 \Pi + \Sig^2 \pi + 2\s\Sig \delb_0.$$
We deduce immediately that 
$$\delb_0 (4\euler_0^2 - 2 \s\Sig) = \euler_0^4 - \s\Sig \euler_0^2 + \s^2\Pi +\Sig^2\pi.$$
Taking the square of this relation and using the relation~(5), 
we get that the polynomial of the Corollary vanishes at $\euler_0$. 
The degree of the minimal polynomial of $\euler_0$ over $P_\bullet$ being equal to $|W|=8$, 
the proof of the Corollary is complete.
\end{proof}

\bigskip

\section{The algebra ${\boldsymbol{Z}}$}\label{section:Q B2}

\medskip

Recall that
$$\euler=xX + yY - A(s+s') - B(t+t')$$
and set
$$\begin{cases}
\euler'=(xY+yX)XY + A(s-s') XY - BtY^2 - Bt'X^2,\\
\euler''=xy(xY+yX) + A xy (s-s') - B y^2 t - B x^2 t',\\
\delb=xyXY - B x t' X - B y t Y + B^2 (1+w_0) + AB (w+w').\\
\end{cases}$$
A brute force computation shows that
\equat\label{Q}
\euler, \euler',\euler'', \delb \in Z=\Zrm(\Hb)
\endequat
and that the following equalities hold:
\equat\label{relations Q}
\begin{cases}
(\Zrm 1) \hskip1cm &\euler ~\euler' = \s\Pi + \Sig ~\delb,\\
(\Zrm 2) \hskip1cm &\euler ~\euler''= \Sig \pi + \s ~\delb,\\
(\Zrm 3) \hskip1cm &\delb~\euler' = \Pi ~\euler'' + B^2 \Sig~\euler,\\
(\Zrm 4) \hskip1cm &\delb~\euler''= \pi~\euler' + B^2 \s~\euler,\\
(\Zrm 5) \hskip1cm &\delb^2=\pi\Pi + B^2 ~\euler^2,\\
(\Zrm 6) \hskip1cm &\euler^{\prime 2} = 
\Pi(4~\delb -\euler^2+ \s\Sig + 4A^2-4B^2) + B^2 \Sig^2,\\
(\Zrm 7) \hskip1cm &\euler^{\prime\prime 2} = 
\pi(4~\delb -\euler^2+ \s\Sig + 4A^2-4B^2)+ B^2 \s^2,\\
(\Zrm 8) \hskip1cm &\euler'~\euler''=
\delb(4~\delb -\euler^2+ \s\Sig + 4A^2-4B^2)-B^2\s\Sig,\\
(\Zrm 9) \hskip1cm &\euler(4~\delb -\euler^2+ \s\Sig  + 4A^2-4B^2)=
\s ~\euler'+\Sig~\euler''.
\end{cases}
\endequat
We immediately see that $\euler_0$, $\euler_0'$, $\euler_0''$ and 
$\delb_0$ are the respective images, in $Z_0=Z/\pG_0 Z$, of the elements 
$\euler$, $\euler'$, $\euler''$ and $\delb$. On the other hand, the 
relations (1), (2),\dots, (9) of Theorem~\ref{theo:invariants B2} are also the images, 
modulo $\pG_0$, of the relations (Z1), (Z2),\dots, (Z9).

\bigskip

\begin{theo}\label{theo:Q}
The $\kb$-algebra $Z$ is generated by $A$, $B$, $\s$, $\pi$, $\Sig$, $\Pi$, 
$\euler$, $\euler'$, $\euler''$ and $\delb$. The ideal of relations is generated by 
$(\Zrm 1)$, $(\Zrm 2)$,\dots, $(\Zrm 9)$. 

Moreover, $Z=P\oplus P\euler\oplus P\euler^2 \oplus P\delb \oplus P\delb \euler 
\oplus P\delb\euler^2 \oplus P\euler' \oplus P \euler''$.
\end{theo}

\begin{proof}
The proof follows exactly the same arguments as the ones of 
Theorem~\ref{theo:invariants B2}, based in part on 
comparisons of bigraded Hilbert series.
%
%
\end{proof}

\bigskip

\begin{coro}\label{Q:intersection complete}
The relations $(\Zrm 1)$, $(\Zrm 2)$,\dots, $(\Zrm 9)$ is a minimal system of relations. 
In particular, the $\kb$-algebra $Z$ is not complete intersection.
\end{coro}

\begin{proof}
This follows immediately from Theorem~\ref{theo:Q}, using the same arguments 
as in the proof of Corollary~\ref{intersection complete}. 
\end{proof}

\bigskip

\begin{coro}\label{euler non pluss}
The minimal polynomial of $\euler$ over $P$ is
\begin{multline*}
 \tb^8 - 2(\s\Sig+4A^2+4B^2)~\tb^6 
+ \bigl(\s^2\Sig^2 + 2(\s^2\Pi + \Sig^2\pi -8\pi\Pi) + 8 (A^2+B^2)\s\Sig +16(A^2-B^2)^2\bigr)~\tb^4 \\
- 2\bigl((\s\Sig+4A^2-4B^2)(\s^2\Pi + \Sig^2\pi)-8\s\Sig\pi\Pi+2B^2 \s^2\Sig^2\bigr)~\tb^2 
+ (\s^2\Pi - \Sig^2\pi)^2.
\end{multline*}
\end{coro}

\begin{proof}
The proof follows exactly the same steps as the proof of Corollary~\ref{euler non}, 
but by starting with the relations $(\Zrm 1)$,\dots, $(\Zrm 9)$ instead of the relations 
(1),\dots, (9).
\end{proof}

\bigskip

\noindent{\bf Acknowledgments --- } 
The above computations (checking that $\euler$, $\euler'$, $\euler''$ and $\delb$ are central, 
and checking the relations $(\Zrm 1)$,\dots, $(\Zrm 9)$) have been done without computer. 
Even though we have been very carefully, the heaviness of the computations imply that 
it might happen that some mistakes occur. However, U. Thiel has developed a {\tt MAGMA} package 
(called {\tt CHAMP}, see~\cite{thiel champ}) for computing in the algebra $\Hb$: hence, he has checked 
{\it independently} that the elements $\euler$, $\euler'$, $\euler''$ and $\delb$ 
are central and that the relations $(\Zrm 1)$,\dots, $(\Zrm 9)$ hold. We wish to thank warmly U. Thiel 
for this checking: he has also checked that the minimal polynomial of $\euler$ 
is given by Corollary~\ref{euler non pluss}.\finl

\bigskip

\section{Calogero-Moser families}

\medskip

The Table~\ref{table omega} gives the values of $\O_\psi$ (for $\psi \in \Irr(W)$) at the generators 
of the $P$-algebra $Z$. They are obtained by computing effectively 
the actions of $\euler$, $\euler'$, $\euler''$ and $\delb$ 
or by using Corollary~\ref{action euler verma} and using the 
relations $(\Zrm 1)$,\dots, $(\Zrm 9)$ 
(knowing that $\O_\psi(\s)=\O_\psi(\pi)=\O_\psi(\Sig)=\O_\psi(\Pi)=0$). 

\bigskip

\begin{table}\refstepcounter{theo}
\centerline{
\begin{tabular}{@{{{\vrule width 2pt}\,\,\,}}c@{{\,\,\,{\vrule width 2pt}\,\,\,}}c|c|c|c@{{\,\,\,{\vrule width 2pt}}}}
\hlinewd{2pt}
\petitespace
$z \in Z$ & $~\euler~$ & $\euler'$ & $\euler''$ & $\delb$ 
\vphantom{$\DS{\frac{a}{A_{\DS{A}}}}$}\\
\hlinewd{2pt}
\petitespace $\O_1 $     & $2(B+A)$ & $0$ & $0$ & $2B(B+A)$ \\
\hline
\petitespace $\O_{\e_s}$ & $2(B-A)$ & $0$ & $0$ & $2B(B-A)$ \\
\hline
\petitespace $\O_{\e_t}$ & $-2(B-A)$ & $0$ & $0$ & $2B(B-A)$ \\
\hline
\petitespace $\O_\e$     &$-2(B+A)$ & $0$ & $0$ & $2B(B+A)$ \\
\hline
\petitespace $\O_\chi$   & $0$      & $0$ & $0$ & $0$ \\
\hlinewd{2pt}
\end{tabular}}

\bigskip

\caption{Table central characters of $\Hbov$}\label{table omega}
\end{table}

\bigskip

Now, let $K$ be a field and fix a morphism $\kb[\CCB] \to K$. 
Let $a$ and $b$ denote the respective images of $A$ and $B$ in $K$. 
The previous Table allows to compute immediately the partitions of $\Irr(W)$ 
into Calogero-Moser $K$-families, according to the values of $a$ and $b$. 
The (well-known) result is given in Table~\ref{table calogero}.

\bigskip

\begin{table}\refstepcounter{theo}
\centerline{
\begin{tabular}{@{{{\vrule width 2pt}\,\,\,}}c@{{\,\,\,{\vrule width 2pt}\,\,\,}}c@{{\,\,\,{\vrule width 2pt}}}}
\hlinewd{2pt}
\petitespace
Conditions & $K$-families
\vphantom{$\DS{\frac{a}{A_{\DS{A}}}}$}\\
\hlinewd{2pt}
\petitespace $a=b=0$     & $\Irr(W)$ \\
\hline
\petitespace $a=0$, $b \neq 0$ & $\{1,\e_s\}$,\quad $\{\e_t,\e\}$ \quad 
et \quad$\{\chi\}$ \\
\hline
\petitespace $a \neq 0$, $b=0$ & $\{1,\e_t\}$,\quad $\{\e_s,\e\}$ \quad 
et \quad $\{\chi\}$ \\
\hline
\petitespace $a=b \neq 0$     &  $\{1\}$,\quad $\{\e\}$ \quad 
et \quad $\{\e_s,\e_t,\chi\}$ \\
\hline
\petitespace $a=-b \neq 0$     &  $\{\e_s\}$,\quad $\{\e_t\}$ \quad 
et \quad $\{1,\e,\chi\}$ \\
\hline
\petitespace $ab \neq 0$, $a^2\neq b^2$     &  
$\{1\}$,\quad $\{\e_s\}$,\quad $\{\e_t\}$,\quad $\{\e\}$ \quad 
et \quad $\{\chi\}$ \\
\hlinewd{2pt}
\end{tabular}}

\bigskip

\caption{Calogero-Moser $K$-families}\label{table calogero}
\end{table}

\bigskip

\section{The group ${\boldsymbol{G}}$}

\medskip

Since $w_0=-\Id_V$ belongs to $W$ (and since all the reflections of $W$ have order $2$), 
the results of \S~\ref{sec:w0} can be applied. 
In particular, if $\t_0=(-1,1,\e) \in \kb^\times \times \kb^\times \times W^\wedge$, 
then $\t_0$ can be seen as the element $w_0 \in W \longinjto G$ and is central in $G$ 
(see Proposition~\ref{tau 0 in G}). Hence, by~(\ref{inclusion w0}), we get 
$$G \subset W_4 ,$$
where $W_4$ is the subgroup of $\SG_W$ consisting of permutations $\s$ of $W$ 
such that $\s(-x)=-\s(x)$ for all $x \in W$. We denote by $N_4$ the (normal) subgroup 
of $W_4$ consisting of permutations $\s \in W_4$ 
such that $\s(x)\in \{x,-x\}$ for all $x \in W$. Then, if we set $\mu_2=\{1,-1\}$, 
$$N_4 \simeq (\mu_2)^4.$$
Moreover,
$$|W_4|=384\qquad\text{and}\qquad |N_4|=16.$$
Let $\e_W : \SG_W \to \mu_2=\{1,-1\}$ be the sign character and let $W_4' = W_4 \cap \Ker \e_W$ and  
$N_4'=W_4' \cap N_4$. Then
$$N_4' =\{(\eta_1,\eta_2,\eta_3,\eta_4) \in (\mu_2)^4~|~\eta_1\eta_2\eta_3\eta_4=1\} \simeq (\mu_2)^3.$$
Moreover,
$$|W_4'|=192\qquad\text{and}\qquad |N_4'|=8.$$
Recall that $H$ is identified with the stabilizer, 
in $G$, of $1 \in W$. Moreover, $G$ contains the image of $W \times W$ in $\SG_W$. 
This image, isomorphic to $(W \times W)/\D\Zrm(W)$, has order $32$ and its intersection 
with $H$, isomorphic to $\D W/\D\Zrm(W) \simeq W/\Zrm(W)$, 
has order $4$. The elements $(s,s)$ and $(t,t)$ 
of $W \times W$ are sent to distinct elements of $N_4$. So $H \cap N_4$ is a subgroup of 
$N_4'$ of order $4$. Since $(w_0,1)$ is sent to an element of 
$N_4'$ which does not belong to $H$, we deduce that 
$$N_4' \subset G.$$

Let $f(\tb) \in P[\tb]$ denote the unique monic polynomial of degree $4$ such that $f(\euler^2)=0$ 
(it is given by Corollary~\ref{euler non pluss}). According to~(\ref{discriminant carre}), we have 
$$\disc(f(\tb^2)) = 256 ~\disc(f)^2 \cdot (\s^2\Pi-\Sig^2\pi)^2,$$
and so the discriminant of the minimal polynomial of $\euler$ is a square in $P$. Hence, 
$$G \subset W_4'.$$
We will show that this inclusion is an equality.

\bigskip

\begin{theo}\label{galois B2}
$G=W_4'$.
\end{theo}

\begin{proof}
It is sufficient to show that $|G|=192$. We already know that 
$N_4' \subset G \subset W_4'$, which shows that $G \cap N_4 = N_4'$. 
To show the Theorem, it is sufficient to show that $G/N_4' \simeq \SG_4$. 
But, $G/N_4'=G/(G \cap N_4)$ is the Galois group of the polynomial 
$f$. So we only need to show that the Galois group of $f$ over 
$\Kb$ is $\SG_4$. Let $\Gba$ denote this Galois group. 

Let $\pG=\langle \s-2, \Sig +2,A-1,B,\Pi-\pi \rangle$. Then $\pG$ is a prime ideal and 
$P/\pG \simeq \kb[\pi]$. Let $\fba$ denote the reduction of 
$f$ modulo $\pG$. Then
$$\fba(\tb)= \tb(\tb^3 + ( 16\,\pi-16\,{\pi}^{2}) \,\tb-64\,{\pi}^{2}).$$
So, by~(\ref{discriminant t}), we have 
$$\disc(\fba)=(64 \pi^2)^2 \cdot \left(- 4 ( 16\,\pi-16\,{\pi}^{2})^3 - 27 \cdot (-64\,{\pi}^{2})^2\right)=2^{24} \,\pi^7 (\pi-4)(2\,\pi+1)^2.$$
So the discriminant of $\fba$ is not a square in $\kb[\pi]$, 
which implies that the discriminant of $f$ is not a square in 
$P$. So $\Gba$ is not contained in the alternating group $\AG_4$.

Since $f$ is irreducible, $\Gba$ is a transitive subgroup of $\SG_4$. In particular, 
$4$ divides $|\Gba|$. Moreover, 
if $c \in \CCB$ is such that $c_s=c_t=1$, then, by Table~\ref{table calogero} 
and Theorem~\ref{theo cellules familles}, 
$G$ admits a subgroup (the inertia group of $\rGba_c$) which admits 
an orbit of length $6$. So $3$ divides $|G|$ and so $3$ divides also $|\Gba|$. Hence, 
$12$ divides $|\Gba|$ and, since $\Gba \not\subset \AG_4$, this forces 
$\Gba = \SG_4$.
\end{proof}

\begin{rema}
Recall that $W_4$ is a Weyl group of type $B_4$ and that 
$W_4'$ is a Weyl group of type $D_4$.\finl
\end{rema}

\section{Calogero-Moser cells, Calogero-Moser cellular characters}
\label{se:cellsB2}
\medskip

\subsection{Results} 
The aim of this section is to show that the Conjectures~\BIL~and~\GAUCHE~hold for $W$. 
If $a$ and $b$ are positive real numbers and if $c_s=-a$ and $c_t=-b$, 
then the description of Kazhdan-Lusztig left, right or two-sided $c$-cells, 
of Kazhdan-Lusztig $c$-families and $c$-cellular characters is easy and can be found, 
for instance, in~\cite{lusztig}. The different cases to be considered are $a > b$, $a=b$ and 
$a < b$: by using the automorphism $\eta$ of $W$ which exchanges $s$ and $t$, we can 
assume that $a \ge b > 0$. The Conjectures~\BIL~and~\GAUCHE~then follow from the description 
of Calogero-Moser left, right or two-sided $c$-cells, 
of Calogero-Moser $c$-families and $c$-cellular characters
given in Table~\ref{table:conjectures-b2}:

\medskip

\begin{theo}\label{theo:conjectures-b2}
Let $c \in \CCB$, set $a=-c_s$ and $b = -c_t$ and assume that $ab\neq 0$. Then 
there exists a choice of prime ideals $\rG_c^\gauche \subset \rGba_c$ such that 
the Calogero-Moser left, right or two-sided $c$-cells, 
of Calogero-Moser $c$-families and $c$-cellular characters
given by Table~\ref{table:conjectures-b2}.

Consequently, Conjectures~\BIL~and~\GAUCHE~hold if $W$ has type $B_2$.
\end{theo}

\medskip

\noindent{\sc Notation - } 
In Table~\ref{table:conjectures-b2}, we have set
$$\G_\chi=\{t,st,ts,sts\},~ \G_\chi^+=\{t,st\},~ \G_\chi^-=\{ts,sts\},~
\G_s=\{s,ts,sts\}\text{~et~} \G_t=\{t,st,tst\}.$$
Moreover:
\begin{itemize}
\item[$\bullet$] $W_3'=H$ denotes the stabilizer of $1 \in W$ in $G=W_4'$ and $W_2'$ 
denotes the stabilizer of $s$ in 
$W_3'$. Note that $W_3'$ (respectively $W_2'$) is a Weyl group of type 
$D_3=A_3$ (respectively $D_2=A_1 \times A_1$). 

\item[$\bullet$] $\SG_3$ denotes the subgroup of $W_3'$ which stabilizes $\G_s$ 
(this is also the stabilizer of $\G_t$): it is isomorphic to the symmetric group 
of degree $3$. 

\item[$\bullet$] $\BZ/2\BZ$ denotes the stabilizer, in $W_2'$, of $\G_\chi^+$ (or $\G_\chi^-$).\finl
\end{itemize}

\vskip1cm
\begin{table}\refstepcounter{theo}\label{table:conjectures-b2}
{\footnotesize
\centerline{\begin{tabular}{@{{{\vrule width 2pt}}}c@{{{\vrule width 2pt}}}c@{{{\vrule width 1pt}}}c|c|c@{{{\vrule width 1pt}}}c|c
@{{{\vrule width 1.5pt}}}c|c@{{{\vrule width 1pt}}}c|c|c@{{{\vrule width 2pt}}}}
\hlinewd{2pt}
\multirow{2}{*}{~Conditions~} & 
\multirow{2}{*}{~$\Dba_c=\Iba_c$~} & 
\multicolumn{3}{c@{{{\vrule width 1pt}}}}{\trespetitespace Two-sided cells} & 
\multirow{2}{*}{$\psi$} & \multirow{2}{*}{$\dim_\kb \LC_c(\psi)$~} & 
\multirow{2}{*}{$I_c^\gauche$} & \multirow{2}{*}{$D_c^\gauche$} & 
\multicolumn{3}{c@{{{\vrule width 2pt}}}}{Left cells} \\
\cline{3-5}  \cline{10-12}
&& \trespetitespace$\G$ & $|\G|$ & $\Irr_\G(W)$ & && &&  $C$ & $|C|$ & 
$\isomorphisme{C}_c^\calo$  \\
\hlinewd{2pt}
\multirow{6}{*}{$\stackrel{\DS{a^2 \neq b^2}}{ab \neq 0}$} & \multirow{5}{*}{~$W_2'$~} 
& \trespetitespace$1$ & 1 & $\unb_W$ & ~$\unb_W$~ & $8$ & 
\multirow{6}{*}{~$\BZ/2\BZ$} & \multirow{6}{*}{$W_2'$~} & 1 & 1 & $\unb_W$ \\
\clinewd{0.1pt}{3-7} \clinewd{0.1pt}{10-12}
&& \trespetitespace$w_0$ & 1 & $\e$ & $\e$ & $8$ & && 
$w_0$ & 1 & $\e$  \\
\clinewd{0.1pt}{3-7} \clinewd{0.1pt}{10-12}
&& \trespetitespace$s$ & 1 & $\e_s $  & $\e_s$ & $8$ & && 
$s$ & 1 & $\e_s$ \\
\clinewd{0.1pt}{3-7} \clinewd{0.1pt}{10-12}
&& \trespetitespace$w_0s$ & 1 & $\e_t $  & $\e_t$ & $8$ & && 
$w_0 s$ & 1 & $\e_t$ \\
\clinewd{0.1pt}{3-7} \clinewd{0.1pt}{10-12}
&& \multirow{2}{*}{$\G_\chi$} & \multirow{2}{*}{4} & \multirow{2}{*}{$\chi$} & 
\multirow{2}{*}{$\chi$} & \multirow{2}{*}{$8$} 
&&& \trespetitespace$\G_\chi^+$ & 2 & $\chi$ 
\\
\clinewd{0.1pt}{10-12}
&&&&&&&&& \trespetitespace$\G_\chi^-$ & 2 & $\chi$ 
\\
\hlinewd{2pt}
\multirow{5}{*}{$\stackrel{\DS{a=b}}{ab \neq 0}$} & \multirow{5}{*}{~$W_3'$~} & \trespetitespace$1$ & 1 & $\unb_W$ & ~$\unb_W$~ & $8$ 
& \multirow{5}{*}{$\SG_3$} & \multirow{5}{*}{$\SG_3$} & 
1 & 1 & $\unb_W$ \\
\clinewd{0.1pt}{3-7} \clinewd{0.1pt}{10-12}
&& \trespetitespace$w_0$ & 1 & $\e$ & $\e$ & $8$ & && 
$w_0$ & 1 & $\e$  \\
\clinewd{0.1pt}{3-7} \clinewd{0.1pt}{10-12}
&& \multirow{3}{*}{~$W \setminus\{1,w_0\}$} & \multirow{3}{*}{6} & 
\multirow{3}{*}{$\e_s, \e_t, \chi$} & \trespetitespace$\e_s$ & $1$ & && 
$\G_s$ & 3 & $\e_s + \chi$  \\
\clinewd{0.1pt}{6-7} \clinewd{0.1pt}{10-12}
&& & && \trespetitespace$\e_t$ & $1$ & && 
$\G_t$ & 3 & $\e_t+\chi$  \\
\clinewd{0.1pt}{6-7} 
&& & && \trespetitespace$\chi$ & $6$ & &&&&\\
\hlinewd{2pt}
\end{tabular}}
}

\vskip0.5cm

\caption{Calogero-Moser cells, families, cellular characters}
\end{table}

We will now concentrate on the proof of Theorem~\ref{theo:conjectures-b2}: 
we will first start by the generic case, by running over the descending chain of prime ideals 
$\pGba \supset \pG^\gauche \supset \langle \pi \rangle$. The use of the ideal $\langle \pi \rangle$ 
will help us to remove some ambiguity for the computation of Calogero-Moser 
left cells. It is natural to ask whether this method can be extended, since the prime ideal 
$\langle \pi \rangle$ is not chosen randomly: 
it is the defining ideal of a $W$-orbit of hyperplanes in $V^*$. 

After studying the generic case, we will specialize our parameters to deduce 
Theorem~\ref{theo:conjectures-b2}. The most difficult step is the computation 
of left cells (see for instance Proposition~\ref{prop:b2-gauche-enfin}). 

\bigskip

\boitegrise{{\bf Notation.} {\it If $z \in Z$ (or $q \in Q$), we denote by $F_z(\tb)$ (or $F_q(\tb)$) 
the minimal polynomial of $z$ (or $q$) over $P$. If $F(\tb) \in P[\tb]$, we will denote by 
$\Fba(\tb)$ (respectively $F^\gauche(\tb)$, respectively $F^\pi(\tb)$) 
the reduction of $F(\tb)$ modulo $\pGba$ (respectively 
$\pG^\gauche$, respectively $\langle \pi \rangle$).}}{0.75\textwidth}

\bigskip

\subsection{Generic two-sided cells} 
%
%
%
We set $\eulerq=\copie(\euler)$, $\eulerq'=\copie(\euler')$, $\eulerq''=\copie(\euler'')$ and 
$\d=\copie(\delb)$. Recall that  $\pGba=\langle \s\,\pi,\Sig,\Pi \rangle_P$, that 
$\zGba = \Ker(\O_1)$ and that $\qGba=\copie(\zGba)$: according to Table~\ref{table omega}, we have 
\equat\label{eq:qba-b2}
\qGba=\pGba Q + \langle \eulerq -2(A+B), \d-2B(A+B), \eulerq',\eulerq'' \rangle_Q.
\endequat
Also, $Q/\qGba=P/\pGba=\kb[A,B]$. Recall that 
\equat\label{fba euler}
\Fba_\euler(\tb)=\tb^4(\tb+2(A+B))(\tb+2(B-A))(\tb-2(A+B))(\tb-2(B-A)).
\endequat
Recall also that, since $w_0=-\Id_V \in W$ and that $W$ is generated by reflections of order $2$, 
we have $\eulerq_{vw_0}=-\eulerq_v$ for all $v \in W$ (see Proposition~\ref{tau 0 in G}). 

\bigskip

\begin{lem}\label{rba b2}
Let $v \in W \setminus\{1,w_0\}$. Then there exists a unique prime ideal $\rGba$ 
of $R$ lying over $\qGba$ and such that $\eulerq_v \equiv 2(B-A) \mod \rGba$. 
\end{lem}

\begin{proof}
Let us first show the existence statement. 
Let $\rGba'$ be a prime ideal of $R$ lying over $\qGba$: 
then $\eulerq \equiv 2(A+B) \mod \rGba'$ and 
$\eulerq_{w_0} \equiv -2(A+B) \mod \rGba'$. 
By~(\ref{fba euler}), there exists a unique element $v' \in W\setminus\{1,w_0\}$ 
such that $\eulerq_{v'} \equiv 2(B-A) \mod \rGba$. 

Recall also that $H$ is the stabilizer of $1 \in W$ in $G=W_4' \subset \SG_W$: 
this is also the stabilizer of $w_0$. 
Then $H$ acts transitively on $W \setminus \{1,w_0\}$ 
(by Theorem~\ref{galois B2}) and so there exists $\s \in H$ 
such that $\s(v')=v$. Let $\rGba=\s(\rGba')$. Then 
$\rGba$ is a prime ideal of $R$ lying over $\qGba$ (since $\s \in H$) 
and $\eulerq_v \equiv 2(B-A) \mod \rGba$. This concludes the proof of the existence statement. 

\medskip

Let us now show the uniqueness statement. So let $\rGba$ and $\rGba'$ be two prime ideals 
of $R$ lying over $\qGba$ such that 
$\eulerq_v + 2(A-B) \in \rGba \cap \rGba'$. Then there exists 
$\s \in H$ such that $\rGba' =\s(\rGba)$. We then have 
$\eulerq_v \equiv \eulerq_{\s(v)} \equiv 2(A-B) \mod \rGba$. 
By~(\ref{fba euler}), we know that $2(B-A)$ is a simple root of $\fba(\tb)$, 
so $\s(v)=v$. Consequently, $\s \in I$, where 
$I$ is the stabilizer of $v$ in $H$. 
By Theorem~\ref{galois B2}, $I$ is the Klein group acting on 
$W\setminus\{1,w_0,v,vw_0\}$ (note that $|I|=4$). 

Let $\Dba$ (respectively $\Iba$) denote the decomposition 
(respectively inertia) group of $\rGba$ (in $G$). By~(\ref{fba euler}), we have 
$\Iba \subset \Dba \subset I$ and it remains to show that $I =\Iba$. 
But the generic two-sided cell covering the generic Calogero-Moser family 
$\{\chi\}$ has cardinality $\chi(1)^2=4$, and it is an orbit under the action of 
$\Iba$. So $|\Iba| \ge 4 = |I|$. Whence the result.
\end{proof}

\bigskip

As a consequence of the proof of the previous Lemma, we obtain 
the next result:

\bigskip

\begin{coro}\label{dba=iba}
Let $v \in W \setminus \{1,w_0\}$. Let $\rGba$ denote the unique prime ideal of $R$ lying over 
$\qGba$ and such that $\eulerq_v \equiv 2(B-A) \mod \rGba$. 
Let $\Dba$ (respectively $\Iba$) denote the decomposition (respectively 
inertia) group of $\rGba$ in $G$. Then:
\begin{itemize}
\itemth{a} $\Dba=\Iba=\{\t \in G~|~\t(1)=1$ and $\t(v)=v\}
\simeq \BZ/2\BZ \times \BZ/2\BZ$. 

\itemth{b} $R/\rGba = Q/\qGba = P/\pGba \simeq \kb[A,B]$.

\itemth{c} The generic Calogero-Moser two-sided cells 
(with respect to $\rGba$) are $\{1\}$, $\{w_0\}$, $\{v\}$, $\{vw_0\}$ 
and $W\setminus \{1,w_0,v,vw_0\}$. Moreover, 
$\Irr_{\{1\}}(W)=\{\unb_W\}$, $\Irr_{\{w_0\}}(W)=\{\e\}$, 
$\Irr_{\{v\}}(W)=\{\e_s\}$, $\Irr_{\{vw_0\}}(W)=\{\e_t\}$ and $\Irr_{\{1\}}(W)=\{\chi\}$. 
\end{itemize}
\end{coro}

\bigskip

\boitegrise{{\bf Choice.} 
{\it From now on, and until the end of this Chapter, we denote by $\rGba$ 
the unique prime ideal of $R$ lying over $\qGba$ and such that 
$\eulerq_s \equiv 2(B-A) \mod \rGba$.}}{0.75\textwidth}

\bigskip

With this choice, 
\equat\label{eq:dba-b2}
\Dba=\Iba=W_2',
\endequat
and the generic Calogero-Moser two-sided cells are $\{1\}$, $\{s\}$, $\{w_0s\}$, $\{w_0\}$ and 
$\G_\chi$, and 
\equat\label{eq:familles-b2}
\begin{cases}
~\Irr_{\{1\}}^\calo(W)=\{\unb_W\},\\
~\Irr_{\{s\}}^\calo(W)=\{\e_s\},\\
~\Irr_{\{w_0s\}}^\calo(W)=\{\e_t\},\\
~\Irr_{\{w_0\}}^\calo(W)=\{\e\},\\
~\Irr_{\G_\chi}^\calo(W)=\{\chi\}.
\end{cases}
\endequat

\bigskip

\subsection{Generic cellular characters} 
Recall that $\pG^\gauche=\langle \Sig,\Pi \rangle_P$. 

\bigskip

\begin{lem}\label{eq:qgauche-b2}
We have
$\qG^\gauche = \pG^\gauche Q + \langle \eulerq-2(B+A),\eulerq'-B\Sig,\eulerq'',\d-2B(A+B) \rangle_Q$. 
\end{lem}
\begin{proof}
Let $\qG'=\langle \eulerq-2(B+A),\eulerq'-B\Sig,\eulerq'',\d-2B(A+B) \rangle_Q$. 
First of all, note that $Q/\pG^\gauche Q$ is the $P/\pG^\gauche=\kb[A,B,\Sig,\Pi]$-algebra 
admitting the following presentation:
\equat\label{eq:presentation-qgauche-b2}
\begin{cases}
(\Qrm 1^\gauche) \hskip1cm &\eulerq ~\eulerq' = \Sig ~\d,\\
(\Qrm 2^\gauche) \hskip1cm &\eulerq ~\eulerq''= 0,\\
(\Qrm 3^\gauche) \hskip1cm &\d~\eulerq' = \Pi ~\eulerq'' + B^2 \Sig~\eulerq,\\
(\Qrm 4^\gauche) \hskip1cm &\d~\eulerq''= 0,\\
(\Qrm 5^\gauche) \hskip1cm &\d^2= B^2 ~\eulerq^2,\\
(\Qrm 6^\gauche) \hskip1cm &\eulerq^{\prime 2} = 
\Pi(4~\d -\eulerq^2+  4A^2-4B^2) + B^2 \Sig^2,\\
(\Qrm 7^\gauche) \hskip1cm &\eulerq^{\prime\prime 2} = 0,\\
(\Qrm 8^\gauche) \hskip1cm &\eulerq'~\eulerq''=
\d(4~\d -\eulerq^2+ 4A^2-4B^2),\\
(\Qrm 9^\gauche) \hskip1cm &\eulerq(4~\d -\eulerq^2+ 4A^2-4B^2)=\Sig~\eulerq''.
\end{cases}
\endequat
A straightforward computation shows that these relations hold in $Q/\qG'$. 
Let $\qG'' = \pG^\gauche Q + \qG'$. Then 
$Q/\qG'' \simeq \kb[\Sig,\Pi,A,B] \simeq P/\pG^\gauche$, so $\qG''$ is a prime 
ideal of $Q$, containing $\pG^\gauche$ and contained in $\qGba$ (by~(\ref{eq:qba-b2})). 
The result then follows from the uniqueness statement in Corollary~\ref{unicite qgauche}.
\end{proof}

%
%

\bigskip

Recall that the Calogero-Moser cellular characters can be defined without computing 
the Calogero-Moser left cells, by using only prime ideals of $Z$ (or $Q$) 
lying over $\pG^\gauche$.
Note also that
\equat\label{eq:minimal-euler-gauche-b2}
F_\euler^\gauche(\tb)=\tb^4(\tb+2(A+B))(\tb+2(B-A))(\tb-2(A+B))(\tb-2(B-A)).
\endequat
This equality allows us to construct other prime ideals of $Q$ lying over $\pG^\gauche$:

\bigskip

\begin{lem}\label{lem:qgauche-b2}
Set
$$
\begin{cases}
\qG_1^\gauche=\qG^\gauche=\pG^\gauche Q + \langle \eulerq-2(B+A),\eulerq'+B\Sig,\eulerq'',\d-2B(B+A) \rangle_Q,\\
\qG_{\e_s}^\gauche=\pG^\gauche Q + \langle \eulerq-2(B-A),\eulerq'+B\Sig,\eulerq'',\d-2B(B-A) \rangle_Q,\\
\qG_{\e_t}^\gauche=\pG^\gauche Q + \langle \eulerq+2(B-A),\eulerq'-B\Sig,\eulerq'',\d-2B(B-A) \rangle_Q,\\
\qG_{\e}^\gauche=\pG^\gauche Q + \langle \eulerq+2(B+A),\eulerq'-B\Sig,\eulerq'',\d-2B(A+B) \rangle_Q,\\
\qG_\chi^\gauche=\pG^\gauche Q + \langle \eulerq,\eulerq'',\d \rangle_Q.\\
\end{cases}
$$
Then:
\begin{itemize}
\itemth{a} If $\g \in \Irr(W)$, then $\qG^\gauche_\g$ is a prime ideal of $Q$ lying over $\pG^\gauche$. 
The associated generic Calogero-Moser cellular character is $\g$.

\itemth{b} If $\g \in \Hom(W,\kb^\times)$, 
then $\qG^\gauche_\g=\Ker(\O_\g^\gauche)$ and $Q/\qG_\g^\gauche = P^\gauche$. 

\itemth{c} If we denote by $\eulerq_\chi'$ the image of $\eulerq'$ in $Q/\qG_\chi^\gauche$, then 
$Q/\qG^\gauche_\chi=(P/\pG^\gauche)[\eulerq'_\chi]$ and the minimal polynomial of $\eulerq'_\chi$ is 
$\tb^2-\Pi(4A^2-4B^2)-B^2\Sig^2$. In particular, $[k_Q(\qG_\chi^\gauche) : k_P(\pG^\gauche)] = 2$.

\itemth{d} If $\qG$ is a prime ideal $Q$ lying over $\pG^\gauche$, then there exists $\g \in \Irr(W)$ 
such that $\qG=\qG_\g^\gauche$. 
\end{itemize}
\end{lem}

\begin{proof}
(b) is easily checked by a direct computation, and it implies~(a) whenever $\g$ is a linear character.

It follows from the presentation~(\ref{eq:presentation-qgauche-b2}) of $Q/\pG^\gauche Q$ that 
$Q/\qG_\chi^\gauche=(P/\pG^\gauche)[\eulerq'_\chi]$ and that the minimal polynomial of 
$\eulerq'_\chi$ is $\tb^2-\Pi(4A^2-4B^2)-B^2\Sig^2$. Since this last polynomial with coefficients 
in $P/\pG^\gauche=\kb[\Sig,\Pi,A,B]$ is irreducible, this implies that $\qG_\chi^\gauche$ is a prime 
ideal lying over $\pG^\gauche$. We deduce~(c) and the first statement of~(a).  
The last statement of~(a) follows from Theorem~\ref{theo:cellulaire-lisse}. 

%
%
%

\medskip

(d) follows from the fact that the sum of the already constructed Calogero-Moser cellular characters 
is equal to the character of the regular representation of $W$.
\end{proof}

\bigskip

\subsection{Generic left cells}
We will now lift the results about cellular characters to results about left cells. 
The first one is a consequence of Theorem~\ref{lem:qgauche-b2}.

\medskip

\begin{coro}\label{coro:rgauche-b2}
Let $\rG^\gauche$ be a prime ideal of $R$ lying over $\qG^\gauche$ and contained in $\rGba$ and let 
$D^\gauche$ (respectively $I^\gauche$) denote its decomposition (respectively inertia) group. 
Then $D^\gauche=W_2'$ and $|I^\gauche|=2$.
\end{coro}

\begin{proof}
First of all, by Corollary~\ref{coro:dgauche-dbarre}, we have $D^\gauche \subset \Dba=W_2'$. 
It follows from Lemma~\ref{lem:qgauche-b2}(c) that, if $C$ is a generic Calogero-Moser cell 
contained in the generic Calogero-Moser two-sided cell $\G$ covering the family $\{\chi\}$, 
then $|C|=2$ and $|C^\DD|=4$. In particular, 
$2$ divides $|I^\gauche|$ and $I^\gauche \varsubsetneq D^\gauche$ by Proposition~\ref{prop:gauche-bilatere}(b). 
The Corollary follows.
\end{proof}

\bigskip

\begin{coro}\label{coro:rgauche-unique-b2}
There exists a unique prime ideal $\rG^\gauche$ of $R$ lying over $\qG^\gauche$ and contained in 
$\rGba$. 
\end{coro}

\begin{proof}
Let $\rG^\gauche$ and $\rG^\gauche_*$ be two prime ideals of $R$ lying over $\qG^\gauche$ and contained 
in $\rGba$. Then there exists $h \in H$ such that $\rG_*^\gauche=h(\rG^\gauche)$. We deduce from 
Proposition~\ref{prop:rgauche-rbarre} that $\rGba=g(\rGba)$. So $h$ belongs to the decomposition group 
of $\rGba$, which is the same as the one of $\rG^\gauche$ (by Corollary~\ref{coro:rgauche-b2}). 
So $\rG^\gauche_*=\rG^\gauche$.
\end{proof}

\bigskip

\boitegrise{{\it We will denote by $\rG^\gauche$ the unique prime ideal of 
$R$ lying over $\qG^\gauche$ and contained in $\rGba$. We will denote by $D^\gauche$ 
(respectively $I^\gauche$) its decomposition (respectively inertia) group.}}{0.75\textwidth}

\bigskip

Corollary~\ref{coro:rgauche-b2} says that 
\equat\label{eq:dgauche-igauche-b2}
D^\gauche=W_2'\qquad\text{and}\qquad |I^\gauche|=2.
\endequat

\bigskip

\begin{coro}\label{coro:r-rgauche-b2}
$R/\rG^\gauche \simeq Q/\qG^\gauche_\chi$ is integrally closed.
\end{coro}

\bigskip

\begin{coro}\label{coro:cellules-gauches-b2}
$\{1\}$, $\{s\}$, $\{tst\}$ and $\{w_0\}$ are generic Calogero-Moser left cells, and their associated 
generic Calogero-Moser cellular characters are given by
$$
\begin{cases}
\isomorphisme{1}^\calo = \unb_W,\\
\isomorphisme{s}^\calo = \e_s,\\
\isomorphisme{tst}^\calo = \e_t,\\
\isomorphisme{w_0}^\calo = \e.\\
\end{cases}
$$
\end{coro}

\begin{proof}
This follows from the fact that the subsets given in the Corollary are also 
generic Calogero-Moser two-sided cells and that their associated generic Calogero-Moser 
families are given by Corollary~\ref{dba=iba}. 
\end{proof}

\bigskip

\begin{coro}\label{coro:euler-modulo-gauche-b2}
The following congruences hold in $R$:
$$
\begin{cases}
\eulerq \equiv 2(B+A) \mod \rG^\gauche,\\
s(\eulerq) \equiv 2(B-A) \mod \rG^\gauche,\\
tst(\eulerq) \equiv -2(B-A) \mod \rG^\gauche,\\
w_0(\eulerq) \equiv -2(B+A) \mod \rG^\gauche,\\
t(\eulerq)\equiv st(\eulerq) \equiv ts(\eulerq) \equiv sts(\eulerq) \equiv 0 \mod \rG^\gauche.
\end{cases}
$$
\end{coro}

\begin{proof}
By~(\ref{eq:minimal-euler-gauche-b2}), the following congruence holds in $R[\tb]$:
$$\prod_{w \in W} (\tb - w(\eulerq)) \equiv \tb^4(\tb+2(A+B))(\tb+2(B-A))(\tb-2(A+B))(\tb-2(B-A)) 
\mod \rG^\gauche R[\tb].\leqno{(*)}$$
We already know that, since $\qG^\gauche \subset \rG^\gauche$, that $\eulerq \equiv +2(B+A) \mod \rG^\gauche$. 
This implies that $s(\eulerq)$ is congruent to $2(B-A)$, $-2(B+A)$, $-2(B-A)$ or $0$ modulo $\rG^\gauche$. 
But, since $s(\eulerq) \equiv 2(B-A) \mod \rGba$ by construction, this forces 
$s(\eulerq) \equiv 2(B-A) \mod \rG^\gauche$. 

The third and fourth congruences are obtained from the first two ones 
by noting that  $tst(\eulerq)=w_0s(\eulerq)=-s(\eulerq)$ and $w_0(\eulerq)=-\eulerq$. 

The last one follows from $(*)$.
\end{proof}

\bigskip

\begin{coro}\label{coro:congruence-gauche-b2}
The following congruences hold in $R$:
$$
\begin{cases}
\d \equiv 2B(B+A) \mod \rG^\gauche,\\
s(\d) \equiv 2B(B-A) \mod \rG^\gauche,\\
tst(\d) \equiv 2B(B-A) \mod \rG^\gauche,\\
w_0(\d) \equiv 2B(B+A) \mod \rG^\gauche,\\
t(\d)\equiv st(\d) \equiv ts(\d) \equiv sts(\d) \equiv 0 \mod \rG^\gauche
\end{cases}
$$
$$
\begin{cases}
\eulerq' \equiv B\Sig \mod \rG^\gauche,\\
s(\eulerq') \equiv B\Sig \mod \rG^\gauche,\\
tst(\eulerq') \equiv -B\Sig \mod \rG^\gauche,\\
w_0(\eulerq') \equiv -B\Sig \mod \rG^\gauche,\\
t(\eulerq')^2\equiv st(\eulerq')^2 \equiv ts(\eulerq')^2 \equiv sts(\eulerq')^2 \equiv B^2(\Sig^2-4\Pi) + 4A^2 \Pi \mod \rG^\gauche.
\end{cases}\leqno{\text{\it and}}
$$
Finally, $g(\eulerq'') \equiv 0 \mod \rG^\gauche$ for all $g \in G$.
\end{coro}

\begin{proof}
The equalities $(Q1^\gauche)$,\dots, $(Q9^\gauche)$ are of course also satisfied 
in the algebra $R/\pG^\gauche R$. Since $\pG^\gauche R$ is a $G$-stable ideal of $R$, 
we can apply any element of $G$ to the equalities $(Q1^\gauche)$,\dots, $(Q9^\gauche)$, 
and then reduce modulo $\rG^\gauche$. 
We deduce for instance from $(Q7^\gauche)$ that $g(\eulerq'') \equiv 0 \mod \rG^\gauche$ for all $g \in G$, 
as desired. 

Whenever $g(\eulerq) \not\equiv 0 \mod \rG^\gauche$, we deduce from $(Q9^\gauche)$ that 
$4g(\d)\equiv g(\eulerq)^2 - 4A^2+4B^2 \mod \rG^\gauche$, 
which allows to show that the congruences of $\d$, $s(\d)$, $tst(\d)$ and $w_0(\d)$ modulo $\rG^\gauche$ 
are the ones expected. Moreover, if $g \in W\setminus\{1,s,tst,w_0\}$, then 
$g(\eulerq) \equiv 0 \mod \rG^\gauche$ and we deduce from~$(Q5^\gauche)$ that 
$g(\d) \equiv 0 \mod \rG^\gauche$.

Finally, whenever $g(\eulerq) \not\equiv 0 \mod \rG^\gauche$, the congruence of $g(\eulerq')$ 
modulo $\rG^\gauche$ is easily determined thanks to $(Q1^\gauche)$, and is as expected. 
However, whenever $g(\eulerq) \equiv 0 \mod \rG^\gauche$, then $g(\d) \equiv 0 \mod \rG^\gauche$ 
by the previous observations, and it follows from~$(\Qrm 6^\gauche)$ that 
$g(\eulerq')^2 \equiv B^2(\Sig^2-4\Pi) + 4A^2 \Pi \mod \rG^\gauche$. 
\end{proof}

\bigskip

\begin{lem}\label{lem:red-pol-b2}
$F_{\eulerq'}^\gauche(\tb) = (\tb-B\Sig)^2(\tb+B\Sig)^2(\tb^2 - B^2\Sig^2 - 4A^2 \Pi  + 4B^2\Pi)^2$.
\end{lem}

\begin{proof}
First of all, by applying the elements of $W \times W$ to $\euler_0' \in \kb[V \times V^*]$, 
we note that $\euler_0'$ has $8$ conjugates, and so the minimal polynomial 
of $\euler_0'$ over $P_\bullet$ has degree $8$. Consequently, $F_{\eulerq'}(\tb)$ has degree $8$: 
it is in fact the characteristic polynomial of the multiplication by 
$\euler'$ in the free $P$-module $Z$. Hence,
$$F_{\eulerq'}(\tb)=\prod_{w \in W} (\tb - w (\eulerq'))$$
and the result then follows from Corollary~\ref{coro:congruence-gauche-b2}.
\end{proof}

\bigskip

As a conclusion, if $g \in \{t,st,ts,sts\}$ and if $q \in \{\eulerq,\eulerq'',\d\}$ then
\equat\label{eq:congru-b2}
g(q) \equiv 0 \mod \rG^\gauche
\endequat
and
\equat\label{eq:congru-carre-b2}
g(\eulerq')^2 \equiv B^2(\Sig^2-4\Pi) + 4A^2 \Pi \mod \rG^\gauche.
\endequat
The next Proposition makes the congruence~(\ref{eq:congru-carre-b2}) more precise: 
it is the most subtle point of this Chapter.

\bigskip

\begin{prop}\label{prop:b2-gauche-enfin}
The following congruences hold in $R$:
$$
\begin{cases}
t(\eulerq') \equiv st(\eulerq') \mod \rG^\gauche,\\
ts(\eulerq') \equiv sts(\eulerq') \mod \rG^\gauche,\\
t(\eulerq') \not\equiv ts(\eulerq') \mod \rG^\gauche.\\
\end{cases}
$$
\end{prop}

\begin{proof}
By Corollary~\ref{coro:congruence-gauche-b2}, 
$$(\tb-t(\eulerq'))(\tb-st(\eulerq'))(\tb-ts(\eulerq'))(\tb-sts(\eulerq'))
\equiv (\tb^2 - B^2\Sig^2 - 4A^2 \Pi  + 4B^2\Pi)^2 \mod \rG^\gauche R[\tb].$$
This shows that there exists a unique $g_0 \in \{st,ts,sts\}$ such that 
$g_0(\eulerq') \equiv t(\eulerq') \mod \rG^\gauche$. 
Since $t(\eulerq') \not\equiv 0 \mod \rG^\gauche$ and
$sts(\eulerq')=tw_0(\eulerq')=-t(\eulerq')$, 
the element $g_0$ is not equal to $sts$. So
$$g_0 \in \{st,ts\}.$$
We only need to show that 
$$g_0=st.\leqno{(*)}$$

Let $F_{\eulerq'}^\pi(\tb)$ be the reduction modulo $\pi P$ of the minimal polynomial of $\eulerq'$. 
Set 
\eqna
F^\pi(\tb)&=&\tb^4 - 2 B \Sig  \tb^3 + (-4 A^2 \Pi   + 4 B^2 \Pi   -
        \s  \Sig  \Pi  ) \tb^2 \\ && + (8 A^2 B \Sig  \Pi   +
        2 B^3 \Sig^3 - 8 B^3 \Sig  \Pi   + 2 B \s  \Sig^2 \Pi  
        - 4 B \s  \Pi^2) \tb \\ && - 4 A^2 B^2 \Sig^2 \Pi   - B^4 \Sig^4
        + 4 B^4 \Sig^2 \Pi   - B^2 \s  \Sig^3 \Pi   +
        4 B^2 \s  \Sig  \Pi^2 + \s ^2 \Pi^3.
\endeqna
Then
$$F_{\eulerq'}^\pi(\tb)=F^\pi(\tb) \cdot F^\pi(-\tb).$$
Let $\rG^\pi$ be a prime ideal of $R$ lying over $\pi P$ and contained in  $\rG^\gauche$. 
Let $\rG^\pi_0$ be a prime ideal of $R$ lying over $\pi P + \pG_0$ and contained in $\rG^\pi$. 
One checks easily that $F^\pi(\tb)$ is prime to $F^\pi(-\tb)$ and is separable, 
so $F^\pi_{\eulerq'}(\tb)$ admits $8$ different roots in $R/\rG^\pi$, which are the classes of the 
$g(\eulerq')$'s, where $g$ runs over $W$. Let $F^\gauche(\tb)$ denote the reduction 
modulo $\pG^\gauche$ of $F^\pi(\tb)$. 
Then
$$F^\gauche(\tb) = (\tb - B\Sig)^2(\tb^2 - B^2\Sig^2 - 4A^2 \Pi  + 4B^2\Pi).$$
So it follows from Corollary~\ref{coro:congruence-gauche-b2} that $\eulerq'$ and $s(\eulerq')$ 
are roots of $F^\pi(\tb)$ in $R/\rG^\pi$. On the other hand, since $W_2'$ acts transitively 
$\{t,st,ts,sts\}$ and stabilizes $\rG^\gauche$, we may, by replacing $\rG^\pi$ by $g(\rG^\pi)$ 
for some $g \in W_2'$ if necessary, 
assume that $t(\eulerq')$ is a root of $F^\pi(\tb)$ modulo $\rG^\pi$. The other roots modulo $\rG^\pi$ 
is then $st(\eulerq')$ or $ts(\eulerq')$ (this cannot be $sts(\eulerq')=-t(\eulerq')$, as this one 
is a root of $F^\pi(-\tb)$ modulo $\rG^\pi$). So let $g_1$ denote the unique element of $\{st,ts\}$ 
such that $g_1(\eulerq')$ is a root of $F^\pi(\tb)$ modulo $\rG^\pi$. By reduction modulo $\rG^\pi_0$, 
we get  
$$\eulerq'\cdot s(\eulerq') \cdot t(\eulerq') 
\cdot g_1(\eulerq') \equiv F^\pi(0) \equiv \s^2\Pi^3 \mod \rG^\pi_0.$$
Moreover, there exists $g \in G$ such that $\rG_0 \subset g(\rG^\pi_0)$. But, since $G=W_4'$, there exists 
signs $\eta_1$, $\eta_2$, $\eta_3$, $\eta_4$ such that 
$\{g(1),g(s),g(t),g(g_1)\}=\{\eta_1, \eta_2 s, \eta_3 t,\eta_4 g_1\}$ 
and $\eta_1\eta_2\eta_3\eta_4=1$. Consequently, 
$$\eulerq' \cdot s(\eulerq') \cdot t(\eulerq') \cdot g_1(\eulerq') \equiv \s^2\Pi^3 \mod \rG_0.$$
The next computation can be performed directly inside $\kb[V \times V^*]^{\D \Zrm(W)} \supset R/\rG_0$: 
\eqna
\eulerq_0' \cdot s(\eulerq_0') \cdot t(\eulerq_0') \cdot ts(\eulerq_0') &=& 
(xY+yX)(xX+yY)(-xY+yX)(-xX+yY)X^4Y^4 \\
&=& (y^2Y^2-x^2X^2)(y^2X^2-x^2Y^2) \Pi^2 \\
&=& ((x^4+y^4) \Pi-\pi(X^4 + Y^4))\Pi^2 \\
&\equiv& \s^2\Pi^3 \mod \pi \kb[V \times V^*].
\endeqna
So $g_1=ts$.

We have therefore proven that 
$$(\tb-\eulerq')(\tb-s(\eulerq'))(\tb-t(\eulerq'))(\tb-ts(\eulerq'))
\equiv F^\pi(\tb) \mod \rG^\pi R[\tb].$$
By reduction modulo $\rG^\gauche$, we get 
$$(\tb-t(\eulerq'))(\tb-ts(\eulerq'))\equiv \tb^2 - B^2\Sig^2 - 4A^2 \Pi  + 4B^2\Pi \mod \rG^\gauche R[\tb].$$
So $t(\eulerq') \equiv -ts(\eulerq') \mod \rG^\gauche$, which shows that $g_0\neq ts$. So $g_0=st$.
\end{proof}

\bigskip

\begin{coro}\label{coro:b2-gauche-enfin}
For the choice of $\rG^\gauche$ made in this section, 
the generic Calogero-Moser left cells are $\{1\}$, $\{s\}$, $\{tst\}$, $\{w_0\}$, 
$\{ts,sts\}$ and $\{t,st\}$.

Let $g_\gauche$ denote the involution of $G$ which leaves $1$, $s$, $tst$ and $w_0$ fixed and 
such that $g_\gauche(t)=st$ and $g_\gauche(ts)=sts$. Then $I^\gauche = \langle g_\gauche\rangle$. 
\end{coro}

\bigskip

\begin{rema}\label{rema:w-w}
We understand better here the convention chosen for the action of $W \times W$ on $\kb(V \times V^*)$ 
(see~\S\ref{subsection:specialization galois 0}). Indeed, if we had chosen the other action 
(the one described in Remark~\ref{rema:action-dif}), the generic Calogero-Moser {\it left} cells 
would have coincided with the Kazhdan-Lusztig {\it right} cells.\finl
\end{rema}

\bigskip

\subsection{Proof of Theorem~\ref{theo:conjectures-b2}} 
Keep here the notation of Theorem~\ref{theo:conjectures-b2} ($a=-c_s$, $b=-c_t$). 
Let us fix for the moment a prime ideal $\rG_c^\gauche$ of $R$ containing $\rG^\gauche$ and $\pG_c R$. 
Since $R/\rGba \simeq P/\pGba$, we deduce that $\rGba_c=\rGba + \pGba_c R$ is the unique 
prime ideal of $R$ lying over $\rGba$ and containing $\rG_c^\gauche$. Let $D_c^\gauche$ (respectively 
$I_c^\gauche$) be the decomposition (respectively inertia) group of $\rG_c^\gauche$ 
and $\Dba_c$ (respectively $\Iba_c$) be the decomposition (respectively inertia) group 
of $\rGba_c$.

\medskip

It follows from Corollaries~\ref{coro:euler-modulo-gauche-b2} and~\ref{coro:congruence-gauche-b2} 
and from Proposition~\ref{prop:b2-gauche-enfin} that
$$
\begin{cases}
\eulerq \equiv 2(b+a) \mod \rG_c^\gauche,\\
s(\eulerq) \equiv 2(b-a) \mod \rG_c^\gauche,\\
tst(\eulerq) \equiv -2(b-a) \mod \rG_c^\gauche,\\
w_0(\eulerq) \equiv -2(b+a) \mod \rG_c^\gauche,\\
t(\eulerq)\equiv st(\eulerq) \equiv ts(\eulerq) \equiv sts(\eulerq) \equiv 0 \mod \rG_c^\gauche.
\end{cases}\leqno{(\clubsuit)}
$$
$$
\begin{cases}
\d \equiv 2b(b+a) \mod \rG_c^\gauche,\\
s(\d) \equiv 2b(b-a) \mod \rG_c^\gauche,\\
tst(\d) \equiv 2b(b-a) \mod \rG_c^\gauche,\\
w_0(\d) \equiv 2b(b+a) \mod \rG_c^\gauche,\\
t(\d)\equiv st(\d) \equiv ts(\d) \equiv sts(\d) \equiv 0 \mod \rG_c^\gauche
\end{cases}\leqno{(\diamondsuit)}
$$
$$
\begin{cases}
\eulerq' \equiv b\Sig \mod \rG_c^\gauche,\\
s(\eulerq') \equiv b\Sig \mod \rG_c^\gauche,\\
tst(\eulerq') \equiv -b\Sig \mod \rG_c^\gauche,\\
w_0(\eulerq') \equiv -b\Sig \mod \rG_c^\gauche,\\
t(\eulerq') \equiv st(\eulerq') \equiv -ts(\eulerq') \equiv -sts(\eulerq') \mod \rG_c^\gauche,\\
t(\eulerq')^2\equiv b^2 \Sig^2+4(a^2-b^2) \Pi \mod \rG_c^\gauche.
\end{cases}\leqno{(\heartsuit)}
$$
and
$$g(\eulerq'') \equiv 0 \mod \rG_c^\gauche\leqno{(\spadesuit)}$$
for all $g \in G$. Recall that we assume that $ab \neq 0$ and that two elements $g$ and $g'$ of $W$ 
are in the same Calogero-Moser left (respectively two-sided) $c$-cell
if and only if $g(q) \equiv g'(q) \mod \rG_c^\gauche$ (respectively $\mod \rGba_c$) 
for all $q \in \{\eulerq,\eulerq,\eulerq'',\d\}$ 
(since $Q=P[\eulerq,\eulerq',\eulerq'',\d]$).

\bigskip

\subsubsection*{\bfit The case ${\boldsymbol{a^2 \neq b^2}}$} 
Assume here that $a^2 \neq b^2$. It follows from the congruences $(\clubsuit)$, $(\diamondsuit)$, 
$(\heartsuit)$ and $(\spadesuit)$ that the Calogero-Moser left $c$-cells are 
$\{1\}$, $\{s\}$, $\{tst\}$, $\{w_0\}$, $\G_\chi^+=\{t,st\}$ and $\G_\chi^-=\{ts,sts\}$ and that the 
Calogero-Moser two-sided $c$-cells are $\{1\}$, $\{s\}$, $\{tst\}$, $\{w_0\}$ and $\G_\chi$.

The results on Calogero-Moser $c$-families Calogero-Moser $c$-cellular characters given by 
Table~\ref{table:conjectures-b2} then follow from Corollary~\ref{coro:famille-semicontinu}, 
from Proposition~\ref{prop:cellulaire-semicontinu}, from~(\ref{eq:familles-b2}) and from 
Corollary~\ref{coro:cellules-gauches-b2}.

Let us now determine $D_c^\gauche$ and $I_c^\gauche$. Note that $I^\gauche \subset I_c^\gauche$ and that, 
according to the description of Calogero-Moser left $c$-cells, that is of $I_c^\gauche$-orbits, 
this forces the equality. On the other hand, since $\rGba_c=\rGba + \pGba_c R$, 
we have $D^\gauche \subset D_c^\gauche$. 
Since the $D_c^\gauche$-orbits are contained in the Calogero-Moser two-sided $c$-cells,
the description of these last ones forces again the equality. We show similarly that 
$\Dba_c=\Dba=W_2'$ and that $\Iba_c=\Iba=W_2'$.

\bigskip

\subsubsection*{\bfit The case where ${\boldsymbol{a=b}}$} 
In this case, the last congruence of $(\heartsuit)$ becomes  
$$t(\eulerq')^2 \equiv b^2 \Sig^2 \mod \rG_c^\gauche.$$
So $t(\eulerq') \equiv b\Sig \mod \rG_c^\gauche$ or $t(\eulerq') \equiv -b\Sig \mod \rG_c^\gauche$. 
By replacing $\rG_c^\gauche$ by $g(\rG_c^\gauche)$, where $g \in W_2'=D^\gauche$ 
exchanges $t$ and $sts$, we can make the following choice:

\bigskip

\boitegrise{\noindent{\bf Choice of ${\boldsymbol{\rG_c^\gauche}}$.} 
{\it We choose the prime ideal $\rG_c^\gauche$ so that 
$t(\eulerq') \equiv -b\Sig \mod \rG_c^\gauche$.}}{0.75\textwidth}

\bigskip
 
The family of congruences $(\clubsuit)$, $(\diamondsuit)$, 
$(\heartsuit)$ and $(\spadesuit)$ show that the Calogero-Moser left $c$-cells 
are $\{1\}$, $\{w_0\}$, $\G_s=\{s,ts,sts\}$ and $\G_t=\{t,st,tst\}$ and that 
the Calogero-Moser two-sided $c$-cells are 
$\{1\}$, $\{w_0\}$ and $W \setminus \{1,w_0\}$. 

As previously, the results on Calogero-Moser $c$-families and Calogero-Moser 
$c$-cellular characters given by Table~\ref{table:conjectures-b2} follow from 
Corollary~\ref{coro:famille-semicontinu}, 
from Proposition~\ref{prop:cellulaire-semicontinu}, from~(\ref{eq:familles-b2}) and 
from Corollary~\ref{coro:cellules-gauches-b2}.

Let us conclude by the description of $D_c^\gauche$, $I_c^\gauche$, $\Dba_c$ and $\Iba_c$. 
First of all, $I_c^\gauche$ has two orbits of length $3$ 
($\G_s$ and $\G_t$) so its order is divisible by $3$. Moreover, it contains $I^\gauche$ 
which has order $2$. So its order is divisible by $6$. The description 
of left $c$-cells then allows to conclude that 
$I_c^\gauche=\SG_3$. On the other hand, $D_c^\gauche$ permutes the left $c$-cells which 
have the same associated $c$-cellular character. So $D_c^\gauche$ stabilizes $\G_s$ and $\G_t$. 
This forces the equality $D_c^\gauche=I_c^\gauche=\SG_3$.

On the two-sided cells side, recall that $\Dba_c=\Iba_c$ because $\Dba_c/\Iba_c$ is a quotient of 
$\Dba/\Iba$. Moreover, the inclusions $W_2' \subset \Iba_c$ and $I_c^\gauche \subset \Iba_c$ 
show that $W_3' \subset \Iba_c$. The equality $\Iba_c=W_3'$ becomes obvious.

\bigskip

\section{Complement: fixed points}\label{sec:fixed-b2}

\medskip

As announced in Example~\ref{ex:root of unity}, we will show that 
Conjecture~\FIX~holds when $W$ is of type $B_2$ and $\t$ is a root of unity. 
If $\t$ is not of order dividing $4$, then $\ZCB^\t=\ZCB^{\CM^\times}$ 
is a union of affine spaces isomorphism $\CCB$, so 
this case is easy. If $\t$ has order dividing $2$, then $\ZCB^\t=\ZCB$, 
and there is again nothing to prove. 

\medskip

\boitegrise{\noindent{\bf Assumption.} {\it We assume in this section that $\kb=\CM$ 
and we fix a primitive $4$-th root of unity $\t$.}}{0.75\textwidth}

\medskip


We identify $\ZCB$ with the set of $(a,b,s,S,p,P,e,e',e'',d) \in \Ab^{10}(\CM)$ 
satisfying the relations $(Z1)$, $(Z2)$,\dots, $(Z9)$ in~(\ref{relations Q}), 
with $A$, $B$, $\s$, $\Sig$, $\pi$, $\Pi$, $\euler$, $\euler'$, $\euler''$ and $\delb$ 
replaced by $a$, $b$, $s$, $S$, $p$, $P$, $e$, $e'$, $e''$ and $d$ respectively. 
Then
$$\ZCB^\t=\{(a,b,s,S,p,P,e,e',e'',d) \in \ZCB~|~s=S=e'=e''=0\}.$$
Therefore, 
$$\ZCB^\t \simeq \{(a,b,p,P,e,d) \in \Ab^6(\CM)~|~
\begin{cases}
d^2=e^2+pP\\
p(4d-e^2+4a^2-4b^2)=0\\
P(4d-e^2+4a^2-4b^2)=0\\
d(4d-e^2+4a^2-4b^2)=0\\
e(4d-e^2+4a^2-4b^2)=0\\
\end{cases} \}.$$
This shows that $\ZCB^\t$ has two irreducible components $\XCB$ and $\XCB_0$, where
$$\XCB=\{(a,b,p,P,e,d) \in \Ab^6(\CM)~|~d^2=e^2+pP~\text{and}~
4d-e^2+4a^2-4b^2=0\}$$
$$\XCB_0=\{(a,b,p,P,e,d) \in \Ab^6(\CM)~|~p=P=d=e=0\}\}.\leqno{\text{and}}$$
Note that the intersection of $\XCB$ and $\XCB_0$ is not empty. 

So $\XCB_0 \simeq \CCB \simeq {\mathrm{pt}} \times_{{\mathrm{pt}}} \CCB$ and 
Conjecture~\FIX~holds for this easy irreducible component. On the other hand,
\eqna
\XCB& \simeq &\{(a,b,p,P,e) \in \Ab^5(\CM)~|~(e^2-4a^2+4b^2)^2=16(e^2+pP)\}\\
&=& \{(a,b,p,P,e) \in \Ab^5(\CM)~|~(e-2(a+b))(e-2(a-b))(e+2(a+b))(e+2(a-b))=pP\}.
\endeqna
Let $V'$ be the $\t$-eigenspace of $w$ and take $W'=\langle w \rangle$. 
Then $(V',W')$ is a reflection subquotient of $(V,W)$, with $\dim_\CM V'=1$. 
We now use the description of $\ZCB(V',W')$ given in Theorem~\ref{center rang 1}:
\eqna
\ZCB(V',W') &\simeq& 
\{(k_0,k_1,k_2,k_3,e,x,y) \in \Ab^7(\CM)~|~(e-4k_0)(e-4k_1)(e-4k_2)(e-4k_3)=xy\\
&&\hskip4cm \text{and}~k_0+k_1+k_2+k_3=0\},\\
\endeqna
where the $k_i$'s are the coordinates in $\CCB(V',W')$. So, if we set 
$\ph : \CCB \to \CCB(V',W')$, $(a,b) \mapsto \frac{1}{2}(a+b,a-b,-a-b,-a+b)$, 
then $\ph$ is linear and 
\equat\label{eq:fixed-b2-compo}
\XCB \simeq \ZCB(V',W') \times_{\CCB(V',W')} \CCB.
\endequat
Note that $\ph$ is well-defined only up to permutation of the four coordinates in $\CCB(V',W')$.

\chapter{Left cells in small dihedral groups}\label{chapter:diedral}

We illustrate here the topological version of the construction of 
Calogero-Moser cells (see Theorem~\ref{theo:cm-cells-topo}) in the 
special case of dihedral groups. In particular, the figures drawn 
in this chapter give evidence that Conjecture L holds for dihedral groups 
of order $\le 10$ (similar computations and drawings can be done 
for larger groups and they give evidence that Conjecture~L holds for dihedral groups of 
order $\le 20$). The underlying computations have been done using
{\sc Magma}~\cite{magma} and {\sc Sage}~\cite{sage}.

\medskip

\boitegrise{{\bf Hypothesis and notation.} {\it We assume in this chapter, and only 
in this chapter, that $W$ is a Coxeter group, that $V$ has dimension $2$ and 
is irreducible as a $\CM W$-module. Write $S=\{s,t\}$ and let $m$ denote 
the order of $st$. We denote by $\t \in \NC$ the unique element of order $2$ such that 
$\t s \t^{-1}=t$ and $\t(C_\RM)=C_\RM$ and we assume that $v_\RM \in C_\RM$ 
is chosen so that $\t(v_\RM)=v_\RM$. We also fix $c \in \CCB(\RM)$, we 
set $a=k_s$ and $b=k_t$ and we assume that $a$, $b \ge 0$.}}{0.75\textwidth}

\medskip

Note that the above conditions imply that $m \ge 3$, and determine $v_\RM$ up to multiplication by an 
element of $\RM_{>0}$. Note also that, if $m$ is odd, then $a=b$. 
We denote by $\gamh_c : [0,1] \longto \CCB(\RM) \times C_\RM \times C_\RM^*$ 
the path defined by
$$\gamh_c(t)=(tc,v_\RM,(1-t) v_\RM^*)$$
and we define $\g_c$ as the image of $\gamh_c$ in $\CCB \times V/W \times V^*/W$. 
In all cases we encounter in this chapter, it turns out that the image of $\gamh_c([0,1[)$ in 
$\CCB \times V/W \times V^*/W$ does not meet the ramification locus of $\Upsilon$: since 
$\gamh_c(1)=(c,v_\RM,0)$, we can define left Calogero-Moser $c$-cells using the path 
$\g_c$ as in \S\ref{se:chemins}. 
We aim here to show pictures of the paths $(\r_w)_{w \in W}$ in $V_\RM$ 
constructed in Theorems~\ref{th:cellsGaudin} and~\ref{theo:cm-cells-topo} 
for $3 \le m \le 6$, and explain how this describes left
Calogero-Moser $c$-cells. This
gives evidence for Conjecture~L in those cases. To be more precise, the 
full paths have not been calculated, but only an approximation of them and we
have 
drawn only the points corresponding to the values $t=j/15$, for $0 \le j \le 15$. 

Before explaining how to read these pictures, we need to introduce
more notations. 
When $m$ is even, we denote by $\e_s$ (resp. $\e_t$) the unique non-trivial 
linear character of $W$ such that $\e_s(t)=1$ (resp. $\e_t(s)=1$). Note 
that $\e_s\e_t=\e$. 
We denote by $\chi$ the character of the reflection representation $V$ 
and, when $m \in \{5,6\}$, we denote by $\chi'$ the unique irreducible 
character of $W$ of degree $2$ different from $\chi$. Recall that 
$$\Irr(W)=
\begin{cases}
\{1,\e,\chi\} & \text{if $m=3$,}\\
\{1,\e,\e_s,\e_t,\chi\} & \text{if $m=4$,}\\
\{1,\e,\chi,\chi'\} & \text{if $m=5$,} \\
\{1,\e,\e_s,\e_t,\chi,\chi'\} & \text{if $m=6$.}
\end{cases}$$
We will provide pictures for $3 \le m \le 6$, noting that, if $m \in \{3,5\}$, then 
$a=b$ (and only one picture will be shown) and, if $m \in \{4,6\}$, 
only some pictures for $a \ge b \ge 1$ will be shown (since the automorphism $\t$ 
exchanges $s$ and $t$ and the roles of $a$ and $b$). 
Then the pictures must be read as follows:
\begin{itemize}
\item[$\bullet$] The plane of the page represents $V_\RM^*$.

\item[$\bullet$] The large black dots are the elements $(w^{-1}(v_\RM^*))_{w \in W}$. 
The element of $w$ is written next to the point (here, $w_0$ stands for 
the longest element).

\item[$\bullet$] For $1 \le j \le 14$, there are $|W|=2m$ elements $\l_1^{(j)}$,\dots, $\l_{2m}^{(j)}$ 
of $V^*$ such that, for any $y \in V$, the eigenvalues of the Gaudin operator 
$D_y^{\gamh_c(j/15)}$ are $(\langle y , \l_{j'}^{(j)} \rangle)_{1 \le j' \le 2m}$, 
counted with multiplicity. The small black dots represent the full collection 
$(\l_{j'}^{(j)})_{\substack{1 \le j \le 14 \\ 1 \le j' \le 2m}}$. 

From this collection, the reader can get an idea from the pictures what the
	paths $\r_w$'s look like.

\item[$\bullet$] The large red dots correspond to the elements $\l_1$,\dots, $\l_r$ 
of $V^*$ such that, for any $y \in V$, the eigenvalues of the Gaudin operator 
$D_y^{\gamh_c(1)}=D_y^{c,v_\RM,0}$ are $(\langle y , \l_{j'} \rangle)_{1 \le j' \le r}$. 
To each such large red dot is associated a Calogero-Moser cellular character 
as in~\S\ref{sub:cellular-gaudin}. The corresponding cellular 
character is written in red next to that dot.

\item[$\bullet$] Two elements $w$ and $w'$ of $W$ are in the same $c$-cell 
if and only if the two paths starting at the large black dots named $w$ and $w'$ 
end at the same large red dot. One can then check (using for 
instance~\cite[Theorem~21.3.1~and~Corollary~21.3.4]{bonnafe livre} for a description 
of left Kazhdan-Lusztig cells and the Kazhdan-Lusztig cellular characters) 
that Conjecture~L holds for $3 \le m \le 6$ and the given values of $(a,b)$. 
Note that similar figures can be reproduced for many values of $(a,b)$ and 
for $m$ up to $12$ using a standard computer. 
\end{itemize}

\medskip
 
\noindent\begin{tabular}{cc}
\input{./a2-pgf.tex}
&
\input{./i25-pgf.tex} \\
Type $A_2$ & Type $I_2(5)$ \\
\end{tabular}

\vskip0.8cm

\centerline{\input{b2-a0-pgf.tex}
\input{b2-atiers-pgf.tex}}

\smallskip

\centerline{\input{b2-adeuxtiers-pgf.tex}
\input{b2-a1-pgf.tex}}

\centerline{Type $B_2=I_2(4)$, $a=1 \ge b \ge 0$}

\newpage

\centerline{\input{g2-a0-pgf.tex}\input{g2-atiers-pgf.tex}}

\centerline{\input{g2-adeuxtiers-pgf.tex}\input{g2-a1-pgf.tex}}

\centerline{Type $G_2=I_2(6)$, $a=1 \ge b \ge 0$}

\part*{Appendices}


\renewcommand\thechapter{\Alph{chapter}}

\def\chaptername{Appendix}\setcounter{chapter}{0}

\chapter{Filtrations}\label{appendice:filtration}

\section{Filtered modules}

\medskip

Let $R$ be a commutative ring.
A {\em filtered $R$-module} is an $R$-module $M$ together with
$R$-submodules $M^{\le i}$ for $i\in\BZ$ such that
$$M^{\le i}\subset M^{\le i+1} \text{ for }i\in\BZ,\ 
M^{\le i}=0\text{ for }i\ll 0 \text{ and } M=\bigcup_{i\in\BZ}M^{\le i}.$$
Given $M$ a filtered $R$-module, the {\em associated $\BZ$-graded} $R$-module
$\grad\,M $ is given by
$$(\grad\,M)_i=M^{\le i}/M^{\le i-1}.$$
The {\em principal symbol map}
$\xi:M\to \grad\,M$ is defined by $\xi(m)=m\mod M^{\le i-1}\in 
(\grad\,M)_i$, where $i$ is minimal such that $m\in M^{\le i}$.
The principal symbol map is injective but not additive.

The {\em Rees}
module associated with $M$ is the $R[\hbar]$-submodule
$\mathrm{Rees}(M)=\sum_{i\in\BZ} \hbar^i
M^{\le i}$\indexnot{Rees}{\mathrm{Rees}(M)} of $R[\hbar^{\pm 1}]\otimes_R M$.
We have $R[\hbar^{\pm 1}]\otimes_{R[\hbar]}\mathrm{Rees}(M)=R[\hbar^{\pm 1}]
\otimes_R M$. In particular,
given $t\in R^\times$, we have an isomorphism of $R$-modules
$$\bigl(R[\hbar]/\langle \hbar-t\rangle\bigr)\otimes_{R[\hbar]}\mathrm{Rees}(M)\longiso M,\ 1\otimes\hbar^im\mapsto t^im.$$

There is an isomorphism of $R$-modules
$$R[\hbar]/\langle \hbar \rangle \otimes_{R[\hbar]}\mathrm{Rees}(M)\longiso \grad\,M,\
1\otimes\hbar^i m\mapsto \begin{cases} 0 & \text{ if }m\in M^{<i} \\ \xi(m) & \text{ otherwise.}\end{cases}$$

\section{Filtered algebras}

\medskip

Let $A$ be an $R$-algebra.
A {\em filtration} on $A$ is the data of a filtered
$R$-module structure on $A$
such that
$$1\in A^{\le 0}\text{ and }A^{\le i}\cdot A^{\le j}
\subset A^{\le i+j}\text{ for all }i,j\in\BZ.$$
We fix a filtration on $A$. 
The associated graded $R$-module $\grad\,A$ is a graded $R$-algebra. The Rees module
associated with $A$ is an $R[\hbar]$-algebra.

\bigskip

\begin{lem}
\label{le:zerodivisorsfilt}
If $\grad\,A$ has no zero divisors, then the principal symbol map
$\xi:A\to \grad\,A$ is multiplicative and $A$ has no zero divisors.
\end{lem}

\begin{proof}
Let $a,b\in A$ be two non-zero elements and let
$i,j$ be minimal such that $a\in A^{\le i}$ and $b\in A^{\le j}$. Since
$\grad\,A$ has no zero divisors, it follows that $\xi(a)\xi(b)\not=0$, hence
$ab{\not\in}A^{<i+j}$. This shows that $\xi(ab)=\xi(a)\xi(b)$, and that $ab\not=0$.
\end{proof}

Let us recall some facts of commutative algebra (cf. 
\cite[Exercises 9.4-9.5]{matsumura}).
Let $R'$ be a commutative domain with field of fractions $K$. An element
$x\in K$ is said to be {\it almost integral}
over $R'$ if there exists $a\in R'$, $a\not=0$, such
 that
for all $n\ge 0$, we have $ax^n\in R'$. If $x$ is integral over $R'$, then
$x$ is almost integral over $R'$, and the converse holds if $R'$ is
noetherian.

We say that $R'$ is {\it completely
integrally closed} if the elements of $K$ that are almost integral over $R'$
are in $R'$.

\begin{lem}
\label{le:normalfilt}
Assume $A$ is a commutative ring. If $\grad\,A$ is a completely integrally
closed domain, then $A$ is a completely integrally closed domain.
\end{lem}

\begin{proof}
Lemma \ref{le:zerodivisorsfilt} shows that $A$ is a domain. Let $K$ be its
field of fractions. Let $x\in K$ be a non-zero element that is
	almost integral over $A$. Let
$c,d\in A$ such that $x=c/d$. Let $i$ (resp. $j$) be minimal such that
$c\in A^{\le i}$ (resp. $d\in A^{\le j}$). We show by induction on
$i$ that $x\in A$.

Let $a\in A$, $a\not=0$, such that $ax^n\in A$ for all $n\ge 0$.
Let $\alpha_n=ax^n$. We have $d^n\alpha_n=ac^n$, hence
$\xi(d)^n\xi(\alpha_n)=\xi(a)\xi(c)^n$ (cf. Lemma \ref{le:zerodivisorsfilt}).
It follows that $\frac{\xi(c)}{\xi(d)}$ is an element of the field
of fractions of $\grad\,A$ that is almost integral over $\grad\,A$.
Consequently, it is in $\grad\,A$. Since $\grad\,A$ has no zero divisors,
it follows that it is homogeneous of degree $i-j$. Let $u\in A$ with 
$\xi(u)=\frac{\xi(c)}{\xi(d)}$. Let $x'=x-u=\frac{c-ud}{d}$. We have
$c-ud\in A^{\le i-1}$ and $x'$ is almost integral over $A$. It follows
by induction that $x'\in A$, hence $x\in A$.
\end{proof}


\begin{lem}
\label{le:gldimfilt}

If $\grad\,A$ is noetherian, then $A$ is noetherian.

	Assume $\grad\,A$ is notherian and $A^{\le i}/A^{\le i-1}$ is
	a flat $R$-module for every $i\in\BZ$.
	If $\grad\,A$ has finite projective dimension as a
	$(\grad\,A\otimes_R (\grad\,A)^\opp)$-module, then 
	$A$ has finite projective dimension as an
	$(A\otimes_R A^\opp)$-module.
\end{lem}

\begin{proof}
	The first statement is \cite[Proposition IV.6]{NaVa}.

	Let us prove the second statement. We consider the
	product filtration on $A\otimes_R A$, so that there is a canonical
	isomorphism
	$\grad\,(A\otimes_R A^\opp)\xrightarrow{\sim}
	\grad\,A\otimes (\grad\,A)^\opp$. The
	$(\grad\,A\otimes (\grad\,A)^\opp)$-module $\grad\,A$ has finite
	projective dimension, hence finite Tor
	dimension. It follows that $A$ has finite Tor dimension as an
	$(A\otimes_R A^\opp)$-module by \cite[Ch. 2, Proposition 3.12]{bjork}.
	We deduce from \cite[Theorem 3.2.7  and Lemma 4.1.6]{We1} that
	$A$ has finite projective dimension as an $(A\otimes_R A^\opp)$-module.
\end{proof}

\section{Filtered modules over filtered algebras}

\medskip

A {\em filtered $A$-module} is an $A$-module $M$ together with a structure
of filtered $R$-module such that
$$A^{\le i}\cdot M^{\le j}\subset M^{\le i+j}\text{ for all }i,j\in\BZ.$$
Given $M$ and $N$ two filtered $A$-modules,
a {\em filtered morphism} of $A$-modules is a morphism of $A$-modules $f:M\to N$
such that $f(M^{\le i})\subset
N^{\le i}$ for all $i\in\BZ$.

\bigskip

\begin{lem}
\label{le:injsurjfiltered}
Let $f:M\to N$ be a filtered morphism of $A$-modules.
If $\grad\,f$ is surjective (resp. injective), then $f$ is surjective
(resp. injective).
\end{lem}

\begin{proof}
Assume $\grad\,f$ is surjective. We have $N^{\le i}=f(M^{\le i})+
N^{\le i-1}$. Since $N^{\le i}=0$ for $i\ll 0$, it follows by induction that
$N^{\le i}=f(M^{\le i})$, hence $f$ is surjective.

Assume $\grad\,f$ is injective.
Let $m\in M-\{0\}$ and let $i$ be minimal such $m\in M^{\le i}$. We have
$f(m)\not\in N^{\le i-1}$, hence $f(m)\not=0$.
\end{proof}

\bigskip

\begin{lem}
\label{le:generatefilt}
Let $M$ be a filtered $A$-module and let $E$ be a subset of $M$. If
$\xi(E)$ generates $\grad\,M$ as a $\grad\,A$-module, then $E$ generates
$M$ as an $A$-module.

Let $F$ be a subset of $A$. If $\xi(F)$ generates $\grad\,A$ as an $R$-algebra,
then $F$ generates $A$ as an $R$-algebra.
\end{lem}

\begin{proof}
We have a canonical morphism of $A$-modules $f:A^{(E)}\to M$. 
Let us endow $A^{(E)}$ with the following filtration: 
$(A^{(E)})^{\le i}=f^{-1}(M^{\le i})$. 
By assumption, $\grad\,f$ is surjective, hence $f$ is surjective
by lemma \ref{le:injsurjfiltered}.

The second assertion follows from the first one by taking $A=R$, $M=A$
and $E$ the set of elements of $A$ that are products of elements of $F$.
\end{proof}

\bigskip

Let $M$ and $N$ be two  filtered $A$-modules. Assume $M$ is a
finitely generated $A$-module. We endow the $R$-module
$\Hom_A(M,N)$ with the filtration given by
$$\Hom_A(M,N)^{\le i}=\{f\in\Hom_A(M,N)\ |\ f(M^{\le j})\subset N^{\le i+j}
\ \forall j\in\BZ\}.$$
A map $f\in\Hom_A(M,N)^{\le i} \setminus \Hom_A(M,N)^{\le i-1}$ 
induces a morphism of $(\grad\, A)$-modules
$\xi(f) : \grad\,M\to\grad\,N$, homogeneous of degree $i$. 

\bigskip

\begin{lem}
\label{le:injgrHom}
The construction above provides an injective morphism of graded $R$-modules
$$\grad\,\Hom_A(M,N)\to\Hom_{\grad\,A}(\grad\,M,\grad\,N).$$
\end{lem}

\begin{lem}
\label{le:Moritafiltered}
Let $e$ be an idempotent of $A^{\le 0}$. Then $\grad\,A\cdot\xi(e)$
is a progenerator for $\grad\,A$ if and only if $Ae$ is a progenerator for $A$.
\end{lem}

\begin{proof}
Note that $\xi(e)$ is an idempotent of $\grad\,A$.
The $A$-module $Ae$ is a progenerator if and only if $e$ generates $A$ as an $(A,A)$-bimodule. The
lemma follows from Lemma \ref{le:generatefilt}.
\end{proof}

\bigskip

\section{Symmetric algebras}
\label{se:symmalg}

\medskip
Let us recall some basic facts about symmetric algebras (cf. for example \cite[\S 2, 3]{broueHighman}).
A {\it symmetric $R$-algebra} is an $R$-algebra $A$, finitely generated and
projective as an $R$-module, and endowed with an $R$-linear map $\tau_A:A\to R$ such that
\begin{itemize}
\item
$\tau_A(ab)=\tau_A(ba)$ for all $a,b\in A$ (i.e., $\tau_A$ is a trace)
and
\item the morphism of $(A,A)$-bimodules
$$\hat{\tau}_A:A\to \Hom_R(A,R),\ a\mapsto (b\mapsto \tau(ab))$$
is an isomorphism.
\end{itemize}
Such a form $\tau_A$ is called a {\em symmetrizing form} for $A$.

\medskip
Consider the sequence of isomorphisms
$$A\otimes_R A\xrightarrow[\sim]{\hat{\tau}\otimes \mathrm{id}}
\Hom_R(A,R)\otimes_R A\xrightarrow[\sim]{f\otimes a\mapsto (b\mapsto f(b)a)}\End_R(A).$$
The {\em Casimir element} is the inverse image of $\mathrm{id}_A$ through the composition
of maps above. The {\it central Casimir element} $\casimir_A$ is its image in $A$ by
the multiplication map $A\otimes_R A\to A$. It is an  element of $\Zrm(A)$.

\smallskip
Assume $A$ is free over $R$, with $R$-basis $\BC$ and dual basis $(b^\vee)_{b\in\BC}$ for
the bilinear form $A\otimes_R A\to R,\ a\otimes a'\mapsto \tau(aa')$. We have
$\casimir_A=\sum_{b\in\BC}bb^\vee$.

\bigskip

\begin{lem}
\label{le:symmtwisted}
Let $(A,\tau_A)$ be a symmetric $R$-algebra and $G$ a finite group acting on the
$R$-algebra $A$ and such that $\tau_A(g(a))=\tau_A(a)$ for all $g\in G$ and $a\in A$.

Let $B=A\rtimes G$ and define an $R$-linear form
$\tau_B:B\to R$ by $\tau_B(a\otimes g)=\tau_A(a)\delta_{1,g}$ for $a\in A$ and $g\in G$.

The form $\tau_B$ is symmetrizing for $B$.
\end{lem}

\begin{proof}
We have 
$$\tau_B((a\otimes g)(a'\otimes g'))=\tau_A(ag(a'))\delta_{g^{-1},g'}=
\tau_A(a'g^{-1}(a))\delta_{g^{-1},g'}=\tau_B((a'\otimes g')(a\otimes g)).$$

Given $g\in G$, let $B_g=A\otimes Rg$ and $C_g=\Hom_R(B_g,R)$. We have $B=\bigoplus_{g\in G}B_g$
and $\Hom_R(B,R)=\bigoplus_g C_g$. Given $g\in G$ and $a,a'\in A$,
	we have 
	$$\bigl(\hat{\tau}_B(a\otimes g)\bigr)(a'\otimes g^{-1})=
	\bigl(\hat{\tau}_A(g^{-1}(a))\bigr)(a').$$
	It follows that the restriction
	of $\hat{\tau}_B$ to $B_g$ is an isomorphism $B_g\longiso C_{g^{-1}}$
	with inverse
	$f\mapsto g\bigl(\hat{\tau}_A^{-1}(a'\mapsto f(a'\otimes g^{-1}))\bigr)\otimes
	g$.
\end{proof}

\bigskip
Let now $A$ be a filtered $R$-algebra with $A^{\le -1}=0$, $A^{\le d-1}\not=A$
and $A^{\le d}=A$ for some $d\ge 0$.
Denote by $p_i:A^{\le i}\to (\grad\,A)_i$ the canonical projection.
Let $\bar{\tau}:(\grad\,A)_d\to R$ be an $R$-linear form. We extend it to an $R$-linear form
on $\grad\,A$ by setting it to $0$ on $(\grad\,A)_i$ for $i<d$.
We define an $R$-linear form $\tau$ on $A$ as the composition
$$\tau:A\xrightarrow{p_d} (\grad\,A)_d\xrightarrow{\bar{\tau}} R.$$

\smallskip
Let $x\in A^{\le i}$ and $y\in A^{\le j}$. We have
$\tau(xy)=\bar{\tau}(p_d(xy))$.
We have $p_d(xy)=0$ if $i+j<d$. If
 $i+j=d$, we have $p_d(xy)=p_i(x)p_j(y)$, hence
$\tau(xy)=\bar{\tau}(p_i(x)p_j(y))$.
The $R$-module $\Hom_R(A,R)$ is filtered with
 $\Hom_R(A,R)^{\le i}=\Hom_R(A/A^{\le d-i-1},R)$
and
$\hat{\tau}$ is a morphism of filtered $R$-modules with
$\grad(\hat{\tau})=\hat{\bar{\tau}}$.

\bigskip

\begin{lem}
\label{le:trace}
Let $L$ be an $R$-submodule of $A^{\le 1}$ such that
$A=A^{\le 0}(R+L)^d$, $L^{d+1}\subset A^{<d}$ and $A^{\le 0}L=LA^{\le 0}$.

If $\bar{\tau}$ is a trace, then $\tau$ is a trace.
\end{lem}

\begin{proof}
Note that $A=A^{\le 0}L^d+A^{<d}$.
We have $p_d(A^{\le 0}L^dL)\subset p_d(A^{<d})=0$ and
$p_d(LA^{\le 0}L^d)=p_d(A^{\le 0}LL^d)=0$.
It follows that $\tau(al)=\tau(la)=0$ for $a\in A^{\le 0}L^d$ and $l\in L$. The
considerations above show that $\tau(al)=\tau(la)$ for $a\in A^{<d}$ and
$l\in L$, hence $\tau(al)=\tau(la)$ for all $a\in A$ and $l\in L$.
	Since $\tau(aa')=\tau(a'a)$ for all $a\in A$ and $a'\in A^{\le 0}$,
	we deduce that $\tau$ is a trace.
\end{proof}

\bigskip

The next proposition is inspired by a result of
Brundan and Kleshchev on degenerate cyclotomic Hecke algebras 
\cite[Theorem~A.2]{BrKl}.

\bigskip

\begin{prop}
\label{pr:filtsymm}
Assume $\grad\, A$ is projective and finitely generated as an 
$R$-module and assume $\tau$ and $\bar{\tau}$ are traces.

If $\bar{\tau}$ is a symmetrizing form for $\grad\,A$, then $\tau$ is a
symmetrizing form for $A$.
\end{prop}

\begin{proof}
Note that $A$ is a finitely generated projective $R$-module.
Since $\hat{\bar{\tau}}$ is an isomorphism,
it follows that $\hat{\tau}$ is an isomorphism (Lemma \ref{le:injsurjfiltered}). 
\end{proof}

\bigskip

\section{Weyl algebras}
\label{se:Weylalgebra}

\medskip

Let $V$ be a finite dimensional vector space over a field $\kb$ of
characteristic $0$.
Let $\DCB(V)=\Hbt_{1,0}$ be the Weyl algebra of $V$. This is the quotient of
the tensor algebra $\Trm_{\kb}(V\oplus V^*)$ by the relations
$$[x,x']=[y,y']=0,\ [y,x]=\langle y,x\rangle \,\text{ for }\,x,x'\in V^*
\,\text{ and }\,y,y'\in V.$$
There is an isomorphism of $\kb$-modules given by multiplication:
$\kb[V]\otimes\kb[V^*]\longiso \DCB(V)$.

The $\kb$-algebra $\DCB(V)$ is filtered, with $\DCB(V)^{\le -1}=0$,
$\DCB(V)^{\le 0}=\kb[V]$, $\DCB(V)^{\le 1}=\kb[V]\oplus \kb[V]\otimes V$
and $\DCB^{\le i}=(\DCB^{\le 1})^i$ for $i\ge 2$. The associated
graded algebra $\grad\,\DCB(V)$ is $\kb[V \times V^*]$. The associated Rees
algebra $\DCB_{\hbar}(V)$ \indexnot{D}{\DCB_{\mathchar'26\mkern-12muh}(V)} 
is the quotient of $\kb[\hbar]\otimes \Trm_{\kb}(V\oplus V^*)$ by the
relations
$$[x,x']=[y,y']=0,\ [y,x]=\hbar\langle y,x\rangle \,\text{ for }\,x,x'\in V^*
\,\text{ and }\,y,y'\in V.$$

\smallskip
Consider the induced $\DCB(V)$-module
$\DCB(V)\otimes_{\kb[V^*]}\kb$, where $\kb[V^*]$ acts on $\kb$ by
evaluation at $0$. Via the canonical isomorphism
$\kb[V]\longiso \DCB(V)\otimes_{\kb[V^*]}\kb,\ a\mapsto a\otimes 1$, we obtain the
faithful action of $\DCB(V)$ by polynomial differential operators on 
$\kb[V]$: an element $x\in V^*$ acts by multiplication, while $y\in V$ acts by $\partial_y=\frac{\partial}{\partial y}$.
As a consequence of the faithfulness of the action, the centralizer of $\kb[V]$ in
$\DCB(V)$ is $\kb[V]$.

\smallskip
Note that there is an injective morphism of $\kb[\hbar]$-algebras
$$\DCB_{\hbar}(V)\hookrightarrow \kb[\hbar]\otimes \DCB(V),\ V^*\ni x\mapsto x,\ V\in y\mapsto \hbar y.$$
This provides by restriction $\kb[\hbar]\otimes\kb[V]$ with the structure of a faithful
representation of $\DCB_{\hbar}(V)$.

\chapter{Galois theory and ramification}\label{chapter:galois-rappels}

\bigskip

\section{Setting}

\medskip

Let $R$ be a commutative ring, $G$ a finite group acting on $R$ and 
$H$ a subgroup of $G$. We set $Q=R^H$ and $P=R^G$, so that 
$P \subset Q \subset R$. 
Let $\rG$ be a prime ideal of $R$. We denote by $k_R(\rG)$  \indexnot{k}{k_R(\rG)}  
the fraction field of $R/\rG$ (that is, the quotient $R_\rG/\rG R_\rG$) and by
$G_\rG^D$   
\indexnot{G}{G_\rG^D, G_\rG^I}  
the stabilizer of $\rG$ in $G$. This subgroup of $G$ acts on $R/\rG$ 
and we denote by $G_\rG^I$ the kernel of this action. In other words,
$$G_\rG^I=\{g \in G~|~\forall~r \in R,~g(r) \equiv r \mod \rG\}.$$
The group $G_\rG^D$ (respectively $G_\rG^I$) is called the 
{\it decomposition group} (respectively the {\it inertia group}) 
of $G$ at $\rG$. 

We fix in this chapter a prime ideal $\rG$ of $R$ and we set 
$\qG=\rG \cap Q$ and $\pG=\rG \cap P = \qG \cap P$:
$$\diagram
\rG \dline & \subset & R \dline \\
\qG \dline & \subset & Q \dline \\
\pG        & \subset & P 
\enddiagram$$
We also set 
$$\r_G : \Spec R \to \Spec P,$$
$$\r_H : \Spec R \to \Spec Q$$
$$\Upsilon : \Spec Q \to \Spec P\leqno{\text{and}}$$
the maps respectively induced by the inclusions $P \subset R$, $Q \subset R$ 
and $P \subset Q$. We have $\r_G = \Upsilon \circ \r_H$: in other words, the diagram 
$$\xymatrix{
\Spec R \ar[rr]^{\DS{\r_H}} \ar@/_1.2cm/[rrrr]_{\DS{\r_G}}&& \Spec Q \ar[rr]^{\DS{\Upsilon}} 
&& \Spec P
}$$
is commutative. For instance,
$$\r_G(\rG)=\pG,\qquad\r_H(\rG)=\qG\qquad \text{and}\qquad \Upsilon(\qG)=\pG.$$
Finally, we set
$$D=G_\rG^D\qquad \text{and} \qquad I=G_\rG^I.$$

\medskip

\section{Around Dedekind's Lemma}\label{sub:dedekind} 

\medskip

Recall that $(R,\times)$ is a monoid. 
Given $M$ a monoid, we denote by $\Hom_{\text{mon}}(M,R)$ the set of morphisms of monoids 
$M \to (R,\times)$, a subset of the 
$R$-module $\FC(M,R)$ of maps from $M$ to $R$.
. Given $A$ a commutative ring, we denote by $\Hom_{\mathrm{ring}}(A,R)$ 
the set of morphisms of rings from $A$ to $R$, a subset of 
$\Hom_{\mathrm{mon}}((A,\times),R)$.
\bigskip

\noindent{\bf Dedekind's Lemma.} 
{\it If $R$ is a domain, then 
$\Hom_{\mathrm{mon}}(M,R)$ is an $R$-linearly independent subset
of $\FC(M,R)$.}

\bigskip

\begin{proof}
Let $\ph_1$,\dots, $\ph_d$ be a family of distinct elements of $\Hom_{\text{mon}}(M,R)$ and  consider
$\l_1$,\dots, $\l_d$ in $R$ such that 
$$\l_1 \ph_1 + \cdots + \l_d \ph_d = 0.\leqno{(*)}$$
We shall show by induction on $d$ that $\l_1=\l_2=\cdots=\l_d=0$. 
When $d=1$, this is clear as $\ph_1(1)=1$. 

So assume that $d \ge 2$ and that there is no non-trivial $R$-linear 
relation of length $\le d-1$ between distinct 
	elements of $\Hom_{\mathrm{mon}}(M,R)$. 
Since $\ph_1 \neq \ph_2$, there exists $m_0 \in M$ such that 
$\ph_1(m_0) \neq \ph_2(m_0)$. It follows from $(*)$ that 
$$ \l_1 \ph_1(m_0 m) + \cdots + \l_d \ph_d(m_0 m) = 0$$
$$\ph_1(m_0) \cdot \bigl(\l_1 \ph_1(m) + \cdots + \l_d \ph_d(m)\bigr) = 0\leqno{\text{and}}$$
for all $m \in M$. By subtracting the second equation to the first one, we get 
$$
\sum_{i=2}^d \l_i (\ph_i(m_0)-\ph_1(m_0)) \ph_i = 0.$$
By the induction hypothesis, we have $\l_2 (\ph_2(m_0)-\ph_1(m_0))=0$, 
hence $\l_2=0$ because $R$ is a domain. The induction hypothesis 
allows also to conclude that $\l_1=\l_3=\cdots=\l_d=0$. 
\end{proof}

\bigskip

\begin{coro}\label{coro:dedekind}
Let $A$ be a commutative ring and assume that $R$ is a domain. Then  
$\Hom_{\mathrm{ring}}(A,R)$ is an $R$-linearly independent subset of $\FC(A,R)$.
\end{coro}

\bigskip

\begin{proof}
Indeed, the set $\Hom_{\mathrm{ring}}(A,R)$ is a subset of the set 
$\Hom_{\mathrm{mon}}((A,\times),R)$ and the corollary follows
from Dedekind's Lemma.
\end{proof}

\bigskip

\section{Decomposition group, inertia group}\label{appendice:galois}

\medskip

We recall some classical results:

\begin{prop}\label{prop:max-max}
The ideal $\rG$ is maximal if and only if $\pG$ is maximal.
\end{prop}

\bigskip

\begin{prop}\label{galois transitif}
The group $G$ acts transitively on the fibers of $\r_G$.
\end{prop}

\begin{proof}
See~\cite[Chapter~5,~\S{2},~Theorem~2(i)]{bourbaki}.
\end{proof}

\begin{rema}\label{H} 
The statement above can also be applied to $H$: the group $H$ acts transitively on 
the fibers of $\r_H$.\finl
\end{rema}

\bigskip

\begin{theo}\label{bourbaki}
The field extension $k_R(\rG)/k_P(\pG)$ is normal, 
with Galois group $D/I$ ($=G_\rG^D/G_\rG^I$). 
\end{theo}

\begin{proof}
See~\cite[Chapter~5,~\S{2},~Theorem~2(ii)]{bourbaki}.
\end{proof}

\bigskip

\begin{coro}\label{coro bourbaki}
Let $\rG'$ be a prime ideal of $R$ containing $\rG$ and let $D'=G_{\rG'}^D$ 
and $I'=G_{\rG'}^I$. Then $D'/I'$ is isomorphic to a subquotient of $D/I$.
\end{coro}

\begin{proof}
By replacing $R$ by $R/\rG$, $Q$ by $Q/\qG$ and $P$ by $P/\pG$, 
	the group $D/I$ can be identified with $G$ (Theorem \ref{bourbaki})
	and the corollary follows from Theorem~\ref{bourbaki}. 
\end{proof}

\bigskip

\begin{theo}\label{raynaud}
If $Q$ is unramified over $P$ at $\qG$ (i.e. if $\pG Q_\qG=\qG Q_\qG$), 
then $I$ is contained in $H$.
\end{theo}

\begin{proof}
See~\cite[Chapter~\MakeUppercase{\romannumeral 10},~Th\'eor\`eme~1]{raynaud}.
\end{proof}

\bigskip

\section{On the $P/\pG$-algebra $Q/\pG Q$}\label{section:Q/P} 

\medskip

\subsection{Double classes} 
The classical proposition~\ref{reduction} below is a basic tool for our 
constructions. We  
provide a proof for the convenience of the reader.
Given $g \in G$, the composed morphism $Q \stackrel{g}{\longto} R \stackrel{{\mathrm{can}}}{\longto} R/\rG$ 
factors through a morphism $\gba : Q/\pG Q \to R/\rG$. The following remark is obvious:
\equat\label{eq:IGH}
\textit{Given $h \in H$ and $i \in I$, we have $\overline{igh}=\gba$.}
\endequat
As a consequence, we have a map
\equat\label{def:IGH}
\fonctio{I \backslash G/H}{\Hom_{(P/\pG)\text{-}\alg}(Q/\pG Q,R/\rG)}{IgH}{\gba}.
\endequat
Note that $\Hom_{(P/\pG)\text{-}\alg}(Q/\pG Q,R/\rG)=\Hom_{P\text{-}\alg}(Q,R/\rG)$. 
Given $\ph \in \Hom_{(P/\pG)-\text{alg}}(Q/\pG Q,R/\rG)$
we denote by $\pht$ the composition 
$Q \stackrel{{\mathrm{can}}}{\longto} Q/\pG Q \stackrel{\ph}{\longto} R/\rG$. This defines a map
\equat\label{def:ker}
\fonctio{\Hom_{(P/\pG)\text{-}\alg}(Q/\pG Q,R/\rG)}{\Upsilon^{-1}(\pG)}{\ph}{\Ker \pht.}
\endequat
Since $R/\rG$ is a domain, $\Ker \pht$ is a prime ideal of $Q$ 
and it is clear that $\Ker \pht \in \Upsilon^{-1}(\pG)$. 

Given $g \in G$, we have $g(\rG) \cap Q \in \Upsilon^{-1}(\pG)$. Moreover, for
$h \in H$ and $d \in D$ we have
$$hgd(\rG) \cap Q=g(\rG) \cap Q.$$
We obtain a map 
\equat\label{def:DGH}
\fonctio{D\backslash G / H}{\Upsilon^{-1}(\pG)}{DgH}{g^{-1}(\rG) \cap Q.}
\endequat

\bigskip

\begin{prop}\label{reduction}
The map $I \backslash G/H \longto \Hom_{(P/\pG)\text{-}\alg}(Q/\pG Q,R/\rG)$ 
	defined in (\ref{def:IGH}) is bijective, as well as the map 
	$D\backslash G / H \to \Upsilon^{-1}(\pG)$ defined in (\ref{def:DGH}).
	Moreover, 
the diagram
$$\diagram
I \backslash G / H \xto[0,3]_{\DS{\sim} \qquad\qquad}^{\DS{IgH \mapsto \gba\quad\quad\quad}} 
\ddto_{\DS{\mathrm{can}}} &&& \Hom_{(P/\pG)\text{-}\alg}(Q/\pG Q,R/\rG) \ddto^{\DS{\ph \mapsto \Ker \pht}} \\
&&&\\
D\backslash G / H \xto[0,3]^{\DS{\sim}}_{\DS{DgH \mapsto g^{-1}(\rG) \cap Q}} &&& \Upsilon^{-1}(\pG) \\
\enddiagram$$
is commutative.
\end{prop}

\bigskip

\begin{proof}
Let us start by showing that the first map is injective. 
Let $g$ and $g'$ be two elements of $G$ such that $\gba=\gba'$. This means that 
$$\forall~q \in Q,~g(q) \equiv g'(q) \mod \rG.$$
Consequently,
$$\forall~r \in R,~\sum_{h \in H} g h(r) \equiv \sum_{h \in H} g'h(r) \mod \rG.$$
By Dedekind's Lemma, the family of morphisms 
	of rings $R \to R/\rG$ is $(R/\rG)$-linearly independent. 
So there exists $h \in H$ such that 
$$\forall~r \in R,~g (r) \equiv g' h(r) \mod \rG$$
or, equivalently, 
$$\forall~r \in R, ~ g' h (g^{-1}(r)) \equiv r \mod \rG.$$
In other words, $g' h g^{-1} \in I$ hence $g' \in I g H$.

\medskip

	Let us now show that the first map is surjective. Let $\ph \in \Hom_{(P/\pG)-\text{alg}}(Q/\pG Q,R/\rG)$ 
and let $\qG'=\Ker \pht$. Since $\ph$ is $(P/\pG)$-linear, we have $\qG' \cap P=\pG$. 
Let $\rG'$ be a prime ideal of $R$ lying over $\qG'$. There exists $g \in G$ such that 
$\rG'=g(\rG)$. So the map $g \circ \pht : Q \to R/\rG'$ 
has $\qG'=\rG' \cap Q$ for kernel and is $P$-linear. By Theorem~\ref{bourbaki}, 
there exists $d \in G_{\rG'}^D$ such that 
$g \circ \pht(q) \equiv d(q) \mod \rG'$ for all $q \in Q$. 
Hence, $\pht(q) \equiv g^{-1}d(q) \mod \rG$, that is, 
$\ph = \overline{g^{-1}d}$. 
%
%
%
%

\medskip

Let us now show that the second map is bijective. 
Given $\qG' \in \Upsilon^{-1}(\pG)$, there exists $\rG' \in \Spec R$ such that $\qG' \cap Q=\rG'$. 
Also, $\rG' \cap P=\qG' \cap P = \pG$ and so, by Proposition~\ref{galois transitif}, 
there exists $g \in G$ such that $\rG' = g(\rG)$. This shows that the bottom 
horizontal row is surjective. The injectivity follows again from 
	Proposition~\ref{galois transitif} (applied to the extension $R/Q$). 

\medskip

The commutativity of the diagram follows from the previous arguments. 
\end{proof}

\bigskip

\subsection{Residue fields}
Let $g \in G$. We put $\qG_g=\rG \cap g(Q)$. 
Note that $\qG_g \cap P=\pG$ and that we obtain a sequence of morphisms of rings 
$P/\pG \longinjto g(Q)/\qG_g \longinjto R/\rG$. So we have a sequence of inclusions of fields 
$$k_P(\pG) \quad \subset \quad k_{g(Q)}(\qG_g) \quad \subset \quad k_R(\rG).$$

\bigskip

\begin{lem}\label{gal:bourbaki}
The extension $k_R(\rG)/k_{g(Q)}(\qG_g)$ is normal with Galois group 
$(D \cap \lexp{g}{H})/(I \cap \lexp{g}{H})$.
\end{lem}


\begin{proof}
This follows from the fact that $g(Q)=R^{\lexp{g}{H}}$ and from Theorem~\ref{bourbaki}.
\end{proof}

\medskip

\begin{coro}\label{dgh igh}
Assume that for all prime ideals $\qG' \in \Upsilon^{-1}(\pG)$ we have 
$k_Q(\qG')=k_P(\pG)$. Then $D\backslash G /H=I\backslash G/H$.
\end{coro}

\begin{proof}
By Lemma~\ref{gal:bourbaki} and Theorem~\ref{bourbaki}, 
it follows from the assumption that, for all $g \in G$, we have
$(D \cap \lexp{g}{H})/(I\cap \lexp{g}{H}) \simeq D/I$. In other words, 
$$\forall~g \in G,~D = I \cdot (D \cap \lexp{g}{H}).$$
Now let $g \in G$ and $d \in D$. Then there exists $i \in I$ and $h \in H$ such that 
$d = ighg^{-1}$, that is $d g = igh$. We deduce that $DgH=IgH$. 
\end{proof}

\medskip

\begin{lem}\label{inertie}
Assume that $Q$ is unramified over $P$ 
	at all prime ideals in $\Upsilon^{-1}(\pG)$.
Then $I \subset \bigcap_{g \in G} \lexp{g}{H}$. 
\end{lem}

\begin{proof}
Let $g \in G$. Since $g(\rG) \cap Q \in \Upsilon^{-1}(\pG)$ it follows from 
Theorem~\ref{raynaud} that $\lexp{g}{I} \subset H$ 
(since $\lexp{g}{I}$ is the inertia group of $g(\rG)$). Hence, $I \subset \lexp{g^{-1}}{H}$. 
\end{proof}

\medskip

\begin{prop}\label{cloture galoisienne}
If $I=\bigcap_{g \in G} \lexp{g}{H} = 1$ and if the extension $k_R(\rG)/k_P(\pG)$ is separable, 
then $k_R(\rG)/k_P(\pG)$ is a Galois closure of the family 
	of extensions $k_{\lexp{g}{Q}}(g(\qG))/k_P(\pG)$ for
$g \in [G / H]$. 
\end{prop}

\noindent{\sc Remark - } 
If $R$ is a domain (which implies that $P$ and $Q$ are domains) and if $G$ 
acts faithfully on $R$, then the assumption
$\bigcap_{g \in G} \lexp{g}{H} = 1$ is equivalent to the fact that 
the extension $\Frac(R)/\Frac(P)$ is a Galois closure of $\Frac(Q)/\Frac(P)$.

Note also that the assumption $I=1$ implies that $G$ acts faithfully.\finl

\bigskip

\begin{proof}
By Theorem~\ref{bourbaki}, 
the extension $k_R(\rG)/k_P(\pG)$ is normal with Galois group $D$. 
By Lemma~\ref{gal:bourbaki}, the extension 
$k_R(\rG)/k_{g(Q)}(\qG_g)$ is normal with Galois group  $D \cap \lexp{g}{H}$. 

Let $k$ be a normal closure of the family of extensions 
$k_{g(Q)}(\qG_g)/k_P(\pG)$, $g \in [ G / D]$. The
Galois group $\Gal(k_R(\rG)/k)$ is the intersection of the $D$-conjugates
of the groups $D \cap \lexp{g}{H}$, for $g$ running in $[ G / D]$. 
So
$$\Gal(k_R(\rG)/k)=\bigcap_{\stackrel{\SS g \in [ G/D]}{d \in D}} 
\lexp{d}{(D \cap \lexp{g}{H})}=\bigcap_{g \in G}
	(D \cap \lexp{g}{H)}=D\cap \bigcap_{g \in G} \lexp{g}{H}=1.$$
The proposition follows.
\end{proof}

\bigskip

\subsection{The case of fields}\label{subsection:corps} 
The discussion above simplifies considerably when $R$ is a field.

\medskip

\boitegrise{{\bf Assumption.} 
{\it In this subsection, and only in this subsection, 
we assume that $R$ is a field: it will be denoted by $M$. We set 
$L=Q=M^H$ and $K=P=M^G$. We also assume that $G$ acts faithfully
on $M$. Hence, $M/K$ is a Galois extension with Galois group $G$ and 
$M/L$ is a Galois extension with Galois group $H$.}}{0.75\textwidth}

\medskip

It follows from the assumption that $\pG=\qG=\rG=0$ and that $D=G$ and $I=1$. 
Hence Proposition~\ref{reduction} provides a bijective map 
$$G/H \stackrel{\sim}{\longleftrightarrow} \Hom_{K-\text{alg}}(L,M).$$
Given $g \in G$, the morphism of $K$-algebras $L \to M$, $q \mapsto g(q)$, extends to 
a morphism of $M$-algebras 
$$\fonction{g_L}{M \otimes_K L}{M}{m \otimes_K l}{mg(l).}$$

\bigskip

\begin{prop}\label{iso galois}
The morphism of $M$-algebras 
$$\sum_{g \in [G/H]} g_L : M \otimes_K L \longto \mathop{\bigoplus}_{g \in [G/H]} M$$
is an isomorphism.
\end{prop}

\begin{proof}
Since $L$ is a $K$-vector space of dimension $|G/H|$, it follows that
$M \otimes_K L$ is an $M$-vector space of dimension $|G/H|$. 
It is enough to show that  $\sum_{g \in [G/H]} g_L$ is injective, 
which is equivalent to the $M$-linear independence of the family of maps 
$L \to M$, $q \mapsto g(q)$, where $g$ runs over $[G/H]$.
Corollary~\ref{coro:dedekind} provides the conclusion.
\end{proof}

\bigskip

\section{Integral closure}\label{section:cloture integrale}

\medskip

\begin{prop}\label{int clos}
Let $f \in P[\tb]$, let $P'$ be a $P$-algebra containing $P$ and let $g \in P'[\tb]$. 
Assume that $f$ and $g$ are monic and that $g$ 
divides $f$ (in $P'[\tb]$). Then the coefficients of $g$ are integral over $P$.
\end{prop}

\begin{proof}
See~\cite[Chapter~5,~\S{1},~Proposition~11]{bourbaki}.
\end{proof}

\begin{coro}\label{minimal clos}
If $P$ is integral and integrally closed, 
with fraction field $K$, if $A$ is a $K$-algebra and if 
$x \in A$ is integral over $P$, then the minimal polynomial of $x$ over $K$ 
belongs to $P[\tb]$.
\end{coro}

\begin{proof}
See~\cite[Chapter~5,~\S{1},~Corollary~of~Proposition~11]{bourbaki}.
\end{proof}

\begin{prop}\label{cloture polynomiale}
If $P$ is a domain and if $f \in P[\tb,\tb^{-1}]$ is integral over $P$, then $f \in P$.
\end{prop}

\begin{proof}
Let $d \ge 1$ and let $p_0$, $p_1$,\dots, $p_{d-1}$ be elements of $P$ 
such that $p_0 + p_1 f + \cdots + p_{d-1} f^{d-1} = f^d$. Let $\d$ be the $\tb$-valuation of 
$f$ and $\d'$ its degree. Since $P$ is a domain, the degree of $f^d$ is $d\d'$, 
and so the above equality can hold only if $\d'=0$. Similarly, $\d=0$. So $f$ is constant.
\end{proof}


\bigskip

\section{On the calculation of Galois groups\label{sec:calcul}}

\medskip

Let $K$ be a field and let 
$f(\tb)=\tb^d + a_{d-1} \tb^{d-1} + \cdots + a_1 \tb + a_0 \in K[\tb]$ 
be a separable polynomial. 
We denote by $M$ a splitting field of $f$ (over $K$) and we put
$$\Gal_K(f)=\Gal(M/K).$$
The group $\Gal_K(f)$ is  the {\it Galois group} of $f$ over $K$. 

Let $t_1$,\dots, $t_d$ be elements of $M$ such that 
$$f(\tb)=\prod_{i=1}^d (\tb-t_i),$$
so that
$$M=K[t_1,\dots,t_d]=K(t_1,\dots,t_d).$$
This provides an injective morphism of groups 
$$\Gal_K(f) ~\longinjto~ \SG_d.$$

Assume now that $P$ is an integrally closed domain with field of fractions $K$
and that $f \in P[\tb]$. 
Let $R$ denote the integral closure of $P$ in  $M$ and let $G=\Gal(M/K)$. 
Since $P$ is integrally closed we have $P=R^G$.
Given $r \in R$, we denote by $\rba$ its image 
in $R/\rG$. Write
$$\fba=\prod_{j=1}^l f_j,$$
where $f_j \in k_P(\pG)[\tb]$ is an irreducible polynomial. 
We have $D/I=\Gal(k_R(\rG)/k_P(\pG))$ by Theorem~\ref{bourbaki}.  
Since $R$ contains $t_1$,\dots, $t_d$ we have
$$\fba(\tb)=\prod_{i=1}^d (\tb-\tba_i).$$
We denote by $\O_j$ the subset of $\{1,2,\dots,d\}$ such that 
$$f_j(\tb)=\prod_{i \in \O_j} (\tb-\tba_i).$$
Let $k_j=k_P(\pG)((\tba_i)_{i \in \O_j})$: it is a splitting field of $f_j$ over $k_P(\pG)$. 
Let $G_j=\Gal(k_j/k_P(\pG))=\Gal(\fba_j)$. Then
\equat\label{eq:galois-surjectif}
\text{\it the canonical morphism $D/I=\Gal(k_R(\rG)/k_P(\pG)) \to \Gal(k_j/k_P(\pG))=G_j$ is surjective}
\endequat
for all $j$. Since $G_j$ acts transitively on $\O_j$, we obtain in particular that 
\equat\label{eq:omega-divise-G}
\text{\it $|\O_j|$ divides $|G|$ for all $j$.}
\endequat

\bigskip

\section{Some facts on discriminants}

\medskip

Let $f(\tb) \in P[\tb]$ be a monic polynomial of degree $d$. We denote by 
$\disc(f)$  \indexnot{d}{\disc}  its discriminant. We have
\equat\label{discriminant carre}
\disc(f(\tb^2))=(-4)^d \disc(f)^2 \cdot f(0).
\endequat
\begin{proof}
By easy specialization arguments, we may assume that $P$ is an algebraically closed field. 
Let 
$E_1$,\dots, $E_d$ be the elements of $P$ such that 
$$f(\tb)=\prod_{i=1}^d (\tb-E_i).$$
We fix a square root $e_i$ of $E_i$ in $P$. So
$$f(\tb^2)=\prod_{\stackrel{\SS{1 \le i \le d}}{\e \in \{1,-1\}}} (\tb-\e e_i)$$
and the discriminant of $f(\tb^2)$ is then equal to 
$$\disc(f(\tb^2)) = 
\Bigl(\prod_{\stackrel{\SS{1 \le i < j \le d}}{\e,\e' \in \{1,-1\}}} (\e e_i -\e' e_j)^2\Bigr)
\cdot \prod_{i=1}^d (e_i - (-e_i))^2.$$
In other words, 
$$\disc(f(\tb^2)) = 4^d \cdot 
\Bigl(\prod_{1 \le i < j \le d} (E_i-E_j)^4\Bigr) \cdot \prod_{i=1}^d E_i
= 4^d ~ \disc(f)^2 \cdot (-1)^d f(0),$$
as expected.
\end{proof}

\bigskip

Let us conclude with another easy result: 
\equat\label{discriminant t}
\disc(\tb f(\tb))=\disc(f) \cdot f(0)^2.
\endequat
\begin{proof}
As in the previous proof, we may assume that $P$ is an algebraically closed field, 
and we denote by $E_1$,\dots, $E_d$ the elements of $P$ such that 
$$f(\tb)=\prod_{i=1}^n (\tb-E_i).$$
Then
$$\disc(\tb f(\tb)) = \Bigl(\prod_{1 \le i < j \le d} (E_i-E_j)^2\Bigr)
\cdot \prod_{i=1}^d (0-E_i)^2.$$
Whence the result.
\end{proof}

\section{Topological version}
\label{se:topologyinertia}

Let $Y$ be a locally simply connected separated topological space endowed
with a faithful left action of a finite group $G$. 
Let $X=G\setminus Y$ and let $\pi:Y\to X$ be the quotient map.

Let $Y^{\nr}=\{y\in Y | \Stab_G(y)=1\}$ be the complement of the
ramification locus and let $X^{\nr}=\pi(Y^{\nr})$.
Fix $y_0\in Y^{\nr}$ and let $F=G\cdot y_0$. We define a right action of
$G$ on $F$ by $(g\cdot y_0)\cdot g'=gg'\cdot y_0$.

\smallskip
Let $y_1$ be a point of $Y$ in the closure of the connected component of 
$Y^{\nr}$ that contains $y_0$.
Let $I=\mathrm{Stab}_G(y_1)$.
The right action of $I$ on $F$ can be described in terms of lifting of paths, as
we recall below.

\medskip
Fix a path $\tilde{\gamma}:[0,1]\to Y$ with $\tilde{\gamma}([0,1))
\subset Y^{\nr}$, $\tilde{\gamma}(0)=y_0$ and $\tilde{\gamma}(1)=y_1$.
Let $\gamma=\pi(\tilde{\gamma})$.

\begin{lem}
Given $y\in F$, there is a unique path $\tilde{\gamma}_y$ in $Y$ starting
at $y$ and lifting $\gamma$.

Given $y',y''\in F$, we have $y''\in y'\cdot I$ if and only if 
$\tilde{\gamma}_{y'}(1)=\tilde{\gamma}_{y''}(1)$.
\end{lem}

\begin{proof}
Let 
$E=\pi^{-1}(\gamma([0,1]))$. We have
$$E=\coprod_{\Omega\in G/I} \bigl(\bigcup_{\omega\in\Omega}
\omega(\tilde{\gamma}([0,1]))\bigr)$$
where the $\bigcup_{\omega\in\Omega} \omega(\tilde{\gamma}([0,1]))$
are the connected components of $E$ and the $\omega(\tilde{\gamma}([0,1)))$
are the connected components of $E \setminus (G\cdot y_1)$.

Let $y\in F$. There is a unique element $g\in G$ such that $y=g\cdot y_0$.
The restricted path $g(\tilde{\gamma})_{|[0,1)}$ is the unique lift of
$\gamma_{|[0,1)}$ starting at $y$. 
It follows from the description of $E$ that $g(\tilde{\gamma})$ is the
unique lift of $\gamma$ starting at $y$. The lemma follows.
\end{proof}

\medskip
We consider now the case of a non-Galois covering.
Let $H$ be a subgroup of $G$ and let $\bar{Y}=H\setminus Y$. We denote by 
$\phi:Y\to \bar{Y}$ the quotient map and by $\psi:\bar{Y}\to X$ the map such
that
$\pi=\psi\circ\phi$. Let $\bar{F}=\phi(F)$. The right action of $I$ on $F$
induces
a right action on $\bar{F}$. We have a bijection
$H\setminus G\xrightarrow{\sim}\bar{F},\
 Hg\mapsto \phi(g\cdot y_0)$ and the right action of $I$ on $\bar{F}$
corresponds to the right action on $H\setminus G$ by right multiplication.

\begin{lem}
\label{le:inertiechemins}
Given $\bar{y}\in \bar{F}$, there is a unique path $\bar{\gamma}_{\bar{y}}$ in
$\bar{Y}$ starting at $\bar{y}$ and lifting $\gamma$.

Given $\bar{y}',\bar{y}''\in \bar{F}$, we have 
$\bar{y}''\in \bar{y}'\cdot I$ if and only if 
$\bar{\gamma}_{\bar{y}'}(1)=\bar{\gamma}_{\bar{y}''}(1)$.
\end{lem}

\begin{proof}
There is an element $g\in G$ such that $\bar{y}=\phi(g\cdot y_0)$, and $Hg$
is uniquely determined by $\bar{y}$. The path $\phi(g(\tilde{\gamma}))$ is the
unique lift of $\gamma$ starting at $\bar{y}$. The lemma follows.
\end{proof}

\medskip

We assume now that $R$ is a finitely generated commutative 
reduced $\CM$-algebra
and $Y$ is the topological space $(\Spec R)(\CM)$, for the classical
topology. We have $\bar{Y}=(\Spec Q)(\CM)$ and $X=(\Spec P)(\CM)$.

The prime ideal $\rG$ of $R$ corresponds to an irreducible subvariety
$Z$ of $\Spec R$. There is a non-empty Zariski open subset $U$ of $Z$ such that
given $y\in U(\CM)$, we have $\mathrm{Stab}_G(y)=G_\rG^I$.

Fix a point $y_1\in U(\CM)$.
We have $I=G_\rG^I$. Lemma \ref{le:inertiechemins} provides
a topological description of the orbits of $I$ on the fibers of
$\psi$.

\chapter{Gradings and integral extensions}\label{appendice graduation}
\setcounter{section}{0}

\bigskip

\section{Idempotents, radical}\label{intro graduation}

\medskip

Let $\Gamma$ be a monoid. Denote by
$\Delta:\BZ[\Gamma]\to\BZ[\Gamma]\otimes_\BZ\BZ[\Gamma],\ \gamma\mapsto\gamma\otimes
\gamma$ the comultiplication. Let $A$ be a ring. Let us recall the equivalence
between the notion of a $\Gamma$-grading on $A$ and that of a coaction of
$\BZ[\Gamma]$ on $A$.

 Put
$A[\Gamma]=A\otimes_\BZ \BZ[\Gamma]$.
Given $A=\mathop{\bigoplus}_{\gamma\in\Gamma} A_\gamma$ a $\Gamma$-graded ring
structure on $A$, we have a morphism of rings 
$$\mu=\mu_A:A\longto A[\Gamma],\ A_\gamma\ni a\mapsto a\otimes \gamma$$
such that
\begin{equat}
\label{eq:gradedmu}
\begin{cases}
& \mu\otimes\id:A\otimes_\BZ \BZ[\Gamma]\to A[\Gamma],\ a\otimes\gamma\mapsto
 \mu(a)\gamma \text{ is an isomorphism and}\\
& (1\otimes\Delta)\circ\mu=(\mu\otimes \id)\circ\mu:
A\longto A\otimes_\BZ \BZ[\Gamma]\otimes_\BZ\BZ[\Gamma].
\end{cases}
\end{equat}

Conversely, consider a morphism of rings
$\mu:A\to A[\Gamma]$ satisfying the two properties (\ref{eq:gradedmu}) above.
Let $A_\gamma=\mu_A^{-1}(A\otimes\gamma)$. Note that given $a\in A_\gamma$
and $a'\in A_{\gamma'}$, we have $aa'\in A_{\gamma+\gamma'}$.
Let $a\in A$ and write $\mu(a)=\sum_{i=1}^n a_i\otimes\gamma_i$ with $a_i\in A$
and $\gamma_i\in\Gamma$. We have $(1\otimes\Delta)\circ\mu(a)=
(\mu\otimes\Id)\circ\mu(a)$, hence
$\sum_i a_i\otimes\gamma_i\otimes\gamma_i=\sum_i \mu(a_i)\otimes\gamma_i$.
It follows that $\mu(a_i)=a_i\otimes\gamma_i$, hence $a_i\in A_{\gamma_i}$.
We deduce that $A=\sum_\gamma A_\gamma$.
The first property shows that $\mu$ is injective, hence $A=\bigoplus_\gamma
A_\gamma$: we have obtained a $\Gamma$-graded ring structure on $A$.

Given $f:\Gamma\to\Gamma'$ a morphism of monoids and given a $\Gamma$-grading
on $A$, we have a $\Gamma'$-grading on $A$ given by
$A'_{\gamma'}=\bigoplus_{\gamma\in f^{-1}(\gamma')}A_\gamma$.

\bigskip
\begin{prop}\label{graduation idem}
Assume that $A$ is commutative and $\Gamma$ is a torsion-free abelian group.
Let $e$ be an idempotent of $A$. Then $e \in A_0$.
\end{prop}

\begin{proof}
Replacing $A$ by the ring generated by the homogeneous components of $e$, we may assume that 
$A$ is noetherian and $\Gamma$ is finitely generated.
Given $d>0$, consider the ring morphism
$m_d:A[\Gamma]\to A[\Gamma],\ a\otimes\gamma\mapsto a\otimes\gamma^d$.
Note that $m_d(\mu(e))$ is an idempotent of $A[\Gamma]$. If
$e{\not\in}A_0$, then $\mu(e){\not\in A}$, hence
the $m_d(\mu(e))$ are distinct for different $d$.
Since $A$ is commutative and noetherian, 
$A[\Gamma]$ is also commutative and noetherian, and so contains only finitely 
many idempotents. We deduce that $e \in A_0$.
\end{proof}

\bigskip

Let us recall a basis result on the homegeneity of the radical 
\cite[Theorem 2.5.40]{Row}.

\begin{prop}\label{graduation radical}
If $\Gamma$ is a free abelian group, then $\Rad(A)$ is a homogeneous ideal of
$A$.
\end{prop}

\begin{proof}
Let $r\in\Rad(A)$. Write $r=\sum_{1\le i\le d}r_i$ with $r_i\in A_{\gamma_i}$ for
some $\gamma_i\in\Gamma$ with $\gamma_i\neq\gamma_j$ for $i\neq j$. Fix $i\in\{1,\ldots,d\}$.
Fix a group morphism $\rho:\Gamma\to\BZ$ such that $\rho(\gamma_j){\not=}
\rho(\gamma_i)$ for $j{\not=}i$.

Let $n$ be a positive integer with $n>|\rho(\gamma_i)-\rho(\gamma_j)|$
for all $j{\not=}i$. Let $\Gamma'=\BZ/n\BZ$ and write $\zeta$ for its generator $1$, so
that $A[\Gamma']=A\otimes_{\BZ}\BZ[\zeta]$. We have $\Rad(A)=\Rad(A[\Gamma'])\cap A$.
Consider the group morphism
$\theta_l:\Gamma\to\Gamma',\ \gamma\mapsto \zeta^{l\rho(\gamma)}$
 and denote by $\theta_l$
again the induced ring morphism $A[\Gamma]\to A[\Gamma']$.
We have 
$$\sum_{l=1}^n \zeta^{-\rho(\gamma_i)l}\theta_l(\mu(r))=
\sum_{1\le j\le d}r_j\sum_{1\le l\le n}
\zeta^{l(\rho(\gamma_j)-\rho(\gamma_i))}=nr_i.$$
Since $\theta_l\circ\mu$ is a morphism of rings, it follows that
$nr_i\in \Rad(A[\Gamma'])$, hence $nr_i\in\Rad(A)$. Similarly, we obtain
$(n+1)r_i\in\Rad(A)$, hence $r_i\in\Rad(A)$. So, $\Rad(A)$ is a homogeneous ideal.
\end{proof}

\bigskip

We assume for the remainder of \S\ref{intro graduation} that $\Gamma=\BZ$.
Let $B$ be a ring containing $A$ and let $\xi \in B^\times$ commuting with $A$.
There exists a unique morphism of rings 
$$\mu_A^\xi : A \longto B  \indexnot{mz}{\mu_A^\xi}  $$
such that $\mu_A^\xi(a)=a\xi^i$ for $a \in A_i$. Note that, if $A$ is $\NM$-graded 
(that is, if $A_i=0$ for $i < 0$), then $\mu_A^\xi$ can be defined also 
when $\xi$ is not invertible. 
Given $\tb$ an indeterminate over $A$, then 
$$\mu_A^\tb : A \longto A[\tb,\tb^{-1}]$$
is a morphism of rings. Denote by $\eval_A^\xi : A[\tb,\tb^{-1}] \to B$ 
the evaluation morphism at $\xi$. We have
\equat\label{mua}
\mu_A^\xi = \eval_A^\xi \circ \mu_A^\tb.
\endequat
In particular, if $B=A$ and $\xi=1$, then  
\equat\label{mua ea}
\mu_A^1 = \Id_A\qquad\text{and}\qquad \eval_A^1 \circ \mu_A^\tb = \Id_A.
\endequat
On the other hand, the morphism $\mu_A^\xi : A \longto B$ can be extended to a 
$\BZ[\tb,\tb^{-1}]$-linear morphism $\mub_A^\xi : A[\tb,\tb^{-1}] \longto B[\tb,\tb^{-1}]$ and
\equat\label{mua mua}
\mub_A^\xi \circ \mu_A^\tb = \mu_A^{\xi \tb}.
\endequat
As particular cases, one can take $B=A[\ub,\ub^{-1}]$ and $\xi=\ub$, where 
$\ub$ is another indeterminate, or take $B=A[\tb,\tb^{-1}]$ 
and $\xi = \tb^{-1}$. We obtain the following equalities:
\equat\label{mua mub}
\mub_A^\ub \circ \mu_A^\tb = \mu_A^{\tb\ub}
\qquad\text{and}\qquad
\mub_A^{\tb^{-1}} \circ \mu_A^\tb (a) = a \in A[\tb,\tb^{-1}]
\endequat
for all $a \in A$. Finally, note that 
\equat\label{ev mutilde}
\eval_A^1 \circ \mub_A^{\tb^{-1}} = \eval_A^1.
\endequat

\bigskip

\section{Extension of gradings}\label{section:graduation integrale}

\medskip

\boitegrise{{\bf Notation.} 
{\it We fix in this section a finitely generated free abelian group $\Gamma$
and a commutative $\Gamma$-graded domain $P$. 
Let $Q$ be a domain containing $P$ and integral over $P$.}}{0.75\textwidth}

\medskip

The aim of this section is to study the gradings on $Q$ which extend the one on $P$. 
We first start with the uniqueness problem.

\bigskip

\begin{lem}\label{unicite graduation}
If $Q=\mathop{\bigoplus}_{\gamma \in \Gamma} \Qti_\gamma = 
\mathop{\bigoplus}_{\gamma \in \Gamma} \Qhat_\gamma$ 
are two gradings on $Q$ extending the one of $P$ (that is, 
$P_\gamma=\Qti_\gamma \cap P = \Qhat_\gamma \cap P$ for all $\gamma\in\Gamma$), then 
$\Qti_\gamma=\Qhat_\gamma$ for all $\gamma\in\Gamma$.
\end{lem}

\begin{proof}
As in \S~\ref{intro graduation} (from which we keep the notation), 
the gradings $Q=\mathop{\bigoplus}_{\gamma\in\Gamma} \Qti_\gamma$ and 
$\mathop{\bigoplus}_{\gamma\in\Gamma} \Qhat_\gamma$ correspond to ring morphisms 
$\mut_Q : Q \to Q[\Gamma]$ and $\muh_Q : Q \to Q[\Gamma]$ extending 
$\mu_P: P \to P[\Gamma]$. Consider the morphism of rings
$\beta:Q[\Gamma]\to Q[\Gamma],\ q\otimes\gamma\mapsto
\muh_Q(q)\gamma^{-1}$ and let
$\alpha=\beta\circ\mut_Q:Q\to Q[\Gamma]$.
Then $\a$ is a morphism of rings and $\a(p)=p$ for all $p \in P$ by
(\ref{eq:gradedmu}).
Therefore, if $q \in Q$, then $\a(q) \in Q[\Gamma]$ is integral over $P$, hence
over $Q$. 
Since $Q$ is a domain, this implies that $\a(q) \in Q$
(Proposition~\ref{cloture polynomiale}). 
On the other hand, (\ref{eq:gradedmu}) shows that the composition
$$Q\xrightarrow{\alpha}Q[\Gamma]\xrightarrow{q\otimes\gamma\mapsto q}Q$$
is the identity, hence $\alpha(q)=q$ for all $q\in Q$. It follows that
$\beta\circ\mut_Q=\beta\circ\muh_Q$, hence $\mut_Q=\muh_Q$ since
$\beta$ is an isomorphism (\ref{eq:gradedmu}).
\end{proof}

\bigskip

\begin{coro}\label{graduation et automorphisme}
If $Q=\mathop{\bigoplus}_{\gamma\in\Gamma} Q_\gamma$ is a grading on 
$Q$ extending the one on $P$ 
and if $G$ is a group acting on $Q$, stabilizing $P$ and its homogeneous components, 
then $G$ stabilizes the homogeneous components of $Q$. 
\end{coro}

\begin{proof}
Indeed, if $g \in G$, then $Q=\mathop{\bigoplus}_{\gamma\in\Gamma} g(Q_\gamma)$ 
is a grading on $Q$ extending the one on $P$. 
According to Lemma~\ref{unicite graduation}, we have  $g(Q_\gamma)=Q_\gamma$ for all
$\gamma$.
\end{proof}

\bigskip

\begin{contre}
The assumption that $Q$ is a {\it domain} is necessary in Lemma~\ref{unicite graduation}. 
Indeed, if $P=P_0$ and if $Q=P \oplus P \e$ with $\e^2=0$, 
then we can endow $Q$ with infinitely many gradings extending the one on $P$:
$\e$ can be made homogeneous of any degree.\finl
\end{contre}

\bigskip

\begin{prop}\label{prop:graduation-positive}
Assume $Q=\bigoplus_{\gamma \in \Gamma} Q_\gamma$ is a
grading on $Q$ extending the one on $P$. If $\Gamma$ is endowed with a
structure of totally ordered group such that
$P_\gamma = 0$ for all $\gamma < 0$, then $Q_\gamma=0$ for all $\gamma < 0$.
\end{prop}

\bigskip

\begin{proof}
See~\cite[Chapter~5,~\S{1},~Proposition~20 and Exercise 25]{bourbaki}.
\end{proof}

\bigskip

We will now be interested in the question of the existence of a grading on $Q$ extending 
the one on $P$. For this, let $K=\Frac(P)$, $L=\Frac(Q)$ and we assume that 
the field extension $L/K$ is finite. 

\bigskip

\begin{lem}\label{g}
The grading on $P$ extends to a grading of its integral closure in $K$. 
\end{lem}

\begin{proof}
See~\cite[Chapter~5,~\S{1},~Proposition~21 and Exercise 25]{bourbaki}. 
\end{proof}

\bigskip

We will now show how the question of the existence can be read on a 
normal closure of the field extension $L/K$. Let $M$ denote a normal closure 
of the field extension $L/K$ and let $R$ be the integral closure  of $P$ in $M$. 

\medskip

\begin{lem}\label{decomposition gradue}
Let $\gamma\in\Gamma$. 
Define a grading on $P[\xb]$ extending the grading on $P$
	by giving $\xb$ the degree $\gamma$. 
Let $F \in P[\xb]$ be a monic polynomial, which is homogeneous for this grading. 
If $F=F_1\cdots F_r$, with $F_i \in P[\xb]$ monic, then $F_i$ is homogeneous for all $i$.
\end{lem} 

\begin{proof}
We have $\mu_{P[\xb]}(\xb a)=\xb\mu_P(a)\gamma$ for $a\in P$.
Let $\gamma_F\in\Gamma$ denote the {\it total} degree of $F$. We have
$$\mu_{P[\xb]}(F)=F\gamma_F = \mu_{P[\xb]}(F_1)\cdots \mu_{P[\xb]}(F_r).$$
Since $P[\xb]$ is a domain with fraction field $K(\xb)$ the fact that 
$K(\xb)[\Gamma]$ is a unique factorization domain implies that there exists
$G_1$,\dots, $G_r \in K(\xb)$ and  $\gamma_1$,\dots, $\gamma_r \in \Gamma$ such that 
$$\mu_{P[\xb]}(F_i)=G_i \gamma_i$$
for all $i$. This forces $F_i$ to be homogeneous of degree $\gamma_i$, 
and $F_i=G_i$.
\end{proof}

\begin{coro}\label{dec hom}
Let $\gamma\in\Gamma$. 
Define a grading on $P[\xb]$ extending the grading on $P$ by giving $\xb$ the degree $\gamma\in\Gamma$. 
Let $F \in P[\xb]$ be a monic polynomial, which is homogeneous for this grading. 
We assume that $M$ is the splitting field of $F$ over $K$. 
Then $R$ admits a grading extending the one on $P$.
\end{coro}

\begin{proof}
By Lemma~\ref{g}, we may assume that $P$ is integrally closed. 
Let $\d$ denote the degree of $F$ {\it in the variable $\xb$}. We shall show the result 
by induction on $\d$, the case where $\d=1$ being trivial 
(because $P=R$ in this case).

So assume that $\d \ge 2$ and let $F_1$ be a monic irreducible polynomial 
of $K[\xb]$ dividing $F$. By Proposition~\ref{int clos}, we have
$F_1 \in P[\xb]$. Set $K'=K[\xb]/< F_1 >$ and let $x$ be the image of $\xb$ in $K'$. 
Then $K'$ is a field which contains the ring $P'=P[\xb]/<F_1>$. 
In fact, $K'$ is the fraction field of $P'$. Since $F_1$ is homogeneous 
by Lemma~\ref{decomposition gradue}, 
$P'$ is graded (with $x$ homogeneous of degree $\gamma$). By Lemma~\ref{g}, 
the integral closure $P''$ of $P'$ in $K'$ 
inherits a grading. On the other hand, $K' \subset M$ and $M$ 
is the splitting field of $F$ over $K'$. 
In $P''[\xb]$, we have 
$$F(\xb)=(\xb-x) F_0(\xb),$$
where $F_0(\xb) \in P''[\xb]$ is homogeneous with degree {\it in the variable $\xb$} 
is equal to $\d-1$. Since the splitting field of $F$ over $K$ is equal to 
the splitting field of $F_0$ over $K'$, the result follows from the induction
hypothesis.
\end{proof}

\bigskip

\begin{prop}\label{R gradue}
Assume that $P$ and $Q$ are integrally closed. 
If the grading on $P$ extends to a grading on $Q$, then this grading 
also extends to a grading on $R$. 
\end{prop}

\begin{proof}
Let $q_1$,\dots, $q_r$ be elements of $Q$, homogeneous of respective degrees 
$\gamma_1$,\dots, $\gamma_r$ and such that $L=K[q_1,\dots,q_r]$. We denote by $F_i \in K[\tb]$ 
the minimal polynomial of $q_i$ over $K$: in fact, $F_i \in P[\tb]$ according to 
Corollary~\ref{minimal clos}. Then $M$ is the splitting field of $F_1 \cdots F_r$. 
By an easy induction argument, we may assume that $r=1$: 
we then write $q=q_1$, $\gamma=\gamma_1$ and $F=F_1$. 

If we give the variable $\tb$ the degree $\gamma$, then we check easily that 
$F$ becomes homogeneous (for the total degree on $P[\tb]$). 
The existence of an extension of the grading follows from Corollary~\ref{dec hom}.
\end{proof}

\bigskip

\begin{lem}\label{lem:homogeneise-premier}
Let $\pG$ be a prime ideal of $P$ and let $\pGt$ be the maximal homogeneous 
ideal of $P$ contained in $\pG$ (that is $\pGt=\bigoplus_{\gamma\in\Gamma} \pG \cap
P_\gamma$). 
Then $\pGt$ is prime.
\end{lem}

\begin{proof}
Indeed, $(P/\pG)[\Gamma]$ is a domain and $\pGt$ is the kernel of the morphism 
obtained by composition $P \xrightarrow{\mu_P} P[\Gamma] \xrightarrow{\mathrm{can}}
(P/\pG)[\Gamma]$.
\end{proof}

\bigskip

\begin{lem}\label{premier homogene}
Let $\qG$ be a prime ideal of $Q$ and let $\pG=\qG \cap P$. Assume that the grading on 
$P$ extends to a grading on $Q$. Then 
$\pG$ is homogeneous if and only if $\qG$ is homogeneous.
\end{lem}

\begin{proof}
If $\qG$ is homogeneous, then $\pG$ is clearly homogeneous. Conversely, 
assume that $\pG$ is homogeneous. Let 
$\qG'=\mathop{\bigoplus}_{\gamma\in\Gamma} (\qG \cap Q_\gamma)$. Then $\qG'$ is a
homogeneous ideal of $Q$ 
contained in $\qG$ and $\qG' \cap P = \pG = \qG \cap P$. By Lemma~\ref{lem:homogeneise-premier}, 
$\qG'$ is a prime ideal, so $\qG'=\qG$ since $Q$ is integral over $P$.
\end{proof}

\bigskip

\begin{lem}
	\label{le:barfromleft}
	Assume $\Gamma=\BZ$. Let 
	$\pG$ be a homogeneous prime ideal of $P$ and
	$\qG$ be a homegeneous prime ideal of $Q$.
	Assume $P/(P\cap\qG)$ is $\NM$-graded and
	$P_{<0}+P_{>0}\subset\pG\subset\qG+\langle Q_{>0},Q_{<0}\rangle$. 

	Then $\qG+\langle Q_{>0},Q_{<0}\rangle$ is a prime ideal of $Q$ and it
	is the unique prime ideal of $Q$ lying over $\pG$ and containing $\qG$.
\end{lem}

\begin{proof}
	Let $I=\langle Q_{>0},Q_{<0}\rangle$ and $I_0=I\cap Q_0=Q_0\cap (Q_{<0}\cdot
	Q_{>0})$.
	Since $Q/\qG$ is an integral extension of $P/(P\cap\qG)$ and
	$P/(P\cap\qG)$ is $\NM$-graded, it follows that $Q/\qG$ is $\NM$-graded
	(Proposition~\ref{prop:graduation-positive}). It follows that
	$Q_{<0}\subset\qG$, hence $I_0\subset\qG$. Since $I$ and $\qG$ are
	homogeneous, we have $Q_0\cap (I+\qG)=I_0+(Q_0\cap\qG)=Q_0\cap\qG$, a prime
	ideal of $Q_0$. We have $Q/(\qG+I)\simeq Q_0/(Q_0\cap (I+\qG))$, hence
	$\qG+I$ is a prime ideal of $Q$.

	\smallskip
	Let $\qG'$ be a prime ideal of $Q$ lying over $\pG$ and containing
	$\langle Q_{>0},Q_{<0}\rangle$. Since $\pG$ is homegeneous, it follows
	that $\qG'$ is homegeneous (Corollary~\ref{premier homogene}). 
	The extension $Q/\qG'$ of $P/\pG$ is integral and, since $P/\pG$ has its 
$\BZ$-grading concentrated in degree $0$, it follows that the $\BZ$-grading of
	$Q/\qG'$ is concentrated in degree $0$. It follows that $I\subset\qG'$,
	hence $\qG+I\subset\qG'$. We have $\pG\subset\qG+I$ and $P\cap\qG'=\pG$, hence
	$P\cap(\qG+I)=\pG$. So, $\qG+I=\qG'$.
\end{proof}

\bigskip

\begin{lem}
\label{le:minpolquotient}
Assume $\Gamma=\BZ$.
Let $\pG$ be a prime ideal of $P$ and let $P'$ be the largest graded subring of $P_\pG$. Assume that the
composition $P'_i\subset P_\pG\xrightarrow{\mathrm{can}}P/\pG$ is bijective for all $i\in\BZ$.

Let $q$ be a homogeneous element of $Q$ and let $F\in P[X]$ be its minimal polynomial.
Then the image of $F$ in $(P/\pG)[X]$ is the minimal polynomial of $q\otimes 1\in Q\otimes_P P/\pG$.
\end{lem}

\begin{proof}
Let $F'=\sum_{i=0}^n a_iX^i\in (P/\pG)[X]$ be the minimal polynomial of $q\otimes 1$, with $a_n=1$.
Note that $\deg F'\le \deg F$.

Let $d$ be the homogeneous degree of $q$ and let $G=\sum_{i=0}^n f_{d(n-i)}^{-1}(a_i)\in P'[X]$, where
$f_j:P'_j\xrightarrow{\sim}P/\pG$ is the canonical bijection. We have
$G(q)\in P'_{dn}\cap \pG P_\pG=0$. It follows that $G=F$.
\end{proof}

\bigskip

We conclude this section with some results about the homogeneization of prime ideals 
of $P$ or $Q$. 

\bigskip

\begin{coro}\label{coro:homogeneise-premier}
Assume that the grading on $P$ extends to $Q$. Let $\qG$ be a prime ideal of $Q$ 
and let $\pG = \qG \cap P$. Let $\pGt$ (respectively 
$\qGt$) be the maximal homogeneous ideal of $P$ (respectively $Q$) contained in 
$\pG$ (respectively $\qG$). Then $\pGt=\qGt \cap P$.
\end{coro}

\begin{proof}
This follows from the proof of Lemma \ref{premier homogene}
and from the fact that the diagram 
$$\xymatrix{
P \ar[r]^-{\mu_P} \ar@{^{(}->}[d] & P[\Gamma] \ar[r]^{\mathrm{can}} & (P/\pG)[\Gamma]
\ar@{^{(}->}[d] \\
Q \ar[r]^-{\mu_Q} & Q[\Gamma] \ar[r]^{\mathrm{can}} & (Q/\qG)[\Gamma]
}$$
is commutative.
\end{proof}

\begin{coro}\label{coro:homogeneise-inertie}
Assume that the grading on $P$ extends to $Q$ and that there exists a finite group $G$ 
acting on $Q$, stabilizing $P$ and preserving the grading. Let $\qG$ be a prime ideal of $Q$ and let 
$\qGt$ be the maximal homogeneous ideal of $Q$ contained in $\qG$. 
Let $D_\qG$ (respectively $D_\qGt$) be the decomposition group of $\qG$ (respectively $\qGt$) 
in $G$ and $I_\qG$ (respectively $I_\qGt$) be the inertia group of $\qG$ (respectively $\qGt$) 
in $G$. Then 
$$D_\qG \subset D_\qGt\qquad\text{and}\qquad I_\qG = I_\qGt.$$
\end{coro}

\begin{proof}
The first inclusion is immediate, since $G$ preserves the grading 
(see Corollary~\ref{graduation et automorphisme}). 
Moreover, $Q/\qG$ is a quotient of $Q/\qGt$, so $I_\qGt \subset I_\qG$. 
Conversely, if $g \in I_\qG \subset D_\qG \subset D_\qGt$ and if $q \in \qG \cap
Q_\gamma$, 
then $g(q)-q \in \qG \cap Q_\gamma \subset \qGt$. So $g \in I_\qGt$.
\end{proof}

\bigskip

\section{Gradings and reflection groups}\label{section:GR}

\medskip

\boitegrise{\noindent{\bf Notation.} 
{\it In this section, we fix a field $k$ of characteristic zero and a commutative 
$\NM$-graded $k$-algebra $R=\bigoplus_{i \in \NM} R_i$. We assume moreover that 
$R$ is a domain, that $R_0=k$ and $R$ is a finitely generated $k$-algebra.
We also fix a finite 
group $G$ acting faithfully on $R$ by automorphisms of graded $k$-algebras 
and we set $P=R^G$. Let $R_+=\bigoplus_{i > 0} R_i$: it is the unique 
graded maximal ideal of $R$. We fix a $G$-stable graded vector subspace 
$E^*$ of $R$ such that $R_+=R_+^2 \oplus E^*$ 
(such a subspace exists because $kG$ is semisimple) and we denote by $E$ 
the $k$-dual of $E^*$.}}{0.75\textwidth}

\bigskip

The group $G$ acts on the vector space $E$ and the aim of this section is to give 
some criterion allowing to determine whether $G$ is a 
reflection subgroup of $\GL_k(E)$. 
Our results are inspired by~\cite{BBR}. 

First of all, the grading on $E^*$ induces a grading on $E$ and a grading on $k[E]$, 
the algebra of polynomial functions on $E$ (that is, the symmetric algebra of $E^*$). 
Similarly, $k[E]$ inherits an action of $G$, which preserves the grading. We denote by 
$k[E]_+$ the unique graded maximal ideal of $k[E]$. The inclusion $E^* \longinjto R$ 
induces a $G$-equivariant morphism of graded $k$-algebras 
$$\pi : k[E] \longto R.$$
It is easily checked that
\equat\label{nombre generateurs R}
\text{\it the minimal number of generators of the $k$-algebra $R$ is $\dim_k E$}
\endequat
(see for instance~\cite[Lemme 2.1]{BBR}). We put
$$I=\Ker \pi,$$
so that
\equat\label{E/I}
R \simeq k[E]/I.
\endequat
In particular, $G$ acts faithfully on $E$. 
Since $I$ is homogeneous, it follows from the graded Nakayama Lemma that 
\equat\label{nombre generateurs}
\text{\it the minimal number of generators of the ideal $I$ is $\dim_k I/k[E]_+ I$}
\endequat
Moreover, it is also easy to check that 
\equat\label{G trivial}
\text{\it $I=k[E] I^G$ if and only if $G$ acts trivially on $I/k[E]_+ I$.}
\endequat
(see for instance~\cite[Lemme 3.1]{BBR}). 
Finally, since $kG$ is semisimple, we have 
\equat\label{PRG}
P \simeq k[E]^G/I^G.
\endequat
We will also need the next lemma.

\bigskip

\begin{lem}\label{liberte G}
If $R$ is a free $P$-module, then the rank of the $P$-module $R$ is $|G|$.
\end{lem}

\begin{proof}
Let $d$ be the $P$-rank of $R$. Since $R$ is a domain, 
$P$ is also a domain and, if we set $K=\Frac(P)$ and $M=\Frac(R)$, 
then $K=L^G$ (and so $[L:K]=|G|$) and $L = K \otimes_P R$ (and so $[L:K]=d$). 
Therefore, $d=|G|$. 
\end{proof}

\bigskip

The main result of this section is the following (compare
 with~\cite[Th\'eor\`eme~3.2]{BBR}, from which we borrow the proof).

\bigskip

\begin{prop}\label{intersection complete R}
Assume that $P$ is regular and that $R$ is a free $P$-module. Then 
the following are equivalent:
\begin{itemize}
\itemth{1} $R$ is complete intersection and $G$ acts trivially on $I/k[E]_+I$. 

\itemth{2} $G$ is a subgroup of $\GL_k(E)$ generated by reflections.
\end{itemize}
\end{prop}

\begin{rema}\label{liberte platitude}
If $P$ is regular and since we are working with graded objects, the 
following statements are equivalent:
\begin{itemize}
\item[$\bullet$] $R$ is a free $P$-module.

\item[$\bullet$] $R$ is a flat $P$-module.

\item[$\bullet$] $R$ is Cohen-Macaulay.
\end{itemize}
Moreover, if $R$ is complete intersection, then $R$ is Cohen-Macaulay.
\end{rema}

\begin{proof}
Let $e=\dim_k E$, let $i=\dim_k I/k[E]_+ I$ and let $d$ denote the Krull dimension 
of $R$ (which is also the one of $P$). Moreover, $e$ is the Krull dimension of 
$k[E]$ and of $k[E]^G$. 

\medskip

Let us first show (1) $\Rightarrow$ (2). So assume that 
$R$ is complete intersection and that $G$ acts trivially on $I/k[E]_+I$. 
Since $R$ is complete intersection, (\ref{E/I}) and (\ref{nombre generateurs}) show
that
$$d=e-i.$$
Moreover, since $G$ acts trivially on $I/k[E]_+I$, the ideal $I$ of $k[E]$ 
can be generated by $i$ homogeneous $G$-invariant elements 
$f_1$,\dots, $f_i$ and so the ideal $I^G$ of $k[E]^G$ is generated by 
$f_1$,\dots, $f_i$. Since $P$ is regular of Krull dimension $d$, the algebra
$P=k[E]^G/I^G$ can be generated by $d$ elements $\pi(g_1)$,\dots, $\pi(g_d)$ 
where $g_j \in k[E]^G$ is homogeneous. Therefore, the $k$-algebra 
$k[E]^G$ is generated by $f_1$,\dots, $f_i$, $g_1$,\dots, $g_d$, 
that is, it is generated by $i+d=e$ elements. 
Since the Krull dimension of $k[E]^G$ is also equal to $e$, 
this shows that $k[E]^G$ is a polynomial algebra, and so  
$G$ is a reflection subgroup of $\GL_k(E)$ by Theorem~\ref{chevalley}.

\medskip

Conversely, let us now show (2) $\Rightarrow$ (1). 
So assume that $G$ is a reflection subgroup of $\GL_k(E)$. Then $k[E]$ is a free 
$k[E]^G$-module of rank $|G|$ (by Theorem~\ref{chevalley}) hence
$(k[E]^G/I^G) \otimes_{k[E]^G} k[E]$ is a free $P$-module 
of rank $|G|$ (see~(\ref{PRG})). Moreover, $k[E]/I=R$ is a free $P$-module of rank $|G|$ 
(see Lemma~\ref{liberte G}). So the canonical surjection 
$(k[E]^G/I^G) \otimes_{k[E]^G} k[E] \longsurto k[E]/I$ (between two free $P$-modules 
of the same rank) is an isomorphism, hence $I$ is generated by 
$I^G$. We deduce that $G$ acts trivially on $I/k[E]_+ I$ (by~(\ref{G trivial})). 

On the other hand, since $k[E]^G$ and $k[E]^G/I^G=P$ are both polynomial algebras 
(see Theorem~\ref{chevalley} for $k[E]^G$), $P$ is complete intersection and so $I^G$ can be generated by 
$e-d$ elements. We deduce from~(\ref{nombre generateurs}) and~(\ref{G trivial}) 
that $i \le e-d$ and so necessarily $i=e-d$ and $R$ is complete intersection.
\end{proof}

\bigskip

\chapter{Blocks, decomposition maps}
\label{appendice: blocs}\setcounter{section}{0}

\bigskip

\boitegrise{{\bf Assumption and notation.} 
{\it We fix in this appendix a commutative ring $R$, which will be assumed 
{\bfit noetherian}, {\bfit integral} and {\bfit integrally closed}. 
Let $\rG$ denote a prime ideal of $R$. We also fix an $R$-algebra $\HC$ which 
will be assumed to be finitely generated and free as an $R$-module.
We denote by $\Zrm(\HC)$ the center of $\HC$. We set  $k=\Frac(R/\rG)=k_R(\rG)$ and 
$K=\Frac(R)=k_R(0)$. Finally, 
the image of an element $h \in \HC$ in $k\HC$ will be denoted by $\hba$.}}{0.75\textwidth}

\bigskip

%

\section{Blocks of $k\HC$}\label{section:definition blocs}

\medskip

If $A$ is a ring (not necessarily commutative), we denote by 
$\blocs(A)$  \indexnot{I}{\blocs(A)}  the set of its primitive idempotents. For instance, 
$\blocs(\Zrm(\HC))$ is the set of primitive central idempotents of $\HC$. 
Since $\Zrm(\HC)$ is noetherian, 
\equat\label{1}
1=\sum_{e \in \blocs(\Zrm(\HC))} e.
\endequat
Moreover, the morphism $\HC \to k\HC$ induces a morphism $\pi_\Zrm : k\Zrm(\HC) \to \Zrm(k\HC)$ 
(which might be neither injective nor surjective). However, the following result 
has been proven by M\"uller~\cite[Theorem~3.7]{muller}:

\bigskip

\begin{prop}[M\"uller]\label{muller}
\begin{itemize}
\itemth{a} If $e \in \blocs(k\Zrm(\HC))$, then $\pi_\Zrm(e) \in \blocs(\Zrm(k\HC))$. 

\itemth{b} The map $\blocs(k\Zrm(\HC)) \to \blocs(\Zrm(k\HC))$, 
$e \mapsto \pi_\Zrm(e)$ is bijective.
\end{itemize}
\end{prop}
%

\bigskip

In a finite dimensional commutative $k$-algebra $\AC$ (for instance, $\Zrm(k\HC)$ 
or $k\Zrm(\HC)$), the prime ideals are maximal and are in one-to-one correspondence with the set 
of primitive idempotents of $\AC$: if $\mG \in \Spec \AC$ and $e \in \blocs(\AC)$, then 
$e$ and $\mG$ are associated through this bijective map if and only if $e \not\in \mG$ 
(that is, if and only if $\mG = \Rad(\AC e) + (1-e)\AC$). 
The Proposition~\ref{muller} shows that $\Spec k\Zrm(\HC)$ is in one-to-one correspondence 
with $\blocs(\Zrm(k\HC))$, that is, with the set of primitive central 
idempotents of $k\HC$. 

Moreover, the natural (and injective) morphism $R \longinjto \Zrm(\HC)$ induces a morphism 
$\Upsilon : \Spec \Zrm(\HC) \to \Spec R$. The map $\Zrm(\HC) \to k\Zrm(\HC)$ induces 
a bijective map between the sets $\Spec k\Zrm(\HC)$ and $\Upsilon^{-1}(\rG)$. Recall that 
$$\Upsilon^{-1}(\rG) = \{\zG \in \Spec \Zrm(\HC)~|~\zG \cap R = \rG\}.$$
Finally, we obtain a bijective map
\equat\label{xi}
\Xi_\rG : \blocs(\Zrm(k\HC)) \stackrel{\sim}{\longto} \Upsilon^{-1}(\rG)
\endequat
which is characterized by the following property: 

\bigskip

\begin{lem}\label{lem:blocs}
If $e \in \blocs(\Zrm(k\HC))$ and if $\zG \in \Upsilon^{-1}(\rG)$, then the following 
are equivalent:
\begin{itemize}
\itemth{1} $\zG=\Xi_\rG(e)$.

\itemth{2} $e \not\in \pi_\Zrm(k \zG)$.

\itemth{3} $\zG$ is the preimage, in $\Zrm(\HC)$, of 
$\pi_\Zrm^{-1}(\Rad(\Zrm(k\HC) e) + (1-e) \Zrm(k\HC))$.
\end{itemize}
\end{lem}

\bigskip

Now, by localization at $\rG$, $\Upsilon^{-1}(\rG)$ is  
in one-to-one correspondence with $\Upsilon_\rG^{-1}(\rG R_\rG)$, where 
$\Upsilon_\rG : \Spec R_\rG \Zrm(\HC) \to \Spec R_\rG$ is 
the map induced by the inclusion $R_\rG \longinjto R_\rG \Zrm(\HC)$. 
The bijective maps, in both directions, between 
$\Upsilon^{-1}(\rG)$ and $\Upsilon_\rG^{-1}(\rG R_\rG)$ 
are given by
$$\fonctio{\Upsilon^{-1}(\rG)}{\Upsilon_\rG^{-1}(\rG R_\rG)}{\zG}{R_\rG \zG}$$
$$\fonctio{\Upsilon_\rG^{-1}(\rG R_\rG)}{\Upsilon^{-1}(\rG)}{\zG}{\zG \cap \Zrm(\HC).}\leqno{\text{and}}$$
The center of the algebra $R_\rG \HC$ is equal to $R_\rG \Zrm(\HC)$ and the 
canonical morphism $\HC \to k\HC$ extends to a morphism $R_\rG \HC \to k\HC$, 
which will still be denoted by $h \mapsto \hba$. 
Finally, we denote by $R_\rG \Zrm(\HC) \to k\Zrm(\HC)$, $z \mapsto \zhat$,  
the canonical morphism (so that $\zba = \pi_\Zrm(\zhat)$ if $z \in R_\rG \Zrm(\HC)$).

To summarize, we obtain a diagram of natural bijective maps 
\equat\label{diagramme bijections}
\diagram
\Upsilon^{-1}(\rG) \ar@{<->}[rr]^{\DS{\sim}} 
&& 
\Upsilon_\rG^{-1}(\rG R_\rG) \ar@{<->}[rr]^{\DS{\sim}} 
&& 
\Spec k\Zrm(\HC) \ar@{<->}[rr]^{\DS{\sim}}
 \ar@{<->}[dd]^{\DS{\reflectbox{\rotatebox[origin=c]{-90}{$\backsim$}}}} 
&& 
\Spec \Zrm(k\HC) \ar@{<->}[dd]^{\DS{\reflectbox{\rotatebox[origin=c]{-90}{$\backsim$}}}} \\
&&&&&&\\
&&
&& \blocs(k\Zrm(\HC)) \ar@{<->}[rr]^{\DS{\sim}}
&&
\blocs(\Zrm(k\HC)). \\
\enddiagram
\endequat
%

\section{Blocks of $R_\rG\HC$}

\medskip

%

\boitegrise{{\bf Assumption.} 
{\it From now on, and until the end of this Appendix, we will
assume that the $K$-algebra $K\HC$ is {\bfit split}.}}{0.75\textwidth}

\bigskip

The question of lifting idempotents whenever the ring $R_\rG$ 
is complete for the $\rG$-adic topology is classical. We propose here another version, 
valid only whenever the $K$-algebra $K\HC$ is split (we only need that 
$R$ is integrally closed: no assumption on the Krull dimension of $R$ 
or on its completeness is necessary).
\bigskip

\subsection{Central characters}\label{section centrale}
If $V$ is a simple $K\HC$-module, and if $z \in K\Zrm(\HC)$, then $z$ acts on 
$V$ by multiplication by a scalar $\o_V(z) \in K$ (indeed, since $K\HC$ 
is split, we have $\End_{K\HC}(V)=K$). 
This defines a morphism of $K$-algebras 
$$\o_V : K\Zrm(\HC) \longto K  \indexnot{oz}{\o_V, \o_V^\rG}  $$
whose restriction to $\Zrm(\HC)$ has valued in $R$ (since $\Zrm(\HC)$ is integral over 
$R$ and $R$ is integrally closed). 
Hence, this defines a morphism of $R$-algebras 
$$\o_V : \Zrm(\HC) \longto R.$$
By composition with the canonical projection $R \to R/\rG$, we obtain  
a morphism of $R$-algebras 
$$\o_V^\rG : \Zrm(\HC) \longto R/\rG.$$
Since $\o_V(1)=1$ and $R/\rG$ is integral, $\Ker \o_V$ is a prime ideal of
$\Zrm(\HC)$ such that $\Ker \o_V \cap R = \rG$. So
\equat\label{ker}
\Ker \o_V^\rG \in \Upsilon^{-1}(\rG).
\endequat
This defines a map
$$\fonction{\KER_\rG}{\Irr(K\HC)}{\Upsilon^{-1}(\rG)}{V}{\Ker \o_V^\rG}.$$

\bigskip

\begin{defi}\label{defi:r-bloc}
The fibers of the map $\KER_\rG$ are called the {\bfit $\rG$-blocks} of $\HC$.
\end{defi}

\bigskip

The $\rG$-blocks of $\HC$ are subsets of $\Irr(K\HC)$, of which they form 
a partition. Note that, since $\Zrm(\HC) = R + \Ker(\o_V^\rG)$, 
the central character $\o_V^\rG$ is determined by its kernel. Hence, two simple  
$K\HC$-modules $V$ and $V'$ belong to the same $\rG$-block if and only if 
$\o_V^\rG=\o_{V'}^\rG$.

\bigskip

\subsection{Lifting idempotents} 
The main result of this section is the following:

\bigskip

\begin{prop}\label{relevement idempotent}
We have:
\begin{itemize}
\itemth{a} If $e \in \blocs(R_\rG \Zrm(\HC))$, then $\ehat \in \blocs(k\Zrm(\HC))$.

\itemth{b} The map $\blocs(R_\rG \Zrm(\HC)) \to \blocs(k\Zrm(\HC))$, $e \mapsto \ehat$ 
is bijective.
\end{itemize}
\end{prop}

\begin{proof}
Let 
$$\fonction{\O}{R_\rG\Zrm(\HC)}{\prod_{V \in \Irr(K\HC)} R_\rG}{z}{(\o_V(z))_{V \in \Irr(K\HC)}.}$$
Then $\O$ is a morphism of $R_\rG$-algebras, whose kernel $I$ is equal to 
$R_\rG\Zrm(\HC) \cap \Rad(K\HC)$ and whose image will be denoted by $A$. 

Consequently, $I$ is nilpotent and so $\O$ induces a bijective map 
$\blocs(R_\rG \Zrm(\HC)) \longbij \blocs(A)$. 
Moreover, by Corollary~\ref{coro:relevement-idempotent} (which will be proven 
in \S~\ref{section:relevements-idempotents}), 
the reduction modulo $\rG$ induces a bijective map $\blocs(A) \longbij \blocs(kA)$. 
It then remains to show that the kernel of the natural 
map $k\Zrm(\HC) \longsurto kA$ is nilpotent: it is obvious as it is the image of 
$I$ in $k\Zrm(\HC)$. 
\end{proof}

\bigskip

\begin{coro}\label{coro:r-blocs}
The map $\KER_\rG : \Irr(K\HC) \to \Upsilon^{-1}(\rG)$ is surjective. 
Its fibers are of the form $\Irr(K\HC e)$, where $e \in \blocs(R_\rG \Zrm(\HC))$. 
\end{coro}

\begin{proof}
The first statement follows from~(\ref{L}) below and the second from the proof 
of Proposition~\ref{relevement idempotent}.
\end{proof}

\bigskip

By combining Propositions~\ref{muller} and~\ref{relevement idempotent}, 
we get the next corollary:

\bigskip

\begin{coro}\label{muller plus}
We have: 
\begin{itemize}
\itemth{a} If $e \in \blocs(R_\rG \Zrm(\HC))$, then $\eba \in \blocs(\Zrm(k\HC))$.

\itemth{b} The map $\blocs(R_\rG \Zrm(\HC)) \to \blocs(\Zrm(k\HC))$, $e \mapsto \eba$ 
is bijective.
\end{itemize}
\end{coro}

\bigskip

Therefore, we get a bijective map 
\equat\label{bijection finale}
\Upsilon^{-1}(\rG) \stackrel{\sim}{\longleftrightarrow} \blocs(R_\rG \Zrm(\HC)).
\endequat
If $\zG \in \Upsilon_\rG^{-1}(\rG R_\rG)$ and if $e \in \blocs(R_\rG \Zrm(\HC))$, then
\equat\label{bij e r}
\text{\it $e$ and $\zG$ are associated through this bijective map 
if and only if $e \not\in R_\rG \zG$.}
\endequat

To summarize, we obtain a diagram of natural bijective maps
\equat\label{diagramme bijections deploye}
\diagram
\Upsilon^{-1}(\rG) \ar@{<->}[r]^{\DS{\sim}} 
\ar@{<-->}[ddr]^{\DS{\reflectbox{\rotatebox[origin=c]{40}{$\backsim$}}}}&
\Upsilon_\rG^{-1}(\rG R_\rG) \ar@{<->}[rr]^{\DS{\sim}} 
\ar@{<-->}[dd]^{\DS{\reflectbox{\rotatebox[origin=c]{-90}{$\backsim$}}}}&& 
\Spec k\Zrm(\HC) \ar@{<->}[rr]^{\DS{\sim}}
 \ar@{<->}[dd]^{\DS{\reflectbox{\rotatebox[origin=c]{-90}{$\backsim$}}}} && 
\Spec \Zrm(k\HC) \ar@{<->}[dd]^{\DS{\reflectbox{\rotatebox[origin=c]{-90}{$\backsim$}}}} \\
&&&&&\\
&\blocs(R_\rG \Zrm(\HC)) \ar@{<-->}[rr]^{\DS{\sim}}&& \blocs(k\Zrm(\HC)) \ar@{<->}[rr]^{\DS{\sim}}&&
\blocs(\Zrm(k\HC)), \\
\enddiagram
\endequat
where the maps with dashed arrows exist only whenever the $K$-algebra 
$K\HC$ is split.

Let
$$\fonctio{\Upsilon^{-1}(\rG)}{\blocs(R_\rG\Zrm(\HC))}{\zG}{e_\zG}$$
denote the bijective map of diagram~\ref{diagramme bijections deploye}. 
We get a partition of $\Irr(K\HC)$ thanks to the action of the central idempotents $e_\zG$:
\equat\label{partition bloc}
\Irr(K\HC) = \coprod_{\zG \in \Upsilon^{-1}(\rG)} \Irr(K\HC e_\zG).
\endequat
The subsets $\Irr(K\HC e_\zG)$ are the $\rG$-blocks of $\HC$.

\bigskip

\begin{exemple}
Whenever $\rG$ is the zero ideal, then $R_\rG=k=K$, 
$\Upsilon^{-1}(\rG) \simeq \Spec K\Zrm(\HC)$, 
$\blocs(R_\rG \HC)=\blocs(K\HC)$ and $\o_V^\rG = \o_V$.\finl
\end{exemple}

\bigskip
%
%

\subsection{Ramification locus} 
The following proposition is certainly classical (but is valid because $R$ 
is integrally closed):

\bigskip

\begin{prop}\label{codimension un}
	Assume that the algebra $K\HC$ is {\bfit split}. Then there exists a
(unique) radical ideal $\aG$ 
of $R$ satisfying the following two properties:
\begin{itemize}
\itemth{1} $\Spec(R/\aG)$ is empty or purely of codimension $1$ in $\Spec(R)$;

\itemth{2} If $\rG$ is a prime ideal of $R$, then $\blocs(R_\rG \Zrm(\HC))=\blocs(K\Zrm(\HC))$ 
if and only if $\aG \not\subset \rG$. 
\end{itemize}

Let $\rG$ be a prime ideal of $R$ containing $\aG$. A subset of
 $\Irr(K\HC)$ is an $\rG$-block if
and only if it
is minimal for the property of being a $\pG$-block for all height
one prime ideals $\pG$ of $R$ with $\aG\subset\pG\subset\rG$.
\end{prop}

\begin{proof}
Let $(b_1,\dots,b_n)$ be an $R$-basis of $\HC$ and let $\blocs(K\Zrm(\HC))=\{e_1,\dots,e_l\}$ 
with $l=|\blocs(K\Zrm(\HC))|$. We write 
$$e_i=\sum_{j=1}^n k_{ij} \, b_j$$
with $k_{ij} \in K$. 

Let us now fix a prime ideal $\rG$ of $R$. Then $\blocs(R_\rG \Zrm(\HC))=\blocs(K\Zrm(\HC))$ 
if and only if
$$\forall~1 \le i \le l,~\forall~1 \le j \le n,~k_{ij} \in R_\rG.\leqno{(*)}$$
If $k \in K$, we set $\aG_k=\{r \in R~|~rk \in R\}$. 
Then $\aG_k$ is an ideal of $R$ and, if $\rG$ is a prime ideal of 
$R$, then $k \in R_\rG$ if and only if $\aG_k \not\subset \rG$. Define
$\aG$ to be the radical of
$$\prod_{\substack{1 \le i \le l \\ 1 \le j \le n}} \aG_{k_{ij}}.$$
Now $(*)$ becomes equivalent to $\aG \not\subset \rG$. 
This proves the statement (2). 

\medskip

Let us now show that $\Spec(R/\aG)$ is empty or purely of codimension $1$ in $\Spec(R)$. 
For this, it is sufficient to prove that $\Spec(R/\aG_k)$ is empty or purely of 
codimension $1$ in $\Spec(R)$. 
If $k \in R$, then $\aG_k=R$ and $\Spec(R/\aG_k)$ is empty. Assume that $k \not\in R$, 
let us show that $\Spec(R/\aG_k)$ is then purely of codimension $1$ in $\Spec(R)$. 
Let $\pG$ be a minimal prime ideal of $R$ containing $\aG_k$. Then $k \not\in R_\pG$. 
We need to prove that $\pG$ has height $1$. But, since $R$ is integrally closed, 
the same holds for $R_\pG$: so $R_\pG$ is the intersection of the localized rings $R_{\pG'}$, 
where $\pG'$ runs over the set of prime ideals of height $1$ of $R$ contained in $\pG$ 
(see~\cite[Theorem~11.5]{matsumura}). So there exists a prime ideal $\pG'$ of $R$ of 
height $1$ contained in $\pG$ and such that $k \not\in R_{\pG'}$. Hence 
$\aG_k \subset \pG' \subset \pG$ and 
the minimality of $\pG$ implies that $\pG=\pG'$, which implies that $\pG$ has height $1$.

\medskip
Assume now $\aG\subset \rG$. Let $I$ be a subset of $\Irr(K\HC)$ that is
a union of $\pG$-blocks for all height one prime ideals $\pG$ of $R$
with $\aG\subset\pG\subset\rG$. There is an idempotent $e$ of
$K\Zrm(\HC)$ such that given $V\in\Irr(K\HC)$, we have
$eV\neq 0$ if and only if $V\in I$.
The coefficients of $e$ in the $K$-basis
$(b_1,\ldots,b_n)$ of $K\HC$ are in $R_{\pG}$ for all
height one prime ideals $\pG$ of $R$ with $\aG\subset\pG\subset\rG$.
	
The discussion above shows that
$e\in R_{\pG}\Zrm(\HC)$ for all height one prime ideals $\pG$ of $R$
that do not contain $\aG$. So, the coefficients of $e$ in the $K$-basis
$(b_1,\ldots,b_n)$ of $K\HC$ are in $\bigcap_{\pG}R_{\pG}$, where 
$\pG$ runs over all height one prime ideals of $R$ contained in $\rG$.
	Since that intersection is $R_\rG$,
	we deduce that $e\in R_{\rG}\Zrm(\HC)$. This shows that $I$ is a 
union of $\rG$-blocks.
\end{proof}

\bigskip

\section{Decomposition maps}\label{section:decomposition}

\medskip

\def\charpol{{\mathrm{Char}}}
\def\reduction{{\mathrm{red}}}

Let $R_1$ be a commutative $R$-algebra and let $\rG_1$ be a prime ideal of $R_1$. 
We set $R_2=R_1/\rG_1$, $K_1=\Frac(R_1)$ and $K_2=\Frac(R_2)=k_{R_1}(\rG_1)$. 
Let $\FC(\HC,K_1[\tb])$ denote the set of maps $\HC \to K_1[\tb]$. 
If $V$ is a $K_1\HC$-module of finite type and if $h \in \HC$, 
we denote by $\charpol_{K_1}^V(h)$ the characteristic polynomial of $h$ for its action 
on the finite dimensional $K_1$-vector space $V$. 
Therefore, $\charpol_{K_1}^V \in \FC(\HC,K_1[\tb])$. Also, $\charpol_{K_1}^V$ 
depends only on the class of $V$ 
in the Grothendieck group $\groth(K_1\HC)$. This defines a map 
$$\charpol_{K_1} : \groth^+(K_1\HC) \longto \FC(\HC,K_1[\tb]),$$
where $\groth^+(K_1\HC)$ denotes the submonoid of $\groth(K_1\HC)$ consisting 
of the isomorphism classes of $K_1\HC$-modules of finite type. 
It is well-known that $\charpol_{K_1}$ is 
injective~\cite[proposition~2.5]{geck rouquier}. 

We will say that the pair $(R_1,\rG_1)$ satisfies the property 
$\propdec$ if the following three statements are fulfilled:
\begin{quotation}
\begin{itemize}
\item[(D1)] $R_1$ if noetherian and integral.

\item[(D2)] If $h \in R_1\Hb$ and if $V$ is a simple $K_1\Hb$-module, then 
$\charpol_{K_1}^V(h) \in R_1[\tb]$ (note that this property 
is automatically satisfied if $R_1$ is integrally closed). 

\item[(D3)] The algebras $K_1\Hb$ and $K_2\Hb$ are split.
\end{itemize}
\end{quotation}
Let $\reduction_{\rG_1} : \FC(\HC,R_1[\tb]) \longto \FC(\HC,R_2[\tb])$ denote the reduction 
modulo $\rG_1$. By assumption~(D3), if $K_2'$ is an extension of $K_2$, the scalar extension 
induces an isomorphism $\groth(K_2\HC) \longiso \groth(K_2'\HC)$, and we will 
identified these two Grothendieck groups.

\bigskip

\begin{prop}[Geck-Rouquier]\label{prop:geck-rouquier}
If $(R_1,\rG_1)$ satisfies $\propdec$, then there exists a unique map 
$\dec_{R_2\HC}^{R_1\HC} : \groth(K_1\HC) \longto \groth(K_2\HC)$  
\indexnot{d}{\dec_{R_2\HC}^{R_1\HC}}  which makes the following diagram
$$\diagram
\groth(K_1\HC) \rrto^{\DS{\charpol_{K_1}}} \ddto_{\DS{\dec_{R_2\HC}^{R_1\HC}}} 
&& \FC(\HC,R_1[\tb]) \ddto^{\DS{\reduction_{\rG_1}}} \\
&&\\
\groth(K_2\HC) \rrto^{\DS{\charpol_{K_2}}} && \FC(\HC,K_2[\tb])
\enddiagram$$
commutative. If $\OC_1$ is a subring of $K_1$ containing $R_1$, if $\mG_1$ is 
a prime ideal of $\OC_1$ such that $\mG_1 \cap R_1 = \rG_1$, and if 
$\LC$ is an $\OC_1\HC$-module which is free of finite rank over $\OC_1$, 
then $k_{\OC_1}(\mG_1)$ is an extension of $K_2$ and 
$$\dec_{R_1\HC}^{R_2\HC} \isomorphisme{K_1\LC}_{K_1\HC} = 
\isomorphisme{k_{\OC_1}(\mG_1) \LC}_{k_{\OC_1}(\mG_1)\HC}.$$
\end{prop}

\begin{proof}
This proposition is proven in~\cite[Proposition~2.11]{geck rouquier} whenever $R_1$ 
is integrally closed. We will show how to deduce our proposition from this case. 
If we assume only that~(D2) holds, let $R_1'$ denote the integral closure 
of $R_1$ in $K_1$. Since $R_1'$ is integral over $R_1$, there exists a prime ideal 
$\rG_1'$ of $R_1'$ such that $\rG_1' \cap R_1 = \rG_1$. Set $R_2'=R_1'/\rG_1'$. 
Then $k_{R_1'}(\rG_1')$ is an extension of $k_{R_1}(\rG_1)$ 
so $k_{R_1'}(\rG_1)\HC$ is split, which means that, by~\cite[Proposition~2.11]{geck rouquier}, 
$\dec_{R_2'\HC}^{R_1'\HC} : \groth(K_1\HC) \longto \groth(k_{R_1'}(\rG_1')\HC)$ is 
well-defined and satisfies the desired properties. We then define 
$\dec_{R_2\HC}^{R_1\HC}$ by using the isomorphism 
$\groth(k_{R_1'}(\rG_1')\HC) \simeq \groth(K_2\HC)$ 
and it is easy to check that this map satisfies the expected properties.
\end{proof}

It then follows a transitivity property~\cite[proposition~2.12]{geck rouquier}:

\bigskip

\begin{coro}[Geck-Rouquier]\label{geck rouquier}
Let $R_1$ be an $R$-algebra, let $\rG_1$ be a prime ideal of $R_1$ and let $\rG_2$ 
be a prime ideal of $R_2=R_1/\rG_1$. We assume that $(R_1,\rG_1)$ and $(R_2,\rG_2)$ 
both satisfy $\propdec$ and we set $R_3=R_2/\rG_2$. 
Then 
$$\dec_{R_3\HC}^{R_1\HC} = \dec_{R_3\HC}^{R_2\HC} \circ \dec_{R_2\HC}^{R_1\HC}.$$
\end{coro}

\bigskip
%
%
%
%
%
%
%
%
%
%
%
%
%

%
%
%

\section{Idempotents and central characters}\label{section:relevements-idempotents}

\bigskip

The aim of this section is to complete the proof of 
Corollary~\ref{coro:r-blocs}.
Let $\OC$ be a local noetherian ring and let 
$A$ be a $\OC$-subalgebra of $\OC^d=\OC \times \OC \times \cdots \times \OC$ ($d$ times). 
Let $\mG=\Rad(\OC)$, $k=\OC/\mG$ and, if $r \in \OC$, let $\rba$ denote its image in $k$.

If $1 \le i \le d$, let $\pi_i : \OC^d \to \OC$ denote the $i$-th projection and 
$$\o_i : A \longto \OC$$
denotes the restriction of $\pi_i$ to $A$. We set 
$$\fonction{\bar{\o}_i}{A}{k}{a}{\overline{\o_i(a)}.}$$
On the set $\{1,2,\dots,d\}$, we denote by $\smile$ the equivalence 
relation defined by
$$\text{$i \smile j$ \quad if and only if \quad $\bar{\o}_i=\bar{\o}_j$.}$$
Finally, we set 
$$e_i=(0,\dots,0,\underbrace{1}_{\text{$i$-th position}},0,\dots,0) \in \OC^d.$$
Then:

\bigskip

\begin{lem}\label{lem:relevement-idempotent}
Let $I \in \{1,2,\dots,d\}/\smile$. Then $\sum_{i \in I} e_i \in A$.
\end{lem}

\bigskip

\begin{proof}
By reordering if necessary the idempotents, we may assume that 
$I=\{1,2,\dots,d'\}$ with $d' \le d$. 
We proceed in several steps:

$$\text{\it If $i \in I$ and $j \not\in I$, then there exists $a_{ij} \in A$ 
such that $\o_i(a_{ij})=1$ and $\o_j(a_{ij})=0$.}\leqno{(\clubsuit)}$$

\medskip

\begin{quotation}
\noindent{\it Proof of $(\clubsuit)$.} 
Since $i \not\smile j$, there exists $a \in A$ such that $\bar{\o}_i(a) \neq \bar{\o}_j(a)$. 
Let $r =\o_j(a)$ and $u = \o_i(a)-\o_j(a)$. Then $u \in \OC^\times$ because $\OC$ is local 
and $a_{ij}=u^{-1}(a - r \cdot 1_A) \in A$ satisfies the two conditions.\finl
\end{quotation}

$$\text{\it There exists $a_1 \in A$ such that $\o_1(a_1)=1$ and $\o_j(a_1)=0$ if $j \not\in I$.}
\leqno{(\diamondsuit)}$$

\medskip

\begin{quotation}
\noindent{\it Proof of $(\diamondsuit)$.} 
By $(\clubsuit)$, there exists, for all $i \in I$ and $j \not\in I$, 
$a_{ij} \in A$ such that $\o_i(a_{ij})=1$ and $\o_j(a_{ij})=0$. Note that, 
if $i' \in I$, then $\o_{i'}(a_{ij}) \equiv 1 \mod \mG$ because $\bar{\o}_i=\bar{\o}_{i'}$. 
Set 
$a=\prod_{i \in I, j \not\in I} a_{ij}$. Then it is clear that 
$\o_j(a)=0$ if $j \not\in I$ and $\o_i(a) \equiv 1 \mod \mG$ if $i \in I$. 
It is then sufficient to take $a_1 = \o_1(a)^{-1} a$.\finl
\end{quotation}

\bigskip

We then define by induction the sequence $(a_i)_{1 \le i \le d'}$ as follows:
$$a_{i+1}=a_i^2 (1 + \o_{i+1}(a_i)^{-2}(1-a_i^2)).$$
We will show by induction on $i \in \{1,2,\dots,d'\}$ the following two facts:
$$\text{\it The element $a_i$ is well-defined and belongs to $A$.}\leqno{(\heartsuit_i)}$$
$$\text{\it If $1 \le i' \le i$ and $j \not\in I$, then $\o_{i'}(a_i)=1$ and $\o_j(a_i)=0$.}
\leqno{(\spadesuit_i)}$$

\bigskip

\begin{quotation}
\noindent{\it Proof of $(\heartsuit_i)$ and $(\spadesuit_i)$.} 
This is obvious if $i=1$. Let us now assume that 
$(\heartsuit_i)$ and $(\spadesuit_i)$ hold (for some $i \le d'-1$). Let us prove that 
this implies that $(\heartsuit_{i+1})$ and $(\spadesuit_{i+1})$ also hold. 

Then $i \smile i+1$ and so $\o_{i+1}(a_i) \equiv \o_i(a_i)=1 \mod \mG$. 
So $\o_{i+1}(a_i)$ is invertible and so $a_{i+1}$ is well-defined 
and belongs to $A$ (this is exactly $(\heartsuit_{i+1})$). 

Now, let us set $r=\o_{i+1}(a_i)$ for simplifying. Then:

$\bullet$ If $1 \le i' \le i$, we have $\o_{i'}(a_{i+1})=1 \cdot (1 + r^{-2}(1-1^2))=1$.

$\bullet$ $\o_{i+1}(a_{i+1}) = r^2 (1 + r^{-2}(1-r^2))=1$.

$\bullet$ If $j \not\in I$, then $\o_j(a_{i+1})= 0 \cdot (1 + r^{-2}(1-0^2))=0$.

\noindent So $(\spadesuit_{i+1})$ holds.\finl
\end{quotation}

\bigskip

Therefore, $a_{d'} = \sum_{i \in I} e_i \in A$.
\end{proof}

\bigskip

\begin{coro}\label{coro:blocs-A}
The map 
$$\fonctio{\{1,2,\dots,d\}/\smile}{\blocs(A)}{I}{\DS{\sum_{i \in I} e_i}}$$
is well-defined and bijective.
\end{coro}

\begin{proof}
The Lemma~\ref{lem:relevement-idempotent} shows that, if $I \in \{1,2,\dots,d\}/\smile$, 
then $e_I=\sum_{i \in I} e_i \in A$. If $e_I$ is not primitive, 
this means that, since $\OC$ is local, there exists two non-empty subsets 
$I_1$ and $I_2$ of $I$ such that $e_{I_1}$, $e_{I_2} \in A$, 
and $I=I_1 \coprod I_2$. But, if $i_1 \in I_1$ and $i_2 \in I_2$, then 
$\omeba_{i_1}(e_{I_1})=1 \neq 0 = \omeba_{i_2}(e_{I_1})$, which is impossible 
because $i_1 \smile i_2$. So the map $I \mapsto e_I$ is well-defined. 
It is now clear that it is bijective.
\end{proof}

\bigskip

If $a \in A$, let $\ahat$ denote its image in $kA=k \otimes_\OC A$.

\bigskip

\begin{coro}\label{coro:relevement-idempotent}
With this notation, we have:
\begin{itemize}
\itemth{a} If $e \in \blocs(A)$, then $\ehat \in \blocs(kA)$.

\itemth{b} The map $\blocs(A) \to \blocs(kA)$, $e \mapsto \ehat$ 
is bijective.
\end{itemize}
\end{coro}

\begin{proof}
(a) Let $e \in \blocs(A)$ and assume that $\ehat=e_1+e_2$, 
where $e_1$ and $e_2$ are two orthogonal idempotents of $kA$. 
The ring $\OC$ being noetherian, $kA$ is a finite dimensional 
commutative $k$-algebra. So there exists two morphisms of $k$-algebras 
$\r_1$, $\r_2 : kA \to k'$ (where $k'$ 
is a finite extension of $k$) such that $\r_i(e_j)=\d_{i,j}$. Let   
$\rhot_i$ denote the composition $A \to kA \stackrel{\r_i}{\longrightarrow} k'$. 

Set $\aG_i=\Ker(\rhot_i)$. The image of $\r_i$ being 
a subfield $k'$, $\aG_i$ is a maximal ideal of $A$. Since $\OC^d$ is integral over 
$A$, there exists a maximal ideal $\mG_i$ of $\OC^d$ such that $\aG_i = \mG_i \cap A$. 
Since $\OC$ is local, $\mG_i$ is of the form 
$\OC \times \cdots \times \OC \times \mG \times \OC \times \cdots \times \OC$, 
where $\mG$ is in $t_i$-th 
position (for some $t_i \in \{1,2,\dots,d\}$), which implies that 
$\rhot_i = \omeba_{t_i}$. 

Since $\r_1 \neq \r_2$ and $\r_i(e_j)=\d_{i,j}$, we get $\omeba_{t_1} \neq \omeba_{t_2}$ 
and $\omeba_{t_1}(e)=\r_1(e_1+e_2)=1=\r_2(e_1+e_2)=\omeba_{t_2}(e)$. This contradicts 
Corollary~\ref{coro:blocs-A}.
\end{proof}

\bigskip

During this proof, the following result has been proven: if $k'$ is a finite extension 
of $k$ and if $\r : kA \to k'$ is a morphism of $k$-algebras, then 
\equat\label{L}
\text{\it there exists $i \in \{1,2,\dots,d\}$ such that $\r(\ahat)=\omeba_i(a)$ 
for all $a \in A$.}
\endequat

\chapter{Invariant rings}\label{appendice:invariant}
Let $\kb$ be a field.
Let $R$ be a $\kb$-algebra acted on by a finite group $G$ whose order
is invertible in $\kb$.
Let $e=\frac{1}{|G|}\sum_{g\in G}g$, a
central idempotent of $\kb G$. Let
$A=R\rtimes G$. The aim of this appendix is to relate 
representations of $A$ and of $R^G$. We are mainly interested 
in the case where $R$ is commutative.

\bigskip

\section{Morita equivalence}

\medskip
The natural action of $G$ on $R$ and the action of $R$ by left multiplication
on $R$ induce a structure of $A$-module on $R$.

\smallskip
The following lemma is classical.
\begin{lem}
	The restriction to $R$ of a semisimple $A$-module is semisimple.
\end{lem}

\begin{proof}
	Let $S$ be a simple $A$-module. Since $\Res_R^A(S)$ is finitely
	generated, it admits a maximal submodule $M$. Since
	$\bigcap_{g\in G}g(M)$ is a proper $A$-submodule of $S$, it is zero.
	Consequently, the canonical morphism of $R$-modules
	$S\to\bigoplus_{g\in G}S/g(M)$ is injective. Since 
	$S/g(M)$ is a simple $R$-module for all $g\in G$, it follows that
	$S$ is a semisimple $R$-module.
\end{proof}

The following lemma is clear.

\bigskip

\begin{lem}
\label{le:idempotentinvariants}
There is an isomorphism of $A$-modules
$R\longiso Ae,\ r\mapsto re$
that restricts to an isomorphism of $\kb$-algebras $R^G\longiso 
eAe$.
\end{lem}

Let $M$ be an $R$-module whose isomorphism class is stable under the action
of a subgroup $H$ of $G$. There are isomorphisms of $A$-modules
$\phi_h:h^*(M)\longiso M$ for $h\in H$, unique up to left multiplication by
$\Aut_R(M)$. Consequently, the elements 
$\phi_h\in N_{\Aut_{R^G}(M)}(\Aut_R(M))$
define a morphism of groups $H\to \Aut_{R^G}(M)/\Aut_R(M)$.

\bigskip

\begin{prop}
\label{pr:equivMorita}
The following assertions are equivalent:
\begin{itemize}
\itemth{1} $Ae$ is a progenerator for $A$
\itemth{2} $Ae$ induces a Morita equivalence between $A$ and $R^G$
\itemth{3} $A=AeA$
\itemth{4} for every simple $A$-module $S$, we have $S^G{\not=0}$.
\itemth{5} for every simple $R$-module $T$ whose isomorphism class is stable
under the action of a subgroup $H$ of $G$ and for every non-zero direct
summand $U$ of $\Ind_R^{R\rtimes H}T$, we have $U^H\not=0$.
\end{itemize}
\end{prop}

\begin{proof}
Note that $Ae$ is a direct summand of $A$, as 
a left $A$-module, hence $Ae$ is a finitely generated projective $A$-module.
The equivalence between (1) and (2) follows from Lemma 
\ref{le:idempotentinvariants}.

If $Ae$ is a progenerator, then 
$A$ is isomorphic to a quotient of a multiple of $Ae$. Since the image of
a morphism of $A$-modules
	$Ae\to A$ is contained in $AeA$, we deduce that if (1) holds, then
(3) holds. Conversely, assume (3). There are $a_1,\ldots,a_n\in A$
such that $1\in Aea_1+\cdots+Aea_n$, hence the morphism
$(Ae)^n\to A,\ (r_1,\ldots,r_n)\mapsto r_1a_1+\cdots+r_na_n$ is surjective and
(1) follows.

We have $A/AeA=0$ if and only if
$A/AeA$ has no simple module, hence if and only if
$e$ does not act by $0$ on any simple $A$-module.  This shows the equivalence
of (3) and (4).

	Assume (4). Let $S$ be a simple $A$-module that is a direct summand
	of $\Ind_{R\rtimes H}^A(U)$. We have $S^G\neq 0$, hence $U^H\neq 0$.
	So (5) holds.

	Assume (5).
Let $S$ be a simple $A$-module. Let $T$ be a simple $R$-module that is a direct
	summand of the semisimple module
	$\Res^A_R(S)$. Let $H$ be the stabilizer of the isomorphism
class of $T$. There is a simple $(R\rtimes H)$-module $U$ that is a direct summand
	of $\Ind_R^{R\rtimes H}(T)$ such that
	$S=\Ind_{R\rtimes H}^A(U)$. Since $U^H\not=0$, we deduce that $S^G\neq 0$.
	This shows (4).
%
%
%
%
%
\end{proof}

\bigskip

\begin{coro}\label{coro:center-ag}
If $AeA=A$, then $\Zrm(A)=\Zrm(R^G)$. 
\end{coro}

\begin{proof}
By Proposition~\ref{pr:equivMorita}, the rings
$A$ and $eAe \simeq R^G$ are Morita 
equivalent via the bimodule $Ae$. This provides an isomorphism
	$\Zrm(A) \xrightarrow{\sim} \Zrm(R^G)$, the 
isomorphism being given by the action on the bimodule $Ae$ (Lemma
\ref{lem:ZA-ZB} below).
The result follows.
\end{proof}

\begin{lem}\label{lem:ZA-ZB}
Let $A$ and $B$ be two rings and $M$ an $(A,B)$-bimodule such that the canonical
maps give isomorphisms
$B\longiso \End_A(M)$ and $A\longiso\End_{B^\opp}(M)^\opp$. Then we have an isomorphism 
$\Zrm(A) \longiso \Zrm(B)$.

In particular, if $e$ is an idempotent of a ring $A$ and if left multiplication
gives an isomorphism $A\longiso\End_{(eAe)^\opp}(Ae)^\opp$, then there is
an isomorphism $\Zrm(A)\longiso \Zrm(eAe),\ a\mapsto ae$.
\end{lem}

\begin{proof}
The left multiplication on $M$ induces a ring morphism 
$\a : \Zrm(A) \to \Zrm(B)$ such that $zm=m\a(z)$ for all $z \in \Zrm(A)$ and $m \in M$. 
Similarly, the right multiplication induces a ring morphism $\b : \Zrm(B) \to \Zrm(A)$ 
such that $mz=\b(z)m$ for all $z \in \Zrm(B)$ and $m \in M$. Hence, if $z \in \Zrm(A)$ and 
$m \in M$, then $zm=\b(\a(z))m$, and so $\b \circ \a = \Id_{\Zrm(A)}$ 
since the action of $A$ on $M$ is faithful by assumption. Similarly $\a \circ \b = \Id_{\Zrm(B)}$.
\end{proof}

\bigskip

\section{Geometric setting}

\medskip

We assume now that $R=\kb[X]$, where $X$ is an affine scheme of finite type over $\kb$,
i.e., $R$ is a finitely generated commutative $\kb$-algebra.
Proposition \ref{pr:equivMorita} has the following consequence.

\bigskip

\begin{coro}
\label{le:Moritafree}
If $G$ acts freely on $X$, then $Ae$ induces a Morita equivalence
between $A$ and $R^G$.
\end{coro}

\bigskip

Let $X^{\reg}=\{x\in X| \Stab_G(x)=1\}$ and let $A^{\reg}=
\kb[X^{\reg}]\rtimes G$. We assume $X^\reg$ is dense in $X$, i.e.,
the pointwise stabilizer of an irreducible component of $X$ is trivial.
The following proposition gives a sufficient condition for a double centralizer theorem.

\bigskip

\begin{prop}
\label{pr:codim2doublecentralizer}
Assume that $X$ is a normal variety, i.e.,
all localizations of $A$ at prime ideals are integral and integrally closed.

\begin{itemize}
\itemth{1}
The canonical morphism of algebras
$A\to\End_{R^G}(R)$ is injective.

\itemth{2}
If the codimension of $X\setminus X^{\reg}$ is $\ge 2$ in each connected
component of $X$, then the morphism above is an isomorphism and
$\Zrm(A)=R^G$.
\end{itemize}
\end{prop}

\begin{proof}
It follows from Corollary~\ref{le:Moritafree} that given $f\in R^G$ such that
$D(f)\subset X^{\reg}$, then the canonical morphism
$R[f^{-1}]\rtimes G\to \End_{R^G[f^{-1}]}(R[f^{-1}])$
is an isomorphism. In particular, the morphism of the proposition
$A\to\End_{(R^G)}(R)$ is an injective
morphism of $R$-modules, since $X^\reg$ is dense in $X$.

Let $K$ be the cokernel of the canonical morphism
$A\to\End_{(R^G)}(R)$. 
We have $K\otimes_R k[X^\reg]=0$, hence the support of
$K$ has codimension $\ge 2$. Since $A$ is normal, it has depth $\ge 2$, hence
$\Ext^1_R(K,R)=0$. We deduce that
$K$ is a direct summand of the torsion free $R$-module
$\End_{(R^G)}(R)$, hence $K=0$.  
The last statement follows from Lemma \ref{lem:ZA-ZB}.
\end{proof}

\medskip
The statement on the center of $A$ can be obtained more directly.

\begin{lem}
If $R$ is an integral domain and $G$ acts faithfully on $X$, then
$\Zrm(A)=R^G$.
\end{lem}

\begin{proof}
Let $a=\sum_{g\in G}r_g g\in \Zrm(A)$ with $r_g\in R$ for $g\in G$.
Given $g_0$ a non-trivial element of $G$, there is $x\in R$
such that $g_0 x g_0^{-1}\neq x$. We have
$$0=[x,a]=\sum_{g\in G}(x-gxg^{-1})r_g g$$
hence $(x-g_0xg_0^{-1})r_{g_0}=0$. Since $R$ is integral, it follows that
$r_{g_0}=0$. We have shown that $a=r_1\in A$. Since $[a,g]=0$ for all
$g\in G$, we deduce that $a\in R^G$.
\end{proof}

\bigskip

We conclude with a description of the simple $A$-modules when $A=AeA$. 
In this case, using the Morita equivalence between $A$ and $R^G$ induced by the bimodule $Ae$, 
we obtain a bijective map
\equat\label{eq:bij-irr}
\bijectio{\Irr(A)}{\Irr(R^G)}{S}{eS.}
\endequat
Since $R$ is commutative, $\Irr(R)$ (respectively $\Irr(R^G)$) is in one-to-one correspondence 
with the set of maximal ideals of $R$ (respectively of $R^G$), so we obtain a bijective map
\equat\label{eq:bij-max-max}
\Irr(R)/G \longbij \Irr(R^G)
\endequat
(see Propositions~\ref{prop:max-max} and~\ref{galois transitif}). By composing the 
two previous bijective maps, we obtain a third bijective map
\equat\label{eq:bij-max-max-max}
\Irr(R)/G \longbij \Irr(R)
\endequat
We will describe more concretely this last map. In order to do that,
let $\O$ be a $G$-orbit 
of (isomorphism classes of) simple $R$-modules. The $R$-module 
$S_\O=R/\cap_{T \in \O} {\mathrm{Ann}}_R(T)$ 
inherits an action of $G$, hence it becomes an $A$-module.

\bigskip

\begin{prop}\label{prop:ag-commutatif}
Assume that $A=AeA$ and that $R$ is commutative and finitely generated.
\begin{itemize}
\itemth{a} If $\O \in \Irr(R)/G$, then $S_\O$ is a simple $A$-module. 

\itemth{b} The map $\Irr(R)/G \longto \Irr(A)$, $\O \mapsto S_\O$ is bijective 
(and coincides with the bijective map~\ref{eq:bij-max-max-max}).

\itemth{c} If $S$ is a simple $A$-module, then $\Res_R^A(S)$ is semisimple and 
multiplicity-free, and two simple $R$-modules occurring in $\Res_R^A(S)$ are 
in the same $G$-orbit.

\itemth{d} If $S$ and $S'$ are two simple $A$-modules, then $S \simeq S'$ if and 
only if $\Res_R^A(S)$ and $\Res_R^A(S')$ have a common irreducible submodule. 
\end{itemize}
\end{prop}

\begin{proof}
(a) By construction, we have a well-defined injective morphism of $R$-modules 
	$S_\O \injto \bigoplus_{T \in \O} T$ (here, we identify $T$ and $R/{\mathrm{Ann}}_R(T)$). 
Let $S$ be a non-zero $A$-submodule of $S_\O$.
Its restriction to $R$ contains a submodule isomorphic to some $T \in \O$. 
Since the action of $G$ stabilizes $S$, it follows that $S=S_\O$ and that
$$\Res_R^A(S_\O)=\bigoplus_{T \in \O} T.\leqno{(*)}$$
This proves (a).

\medskip

(b) It follows from $(*)$ that the map $\Irr(R)/G \longto \Irr(R)$, $\O \mapsto S_\O$ is injective. 
Now, let $\O \in \Irr(R)/G$, let $T \in \O$ and let $\mG={\mathrm{Ann}}_R(T)$. 
Denote by $H$ the stabilizer of $\mG$ in $G$ (that is, the decomposition group 
of $\mG$). We have $eS_\O=S_\O^G \simeq T^H=(R/\mG)^H$. By Theorem~\ref{bourbaki}, 
$(R/\mG)^H = R^G/(\mG \cap R^G)$. This proves that $eS_\O$ is the simple $R^G$-module 
associated with the maximal ideal $\mG \cap R^G$ of $R^G$ or, in other words, 
is the simple $R^G$-module associated with $\O$ through the bijective map~\ref{eq:bij-max-max}. 
This completes the proof of (b).

\medskip

(c) and (d) now follow from (a), (b) and $(*)$.
\end{proof}

\chapter{Highest weight categories}\label{ap:hw}

We fix in Appendix \ref{ap:hw} a commutative noetherian ring $k$.

\section{General theory}
\subsection{Definitions and first properties}

We say that a poset $\Delta$ is {\em locally finite} if given
any $D,D'\in\Delta$, then there are only finitely many $D''\in\Delta$
such that $D<D''<D'$. We say that a subset $\Gamma$ of $\Delta$ is an
{\em ideal} if given $D\in\Delta$ and $D'\in\Gamma$ with $D<D'$,
then $D\in\Gamma$. Given $D\in\Delta$, we define
$\Delta_{\le D}=\{D'\in\Delta\ |\ D'\le D\}$, and we define similarly
$\Delta_{<D}$, $\Delta_{\ge D}$ and $\Delta_{>D}$. We say that an ideal
$\Gamma$ is {\em finitely generated} if there are
$D_1,\ldots,D_n\in\Delta$ such that $\Gamma=\Delta_{\le D_1}\cup\cdots\cup
\Delta_{\le D_n}$.

\medskip

Let $\CC$ be a $k$-linear abelian category. We say that $\CC$ is
{\em noetherian} if all its objects are noetherian.

\smallskip
Let
$\Delta$ be a family of isomorphism classes of objects of $\CC$ (the
{\em standard objects}). We assume $\Delta$ is endowed
with a locally finite poset structure.

\medskip
Given $\Gamma\subset\Delta$, we denote by 
\begin{itemize}
\item $\CC[\Gamma]$ the full
subcategory of $\CC$ of objects $M$ such that $\Hom(D,M)=0$ for all
$D\in\Delta\setminus\Gamma$
\item $\CC^\Gamma$ the full subcategory of $\CC$ of objects 
$M$ that have
a filtration $0=M_0\subset M_1\subset\cdots\subset M_r=M$ such that
$M_i/M_{i-1}\in\Gamma$ for $1\le i\le r$
\item $i(\CC^\Gamma)$ the full subcategory of $\CC$ of objects that
are direct summands of objects of $\CC^\Gamma$.
\end{itemize}

We extend now the definition of (split) highest weight categories over $k$
of \cite{rouquier schur} to the case of a non-necessarily finite $\Delta$.

\begin{defi}
\label{de:hwcat}
We say that $\CC$, endowed with the poset of standard objects $\Delta$, is a
highest weight category if
\begin{itemize}
\item[(i)] for all $D\in\Delta$, we have $\End(D)=k$
\item[(ii)] given $D_1,D_2\in\Delta$ such that $\Hom(D_1,D_2)\not=0$, we have $D_1\le D_2$
\item[(iii)] every object of $\CC$ is the quotient of an object of 
$\CC^\Delta$
\item[(iv)] for all $M\in\CC$, $D,D'\in\Delta$, there is a
surjection $R\twoheadrightarrow D$ with kernel in
$\CC^{\Delta_{>D}}$ such that $\Hom(R,D')$ is a finitely generated
projective $k$-module and $\Ext^1(R,M)=0$.
\end{itemize}
\end{defi}

\begin{rema}
Assumption (iv) is aimed at making sense of 
the existence of an approximation of a projective module,
and of the requirement that objects of $\Delta$ are finitely generated
and projective over $k$.\finl
\end{rema}

We assume from now on that $\CC$ is a highest weight category.

\medskip
Note first that Definition \ref{de:hwcat} (iv) admits a version where
$D$ is replaced by an arbitrary object of $\CC^\Delta$. The next lemma shows
that this stronger version is actually a consequence of 
Definition \ref{de:hwcat} (iv).

\begin{lem}
\label{le:almostcover}
Let $N\in\CC^\Delta$, $D'\in\Delta$ and $M\in\CC$.
Then, there exists a surjection
$R\twoheadrightarrow N$ with kernel in $\CC^\Delta$ such that 
$\Hom(R,D')$ is a finitely generated projective $k$-module and
$\Ext^1(R,M)=0$.
\end{lem}

\begin{proof}
Fix a filtration
$0=N_0\subset N_1\subset\cdots\subset N_r=N$ such that $N_i/N_{i-1}\in\Delta$
for $1\le i\le r$. Given $i$, there exists a surjection $f_i:R_i
\twoheadrightarrow N_i/N_{i-1}$ such that $\ker f_i\in\CC^\Delta$,
$\Hom(R_i,D')$ is a finitely generated projective $k$-module and
$\Ext^1(R_i,N_{i-1}\oplus M)=0$.
So, $f_i$ lifts to a map $g_i:R_i\to N_i$ and the sum
$g=\sum_i g_i: \bigoplus_i R_i\to N$ is surjective. Let
	$L_i=\ker(g_1+\cdots+g_i)$ for $1\le i\le r$. We have
a filtration $0=L_0\subset L_1\subset\cdots\subset L_r$
such that $L_i/L_{i-1}\simeq \ker f_i$ for $1\le i\le r$.
It follows that $L_r\in\CC^\Delta$. Note finally that $\Ext^1(
\bigoplus_i R_i,M)=0$ and $\Hom(\bigoplus_i R_i,D')$ is a finitely
generated projective $k$-module.
\end{proof}

\begin{lem}
\label{le:kernelDeltafiltered}
Let $0\to M\to L\to N\to 0$ be an exact sequence in $\CC$ with
$L,N\in i(\CC^\Delta)$. Then $M\in i(\CC^\Delta)$.
\end{lem}

\begin{proof}
It is enough to prove the lemma for $L,N\in\CC^\Delta$.
By Lemma \ref{le:almostcover}, there is a surjection $R\twoheadrightarrow
N$ with kernel $N'\in \CC^\Delta$ and with $\Ext^1(R,M)=0$.
Let $L'$ be the kernel of the canonical
map $L\oplus R\twoheadrightarrow N$. The composition of canonical maps
$L'\hookrightarrow L'\oplus R\twoheadrightarrow L$ is surjective and its
	kernel is isomorphic to $N'$, hence $L'\in\CC^\Delta$. The
kernel of the canonical map $L'\twoheadrightarrow R$ is isomorphic to $M$.
Since $\Ext^1(R,M)=0$, it follows that $M$ is a direct summand of $L'$,
hence $M\in i(\CC^\Delta)$.
\end{proof}

\begin{lem}
\label{le:vanishingExti}
Let $\Gamma$ be an ideal of $\CC$ and
let $M\in\CC[\Gamma]$. Let $D\in\Delta$ and let $i\ge 0$.
If $\Ext^i(D,M){\not=}0$, then there exists
$D_0,\ldots,D_i\in\Gamma$ with $D=D_0<D_1<\cdots<D_i$.
\end{lem}

\begin{proof}
When $i=0$, we have
$D\in\Gamma$. We proceed now by induction on $i\ge 1$.

There exists $M'\in\CC$, $\zeta\in\Ext^1(D,M')$ and
$\xi\in\Ext^{i-1}(M',M)$ such that $\xi\circ\zeta{\not=}0$.
There exists a surjection $R\twoheadrightarrow D$ with kernel
$L$ in $\CC^{\Delta_{>D}}$ such that $\Ext^1(R,M')=0$.
So, $\zeta$ factors through a map $f:L\to M'$.
Since $\xi\circ\zeta{\not=}0$, it follows that $\xi\circ f{\not=}0$.
So, there exists $D'>D$ such that $\Ext^{i-1}(D',M){\not=}0$.
By induction, we deduce that there exists $D'_0,\ldots,D'_{i-1}
\in\Gamma$ such that $D'=D'_0<D'_1<\cdots<D'_{i-1}$. It follows that
the lemma holds for $D$, $M$ and $i$.
\end{proof}

\begin{lem}
	Given $M,N\in\CC$, the $k$-module $\Hom(M,N)$ is finitely generated.
\end{lem}

\begin{proof}
	Recall that $k$ is noetherian.
By Definition \ref{de:hwcat} (iv), the $k$-module $\Hom(M,N)$
is finitely generated if $M,N\in\Delta$. Consequently, the lemma holds
if $M,N\in\CC^\Delta$.

Assume now $M\in\CC^\Delta$. There exists a surjection 
	$N'\twoheadrightarrow N$ with $N'\in\CC^\Delta$ and we denote by
	$N''$ the kernel of that surjection.
	Lemma \ref{le:almostcover} shows there exists a surjection
	$R\twoheadrightarrow M$ with $R\in\CC^\Delta$ and $\Ext^1(R,N'')=0$.
	The canonical map $\Hom(R,N')\to\Hom(R,N)$ is surjective and
	$\Hom(R,N')$ is finitely generated, hence $\Hom(R,N)$ is finitely
	generated. It follows that $\Hom(M,N)$ is finitely generated.

	A general $M$ is a quotient of an object of $\CC^\Delta$, hence
	$\Hom(M,N)$ is finitely generated.
\end{proof}

\begin{lem}
\label{le:finiteExt1}
Let $M,M'\in\CC^\Delta$.
The $k$-module $\Ext^1(M,M')$ is finitely
generated. If it is non-zero and $M,M'\in\Delta$, then $M<M'$.
\end{lem}

\begin{proof}
	Assume first $M,M'\in\Delta$.
Fix a surjection $R\twoheadrightarrow M$ with kernel $L\in\CC^{\Delta_{>M}}$
such that $\Ext^1(R,M')=0$.
Since $\Hom(L,M')$ is a finitely generated $k$-module, we deduce that
$\Ext^1(M,M')$ is a finitely generated $k$-module. If $\Ext^1(M,M'){\not=0}$,
then $\Hom(L,M'){\not=}0$, hence there is $M''>M$ such that $M''\le M'$.
So, $M<M'$.

The general case follows by induction on the length of a $\Delta$-filtration.
\end{proof}

\begin{lem}
\label{le:orderfiltration}
Let $I$ be a finite subset of $\Delta$. Fix a total order
	$\prec$ on $I$ such that $\Ext^1(D,D')\neq 0$ implies $D\prec D'$.

Consider $M\in\CC$ with a filtration
$0=M_0\subset M_1\subset\cdots\subset M_r=M$ such that
$M_i/M_{i-1}\in I$ for $1\le i\le r$. Then,
$M$ has another filtration
$0=M'_0\subset M'_1\subset\cdots\subset M'_r=M$ such that
	$M'_i/M'_{i-1}\in I$ for $1\le i\le r$ and $M'_{i+1}/M'_i\preceq
	M'_i/M'_{i-1}$
for $1\le i<r$.
\end{lem}

\begin{proof}
We prove the result by induction on $r$. Take $D\in I$ maximal for $\prec$
such that $D\simeq M_i/M_{i-1}$ for some $i$, and consider $i$ minimal
with this property. We have $D{\not\prec} M_j/M_{j-1}$ for $j<i$, hence
$\Ext^1(D,M_j/M_{j-1})=0$ for $j<i$.
It follows that
$\Ext^1(D,M_{i-1})=0$, hence $M_i$ has a subobject $L$
isomorphic to $D$ such that $M_i=L\oplus M_{i-1}$.
The object $M/L$ has a filtration with $(M/L)_j=M_j$ for $j<i$ and
$(M/L)_j=M_{j+1}/L$ for $j\ge i$. That filtration has length $r-1$,
and the subquotients are in $I$. By induction, $N=M/L$ has another filtration
$0=N'_0\subset\cdots\subset N'_{r-1}=N$ with
	$N'_i/N'_{i-1}\in I$ for $1\le i<r$ and $N'_{i+1}/N'_i\preceq
	N'_i/N'_{i-1}$
for $1\le i<r-1$. We obtain an appropriate filtration of $M$ by
taking $M'_1=L$ and $M'_i$ the inverse image of $N'_{i-1}$ for $i>1$.
\end{proof}

We define the partial order $\lessdot$ on $\Delta$ as the one generated
by $D\lessdot D'$ if $\Ext^i(D,D')\neq 0$ for some $i\in\{0,1\}$.

The following proposition shows that $\lessdot$ is the coarsest order
on $\Delta$ that makes $(\CC,\Delta)$ into a highest weight category.

\begin{prop}
\label{pr:coarsestorder}
Let $\lhd$ be a partial order on $\Delta$. 
The category $\CC$
equipped with the poset $(\Delta,\lhd)$ is a highest weight category if and
	only if $\lhd$ is finer than $\lessdot$.
\end{prop}

\begin{proof}
	If $(\Delta,\lhd)$ defines a highest weight category structure on
	$\CC$, then it follows from Lemma \ref{le:finiteExt1} that
	$\lhd$ is finer than $\lessdot$.

	Assume now $\lhd$ is finer than $\lessdot$.
Consider $M\in\CC$ and $D,D'\in\Delta$. There is a surjection
$R\twoheadrightarrow D$ with kernel $N\in\CC^{\Delta_{>D}}$ and
such that $\Hom(R,D')$ is a finitely generated projective $k$-module
and $\Ext^1(R,M)=0$.
By Lemma \ref{le:orderfiltration}, there is a filtration
$0=R_0\subset R_1\subset\cdots\subset R_{r-2}\subset R_{r-1}=N\subset
R_r=R$ with subquotients in $\Delta$
and there is an integer $i\in\{1,\ldots,r-1\}$
such that given $j\in\{1,\ldots,r-1\}$, we have
$D\lhd (R_j/R_{j-1})$ if and only if $j>i$.

Consider $j>i$ and $l\in\{1,\ldots,r-1\}$ such that
$\Ext^1(R_j/R_{j-1},R_l/R_{l-1})\neq 0$. We have
$(R_j/R_{j-1})\lhd (R_l/R_{l-1})\neq 0$, hence
$D\lhd (R_l/R_{l-1})$. It follows that $l>i$. We deduce that
$\Ext^1(R/R_i,R_i)=0$, hence $R\simeq R/R_i\oplus R_i$. The $k$-module
$\Hom(R/R_i,D')$ is finitely generated and projective and 
$\Ext^1(R/R_i,M)=0$. So (iv) holds for $(\CC,\Delta,\lhd)$. The
conditions (i)--(iii) are clear. This completes the proof of the
proposition.
\end{proof}

\subsection{Ideals and Serre subcategories}

\begin{prop}
\label{pr:idealshw}
Let $\Gamma$ be an ideal of $\Delta$.

\begin{itemize}
\item[(i)]
Every object of $\CC^\Delta$ has a subobject in $\CC^{\Delta\setminus\Gamma}$
whose quotient is in $\CC^\Gamma$.
\item[(ii)]
An object of $\CC$ is in $\CC[\Gamma]$ if and only if it is a quotient of an object of $\CC^{\Gamma}$.
\item[(iii)]
$\CC[\Gamma]$ is a Serre subcategory of $\CC$. It is a highest weight
category with poset of standard objects $\Gamma$.
\item[(iv)]
The inclusion functor $i_\Gamma:\CC[\Gamma]\hookrightarrow\CC$ has a left
		adjoint ${^\vee i}_\Gamma$
sending an object $M\in\CC$ to its largest quotient
in $\CC[\Gamma]$. If $\CC$ is noetherian, it has a right 
 adjoint $i_\Gamma^\vee$ sending an object $M\in\CC$ to its largest subobject
in $\CC[\Gamma]$.
\item[(v)] We have ${^\vee i}_\Gamma(\CC^\Delta)=\CC^{\Gamma}$ and
	${^\vee i}_\Gamma(\Proj(\CC))\subset\Proj(\CC[\Gamma])$.
\end{itemize}
\end{prop}

\begin{proof}
Statement (i) follows immediately from Lemma \ref{le:orderfiltration}.

\smallskip
Let $f:P\twoheadrightarrow M$ be a surjection with $P\in\CC^\Gamma$ and
$M\in\CC$.
Let $D\in\Delta\setminus\Gamma$. There exists a surjection
$g:R\twoheadrightarrow D$ with kernel in $\CC^{\Delta_{>D}}$ such that
$\Ext^1(R,\ker f)=0$.

Consider now $h:D\to M$ and let $h'=h\circ g:R\to M$. The map $h'$ factors
through a map $h'':R\to P$. Since $R\in\CC^{\Delta\setminus\Gamma}$ and
$P\in\CC^\Gamma$, we deduce that $h''=0$, hence $h=0$. So, 
$M\in\CC[\Gamma]$.
We have shown that every quotient of an object of $\CC^\Gamma$ is in
$\CC[\Gamma]$.

\smallskip

Let $M\in\CC[\Gamma]$. Consider a surjection $f:R\twoheadrightarrow M$
with $R\in\CC^\Delta$.
By (i), there is $R'\le R$ such that
$R'\in\CC^{\Delta\setminus \Gamma}$ and $R/R'\in\CC^{\Gamma}$.
By induction on the length of the filtration,
we see that $\Hom(N,M)=0$ for all $N\in\CC^{\Delta\setminus \Gamma}$. In
particular,
$\Hom(R',M)=0$. So, $f$ factors through a
surjection $R/R'\twoheadrightarrow M$, hence $M$
is a quotient of an object of $\CC^\Gamma$. 
This shows (ii).

\smallskip
Note that $\CC[\Gamma]$ is closed under subobjects and extensions,
while the category of quotients of objects of $\CC^\Gamma$ is closed under
taking quotients. It follows that $\CC[\Gamma]$ is a Serre subcategory.

\smallskip
Let $M\in\CC[\Gamma]$ and $D,D'\in\Gamma$. We fix a
surjection $f:R\twoheadrightarrow D$ as in Definition \ref{de:hwcat}(iv).
By (i),
there is $R'\le \ker f$ such that $R'\in\CC^{\Delta\setminus \Gamma}$ and
$\ker f/R'\in\CC^{\Gamma_{>D}}$. The map $f$ factors through a
surjection $R/R'\twoheadrightarrow D$.

We have $\Hom(R',D')=0$, hence
$\Hom(R/R',D')\simeq\Hom(R,D')$ is a finitely generated projective $k$-module.
Since $\Ext^1(R,M)=0$ and $\Hom(R',M)=0$, we deduce that $\Ext^1(R/R',M)=0$.
So, the surjection $R/R'
\twoheadrightarrow D$ satisfies Definition \ref{de:hwcat}(iv) for $\CC[\Gamma]$.
We deduce that $\CC[\Gamma]$ is a highest weight category, hence (iii) holds.

\medskip
Let $M\in\CC$. There exists a surjection $f:R\twoheadrightarrow M$ with
$R\in\CC^\Delta$. As above, there is $R'\le R$ such that 
$R'\in\CC^{\Delta\setminus \Gamma}$ and $R/R'\in\CC^{\Gamma}$.
Let $M'=f(R')$. Note that $M/M'$ is a quotient of $R/R'$, so
$M/M'\in\CC[\Gamma]$. Consider now $N\in\CC[\Gamma]$. Since
$\Hom(R',N)=0$, we have $\Hom(M',N)=0$, hence every map $M\to N$
factors through $M/M'$. We deduce that $M/M'$ is the largest quotient of
$M$ that is in $\CC[\Gamma]$. The functor $M\mapsto M/M'$ is left adjoint
to the inclusion functor. 

Assume now all objects of $\CC$ are noetherian.
Consider now the family $I$ of subobjects of $M$ that are in
$\CC[\Gamma]$. Since $M$ is noetherian, the family of finite sums of
objects of $I$ has a supremum $M''$. Given $N\in\CC[\Gamma]$,
we have an isomorphism $\Hom_{\CC[\Gamma]}(N,M'')\xrightarrow{\sim}\Hom_\CC(N,M)$.
We deduce that the inclusion functor $\CC[\Gamma]\hookrightarrow\CC$ has
a right adjoint, sending an object $M$ to $M''$.
This shows (iv).

\smallskip
Let $M\in\CC^\Delta$. It follows from (i) that $M$ has a subobject 
$M'\in\CC^{\Delta\setminus\Gamma}$ with $M/M'\in\CC^{\Gamma}$. Since
${^\vee i}_\Gamma$ is right exact, we have
an exact sequence ${^\vee i}_\Gamma(M')\to {^\vee i}_\Gamma(M)\to
M/M'\to 0$. On the other hand, $\Hom(M',{^\vee i}_\Gamma(M'))=0$, hence
${^\vee i}_\Gamma(M')=0$. It follows that ${^\vee i}_\Gamma(M)\in
\CC^{\Gamma}$.

The last statement of (v) is a standard property of the left adjoint
of the inclusion of a full abelian subcategory.
\end{proof}

\begin{coro}
\label{co:CCunion}
The category $\CC$ is the union of the full subcategories
$\CC[\Gamma]$, where $\Gamma$ runs over finitely 
generated ideals of $\Delta$.
\end{coro}

\subsection{Projective objects}
Note first that since every object of $\CC$ is a quotient of an object
of $\CC^\Delta$, it follows that every projective object of $\CC$ is a direct
summand of an object of $\CC^\Delta$, that is, is an object of $i(\CC^\Delta)$.

\smallskip
We start with a projectivity criterion.

\begin{lem}
\label{le:critproj}
Let $N\in\CC^\Delta$ such that $\Ext^1(N,D)=0$ for all $D\in\Delta$.
Then, $N$ is projective.
\end{lem}

\begin{proof}
Let $M\in\CC$. By Lemma \ref{le:almostcover}, there exists
a surjection $f:R\twoheadrightarrow N$ such that $\ker f\in\CC^\Delta$
and $\Ext^1(R,M)=0$. We have $\Ext^1(N,\ker f)=0$, hence $f$ is
a split surjection and $\Ext^1(N,M)=0$. It follows that $N$ is projective.
\end{proof}

\begin{lem}
\label{le:projcoverD}
Let $D$ be an object of $\Delta$ such that $\Delta_{>D}$ is finite.
Then, there is a projective object $P$ of $\CC$ and
a surjective map $P\twoheadrightarrow D$ whose kernel is in
$\CC^{\Delta_{>D}}$.
\end{lem}

\begin{proof}
Fix $r\ge 0$ and an increasing bijection $\phi:\Delta_{>D}\xrightarrow{\sim}
\{1,2,\ldots,r\}$.
Let $P_0=D$.
We construct by induction on $i\in \{1,2,\ldots,r\}$
a family of objects $P_1,\ldots,P_r$ in $\CC^{\Delta_{>D}}$ and
surjections $f_i:P_i\twoheadrightarrow P_{i-1}$ 
such that $\Ext^1(P_i,\phi^{-1}(i))=0$ and $\ker f_i$ is a finite multiple of
$\phi^{-1}(i)$.

Assume $P_i$ has been constructed.
Since $\Ext^1(P_i,\phi^{-1}(i+1))$ is
a finitely generated $k$-module (Lemma \ref{le:finiteExt1}),
there exists an object
$P_{i+1}$ of $\CC$ and a surjection $f_{i+1}:P_{i+1}
\twoheadrightarrow P_i$ 
such that the canonical map 
$\Ext^1(P_i,\phi^{-1}(i+1))\to \Ext^1(P_{i+1},\phi^{-1}(i+1))$ vanishes
and $\ker f_{i+1}$ is a finite multiple of $\phi^{-1}(i+1)$. Since
$\Ext^1(\phi^{-1}(i+1),\phi^{-1}(i+1))=0$, it follows that
$\Ext^1(P_{i+1},\phi^{-1}(i+1))=0$.

We put $P=P_r$ and $g_i=f_i\circ\cdots\circ f_r:P\twoheadrightarrow
P_{i-1}$. Note that $\ker g_i\in\CC^{\phi^{-1}(\{i,\ldots,r\})}$.

Given $D'\in\Delta$ with $D\nless D'$, we have also $\phi^{-1}(i)\nless D'$
for all $i$, hence $\Ext^1(P,D')=0$.

Let $i\in \{1,2,\ldots,r\}$. We have $\Ext^1(\ker g_{i+1},\phi^{-1}(i))=0$ and
$\Ext^1(P_i,\phi^{-1}(i))=0$, hence $\Ext^1(P,\phi^{-1}(i))=0$.
So, $\Ext^1(P,D')=0$ for all $D'\in\Delta$. We deduce from Lemma
\ref{le:critproj} that $P$ is projective.
\end{proof}

Let us now provide a criterion for the existence of enough projective objects.

\begin{prop}
\label{pr:enoughproj}
If $\Delta_{>D}$ is finite for all $D\in\Delta$,
then $\CC$ has enough projective objects. More precisely, fix
$P_D$ a projective object with quotient $D$ for every $D\in\Delta$. Then,
$\{P_D\}_{D\in\Delta}$ is a generating family of projective objects.
Furthermore, every object of $\CC^\Delta$ has a finite projective resolution.
\end{prop}

\begin{proof}
Given $D\in\Delta$, Lemma \ref{le:projcoverD} shows there is a projective
object $P_D$ and a surjection $P_D\twoheadrightarrow D$.
Let $M\in\CC^\Delta$ with a filtration
$0=M_0\subset M_1\subset\cdots\subset M_r=M$ such that
$M_i/M_{i-1}\in\Delta$ for $1\le i\le r$.
We have a surjection $P_{M_i/M_{i-1}}\twoheadrightarrow M_i/M_{i-1}$.
It lifts to a map $f_i:P_{M_i/M_{i-1}}\to M_i$ and the sum
$\sum f_i:\bigoplus P_{M_i/M_{i-1}}\to M$ is surjective. So,
every object of $\CC^\Delta$, hence every object of $\CC$,
	is a quotient of a projective object.
The last statement follows from Lemma \ref{le:vanishingExti}.
\end{proof}

\begin{prop}
\label{le:extifaithful}
Let $\Gamma$ be an ideal of $\Delta$. Given $M,N\in\CC[\Gamma]$,
we have $\Ext^i_{\CC[\Gamma]}(M,N)=\Ext^i_{\CC}(M,N)$ for all $i\ge 0$.
\end{prop}

\begin{proof}
Note that the statement of the proposition is clear when $i\le 1$ since
$\CC[\Gamma]$ is a Serre subcategory of $\CC$ (Proposition \ref{pr:idealshw}).

\noindent
$\star\ $Assume first $\Gamma$ is a finitely generated ideal.

$\bullet\ $Assume further that $M\in\CC^\Delta$ is a projective object of
$\CC[\Gamma]$.
Assume $\Ext^i_\CC(M,N)\not=0$ for some $i>1$. There is $N'\in\CC$ and
$\zeta\in\Ext^1_\CC(M,N')$ and $\xi\in\Ext^{i-1}_\CC(N',N)$ with
$\xi\circ\zeta\not=0$.
By Lemma \ref{le:almostcover}, there exists
a surjection $f:P\twoheadrightarrow M$ with $P\in\CC^\Delta$ such that
$\Ext^1(P,N')=0$. By
Proposition \ref{pr:idealshw}(i), there exists $P'\le P$ such that
$P'\in\CC^{\Delta\setminus\Gamma}$ and $P/P'\in\CC^\Gamma$.
Since $\Hom(P',M)=0$, we deduce that $f$ factors through a surjection
$P/P'\twoheadrightarrow M$. As $M$ is projective in $\CC[\Gamma]$, that
last surjection splits, hence there is $\tilde{P}\le P$ such that $f$
restricts to a surjection $g:\tilde{P}\twoheadrightarrow M$ with kernel
$P'$. Since the composition
$\zeta\circ f$ vanishes, we deduce that $\zeta\circ g$ vanishes, hence
$\zeta$ factors through a map $h:P'\to N'$ and $\xi\circ h\not=0$.
By Lemma \ref{le:vanishingExti}, we have $\Ext^{i-1}(P',N)=0$, hence a
contradiction.
So, $\Ext^i_\CC(M,N)=0$ for all $i>0$.

\smallskip
$\bullet\ $Consider now an arbitrary $M\in\CC[\Gamma]$. Since $\Delta$
	is locally finite, the sets $\Delta_{>D}\cap\Gamma$ are finite
	for all $D\in\Gamma$.
By Proposition \ref{pr:enoughproj}, there exist $P\in\Proj(\CC[\Gamma])$ and
a surjection $f:P\twoheadrightarrow M$. Given $i>1$, there are
isomorphisms $\Ext^{i-1}_\CC(\ker f,N)\xrightarrow{\sim}\Ext^i_\CC(M,N)$ (by
the discussion above) and
$\Ext^{i-1}_{\CC[\Gamma]}(\ker f,N)\xrightarrow{\sim}
\Ext^i_{\CC[\Gamma]}(M,N)$. By
induction on $i$, we have $\Ext^{i-1}_{\CC[\Gamma]}(\ker f,N)=\Ext^{i-1}_\CC
(\ker f,N)$. It follows that $\Ext^i_{\CC[\Gamma]}(M,N)=\Ext^i_{\CC}(M,N)$.

\smallskip
\noindent
$\star\ $Consider finally the case of an arbitrary ideal $\Gamma$. There exists
a finitely generated ideal $\Gamma'$ contained in $\Gamma$ and such that
$M,N\in\CC[\Gamma']$. We have $\Ext^i_{\CC[\Gamma']}(M,N)=
\Ext^i_{\CC[\Gamma]}(M,N)$ by what we have proven already in the case
	of $\CC[\Gamma]$ and $\Gamma'$. Since
$\Ext^i_{\CC[\Gamma']}(M,N)=\Ext^i_{\CC}(M,N)$,
the proposition follows.
\end{proof}

\begin{lem}
\label{le:hwfinitegen}
Let $M,N\in\CC$.
\begin{itemize}
\item[(i)] If $M\in\Proj(\CC)$ and $N\in\CC^\Delta$, then the $k$-module
$\Hom(M,N)$ is projective.
\item[(ii)] Given $i\ge 0$, the $k$-module $\Ext^i(M,N)$ is finitely generated.
\item[(iii)] If $M\in\CC^\Delta$, then $\Ext^i(M,N)=0$ for $i\gg 0$.
\item[(iv)] Let $\Gamma$ be an ideal of $\Delta$ such that $M\in
	\CC[\Gamma]$. If $\Ext^i(D,M)\neq 0$ for
some $D\in\Delta$ and $i\ge 0$, then $D\in\Gamma$.
\end{itemize}
\end{lem}

\begin{proof}
Assume $M\in\Proj(\CC)$ and $N\in\Delta$. By Lemma \ref{le:almostcover},
there exists a surjection $R\twoheadrightarrow M$ such that
$\Hom(R,N)$ is a finitely generated projective $k$-module. Since
$\Hom(M,N)$ is a direct summand of $\Hom(R,N)$, we deduce that it is
a finitely generated projective $k$-module as well.

When $M\in\Proj(\CC)$ and $N\in\CC^\Delta$, it follows by induction on
the length of a filtration of $N$ that $\Hom(M,N)$ is a finitely
generated projective $k$-module. This shows (i).

When $M\in\Proj(\CC)$ and $N\in\CC$, there exists $N'\in\CC^\Delta$
such that $N$ is a quotient of $N'$. We deduce that $\Hom(M,N)$ is
a finitely generated $k$-module.

\smallskip
Let us now prove (ii) and (iii).
Thanks to Proposition \ref{le:extifaithful}, we can assume that
$\Delta$ is finitely generated as an ideal, hence 
$\CC$ has enough projectives by Proposition \ref{pr:enoughproj}. 

Consider a surjection $f:P\twoheadrightarrow M$ with $P$ projective.
Since $\Hom(P,N)$ is a finitely generated $k$-module, so is
	$\Hom(M,N)$. We deduce (ii) by induction on $i>0$.
We have a surjection $\Ext^{i-1}(\ker f,N)\twoheadrightarrow\Ext^i(M,N)$.
By induction, $\Ext^{i-1}(\ker f,N)$ is finitely generated, hence so
is $\Ext^i(M,N)$. This shows (ii).

It is enough to prove (iii) for $M\in\Delta$. We proceed by induction: we
assume that given $D\in\Delta_{>M}$, we have
	$\Ext^i(D,N)=0$ for $i\gg 0$. We can assume that 
	$\ker f\in\CC^{\Delta_{>M}}$,
hence $\Ext^i(\ker f,N)=0$ for $i\gg 0$. So, 
$\Ext^i(M,N)=0$ for $i\gg 0$. This shows (iii).

\smallskip
Let us show (iv). We can assume that $\Gamma$ is finitely generated.
Assume $D{\not\in}\Gamma$.
Let $\Gamma'=\Delta_{\le D}\cup\Gamma$. Since $D$ is projective in
$\CC[\Gamma']$, we have $\Ext^i_{\CC[\Gamma']}(D,M)=
\Ext^i_\CC(D,M)=0$ if $i>0$ (Proposition \ref{le:extifaithful}). So we have
$i=0$. There is $M'\in\CC^{\Gamma}$ and a surjection $M'\twoheadrightarrow M$.
Since $\Hom(D,M)\neq 0$, it follows that $\Hom(D,M')\neq 0$, a
contradiction. So (iv) holds.
\end{proof}

\begin{prop}
\label{pr:qh}
Assume $\Delta$ is finite. Then,
$(\CC,\Delta)$ is a highest category over $k$ as in
\cite[Definition 4.11]{rouquier schur}. There
is a split quasi-hereditary
$k$-algebra $A$ \cite[Definition 3.2]{CPS2} and an equivalence $A\mmod\simeq\CC$.
\end{prop}

\begin{proof}
It follows from Proposition \ref{pr:enoughproj} that $\CC$ has a progenerator $P$ and
from Lemma \ref{le:hwfinitegen} that
$A=\End(P)$ is finitely generated and projective as a $k$-module.
So, we have an equivalence $\Hom(P,-):\CC\xrightarrow{\sim}A\mmod$.
By Lemma \ref{le:hwfinitegen},
$\Hom(P,D)$ is a finitely generated projective $k$-module
for all $D\in\Delta$. Finally, Lemma \ref{le:projcoverD} shows that given $D\in\Delta$,
there is $P\in\Proj(\CC)$ and a surjection $P\twoheadrightarrow D$
with kernel in $\CC^{\Delta_{>D}}$. It follows that $(\CC,\Delta)$
is a highest category over $k$ as in \cite[Definition 4.11]{rouquier schur}.
The statement about quasi-hereditary algebras is
\cite[Theorem 4.16]{rouquier schur}.
\end{proof}

\begin{rema}
	\label{re:hwcard1}
	If $\Delta$ contains a unique object $D$, then
	there is an equivalence $\Hom(D,-):\CC\xrightarrow{\sim} k\mmod$.
\end{rema}

Given $\AC$ an abelian category, $M$ an object of $\AC$ and $L$ a simple
object of $\AC$, we denote by $[M:L]\in\BZ_{\ge 0}\cup\{\infty\}$
the maximum of the set of integers $n$ such that $M$ has a filtration
$0=M_{-1}\subset M_0\subset\cdots\subset M_{2n}=M$ with
$M_{2i-1}/M_{2i-2}\simeq L$ for $1\le i\le n$.


\begin{prop}
\label{pr:fieldsimple}
Assume $k$ is a field and $\CC$ is noetherian. Then,
\begin{itemize}
\item any object $D\in\Delta$ has a unique
	simple quotient $L(D)$ and we have $\End(L(D))=k$ and $[D:L(D)]=1$
\item every simple object of $\CC$ is isomorphic
to $L(D)$ for a unique $D\in\Delta$
\item Let $D,D'\in\Delta$ such that $[D:L(D')]{\not=}0$. Then, we have
$D'\le D$ and $[D:L(D')]<\infty$.
\end{itemize}
Assume furthermore that $\Delta$ is finitely generated as an ideal. Then
every object of $\CC$ has a projective cover. In particular, given $D\in\Delta$,
the simple object $L(D)$ has
a projective cover $P(D)$ and
$[M:L(D)]=\dim_k\Hom_\CC(P(D),M)<\infty$ for all $M\in\CC$.
\end{prop}

\begin{proof}
Let $L$ be a simple object of $\CC$. There is $D\in\Delta$ such that
$\Hom(D,L){\not=}0$, since $L$ is a quotient of an object of $\CC^\Delta$.
Consider now $D'\in\Delta$ such that $\Hom(D',L){\not=0}$. Assume
$D{\nleq}D'$ and let $\Gamma$ be the ideal of $\Delta$ generated by $D$ and
$D'$. Then $D$ is projective in $\CC[\Gamma]$ hence a
surjection $D\twoheadrightarrow L$ factors as the composite of a surjection 
$D'\twoheadrightarrow L$ and a non-zero map $D\to D'$.
	Since $\Hom(D,D')=0$, we obtain a contradiction. It follows that
	$D\le D'$. By symmetry, we obtain $D'\le D$, hence $D=D'$.
It follows that there is a unique $D\in\Delta$ such that $L$ is a quotient of
$D$.

Fix now $D\in\Delta$ and assume there are simple objects $L,L'\in\CC$
and a surjective map $f:D\twoheadrightarrow L\oplus L'$. We have
$L,L'\in\CC[\Delta_{\le D}]$ and $D$ is projective in $\CC[\Delta_{\le D}]$.
It follows that composition with $f$ induces a surjection
$\End(D)\twoheadrightarrow\Hom(D,L\oplus L')$. Since $\End(D)=k$, we obtain
a contradiction: $D$ has at most one simple quotient and
	$\dim_k\Hom(D,L)\le 1$ for all simple objects $L$.

Let $M$ be the largest subobject of $D$ that is in $\CC[\Delta_{<D}]$. Since
$\Hom(D,M)=0$, we deduce that $D/M{\not=}0$. Since $\CC[\Delta_{<D}]$ is
a Serre subcategory of $\CC$, we have $\Hom(N,D/M)=0$ for all
$N\in\CC[\Delta_{<D}]$. Let $L$ be a non-zero subobject of $D/M$.
	Since $L\in\CC[\Delta_{\le D}]$ but
	$L{\not\in}\CC[\Delta_{<D}]$, 
there exists a non-zero map $D\to L$. That map lifts to a non-zero
map $D\to D$, hence an isomorphism since $\End(D)=k$. So, $L=D/M$, hence
$D/M$ is simple.

We have shown that every $D\in\Delta$ has a unique simple quotient
$L(D)$, that $L(D)\simeq L(D')$ implies $D\simeq D'$ and every simple object
of $\CC$ is isomorphic to $L(D)$ for some $D\in\Delta$.

	Since $D$ is projective in $\CC[\Delta_{\le D}]$ and $\End(D)=k$, we deduce that 
	$[D:L(D)]=1$ and $\End(L(D))=k$, since $\dim_k\Hom(D,L(D))\le 1$.

Let $D'\in\Delta$ such that $[D:L(D')]{\not=0}$.
Let $\Gamma$ be the ideal of $\Delta$ generated by $D$ and $D'$.
By Lemma \ref{le:critproj}, there is $P\in\Proj(\CC[\Gamma])$ and
a surjection $P\twoheadrightarrow D'$ with kernel in $\CC^{\Gamma_{>D'}}$.
We have $\dim_k\Hom(P,D)<\infty$ (Lemma \ref{le:hwfinitegen}), hence $[D:L(D')]<\infty$.
Also, $\Hom(P,D){\not=}0$, hence $D'\le D$.

\smallskip
Assume now $\Delta$ is finitely generated as an ideal.
Let $D\in\Delta$. There is a projective object $P$ of $\CC$ and
a surjection $P\twoheadrightarrow L(D)$. Since $\End_\CC(P)$ is 
finite-dimensional, it follows that there is $P(D)$ an indecomposable direct 
summand of $P$ and a surjection $f:P(D)\twoheadrightarrow L(D)$.
As $\End_\CC(P(D))$ is a local $k$-algebra, it follows that $f$ is a
projective cover.
\end{proof}

\subsection{Ideals and quotients}

Let $\Gamma$ be an ideal of $\Delta$. Recall that
$\CC[\Gamma]$ is a Serre subcategory of $\CC$ (Proposition
\ref{pr:idealshw}). We put
$\CC(\Delta\setminus\Gamma)=\CC/\CC[\Gamma]$.

\begin{prop}
\label{pr:quotienthw}
Let $\Gamma$ be an ideal of $\Delta$. Then
$\CC(\Delta\setminus\Gamma)$ is a highest weight category with poset of
standard objects
$\Delta\setminus\Gamma$ and the quotient functor induces an equivalence
	$\CC^{\Delta\setminus\Gamma}\xrightarrow{\sim}
	\CC(\Delta\setminus\Gamma)^{\Delta\setminus\Gamma}$.
	
Given $\Gamma'$ an ideal of
$\Delta$, there is a canonical equivalence
$$(\CC[\Gamma'])\bigl(\Gamma'\setminus(\Gamma\cap\Gamma')\bigr)
\xrightarrow{\sim}
\bigl(\CC(\Delta\setminus\Gamma)\bigr)[\Gamma'\setminus(\Gamma\cap\Gamma')].$$
\end{prop}

\begin{proof}
Given $M\in\CC^{\Delta\setminus\Gamma}$ and $N\in\CC$, we have an
isomorphism $\Hom_\CC(M,N)\xrightarrow{\sim}
\Hom_{\CC(\Delta\setminus\Gamma)}(M,N)$
since $\Hom_\CC(M,M')=\Ext^1_\CC(M,M')=0$ for all $M'\in\CC[\Gamma]$
(Lemma \ref{le:vanishingExti}).

	We deduce that $\Hom_\CC(D,D')\simeq\Hom_{\CC(\Delta\setminus\Gamma)}(D,D')$ for all
$D,D'\in\Delta\setminus\Gamma$. It follows that (i) and (ii) in Definition
\ref{de:hwcat} hold for $\CC(\Delta\setminus\Gamma)$.

Let $M\in\CC$. By Proposition \ref{pr:idealshw}(iv),
there is $N\subset M$ such that $M/N$ is the
largest quotient of $M$ in $\CC[\Gamma]$.
There is a surjection $f:R\twoheadrightarrow N$ with
$R\in\CC^\Delta$. By Proposition \ref{pr:idealshw}(i), there is a
subobject $R'\subset R$ with $R'\in\CC^{\Delta\setminus\Gamma}$ and $R/R'\in
\CC^{\Gamma}$. Let $N'=f(R')$. We have a surjection $R/R'\twoheadrightarrow
N/N'$, hence $N/N'\in\CC[\Gamma]$. It follows that $N/N'=0$,
hence $N$ is a quotient of $R'$. The image in $\CC(\Delta\setminus\Gamma)$ of
the map $R'\to M$ is a surjection.
So, (iii) in Definition \ref{de:hwcat} holds for $\CC(\Delta\setminus\Gamma)$.

\medskip

Let $\IC$ (resp. $\JC$) be the thick subcategory of $D^b(\CC)$ generated by
	$\CC[\Gamma]$ (resp. $\Delta\setminus\Gamma$). Since
$\Ext^i_\CC(D,D')=0$ for all $i\ge 0$, $D\in\Delta\setminus\Gamma$ and 
$D'\in\Gamma$ (Lemma \ref{le:vanishingExti}), it follows that
$\Hom_{D^b(\CC)}(C,C')=0$ for $C\in\JC$ and $C'\in\IC$. We deduce that
$\Hom_{D^b(\CC)}(M,M')\xrightarrow{\sim}\Hom_{D^b(\CC)/\IC}(M,M')$ for all
$M\in\JC$ and $M'\in D^b(\CC)$.
Since $\CC(\Delta\setminus\Gamma)$ is the heart of the canonical quotient
$t$-structure
on $D^b(\CC)/\IC$, we deduce that $\Ext^1_\CC(M,M')\simeq
\Ext^1_{\CC(\Delta\setminus\Gamma)}(M,M')$ for all
$M\in\CC^{\Delta\setminus\Gamma}$
and $M'\in\CC$. In particular, if $M\in\CC^{\Delta\setminus\Gamma}$ is
projective in $\CC$, then it is projective in $\CC(\Delta\setminus\Gamma)$. Also,
$\CC^{\Delta\setminus\Gamma}\xrightarrow{\sim}\bigl(\CC(\Delta\setminus\Gamma)
\bigr)^{\Delta\setminus\Gamma}$.

\smallskip
Let us now assume that $\Delta$ is finitely generated as an ideal.
Let $M\in\CC$. We have seen that there exists $R\in\CC^{\Delta\setminus
\Gamma}$ and a map $R\to M$ that becomes surjective in $\CC(\Delta\setminus
\Gamma)$.
In the proof of Proposition \ref{pr:enoughproj}, we saw that
there exists a surjection $P\twoheadrightarrow R$ with 
$P\in\Proj(\CC)\cap \CC^{\Delta\setminus \Gamma}$. The composition
$P\to R\to M$ is surjective in $\CC(\Delta\setminus\Gamma)$ and $P$ is
projective in $\CC(\Delta\setminus\Gamma)$. 
We deduce that the quotient functor induces an equivalence
$\Proj(\CC)\cap\CC^{\Delta\setminus\Gamma}\xrightarrow{\sim}
\Proj(\CC(\Delta\setminus\Gamma))$ and that
$\CC(\Delta\setminus\Gamma)$ has enough projectives.
So, (iv) in Definition \ref{de:hwcat} holds
for $\CC(\Delta\setminus\Gamma)$.

\smallskip
Let us consider again an arbitrary $\Delta$.
We have $\CC[\Gamma']^{\Gamma'\setminus(\Gamma\cap\Gamma')}
\xrightarrow{\sim}\bigl(\CC[\Gamma'](\Gamma'\setminus(\Gamma\cap\Gamma'))
\bigr)^{\Gamma'\setminus(\Gamma\cap\Gamma')}$ and 
$\CC^{\Gamma'\setminus(\Gamma\cap\Gamma')}=
\CC[\Gamma']^{\Gamma'\setminus(\Gamma\cap\Gamma')}\xrightarrow{\sim}
\bigl(\CC(\Delta\setminus\Gamma)\bigr)^{\Gamma'
\setminus(\Gamma\cap\Gamma')}$. So, we have an equivalence
$$\bigl(\CC[\Gamma'](\Gamma'\setminus(\Gamma\cap\Gamma'))\bigr)^{\Gamma'
\setminus(\Gamma\cap\Gamma')}\xrightarrow{\sim}
\bigl(\CC(\Delta\setminus\Gamma)\bigr)^{\Gamma'
\setminus(\Gamma\cap\Gamma')}.$$
 We have shown that
every object of $(\CC[\Gamma'])(\Gamma'\setminus(\Gamma\cap\Gamma'))$ is the
cokernel
of a morphism in $(\CC[\Gamma'])(\Gamma'\setminus(\Gamma\cap\Gamma'))^{\Gamma'\setminus(\Gamma\cap\Gamma')}$, hence
the canonical functor $(\CC[\Gamma'])(\Gamma'\setminus(\Gamma\cap\Gamma'))\to
\CC(\Delta\setminus\Gamma)$ is fully faithful.

Let $D,D'\in\Delta\setminus\Gamma$ and $M\in\CC$. There is a finitely
generated ideal $\Gamma'$ of $\Gamma$ containing $D$, $D'$ and
such that $M\in\CC[\Gamma']$. 

There exists a surjection
$P\twoheadrightarrow D$ with kernel in $\CC^{\Gamma'_{>D}}$ and
$P\in\Proj(\CC[\Gamma'])$. Since $(\CC[\Gamma'])(\Gamma'\setminus
(\Gamma\cap\Gamma'))$ is
a highest weight category, we deduce that (iv) in Definition \ref{de:hwcat}
holds for $\CC(\Delta\setminus\Gamma)$. The last statement of the proposition
follows.
\end{proof}

\begin{lem}
\label{le:Gammalocal}
Let $\Gamma$ be an ideal of $\Delta$. The quotient functor
$q_\Gamma:\CC\to\CC(\Delta\setminus\Gamma)$ has a left adjoint
${^\vee q_\Gamma}$.

Assume $\Delta$ is finitely generated as an ideal.
Given $M\in\CC$, there is a finitely generated ideal $\Gamma$ of $\Delta$
such that $\Delta\setminus\Gamma$ is finite and such that the canonical map
${^\vee q_\Gamma}q_\Gamma(M)\to M$ is an isomorphism.
\end{lem}

\begin{proof}
Thanks to Corollary \ref{co:CCunion} and Proposition \ref{pr:quotienthw},
it is enough to prove the lemma when $\Delta$ is finitely generated as
an ideal, and we make now that assumption.

Let $\PC_\Gamma$ be the full subcategory of projective objects of
$\CC$ that are in $i(\CC^{\Delta\setminus\Gamma})$. The quotient functor
$q_\Gamma$ restricts to an equivalence $\phi$ from $\PC_\Gamma$ to the
category of projective objects of $\CC(\Delta\setminus\Gamma)$.
Let $N\in\CC(\Delta\setminus\Gamma)$. Fix a projective presentation
$P\xrightarrow{f} Q\to N\to 0$. Define ${^\vee q}_\Gamma(N)=
\mathrm{coker}(\phi^{-1}(f))$. It is easy to check that this defines
a functor $\CC(\Delta\setminus\Gamma)\to\CC$ that is left adjoint to $q_\Gamma$.

Let $M\in\CC$. Let $P\to Q\to M\to 0$ be a projective presentation of $M$.
There is an ideal $\Gamma$ of $\Delta$ such that $\Delta\setminus\Gamma$
is finite and $P,Q\in i(\CC^{\Delta\setminus\Gamma})$. It
follows that the canonical map ${^\vee q_\Gamma}q_\Gamma(M)\to M$ is an
isomorphism.
\end{proof}

\subsection{Base change}
\label{se:basechange}
%
%
Let $\DC$ be an additive category.
We denote by $\DC\mMOD$ the abelian category
of additive functors $\DC^\opp\to\ZM\mMOD$ and
by $\DC\mmod$ its full subcategory of functors
that are quotients of representable functors.
The Yoneda functor defines a fully faithful embedding
$$\DC\hookrightarrow \DC\mmod,\ M\mapsto\Hom(-,M).$$
The family of representable functors is a generating family of
projective objects of $\DC\mMOD$ and we will identify $\DC$ with the 
corresponding full subcategory of $\DC\mMOD$.

The full subcategory $\DC\mmod$ of $\DC\mMOD$ is closed under extensions
and quotients.
We say that $\DC$ is {\em locally noetherian} if
$\DC\mmod$ is closed under taking subobjects in $\DC\mMOD$.
When $\DC$ is locally noetherian, the subcategory
$\DC\mmod$ of $\DC\mMOD$ is a Serre subcategory.

\begin{lem}
\label{le:functorcat}
Let $\AC$ be an abelian category with enough projectives.
The canonical functor
$\AC\to \Proj(\AC)\mMOD,\ M\mapsto\Hom(-,M)$
is exact and fully faithful and it takes values in $\Proj(\AC)\mmod$.

The following conditions are equivalent
\begin{itemize}
\item[(i)] $\AC$ is noetherian
\item[(ii)] $\Proj(\AC)$ is locally noetherian
\item[(iii)] the Yoneda functor gives an equivalence
$\AC\xrightarrow{\sim} \Proj(\AC)\mmod$.
\end{itemize}
\end{lem}

\begin{proof}
The first part of the lemma is clear.

Assume (i). Let $P\in\Proj(\AC)$ and let $\Psi$ be a subobject of
$\Hom(-,P)$. There is a family $I$ of objects of $\Proj(\AC)$ and
maps $f_Q:\Hom(-,Q)\to\Psi$ for $Q\in I$ such that $\sum f_Q:\bigoplus_{Q\in I}
\Hom(-,Q)\to\Psi$
is surjective. Define $g_Q:Q\to P$ such that $\Hom(-,g_Q)$ is the
composition of $f_Q$ with the inclusion $\Psi\hookrightarrow\Hom(-,P)$.
By assumption, there is a finite subset $I'$ of $I$ such that
	$\im(g_Q)\subset \sum_{Q'\in I'}\im(g_{Q'})$ for all $Q\in I$.
	It follows that $\Psi$ is a quotient of $\Hom(-,\bigoplus_{Q'\in I'}
	Q')$, hence
$L\in\Proj(\AC)\mmod$. This shows that (ii) holds.

\smallskip
Assume (ii). Every object $M$ of $\Proj(\AC)\mmod$ is isomorphic to
the cokernel of a map $f:\Hom(-,P)\to\Hom(-,Q)$ with $P,Q\in\Proj(\AC)$.
There is $g\in\Hom(P,Q)$ such that $f=\Hom(-,g)$. We have $M\simeq
\mathrm{coker} f=\Hom(-,\mathrm{coker} g)$ and (iii) follows.

\smallskip
Assume (iii). Let $M\in\AC$ and let $I$ be a family of subobjects of $M$.
Since sums of subobjects exist in $\ZM\mMOD$, they exist in 
$\Proj(\AC)\mMOD$. So, there is a subobject $\Psi=\sum_{Q\in I}
\Hom(-,Q)$ of $\Hom(-,M)$. Let $\Phi=\Hom(-,M)/\Psi$. Since
$\Phi\in\Proj(\AC)\mmod$, there is $N\in\AC$ and an isomorphism
$\Phi\simeq\Hom(-,N)$. There is a map $f\in\Hom(M,N)$ such that
$\Hom(-,f)$ corresponds to the quotient map $\Hom(-,M)\twoheadrightarrow\Phi$.
We have $\Psi=\ker(\Hom(-,f))\simeq\Hom(-,\ker f)$. We deduce that
$\ker f=\sum_{Q\in I}Q$. So, (i) holds.
\end{proof}

Let $\DC$ be a $k$-linear category. Note that the forgetful functor
from the category of $k$-linear functors $\DC^\opp\to k\mMOD$ to the
category $\DC\mMOD$ is an isomorphism of categories.

Let $k'$ be a commutative $k$-algebra.
We denote by $k'\DC$ the $k'$-linear
category with set of objects $\{k'M\}$ where $M$ runs over the set
of objects of $\DC$ and with $\Hom_{k'\DC}(k'M,k'N)=k'\otimes\Hom_\DC(M,N)$.
There is a base change functor $k'\otimes -:\DC\to k'\DC,\ M\mapsto k'M$
that is compatible, via the Yoneda embedding, with the base change
functor $k'\otimes -:\DC\mMOD\to (k'\DC)\mMOD,\ F\mapsto (k'N\mapsto 
k'\otimes F(N))$.

\begin{lem}
\label{le:lattice}
Assume $k'$ is a localization of $k$. Let $M\in (k'\DC)\mmod$. There exists
$\tilde{M}\in\DC\mmod$ such that $k'\tilde{M}\simeq M$.
\end{lem}

\begin{proof}
Let $Q\in\DC$ and $N$ a subobject of $k'Q$ in
$(k'\DC)\mMOD$ such that $M\simeq (k'Q)/N$.

Let $\phi:k\to k'$ be the
canonical algebra map and $\phi_Q=\phi\otimes\id_Q:Q\to k'\otimes Q$.
Let $L=\phi_Q^{-1}(N)\subset Q$. We have a canonical isomorphism
$k'L\xrightarrow{\sim} N$, hence $M\simeq k'(Q/L)$.
\end{proof}

\begin{lem}
\label{le:basechangenoetherian}
Assume $k'$ is a finitely generated module over a localization of $k$.

If $\DC$ is locally noetherian, then $k'\DC$ is locally noetherian.
\end{lem}

\begin{proof}
Let $M\in\DC$ and let $N$ be a subobject of $k'M$ in
$(k'\DC)\mMOD$.

Assume $k'$ is a localization of $k$. The proof of Lemma \ref{le:lattice}
shows that there exists a subobject $L$ of $M$ such that $k'L\simeq N$.
Since $L\in\DC\mmod$, it follows that $N\in(k'\DC)\mmod$.

Assume $k'$ is a finitely generated $k$-module. The restriction of $k'M$
to $\DC\mMOD$ is a quotient of a finite direct sum of copies of
$M$, hence the restriction $N_0$ of $N$ to $\DC\mMOD$ is the quotient
of an object of $\DC\mmod$. So, there exists
$P\in \DC$ and a surjective map $P\twoheadrightarrow N_0$ in $\DC\mmod$.
By adjunction, we obtain a surjective map
$k'P\twoheadrightarrow N$ in $(k'\DC)\mMOD$. So $N\in (k'\DC)\mmod$.

The general case follows.
\end{proof}

Assume $k$ is a discrete valuation ring with residue field $\bar{k}$ and
field of fractions $K$.

\begin{lem}
Assume $\Hom$-spaces in $K\DC$ are finite-dimensional vector spaces over $K$
and $\DC$ is locally noetherian. Then $K\DC$ and $\bar{k}\DC$ are
locally noetherian.

Let $M\in K\DC$. There exists $\tilde{M}\in\DC$ such that 
$K\tilde{M}\simeq M$ and $\Hom(P,M)$ is a free 
$k$-module of finite rank for every $P\in\DC$.

Given any such $\tilde{M}$, the class $[\bar{k}\tilde{M}]\in
K_0((\bar{k}\DC)\mmod)$ depends only on $[M]\in K_0((K\DC)\mmod)$.
\end{lem}

\begin{proof}
	Note that  $K\DC$ and $\bar{k}\DC$ are
locally noetherian by Lemma \ref{le:basechangenoetherian}.

Lemma \ref{le:lattice} ensures the existence of $N\in\DC\mmod$ such that
$k'N\simeq M$. Let $N'$ be the torsion subobject of $N$ and $\tilde{M}=N/N'$.
We have $k'\tilde{M}\simeq M$. Let $P\in\DC$. Since
$\Hom(P,M)$ is a torsion-free $k$-module
such that $K\otimes\Hom(P,M)$ is a finite-dimensional $K$-vector space,
it follows that $\Hom(P,M)$ is a free $k$-module of finite rank.

\smallskip
Let $\tilde{M}_1$ and $\tilde{M}_2$ be two objects of $\DC\mmod$. Given
	$P\in\DC$, the canonical map $K\Hom(P,\tilde{M}_2)\to
	\Hom(KP,K\tilde{M}_2)$ is an isomorphism. Since
	$\tilde{M}_1$ has a resolution
	$P'_1\to P_1\to \tilde{M}_1\to 0$ with
	$P'_1,P_1\in\DC$, it follows that the canonical
	map $K\Hom(\tilde{M}_1,\tilde{M}_2)\to\Hom(K\tilde{M}_1,K\tilde{M}_2)$
	is an isomorphism.

	Assume now we have an isomorphism $g:K\tilde{M}_1
	\xrightarrow{\sim}K\tilde{M}_2$. There is $f:M_1\to M_2$ and 
	$m\ge 0$ such that $Kf=\pi^m g$, where
$\pi$ is a generator of the maximal ideal of $k$. Similarly, there is
	$f':M_2\to M_1$ and $m'\ge 0$ such that $Kf'=\pi^{m'}g^{-1}$.
	Let $n=m+m'$.  Since $f'\circ f=\id^n$,
	it follows that $f$ is injective. We have
	$f\circ f'=\pi^n \id$, hence
$\pi^n\tilde{M}_2\subset\im(f)$.

	We proceed by induction on $n$ 
to show that $[\bar{k}\tilde{M}_1]=[\bar{k}\tilde{M}_2]$.

Assume $n=1$. Let $L=\mathrm{coker} f$.
There is an exact sequence 
$$0\to \Tor_1^k(\bar{k},L)\to \bar{k}\tilde{M}_1\to
\bar{k}\tilde{M}_2\to \bar{k}L\to 0.$$
Since $\pi L=0$, we have $\Tor_1^k(\bar{k},L)\simeq \bar{k}L$,
hence $[\bar{k}\tilde{M}_1]=[\bar{k}\tilde{M}_2]$.

In the general case, let $\tilde{M}_3=\im(f)+\pi^{n-1}\tilde{M}_2$.
We have $\pi \tilde{M}_3\subset \im(f)$, hence
$[\bar{k}\tilde{M}_1]=[\bar{k}\tilde{M}_3]$. We have
$\mG^{n-1}\tilde{M}_2\subset \tilde{M}_3$, so it follows by induction that
$[\bar{k}\tilde{M}_2]=[\bar{k}\tilde{M}_3]$. This completes the proof of the
lemma.
\end{proof}

The previous lemma provides a decomposition map
$$d:K_0((K\DC)\mmod)\to K_0((\bar{k}\DC)\mmod)$$
with the property that $d([KM])=[\bar{k}M]$ when $M\in\DC\mmod$ and
$\Hom(P,M)$ is a projective $k$-module for all $P\in\DC$.

\medskip
We assume from now on that $\CC$ is a highest weight category over $k$ with
poset of standard objects $\Delta$.

Proposition \ref{pr:enoughproj} and Lemma \ref{le:functorcat} imply the
following result.

\begin{lem}
\label{le:projfunctors}
Assume $\Delta$ is finitely generated as an ideal and $\CC$ is noetherian.
The Yoneda functor induces an equivalence
$\CC\xrightarrow{\sim} \Proj(\CC)\mmod$.
\end{lem}

\smallskip
Given two ideals $\Gamma\subset\Gamma'$, we have a functor
$F_{\Gamma\subset\Gamma'}:\CC[\Gamma']\to\CC[\Gamma]$ sending
an object of $\CC[\Gamma']$ to its largest quotient in $\CC[\Gamma]$
(Proposition \ref{pr:idealshw}).
This functor sends projective objects to projective objects.

When $\Gamma$ and $\Gamma'$ are finitely generated, we
have a commutative diagram
$$\xymatrix{
\CC[\Gamma]\ar[rr]^-{M\mapsto\Hom(-,M)}_-\sim \ar@{^{(}->}[d] &&
 \Proj(\CC[\Gamma])\mmod
\ar[d]^{-\circ F_{\Gamma\subset\Gamma'}} \\
\CC[\Gamma']\ar[rr]_-{M\mapsto\Hom(-,M)}^-\sim &&
 \Proj(\CC[\Gamma'])\mmod
}$$
and the vertical arrow $-\circ F_{\Gamma\subset\Gamma'}$ is fully faithful.

\smallskip
Given $M\in\CC$,
there is a finitely generated ideal $\Gamma'$ of $\Delta$ such that
$M\in\CC[\Gamma']$ (Corollary \ref{co:CCunion}) and $\Hom(-,M)$ defines an
object of $\Proj(\CC[\Gamma'])\mmod$.
So, we obtain a functor
$$\CC\to\colim_{\Gamma}\ \Proj(\CC[\Gamma])\mmod,$$
where the colimit is taken using the system of strictly transitive
transition functors $-\circ F_{\Gamma\subset\Gamma'}$.
Lemma \ref{le:projfunctors} shows the following result.

\begin{coro}
The Yoneda functor gives an equivalence
$$\CC\xrightarrow{\sim}\colim_{\Gamma}\ \Proj(\CC[\Gamma])\mmod,$$
where $\Gamma$ runs over finitely generated ideals of $\Delta$.
\end{coro}

Let $k'$ be a noetherian commutative $k$-algebra.
Given $\Gamma$ a finitely generated ideal of $\Delta$, we put
$(k'\CC)_{\Gamma}=\bigl(k'\Proj(\CC[\Gamma])\bigr)\mmod$.
We define
$k'\CC=\colim_\Gamma (k'\CC)_{\Gamma}$, where $\Gamma$ runs 
over finitely generated ideals of $\Delta$.
Note that, in this colimit, given $\Gamma_1\subset\Gamma_2$, the functors
$(k'\CC)_{\Gamma_1}\to (k'\CC)_{\Gamma_2}$ are fully faithful.

The base change functor $k'\otimes-:
\Proj(\CC[\Gamma])\mmod\to
\bigl(k'\Proj(\CC[\Gamma])\bigr)\mmod$ induces a
base change functor $k'\otimes -:\CC\to k'\CC$.

\begin{prop}
\label{pr:basechangehw}
	Assume $\mathrm{Proj}(\CC[\Gamma])$ and
	$k'\mathrm{Proj}(\CC[\Gamma])$ are locally noetherian
	for all finitely generated ideals $\Gamma$ of $\Delta$ (note that
	that the assertion over $k'$ is a consequence of the one over $k$ when
	$k'$ is a finitely generated module over a localization of $k$).

The category $k'\CC$ is a highest weight category with set
of standard objects $k'\Delta=\{k'D\}_{D\in\Delta}$.
\end{prop}

\begin{proof}
Let $D\in\Delta$. The object $D$
is projective in $\CC[\Delta_{\le D}]$, hence 
$$\End_{k'\CC}(k'D)=\End_{(k'\CC)_{\Delta_{\le D}}}(k'D)=
k'\End_{\CC[\Delta_{\le D}]}(D)=k'.$$
So, (i) in Definition \ref{de:hwcat} holds for $k'\CC$.

Let $D_1,D_2\in\Delta$ and let $\Gamma$ be the ideal of
$\Delta$ generated by $D_1$ and $D_2$. If $D_1\nleq D_2$, then
$D_1$ is projective in $\CC[\Gamma]$, hence as above we have
$\Hom(k'D_1,k'D_2)\simeq k'\Hom(D_1,D_2)=0$.
We deduce that (ii) holds.

By assumption, every object of $k'\CC$ is a quotient of an object
of the form $k'P$, where $P$ is a projective object of
$\CC[\Gamma]$ for some finitely generated $\Gamma$. Note that
$k'P\in (k'\CC)^{k'\Delta}$ and (iii) holds.

\smallskip

Let $\Gamma$ be a finitely generated ideal of $\Delta$.
Let $M\in (k'\CC)_\Gamma$. Let $D\in\Delta\setminus\Gamma$ and
consider $g:k'D\to M$.
Let $P\in \Proj(\CC[\Gamma])$ and $f:k'P\twoheadrightarrow M$ be
a surjection. Let $\Gamma'=\Gamma\cup\Delta_{\le D}$.
Since $D$ is projective in $\CC[\Gamma']$ (Lemma \ref{le:projcoverD}),
it follows that $k'D$ is projective in $(k'\CC)_{\Gamma'}$, hence
there is $h:k'D\to k'P$ such that $g=f\circ h$.
On the other hand, $\Hom_{(k'\CC)_{\Gamma'}}(k'D,k'P)
\simeq k'\Hom_{\CC}(D,P)=0$. So, $g=0$.
We deduce that $\Hom(k'D,M)=0$, hence
$(k'\CC)_\Gamma\subset (k'\CC)[k'\Gamma]$.

\smallskip
Let $D,D'\in\Delta$ and $M\in k'\CC$. Let 
$\Gamma$ be a finitely generated ideal of $\Delta$ containing $D$ and $D'$
and such that $M\in (k'\CC)_\Gamma$. By Lemma \ref{le:projcoverD}, there
is $P\in\Proj(\CC[\Gamma])$ and a surjective
map $P\twoheadrightarrow D$ whose kernel is in $\CC^{\Delta_{>D}}$.
So, we have a surjection $k'P\twoheadrightarrow k'D$
whose kernel is in $(k'\CC)^{k'\Delta_{>D}}$.
We have 
$\Hom_{(k'\CC)_\Gamma}(k'P,k'D')\simeq
k'\Hom_{\CC[\Gamma]}(P,D')$ and $\Hom_\CC(P,D')$ is a finitely
generated projective $k$-module (Lemma \ref{le:hwfinitegen}). It follows
that $\Hom_{k'\CC}(k'P,k'D')$ is a finitely generated 
projective $k'$-module.

 Consider an exact sequence $0\to M\to N
\to k'P\to 0$ with $N\in k'\CC$. There is a finitely generated
ideal $\Gamma'$ of $\Delta$ containing $\Gamma$ and such that
$N\in(k'\CC)_{\Gamma'}$. Now, there is $R\in\Proj(\CC[\Gamma'])$
and a surjection $f:k'R\twoheadrightarrow N$. By Proposition
\ref{pr:idealshw}(i),
there is $R'\le R$ such that $R'\in\CC^{\Gamma'\setminus\Gamma}$ and
$R/R'\in\CC^\Gamma$. We have $\Hom(k'R',k'P)=0$ and
$\Hom(k'R',M)=0$, hence $f$ factors through a surjection
$k'(R/R')\twoheadrightarrow N$. We have $k'(R/R')\in
(k'\CC)_{\Gamma}$, hence $N\in (k'\CC)_{\Gamma}$, so the
surjection $N\twoheadrightarrow k'P$ splits. This shows that
$\Ext^1_{k'\CC}(k',M)=0$, hence (iv) holds.
This completes the proof of
the proposition.
\end{proof}

\subsection{Grothendieck groups}
\label{hw:K0}
The next lemma follows from \cite[Lemma II.6.2.7]{We2}.

\begin{lem}
\label{le:K0colim}
We have $K_0(\CC)=\colim_{\Gamma}K_0(\CC[\Gamma])$, where $\Gamma$
runs over finitely generated ideals of $\Delta$.
\end{lem}

\smallskip
Let $V\in k\mmod$. Given $M\in\DC$, the object $V\otimes_k\Hom(-,M)$
is representable by an object $V\otimes_k M$:
given $k^r\xrightarrow{f}k^s\to V\to 0$ an exact sequence
in $k\mmod$, we have $V\otimes_k M=\mathrm{coker}(f\otimes M:M^r\to M^s)$.

\begin{lem}
\label{le:K0action}
Let $M\in\CC^\Delta$ and let $0\to V_1\to V\to V_2\to 0$ be an exact
sequence in $k\mmod$. We have an exact sequence
$0\to V_1\otimes_k M\to V\otimes_k M\to V_2\otimes_k M\to 0$.
\end{lem}

\begin{proof}
We can assume that $\Delta$ is finitely generated as an ideal, so that
$\CC$ has enough projectives.
Let $P\in\Proj(\CC)$. Since $\Hom(P,M)$ is projective over $k$,
it follows that we have an exact sequence
$0\to V_1\otimes_k \Hom(P,M)\to V\otimes_k \Hom(P,M)\to V_2\otimes_k\Hom(P,M)
\to 0$, hence an exact sequence
$0\to \Hom(P,V_1\otimes_k M)\to \Hom(P,V\otimes_k M)\to \Hom(P,V_2\otimes_k M)
\to 0$. The lemma follows.
\end{proof}

\begin{lem}
\label{le:exactsequenceK0}
Let $\Gamma$ be an ideal of $\Delta$ such that $\Delta\setminus\Gamma$
is finite. We have an exact sequence
$$0\to K_0(\CC[\Gamma])\to K_0(\CC)\to K_0(\CC(\Delta\setminus\Gamma))\to 0.$$
\end{lem}

\begin{proof}
Without assumption on $\Delta\setminus\Gamma$, there is an exact sequence
\cite[Theorem II.6.4]{We2}
$$K_0(\CC[\Gamma])\to K_0(\CC)\to K_0(\CC(\Delta\setminus\Gamma))\to 0.$$

Assume $\Delta\setminus\Gamma$ has a single element $D_0$.
It follows from Lemma \ref{le:K0action} that the functor
$D_0\otimes_k -:k\mmod\to\CC$ is exact.
Let $M\in\CC$. Let $N$ be the cone of the adjunction map
	$f:D_0\otimes_k\Hom_\CC(D_0,M)\to M$. Since $\Hom(D_0,f)$ is an
	isomorphism and $D_0$ is projective, it follows that
$\Hom_{\CC}(D_0,H^i(N))=0$ for all $i$. So, $N$ is a
bounded complex with cohomology contained in $\CC[\Gamma]$.
It follows that $M\mapsto N$ provides
a left adjoint to the inclusion functor
of the thick subcategory of $D^b(\CC)$ of complexes with cohomology
in $\CC[\Gamma]$. As a consequence, the canonical map
$K_0(\CC[\Gamma])\to K_0(\CC)$ is a split injection. So, the lemma holds
when $|\Delta\setminus\Gamma|=1$. The general case follows by induction
on $|\Delta\setminus\Gamma|$.
\end{proof}

It follows from Lemma \ref{le:K0action} that given $V\in k\mmod$ and
$D\in\Delta$, the class $[V\otimes_k D]\in K_0(\CC)$ depends only
on $[V]\in G_0(k)$ and $[D]$. We denote it by $[V]\cdot [D]$. This
provides $K_0(\CC)$ with a structure of $G_0(k)$-module.

\begin{lem}
\label{le:injK0}
The morphism of $G_0(k)$-modules
$$G_0(k)^{(\Delta)}\to K_0(\CC),\ (a_D)_{D\in\Delta}\mapsto
\sum_D a_D\cdot [D]$$
is injective.

If $\Delta$ is finitely generated as an ideal, then
the canonical map $K_0(\CC\mproj)\to K_0(\CC)$ is injective with
image the free $\ZM$-submodule generated by $\{[D]\}_{D\in\Delta}$.
\end{lem}

\begin{proof}
When $\Delta$ is finite, the lemma follows from Lemma
\ref{le:exactsequenceK0}.
When $\Delta$ is finitely generated,
given a finite subset $I$ of $\Delta$, there is an ideal $\Gamma$
of $\Delta$ such that $\Delta\setminus\Gamma$ is finite and contains $I$.
	Since the lemma holds for $\CC(\Delta\setminus\Gamma)$, we deduce that the lemma holds
for $\CC$.
The general case follows from Lemma \ref{le:K0colim}.
\end{proof}

\begin{defi}
We say that $\CC$ is {\em separated} if $D^b(\CC)$ is generated as
a triangulated category 
by $\{V\otimes_k D\}$, where $V\in k\mmod$ and $D\in\Delta$.
\end{defi}

\begin{lem}
\label{le:finiteseparated}
If all the finitely generated ideals of $\Gamma$ are finite, then
$\CC$ is separated.
\end{lem}

\begin{proof}
When $\Delta$ is finite, the statement follows by induction as in the proof
of Lemma \ref{le:exactsequenceK0}, cf Remark \ref{re:hwcard1}.

Let $M$ be a bounded complex of objects of $\CC$. There is a finitely
generated ideal $\Gamma$ of $\Delta$ such that $M$ is
a bounded complex of objects of $\CC[\Gamma]$. Since the lemma
holds for $\CC[\Gamma]$, we are done.
\end{proof}

\begin{lem}
\label{le:basisK0separated}
If $\CC$ is separated, then $K_0(\CC)$ is a free
$K_0(k\mmod)$-module with basis $\{[D]\}_{D\in\Delta}$.

If in addition $\Delta$ is finite and $\Spec k$ is connected,
then there is an isomorphism
$$K_0(k\mmod)\otimes_\ZM K_0(\CC\mproj)\xrightarrow{\sim}K_0(\CC),\
a\otimes [P]\mapsto a\cdot [P].$$
\end{lem}

\begin{proof}
The first statement follows immediately from Lemma \ref{le:injK0}.
The second statement follows by induction from Lemma \ref{le:exactsequenceK0}.
\end{proof}

\begin{rema}
If $D^b(\CC)$ is generated, as
a triangulated category closed under taking direct summands,
by $\{V\otimes_k D\}$, where $V\in k\mmod$ and $D\in\Delta$,
then it is separated: this is shown by the proof of Lemma
\ref{le:exactsequenceK0}
when $\Delta$ is finite, and the general case follows.\finl
\end{rema}

\subsection{Completed Grothendieck groups}
\label{se:completedK0}

\smallskip
We define the {\em completed Grothendieck group} of
$\CC$ as
$\hat{K}_0(\CC)=
\colim_{\Gamma}\lim_{\Omega\subset\Gamma} K_0(\CC[\Gamma]/\CC[\Omega])$
where $\Gamma$ runs over finitely generated ideals of $\Delta$
and $\Omega$ runs over ideals of $\Gamma$ such that
$\Gamma\setminus\Omega$ is finite.

Note that given $\Omega\subset\Omega'\subset\Gamma$,
the transition map $K_0(\CC(\Gamma\setminus\Omega))\to
K_0(\CC(\Gamma\setminus\Omega'))$ is surjective, while given 
$\Gamma'\subset\Gamma$, the  transition map 
$\lim_{\Omega\subset\Gamma} K_0(\CC(\Gamma\setminus\Omega))\to
\lim_{\Omega\subset\Gamma'} K_0(\CC(\Gamma'\setminus\Omega))$ is injective.

There is a canonical morphism of groups 
$$K_0(\CC)\to \hat{K}_0(\CC),\ [M]\mapsto [[M]]$$
since $\CC=\colim_\Gamma\CC[\Gamma]$, where
$\Gamma$ runs over finitely generated ideals of $\Delta$.

\smallskip
When $\Gamma$ is finitely generated, we have
$\hat{K}_0(\CC)=\lim_{\Omega\subset\Gamma} K_0(\CC[\Gamma]/\CC[\Omega])$
where $\Omega$ runs over ideals of $\Delta$ such that
$\Delta\setminus\Omega$ is finite.

When $\Delta$ is finite, we have a canonical isomorphism
$K_0(\CC)\xrightarrow{\sim} \hat{K}_0(\CC)$.

\smallskip
Let $\mathrm{Map}^{\mathrm{fg}}(\Delta,K_0(k\mmod))$ be the
abelian group of maps $\chi:\Delta\to K_0(k\mmod)$ such that
$\{D\in\Delta\ |\ \chi(D)\neq 0\}$ is contained in a finitely generated ideal
of $\Delta$.

\begin{lem}
\label{le:completedK0maps}
There is an isomorphism
$$\sigma:\mathrm{Map}^{\mathrm{fg}}(\Delta,K_0(k\mmod))
\xrightarrow{\sim} \hat{K}_0(\CC),\
\chi\mapsto \sum_{D\in\Delta}[\chi(D)]\cdot [[D]]$$
and an isomorphism
$$\hat{K}_0(\CC)\xrightarrow{\sim} \mathrm{Map}^{\mathrm{fg}}(\Delta,K_0(k\mmod))
,\ [[M]]\mapsto
(D\mapsto \sum_{i\ge 0}(-1)^i[\Ext^i(D,M)]).$$
\end{lem}

\begin{proof}
Assume first $\Delta$ is finite. Consider the morphisms
$$f:K_0(k\mmod)^\Delta\xrightarrow{\sim} K_0(\CC),\ ([V_D])_{D\in\Delta}\mapsto
\sum_{D\in\Delta}[V_D]\cdot [D]$$
and
$$g:K_0(\CC)\xrightarrow{\sim} K_0(k\mmod)^\Delta,\ [M]\mapsto 
(D\mapsto \sum_{i\ge 0}(-1)^i[\Ext^i(D,M)]).$$
Lemma \ref{le:basisK0separated} shows that $f$ is an isomorphism. Since
$g\circ f$ has a triangular matrix with entries $1$ on the diagonal,
it follows that $g$ is an isomorphism.

Consider now a general $\Delta$.
We have $\mathrm{Map}^{\mathrm{fg}}(\Delta,K_0(k\mmod))=
\colim_\Gamma\lim_\Omega K_0(k\mmod)^\Omega$
where $\Gamma$ runs over finitely generated ideals of $\Delta$
and $\Omega$ runs over ideals of $\Gamma$ such that
$\Gamma\setminus\Omega$ is finite. The lemma follows from the case where
$\Delta$ is finite.
\end{proof}

Lemmas \ref{le:basisK0separated} and \ref{le:completedK0maps} have the
following consequence.

\begin{prop}
If $\CC$ is separated, then we have a canonical injection
$K_0(\CC)\hookrightarrow\hat{K}_0(\CC)$.
\end{prop}

\begin{lem}
\label{le:K0hatmult}
Assume $k$ is a field and $\CC$ is noetherian.
	The map $M\mapsto ([M:L(D)])_{D\in\Delta}$ induces
an injection
$\hat{K}_0(\CC)\hookrightarrow \BZ^\Delta$.
\end{lem}

\begin{proof}
Let $M\in\CC$ and $D\in\Delta$. There is a finitely generated ideal
$\Gamma$ of $\Delta$ and an ideal $\Omega$ of $\Gamma$ such that
$\Gamma\setminus\Omega$ is finite and contains $D$ and such that
$M\in\CC[\Gamma]$. We have $[M:L(D)]=[M':L(D)]$, where $M'$ is the image
of $M$ in $\CC(\Gamma\setminus\Omega)$. It follows that
$[M:L(D)]$ depends only on the class of $M$ in $\hat{K}_0(\CC)$.

When $\Delta$ is finite, $\CC$ is equivalent to the category
of finite-dimensional modules over a finite-dimensional $k$-algebra, hence
the class of a module in $K_0$ is determined by the multiplicities
of simple modules in a composition series. So, we obtain the injectivity
when $\Delta$ is finite, and the general case follows.
\end{proof}

%
%

Let $k'$ be a noetherian commutative flat $k$-algebra. There is a commutative
diagram
$$\xymatrix{
K_0(\CC)\ar[r]\ar[d] & \hat{K}_0(\CC)\ar[r]^-{\sigma^{-1}}_-\sim\ar[d] &
\mathrm{Map}^{\mathrm{fg}}(\Delta,K_0(k\mmod))\ar[d] \\
K_0(k'\CC)\ar[r] & \hat{K}_0(k'\CC)\ar[r]_-{\sigma^{-1}}^-\sim &
\mathrm{Map}^{\mathrm{fg}}(\Delta,K_0(k'\mmod))
}$$
where the vertical maps are induced by the functor $k'\otimes_k-$.

\subsection{Decomposition maps}
\label{se:hwdecmap}
We assume in \S\ref{se:hwdecmap} that $\CC$ is noetherian.

If $k$ is a discrete valuation ring with residue field $\bar{k}$ and
field of fractions $K$, it 
follows from \S \ref{se:basechange} that there is a decomposition map
$$d:K_0(K\CC)\to K_0(\bar{k}\CC)$$
with the property that $d([KM])=[\bar{k}M]$ when $M$ is an object of
$\CC[\Gamma]$ where $\Gamma$ is a finitely generated ideal of $\Delta$
and 
$\Hom(P,M)\in k\mproj$ for all $P\in\Proj(\CC[\Gamma])$.

\medskip
We construct now decomposition maps over more general local rings $k$,
using completed Grothendieck groups.

Assume $k$ is a local integral ring with residue field $\bar{k}$ and
field of fractions $K$. We have canonical isomorphisms 
$\dim_K:K_0(K\mmod)\xrightarrow{\sim}\ZM$ and
$\dim_{\bar{k}}:K_0(\bar{k}\mmod)\xrightarrow{\sim}\ZM$.
We define $\hat{d}:K_0(K\CC)\xrightarrow{\sim} K_0(\bar{k}\CC)$ as the map
making the following diagram commutative
$$\xymatrix{
\hat{K}_0(K\CC)\ar[r]^-{\sigma^{-1}}_-\sim\ar[d]_{\hat{d}} &
\mathrm{Map}^{\mathrm{fg}}(\Delta,\ZM)\ar@{=}[d] \\
\hat{K}_0(k'\CC)\ar[r]_-{\sigma^{-1}}^-\sim &
\mathrm{Map}^{\mathrm{fg}}(\Delta,\ZM)
}$$

\begin{prop}
Let $M\in\CC$ and let $\Gamma$ be a finitely generated ideal of $\Delta$
such that $M\in\CC[\Gamma]$. Assume $\Hom(P,M)\in k\mproj$ for all
$P\in\Proj(\CC[\Gamma])$. Then $\hat{d}([[KM]])=[[kM]]$.
\end{prop}

\begin{proof}
Assume first $\Delta$ is finite. There are objects $\bar{D}$ of $\CC$
for $D\in\Delta$ such that
$\Ext^i(D,\bar{D}')=\delta_{0i}\delta_{D,D'}$ for all $D,D'\in\Delta$
and $i\ge 0$ (Proposition \ref{pr:qh} and 
\cite[Proposition 4.19]{rouquier schur}).
The assumption on $M$ guarantees that it has a finite projective resolution
$0\to P^{-n}\to\cdots\to P^0\to M\to 0$
\cite[Proposition 4.23]{rouquier schur}.
As a consequence, $[[M]]=\sigma(D\mapsto \sum_{i\ge 0}(-1)^i
[\Hom(P^i,\bar{D})])$. Since
$\dim_K [\Hom_{K\CC}(KP^i,K\bar{D})]=
\dim_k [\Hom_{k\CC}(kP^i,k\bar{D})]$, the proposition follows.

Consider the general case. Let $\Gamma$ be a finitely generated ideal
of $\Delta$ such that $M\in\CC[\Gamma]$. Let $\Omega$ be a finitely
generated ideal of $\Gamma$ such that $\Gamma\setminus\Omega$ is finite.
There exists a projective object of $\CC[\Gamma]$ whose image in
$\CC[\Gamma]/\CC[\Omega]$ is a progenerator, and
$\Hom(P,M)\simeq\Hom(q_\Omega(P),q_\Omega(M))$ (cf. Lemma
\ref{le:Gammalocal}). Since the proposition holds for $q_\Omega(M)$,
we deduce that it holds for $M$.
\end{proof}

When $k$ is a discrete valuation ring, there is a
commutative diagram
$$\xymatrix{
K_0(K\CC)\ar[r]\ar[d]_d &\hat{K}_0(K\CC)\ar[r]^-{\sigma^{-1}}_-\sim
\ar[d]_{\hat{d}} &
\mathrm{Map}^{\mathrm{fg}}(\Delta,\ZM)\ar@{=}[d] \\
K_0(\bar{k}\CC)\ar[r] &\hat{K}_0(k'\CC)\ar[r]_-{\sigma^{-1}}^-\sim &
\mathrm{Map}^{\mathrm{fg}}(\Delta,\ZM)
}$$

\subsection{Blocks}
\label{se:hwblocks}

We assume in \S \ref{se:hwblocks} that $\Spec k$ is connected.

We define the equivalence relation $\sim$ on $\Delta$ as the one generated
by $D\sim D'$ when $\Ext^i_{\CC}(D,D')\neq 0$ for some $i\in\{0,1\}$. This
is the equivalence relation generated by the partial order $\lessdot$
(cf. Proposition \ref{pr:coarsestorder}). 

\begin{prop}
\label{pr:blockshw}
Given $I\in\Delta/\sim$, the full subcategory $\CC[I]$ of $\CC$ is an
indecomposable
Serre subcategory whose objects are the quotients of objects of $\CC^I$. It
is a highest weight category with poset of standard objects $I$.
We have $\CC=\bigoplus_{I\in\Delta/\sim}\CC[I]$.
\end{prop}

\begin{proof}
Let $I,J\in\Delta/\sim$ with $I\neq J$. Given $M\in\CC^I$ and $N\in\CC^J$,
we have $\Ext^1(M,N)=0$ and $\Hom(M,N)=0$. It follows that 
$\CC^\Delta=\bigoplus_{I\in\Delta/\sim}\CC^I$. We deduce that
a quotient of an object of $\CC^I$ is in $\CC[I]$.

Let $L\in\CC$. There is an exact sequence $M\xrightarrow{f} N\to L\to 0$
with $M,N\in\CC^\Delta$. We have decompositions
$M=\bigoplus_{I\in\Delta/\sim}M^I$,
$N=\bigoplus_{I\in\Delta/\sim}N^I$ and $f=\sum_I f^I$ with 
$f^I:M^I\to N^I$ and $M^I,N^I\in\CC^I$. It follows that $L=\bigoplus_I
\mathrm{coker} f^I$ and $\mathrm{coker} f^I$ is a quotient of an object of
$\CC^I$.
Assume $L\in\CC[I]$. Consider $J\neq I$.
We have $\Hom(N^J,L)=0$, hence $\mathrm{coker} f^J=0$. It follows that
$L=\mathrm{coker} f^I$ is the quotient of an object of $\CC^I$.
We have shown that $\CC=\bigoplus_{I\in\Delta/\sim}\CC[I]$.

Let $e$ be an idempotent of the center of $\CC[I]$. Note that
$e$ acts by $0$ or $1$ on an object of $\Delta$. Let
$I_e$ be the subset of $I$ of objects on which $e$ acts by $0$.
Given $D,D'\in I$ with $\Ext^i(D,D')\neq 0$ for some $i\in\{0,1\}$, we have
$D,D'\in I_e$ or $D,D'\in I\setminus I_e$. We deduce that
$I=I_e$ or $I_e=\emptyset$. Since every object of $\CC[I]$ is the quotient
of an object of $\CC^I$, it follows that $e=0$ or $e=1$. So,
$\CC[I]$ is indecomposable.

Since $(\CC,\Delta,\lessdot)$ is a highest weight category 
(Proposition \ref{pr:coarsestorder}), it
follows that $\CC[I]$ is a highest weight category with set of standard
$I$ and partial order $\lessdot$, hence it is a highest weight category
with the partial order $<$.
\end{proof}

\begin{lem}
\label{le:hwblocksfinte}
Let $D,D'\in\Delta$. 

We have $D\sim D'$ if and only if there is
a finitely generated ideal $\Gamma$ of $\Delta$ and an ideal $\Omega$
of $\Gamma$ such that $\Gamma\setminus\Omega$ is a finite set
containing $D$ and $D'$ and such that
$q_\Omega(D)$ and $q_\Omega(D')$
are in the same block of $\CC[\Gamma](\Gamma\setminus\Omega)$.
\end{lem}

\begin{proof}
Let $\Gamma$ be a finitely generated ideal of $\Delta$ and $\Omega$ an ideal
of $\Gamma$ such that $\Gamma\setminus\Omega$ is a finite set
containing $D$ and $D'$.
We have $\Ext^i_\CC(D,D')\simeq\Ext^i_{\CC[\Gamma](\Omega)}(q_\Omega(D),
q_\Omega(D'))$ for $i\in\{0,1\}$ (cf. proof of Proposition \ref{pr:quotienthw}).

We have $D\sim D'$ if and only if there exists
$D_0=D,D_1,\ldots,D_n=D'$ in $\Delta$ such that 
$\Ext^*(D_i,D_{i+1})\neq 0$ or $\Ext^*(D_{i+1},D_i)\neq 0$ for
some $*\in\{0,1\}$, for all $i\in\{0,\ldots,n-1\}$. So,
$D\sim D'$ if and only if there is a finitely generated ideal $\Gamma$ of
$\Delta$, an ideal $\Omega$ of $\Gamma$ and
$D_0=D,D_1,\ldots,D_n=D'$ in $\Gamma\setminus\Omega$ such that 
$\Ext^*_{\CC[\Gamma](\Omega)}(q_\Omega(D_i),q_\Omega(D_{i+1}))\neq 0$ or
$\Ext^*_{\CC[\Gamma](\Omega)}(q_\Omega(D_{i+1}),q_\Omega(D_i))\neq 0$ for
some $*\in\{0,1\}$, for all $i\in\{0,\ldots,n-1\}$: this is equivalent
to the requirement that $q_\Omega(D)$ and $q_\Omega(D')$ are in the same
block of $\CC[\Gamma](\Gamma\setminus\Omega)$ by Proposition \ref{pr:blockshw}.
\end{proof}

We assume for the remainder of \S \ref{se:hwblocks} that $\CC$ is noetherian.

\smallskip
Given
$\Gamma$ a finitely generated ideal of $\Delta$ and $\Omega$ an ideal
of $\Gamma$ such that $\Gamma\setminus\Omega$ is a finite set
containing $D$ and $D'$, there is a finite family $B_{\Gamma,\Omega}$
of prime ideals of $k$ such that  given a prime ideal $\qG$ of $k$, we have
$k_\qG\otimes_k D\sim k_\qG\otimes_k D'$ if and only if $\pG\subset\qG$.

\begin{lem}
\label{pr:blockshwspec}
Let $D,D'\in\Delta$ and let $k'$ be a commutative noetherian $k$-algebra
that is a finitely generated module over a localization of $k$.
If $k'\otimes_k D\sim k'\otimes_k D'$ in $k'\Delta$
then $D\sim D'$.
\end{lem}

\begin{proof}
Assume first $\Delta$ is finite. In that case, the result is classical:
we have $R\Hom^\bullet_{k'\CC}(k'D,k'D')\simeq
k'\otimes^\LM R\Hom^\bullet_\CC(D,D')$. So, if $\Ext^*_{k'\CC}(k'D,k'D')
\neq 0$ for some $*\in\{0,1\}$, then $\Ext^*_\CC(D,D')\neq 0$ for some
$*\ge 0$.

Consider $D=D_0,\ldots,D_n=D'$ in $\Delta$ with 
$\Ext^*_{k'\CC}(k'D_i,k'D_{i+1})\neq 0$ or 
$\Ext^*_{k'\CC}(k'D_{i+1},k'D_i)\neq 0$ for some $*\in\{0,1\}$ for all
$i\in\{0,n-1\}$. We have
$\Ext^*_{\CC}(D_i,D_{i+1})\neq 0$ or 
$\Ext^*_{\CC}(D_{i+1},D_i)\neq 0$ for some $*\ge 0$ for all
$i\in\{0,n-1\}$. It follows that $D$ and $D'$ are in the same block of
$\CC$, hence $D\sim D'$ by Proposition \ref{pr:blockshw}.

The general case follows from Proposition \ref{pr:basechangehw} and its proof.
\end{proof}

\begin{prop}
Assume $k$ is integral and integrally closed and let $K$ be its field
of fractions.

	There exists a (unique) family $\FC$ 
of height one prime ideals of $k$ with the following property:

given $\pG$ a prime ideal of $k$, the partition of $\Delta$ into blocks
of $k_\pG\CC$ is the same as the partitions into blocks of
$K\CC$ if and only if $\pG$ contains no element of $\FC$.

\smallskip
Given $\pG$ a prime ideal of $k$ containing an element of $\FC$, 
a subset of $\Delta$ 
corresponds to a block of $k_\pG\CC$ if and only if it is a union
of blocks of $k_\qG\CC$ for all $\qG\in\FC$ with $\qG\subset\pG$.

The set $\FC$ has the cardinality of a union over the set of finite subsets of
$\Delta$, of finite sets.
\end{prop}

\begin{proof}
Assume $\Delta$ is finite. We have $\CC\simeq A\mmod$ where $A$ is
a $k$-algebra that is finitely generated and projective as a $k$-module.
The proposition follows from 
Proposition \ref{pr:blockshw} and
Proposition \ref{codimension un}

The general case follows from Lemma \ref{le:hwblocksfinte}. The set $\FC$
is the union of the finite sets associated with the categories
$\CC[\Gamma](\Omega)$, where $\Gamma$ runs over finitely generated ideals
of $\Delta$, $\Omega$ runs over ideals of $\Gamma$ such that
$\Gamma\setminus\Omega$ is finite and generates $\Gamma$ as an ideal
of $\Delta$.
\end{proof}

\begin{rema}
Proposition \ref{pr:blockshwspec} shows that if $D{\not\sim D'}$, then
$k_\qG\otimes_k D{\not\sim} k_\qG\otimes_k D'$ for all prime ideals
$\qG$ of $k$.\finl
\end{rema}

\section{Triangular algebras}
\label{se:Appendixtriangular}

The construction of a highest weight category from representations of triangular
algebras in \cite[\S 2]{ggor} is done under the presence of a grading coming from an
inner derivation. The method used there does not actually
use that the gradings are inner, and we describe here constructions following
\cite[\S 2]{ggor}, with a more general setting.

\smallskip

\subsection{Definition}
\label{de:triangular}
Let $A$ be a graded $k$-algebra with three graded
subalgebras $B^+$, $B^-$ and $H$ such that
\begin{itemize}
\item[($\bigtriangleup$i)] $H$, $B^-$ and $B^+$ are flat $k$-modules
\item[($\bigtriangleup$ii)] the multiplication map $\mu:B^+\otimes H\otimes B^-\to A$ is an
isomorphism of $k$-modules
\item[($\bigtriangleup$iii)] $\mu(B^+\otimes H)=\mu(H\otimes B^+)$
and $\mu(B^-\otimes H)=\mu(H\otimes B^-)$
\item[($\bigtriangleup$iv)] $B^-_{>0}=B^+_{<0}=0$, $B^-_0=B^+_0=k$ and $H=H_0$.
\end{itemize}

\smallskip
The canonical maps
$H\to (B^+H)/(B^+H)_{<0}$ and $H\to (B^-H)/(B^-H)_{>0}$ are isomorphisms.
Composing their inverse with the quotient maps,
we obtain a canonical morphism of graded $k$-algebras
$B^\epsilon H\twoheadrightarrow H$ for $\epsilon\in\{+,-\}$.

\smallskip
We identify the category of $H$-modules with the category of graded
$H$-modules that are concentrated in degree $0$.

\subsection{Induced modules}
\subsubsection{}
Given $F$ a graded $(B^-H)$-module, we put $\Delta(F)=A\otimes_{B^-H}F$. 
Note that $\Delta$ is an exact functor that is left adjoint to the restriction functor from the category of
graded $A$-modules to the category of graded $(B^-H)$-modules.

There is a canonical isomorphism of graded $(B^+H)$-modules
$$B^+H\otimes_H F\xrightarrow{\sim} \Delta(F)$$
and a canonical isomorphism of graded $B^+$-modules
$$B^+\otimes F\xrightarrow{\sim} \Delta(F).$$

\medskip
When $E$ is an $H$-module, we view $E$ as a graded $(B^-H)$-module concentrated in degree $0$
through the canonical map $p:B^-H\twoheadrightarrow H$ and we
put $\Delta(E)=\Delta(p^*(E))$.
There is a canonical isomorphism of graded $(B^+H)$-modules
$$B^+H\otimes_H E\xrightarrow{\sim}\Delta(E).$$

\medskip
Let $n\ge 0$. We put
$$\Delta_n(E)=\Delta\Bigl(\bigl((B^-H/(B^-H)_{<-n}\bigr)\otimes_H E\Bigr).$$
Note that $\Delta_0(E)=\Delta(E)$ and $\Delta_n(E)$ has a filtration
with subquotients 
$$\Delta(E), \Delta\bigl((B^-H)_{-1}\otimes_H E\bigr),\ldots,
\Delta\bigl((B^-H)_{-n}\otimes_H E\bigr).$$

\subsubsection{}
Given $M$ a $B^-$-module and $n$ a non-positive integer, let
$\mathrm{Ann}_{B^-_{<n}}(M)=\{m\in M | B^-_{<n}m=0\}$, a $B^-$-submodule of $M$.
We put $M_{\ln}=\bigcup_{n<0} \mathrm{Ann}_{B^-_{<n}}(M)$ and we
say that $M$ is {\em locally nilpotent} for $B^-$ if $M=M_{\ln}$.
The functor $M\mapsto M_{\ln}$ is right adjoint to the inclusion functor from the 
category of locally nilpotent $B^-$-modules to the category of $B^-$-modules.

\smallskip
Note that
\begin{itemize}
\item When $M$ is a graded $B^-$-module, $M_{\ln}$ is a graded $B^-$-submodule of $M$.
\item If $M$ is a $(B^-H)$-module, then $M_{\ln}$ is a $(B^-H)$-submodule of $M$, since
$HB^-_{<n}=(HB^-)_{<n}=(B^-H)_{<n}=B^-_{<n}H$.
\end{itemize}

\smallskip
Assume now $M$ is an $A$-module. Since
$(B^-)_{<i}H(B^+)_j\subset A_{<i+j}\subset A(B^-)_{<i+j}$ for $i+j<0$,
it follows
that $M_{\ln}$ is an $A$-submodule of $M$. We deduce that $M\mapsto M_{\ln}$ is right adjoint to the inclusion functor from
the category of $A$-modules (resp. graded $A$-modules) that are locally nilpotent as $B^-$-modules to the category of
$A$-modules (resp. graded $A$-modules).

\medskip
Let $M\in A\mmodgr$ and $E\in H\mmodgr$. 
We have an isomorphism of $k$-modules
$$\Hom_{A\mmodgr}(\Delta_n(E),M)\xrightarrow{\sim}
\Hom_{H\mmodgr}(E,\mathrm{Ann}_{B^-_{<-n}}(M)),\ f\mapsto f_{|1\otimes E}
$$

If $E$ is concentrated in degree $d$, then
$$\Hom_{H\mmodgr}(E,\mathrm{Ann}_{B^-_{<-n}}(M))=
\Hom_H(E,\bigl(\mathrm{Ann}_{B^-_{<-n}}(M)\bigr)_d)\subset
\Hom_H(E,M_d).$$


\smallskip
Note that if
$M$ is a finitely generated graded $A$-module that is locally nilpotent for $B^-$, then 
$M_{<i}=0$ for $i\ll 0$.

\begin{lem}
Let $M\in A\mmodgr$, let $E$ be an $H$-module, let $d\in\BZ$ and let $i\in\BZ$ such that
$M_{<i}=0$. Then the canonical map
$\Hom_{A\mmodgr}(\Delta_n(E)\langle -d\rangle,M)\to \Hom_H(E,M_d)$
is an isomorphism for $n\ge d-i$.
\end{lem}

\begin{proof}
We have $\bigl(\mathrm{Ann}_{B^-_{<-n}}(M)\bigr)_d=M_d$ since
$(B^-)_{<-n}M_d\subset M_{<d-n}=0$ and
the result follows.
\end{proof}

\begin{lem}
\label{le:HomDelta}
Let $E,E'\in H\mmod$ and let $d\in\BZ$. 

If $d<0$, then
$\Hom_{A\mmodgr}(\Delta(E),\Delta(E'\langle d\rangle))=0$.

The functor $\Delta$ induces an isomorphism
$$\Hom_H(E,E')\xrightarrow{\sim}\Hom_{A\mmodgr}(\Delta(E),\Delta(E')).$$
\end{lem}

\begin{proof}
We have 
$$\Hom_{A\mmodgr}(\Delta(E),\Delta(E'\langle d\rangle))\simeq
\Hom_{B^-H\mmodgr}(E,\Delta(E')\langle d\rangle)$$
If $d<0$ then $(\Delta(E')\langle d\rangle)_0=0$, hence $\Hom_{B^-H\mmodgr}(E,\Delta(E')\langle d\rangle)=0$. This
shows the first statement.

Note that $E'=\Delta(E')_0$ is a $(B^-H)$-submodule of $\Delta(E')$. It follows that
$$\Hom_{B^-H\mmodgr}(E,\Delta(E'))\simeq\Hom_H(E,E')$$
and the second statement follows.
\end{proof}

\subsection{Category $\OC^{gr}$}
We fix a set $I$ of isomorphism classes of
$H$-modules that are finitely generated as $k$-modules.

\medskip

We assume that
\begin{itemize}
\item[($\bigtriangleup$v)] given $E\in I$, every submodule of $E$ is a quotient of a finite
multiple of $E$
\item[($\bigtriangleup$vi)] $\End_H(E)=k$ for all $E\in I$
\item[($\bigtriangleup$vii)] $\Hom_H(E,F)=0$ for all $E, F\in I$ with $E{\not\simeq F}$
\item[($\bigtriangleup$viii)] 
the $H$-modules $\bigl((B^+H)_{\ge n}/(B^+H)_{>n}\bigr)\otimes_H E$ and
$\bigl((B^-H)_{\le -n}/(B^-H)_{<-n}\bigr)\otimes_H E$ are direct summands
of finite direct sums of objects of $I$ for all $E\in I$ and $n\ge 0$
\item[($\bigtriangleup$ix)] the $A$-module $\Delta(E)$ is noetherian for all $E\in I$.
\end{itemize}

 We denote by
$\OC_H$ the full subcategory of $H\mMOD$ with objects the quotients of 
finite direct sums of objects in $I$. Note that $\OC_H$ is
an abelian subcategory of $H\mMOD$ closed under quotients and subobjects,
and $I$ is a set of projective
generators for $\OC_H$. There is an equivalence of
categories
\begin{equation}
	\label{eq:conditionIH}
	\bigoplus_{E\in I}\Hom(E,-):\OC_H\xrightarrow{\sim} (k\mmod)^{(I)}.
\end{equation}

\begin{exemple}
\label{ex:Hsplit}
Let $I$ be a family of isomorphism classes of $H$-modules. Assume
for every $E\in I$, the quotient of $H$ by the annihilator of $E$ is a
matrix algebra over $k$ and $E$ is the pullback of its vector representation.

Note that the modules in $I$ are free $k$-modules of finite rank. We
	have the equivalence of categories (\ref{eq:conditionIH})
	and the conditions ($\bigtriangleup$v)-($\bigtriangleup$vii) hold.\finl
\end{exemple}

\begin{rema}
Note that Assertion ($\bigtriangleup$ix) is satisfied if $B^+$ is noetherian:
since $E$ is a finitely generated $k$-module, it follows that 
$\Delta(E)\simeq B^+\otimes E$ is a finitely generated $B^+$-module,
hence a noetherian $B^+$-module.\finl
\end{rema}

\smallskip
We denote by $\OC^{gr}$ the category of finitely
generated graded $A$-modules $M$ that are
locally nilpotent for $B^-$ and satisfy $M_i\in\OC_H$ for all $i\in\BZ$.
Since $\OC_H$ is closed under taking quotients, we deduce
that $\OC^{gr}$ is a full subcategory of $A\mmodgr$ closed under taking
quotients.
Note that if $\OC_H$ is closed under extensions, then $\OC^{gr}$ will also
be closed under extensions.


\begin{lem}
\label{le:OH}
Let $M$ be a graded $A$-module. The following conditions are equivalent
\begin{itemize}
\item[(i)] $M\in\OC^{gr}$
\item[(ii)] there exists a finite family $S$ of objects of $I$,
$d_E\in\BZ$ and $n_E\in\BZ_{\ge 0}$ for
$E\in S$
such that $M$ is a quotient of $\bigoplus_{E\in S}
\Delta_{n_E}(E\langle d_E\rangle)$.
\end{itemize}
\end{lem}

\begin{proof}
Assume (i). There is a finite subset $J$ of $\BZ$ such that $M$ is generated
by $\bigoplus_{j\in J}M_j$ as an $A$-module.
Given $j\in J$, there is a finite family $S_j$ of objects of
$I$ and a surjective morphism of $H$-modules
$f_j:\bigoplus_{E\in S_j}E\twoheadrightarrow M_j$.
By adjunction, we obtain a morphism
of graded $(B^-H)$-modules $g_j:\bigoplus_{E\in S_j} 
B^-H\otimes_H E\langle -j\rangle\to M$.
Let $E\in S_j$. Since $f_j(E)$ is a finitely generated $H$-module,
there is a non-negative integer $n_E$ such that $(B^-)_{\le -n_E}f_j(E)=0$.
So, $g_j$ factors through a morphism of graded $(B^-H)$-modules
$h_j:\bigoplus_{E\in S_j} \bigl((B^-H)/(B^-H)_{\le -n_E}\bigr)\otimes_H E\langle -j\rangle\to M$.
Let 
$$h'_j:\bigoplus_{E\in S_j} A\otimes_{B^-H}
\bigl((B^-H)/(B^-H)_{\le -n_E}\bigr)\otimes_H E\langle -j\rangle\to M$$
be the morphism of graded $A$-modules obtained from $h_j$ by adjunction.
The map $\sum_{j\in J}h'_j$ is surjective and this shows (ii) holds.

Assume (ii). Let $E\in I$.
Note that $(\Delta_n(E))_{<-n}=0$, hence $B^-_{<-n-i}(\Delta_n(E))_i=0$. It follows that 
$\Delta_n(E)$ is locally nilpotent for $B^-$. There is an isomorphism of $H$-modules
$(B^-H)/(B^-H)_{<-n}\simeq\bigoplus_{i=0}^n(B^-H)_{\le -i}/(B^-H)_{<-i}$,
hence
$$\Delta_n(E)_i\simeq\bigoplus_j(B^+H)_{i+j}\otimes_H \bigl((B^-H)_{-j}
\otimes_H E\bigr)$$
is isomorphic to a direct summand of a finite direct sum of objects of $I$.
We deduce that $\Delta_n(E)\in\OC^{gr}$.

Since $\OC^{gr}$ is closed under taking finite
direct sums, quotients and shifts, it follows that (i) holds.
\end{proof}

\begin{lem}
\label{le:vanishingExt1Delta}
Let $E\in I$, $n\in\ZM_{\ge 0}$ and let $M\in \OC^{gr}$.
If $M_{<-n}=0$, then $\Ext^1_{\OC^{gr}}(\Delta_n(E),M)=0$.

In particular, given $E'\in I$ and $d\le n$, 
we have
$\Ext^1_{A\mmodgr}(\Delta_n(E),\Delta(E'\langle d\rangle))=0$.
\end{lem}

\begin{proof}
We have 
$$\Hom_{A\mmodgr}(A\otimes_{B^-H}(B^-H)_{<-n}\otimes_H E,M)\simeq
\Hom_{(B^-H)\mmodgr}((B^-H)_{<-n}\otimes_H E,M)=0$$
since
$((B^-H)_{<-n}\otimes_H E)_i=0$ for $i\ge -n$, while
$M_i=0$ for $i<-n$.

There is an exact sequence of graded $A$-modules
$$0\to A\otimes_{B^-H}(B^-H)_{<-n}\otimes_H E\to A\otimes_H E\to \Delta_n(E)\to 0.$$
Since $A\otimes_H E$ is projective in the category
of graded $A$-modules
whose restriction to $H$ is in $\OC_H$, we deduce that
$\Ext^1_{\OC^{gr}}(\Delta_n(E),M)=0$.
\end{proof}

Let $\Delta^{gr}=\{\Delta(E\langle n\rangle)\}_{E\in I,\ n\in\ZM}$,
a set of objects of $\OC^{gr}$ (cf. Lemma \ref{le:OH}). We put a partial order
on $\Delta^{gr}$: given $E,F\in I$ and
$i,j\in\ZM$, we put $E\langle i\rangle < F\langle j\rangle$ if $i<j$.
Given $n\in\BZ$, we put $\Delta^{gr,\le n}=
\{\Delta(E\langle r\rangle)\}_{E\in I,r\le n}$.

\begin{theo}
\label{th:hwgraded}
$\OC^{gr}$ is a noetherian
highest weight category with poset of standard objects $\Delta^{gr}$.

	Given $n\in\BZ$, the set
	$\{\Delta_{n-r}(E\langle r\rangle)\}_{r\le n,E\in I}$ is a set
	of projective generators for $\OC^{gr}[\Delta^{gr,\le n}]$.
\end{theo}

\begin{proof}
Let $E\in I$. By ($\bigtriangleup$ix), every
$A$-submodule of $\Delta_n(E)$ is finitely generated. So, 
all subobjects of objects of $\OC^{gr}$ are finitely generated as
$B^+$-modules by Lemma \ref{le:OH}. We deduce that $\OC^{gr}$ is closed
under taking subobjects, since $\OC_H$ has that property.

We check now the conditions of Definition \ref{de:hwcat}.
Conditions (i) and (ii) are given by Lemma \ref{le:HomDelta}, and
(iii) by Lemma \ref{le:OH}.

Let $E,E'\in I$, $d\in\ZM$ and $M\in\OC^{gr}$. By Lemma \ref{le:OH},
there is $m\ge 0$ such that
$M_{<-m}=0$. It follows from Lemma \ref{le:vanishingExt1Delta}
that $\Ext^1_{\OC^{gr}}(\Delta_n(E),M)=0$ for $n\ge m$.
There is $m'\ge 0$ such that $\Delta(E')_d$ is killed by $B^-_{<-m'}$.
Consequently given $n\ge m'$, we have
$$\Hom_{\OC^{gr}}(\Delta_n(E),\Delta(E'\langle d\rangle))\simeq
\Hom_{H\mmodgr}(E,\mathrm{Ann}_{B^-_{<-n}}(\Delta(E'\langle d\rangle)))
=\Hom_{H\mmodgr}(E,\Delta(E')_d).$$
Since $\Delta(E')_d$ is a direct summand of a finite direct sum of objects of
$I$,
we deduce that $\Hom_{\OC^{gr}}(\Delta_n(E),\Delta(E'\langle d\rangle))\in
k\mproj$. So, (iv) holds.

\smallskip
The projectivity of the objects $\Delta_{n-r}(E\langle r\rangle)$ follows from
Lemmas \ref{le:vanishingExt1Delta} and \ref{le:critproj}.
Since every object of $\OC^{gr}[\Delta^{gr,\le n}]$ is a quotient of an
object filtered by $\Delta(E\langle r\rangle)$'s with $r\le n$, the
generation follows.
\end{proof}

%

\medskip
Let $k'$ be a noetherian commutative $k$-algebra. We put
$A'=k'A$, $B^{\prime\pm}=k'B^\pm$ and $H'=k'H$. The conditions
($\bigtriangleup$i)-($\bigtriangleup$iv) are satisfied for those $k'$-algebras. 

Let $I'=\{k'E\}_{E\in I}$. Note that assumption ($\bigtriangleup$viii) is satisfied for $I'$.
We assume that assumptions ($\bigtriangleup$v)-($\bigtriangleup$vii) hold for $I'$ (this will be the case
if $I$ is as in Remark \ref{ex:Hsplit}) and that assumption ($\bigtriangleup$ix) holds.

We obtain a highest category $\OC^{\prime gr}$ over $k'$.
Given $E$ an $H$-module and given $n\ge 0$, we have a canonical isomorphism
$k'\Delta_n(E)\xrightarrow{\sim} \Delta'_n(k'E)$. It follows that
tensoring by $k'$ the projective generating family for
$\OC^{gr}[\Delta^{gr,\le n}]$
given by
Theorem \ref{th:hwgraded} gives a projective generating family for
$\OC^{\prime gr}[\Delta^{\prime gr,\le n}]$. We deduce
that $k'\OC^{gr}\xrightarrow{\sim} \OC^{\prime gr}$.

\subsection{Inner grading}
\label{se:innergrading}
We assume in \S \ref{se:innergrading} that $k$ is a $\QM$-algebra
and that there is $h\in A$ such that
$A_i=\{a\in A\ |\ ha-ah=ia\}$ for all $i\in\BZ$. We assume that
$A\neq A_0$.

\begin{lem}
There is a unique decomposition $h=h'+h_0$ where
$h'\in B^+\otimes H\otimes B^-_{<0}$ and $h_0\in \Zrm(H)$.
\end{lem}

\begin{proof}
We have $h\in A_0$. Write $h=h'+h_0$ where
$h'\in B^+\otimes H\otimes B^-_{<0}$ and $h_0\in B^+\otimes H$.
Since $h_0\in A_0$, it follows that $h_0\in H$. Let $a\in H$.
We have $0=[h,a]=[h',a]+[h_0,a]$. Since 
$[h',a]\in B^+\otimes H\otimes B^-_{<0}$ and $[h_0,a]\in H$, we deduce
that $[h_0,a]=0$. This shows that $h_0\in \Zrm(H)$.
\end{proof}

Let $E\in I$. The action of $h_0$ on $E$ is given by multiplication by
an element $C_E$ of $k$. Note that $h$ acts by $C_E+i$ on $\Delta(E)_i$.

\smallskip
We put $L=\bigcup_{E\in I}(C_E+\ZM)\subset k$.
Let $\sim$ be the equivalence relation on $L$ generated by
$\lambda\sim\lambda'$ if $\lambda-\lambda'{\not\in}k^\times$.
We assume in \S \ref{se:innergrading} that
\begin{itemize}
\item[(x)]
$C_E{\not\sim}C_E+i$ for $i\in\ZM-\{0\}$.
\end{itemize}

\smallskip
We define a relation on $I$ as the transitive closure of the relation 
$E>F$ if
$C_E-C_F+\ZM_{<0}{\not\in}k^\times$.
Our assumption above ensures that $>$ is a partial order on $I$.

\medskip
Let $\OC$ be the category of finitely generated
$A$-modules that are locally nilpotent for $B^-$ and whose restriction to
$H$ is the quotient of a (possibly infinite) direct sum of objects of $I$.

\smallskip
Let $M\in A\mMOD$. Given $\lambda\in k$, we put
$$\WC_\lambda(M)=\{m\in M\ |\ (h-\lambda)^nm=0\ \text{ for }n\gg 0\}.$$

Given $\alpha\in L/\sim$, we put
$\WC_\alpha(M)=\sum_{\lambda\in\alpha}\WC_\lambda(M)$.

\smallskip
We denote by $\OC^{gr,\alpha}$ the full subcategory of $\OC^{gr}$
of objects $M$ such that $M_i\subset\WC_{i+\alpha}(M)$ for all
$i\in\ZM$. We denote by $\OC^{\alpha+\ZM}$ the full subcategory
of $\OC$ of objects $M$ such that $M=\sum_{i\in\ZM}\WC_{\alpha+i}(M)$.

\begin{prop}
\label{pr:gradedungraded}
We have $\OC^{gr}=\bigoplus_{\alpha\in L/\sim}\OC^{gr,\alpha}$ and
$\OC=\bigoplus_{\alpha+\ZM\in (L/\sim)/\ZM}\OC^{\alpha+\ZM}$.
Furthermore, given $\alpha\in L/\sim$, the forgetful functor gives an
equivalence $\OC^{gr,\alpha}\xrightarrow{\sim}\OC^{\alpha+\ZM}$.

Fix an element $\tilde{\beta}\in L/\sim$ for each $\beta\in(L/\sim)/\ZM$.
There is an equivalence of graded categories
$\OC^{(\ZM)}\xrightarrow{\sim}\OC^{gr}$ sending $\Delta(E)$ to
$\Delta(E)\langle \widetilde{C_E+\ZM}-C_E\rangle$.
\end{prop}

\begin{proof}
Let $M\in\OC$.
Let us first show that $M=\sum_{\lambda\in L}\WC_\lambda(M)$. This is clear
when $M=\Delta(E)$ for some $E$. It follows that it holds also when
$M=\Delta_n(E)$, hence for $M$ a direct sum of $\Delta_n(E)$'s.
One shows as in Lemma \ref{le:OH}
that every object $M\in\OC$ is a quotient of a finite direct sum of
$\Delta_n(E)$'s, hence the result holds for $M$.

	Note that if $\lambda{\not\sim}\lambda'$, then $\WC_\lambda(M)\cap
	\WC_{\lambda'}(M)=0$.
It follows that we have an $H$-module decomposition
$M=\bigoplus_{\alpha\in L/\sim}\WC_\alpha(M)$.

Let $M(\alpha+\ZM)=\bigoplus_{i\in\ZM}\WC_{\alpha+i}(M)$: this is
an $A$-submodule of $M$ in $\OC^{\alpha+\ZM}$ and 
$M=\bigoplus_{\alpha+\ZM\in (L/\sim)/\ZM}M(\alpha+\ZM)$. This gives
the required decomposition of $\OC$.

\smallskip
Assume now $M\in\OC^{gr}$. Let $M(\alpha)_i=\WC_{\alpha+i}(M)\cap M_i$
and $M(\alpha)=\bigoplus_{i\in\ZM}M(\alpha)_i$, a graded $A$-submodule of $M$
contained in $\OC^{gr,\alpha}$. We have $M=\bigoplus_{\alpha\in L/\sim}M(\alpha)$,
and this provides the required decomposition of $\OC^{gr}$.

Let $M\in\OC^{\alpha+\ZM}$. Let $M^i=\WC_{\alpha+i}(M)$. This defines
a structure of graded $A$-module on $M$, and gives an object $M'$ of
$\OC^{gr,\alpha}$. The construction $M\mapsto M'$ is inverse to
the forgetful functor $\OC^{gr,\alpha}\to\OC^{\alpha+\ZM}$.

\smallskip
Consider now an element $\tilde{\beta}\in L/\sim$ for each
$\beta\in(L/\sim)/\ZM$. We have constructed an equivalence
$F:\OC^\beta\xrightarrow{\sim}\OC^{gr,\tilde{\beta}}$. These functors
extend uniquely to the required equivalence of graded categories
$\OC^{(\ZM)}\xrightarrow{\sim}\OC^{gr}$.
\end{proof}

From Proposition \ref{pr:gradedungraded}
and Theorem \ref{th:hwgraded}, we deduce the
following result (\cite[Theorem 2.19]{ggor} when $k$ is a field).
Let $\Delta=\{\Delta(E)\}_{E\in I}$, with the poset structure of $I$.

\begin{theo}
\label{th:hwinner}
$\OC$ is a noetherian highest weight category with poset of standard objects
$\Delta$.
\end{theo}

\begin{exemple}
Let $\gG$ be a finite dimensional reductive Lie algebra over $k=\CM$. Let $\bG_+$ be a Borel subalgebra
and $\hG\subset\bG_+$ a Cartan subalgebra. Let $\bG^-$ be the opposite Borel subalgebra.
Let $A=U(\gG)$, $B^\pm=U(\bG^\pm)$ and $H=U(\hG)$.
Let $h\in\hG$ be the sum of the simple coroots. We consider the inner grading on $\gG$, hence on
$A$ defined by $\ad(h)$. We have $\bG^+_{<0}=\bG^-_{>0}=\bG_{\not=0}=0$ and $\bG^-_0=\bG^+_0=\CM$.

We take for $I$ the set of isomorphism classes of simple $\hG$-modules, so that $\OC_H$ is the category
of semisimple $\hG$-modules. Then $\OC$ is the usual BGG category.\finl
\end{exemple}

\begin{rema}
If the grading on $A$ is not inner, then $\OC$ is not a highest weight
category in general.\finl
\end{rema}

\newpage
\renewcommand\thechapter{}

\def\chaptername{}\setcounter{chapter}{0}

\chapter{Prime ideals and geometry}

$$
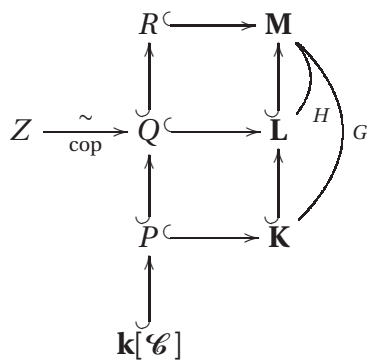
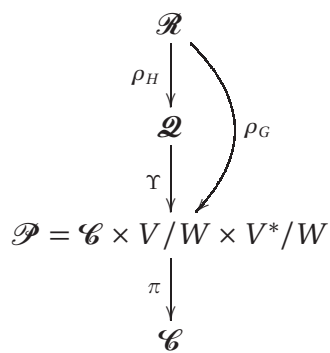
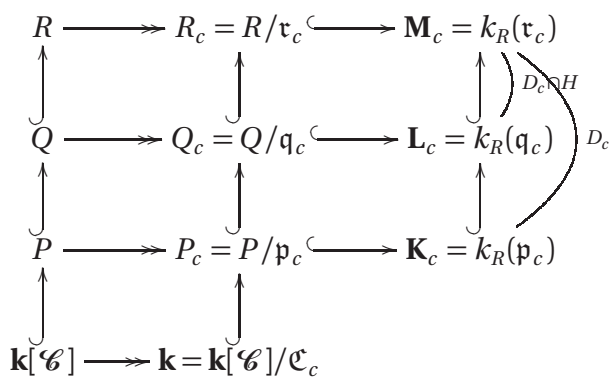
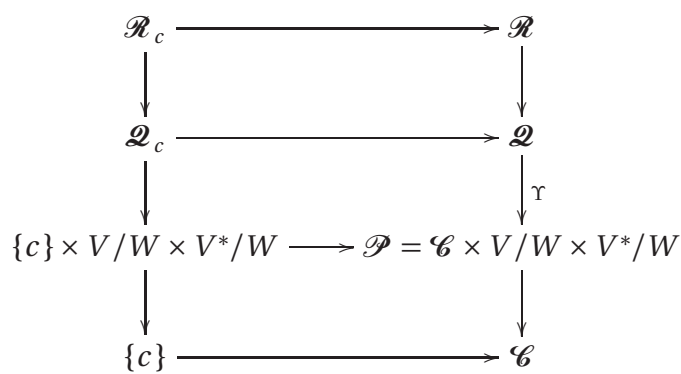
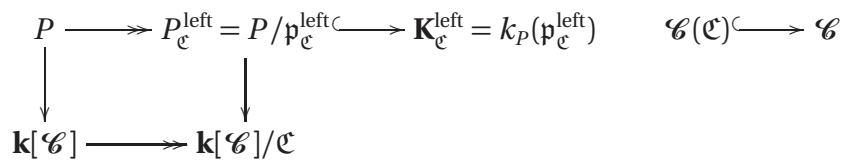
\xymatrix{
& R\ar@{^{(}->}[r] & \Mb\ar@{-}@/^1pc/[d]^(0.75)H \ar@{-}@/^2pc/[dd]^G\\
Z\ar[r]^{\sim}_{\copie} & Q \ar@{^{(}->}[r]\ar@{^{(}->}[u] & \Lb \ar@{^{(}->}[u] \\
& P \ar@{^{(}->}[r]\ar@{^{(}->}[u] & \Kb \ar@{^{(}->}[u]\\
&\kb[\CCB]\ar@{^{(}->}[u]
}
\ \ \ \ \ \ 
\xymatrix{
\RCB \ar[d]_{\rho_H} \ar@/^2pc/[dd]^{\rho_G} \\
\QCB \ar[d]_\Upsilon \\
\PCB=\CCB\times V/W\times V^*/W\ar[d]_\pi \\
\CCB}
$$

$$\xymatrix{
R\ar@{->>}[r] & R_c=R/\rG_c \ar@{^{(}->}[r] &\Mb_c=k_R(\rG_c) \ar@{-}@/^1pc/[d]^{D_c\cap H}
 \ar@{-}@/^3pc/[dd]^{D_c} \\
Q \ar@{->>}[r]\ar@{^{(}->}[u] & Q_c=Q/\qG_c \ar@{^{(}->}[u]\ar@{^{(}->}[r] &\Lb_c=k_Q(\qG_c)\ar@{^{(}->}[u] \\
P \ar@{->>}[r]\ar@{^{(}->}[u] & P_c=P/\pG_c \ar@{^{(}->}[u]\ar@{^{(}->}[r] &\Kb_c=k_P(\pG_c)\ar@{^{(}->}[u] \\
\kb[\CCB] \ar@{->>}[r]\ar@{^{(}->}[u] & \kb=\kb[\CCB]/\CG_c \ar@{^{(}->}[u]
}
\ \ \ \ \ \ 
\xymatrix{
\RCB_c\ar[r]\ar[d] & \RCB \ar[d] \\
\QCB_c\ar[r]\ar[d] & \QCB \ar[d]^\Upsilon \\
\{c\}\times V/W\times V^*/W\ar[r]\ar[d] & \PCB=\CCB\times V/W\times V^*/W\ar[d] \\
\{c\} \ar[r] & \CCB}
$$

$$\xymatrix{
P\ar@{->>}[r]\ar[d] & P_\CG^{\gauche}=P/\pG_\CG^\gauche \ar[d]\ar@{^{(}->}[r] & \Kb_\CG^\gauche=k_P(\pG_\CG^\gauche) \\
\kb[\CCB] \ar@{->>}[r] & \kb[\CCB]/\CG
}
\ \ \ \ \ \ 
\xymatrix{
\CCB(\CG) \ar@{^{(}->}[r] & \CCB
}$$

\newpage

\printindex


\newpage

\begin{thebibliography}{AAAA}
	\bibitem[AgFeVe]{AgFeVe} {\sc L.~Aguirre, G.~Felder \& A.P.~Veselov},
	{\em Gaudin subalgebras and wonderful models},
	Sel. Math. New Ser. {\bf 22} (2016), 1057--1071.
\bibitem[AlFo]{alev} {\sc J. Alev \& L. Foissy}, 
{\it Le groupe des traces de Poisson de certaines alg\`ebres d'invariants}, 
Comm. Algebra {\bf 37} (2009), 368--388. 

\bibitem[Ari]{ariki} {\sc S. Ariki}, 
{\it Representation theory of a Hecke algebra of $G(r, p, n)$}, 
J. Algebra {\bf 177} (1995), 164--185.

\bibitem[ArKo]{ariki-koike} {\sc S. Ariki \& K. Koike}, 
{\it A Hecke algebra of $(\ZM /r \RM ) \rtimes \SG_n$ 
and construction of its irreducible representations}, 
Adv. in Math. {\bf 106} (1994), 216--243.

\bibitem[Bel1]{bellamy these} {\sc G. Bellamy}, 
{\it Generalized Calogero-Moser spaces and rational Cherednik algebras}, PhD thesis, 
University of Edinburgh, 2010.

\bibitem[Bel2]{bellamy g4} 
{\sc G. Bellamy}, {\it On singular Calogero-Moser spaces}, 
Bull. of the London Math. Soc. {\bf 41} (2009), 315--326.

\bibitem[Bel3]{bellamy factorization} 
{\sc G. Bellamy}, {\it Factorisation in generalized Calogero-Moser spaces}, 
J. Algebra {\bf 321} (2009), 338--344.

\bibitem[Bel4]{bellamy cuspidal}
{\sc G. Bellamy}, {\it Cuspidal representations of rational Cherednik algebras at $t=0$}, 
Math. Z. {\bf 269} (2011), 609--627.

\bibitem[Bel5]{bellamy} {\sc G. Bellamy}, 
{\it The Calogero-Moser partition for $G(m,d,n)$}, 
Nagoya Math. J. {\bf 207} (2012), 47--77.

\bibitem[Bel6]{bellamyVerma} {\sc G. Bellamy}, 
{\it Endomorphisms of Verma modules for rational Cherednik algebras},
Transform. Groups {\bf 19} (2014), 699--720.

\bibitem[BeTh1]{BelTh1} {\sc G. Bellamy \& U. Thiel}, 
{\it Highest weight theory for finite-dimensional graded algebras with triangular
decomposition},
Adv. Math. {\bf 330} (2018) 361--419.
\bibitem[BeTh2]{BelTh2} {\sc G. Bellamy \& U. Thiel}, 
	{\it Cores of graded algebras with triangular decomposition},
		preprint arXiv:1711.00780.


\bibitem[BST]{BelSchTh} {\sc G. Bellamy, T. Schedler \& U. Thiel}, 
{\it Hyperplane arrangements associated to symplectic quotient singularities},
in {\it Phenomenological approach to algebraic geometry}, 25--45, 
Banach Center Publ., {\bf 116},
Polish Acad. Sci. Inst. Math., Warsaw, 2018.


\bibitem[Ben]{benard} {\sc M. Benard}, 
{\it Schur indices and splitting fields of the unitary reflection groups}, 
J. Algebra {\bf 38} (1976), 318--342. 

\bibitem[Besi1]{bessis} {\sc D. Bessis}, 
{\it Sur le corps de d\'efinition d'un groupe de r\'eflexions complexe}, 
Comm. Algebra {\bf 25} (1997), 2703--2716.
 
\bibitem[Bes2]{bessis diagram} {\sc D. Bessis},
{\it Zariski theorems and diagrams for braid groups}, 
Invent. Math. {\bf 145} (2001), 487--507.

\bibitem[BBR]{BBR} {\sc D. Bessis, C. Bonnaf\'e \& R. Rouquier}, 
{\it Quotients et extensions de groupes de r\'eflexion complexes}, 
Math. Ann. {\bf 323} (2002), 405--436.

\bibitem[Bia]{bialynicki} {\sc A. Bialynicki-Birula}, 
{\it Some theorems on actions of algebraic groups}, 
Ann. of Math. {\bf 98} (1973), 480--497.

\bibitem[Bj]{bjork} {\sc J.E. Bj\"ork},
	{\it Rings of differential operators}, North-Holland, 1979.

\bibitem[Bon1]{bonnafe two} {\sc C. Bonnaf\'e}, 
{\it Two-sided cells in type $B$ (asymptotic case)}, 
J. Algebra {\bf 304} (2006), 216--236.

\bibitem[Bon2]{bonnafe continu} 
{\sc C. Bonnaf\'e}, 
{\it Semicontinuity properties of Kazhdan-Lusztig cells}, 
New-Zealand J. Math. {\bf 39} (2009), 171--192.

\bibitem[Bon3]{bonnafe faux} {\sc C. Bonnaf\'e}, 
{\it On Kazhdan-Lusztig cells in type $B$}, 
J. Algebraic Combin. {\bf 31} (2010), 53--82; 
{\it Erratum to: On Kazhdan-Lusztig cells in type $B$}, 
J. Algebraic Combin. {\bf 35} (2012), 515--517.

\bibitem[Bon4]{bonnafe b} {\sc C. Bonnaf\'e}, 
{\it Constructible characters and $\bb$-invariant}, 
Bull. Belg. Math. Soc. {\bf 22} (2015), 377--390.

\bibitem[Bon5]{bonnafe livre} {\sc C. Bonnaf\'e}, 
{\it Kazhdan-Lusztig cells with unequal parameters}, 
Algebra and Applications {\bf 24}, Springer, Cham, 2017. xxv+348 pages,

\bibitem[Bon6]{bonnafe diedral} {\sc C. Bonnaf\'e}, 
{\it On the Calogero-Moser space associated with dihedral groups}, 
Ann. Math. Blaise Pascal {\bf 25} (2018), 265--298.

\bibitem[Bon7]{bonnafe autosymp} {\sc C. Bonnaf\'e}, 
{\it Automorphisms and symplectic leaves of Calogero-Moser spaces}, 
preprint arXiv:2112.12405. 

\bibitem[Bon8]{bonnafe dihedral2} {\sc C. Bonnaf\'e}, 
	{\em On the Calogero-Moser space associated with dihedral groups II.
	The equal parameter case}, preprint arXiv:2112.12401.
\bibitem[BoDy]{bonnafe dyer} 
{\sc C. Bonnaf\'e \& M. Dyer}, 
{\it Semidirect product decomposition of Coxeter groups}, 
Comm. in Algebra {\bf 38} (2010), 1549--1574.

\bibitem[BoGec]{bonnafe geck} {\sc C. Bonnaf\'e \& M. Geck}, 
{\it Conjugacy classes of involutions and Kazhdan-Lusztig cells}, 
Represent. Theory {\bf 18} (2014), 155--182. 

\bibitem[BoGer]{bonnafe germoni} {\sc C. Bonnaf\'e \& J. Germoni}, 
{\it Calogero-Moser cells of dihedral groups at equal parameters}, 
preprint arXiv:2112.13683.

\bibitem[BGIL]{BGIL} {\sc C. Bonnaf\'e, M. Geck, L. Iancu \& T. Lam}, 
{\it On domino insertion and Kazhdan-Lusztig cells in type $B_n$}, 
in {\it Representation theory of algebraic groups and quantum groups}, 
33--54, Progress in Mathematics {\bf 284}, Birkh\"auser/Springer, New York, 2010.

\bibitem[BoIa]{bonnafe iancu} {\sc C. Bonnaf\'e \& L. Iancu}, 
{\it Left cells in type $B_n$ with unequal parameters}, 
Representation Theory {\bf 7} (2003), 587--609.

\bibitem[BoMa]{bonnafe maksimau} {\sc C. Bonnaf\'e \& R. Maksimau}, 
{\it Fixed points in smooth Calogero-Moser spaces}, 
Ann. Inst. Fourier {\bf 71} (2021), 643--678.

\bibitem[BoRo1]{cm 1} {\sc C. Bonnaf\'e \& R. Rouquier}, 
{\it Cellules de Calogero-Moser}, preprint arXiv:1302.2720.

\bibitem[BoRo]{BR} {\sc C. Bonnaf\'e \& R. Rouquier},
{\it An asympotic cell category for cyclic groups}, 
J. Algebra {\bf 558} (2020), 102--128.

\bibitem[BoSh]{bonnafe shan} {\sc C. Bonnaf\'e \& P. Shan}, 
{\it On the cohomology of Calogero-Moser spaces}, 
Int. Math. Res. Not. IMRN {\bf 2020}, 1091--1111.

\bibitem[BoTh]{bonnafe thiel} {\sc C. Bonnaf\'e \& U. Thiel}, 
{\it Computational aspects of Calogero-Moser spaces}, 
preprint arXiv:2112.15495. 


\bibitem[BoChCh]{BCC} {\sc C. Boura, E. Chavli \& M. Chlouveraki}, 
{\it The BMM symmetrising trace conjecture for the exceptional $2$-reflection groups
of rank $2$}, J. Algebra {\bf 558} (2020), 176--198.

\bibitem[BoChChKa]{BCCK} {\sc C. Boura, E. Chavli, M. Chlouveraki \& K. Karvounis}, 
{\it The BMM symmetrising trace conjecture for groups $G_4$, $G_5$, $G_6$, $G_7$, $G_8$}, 
J. Symbolic Comput. {\bf 96} (2020), 62--84.

\bibitem[Bou]{bourbaki} {\sc N. Bourbaki}, 
{\it Commutative algebra, chapters 5, 6, 7}.

\bibitem[Bou10]{bourbaki10} {\sc N. Bourbaki}, {\it Commutative algebra, chapter 10}.

\bibitem[BreMa]{bremke malle} {\sc K. Bremke \& G. Malle}, 
{\it Reduced words and a length function for $G(e, 1, n)$}, 
Indag. Math. {\bf 8} (1997), 453--469.

\bibitem[Bri]{brieskorn} {\sc E. Brieskorn}, 
{\it Die Fundamentalgruppe des Raumes der regul\"aren Orbits 
einer endlichen komplexen Spiegelungsgruppe}, 
Invent. Math. {\bf 12} (1971), 57--61.

\bibitem[BrGoWh]{BGW} {\sc A. Brochier, I. Gordon \& N. White},
	{\it Calogero-Moser cells in type A and the RSK correspondence},
preprint arXiv:2012.10177.

\bibitem[Bro1]{broueHighman} {\sc M. Brou\'e}, 
{\it Higman criterion revisited}, 
Michigan Journal of Mathematics {\bf 58} (2009), 
125--179. 

\bibitem[Bro2]{broue} {\sc M. Brou\'e}, 
{\it Introduction to complex reflection groups and their 
braid groups}, Lecture Notes in Mathematics {\bf 1988}, 2010, Springer.

\bibitem[BrKi]{broue kim} {\sc M. Brou\'e \& S. Kim}, 
{\it Familles de caract\`eres des alg\`ebres de Hecke cyclotomiques}, 
Adv. Math. {\bf 172} (2002), 53--136. 

\bibitem[BroMa]{BrMa} {\sc M. Brou\'e \& G. Malle},
{\it Zyklotomische Heckealgebren},
Ast\'erisque {\bf 212} (1993), 119--189.

\bibitem[BMM1]{BMM} {\sc M. Brou\'e, G. Malle \& J. Michel}, 
{\it Towards Spetses I}, 
Transformation Groups {\bf 4} (1999) 157--218. 

\bibitem[BMM2]{BMM2} {\sc M. Brou\'e, G. Malle \& J. Michel}, 
{\it Split spetses for primitive reflection groups, With an erratum to [MR1712862]},
Ast\'erisque {\bf 359} (2014), vi+146 pp.

\bibitem[BrMaRo]{BMR} {\sc M. Brou\'e, G. Malle \& R. Rouquier}, 
{\it Complex reflection groups, braid groups, Hecke algebras}, 
J. Reine Angew. Math. {\bf 500} (1998), 127--190.

\bibitem[BrMi]{broue-michel} {\sc M. Brou\'e \& J. Michel}, 
{\it Sur certains \'el\'ements r\'eguliers des groupes de Weyl 
et les vari\'et\'es de Deligne-Lusztig associ\'ees}, in 
{\it Finite reductive groups} (Luminy, 1994), 73--139, Progr. Math. {\bf 141}, 
Birkh\"auser, Boston, MA, 1997.

\bibitem[BrGo]{BG} {\sc K. A. Brown \& I. G. Gordon}, 
{\it The ramification of centres: Lie algebras in positive
characteristic and quantised enveloping algebras}, 
Math. Z. {\bf 238} (2001), 733--779.

\bibitem[BGS]{BGS} {\sc K. A. Brown, I. G. Gordon \& C. H. Stroppel}, 
{\it Cherednik, Hecke and quantum algebras as free Frobenius and Calabi-Yau extensions}, 
J. Algebra {\bf 319} (2008), 1007--1034. 

\bibitem[BrKl]{BrKl} {\sc J. Brundan, Jonathan \& A. Kleshchev}, 
{\it Schur-Weyl duality for higher levels}, 
Selecta Math. (N. S.) {\bf 14} (2008), 1--57.


\bibitem[Cha1]{chavli1} {\sc E. Chavli}, 
{\it The BMR freeness conjecture for exceptional groups of rank $2$}, 
doctoral thesis, Univ. Paris Diderot (Paris 7), 2016.

\bibitem[Cha2]{chavli2} {\sc E. Chavli}, 
{\it The BMR freeness conjecture for the first two families 
of the exceptional groups of rank $2$},
Comptes Rendus Math\'ematiques {\bf 355} (2017), 1--4.

\bibitem[Cha3]{chavli3} {\sc E. Chavli}, 
{\it The BMR freeness conjecture for the tetrahedral and octahedral family}, 
Comm. Algebra {\bf 46} (2018), 386--464.


\bibitem[Chl2]{chlouveraki} {\sc M. Chlouveraki}, 
{\it Rouquier blocks of the cyclotomic Hecke algebras}, 
C. R. Math. Acad. Sci. Paris {\bf 344} (2007), 615--620.

\bibitem[Chl3]{chlouveraki B} {\sc M. Chlouveraki}, 
{\it Rouquier blocks of the cyclotomic Ariki-Koike algebras}, 
Algebra Number Theory {\bf 2} (2008), 689--720.

\bibitem[Chl4]{chlouveraki LNM} {\sc M. Chlouveraki}, 
{\it Blocks and families for cyclotomic Hecke algebras}, 
Lecture Notes in Mathematics {\bf 1981}, 2009, Springer.

\bibitem[Chl5]{chlouveraki D} {\sc M. Chlouveraki}, 
{\it Rouquier blocks of the cyclotomic Hecke algebras of $G(de,e,r)$}, 
Nagoya Math. J. {\bf 197} (2010), 175--212.

\bibitem[CPS1]{CPS} {\sc E. Cline, B. Parshall \& L. Scott},
{\it Finite-dimensional algebras and highest weight categories},
J. Reine Angew. Math. {\bf 391} (1988), 85--99.

\bibitem[CPS2]{CPS2} {\sc E. Cline, B. Parshall \& L. Scott},
{\it Integral and graded quasi-hereditary algebras, I},
 J. of Alg. {\bf 131} (1990), 126--160.

\bibitem[Eti]{Et} {\sc P.~Etingof},
{\it Proof of the Brou\'e-Malle-Rouquier conjecture in characteristic zero
(after I. Losev and I. Marin - G. Pfeiffer)}, 
Arnold Math. J. {\bf 3} (2017), 445--449.

\bibitem[EtGi]{EG} {\sc P. Etingof \& V. Ginzburg}, 
{\it Symplectic reflection algebras, 
Calogero-Moser space, and deformed Harish-Chandra homomorphism}, 
Invent. Math. {\bf 147} (2002), 243--348.

\bibitem[EtRa]{ER} {\sc P.~Etingof \& E.~Rains},
{\it Central extensions of preprojective algebras, the quantum Heisenberg algebra,
and $2$-dimensional complex reflection groups},
J. of Alg. {\bf 299} (2006), 570--588.


\bibitem[Gec1]{geck induction} {\sc M. Geck}, 
{\it On the induction of Kazhdan-Lusztig cells}, 
Bull. London Math. Soc. {\bf 35} (2003), 608--614. 

\bibitem[Gec2]{geck f4} {\sc M. Geck}, 
{\it Computing Kazhdan--Lusztig cells for unequal parameters}, 
J. Algebra {\bf 281} (2004) 342--365.

\bibitem[Gec3]{geck plus} {\sc M. Geck}, 
{\it Left cells and constructible representations}, 
Represent. Theory {\bf 9} (2005), 385--416; 
{\it Erratum to: "Left cells and constructible representations''}, 
Represent. Theory {\bf 11} (2007), 172--173.

\bibitem[GeIa]{geck iancu} {\sc M. Geck \& L. Iancu}, 
{\it Lusztig's $\ab$-function in type $B_n$ in the asymptotic case}, 
Nagoya Math. J. {\bf 182} (2006), 199--240.

\bibitem[GePf]{geck} {\sc M. Geck \& G. Pfeiffer}, 
{\it Characters of finite Coxeter groups and Iwahori-Hecke algebras}, 
London Mathematical Society Monographs, New Series {\bf 21}, The Clarendon Press, 
Oxford University Press, New York, 2000, xvi+446 pp.

\bibitem[GeRo]{geck rouquier} 
{\sc M. Geck \& R. Rouquier}, 
{\it Centers and simple modules for Iwahori-Hecke algebras}, in 
{\it Finite reductive groups} (Luminy, 1994), 
251--272, Progr. in Math. {\bf 141}, Birkh\"auser Boston, 
Boston, MA, 1997.

\bibitem[GGOR]{ggor} {\sc V. Ginzburg, N. Guay, E. Opdam \& R. Rouquier}, 
{\it On the category $\OC$ for rational Cherednik algebras}, 
{\it Invent. Math.} {\bf 154} (2003), 617--651.

\bibitem[GiKa]{GK} {\sc V. Ginzburg \& D. Kaledin}, 
{\it Poisson deformations of symplectic quotient singularities}, 
Adv. in Math. {\bf 186} (2004), 1--57.

\bibitem[Gor1]{gordon} {\sc I. Gordon}, 
{\it Baby Verma modules for rational Cherednik algebras}, 
Bull. London Math. Soc. {\bf 35} (2003), 321--336.

\bibitem[Gor2]{gordon B} {\sc I. Gordon}, 
{\it Quiver varieties, category $\OC$ for rational Cherednik algebras, and Hecke algebras}, 
Int. Math. Res. Pap. IMRP (2008), no. 3, Art. ID rpn006, 69 pp.

\bibitem[GoMa]{gordon martino} {\sc I. G. Gordon \& M. Martino}, 
{\it Calogero-Moser space, restricted rational Cherednik algebras and two-sided cells},
Math. Res. Lett. {\bf 16} (2009), 255--262. 

\bibitem[HKRW]{HaKaRyWe} 
	{\sc I. Halacheva, J. Kamnitzer, L. Rybnikov \& A. Weekes},
{\it Crystals and monodromy of Bethe vectors},
Duke Math. J. {\bf 169} (2020), 2337--2419.

\bibitem[HoNa]{HN} {\sc R. R. Holmes \& D. K. Nakano}, 
{\it Brauer-type reciprocity for a class of graded associative algebras}, 
J. of Algebra {\bf 144} (1991), 117--126.

\bibitem[KaSc]{KS} {\sc M.~Kashiwara \& P.~Schapira},
{\it Categories and sheaves}, Grundlehren der Mathematischen Wissenschaften [Fundamental Principles of Mathematical Sciences], {\bf 332}. Springer-Verlag, Berlin, 2006. x+497 pp.

\bibitem[KaLu]{KL} {\sc D. Kazhdan \& G. Lusztig}, 
{\it Representations of Coxeter groups and Hecke algebras}, 
Invent. Math. {\bf 53} (1979), 165--184.

\bibitem[Lac1]{lacabanne tordu} {\sc A. Lacabanne}, 
{\it Crossed S-matrices and Fourier matrices for Coxeter groups with 
automorphism}, J. Algebra {\bf 558} (2020), 550--581.

\bibitem[Lac2]{lacabanne super} {\sc A. Lacabanne}, 
{\it Slightly degenerate categories and $\ZM$-modular data}, 
Int. Math. Res. Not. IMRN 2021, 9340--9374.

\bibitem[Lac3]{lacabanne gd1n} {\sc A. Lacabanne}, 
{\it Drinfeld double of quantum groups, tilting modules and $\ZM$-modular 
data associated to complex reflection groups},
J. Comb. Algebra {\bf 4} (2020), 269--323.

\bibitem[Lac4]{lacabanne Fourier} {\sc A. Lacabanne}, 
{\it Fourier matrices for $G(d,1,n)$ from quantum general linear groups},
preprint arXiv:2011.11332, to appear in Journal of Algebra.

\bibitem[LeMi]{LeMi} {\sc B.~Leclerc \& H.~Miyachi},
{\it Constructible characters and canonical bases},
J. of Alg. {\bf 277} (2004), 298--317.

\bibitem[Los]{Lo} {\sc I.~Losev},
{\it Finite-dimensional quotients of Hecke algebras},
Algebra and Number Theory {\bf 9} (2015), 493--502.

\bibitem[Lus1]{Lu2} {\sc G.~Lusztig},
{\it Left cells in Weyl groups}, in {\it Lie group representations, I}, 99--111,
Lecture Notes in Math. {\bf 1024}, Springer, Berlin, 1983.

\bibitem[Lus2]{lusztig orange} {\sc G. Lusztig}, 
{\it Characters of reductive groups over finite fields}, 
Ann. Math. Studies {\bf 107}, Princeton UP (1984), 384 pp.

\bibitem[Lus3]{lusztigexotic} {\sc G. Lusztig},
{\it Exotic Fourier transform},
Duke Math. J. {\bf 73} (1994), 227--242.

\bibitem[Lus4]{lusztig} 
{\sc G.~Lusztig}, 
\emph{Hecke algebras with unequal parameters}, CRM Monograph Series 
{\bf 18}, American Mathematical Society, Providence, RI (2003), 136 pp.

\bibitem[Lus5]{lusztigdisconnected} 
{\sc G.~Lusztig}, 
{\it Character sheaves on disconnected groups I--X},
Representation Theory {\bf 7} (2003), 374--403; {\bf 8} (2004), 72--124;
{\bf 8} (2004), 125--144;  {\bf 8} (2004) 145--178; Errata {\bf 8} (2004),
179; {\bf 8} (2004), 346--376; {\bf 8} (2004), 377--413; {\bf 9} (2005),
209--266; {\bf 10} (2006), 314--352; {\bf 10} (2006), 353--379; 
{\bf 13} (2009), 82--140.

\bibitem[Magma]{magma}
{\sc W. Bosma, J. Cannon \& C. Playoust}, 
{\it The Magma algebra system. I. The user language}, 
J. Symbolic Comput. {\bf 24} (1997), 235--265.

\bibitem[Mal1]{malleexotic} {\sc G. Malle}, 
{\it Appendix: An exotic Fourier transform for $H_4$},
Duke J. Math. {\bf 73} (1994), 243--248.

\bibitem[Mal2]{malleunipotent} {\sc G. Malle}, 
{\it Unipotente Grade imprimitiver komplexer Spiegelungsgruppen},
J. Alg. {\bf 177} (1995), 768--826.
                
\bibitem[Mal3]{malle} {\sc G. Malle}, 
{\it On the rationality and fake degrees of characters of cyclotomic algebras}, 
J. Math. Sci. Univ. Tokyo {\bf 6} (1999), 647--677.

\bibitem[Mal4]{malle generic} {\sc G. Malle}, 
{\it On the generic degrees of cyclotomic algebras}, 
Represent. Theory {\bf 4} (2000), 342--369.

\bibitem[MalMath]{malle mathas} {\sc G. Malle \& A. Mathas}, 
{\it Symmetric cyclotomic Hecke algebras}, 
J. Algebra {\bf 205} (1998), 275--293.

\bibitem[MalMatz]{mama} {\sc G. Malle \& B. H. Matzat}, 
{\it Inverse Galois theory}, 
Springer Monographs in Mathematics, 
Springer-Verlag, Berlin (1999), xvi+436 pp.

\bibitem[MaMi]{malle michel} {\sc G. Malle \& J. Michel}, 
{\it Constructing representations of Hecke algebras for complex reflection groups}, 
LMS J. Comput. Math. {\bf 13} (2010), 426--450.

\bibitem[Mar1]{marin} {\sc I. Marin}, 
{\it Report on the Brou\'e-Malle-Rouquier conjectures}, 
in {\it Perspectives in Lie theory}, Springer INdAM Series, pp. 351--359, 2017.

\bibitem[Mar2]{marin1} {\sc I. Marin}, 
{\it The cubic Hecke algebra on at most $5$ strands}, 
J. Pure Appl. Algebra {\bf 216} (2012), 2754--2782.

\bibitem[Mar3]{marin2} {\sc I. Marin}, 
{\it The freeness conjecture for Hecke algebras of complex reflection groups, 
and the case of the Hessian group $G_{26}$}, 
J. Pure Appl. Algebra {\bf 218} (2014), 704--720.

\bibitem[Mar4]{marin3} {\sc I. Marin}, 
{\it Proof of the BMR conjecture for $G_{20}$ and $G_{21}$}, 
J. Symbolic Comput. {\bf 92} (2019), 1--14. 

\bibitem[MaPf]{marin-pfeiffer} {\sc I.~Marin \& G.~Pfeiffer},
{\it The BMR freeness conjecture for the $2$-reflection groups},
Math. Comp. {\bf 86} (2017), 2005--2023.

\bibitem[MaWa]{marin wagner} {\sc I. Marin \& E. Wagner}, 
{\it Markov traces on the BMW algebras}, 
preprint arXiv:1403.4021.

\bibitem[Mart1]{martino} {\sc M. Martino}, 
{\it The Calogero-Moser partition and Rouquier 
families for complex reflection groups}, 
J. of Algebra {\bf 323} (2010), 193--205.

\bibitem[Mart2]{martino 2} {\sc M. Martino}, 
{\it Blocks of restricted rational Cherednik algebras for $G(m,d,n)$}, 
J. Algebra {\bf 397} (2014), 209--224. 

\bibitem[Mat]{matsumura} {\sc H. Matsumura}, {\it Commutative ring theory}, 
Cambridge Studies in Advanced Mathematics {\bf 8}, Cambridge University Press, 
Cambridge, 1986, xiv+320 pp.

\bibitem[MuTaVa1]{MuTaVa1} {\sc E. Mukhin, V. Tarasov \& A. Varchenko},
{\it Gaudin Hamiltonians generate the Bethe algebra of a tensor
power of the vector representation of $\gG\lG_N$},
Algebra i Analiz {\bf 22} (2010), 177--190; translation in 
St. Petersburg Math. J. {\bf 22} (2011), 463--472.

\bibitem[MuTaVa2]{MuTaVa2} {\sc E. Mukhin, V. Tarasov \& A. Varchenko},
	{\it Bethe subalgebras of the group algebra of the symmetric group},
	Transform. Groups {\bf 18} (2013), 767--801.

\bibitem[MuTaVa3]{MuTaVa3} {\sc E. Mukhin, V. Tarasov \& A. Varchenko},
	{\it Bethe algebra of Gaudin model, Calogero-Moser space, and
	Cherednik algebra},
	Int. Math. Res. Not. {\bf 5} (2014),1174--1204.
\bibitem[M\"ul]{muller} {\sc B. J. M\"uller}, 
{\it Localisation in non-commutative Noetherian rings}, 
Canad. J. Math. {\bf 28} (1976), 600--610.

\bibitem[NaVa]{NaVa} {\sc C. Nastasescu \& F. Van Oystaeyen},
	{\it Graded ring theory}, North-Holland, 1982.
%


\bibitem[Ray]{raynaud} {\sc M. Raynaud}, {\it Anneaux locaux hens\'eliens}, 
Lecture Notes in Mathematics {\bf 169}, Springer-Verlag, Berlin-New York (1970), v+129 pp. 

\bibitem[Rou]{rouquier schur} {\sc R. Rouquier}, 
{\it $q$-Schur algebras and complex reflection groups}, 
Mosc. Math. J. {\bf 8} (2008), 119--158. 

\bibitem[Row]{Row} {\sc L.~Rowen},
{\it Ring Theory, Volume I}, Academic Press, 1988.

\bibitem[Sage]{sage} {\sc Developers, W. Stein, D. Joyner, D. Kohel, J. Cremona \& B. Er\"ocal} 
(2020). SageMath, version 9.0. Retrieved from {\tt http://www.sagemath.org}.

\bibitem[Ser]{serre} {\sc J.-P. Serre}, {\it Alg\`ebre locale. Multiplicit\'es}, 
Cours au Coll\`ege de France, 1957--1958, r\'edig\'e par P. Gabriel, 
Lecture Notes in Mathematics {\bf 11}, Springer-Verlag, Berlin-New York (1965), vii+188 pp.

\bibitem[SGA1]{sga} {\sc A. Grothendieck}, {\it Rev\^etements \'etales et groupe 
fondamental}, S\'eminaire de G\'eom\'etrie alg\'ebrique du Bois Marie 1960-1961, 
Documents Math\'ematiques {\bf 3}, Soci\'et\'e Math\'ematique de France, Paris, 2003, 
327+xviii pages.

\bibitem[ShTo]{ST} {\sc G. C. Shephard \& J. A. Todd}, 
{\it Finite unitary reflection groups},
Canad. J. Math. {\bf 6} (1954), 274--304.


\bibitem[Soe]{Soe} {\sc W. Soergel},
{\it The combinatorics of Harish-Chandra bimodules},
J. Reine Angew. Math. {\bf 429} (1992), 49--74.

\bibitem[Spr]{springer} {\sc T.A. Springer}, 
{\it Regular elements of finite reflection groups}, 
Invent. Math. {\bf 25} (1974), 159--198. 

\bibitem[Thi1]{thiel} {\sc U. Thiel}, 
{\it A counter-example to Martino's conjecture about generic Calogero-Moser families}, 
Algebr. Represent. Theory {\bf 17} (2014), 1323--1348. 

\bibitem[Thi2]{thiel thesis} {\sc U. Thiel}, {\it On restricted rational
Cherednik algebras}, Ph.D. Thesis, Kaiserslautern (2014).

\bibitem[Thi3]{thiel champ} {\sc U. Thiel}, 
{\it {\tt CHAMP}: a Cherednik algebra package}, 
LMS J. Comput. Math. {\bf 18} (2015), 266--307. 
Available at 

\centerline{\tt https://github.com/ulthiel/Champ}

\bibitem[Tsu]{tsuchioka} {\sc S. Tsuchioka}, 
{\it BMR freeness for icosahedral family}, 
Exp. Math. {\bf 29} (2020), 234--245.

\bibitem[Wei1]{We1} {\sc C.A.Weibel},
	{\it An introduction to homological algebra}, Cambridge Univ. Press,
	1994.
\bibitem[Wei2]{We2} {\sc C.A.Weibel},
{\it The K-Book},
American Math. Soc., 2013.

\bibitem[Whi]{Wh} {\sc N. White},
{\it The Monodromy of real Bethe vectors for the Gaudin model},
J. Comb. Algebra {\bf 2} (2018), 259--300. 

\end{thebibliography}
\end{document}